\documentclass[smallextended]{svjour3}
\usepackage[active]{srcltx}
\usepackage{amsfonts,amssymb,amsmath,amssymb,amsopn,amsxtra,amstext,amsbsy}
\usepackage[cp1251]{inputenc}
\usepackage[english]{babel}
\usepackage{setspace}
\usepackage{tikz}
\allowdisplaybreaks[4]
\makeatletter

\renewcommand\section{\@startsection{section}{1}{\z@}%
    {-21dd plus-8pt minus-4pt}{10.5dd}
     {\large\bfseries\boldmath}}

\renewcommand\subsection{\@startsection{subsection}{2}{\z@}%
  {-13pt plus-8pt minus-4pt}{\z@}{\bfseries\boldmath}}

\renewcommand\subsubsection{\@startsection{subsubsection}{3}{\z@}%
    {-13pt plus-8pt minus-4pt}{\z@}{\bfseries\boldmath}}

\makeatother

\DeclareSymbolFont{AMSb}{U}{msb}{m}{n}
\DeclareMathAlphabet{\Bb}{U}{msb}{m}{n}
\DeclareSymbolFont{AMSb}{U}{eur}{m}{n}
\DeclareMathAlphabet{\eurm}{U}{eur}{m}{n}
\DeclareSymbolFont{AMSb}{U}{eus}{m}{n}
\DeclareMathAlphabet{\eusm}{U}{eus}{m}{n}
\DeclareSymbolFont{AMSb}{U}{euf}{m}{n}
\DeclareMathAlphabet{\eufm}{U}{euf}{m}{n}
\numberwithin{equation}{section}
\numberwithin{theorem}{section}
\numberwithin{lemma}{section}
\newtheorem{rem}[theorem]{Remark}
\newtheorem{clash}{}
\newtheorem{thmx}{Theorem}

\newtheorem{lemmx}{Lemma}

\newtheorem{corsecttwo}[theorem]{Corrollary}
\newtheorem{lemsectfour}[theorem]{Lemma}
\newtheorem{lemsectfive}[theorem]{Lemma}
\newtheorem{lemsectsix}[theorem]{Lemma}
\newtheorem{defsectsix}[theorem]{Definition}
\newtheorem{lemsectseven}[theorem]{Lemma}
\newtheorem{lemsecteight}[theorem]{Lemma}

\newtheorem{corsecteight}[theorem]{Corrollary}
\newtheorem{lemsectthree}[theorem]{Lemma}
\newtheorem{corsectthree}[theorem]{Corrollary}
\newtheorem{lemsectnine}[theorem]{Lemma}
\newtheorem{defsectnine}[theorem]{Definition}
\newtheorem{lemsectten}[theorem]{Lemma}

\newcommand{\fo}[1]{{\mbox{\footnotesize{$#1$}}}}
\newcommand{\tn}[1]{{\mbox{\tiny{$#1$}}}}
\newcommand{\nor}[1]{{\mbox{\normalsize{$#1$}}}}
\newcommand{\sm}[1]{{\mbox{\small{$#1$}}}}

\newcommand{\im}{{\rm{Im}}}
\newcommand{\re}{{\rm{Re}}}

\newcommand{\ie}{\lambda_{{\tn{\triangle}}}}
\newcommand{\he}{F_{\!{\tn{\triangle}}}}
\newcommand{\het}{F_{\!{\tn{\triangle}}}}
\newcommand{\fet}{\mathcal{F}_{{\tn{\square}}}}
\newcommand{\fetr}{\widehat{\mathcal{F}}_{{\tn{\square}}}}
\newcommand{\fetu}{\mathcal{F}_{{\tn{\bigtriangledown}}}}
\newcommand{\fetur}{\widehat{\mathcal{F}}_{{\tn{\bigtriangledown}}}}
\newcommand{\fetd}{\mathcal{F}_{{\tn{\triangle}}}}
\newcommand{\get}{\mathcal{F}_{{\tn{\square}}}}
\newcommand{\getu}{\mathcal{F}_{{\tn{\bigtriangledown}}}}

\newcommand{\Log}{{\rm{Log}}\,}
\newcommand{\Arg}{{\rm{Arg}}\,}
\newcommand{\diff}{\mathrm{d}}
\newcommand{\imag}{ \hspace{0,02cm}\mathrm{i} \hspace{0,015cm}}

\newcommand{\oh}{\mathrm{o}}
\newcommand{\e}{ \hspace{0,02cm}\mathrm{e} \hspace{0,015cm}}

\DeclareMathOperator*{\limsp}{\overline{\lim}}
\DeclareMathOperator*{\limlw}{\underline{\lim}}

\usepackage[linktocpage=true,bookmarksopenlevel=1]{hyperref}
\hypersetup{bookmarksnumbered}

\journalname{PREPRINT}

\begin{document}
\title{ Hyperbolic Fourier series}



\titlerunning{Hyperbolic Fourier series}        

\author{$\phantom{a}$\\[-0,5cm]Andrew Bakan \and H{\aa}kan Hedenmalm \and\\[0,1cm] Alfonso Montes-Rodr\'{\i}guez
 \and \\[0,12cm] Danylo Radchenko \and Maryna Viazovska}

\authorrunning{Bakan, Hedenmalm, Montes-Rodr\'{\i}guez, Radchenko, Viazovska}


\institute{Andrew Bakan \at
               Institute of Mathematics, National Academy of Sciences of Ukraine,
01601 Kyiv, Ukraine\\
              \email{andrew@bakan.kiev.ua}           
           \and
          H{\aa}kan Hedenmalm \at
              KTH Royal Institute of Technology, SE–10044 Stockholm,
Sweden\\  \email{haakanh@math.kth.se}
 \and
         Alfonso Montes-Rodr\'{\i}guez \at
             University of Sevilla, 4180 Sevilla,
Spain\\  \email{amontes@us.es}
 \and
         Danylo Radchenko \at
              ETH Z\"{u}rich,  Mathematics Department, 8092 Z\"{u}rich,
Switzerland\\  \email{danradchenko@gmail.com}
 \and
         Maryna Viazovska \at
              \'{E}cole Polytechnique F\'{e}d\'{e}rale de Lausanne,
1015 Lausanne, Switzerland\\  \email{viazovska@gmail.com}
}


\maketitle

\vspace{-1,95cm}
\begin{abstract}\hspace{-0,17cm}{\bf{.}}
In this article we explain the essence of the interrelation described in [PNAS \textbf{118}, 15 (2021)]
 on how to write explicit interpolation formula for solutions of the Klein-Gordon equation by using the recent Fourier pair  interpolation formula of  Viazovska and Radchenko from [Publ Math-Paris \textbf{129}, 1 (2019)].
 By the words of  Fourier pair  interpolation, we mean the interpolation of a pair of functions of one variable, related to each other by the Fourier transform.

   Hedenmalm and Montes-Rodr\'{\i}guez
  established
in 2011 the weak-star completeness  in $L^{\infty}\! (\Bb{R})$ of the sequence
$1$, $\exp (\imag \pi n x)$, $\exp (\imag \pi n/ x)$, $n \in \Bb{Z} \setminus  \{0\}$,
 which is referred to as the {\emph{hyperbolic trigonometric system}}.
We construct explicitly  the sequence in  $L^1  (\Bb{R} )$ which is  biorthogonal
to this system and show that it is   complete in  $L^1  (\Bb{R} )$.
 The construction involves integrals with the kernel studied by Chib\-ri\-ko\-va
 in  1956,  similar to
 what was used
  by Viazovska and Radchen\-ko
 for the Fourier pair
 interpolation on the real line. We associate with each
 $f \in L^1 (\Bb{R}, (1+x^2)^{-1}{\rm{d}} x)$
 its {\emph{hyperbolic Fourier series}} \vspace{0,1cm}
 \begin{center}
 $\eurm{h}_{0}(f) + \sum_{n \in \Bb{Z}\setminus \{0\}}
 \left(\eurm{h}_{n}(f)\,{\rm{e}}^{{\fo{\imag \pi n x}}} +
  \eurm{m}_{n}(f)\,{\rm{e}}^{{\fo{-\imag \pi n/ x}}} \right)$
 \end{center}

 \vspace{0,2cm} \noindent
and prove that it converges to $f$ in the space of tempered
distributions on the real line.
The integral transform of   $\varphi \in L^{1} (\Bb{R}) $ given by
\begin{align*}
     &U_{\varphi} (x, y):=
     \begin{textstyle}
     \int_{\Bb{R}}
     \end{textstyle}\,
    \varphi (t)  \exp \left(\imag xt + {\imag  y}/{t}\right) {\rm{d}} t \,,
    \quad (x, y) \in \Bb{R}^{2}\,,
\end{align*}
\noindent
 is   continuous and  bounded on $\Bb{R}^{2}$, vanishes
at infinity, $u\!=\!U_\varphi$  solves the   Klein-Gor\-don equation  $u_{xy}\!+\!u\!=\!0$,
 and, by the theorem of Hedenmalm and Montes-Rodr\'{\i}\-guez,
it is determined uniquely by its values at
$\{(\pi n, 0), (0, \pi n)\}_{n \in \Bb{Z}}\!\subset\!\Bb{R}^{2}$.
Applied to the above mentioned biorthogonal system, this transform supplies interpolating
functions $\{\eusm{R}_{n}\}_{n \geqslant 0}$ for the Klein-Gordon
equation   on the characteristics with
  $\eusm{R}_{n}(\pi m , 0)\! = \!\delta_{n, m}$,
   $\eusm{R}_{n}(0 , \pi m)\! = \!0$, $m \!\in\! \Bb{Z}$,
   where $\delta_{n, m}$ is the Kronecker delta symbol.
  Under additional  decay conditions on $U_{\varphi}$,
  these interpolating functions allow us to  restore  $U_{\varphi}$
  from its values at  $\{(\pi n, 0), (0, \pi n)\}_{n \in \Bb{Z}}$.
   The restriction of any smooth solution of the Klein-Gordon
   equation to a bounded rectangle can be represented  in the form of $U_{\varphi}$, but this matter  will be pursued elsewhere.


\end{abstract}
\newpage

\setcounter{tocdepth}{3}
 \tableofcontents


\newpage

\section[\hspace{-0,30cm}. \hspace{0,11cm}Introduction]{\hspace{-0,095cm}{\bf{.}}  Introduction}\label{int}


\hspace{-0,2cm}
 An important unsolved issue in mathematics is
 the problem of interpolating oscillatory processes.
  Uniqueness holds a special place in interpolation.
  For example, when gas fluctuates in a gas storage, it is
  necessary to select the position of the pressure gauges
  in order to fully know about the state of the gas at any
  point in the storage. Unfortunately, nothing but a few
  numerical methods, based on considering solutions with
  extreme value of  entropy, are known (see, e.g., \cite{lio}, \cite{ger}).
  There appear to have been no other approaches to
  this topic  until 2011, when the second and the third authors
  \cite{hed}  considered uniformly bounded
  solutions of the Klein-Gordon equation
$ U_{xy} + U = 0$ in the plane of the form
\begin{align}\label{f0int}
    & U_{\varphi} (x, y) = \int\nolimits_{\Bb{R}}\
    {\rm{e}}^{{\fo{{\rm{i}}  x t + {\rm{i}}  y / t}}} \varphi (t) {\rm{d}} t  \ ,
    \quad  \varphi \in L^{1}(\Bb{R}) \,,
\end{align}

\vspace{-0,1cm}
\noindent On a compact subset, such solutions $U_{\varphi}$  can approximate  any continuous solution
 of this equation.
They pioneered  the study of the {\it{discretized Goursat problem}}
which   instead of prescribing the values of a solution on the
two intersecting characteristics $x=0$ and $y=0$  assumes that
these values are known only along the discrete subset of the characteristics
consisting of  equidistant points $\{(\pi n, 0)\}_{n \in \Bb{Z}}$,
 $\{(0, \pi n)\}_{n \in \Bb{Z}}$. It was established
 in \cite{hed} that the values
 $U_{\varphi} (\pi n, 0)$, $n \in \Bb{Z}$, and
 $U_{\varphi} (0, \pi n)$, $n \in \Bb{Z}\setminus \{0\} $,
 determine the function  $U_{\varphi}$ uniquely. They obtained the following result, which we develop further
  in Section~\ref{refa}.
  \newcommand{\reda}{\,\cite[p.\! 1517, Theorem 3.1]{hed}\,}
 \begin{thmx}[\reda]\hspace{-0,15cm}{\bf{.}}\label{intth0}
 Let $ \varphi  \in L^{1} (\Bb{R})$ have
\begin{align*}
    &  \int\limits_{\Bb{R}}  {\rm{e}}^{{\fo{{\rm{i}} \pi n x}}}
    \varphi (x) {\rm{d}} x  =\int\limits_{\Bb{R}}  {\rm{e}}^{{\fo{{-\rm{i}}
     \pi n / x}}} \varphi (x) {\rm{d}} x=0 \, , \quad  n\in \Bb{Z} \, .
\end{align*}

\vspace{-0,3cm}
\noindent Then $ \varphi  = 0$.
\end{thmx}

\vspace{-0,1cm}
In the present paper  we explicitly construct a  system of functions
$\eurm{H}_{0}$, $\eurm{H}_{n}$, $\eurm{M}_{n} \subset
L^1 (\Bb{R}) \cap C^{\infty} (\Bb{R})$, $n\in \Bb{Z}\setminus \{0\}$,
 which is biorthogonal to the {\emph{hyperbolic
trigonometric system}} $1$, $\exp (\imag \pi n x)$,
$\exp (\imag \pi n/ x) \subset L^{\infty} (\Bb{R})$,
$n \in \Bb{Z}\setminus \{0\}$, such that
\begin{align}\label{f02int}
    & \hspace{-0,15cm} \begin{array}{lll}
       \begin{displaystyle}
       \int\limits_{\Bb{R}} {\rm{e}}^{{\fo{-\imag  \pi n x}}}
       \eurm{H}_{p} (x) {\rm{d}} x = \delta_{n, p} \, ,
       \end{displaystyle}
         &
      \begin{displaystyle}
       \int\limits_{\Bb{R}} {\rm{e}}^{{\fo{ {\imag  \pi m}/{x} }}}
       \eurm{H}_{p} (x) {\rm{d}} x = 0 \,,
      \end{displaystyle}  & \  \left|\eurm{H}_{q} (x)
      \right|\leqslant \dfrac{\pi^{6}  |q\,|^{2} }{1+x^{2}}\ , \
      \\[0,5cm]
            \begin{displaystyle}
       \int\limits_{\Bb{R}} {\rm{e}}^{{\fo{\imag  \pi n x}}} \eurm{M}_{q} (x) {\rm{d}} x = 0 \, ,
       \end{displaystyle}
         &
      \begin{displaystyle}
       \int\limits_{\Bb{R}} {\rm{e}}^{{\fo{ {\imag  \pi m}/{x} }}}
       \eurm{M}_{q} (x) {\rm{d}} x = \delta_{m,  q} \,,
      \end{displaystyle}  &  \  \left|\eurm{M}_{q} (x)
      \right|\leqslant \dfrac{\pi^{6} |q\,|^{2} }{1+x^{2}}\ ,
             \end{array}\hspace{-0,2cm}
\end{align}

\vspace{-0,2cm}
\noindent $\left|\eurm{H}_{0} (x) \right|\!\leqslant\!3/ (1+x^{2})$
for all $x\!\in\! \Bb{R} $,  $n, p \!\in \!\Bb{Z}$ and $m, q \!\in\! \Bb{Z}\setminus \{0\}$,
where $ \delta_{n,  m}=1$, if $n=m$, and
$ \delta_{n,  m}=0$, if $n\neq m$, for arbitrary $n, m \in \Bb{Z}$.
We prove that for each $f \in L^1 (\Bb{R}, (1+x^2)^{-1}{\rm{d}} x)$
 its   {\emph{hyperbolic  Fourier series}}
 \begin{align}\label{f03int}
    & \eurm{h}_{0}(f) + \sum_{n \in \Bb{Z}\setminus \{0\}}
    \left(\eurm{h}_{n}(f)\,{\rm{e}}^{{\fo{\imag \pi n x}}} +
  \eurm{m}_{n}(f)\,{\rm{e}}^{{\fo{-\imag \pi n/ x}}} \right)  ,
  \\[0,2cm]  & \nonumber
\eurm{h}_{n}(f):=  \int\limits_{\Bb{R}}  f (t) \eurm{H}_{-n} (t) {\rm{d}} t
 \ , \ n\in \Bb{Z} \ ; \quad \eurm{m}_{n}(f) :=
  \int\limits_{\Bb{R}}  f (t) \eurm{M}_{-n} (t) {\rm{d}} t \ , \ n\in \Bb{Z}\setminus \{0\} \ ,
\end{align}

\noindent
converges to $f$ in the space $\eurm{S}^{\,\prime} (\Bb{R})$ of
tempered distributions on the real line.
We also show in Remark~\ref{rem1} that in  case  $f (x)$
or $f (-1/x)$ is $2$-periodic and integrable on the period, the
series  \eqref{f03int} coincides with its usual Fourier series
of the  argument $x$ or $-1/x$, respectively. Using \eqref{f09int} below
 (see also Theorem~\ref{intforth2}),  we explicitly restore $U_{\varphi}$
  in \eqref{f0int} from its values at the points
  $\{(\pi n, 0), (0, \pi n)\}_{n \in \Bb{Z}}$ under
  the condition \eqref{f08int}. The proof
  of the assertion that any
 smooth solution of the Klein-Gordon equation
  on a bounded rectangle can be represented in the form \eqref{f0int} will be supplied
  elsewhere. Another aspect of the theory which will be developed elsewhere is the series expansion of $U_{\varphi}$ which is obtained by insertion of the series \eqref{f03int}  for the function $\varphi$ in the integral \eqref{f0int}.

\vspace{0.25cm}

\subsection[\hspace{-0,31cm}. \hspace{0,11cm}Key proofs.]{\hspace{-0,11cm}{\bf{.}} Key proofs.}\label{baspro}\!\!
Given the properties \eqref{f02int}, it is easy to deduce from
Theorem~\ref{intth0} that   for every $f \!\in \!L^1 (\Bb{R}, (1\!+\!x^2)^{-1}{\rm{d}} x)$  the series  \eqref{f03int} converges in  $\eurm{S}^{\,\prime} (\Bb{R})$ and  that the biorthogonal system $\{\eurm{H}_{0}\}\cup\{\eurm{H}_{n}, \eurm{M}_{n}\}_{n\in \Bb{Z}\setminus \{0\}} $ is complete in $L^1 (\Bb{R})$ as well.
We indicate briefly the necessary arguments. The tempered test functions form the so-called Schwartz class
$\eurm{S} (\Bb{R})$ on the real axis (see \cite[p.\! 74]{vlad}), so that for all positive integers
$n$, $ k $ and arbitrary $\varphi \in \eurm{S} (\Bb{R})$ integration by part gives\vspace{-0,02cm}
\begin{align*}
    &  \int\nolimits_{\Bb{R}} \varphi (x) {\rm{e}}^{{\fo{{-\rm{i}}\pi  n x }}} {\rm{d}} x = \dfrac{{\rm{i}}}{\pi  n }
     \int\nolimits_{\Bb{R}}  \varphi (x) {\rm{d}} {\rm{e}}^{{\fo{-{\rm{i}}\pi  n x }}}  = \dfrac{1}{(i\,\pi\,  n)^{k}}  \int\nolimits_{\Bb{R}} \varphi^{(k)} (x) {\rm{e}}^{{\fo{{-\rm{i}}\pi  n x }}} {\rm{d}} x \, , \\  &
  \int\nolimits_{\Bb{R}} \varphi (x) {\rm{e}}^{{\fo{  \dfrac{{\rm{i}} \pi n }{x} }}}  {\rm{d}} x \!= \!
  \dfrac{{\rm{i}}}{\pi  n }
     \int\nolimits_{\Bb{R}} x^{2} \varphi (x) {\rm{d}} {\rm{e}}^{{\fo{  \dfrac{{\rm{i}} \pi n }{x} }}}  \!  = \! \dfrac{1}{(\imag \,\pi\,  n)^{k}}   \int\nolimits_{\Bb{R}} \left(\left(x^{2}\dfrac{{\rm{d}}}{{\rm{d}} x}\! +\!2x \right)^{k}\!\!\varphi (x)\!\right) {\rm{e}}^{{\fo{  \dfrac{{\rm{i}} \pi n }{x} }}}  {\rm{d}} x\, .
 \end{align*}

\vspace{-0,15cm}
\noindent As the integrands here are in $L^1 (\Bb{R})$,  the coefficients
 \begin{align}\label{f05aint}
    &   \hspace{-0,15cm} \eurm{h}_{n}^{\star}(\varphi)\!:=\!\int\limits_{\Bb{R}}\! \varphi (x)\, {\rm{ e}}^{{\fo{{-\rm{i}}\pi  n x }}} {\rm{d}} x
      \, ,  \    n \!\in\! \Bb{Z}\,;  \ \
      \eurm{m}_{n}^{\star}(\varphi)\!:=\! \int\limits_{\Bb{R}}\! \varphi (x) \,{\rm{e}}^{{\fo{  \dfrac{{\rm{i}} \pi n }{x} }}}  {\rm{d}} x \, , \    n\! \in\! \Bb{Z}\setminus \{0\},\hspace{-0,1cm}
\end{align}

\vspace{-0,1cm}
\noindent  satisfy, for  any positive integer $k$, \vspace{-0,05cm}
\begin{align}\label{f05bint}
    & \eurm{h}_{n}^{\star}(\varphi)\!=\! {\rm{O}} \left(n^{-k}\right)  \ , \quad
     \eurm{m}_{n}^{\star}(\varphi) \!  =\! {\rm{O}} \left(n^{-k}\right)  \, , \quad n \to \infty\,.
\end{align}

\vspace{-0,05cm}
\noindent Hence,   for all $\{a_{0}\} \cup
 \{a_{n}, b_{n}\}_{n\in \Bb{Z}\setminus \{0\}} \!\subset\! \Bb{C}$ and any  positive integer $N$  we have
\begin{align}\label{f04int}
    &  \hspace{-0,2cm}
    a_{n}, b_{n}\!=\! {\rm{O}} (n^{N}), \ n\!\to\! \infty   \ \Rightarrow \
     a_{0} + \sum\nolimits_{n \in \Bb{Z}\setminus \{0\}} \!\! \left(a_{n} {\rm{e}}^{{\fo{\imag \pi n x}}}\! +\!
 b_{n} {\rm{e}}^{{\fo{-\imag \pi n/ x}}} \right)\!\in\! \eurm{S}^{\,\prime} (\Bb{R}), \hspace{-0,1cm}
\end{align}

\noindent as the series converges in the space $\eurm{S}^{\,\prime} (\Bb{R})$ (see
\cite[p.\! 77]{vlad}){\hyperlink{r012}{${}^{\ref*{case012}}$}}\hypertarget{br012}{}.  In view  of \eqref{f02int}, we see that
 for each $\varphi \in \eurm{S} (\Bb{R})$
 the function
\begin{align*}
    &  \rho_{\varphi} (x)\! :=\!\varphi (x)\! -\! \eurm{h}_{0}^{\star}(\varphi)\eurm{H}_{0}(x) \!-\!
    \sum\nolimits_{n\in \Bb{Z}\setminus \{0\} }\!
    \Big( \eurm{h}_{n}^{\star}(\varphi) \eurm{H}_{n}(x)\! + \!\eurm{m}_{n}^{\star}(\varphi) \eurm{M}_{n}(x)\Big) \, ,    \ \ x\!\in\! \Bb{R} \,,
\end{align*}

\noindent is in   $L^{1} (\Bb{R})$ and satisfies
\begin{align*}
    &  \int\nolimits_{\Bb{R}}  {\rm{e}}^{{\fo{{\rm{i}} \pi n x}}} \rho_{\varphi} (x) {\rm{d}} x  =\int\nolimits_{\Bb{R}}  {\rm{e}}^{{\fo{-{\rm{i}} \pi n / x}}} \rho_{\varphi} (x) {\rm{d}} x=0 \ , \quad  n\in \Bb{Z} \,.
\end{align*}

\noindent From  Theorem~\ref{intth0},  it follows  that $\rho_{\varphi} =0$ and consequently,
for any $\varphi \in \eurm{S} (\Bb{R})$,
\begin{align}\label{f06int}\hspace{-0,2cm}
    &  \varphi (x)\! =\! \eurm{h}_{0}^{\star}(\varphi)\eurm{H}_{0}(x)\! +\!
    \sum\nolimits_{n\in \Bb{Z}\setminus \{0\} }\!
    \Big( \eurm{h}_{n}^{\star}(\varphi) \eurm{H}_{n}(x)\! +\! \eurm{m}_{n}^{\star}(\varphi) \eurm{M}_{n}(x)\Big) \, ,
    \ \ x\!\in \!\Bb{R} \,,
\end{align}

\noindent where the series converges in the weighted uniform norm
$\|\cdot\|_{C_{2}(\Bb{R})}$, in view of \eqref{f02int} and \eqref{f05bint}.
Here, $\|f \|_{C_{2}(\Bb{R})}:=\sup_{x\in \Bb{R}} (1\!+\!x^{2}) |f (x)|$, $f \in C (\Bb{R})$.
 Since  $\eurm{S} (\Bb{R})$ is  dense in
$L^{1} (\Bb{R})$
we deduce the completeness of the system $\{\eurm{H}_{0}\}\cup\{\eurm{H}_{n}, \eurm{M}_{n}\}_{n\in \Bb{Z}\setminus \{0\}} $ in $L^{1} (\Bb{R})$, as claimed.  As for the convergence of the  series \eqref{f03int}, we take an arbitrary
$f \in L^1 (\Bb{R}, (1+x^2)^{-1}{\rm{d}} x)$ and  for  $N \geqslant 1$ we write
\begin{align*}
     & \mathcal{F}_{N} [f] (x) := \eurm{h}_{0}(f) + \sum_{
     n \in \Bb{Z}\setminus \{0\}, \, |n| \leqslant N}\left(\eurm{h}_{n}(f)\,{\rm{e}}^{{\fo{\imag \pi n x}}} +
  \eurm{m}_{n}(f)\,{\rm{e}}^{{\fo{-\imag \pi n/ x}}} \right)\, ,
    \ \ x\!\in\! \Bb{R} \,.
\end{align*}

\noindent It follows from
\eqref{f02int}  and \eqref{f04int}  that, as $N \to +\infty$,   $ \mathcal{F}_{N}[f]$ converges in the space $\eurm{S}^{\,\prime} (\Bb{R})$ to an element  $\mathcal{F} [f] \in \eurm{S}^{\,\prime} (\Bb{R})$.
For any test function $\varphi \in \eurm{S} (\Bb{R})$ we have\vspace{-0,1cm}
\begin{align*}
     &  \int\limits_{\Bb{R}} \mathcal{F}_{N} [f] (x) \varphi (x) {\rm{d}}x =
      \eurm{h}_{0}(f)\eurm{h}_{0}^{\star}(\varphi) +\!\!\! \sum_{
     n \in \Bb{Z}\setminus \{0\}, \, |n| \leqslant N}\left(\eurm{h}_{n}(f)\,\eurm{h}_{n}^{\star}(\varphi) +
  \eurm{m}_{n}(f)\,\eurm{m}_{n}^{\star}(\varphi) \right) \\  &
  = \int\limits_{\Bb{R}}f (x)
  \bigg[
  \eurm{h}_{0}^{\star}(\varphi)\eurm{H}_{0}(x) +
    \sum_{
     n \in \Bb{Z}\setminus \{0\}, \, |n| \leqslant N}
    \Big( \eurm{h}_{n}^{\star}(\varphi) \eurm{H}_{n}(x) + \eurm{m}_{n}^{\star}(\varphi) \eurm{M}_{n}(x)\Big)
  \bigg]
  {\rm{d}} x \,.
\end{align*}

\vspace{-0,1cm}
\noindent From \eqref{f06int} combined with \eqref{f02int}, \eqref{f05bint} and
the Lebesgue dominated convergence theorem \cite[p.\! 161]{nat},    we obtain,  by passing to the limit as $N \to +\infty$,\vspace{0,1cm}
\begin{center}
   $\mathcal{F} [f] \big( \varphi\big) =
     \int_{\Bb{R}} f (x) \varphi (x) {\rm{d}}x \,, \quad \varphi \in \eurm{S} (\Bb{R})$.
\end{center}

\vspace{0,1cm}
\noindent This tells us that $\mathcal{F} [f] = f$ holds in the space $\eurm{S}^{\,\prime} (\Bb{R})$.
In conclusion,  the hyperbolic  Fourier series \eqref{f03int} converges to $f$ in the space  of tempered distributions on the real line.

\newpage
\subsection[\hspace{-0,31cm}. \hspace{0,11cm}Interpolation formula.]{\hspace{-0,11cm}{\bf{.}}
Interpolation formula.}\label{basint}\!\!
We proceed with the interpolation formula for the Klein-Gordon equation.
Let  $\varphi \in L^{1} (\Bb{R})$ be arbitrary  and consider the solution  $U_{\varphi}$ of the Klein-Gordon equation given by \eqref{f0int}. Then the coefficients  of $\varphi$ in \eqref{f05aint} are called the {\emph{conjugate hyperbolic Fourier coefficients}} and they equal the values of $U_{\varphi}$ at the points $\{(\pi n, 0), (0, \pi n)\}_{n \in \Bb{Z}}$,
\begin{align}\label{f07int}
     &\eurm{h}_{n}^{\star}(\varphi)=U_{\varphi}(-\pi  n, 0 )
      \, ,  \    n \!\in\! \Bb{Z}\,;  \ \
      \eurm{m}_{n}^{\star}(\varphi)= U_{\varphi} (0, \pi n)  \, , \    n\! \in\! \Bb{Z}\setminus \{0\}\,.
\end{align}

\noindent   By  the estimates \eqref{f02int} and
manipulations  similar to those employed in the proof of \eqref{f06int} we get that
  the {\emph{conjugate hyperbolic  Fourier series}} of $\varphi$ in the right-hand side of the equality \eqref{f06int} converges to $\varphi$ in the weighted uniform norm $\|\cdot\|_{C_{2}(\Bb{R})}$ on $\Bb{R}$  provided that\vspace{-0,1cm}
\begin{align}\label{f08int}
     &  \sum\nolimits_{n \in \Bb{Z}\setminus \{0\}}n^{2}\left(\,\left|U_{\varphi}(\pi  n, 0 )\right|+ \left|U_{\varphi} (0, \pi n)\right|\,\right) < \infty .
\end{align}

\vspace{-0,1cm}
\noindent If we multiply both sides of  the equality \eqref{f06int} by
$\exp (\imag xt + \imag y/t)$ and integrate with respect to $t$ over the real line,
we restore  $U_{\varphi}$  from its values at the points $\{(\pi n, 0), (0, \pi n)\}_{n \in \Bb{Z}}$ (see Theorem~\ref{intforth2}), by using the identities \eqref{f011int} below,
\begin{multline}\hspace{-0,25cm}
  U_{\varphi} (x, y)\! = \!  U_{\varphi} (0, 0)\, \eusm{R}_{\hspace{0.025cm}0} (x, y)
   \! + \! \sum\limits_{n \geqslant 1} \!  \Big[  U_{\varphi} (\pi n, 0)\,  \eusm{R}_{\hspace{0.025cm}n} (  x,  y)   + U_{\varphi} (0, -\pi n)\, \eusm{R}_{\hspace{0.025cm}n} ( -y,  -x)      \Big]+\\[0,2cm]  \! + \!
   \sum\limits_{n \geqslant 1} \!  \Big[\,  U_{\varphi} (-\pi n, 0)\,  \eusm{R}_{\hspace{0.025cm}n} (- x,  -y)   + U_{\varphi} (0, \pi n)\, \eusm{R}_{\hspace{0.025cm}n} ( y,  x) \Big] \,, \quad (x, y) \in  \Bb{R}^{2} \,. \label{f09int}\end{multline}

\noindent Here,  $\{\eusm{R}_{\hspace{0.025cm}n}\}_{n\geqslant 0}$ are the {\emph{interpolating functions for the Klein-Gordon equation}} given
by\vspace{-0,2cm}
\begin{align*}
    &   \eusm{R}_{n}(x, y)
\!:= \! \int\nolimits_{\Bb{R}}\  {\rm{e}}^{{\fo{\imag  x t+ \imag  y/t}}} \eurm{H}_{-n} (t) {\rm{d}} t\, , \quad
 n\geqslant 0 \, , \quad x,y \in \Bb{R}\,.
\end{align*}

\noindent The biorthogonal system has the  symmetry properties
\begin{align*}
     & \eurm{H}_{-n}(x)\! = \!\eurm{H}_{n}(-x) , \  n\!\geqslant\!0 \, , \  x\!\in\! \Bb{R} \,;  \ \
\eurm{M}_{n}(x) \!= \! \eurm{H}_{n} (-1/x)/x^{2} \, , \   n\!\geqslant\!  1 \, , \  x\!\in\! \Bb{R}_{\neq 0} \,,
\end{align*}

\noindent which for arbitrary $n \geqslant 1$ and $x, y \in \Bb{R}$  lead to the  identities,
\begin{align}\label{f011int}
    &  \hspace{-0,5cm}
  \begin{array}{ll}
    \begin{displaystyle}   \int\limits_{\Bb{R}}  {\rm{e}}^{{\fo{\imag  x t\! +\! \imag  y/t}}} \eurm{H}_{n} (t) {\rm{d}} t\!= \! \eusm{R}_{n}(-x, -y),      \end{displaystyle}  &
    \begin{displaystyle}\    \int\limits_{\Bb{R}}  {\rm{e}}^{{\fo{\imag  x t \!+\! \imag  y/t}}} \eurm{M}_{n} (t) {\rm{d}} t\!= \! \eusm{R}_{n}(y, x),  \end{displaystyle} \\
   \begin{displaystyle}   \int\limits_{\Bb{R}}  {\rm{e}}^{{\fo{\imag  x t \!+\! \imag  y/t}}} \eurm{H}_{-n} (t) {\rm{d}} t\!=\! \eusm{R}_{n}(x, y), \end{displaystyle} &
    \begin{displaystyle} \   \int\limits_{\Bb{R}}  {\rm{e}}^{{\fo{\imag  x t\! +\! \imag  y/t}}} \eurm{M}_{-n} (t) {\rm{d}} t\!= \! \eusm{R}_{n}(-y, -x).\end{displaystyle}
  \end{array}  \vspace{-0,2cm}
\end{align}


\subsection[\hspace{-0,31cm}. \hspace{0,11cm}Basic result.]{\hspace{-0,11cm}{\bf{.}} Basic  result.}\label{intmres}
Let $\Bb{D}:=\{z\in\Bb{C}\ |\  |z| <1\}$ and
\begin{align}
\label{f1int} \het (z) \!:=\!
F (1/2,1/2;1; z) \!=\!1\! +\! \dfrac{1}{\pi}
\sum\limits_{n = 1}^{\infty}
\dfrac{\Gamma(n\!+\!1/2)^{2}}{ n!^{2}}z^{n} \ , \quad
 z\! \in\! \Bb{D}\,,
\end{align}

\noindent be  the Gauss hypergeometric function which can be extended  to a nonvanishing holomorphic function on the set $\Bb{C}\!\setminus\![1,+\infty)$ (see \cite[p.\! 597, (1.19)(b)]{bh2}). In view of the Barnes expansion \cite[p.\! 173]{bar}(1908) (see also \cite[p.\! 299]{whi}, \eqref{f11inttrian}),
\begin{align*}
     &  \dfrac{\pi\het (1\!-\!z)}{\het (z)}= \left(\log \dfrac{16}{z}\right) \!-\!
\dfrac{
  \dfrac{2}{\pi} \sum\limits_{n=1}^{\infty} \dfrac{\Gamma (n+1/2)^{2}}{(n !)^{2}} \left[ \sum\limits_{k=1}^{n} \dfrac{1}{(2k -1)k}\right] z^{n}}{1 + \dfrac{1}{ \pi}\sum\limits_{n = 1}^{\infty}
\dfrac{\Gamma(n+1/2)^{2}}{(n !)^{2}} z^{n}} \ , \ z\! \in\! \Bb{D}\!\setminus\! (-1,0]\,,
\end{align*}

\noindent   for each positive integer $n$,  the function $\exp (n\pi\het (1-z)/\het (z))$ is meromorphic in $ \Bb{D}$. Moreover, there exists an algebraic polynomial $S^{{\tn{\triangle}}}_n$ of degree $n$  with real coefficients and leading coefficient equal to $16^{n}$ such that $S^{{\tn{\triangle}}}_n (0)=0$ and
\begin{align}\nonumber
       {\rm{e}}^{{\fo{ n \pi \ \dfrac{\he(1-z)}{\he(z)}}}} & =
\dfrac{16^{n}}{z^{n}} \  \exp\left(-n z \
\dfrac{\dfrac{1}{2} +   \dfrac{2}{\pi} \sum\limits_{m=1}^{\infty} \dfrac{\Gamma (m+3/2)^{2}}{(m+1) !^{2}} \left[ \sum\limits_{k=1}^{m+1} \dfrac{1}{(2k -1)k}\right] z^{m}}{ 1 + \dfrac{1}{\pi}\sum\limits_{m = 1}^{\infty}
\dfrac{\Gamma(m+1/2)^{2}}{(m !)^{2}} z^{m}}\right)
      \\[0,3cm]  &
    = S^{{\tn{\triangle}}}_{n} (1/z) + \Delta_{n}^{S}(z)\ , \ z \in \Bb{D}\setminus \{0\} \ , \
    \Delta_{n}^{S}\in {\rm{Hol}}(\Bb{D})\, ,
    \label{f1aint}\end{align}

\noindent where ${\rm{Hol}}(D)$ denotes the  space of all holomorphic functions in a domain $D\!\subset\!\Bb{C}$. In addition, all the Taylor coefficients of $\Delta_{n}^{S}$ are real (see \cite[p.\! 199]{rud}, \cite[p.\! 72]{con}). Our main result follows.
\begin{theorem}\hspace{-0,18cm}{\bf{.}}\label{intth1} The system  $\left\{\eurm{H}_{n}(x)\right\}_{n\in \Bb{Z}}\cup \left\{\eurm{M}_{n}(x)\right\}_{n\in \Bb{Z}\setminus \{0\} }\subset L^1 (\Bb{R})$ given by
\begin{align}
    &  \nonumber   \eurm{H}_{0} (x):=\dfrac{1}{2 \pi^{2 }}\int\limits_{ - \infty}^{ +  \infty}
 \frac{\left|\he  (1/2 + \imag  t)\right|^{2} {\rm{d}} t}{  \big(t^{2}+1/4\big)   \Big(
\he  (1/2 - \imag  t)^{2}+  x^{2}\he  (1/2 + \imag  t)^{2}\Big)}  \, , \\[0,4cm]    &\nonumber
\eurm{H}_{n} (x)  =  \dfrac{1}{4 \pi^{3} n}
  \int\limits_{ - \infty}^{ +  \infty}
\dfrac{  S^{{\tn{\triangle}}}_{n} \left(\dfrac{1}{1/2 + \imag  t}\right) {\rm{d}} t}{ \big(t^{2}+1/4\big)
\Big( \he  (1/2 - \imag  t) -\imag   x \he  (1/2 + \imag  t)\Big)^{2} }   \ , \\[0,4cm]    &
\label{f2intth1}
\eurm{H}_{-n}(x) := \eurm{H}_{n}(-x) \ , \quad  n\geqslant 1 \ , \  x\!\in\! \Bb{R} \,;  \\[0,4cm]  &
\label{f3intth1}
\eurm{M}_{n}(x) :=  \eurm{H}_{n} (-1/x)/x^{2} \ , \  n\in \Bb{Z}\setminus \{0\} \ , \  x\!\in\! \Bb{R}\setminus\{0\} \,,
\end{align}

\noindent
is  biorthogonal to the sequence $1$, $\exp ({\rm{i}}\pi n x)$, $\exp ({\rm{i}}\pi n/ x)$, $n \in \Bb{Z}\setminus \{0\}$, and satisfies the equalities \eqref{f02int}. Here, all the integrals are absolutely convergent for every $x\!\in\! \Bb{R}$ and  $\{S^{{\tn{\triangle}}}_{n}\}_{n\geqslant 1}$ are the algebraic polynomials
defined in \eqref{f1aint}. This system is complete in $L^1 (\Bb{R})$ and enjoys the estimates
\begin{align}\label{f1intth1}
     &
   \hspace{-0.7cm}   \left|\eurm{H}_{0} (x) \right|\leqslant \frac{3 }{1+x^{2}}\ , \
     \left|\eurm{H}_{n} (x) \right| + \left|\eurm{M}_{n} (x) \right| \leqslant \frac{\pi^{6} |n|^{2}}{1+x^{2}}\ , \quad  n\in \Bb{Z}\setminus \{0\} \ , \  x\!\in\! \Bb{R} \,.\hspace{-0.3cm}
\end{align}

\end{theorem}

\begin{corsecttwo}\hspace{-0,18cm}{\bf{.}}\label{intcor1} Suppose  $f \in L^1 (\Bb{R}, (1+x^2)^{-1}{\rm{d}} x)$.
 Then the series
 \begin{align*}
 \sum\limits_{n\in \Bb{Z}}   {\rm{e}}^{{\fo{{\rm{i}}\pi  n x }}}\int\limits_{\Bb{R}}  f (t) \eurm{H}_{-n} (t) {\rm{d}} t +\sum\limits_{n\in \Bb{Z}\setminus \{0\}}   {\rm{e}}^{{\fo{  -  \dfrac{{\rm{i}} \pi n }{x} }}} \int\limits_{\Bb{R}}  f (t) \eurm{M}_{-n} (t) {\rm{d}} t
 \end{align*}

 \noindent
represents a tempered distribution  on $\Bb{R}$ associated with the regular function $f$.
\end{corsecttwo}

\begin{rem}\hspace{-0,15cm}{\bf{.}}\label{remz}{\rm{
To prove Theorem~\ref{intth1} we first write down in
\eqref{f4int} and \eqref{f7int} the integrands of the integrals \eqref{f2intth1} in terms of the Schwarz triangle function.
Then using the change of variables we express in \eqref{f9int}  these integrals through
the elliptic modular function which is the inverse to the Schwarz triangle function.
This allows us to prove in Theorem~\ref{bslem3}  the fact of biorthogonality indicated in
Theorem~\ref{intth1}.

To establish the estimates \eqref{f1intth1}, we first find in \eqref{f3genfunbs} and \eqref{f4genfunbs} the generating functions of $\left\{\eurm{H}_{n}(x)\right\}_{n\in \Bb{N}}$ and $ \left\{\eurm{M}_{n}(x)\right\}_{n\in \Bb{N}  }$.
And then, having found in Section~\ref{contgen}  a constructive description of the analytic extension of these generating functions, we get in \eqref{f2fevagen} and \eqref{f3fevagen} the desired estimates \eqref{f1intth1}.
}}\end{rem}

The relationship{\hyperlink{r12}{${}^{\ref*{case12}}$}}\hypertarget{br12}{}
\begin{align}\label{f2intcor1}
    &  \he  \left(\dfrac{1}{2} +
\imag t\right)=\dfrac{\he \left(\dfrac{1}{2} + \dfrac{ t}{\sqrt{4 t^{2} +1}}\right)+\imag
\he\left(\dfrac{1}{2} - \dfrac{ t}{\sqrt{4 t^{2} +1}}\right)}{(1+ \imag )\,\sqrt[4]{ 4 t^{2} +1}} \ , \quad t \in \Bb{R}\,,
\end{align}

\noindent allows us{\hyperlink{r14}{${}^{\ref*{case14}}$}}\hypertarget{br14}{} to write the formulas \eqref{f2intth1} by using only the values of $\he$ on the interval
$(0,1)$,\vspace{-0,2cm}
\begin{align*}
    \eurm{H}_{0} (x) &=\dfrac{\imag }{2 \pi^{2 }}\int\limits_{\Bb{R}}
\dfrac{\dfrac{1\!-\!\imag \eurm{y}(t)}{1\!+\!\imag \eurm{y}(t)}
}{x^{2}- \left(\dfrac{1\!-\!\imag \eurm{y}(t)}{1\!+\!\imag \eurm{y}(t)}\right)^{2}}\,\dfrac{{\rm{d}} t}{t^{2}+1/4} \ , \quad
\eurm{y}(t)\! :=\!
\dfrac{ \he\left(\dfrac{1}{2}\! -\! \dfrac{ t}{\sqrt{4 t^{2}\! +\!1}}\right)}{\he \left(\dfrac{1}{2}\! + \!\dfrac{ t}{\sqrt{4 t^{2} \!+\!1}}\right)} \ , \
 \\[0,4cm]
      \eurm{H}_{n} (x)  & =  \dfrac{\imag }{ \pi^{3} n}\int\limits_{\Bb{R}} \dfrac{  S^{{\tn{\triangle}}}_{n} \left(\dfrac{1}{1/2 + \imag  t}\right) \dfrac{{\rm{d}} t}{\sqrt{t^{2}+1/4}}}{
\left((1\!+\!x)\he \left(\dfrac{1}{2} \!- \!\dfrac{ t}{\sqrt{4 t^{2}\! +\!1}}\right)\!+\!\imag  (1\!-\!x)
\he\left(\dfrac{1}{2} \!+\! \dfrac{ t}{\sqrt{4 t^{2} \!+\!1}}\right)\right)^{\!\!2} } \ .
\end{align*}

\noindent Since   $\he (0)=1$, $\he $ increases on $[0,1)$, $ \pi \he (1-x)-\ln (1/x)\in (0, 3\pi )$, $x\in (0,1/3)$
(see \cite[p.\! 602, (2.4)]{bh2}, \cite[p.\! 30, (A.6i)]{bh1}), $S^{{\tn{\triangle}}}_{n}(0)=0$, we obviously get an absolute convergence of these integrals for every $x\!\in\! \Bb{R}$ and $n\geqslant 1$. Here $\eurm{y}:\Bb{R}\mapsto (0, +\infty)$ decreases from $+\infty$ to $0$.

\vspace{-0.2cm}
\subsection[\hspace{-0,31cm}. \hspace{0,11cm}Connection with the Perron-Frobenius-Ruelle operator.]{\hspace{-0,11cm}{\bf{.}} Connection with  the Perron-Frobenius-Ruelle operator.}\label{frobop}

It can be easily seen that
   for arbitrary  $f \in L_1 \big([-1,1], {\rm{d}} x\big)$ in the space $L_1 \big([-1,1], {\rm{d}} x\big)$  there exists the limit
    \begin{align}\label{f1bsfir}
        &   {\mathbf{T}}_{1} [f] (x) \ := \ \lim\limits_{n \to +\infty}  \sum_{\substack{{\fo{k=-n}}\\[0.05cm] {\fo{k\neq 0}}}}^{{\fo{n}}} \dfrac{ f \left(\dfrac{1}{2k- x}\right)}{(2k-x)^{2}} \in L_1 \big([-1,1], {\rm{d}} x\big) \, ,
    \end{align}

\noindent and hence the mapping $f  \mapsto  {\mathbf{T}}_{1} [f]$  defines the  operator ${\mathbf{T}}_{1}: L_1 \big([-1,1], {\rm{d}} x\big) \mapsto L_1 \big([-1,1], {\rm{d}} x\big)$, where ${{\diff}}  x := {\rm{d }}  m (x)$ and  $m$ denotes the
Lebesgue measure on the real line. This operator is known as the {\it{Perron-Frobenius-Ruelle operator}} corresponding to the even Gauss map  $G_{2}: (-1,1]\to(-1,1]$ defined by the formula $G_{2} (x) = \{-1/x\}_{2}$ if $x\neq 0$ and $G_{2} (0) =0$. Here,  $\{x\}_{2}$ assigns to each real $x$ the  unique number
in the interval  $(-1,1]$ such that $x - \{x\}_{2}$ is an even integer {\rm{(}}see \cite[p.\! 81]{kes}, \cite[p.\! 57]{ios}{\rm{)}}.
The properties stated in Lemma~\ref{bslem1}{\rm{(4)}} and Lemma~\ref{bslem2}{\rm{(3,4)}} below tell us that
\begin{align}\label{f2bsfir}
    & \hspace{-0,25cm} \begin{array}{rcll}
 \begin{displaystyle}
 \left(I + {\mathbf{T}}_{1}\right) \left[\eurm{H}_{n}^{+}\right](x) \end{displaystyle} &    =   &   \begin{displaystyle} {\rm{e}}^{{\fo{{\rm{i}} \pi n x}}} \, ,
 \end{displaystyle}  &
 \begin{displaystyle}
 \quad   \left(I - {\mathbf{T}}_{1}\right) \left[\eurm{H}_{n}^{-}\right](x) = {\rm{e}}^{{\fo{{\rm{i}} \pi n x}}}  \ , \ \  n\!\in\! \Bb{Z}\setminus \{0\}\, ,
 \end{displaystyle}   \\[0,2cm]
  \begin{displaystyle}
 \left(I + {\mathbf{T}}_{1}\right) \left[2 \eurm{H}_{0}\right](x) \end{displaystyle}  &    =  &  1 \, ,
  &      \quad   x \in [-1,1] \,,
       \end{array}
\end{align}

\noindent where
\begin{align}\label{f3bsfir}
    & \left\{ \begin{array}{ll}
       \eurm{H}_{n}^{\pm}(-1/x)      =
                    \pm x^{2}\eurm{H}_{n}^{\pm}(x)  \ ,   &    \quad \eurm{H}_{0} (-1/x)\!=\! \eurm{H}_{0} (x) x^{2}\,,     \\[0,2cm]
   \eurm{H}_{n}^{\pm}     :=     2\eurm{H}_{n} \pm 2 \eurm{M}_{n} \ ,    &    \quad
    n\in \Bb{Z}\setminus \{0\} \, , \  x\in \Bb{R}\setminus\{0\} \,.
       \end{array}\right.
\end{align}

\noindent Hence another way to approach the functions  $\eurm{H}_{0}$, $\eurm{H}_{n}$, $\eurm{M}_{n}$, $n\in \Bb{Z}\setminus \{0\}$, is to solve the operator equations \eqref{f2bsfir} and  calculate  the functions  $\eurm{H}_{0}$, $\eurm{H}^{\pm}_{n}$, $n\in \Bb{Z}\setminus \{0\}$, on the interval $(-1,1)$. The symmetry in
\eqref{f3bsfir} gives their values on $\Bb{R}\setminus [-1,1]$ and hence extends them to the real line $\Bb{R}$.


\vspace{-0,2cm}
\subsection[\hspace{-0,31cm}. \hspace{0,11cm}Outline of the paper.]{\hspace{-0,11cm}{\bf{.}} Outline of the paper.}\label{outpap}
 The paper is organized as follows.  In Section~\ref{inttrian} we list a few facts about the elliptic modular function  $\lambda$ following an approach suggested by the first and the second authors in \cite{bh2}. The polynomials introduced and investigated by the fourth and fifth author in \cite{rad} have given reasons for us in Section~\ref{intprel} to introduce and explore a sequence of simpler polynomials, from which the polynomials in \cite{rad} can be obtained by symmetrization. We devote Section~\ref{bs} to various properties of the biorthogonal system.  Two partitions of the upper half-plane  are introduced and developed in Section~\ref{bsden}. The results of Section~\ref{bsden} are needed in Section~\ref{contgen} to make a constructive description of the analytic extension of the generating function for the biorthogonal system. This extension is similar to what was  proposed and implemented   by the fourth and the fifth authors in \cite{rad}. The results of Sections~\ref{bsden} and~\ref{contgen} are presented in the language of the even-integer continued fractions as suggested by the first author. Some estimates of  biorthogonal functions are derived in Section~\ref{evagen}. Applications to the Klein-Gordon equation of the obtained results are made in Section~\ref{mres}. In Section~\ref{refa},  Theorem~\ref{intth0} is extended with the purpose to  apply it for
partial solution of the issues raised in \cite{bhm}. Section~\ref{prgen} elaborates on the proofs of the main results. Further explanatory notes are supplied   in  Section~\ref{detpro}.




\subsection[\hspace{-0,31cm}. \hspace{0,11cm}Notation.]{\hspace{-0,11cm}{\bf{.}} Notation.}\label{not}
We begin this section by listing the frequently used notations.
\begin{align*}
    &
    \begin{array}{ll}\Bb{C},\, \Bb{R},\, \Bb{Z},\, \Bb{N}   & \  \mbox{Complex\! plane, real\! line,
     all\! integers, positive\! integers, respectively.}  \\
      \Bb{D},\ \overline{\Bb{D}},\ \Bb{H}  & \  \Bb{D}\!:=\!\{z\!\in\! \Bb{C} \mid |z |\!< \!1\}, \
        \overline{\Bb{D}}\!:=\!\{z\!\in\! \Bb{C} \mid |z |\!\leqslant \!1\}
        \, \mbox{and} \,
      \Bb{H}\!:=\!\{z\!\in\! \Bb{C} \mid \im\, z \!>\! 0\}.      \\
   \Bb{R}_{>a},\, \Bb{R}_{\geqslant a}\,,\,     & \
   \Bb{R}_{>a}:= (a, +\infty),   \quad  \ \, \Bb{R}_{\geqslant a}:= [a, +\infty),  \, \quad  a\in \Bb{R}\,.
   \\
   \Bb{R}_{<a},\, \Bb{R}_{\leqslant a}  &  \
      \Bb{R}_{<a}:= (-\infty, a),   \quad  \  \,\Bb{R}_{\leqslant a}:= (-\infty, a],   \quad  \, a\in \Bb{R}\,.
      \\
  \Bb{Z}_{\neq 0},\, \Bb{Z}_{\geqslant q}    &  \ \Bb{Z}_{\neq 0}\ :=  \Bb{Z}\setminus\{0\}, \ \ \ \ \
  \Bb{Z}_{\geqslant q}\ := \{\, n \in \Bb{Z} \, | \,  n \geqslant q \, \},  \ q\in \Bb{Z}\,.
      \\
 \Bb{Z}_{\neq 0}^{\hspace{0,02cm}\Bb{N}_{\hspace{-0,02cm}\eurm{f}}} ,\,  \Bb{Z}_{\leqslant q}    &  \   \Bb{Z}_{\neq 0}^{\hspace{0,02cm}\Bb{N}_{\hspace{-0,02cm}\eurm{f}}}\!:=\!\sqcup_{k\geqslant 1}\! \left(\Bb{Z}_{\neq 0}\right)^{k}\!, \
  \Bb{Z}_{\leqslant q}\ := \{\, n \in \Bb{Z} \, | \,  n \leqslant q \, \},  \ q\in \Bb{Z}\,.
       \\
  A_{\re>a},  &    \   A_{\re>a}\ := \,\{\, z\in A \ | \  \re\, z >a  \ \ \}   \, , \ \,   a \in \Bb{R}\, ,  \hspace{0,53cm}  A \subset \Bb{C}\,.
    \\
  A_{\re\geqslant a},  &  \ A_{\re\geqslant a}\ := \,\{\, z\in A \ | \  \re\, z \geqslant a   \ \ \}   \, , \ \,   a \in \Bb{R}\, ,  \hspace{0,53cm}  A \subset \Bb{C}\,.
     \\
   A_{\re<a}   &    \ A_{\re<a}\ := \,\{\, z\in A \ | \  \re\, z <a   \ \ \}   \, , \ \,  a \in \Bb{R}\, ,  \hspace{0,53cm}  A \subset \Bb{C}\,.
     \\
   A_{\re\leqslant a}   &    \ A_{\re\leqslant a}\ := \,\{\, z\in A \ | \  \re\, z \leqslant a   \ \ \}   \, , \ \,  a \in \Bb{R}\, , \hspace{0,53cm} A \subset \Bb{C}\,.
 \\
  A_{|\re|<a},  &  \ A_{|\re|<a}:=\{\, z\in A \ | \  |\re\, z| <a  \, \} \, , \ \,  a \in \Bb{R}_{>0}\, , \ \, A \subset \Bb{C}\,.
   \\
  A_{|\re|\leqslant a},  &  \ A_{|\re|\leqslant a}:=\{\, z\in A \ | \ |\re\, z |\leqslant a  \, \}  \, , \ \, a \in \Bb{R}_{>0}\, , \ \, A \subset \Bb{C}\,.
    \\
   \Bb{D}_{\im > 0}  & \  \Bb{D}_{\im > 0}:= \Bb{D} \cap \Bb{H}\,.   \\
    \overline{\Bb{D}}_{\im > 0}  & \  \overline{\Bb{D}}_{\im > 0}:= \overline{\Bb{D}} \cap \Bb{H}\,.
    \\
   {\rm{sign}} (x)  & \   {\rm{sign}} (x)\  \mbox{is equal to   $-1$ if $x < 0$, $0$ if $x=0$ and $1$ if $x>0$.}
   \\
    \chi_{A}(x)  & \   \chi_{A}(x)\  \mbox{is equal to   $1$ if $x \in A$ and  $0$ if $x \not\in A$ for $A \subset \Bb{C}$.}
     \\
    {\rm{gcd}}(a, b)  & \   {\rm{gcd}}(a, b)\in \Bb{N}\  \mbox{is the greatest common divisor of $a, b \in \Bb{N}$ and}
        \\   & \   \mbox{${\rm{gcd}}(a, b)\! :=\! {\rm{gcd}}(|a|, |b|)$, ${\rm{gcd}}(0, b) \!:=\!|b|$ for arbitrary  $a, b \in \Bb{Z}_{\neq 0}$.}
        \\ {\rm{int}} A  & \   \mbox{${\rm{int}} A\!:= \!\{a \in A \,| \, \exists\, \varepsilon \!>\! 0 : a\! + \!\varepsilon\Bb{D}\subset A\}$, \ $A\!\subset\!\Bb{C}$.}
    \end{array}
\end{align*}

For any $z_{1}, z_{2} \in \Bb{C}$, $z_{1}\neq z_{2}$ the straight line segment  from $z_{1}$ to
$z_{2}$ is denoted by $[z_{1}, z_{2}]$. We extend interval notation on the reals to line segment notation, so that, e.g., for any two points
   $z_{1}, z_{2} \in \Bb{C}$, we write $(z_{1}, z_{2}] := [z_{1}, z_{2}]\setminus \{z_{1}\}$ and $[z_{1}, z_{2}) := [z_{1}, z_{2}]\setminus \{z_{2}\}$.

Following the definitions of \cite[pp.\! 6, 40]{sar}, we denote by
 $\ln : \Bb{R}_{>0}\to \Bb{R}$ the real-valued logarithm defined
 on $\Bb{R}_{>0}$, and let $\Log  ( z )\!=
 \!\ln |z| \!+\! \imag  \Arg  (z)$
  be the principal branch of the logarithm defined for
  $z\!\in \!\Bb{C}\setminus (-\infty, 0]$ with $\Arg  (z)\! \in \!(-\pi, \pi)$.
 Furthermore,  for a simply connected domain $D\!\subset\!\Bb{C}$,
  a point $a\!\in\! D$, and a  function $f\!\in\! {\rm{Hol}}(D)$ which
  is zero-free  in $D$ with $f (a)\!>\!0$, we write  $\log f (z)$  for
  the holomorphic function in $D$ such that  $\exp (\log f (z))\! =\! f (z)$,
  $z\!\in \!D$, and $\log f (a)\! = \!\ln f (a)$ (see \cite[p.\! 94]{con}).
  Then ${\rm{Re}} \, \log f (z) \!= \!\ln |f (z)|$ and
  $\arg f (z) \!:=\! {\rm{Im}} \, \log f (z)$ for each  $z\!\in \!D$.

As for topology, we denote by ${\rm{clos}}(A)$ (or $\overline{A}$),
$\mathrm{int}(A)$, and
$\partial A$ the closure, interior, and boundary of a subset $A\subset\Bb{C}$,
respectively.

Let $C(\Bb{R})$ denote the linear space of all continuous complex-valued functions on $\Bb{R}$. For a non-negative Borel measure $\mu$ on $\Bb{R}$ we denote by $L^{1}(\Bb{R},  \mu)$   the standard normed space of integrable  with respect to $\mu$ complex-valued Borel functions.
For arbitrary non-negative Borel measures $\mu$, $\nu$ on $\Bb{R}$ and $g \in L^{1} (\Bb{R},  \mu)$, we write ${\rm{d}} \nu (x) = g (x) {\rm{d}} \mu (x) $ if $\nu (A) = \int_{A} g (x) {\rm{d}} \mu (x)$ for arbitrary Borel subset $A$ of $\Bb{R}$. If here $\mu$ is the
Lebesgue measure $m$ on the real line and $g (x)= 1/(1+x^{2})$, $x\in \Bb{R}$, then instead of $L^{1}(\Bb{R}, \mu)$ and $L^{1}(\Bb{R},  \nu)$ we write $L^{1}(\Bb{R}, {\rm{d}} x)$ (or $L^{1}(\Bb{R})$) and $L^1 (\Bb{R}, (1+x^2)^{-1}{\rm{d}} x)$, respectively.

For $1 \leqslant p<+\infty$, we denote by   ${\rm{H}}^p_{+}(\Bb{R})$ the class  of all nontangential limits on $\Bb{R}$  of  functions from the Hardy space  (see \cite[p.\! 112]{ko}, \cite[p.\! 57]{gar})\vspace{-0,1cm}
  \begin{align*}
    &     {H}^p (\Bb{H}):= \left\{ f\in {\rm{Hol}} (\Bb{H}) \ \Big| \ \|f\|_{H^p_{+}}:=\sup\limits_{y>0}
    \int_{\Bb{R}} |f(x+\mathrm{i} \hspace{0,015cm} y)|^{p}\, \mathrm{d} x < +\infty \right\} \, ,
\end{align*}

\vspace{-0,1cm}
\noindent   and  ${\rm{H}}^p_{-}(\Bb{R})\! :=\{f (-x)\  | \ f\!\in \! {\rm{H}}^p_{+}(\Bb{R})\}$.



\section[\hspace{-0,30cm}. \hspace{0,11cm}The Schwarz triangle function and  its inverse]{\hspace{-0,095cm}{\bf{.}} The Schwarz triangle function and  its inverse}\label{inttrian}

 We denote  $\gamma (a,\infty):= a+ \imag  (0, +\infty)$, $a\in  \Bb{R}$, and \vspace{-0,1cm}
\begin{align}\label{f1schw}
\gamma (a,b):=\left\{\, z\in \Bb{H}\ \ \big| \ \ \left|z - (a+b)/2\right| =
|b-a|/2 \,\right\} \, ,\hspace{0,22cm} \quad  a,b \in \Bb{R} \, , \ a\neq b \,.
\end{align}

\vspace{-0,1cm}
\noindent Everywhere below we also use the notation $\gamma (a, b)$ for the contour of integration which is extended from $a$ to $b$ along  the open semicircle $\gamma (a, b)$.

\subsection[\hspace{-0,31cm}. \hspace{0,11cm}The Gauss hypergeometric function.]{\hspace{-0,12cm}{\bf{.}} The Gauss hypergeometric function.}\label{gau}

\ Euler's integral representation
\begin{align*}
    &   \he  (z) = \dfrac{1}{\pi }  \int_{0}^{1} \dfrac{ {\diff}
      t}{\sqrt{t(1-t)(1-tz)  }}
    \ , \quad z \in \Bb{C}\setminus \Bb{R}_{\geqslant 1} \, , \
\end{align*}

\noindent of the Gauss hypergeometric function \eqref{f1int} gives its analytic extension from $\Bb{D}$ to
$\Bb{C}\setminus \Bb{R}_{\geqslant 1}$ satisfying
\begin{align}\label{f3inttrian}
    &   \he  (x) > 0 \ , \quad  x \in \Bb{R}_{<\, 1} \ , \
\end{align}

\noindent and together with the Pfaff formula (see \cite[p.\! 79]{and})
\begin{align*}
    &  \he  (z) = \dfrac{1}{\sqrt{1-z}} \ \he
\left(\dfrac{z}{z-1}\right)   \ , \quad
 z \in \Bb{C}\setminus \Bb{R}_{\geqslant 1}  \ , \
\end{align*}

\noindent it allows us to find out  the boundary values of $\he $ on the both sides of the cut along
$\Bb{R}_{\geqslant 1}$
(see \cite[p.\! 491, 19.7.3; p.\! 490, 19.5.1]{olv}),
\begin{align*}
    &  
\he  (x\!\pm\!  \imag   0)\!=\!  \dfrac{1}{\sqrt{x}} \, \he
\left(\dfrac{1}{x}\right)\! \pm \!\dfrac{ \imag  }{\sqrt{x}} \,
\he \left( 1\!- \!\dfrac{1}{x}\right)  ,
\quad x  \in \Bb{R}_{> 1} \,.
\end{align*}

\noindent Moreover, it also allows us to estimate $\he$ for $x\!+\!\imag y\!\in\!\Bb{C}\setminus \Bb{R}_{\geqslant 1}$ as follows
\begin{align*}
    &
 \left|\he (x \!+ \!\imag  y)\right| \!\leq \!  \dfrac{6 \!+ \!\ln \dfrac{x^{2}}{x\!-\!1}}{\sqrt{x}}\chi_{{\fo{\Bb{R}_{>1}}}}(x)\!+\chi_{[0,1)}(x) \!\ln \dfrac{21}{1\!-\!x}\!+ \!\dfrac{ 3 \!+\! \ln \left(1\! +\! |x|\right)}{\sqrt{1+|x|}} \chi_{{\fo{\Bb{R}_{<0}}}}(x)
\end{align*}

\noindent (see \cite[p.\! 35, (A.9e)]{bh1}), and conclude that  $\he $ belongs to the Hardy space $ H^{p}_{+} (\Bb{H})  $  for
arbitrary $2 < p < \infty$. Then the Schwarz integral formula (see
\cite[p.\! 128]{ko})
\begin{align}\label{f6inttrian}
    &   \he  (z)=
\dfrac{1}{\pi} \int\nolimits_{1}^{\infty} \he \left( 1-\dfrac{1}{t} \right)
 \dfrac{ {\diff}  t}{(t-z)\sqrt{t}}  \ \ , \quad z \in \Bb{C}\setminus \Bb{R}_{\geqslant 1} \,,
\end{align}

\noindent  can be applied
 to restore $\he$ on $\Bb{C}\setminus \Bb{R}_{\geqslant 1}$ from the values of $\im\,  \he $ on $\Bb{R}$ (see \cite[p.\! 604, (2.12)]{bh2}). The similar quadratic formula
\begin{align}\label{f7inttrian}
    &   \he  (z)^{2}=
\dfrac{2}{\pi} \int\nolimits_{1}^{\infty} \he \left( 1-\dfrac{1}{t} \right)\he \left( \dfrac{1}{t} \right)
 \dfrac{ {\diff}  t}{(t-z) t} \ \ , \quad z \in \Bb{C}\setminus \Bb{R}_{\geqslant 1}\,,
\end{align}

\noindent  also applies, because  $\he^{2} $ is in the Hardy space $ H^{p}_{+} (\Bb{H})  $  for all $1 < p < \infty$. For arbitrary  $z = x+\imag  y \in \Bb{C}\setminus \Bb{R}_{\geqslant 1}$
we use the relation
\begin{align*}
    &  \im \dfrac{z(1-z)}{1-tz}= y \dfrac{t x^{2}-2x+1}{(1-tx)^{2} + t^{2}y^{2}} \ , \quad t \in [0,1]  \, , \
\end{align*}

\noindent to derive from \eqref{f7inttrian} that
\begin{align*}
    &  \im \left( z(1\!-\!z) \he  (z)^{2} \right)\!=\!
\dfrac{2 y }{\pi} \int\limits_{0}^{1}
\dfrac{\he \left( 1\!-\!t \right)\he \left( t \right)\left(\dfrac{1}{1\!+ \!\sqrt{ 1\!-\!t}}\!-\!x\right)
\left(1\!-\!tx\! +\! \sqrt{ 1\!-\!t}\right){\diff}  t}{(1-tx)^{2} + t^{2}y^{2}}\ ,
\end{align*}

\noindent from which we conclude that
\begin{align}\label{f8inttrian}
    &     {\rm{sign}} \,\im \Big( z(1\!-\!z) \he  (z)^{2} \Big)=
{\rm{sign}}\,\im \left(z\right)
  \, , \quad z \in
\Bb{C}_{\re\,\leqslant 1/2} \,.
\end{align}

\noindent
It follows from \eqref{f6inttrian} that $\he$ has  property
\begin{align*}
    &  \hspace{-0,2cm} \Arg   \he  (z) \in
 \!\left(-\dfrac{\pi}{2},  \dfrac{\pi}{2}\right) \, , \ \
    {{\rm{Re}}} \, \he   (z) \!>\! 0 \, , \quad z\!\in\!\Bb{C}\setminus
    \Bb{R}_{\geqslant 1} \,,
\end{align*}

\noindent  which reinforces \eqref{f3inttrian}, and allows us  to form   the logarithm of it and to obtain the integral representation
\begin{align}\label{f10inttrian}
    &  \Log   \he   (z) = \frac{1}{\pi^{2 }} \!\int\nolimits_{0}^{1} \
\frac{\Log  \left(\dfrac{1}{1-tz}\right)
 }{t(1-t)\left(\vphantom{\dfrac{a}{a}} \he  (t)^{2} +
 \he  (1-t)^{2}\right)} \, {\diff} t \, , \quad z\!\in\!\Bb{C}\setminus
    \Bb{R}_{\geqslant 1} \ , \ \hspace{-0,15cm}
\end{align}

\noindent (see \cite[p.\! 595, (1.15)]{bh2}), where
\begin{align*}
  \dfrac{1}{\pi^{2} } \int\nolimits_{0}^{1} \   \dfrac{{\rm{d}} t}{t(1-t)\left(\vphantom{\dfrac{a}{a}} \he  (t)^{2} +
 \he  (1-t)^{2}\right)}=  \dfrac{1}{2}\,,
\end{align*}

\noindent because\vspace{-0,15cm}
\begin{align}\label{f3pinttheor1}
     & \dfrac{{{\diff}}}{{{\diff}} x} \dfrac{\he
    (1-x)}{\he  (x)} = -  \dfrac{1}{\pi x (1-x) \he  (x)^{2}} < 0 \ ,
     \quad  x \in (0,1) \ , \
\end{align}

\noindent
(see \cite[p.\! 606, (3.9)]{bh2}) implies
\begin{align*}
    &  \dfrac{d}{d x}\left(\dfrac{\pi}{2}- \arctan \dfrac{\he \left(1\!-\!x\right)}{\he \left(x\right)}\right)=
\dfrac{1}{\pi x (1-x)} \dfrac{1}{\he \left(x\right)^{2} + \he \left(1-x\right)^{2}} \ , \quad  x\in (0,1)\,.
\end{align*}

\noindent Here, in view of the Barnes expansion (see \cite[p.\! 173]{bar}, \cite[p.\! 299]{whi}),
\begin{align}\label{f11inttrian}
    &  \he (1-z) =  \dfrac{\he (z)}{\pi}\Log \dfrac{16}{z}- \dfrac{2}{\pi^{2}}\sum\limits_{n=1}^{\infty} \dfrac{\Gamma (n+1/2)^{2}}{(n !)^{2}}\left(\sum\limits_{k= 1}^{n} \dfrac{1}{ (2k -1) k}\right) z^{n} \, , \
\end{align}

\noindent  which is valid for $z \in \Bb{D}\setminus (-1,0]$, we have  (see \cite[p.\! 602, (2.5)]{bh2})
\begin{align}\label{f5pinttheor1}
    & \hspace{-0,2cm}   \he  (1-x)\! =\!\dfrac{1}{\pi}\,
\ln  \dfrac{16}{x}\ +\!
{\rm{O}} \Big( x \ln \dfrac{1}{x}\Big) \ , \ \  \he  (x) \!= \!1\! + {\rm{O}}  (x)  \ , \    x \to 0 , \, x\in(0,1) \, .\hspace{-0,1cm}
\end{align}

\vspace{0.2cm}
\subsection[\hspace{-0,31cm}. \hspace{0,11cm}The Schwarz triangle function.]{\hspace{-0,11cm}{\bf{.}} The Schwarz triangle function.} \ \ \
  In 1873, Schwarz \cite{sch} established the following fact
(see \cite[p.\! 97]{erd1}).
\begin{thmx}\hspace{-0,15cm}{\bf{.}}\label{intthA}
The  Schwarz triangle function
\begin{align}
\label{f1intthA}
       \ie (z) :=   \imag \, \frac{\het (1-z)}{\het (z)}\,,
\quad  z \in  (0,1)\cup \left(\Bb{C}\setminus\Bb{R}\right),
\end{align} \noindent
maps the set $(0,1)\cup \left(\Bb{C}\setminus\Bb{R}\right) $ one-to-one
onto the  ideal hyperbolic quadrilateral
\begin{align}\label{f2intthA}
\fet := \big\{\, z \in\Bb{H}\ \big| \,\  -1 < {\rm{Re}}\,  z < 1 \, , \,  \
|2 z -  1| > 1 , \ | 2 z +  1| > 1 \, \big\}  \, ,
\end{align}

\noindent which will be referred to as Schwarz quadrilateral (see Fig.~\ref{figure:fundamental_domain}).
\end{thmx}\vspace{-0,5cm}

\begin{figure}[htbp]
    \begin{tikzpicture}
    \definecolor{cv0}{rgb}{0.95,0.95,0.95}
    \begin{scope}[scale=1]
  \clip(-6,-0.5) rectangle (6,3.5);
    \draw (-3,0) circle (0.04) node[below] {$-1$};
    \draw (0,0) circle (0.04) node[below] {$0$};
    \draw (-1.5,0) circle (0.04) node[below] {$-1/2$};
    \draw (1.5,0) circle (0.04) node[below] {$1/2$};
    \draw (3,0) circle (0.04) node[below] {$1$};
        \draw (-5,0) -- (5,0);
    \draw (9,0) --  (9,3);
    \draw (-9,0)  --  (-9,3);
       \fill[color=cv0] (-3,0) arc (180:0:1.5) arc (180:0:1.5) -- (3,3) -- (-3,3);
    \draw (3,0)  --  ( 3,3);
    \draw (-3,0)  -- (-3,3);
    \draw[dotted] (0,0)  -- (0,2.4);
    \draw[dotted] (-1.5,0)  -- (-1.5,1.5);
    \draw[dotted] (1.5,0)  -- (1.5,1.5);
    \draw[dotted] (-3,1.5)  -- (3,1.5);
    \draw (-1.5,1.5) circle (0.04);
     \draw (1.5,1.5) circle (0.04);
      \draw (0,1.5) circle (0.04);
     \draw (0,1.7) node[right] {$i/2$};

    \draw (3,0) arc (0:180:1.5);
    \draw (-3,0) arc (180:0:1.5);
       \draw (0,2.4) node[above]{$\fet$};
    \end{scope}
    \end{tikzpicture}
    \caption{\hspace{-0,2cm}{\bf{.}} The  Schwarz quadrilateral  $\fet$.}
    \label{figure:fundamental_domain}\vspace{-0,3cm}
\end{figure}


\noindent
Theorem~\ref{intthA},  \eqref{f1int}, \eqref{f5pinttheor1}, \eqref{f3pinttheor1}  and the partial case
\begin{align*}
    \ie (x)\! = \!
    \imag \,   \frac{\het (1-x)}{\het (x)}\in {\rm{i }}\,\Bb{R}_{>0} \ ,
     \quad  x \in (0,1) \ ,
\end{align*}

\noindent of \eqref{f1intthA},
 imply
\begin{align}\label{f0pinttheor1}
   & \ie \big((0,1)\big) = \imag  \Bb{R}_{>0} \ , \quad  \ie \big(1/2\big) = \imag  \ ,
\end{align}

\noindent and that the function $\ie ( t)/\imag $ on the interval $(0, 1)$ decreases from $+\infty$ to $0$,
\begin{align}\label{f6pinttheor1}
    &  0 < \ie ( t_{2})/\imag  < \ie ( t_{1})/\imag  < +\infty \ , \quad  0 < t_{1}< t_{2} < 1 \,.
\end{align}

To make other known  properties of $\ie$  obvious, it is necessary to use the
integral representation (see \cite[p.\! 608, (3.13)]{bh2})
\begin{align*}
    &   \Log\dfrac{\ie (z)}{\imag } = \Log   \dfrac{ \he  (1\!-\! z)}{\he  (z)}\!=\!\dfrac{1}{
    \pi^{2 }}\! \int\limits_{0}^{1}\!\! \dfrac{
    \Log  \dfrac{1-tz}{1-t+tz}}{t (1-t)\left(\he  \left(t \right)^{2}
    + \he  \left(1-t \right)^{2}\right)} \,  {\diff}  t \ ,
\end{align*}

\noindent
which is immediate  from \eqref{f10inttrian} for every $z \in (0,1)\cup \left(\Bb{C}\setminus\Bb{R}\right)$.
This  representation has the following convenient properties (see Section~\ref{pinttrianlem1}).
\begin{lemsectthree}\hspace{-0,17cm}{\bf{.}}\label{inttrianlem1} Denote $\Lambda := (0,1)\cup \left(\Bb{C}\setminus\Bb{R}\right)$.
 The integrands in
\begin{align}\label{f1inttrianlem1}
    &  \ln\left|\ie (z)\right|= \ln   \dfrac{ \left|\he  (1\!-\! z)\right|}{\left|\he  (z)\right|}\!=\!\dfrac{1}{
    \pi^{2 }}\! \int\limits_{0}^{1}\!\! \dfrac{
    \ln  \dfrac{|1-tz|}{|1-t+tz|}}{t (1-t)\left(\he  \left(t \right)^{2}
    + \he  \left(1-t \right)^{2}\right)} \,  {\diff}  t \ ,
\end{align}

\noindent and in
\begin{align}\label{f2inttrianlem1}
    &  \Arg\dfrac{\ie (z)}{\imag } = \Arg   \dfrac{ \he  (1\!-\! z)}{\he  (z)}\!=\!\dfrac{1}{
    \pi^{2 }}\! \int\limits_{0}^{1}\!\! \dfrac{
    \Arg \dfrac{1-tz}{1-t+tz}}{t (1-t)\left(\he  \left(t \right)^{2}
    + \he  \left(1-t \right)^{2}\right)} \,  {\diff}  t
\end{align}

\noindent  are of constant sign for each $z \in \Lambda$.
More precisely, if $z \in \Lambda$ then
\begin{align*}
    &
  \dfrac{|1-tz|}{|1-t+tz|}
\left\{
  \begin{array}{ll}
  > 1  \ , & \hbox{if}\quad \re\, z < 1/2\, ,\\
 = 1 , & \hbox{if}\quad \re\, z = 1/2\, ,\\
  < 1 , & \hbox{if}\quad  \re\, z > 1/2\,,
  \end{array}
\right.\quad
\Arg \dfrac{1-tz}{1-t+tz}\left\{
  \begin{array}{ll}
   \in (-\pi, 0) , & \hbox{if}\quad \im\, z > 0\,,\\
  = 0 , & \hbox{if}\quad z \in (0,1)\,,\\
  \in (0, \pi) , & \hbox{if}\quad \im\, z < 0\,,
  \end{array}
\right.
\end{align*}

\noindent for arbitrary $t \in (0,1)$.
\end{lemsectthree}

The properties established in Lemma~\ref{inttrianlem1} allow us to see from \eqref{f2inttrianlem1} that
\begin{align}\label{f13inttrian}
    &   {\rm{sign}}\left( \re\, \ie (z)\right) =   {\rm{sign}} \left(\im\, z\right)  \ , \quad  z  \in (0,1)\cup \left(\Bb{C}\setminus\Bb{R}\right) \,,
\end{align}

\noindent  while \eqref{f1inttrianlem1} yields that $|\ie (z)|\lesseqqgtr 1$ holds if and only of  $\re\, z  \gtreqless  1/2$, respectively, and according to Theorem~\ref{intthA} we obtain (see \cite[p.\! 41, (A11.i)]{bh1})
\begin{align}\label{f5aint}
        &
  \begin{array}{l}
  {\rm{(a)}}\	  \ie\left(\Bb{C}_{\re>1/2}\setminus \Bb{R}_{\geqslant 1}\right)=
        \fet \cap \Bb{D}\,,  \\[0,1cm]
  {\rm{(b)}}\	  \ie\left(\Bb{C}_{\re< 1/2}\setminus \Bb{R}_{\leqslant 0}\right)=
        \fet \setminus \overline{\Bb{D}}
           \,,
  \end{array}
       \qquad
       {\rm{(c)}}\	    \ie\big((1/2) + \imag  \Bb{R}\big) = \Bb{H}\cap \partial \Bb{D}  \,.
    \end{align}

\subsection[\hspace{-0,31cm}. \hspace{0,11cm}The elliptic modular function.]{\hspace{-0,11cm}{\bf{.}} The elliptic modular function $\lambda$.}
The function
$\lambda\! :\! \fet\! \to\! (0,1)\cup \left(\Bb{C}\setminus\Bb{R}\right)$
which is the inverse to $ \ie$, i.e.,
\begin{align}\label{f2aint}\hspace{-0,2cm}
    &  {\rm{(a)}}\ \   \ie (\lambda (y)) = y \,,  \ \ y \in \fet \,  , \ \ {\rm{(b)}}\ \   \lambda \big(\ie (z)\big) = z  \ , \ \  z \in
      (0,1)\cup \left(\Bb{C}\setminus\Bb{R}\right),\hspace{-0,1cm}
\end{align}

\noindent is called the  {\emph{{\rm{(}}elliptic{\rm{)}} modular function}}  $\lambda$ (see  \cite[p.\! 99]{erd1} and \cite[p.\! 579]{olv}).

The  modular function $\lambda$ extends to a periodic nonvanishing
holomorphic function in $\Bb{H}$ with period $2$ and (see
\cite[p.\! 598, (1.29)]{bh2})
\begin{align}\label{f2int}
     &\hspace{-0,2cm} \lambda (z)\! = \! {\Theta_{2}(z)^{4}}\big/{\Theta_{3}(z)^{4}}\!\neq\! 0  \, , \
\lambda (-1/z)\!=\!  1\!- \!\lambda (z) \!= \! {\Theta_{4}(z)^{4}}\big/{\Theta_{3}(z)^{4}}\!\neq\! 0   \, , \   z \!\in\!  \Bb{H} \,,\hspace{-0,1cm}
\end{align}

\noindent where for $z \in  \Bb{H}$ and $u \in \Bb{D}$ (see
\cite[pp.\! 612--614, (6.1),(6.7),(6.8)]{bh2}),
\begin{align*}
          \Theta_{3}  (z) &\! := \!   \theta_{3}\left(
          {\rm{e}}^{\imag  \pi z}\right)\, ,  &
      \Theta_{2}  (z) &  \! := \! 2 {\rm{e}}^{\imag  \pi z /4}
 \theta_{2}\left({\rm{e}}^{\imag  \pi z}\right)\, ,
      &
     \Theta_{4}  (z) & \! := \!\theta_{4}\left({\rm{e}}^{\imag  \pi z}\right)\, ,
 \\
\theta_{3} (u)& \! := \!
1  \!+ \! 2\sum\limits_{n\geqslant 1} u^{ n^2  } \, ,  &
\theta_{2} (u) &  \!:=
\!1 \!+\! \sum\limits_{n\geqslant 1} u^{n^2 +n} \, ,  &
 \theta_{4} (u)  & \! :=  \!1  \!+ \! 2\sum\limits_{n\geqslant 1}
(-1)^{n}u^{n^2}\!.
    \end{align*}

\vspace{-0,1cm}
 \noindent Regarding these nonvanishing holomorphic  functions in $\Bb{H}$ and in $\Bb{D}$, correspondingly (see \cite[p.\! 598, (1.26)]{bh2}),
 the main  relationships between them can be written for arbitrary
 $z \in \Bb{H}$ as follows, by using the principal branch
  of the square root,
\begin{align}\label{f3bint}
    &  \hspace{-0.45cm} \begin{array}{llll}
	{\rm{(a)}}\	\Theta_2(-1/z) &= (z/\imag )^{1/2}\Theta_4 (z)\, ,
 \quad   &\quad {\rm{(b)}}\
	\Theta_3(-1/z) &= (z/\imag )^{1/2}\Theta_3(z)\, ,\ \\[0,1cm]
  {\rm{(c)}} \
	\Theta_4(-1/z) &= (z/\imag )^{1/2}\Theta_2(z) \, , &\quad
{\rm{(d)}} \
	\Theta_2(z+1) &= {\rm{e}}^{\imag \pi/4}\Theta_2(z) \, ,
\\[0,1cm]   {\rm{(e)}} \
	\Theta_3(z+1) &= \Theta_4(z) \, , &\quad {\rm{(g)}} \
	\Theta_4(z+1)&= \Theta_3(z) \, ,
\end{array}\hspace{-0.2cm}
\end{align}

\vspace{-0,1cm}\noindent(see \cite[p.\! 614, (6.8)]{bh2}, $\re \,( z/i) > 0$).
In addition,  they are called the {\emph{theta functions}} and   meet the Jacobi identity
\begin{gather}\label{f3cint}
   \Theta_3 (z)^4 \! =\! \Theta_2 (z)^4 \!+\! \Theta_4 (z)^4 , \   z\!\in \!\Bb{H} \, ; \ \   \theta_{3} (u)^{4}  \!= \!16 u\, \theta_{2} (u)^{4}\!+\! \theta_{3} (-u)^{4} , \    u \!\in\! \Bb{D}\,,
\end{gather}

\noindent (see \cite[p.\! 599,\! (1.29)]{bh2}),  which gives for  $u\! \in\! \Bb{D}$  that (see
\cite[p.\! 157,\! (2.1)]{cha})
\begin{align}\label{f3wcint}
    & \begin{array}{lrcl}
     {\rm{(a)}}   &
       \ \  1+ \sum\nolimits_{n=1}^{\infty}\ r_{4}(n)\,u^{n}   &   :=  &   \theta_{3} (u)^{4}  ; \\[0,2cm]
        {\rm{(b)}}   &
       \ \ 1 +  \sum\nolimits_{n=1}^{\infty} (-1)^{n}  r_{4}(n) u^{n}  &   =   &   \theta_{4}(u)^{4} =\theta_{3}(-u)^{4} ; \\[0,2cm] {\rm{(c)}}   &
     \ \ \sum\nolimits_{n=0}^{\infty}  r_{4}(2n+1) u^{2n}&   =   &8 \, \theta_{2} (u)^{4} .
      \end{array}
\end{align}

\vspace{0,1cm}
It can be seen from \eqref{f6pinttheor1}, \eqref{f2int} and \eqref{f19int}  that (see Section~\ref{pinttrianlem1})
\begin{align}\label{f0apinttheor1}
   &\hspace{-0,3cm}  {\rm{(a)}}\ \ \lambda \big( \imag \Bb{R}_{>0}\big)\! =\!(0,1) , \ \
   \lambda \big( \imag \big)\! =\!1/2  , \ \ {\rm{(b)}}\ \ {\rm{e}}^{{\fo{-  \pi y }}} \!< \! \lambda (\imag y)\! <\! 16\, {\rm{e}}^{{\fo{-  \pi y }}} , \ \
 0 \!<\! y \!<\! \infty     \, ,\hspace{-0,2cm}
\end{align}

\noindent and that the function $\lambda (\imag  t)$ on the interval $(0, +\!\infty)$ strictly decreases from $1$ to $0$,
\begin{align}\label{f6apinttheor1}
    & \lambda (\imag  y_{1})+ \lambda (\imag / y_{1})=1 \ , \quad   0 < \lambda (\imag  y_{2}) < \lambda (\imag  y_{1}) < 1 \ , \quad  0 < y_{1}< y_{2} < +\infty \,.
\end{align}

\noindent In addition, \eqref{f5aint} and \eqref{f13inttrian} can be written as follows (\,$\gamma (-1,1):=\Bb{H} \cap\partial\Bb{D}$\,)
\begin{align}\label{f5int}
        &\hspace{-0,2cm}
  \begin{array}{ll}
   {\rm{(a)}}\ \lambda\left(\fet \cap \Bb{D}\right)\!=\!
       \Bb{C}_{\re>1/2}\!\setminus\! \Bb{R}_{\geqslant 1} \,;  &
   \ {\rm{(b)}}\   \lambda\big(\gamma (-1,1)\big)\! =  \!(1/2)\! + \!i \Bb{R} \,;
        \\[0,1cm]
 {\rm{(c)}}\   \lambda\left(\fet \setminus \overline{\Bb{D}}\right)\!=\!\Bb{C}_{\re< 1/2}
 \!\setminus\! \Bb{R}_{\leqslant 0}
           \,; &   \ {\rm{(d)}}\
         {\rm{sign}} \left(\im\, \lambda(z)\right) \! =\!  {\rm{sign}}\left( \re\, z\right)\, , \   z\!  \in \!\fet \,.
  \end{array}
     \hspace{-0,1cm}
    \end{align}

\vspace{-0,05cm}
The known relationships (see \cite[p.\! 598, (1.25); p.\! 599, (1.32)]{bh2}){\hyperlink{r13}{${}^{\ref*{case13}}$}}\hypertarget{br13}{}
\begin{align}\label{f14inttrian}
    &  \hspace{-0,2cm}
    \begin{array}{ll}
  {\rm{(a)}}\	  \Theta_{2}\big(\ie (z)\big)\!= \!z^{1/4} \he  (z)^{1/2}\!  ,  &
 \qquad
     {\rm{(b)}}\	   \Theta_{3}\big(\ie (z)\big)\!= \!\he  (z)^{1/2}\, ,
     \\[0,1cm]
 {\rm{(c)}}\	  \Theta_{4}\big(\ie (z)\big)\!=\! (1\!-\!z)^{1/4}\he  (z)^{1/2}\!  , &
\phantom{{\rm{(c)}}a}
 \qquad z\in (0,1)\cup \left(\Bb{C}\setminus\Bb{R}\right),
    \end{array}\  \  \hspace{-0,1cm}
\end{align}

\noindent for the principal branches of the quadratic and of  the fourth roots together with \eqref{f2aint} and \eqref{f2int}, show that the Schwarz triangle  function $\ie$ is a key  that links the theta functions $\{\Theta_{k}\}_{k=2}^{4}$ with the Gauss hypergeometric function $\he$.
This allows each property of $\he$ to be formulated as a property of theta functions and vice versa.

\vspace{0,1cm}
For instance, by virtue of  \eqref{f7inttrian}, we have
${\rm{sign}}\, \im\, \he (z)^{2} = {\rm{sign}}\, \im\,z$ for all $z\! \in \! \Bb{C}\!\setminus\! \Bb{R}_{\geqslant 1}$, and
as a consequence, if we take into account \eqref{f13inttrian}
and \eqref{f14inttrian}(b), we get the result that\vspace{-0,2cm}
\begin{align}\label{f9bint}
    &  {\rm{sign}}\Big( \im\, \Theta_{3} \big(z\big)^{4}\Big) = {\rm{sign}} \left(\re\,z\right)    \, , \quad  z  \in \fet \,,
\end{align}

\noindent which was  established by the first and the second authors  in Corollary~1.2 of \cite[p.\! 599]{bh2}.
Our second example is the property
\begin{align}\label{f16inttrian}
  {\rm{sign}}\big( \re\,  \lambda^{\,\prime}  (z) \big)= -
 {\rm{sign}}  \left(\re\, z\right) \ , \quad
  z\in  \fet \setminus {\rm{clos}}\left(\Bb{D}\right) \,,
\end{align}

\noindent
which follows from \eqref{f8inttrian} and the identity
\begin{align}\label{f19int}
    &     \lambda^{\,\prime} (z)=  \imag \, \pi \,\lambda (z) \,\left(1 -\lambda (z)\right)\,\Theta_{3}\left(z\right)^{4} =  {\rm{i}}\, \pi \,\dfrac{\Theta_{2}(z)^{4} \Theta_{4}(z)^{4}}{\Theta_{3}(z)^{4}}\ , \ \   z \in \Bb{H}
\end{align}

\noindent (see  \cite[p.\! 599;  (1.30)]{bh2}), by application to them  \eqref{f14inttrian}, \eqref{f5aint}(b) and \eqref{f13inttrian}, as with the help of \eqref{f14inttrian},  the identity \eqref{f19int} can be written in the form
\begin{align*}
     &  \dfrac{ \lambda^{\,\prime} \left(\ie (z)\right)  }{\imag \, \pi}=\dfrac{\Theta_{2}\left(\ie (z)\right)^{4}\Theta_{4}
       \left(\ie (z)\right)^{4}}{\Theta_{3}\left(\ie (z)\right)^{4}} = z (1\!-\!z)\he  (z)^{2} , \quad  z  \in (0,1)\cup \left(\Bb{C}\setminus\Bb{R}\right) ,
\end{align*}

\noindent and then, in a second step, using \eqref{f8inttrian}, \eqref{f13inttrian} and \eqref{f5aint}(b), we obtain
\begin{align*}
     &  -{\rm{sign}} \,\re\,\lambda^{\,\prime} \left(\ie (z)\right)=
 {\rm{sign}} \,\im \dfrac{ \lambda^{\,\prime} \left(\ie (z)\right)  }{\imag \, \pi}= {\rm{sign}} \,\im \big(z (1\!-\!z)\he  (z)^{2}\big) = \\ & =  {\rm{sign}}\,\im \left(z\right)={\rm{sign}}\left( \re\, \ie (z)\right)  , \ \  \ie (z) \in
 \ie\left(\Bb{C}_{\re< 1/2}\setminus \Bb{R}_{\leqslant 0}\right)=
        \fet \setminus \overline{\Bb{D}} \,.
\end{align*}

\noindent
As for conclusions going in the opposite direction, here we can mention that
 the combination of \eqref{f14inttrian}(b), \eqref{f2aint}, \eqref{f2int} and  the  Landen trans\-for\-mation equations
\begin{align}\label{f3d1int}
    & \begin{array}{ll}
   {\rm{(a)}} \  2 \Theta_{2}(2 z)^{2}\! = \!\Theta_{3}( z)^{2}\!-\! \Theta_{4}( z)^{2},
   &   \quad  {\rm{(b)}}\ 2 \Theta_{3}(2 z)^{2} \!= \!\Theta_{3}( z)^{2}\!+ \!\Theta_{4}( z)^{2},
   \\[0,2cm] {\rm{(c)}}\ \Theta_{4}(2 z)^{2}\! =\! \Theta_{3}( z)\Theta_{4}( z)  ,
     & \quad \phantom{{\rm{(c)}}a}   z\in \Bb{H}\,,
  \end{array}
\end{align}

\noindent (see  \cite[p.\! 18]{law}{\hyperlink{r38}{${}^{\ref*{case38}}$}}\hypertarget{br38}{}),  gives{\hyperlink{r1}{${}^{\ref*{case1}}$}}\hypertarget{br1}{} the quadratic   transformation relation (3.1.10) with $a\!=\!b\!=\!1$ of   \cite[p.\! 128]{and} for the hypergeometric function $\he$ and equality \eqref{f2intcor1} as well. At the same time, each of the four
nontrivial functional relationships for the modular function $\lambda$ in the table
of
\cite[p.\! 111]{cha} (except for the first and third, considered as trivial)
can be written as the corresponding
Kummer transformation rule for $\he $ (see \cite[p.\! 106]{erd1}). For instance, in \cite[p.\! 33]{bh1} the Kummer identity
(27) of \cite[p.\! 106]{erd1})  was derived
from \eqref{f2aqinttheor1}.
It also follows{\hyperlink{r2}{${}^{\ref*{case2}}$}}\hypertarget{br2}{} from \eqref{f3d1int} that
\begin{align}\label{f3dint}
    & \lambda \left(\dfrac{1+z}{1-z}\right)=\dfrac{1}{2} +
\dfrac{i\lambda_{1} (z)}{4  \sqrt{\lambda_{2} (z)}}  \, ,   &     &
\lambda \left(\dfrac{z-1}{z+1}\right)=\dfrac{1}{2} -
\dfrac{i\lambda_{1} (z)}{4  \sqrt{\lambda_{2} (z)}}  \, ,    &     &    z\in \fet\,,
 \\    &
 \lambda_{2} (z):= \lambda (z)\big(1-\lambda (z)\big) \, ,    &     &   \lambda_{1} (z):= 1-2\lambda (z)\, ,    &     & z\in \Bb{H}\,,
\label{f3eint} \\    &
 \lambda_{2} (-1/z)\!=\!\lambda_{2} (z) \, ,    &     &  \lambda_{1} (-1/z)\!=\!-\!\lambda_{1} (z)\, ,    &     & z\in \Bb{H}\,.
\nonumber 
\end{align}

\noindent Observe that  the principal branch of the square root  can be used in \eqref{f3dint}  because
$\lambda(\fet) = (0,1)\cup \left(\Bb{C}\setminus\Bb{R}\right)$ and the two inclusions  $z\!\in\! \Bb{C}$,
$z(1-z)\!\in \!\Bb{R}_{<0}$ hold  if and only if $z\! \in \! \Bb{R}_{<0}\!\cup \Bb{R}_{>1}$. A combination of \eqref{f0apinttheor1}(b) and \eqref{f3dint} gives that{\hyperlink{r4}{${}^{\ref*{case4}}$}}\hypertarget{br4}{}
\begin{align}\label{f3zint}
 64  \left|\lambda \left(\dfrac{t+i}{t-i}\right)\right|=16\,\big/\!\sqrt{\lambda (it)\lambda (i/t)}  \geqslant \exp \dfrac{\pi }{2}\left(t + \dfrac{1}{t}\right) \ , \quad  t > 0\,,
\end{align}

\noindent while the relations \eqref{f3d1int} and \eqref{f3bint} imply{\hyperlink{r3}{${}^{\ref*{case3}}$}}\hypertarget{br3}{} that
\begin{align}\label{f3yint}
\left|\Theta_{3} \left(\dfrac{t+i}{t-i}\right)\right|^{4}  \leqslant 5 \left(t + \dfrac{1}{t}\right)
 \exp \left(-\dfrac{\pi }{4}\left(t + \dfrac{1}{t}\right)\right)  \ , \quad  t > 0 \,.
\end{align}

The relation
 \begin{align}
\label{f3int}
       \imag  \, \pi\, \ie^{\,\prime} (z)\, \he  (z)^{2}
       =\frac{1}{ z(1-z)}
      \ , \quad z \in (0,1)\cup \left(\Bb{C}\setminus\Bb{R}\right),
\end{align}

\noindent  (see  \cite[p.\! 597, (1.20)]{bh2})
allow us in Section~\ref{pf4aint}  to show that
the integral formulas of Theorem~\ref{intth1} for every $x\!\in\! \Bb{R}$ can be written as
follows, 
\begin{align}\label{f4int}
      \eurm{H}_{0} (x)  & =\dfrac{1}{ 2 \pi\imag}\!\!\!\!\int\limits_{ 1/2 - {\rm{i}} \infty}^{1/2 + {\rm{i}} \infty}
 \frac{\ie (y)\ie^{\,\prime} (y) \, \he  (y)^{2}}{ x^{2}- \ie (y)^{2}  } d y \, ,  \\
  \eurm{H}_{n} (x)   & =  \dfrac{(-1)}{4 \pi^{2} n}  \!\!\!\!   \int\limits_{ 1/2 - {\rm{i}} \infty}^{1/2 + {\rm{i}} \infty} \
\frac{S^{{\tn{\triangle}}}_{n}\! \left({1}\big/{y}\right)\ie^{\,\prime} (y) d y}{\big(x+\ie (y)\big)^{2} }
 \, , \quad n \geqslant 1 \,.
\label{f7int}\end{align}

\noindent
In view of \eqref{f14inttrian}(b), we can  use \eqref{f5int} and Theorem~\ref{intthA} to change the variable
$z = \ie (y)$ in the integrals \eqref{f4int}, \eqref{f7int} and for arbitrary $x\!\in\! \Bb{R}$ obtain (see \eqref{f5intlem3})
\begin{align}\label{f9int}
    &\hspace{-0,2cm} \eurm{H}_{0} (x)\!=\dfrac{1}{2 \pi\imag} \!\!\! \!\!\! \int\limits_{\gamma (-1,1)} \!\!\!\! \frac{z \, \Theta_{3}\left(z\right)^{4}}{  x^{2} - z^{2}} d z \, , \  \  \eurm{H}_{n} (x) \! = \dfrac{(-1)}{4 \pi^{2} n} \!\!\! \!\!\!  \int\limits_{\gamma (-1,1)}\!\!\!\!\!
\frac{S^{{\tn{\triangle}}}_{n} \!\left(\dfrac{1}{\lambda (z)}\right) d z}{(x+z)^{2}}\, , \quad n \geqslant 1 \,.\hspace{-0,1cm}
\end{align}

\noindent Here, the integrals are absolutely convergent for all $x\in \Bb{R}$ as follows from \eqref{f5intlem3}, \eqref{f3yint},
\eqref{f3zint} and the parametrization $\gamma (-1,1) \ni z = (t+i)/(t-i)$, $t \in \Bb{R}_{> 0}$.

\subsection[\hspace{-0,31cm}. \hspace{0,11cm}Monotonicity properties of the modular function.]{\hspace{-0,1cm}{\bf{.}} Monotonicity properties of the modular function $\lambda$.}
The equality \eqref{f19int}
  together with  \eqref{f9bint},  the identities
\begin{align}\label{f9cint} &
\begin{array}{l}
  \begin{displaystyle}
 \dfrac{d}{d x}\log\left( \left|\lambda (i y\! +\! x)\right|  \left|1\!-\!\lambda (i y\! +\! x)\right|\right)\!=\!
    \pi\,\im\,\left(\Theta_{3}^{4}(i y \!+\! x)\!-\!2\Theta_{3}^{4}(i y \!+\! x\!-\!1)\right)   \, ,
 \end{displaystyle}
 \\[0,2cm]
   \begin{displaystyle}
 \dfrac{d}{d x}\log \left|\lambda (i y\! + \!x)\right|= -  \pi\,\im \, \Theta_{3}(i y\! + \!x\!-\!1)^{4} \, , \quad x + {\rm{i}} y\in \Bb{H} \, ,
 \end{displaystyle}
\end{array}
\end{align}

\noindent and \eqref{f16inttrian} allow us to obtain in Section~\ref{pinttheor1} the following monotonicity properties  of $\lambda$
and the estimates of Corollary~\ref{intcorol1} below as well.

 \vspace{-0,06cm}
\begin{theorem}\hspace{-0,16cm}{\bf{.}}\label{inttheor1}
    Let $a \geqslant  1/2$ and $\lambda_{2} (z):= \lambda (z)\big(1-\lambda (z)\big)$. Then
\begin{align}\label{f1inttheor1}
    &\hspace{-1cm}  x \dfrac{d}{d x}\ \left|\lambda (x + {\rm{i}} y)\right|\ \,> 0,    \hspace{0,37cm}  x\!+\!i y\in \fet \!\setminus \!\left(
    {\rm{i}} \Bb{R}_{> 0}\right)\!,  \\    &  \label{f4inttheor1}
    \hspace{-1cm}  x \dfrac{d}{d x}\, \left|\lambda_{2} (x + {\rm{i}} y)\right|\, > 0,    \hspace{0,37cm}  x\!+\!i y\in \fet \!\setminus \!\left(
    {\rm{i}} \Bb{R}_{> 0}\right)\!,
     \\  &\hspace{-1cm}
    \label{f2inttheor1}
 x \dfrac{d}{d x} \ \re\, \lambda (x\! + \!i y)\, <  0 ,   \hspace{0,4cm}    x\!+\!i y\in \fet \!\setminus \!\left( \Bb{D}\cup {\rm{i}} \Bb{R}_{> 0}\right), \\  &\hspace{-0,45cm}
       \max_{{\fo{ \im z \!\geqslant\! a}} } \left|\lambda (z)\right| \ =  \left|\lambda ( 1 + {\rm{i}} a)\right|\  =  \frac{\lambda (i a)}{1-\lambda (i a) }  \ \,, \label{f3inttheor1}
 \\  &\hspace{-0,5cm}
       \max_{{\fo{ \im z \!\geqslant\! a}} } \left|\lambda_{2} (z)\right|  =  \left|\lambda_{2} ( 1 + {\rm{i}} a)\right|  =  \frac{\lambda (i a)}{\left(1-\lambda (i a)\right)^{2} }  \ \,. \label{f6inttheor1}
\end{align}

\end{theorem}

 Since $0\!< \!\lambda (i y)\!< \!1/2\!=\! \lambda (i )$ holds for  $y\!> \!1$, the next properties can  be  derived from Theorem~\ref{inttheor1}, \eqref{f19int} and $\lambda\big(\gamma (-1,1)\big)\! = \! (1/2)\! +\! {\rm{i}} \,\Bb{R}$ (see \eqref{f0apinttheor1}--\eqref{f5int}).

\begin{corsectthree}\hspace{-0,17cm}{\bf{.}}\label{intcorol1}
Suppose $t>0$, $x \in [-1,1]$,  $y> 1$ and put
\begin{align*}
    &  \eusm{E}^{0}_{\!{\tn{\frown}}}:= -1/\eusm{E}^{\infty}_{\!{\tn{\frown}}},\quad \eusm{E}^{\infty}_{\!{\tn{\frown}}}:=
\Bb{H} \setminus\underset{{\fo{m \in \Bb{Z}}}}{\cup } \left(2m +
\overline{\Bb{D}}\right),\quad\fet \cap \Bb{D}\subset \eusm{E}^{0}_{\!{\tn{\frown}}}.
\end{align*}

\noindent
Then $\lambda (i y)\in (0,1/2)$ and
\begin{align}\label{f1intcorol1} &\hspace{-0,25cm}
\begin{array}{ll}
\begin{displaystyle}
   {\rm{(a)}} \  \lambda (x\! + \!i y)\in \Bb{C}_{\re \leqslant \lambda (i y)} \bigcap
   \frac{\lambda (i y)}{1-\lambda (i y)}\,{\rm{clos}}\left(\Bb{D}\right) ,
   \end{displaystyle}
   &  \  {\rm{(b)}} \   \left|\lambda^{\,\prime} (x\! + \!i y)\right| \!\leqslant\!  9\ ,\hspace{-0,1cm} \\[0,5cm]
 {\rm{(c)}} \ \hspace{0,3cm}
  4 \left|\lambda(z)-\lambda(x+iy)\right|\geqslant \left|\lambda(z)\right|\, \left|1-2\lambda(iy)\right| \, ,  &  \  z \in \eusm{E}^{0}_{\!{\tn{\frown}}}\,, \\[0,25cm]
 {\rm{(d)}} \
  \sqrt{2}\, \left|\lambda(z)-\lambda(1+it)\right|\geqslant \left|\lambda(z)\right|+\left|\lambda(1+it)\right| \, ,  &  \ z \in \eusm{E}^{0}_{\!{\tn{\frown}}}\,.
\end{array}
\end{align}
\end{corsectthree}

\subsection[\hspace{-0,31cm}. \hspace{0,11cm}Imaginary part of the Schwarz triangle function.]{\hspace{-0,11cm}{\bf{.}} Imaginary part of the Schwarz triangle function.}
We note that   $\ie$ enjoys the functional relation\vspace{-0,1cm}
\begin{align}
\label{f3intthA}
\ie (z)\ie (1-z) =-1, \quad z \in
 (0,1)\cup \left(\Bb{C}\setminus\Bb{R}\right)\,,
\end{align}

\vspace{-0,1cm}
\noindent
 (see \cite[p.\! 597; (1.21)]{bh2}) and for every $x \! \in  \!\Bb{R} \setminus \!\{0, 1\}$ there exist
  $  \ie \! (x\pm\!  \mathrm{i}  \hspace{0,015cm}   0)
 :=\lim_{\, z\in \Bb{H}, \,  z \to 0} \  \ie  (x\pm z)$ (see \cite[p.\! 604]{bh2}). Moreover, there are relationships
 between the values of
$\ie$ on the  two sides of the cut along $\Bb{R}_{< 0}$:
\begin{align}
\label{f4intthA}
\ie (-x +  \imag  0 ) = 2 +  \ie (-x -
 \imag  0 ) \ , \quad  x>0\,,
\end{align}
\noindent  (see \cite[p.\! 597; (1.22)]{bh2}),  and  along the other cut  $\Bb{R}_{> 1}$ as well,
\begin{align}
\label{f5intthA}\dfrac{1}{\ie (1 + x+  \imag  0 )} = 2 +
\dfrac{1}{\ie (1+ x - \imag  0 )}
\  ,
\quad x>0\,,
\end{align}

 \noindent  (see \cite[p.\! 597; (1.23)]{bh2}). It was also explained in \cite[p.\! 601; (4.1)]{bh2} that  the Pfaff formula (see \cite[p.\! 79]{and}) gives that\vspace{-0,1cm}
\begin{align}\label{f6intthA}
    &  \ie (-x \pm  \imag   0 ) = \pm 1 +  \imag   \,
 {\fo{\Delta}} (x)  \, , \ \  x > 0 \,,
\end{align}

\vspace{-0,1cm}
\noindent where for  $x > 0$ we write (see \cite[p.\! 599]{bh2})
 \begin{align}\label{f15int}
    & \hspace{-0,2cm}{\fo{\Delta}} (x)\! := \!   \dfrac{\he   \left( {1}\big/{(1\!  +\!  x)}\right)}{\he
 \left(  {x}\big/{(1\!  +\!  x)}\right)} \ ,
    \quad
    \left\{\begin{array}{l}
    {\fo{\Delta}} (0)\! =\!  +\! \infty \ ,\\[0,2cm]
    {\fo{\Delta}} (+\infty)\!  = \! 0 \ ,
    \end{array}\right.
    \ \  \ \dfrac{{{\diff}} {\fo{\Delta}} (x)}{{{\diff}} x}\!  <\!  0, \ \
     {\fo{\Delta}} (x){\fo{\Delta}} (1/x)\!  = \!1.\hspace{-0,1cm}
\end{align}

\noindent
Then, by  \eqref{f3intthA},
\begin{align}\label{f2intlemma1}
     &
  \ie (1 + x \pm \imag  0)  =-\frac{1}{\ie (-x \mp \imag  0 )} =
     - \frac{1}{\mp 1 +  \imag   \,  {\fo{\Delta}} (x)} =
     \dfrac{\pm 1 +  \imag   \,  {\fo{\Delta}} (x)}{1 + {\fo{\Delta}} (x)^{2}} \ ,
\end{align}

\noindent i.e., we have
\begin{align}\label{f7intthA}
    &  \ie (1 + x + \imag  0)  =\frac{1 +  \imag   \,  {\fo{\Delta}} (x)}{1 + {\fo{\Delta}} (x)^{2}}\,, \ \
  \ie (1 + x - \imag  0)  =  \frac{-1 +  \imag   \,  {\fo{\Delta}} (x)}{1 + {\fo{\Delta}} (x)^{2}}\,, \ \  x > 0 \,.
\end{align}

\noindent Together with \eqref{f6intthA} this shows that the values of
$\im\,\ie$ on the two sides of the  cuts at $\Bb{R}_{< 0}$  and  at $\Bb{R}_{> 1}$ coincide. Hence,
 we extend the function $\im\,\ie$, initially defined on $(0,1)\cup(\Bb{C}\setminus\Bb{R})$, to $\Bb{R}_{< 0}\cup\Bb{R}_{> 1}$ declaring its values  to be given by
\begin{align}\label{f8intthA}
    &\hspace{-0,3cm}
   \begin{array}{lccclcl}
   \im\, \ie (  1\!+\! x  )  & \!  := \!  & \im\,\ie (1\! +\! x \!+ \!\imag  0 )  & \!  =\!   &
\im\, \ie (  1\!+\! x \!-\!  \imag  0 )& \!  = \!  & \dfrac{ {\fo{\Delta}} (x)}{1\! + \!{\fo{\Delta}} (x)^{2}} \ ,
   \\[0,3cm]
    \im\, \ie (-x  )  &  \! :=  \! &   \im\, \ie (-x \!+ \! \imag  0 )   & \!  = \!  &  \im\,\,  \ie (-x \!-\!
 \imag  0 ) &  \! = \!  & {\fo{\Delta}} (x)  \ ,   \ \   x >0 \,.
   \end{array}
\hspace{-0,1cm}\end{align}

\noindent We find that the resulting function $\im\,  \ie$ is continuous on $\Bb{C}\!\setminus\!\{0,1\}$. Taking the limit in the relation \eqref{f2aint}(b) as
$z \in (0,1)\!\cup\! \left(\Bb{C}\!\setminus\!\Bb{R}\right),
 z\! \to \!\Bb{R}_{<0}\!\cup \Bb{R}_{>1}$,  we obtain from \eqref{f6intthA} and \eqref{f2intlemma1} that
 \begin{align}\label{f1intlemma1}
    &
   \left\{ \begin{array}{rcccl} \lambda \big(\ie (-x \pm  \imag   0 )\big) &   =   &
    \lambda \left(\pm 1+i{\fo{\Delta}} (x)\right) &   =   & -x \,, \\[0,1cm]
     \lambda \big(\ie (1+ x \pm  \imag   0 )\big) &   =   &
   \lambda \left(\,1 \big/ \big(\pm 1-i{\fo{\Delta}} (x)\hspace{0,015cm}\big)\hspace{0,015cm}\right) &   =   & 1 + x\,,
    \end{array}\right.
    \qquad    x > 0 \,,
 \end{align}

\noindent where ${\fo{\Delta}}$ maps $\Bb{R}_{>0}$   onto itself in a one-to-one fashion, by \eqref{f15int}. Together with Theorem~\ref{intthA} this  means the following (see \eqref{f1schw}).

\begin{lemsectthree}\hspace{-0,17cm}{\bf{.}}\label{lemoneto} The modular function $\lambda$ maps each of the sets
$\gamma (1,\infty)=1+\imag \Bb{R}_{>0}$ and
 $\gamma (-1,\infty)=-1+\imag \Bb{R}_{>0}$ one-to-one onto $\Bb{R}_{<0}$ and each of the set \vspace{0,1cm} \\
$\begin{array}{lcccr} \hspace{1.9cm}
   \gamma (0,1) & = & 1/(1-\imag \Bb{R}_{>0})& = &1/2 + (1/2)(\Bb{H} \cap\partial \Bb{D}),\\  \hspace{1.9cm}
   \gamma (-1,0)& = &1/(-1-\imag \Bb{R}_{>0})& = &-1/2 + (1/2)(\Bb{H} \cap\partial \Bb{D})
\end{array}$

 \vspace{0,1cm}
 \noindent
 one-to-one onto $\Bb{R}_{>1}$. As a consequence, $\lambda$ maps each of the sets
\begin{align}\label{f1lemoneto}
     & \begin{array}{ll}
   \fet \sqcup \gamma (1,\infty) \sqcup \gamma (0,1)  \ , \   &
   \fet \sqcup \gamma (-1,\infty) \sqcup \gamma (0,1)  \ , \   \\[0,1cm]
    \fet \sqcup \gamma (1,\infty) \sqcup \gamma (-1,0)  \ , \      &
    \fet \sqcup \gamma (-1,\infty) \sqcup \gamma (-1,0) \ , \
       \end{array}
\end{align}

\noindent one-to-one onto $\Bb{C}\setminus\{0, 1\}$.
\end{lemsectthree}

By using the property  ${\fo{\Delta}}(x)\! \to \! +\infty$ as $x\!> \!0 $, $x\! \to\! 0$, which comes  from \eqref{f15int},
 we obtain from \eqref{f8intthA}  that $\im\, \ie (-x  )\!\to \! +\infty$  and $\im\, \ie (1\!+\!x  )\!\to\! 0$ as $x\!> \!0 $, $x\! \to\! 0$, which together with the properties
\begin{align*}
    &  \lim\nolimits_{{\fo{z\! \in\!\Lambda, z\! \to \!1}}} \ie (z) =0 \, , \
     \lim\nolimits_{{\fo{z\! \in\!\Lambda,  z\! \to \!0}}} \im\,\ie (z) = +\infty \, , \  \Lambda\!:=\! (0,1)\!\cup\! \left(\Bb{C}\!\setminus\!\Bb{R}\right),
\end{align*}

\noindent (see \cite[p.\! 609; Lemmas 4.1, 4.2]{bh2})
means that the  function $\im\, \ie$, which is continuous on
 $\Bb{C}\setminus\{0,1\}$, can be continuously extended to the point $1$ as well, with  value $\im\, \ie (1)=0 $,  while
$\lim_{ z\to \,0} \im\, \ie (z)=+\infty$. As a result, we get the following statement.

\vspace{-0,2cm}
\begin{lemsectthree}\hspace{-0,18cm}{\bf{.}}\label{intlemma1} Let ${\fo{\Delta}}$ be defined as in \eqref{f15int}. The harmonic function $\im\,\ie$ on
 $(0,1)\cup \left(\Bb{C}\setminus\Bb{R}\right)$ can be continuously extended to   $\Bb{C}\setminus\{0\}$
 with  values on $\Bb{R}_{< 0}\!\cup \Bb{R}_{\geqslant 1}$ given by \eqref{f8intthA} and
 $\im\, \ie (1)\!:=\!0$. The extended  function $\im\, \ie$  is positive on $\Bb{C}\setminus\{0, 1\}$ and satisfies

 \vspace{-0,5cm}
\begin{align}\label{f9intthA}
    &\hspace{-0,2cm}
     {\rm{(a)}} \ \lim_{ z\to \,0} \im\, \ie (z)=+\infty   \, , \ \  {\rm{(b)}} \  \im\, \ie \big(\lambda(y)\big) \!=\!\im\, y  \, , \ \  y \!\in\! \Bb{H}\cap {\rm{clos}}\left(\mathcal{F}_{{\tn{\square}}}\right)
.\hspace{-0,1cm}
\end{align}
\end{lemsectthree}

\vspace{-0,1cm}
It follows from Theorem~\ref{intthA}, \eqref{f1intlemma1} and \eqref{zbssqlem1} that for every $a \in \Bb{R}_{<0}\cup \Bb{R}_{>1}$ the equation
$\lambda (z) = a$ with $z\in \Bb{H}\cap {\rm{clos}}\left(\mathcal{F}_{{\tn{\square}}}\right)$ has exactly two solutions $z=\ie (a \pm {\rm{i}} 0)$ with equal imaginary parts, by virtue of \eqref{f8intthA}, while for every
$a \in  (0,1)\!\cup\! \left(\Bb{C}\!\setminus\!\Bb{R}\right)$ it has a unique solution given by $z=\ie (a )$.  In fact,   Lemma~\ref{lemdhp} below implies the following stronger property, which we obtain later on  in Section~\ref{pinttheor2}.
\begin{theorem}\hspace{-0,15cm}{\bf{.}}\label{inttheor2} For every
$a \in \Bb{C}\setminus\{0,1\}$ there exists a finite maximum

\vspace{-0,4cm}
\begin{align*}
     & \max \left\{\,\im\, z \, \mid \, \lambda(z)=a \ , \
     z\in \Bb{H}_{|\re|\leqslant 1}
     \,\right\} , \ \ \Bb{H}_{|\re|\leqslant 1}:=\left\{\,z\in \Bb{H} \,\mid \, \re\, z \in [-1,1]\,\vphantom{\Bb{H}_{|\re|\leqslant 1}}\right\} ,
\end{align*}

\vspace{-0,1cm}
\noindent   attained at one or two points, which belong to $\Bb{H}\cap {\rm{clos}}\left(\mathcal{F}_{{\tn{\square}}}\right)$. There is  one such  point $\ie (a)$ if $a \in (0,1)\cup \left(\Bb{C}\setminus\Bb{R}\right)$ and two points $\ie (a \pm {\rm{i}} 0)$ if $a  \in (-\infty, 0) \cup (1, +\infty)$.
 In particular

 \vspace{-0,4cm}
\begin{align}\label{f1inttheor2}
     & z \in \Bb{C}\setminus\{0,1\} , \ y\in \Bb{H}_{|\re|\leqslant 1}\setminus{\rm{clos}}\left(\mathcal{F}_{{\tn{\square}}}\right)  , \
     \lambda(y)=z \ \Rightarrow \ \im\, y <   \im\, \ie (z) \,,
\end{align}

\vspace{-0,1cm}
\noindent where $\im\, \ie $ is  as in  Lemma~{\rm{\ref{intlemma1}}}.
\end{theorem}

\vspace{-0,2cm}
\begin{corsectthree}\hspace{-0,15cm}{\bf{.}}\label{intcorol2}
For all pairs $(y,z)$ with $y \in \Bb{H}\cap{\rm{clos}}\left(\mathcal{F}_{{\tn{\square}}}\right)$,
$z \in \Bb{H}_{|\re|\leqslant 1}$, and $\im\, z > \im\, y $, we have  $\lambda (z)- \lambda(y)\neq 0$.
\end{corsectthree}


\vspace{-0.25cm}

\section[\hspace{-0,30cm}. \hspace{0,11cm}Schwarz triangle polynomials]{\hspace{-0,095cm}{\bf{.}} Schwarz triangle polynomials}\label{intprel}

\vspace{-0,2cm}
We begin this section by recalling the following property of periodic functions.{\hyperlink{r711}{${}^{\ref*{case711}}$}}\hypertarget{br711}{}

\vspace{-0,1cm}
\begin{lemsectfour}\hspace{-0,2cm}{\bf{.}}\label{lemfou}
 Let $a \in \Bb{R}$ and $f_{0}\!\in\! {\rm{Hol}} ({\rm{i}}a + \Bb{H})$ be periodic with period $2$. Then

 \vspace{-0,2cm}
 \begin{align}\label{f1lemfou}
    &  \int_{z_{1}}^{z_{1}+2} f_{0} (\zeta) d \zeta = \int_{z_{2}}^{z_{2}+2} f_{0} (\zeta) d \zeta \ , \quad
    z_{1}, z_{2} \in {\rm{i}}a + \Bb{H} \,.
 \end{align}

\vspace{-0,2cm}
 \noindent If $\varphi : \Bb{D}\setminus\{0\} \mapsto \Bb{C}$ is such that $\varphi (\exp ({\rm{i}}\pi z)) \in {\rm{Hol}} (\Bb{H})$ then $\varphi \in {\rm{Hol}} (\Bb{D}\setminus\{0\}) $.
\end{lemsectfour}

\subsection[\hspace{-0,31cm}. \hspace{0,11cm}Definition and connection with Faber polynomials.]{\hspace{-0,11cm}{\bf{.}} Definition and connection with Faber polynomials.}
The notation \eqref{f1intthA}
permits us for arbitrary positive integer $n$  to write the decomposition \eqref{f1aint} in the form
\begin{align}\label{f6int}    &
  {\rm{e}}^{{\fo{- n \pi \imag \ie(z)  }}} \!=\! S^{{\tn{\triangle}}}_{n} (1/z)\! + \!\Delta_{n}^{S}(z)\ , \ z \!\in \! \Bb{D}\setminus \{0\} \ , \
 S^{{\tn{\triangle}}}_{n} (0)\!=\!0     \, , \ \Delta_{n}^{S}\!\in\! {\rm{Hol}}(\Bb{D}) \,,
\end{align}

\noindent
  and we will call $S^{{\tn{\triangle}}}_{n}$  the $n$-th {\it{Schwarz triangle polynomial}}.

\vspace{0.05cm}
A special symmetrization of algebraic polynomials occurs when we
consider the expression $S^{{\tn{\triangle}}}_{n} \!\left({1}/{\lambda (z)}\right)\pm S^{{\tn{\triangle}}}_{n} \!\left({1}/{\lambda (-1/z)}\right)$
to  make  the function $S^{{\tn{\triangle}}}_{n} \!\left({1}/{\lambda (z)}\right)$ in
\eqref{f9int} invariant up to a multiplier $\pm 1$ with respect to the argument reversal
$z\mapsto-1/z$. In view of \eqref{f2int}, we have $\lambda (-1/z)\!=\!  1\!- \!\lambda (z)$,  which suggests the following algebraic
operations (see Section~\ref{pintlem1}).

\begin{lemsectfour}\hspace{-0,18cm}{\bf{.}}\label{intlem1} Let $n$ be a positive integer and $P_n$ be an algebraic polynomial of degree $n$ with real coefficients. Then there exist two algebraic
polynomials $\eusm{R}^{\pm}\left[P_n\right]$ of degree $n$ with real coefficients such that for every
$x \in \Bb{C}\setminus\{0, 1\}$, we have
\begin{align*} &
   P_{n} \left(\frac{1}{x}\right) +  P_{n} \left(\frac{1}{1-x}\right) =
  \eusm{R}^{+}\!\!\left[P_n\right] \left(\frac{1}{x (1-x)}\right)  ,     &     &   \eusm{R}^{+}\!\left[P_n\right](0)= 2 P_n (0) \, , \      \\  &
   P_{n} \left(\frac{1}{x}\right) -  P_{n} \left(\frac{1}{1-x}\right) =
 (1-2x) \,\eusm{R}^{-}\!\left[P_n\right] \left(\frac{1}{x (1-x)}\right)  \, ,   &     &      \eusm{R}^{-}\!\left[P_n\right]  (0)=0  .
\end{align*}
\end{lemsectfour}

\noindent For each $ n \!\geqslant\! 1$, the
symmetrized functions
\begin{align}\label{f10int}\hspace{-0,2cm}
\eurm{R}^{\pm}_{n}(z)\!:=\! S^{{\tn{\triangle}}}_{n} \left(\frac{1}{\lambda (z)}\right) \!\pm \! S^{{\tn{\triangle}}}_{n} \left(\frac{1}{1\!-\!\lambda (z)}\right) ,
 \ \eurm{R}^{\pm}_{n}(-1/z)\!= \! \pm\eurm{R}^{\pm}_{n}(z), \  \  z\!\in\! \Bb{H},\hspace{-0,1cm}
\end{align}

\noindent  are holomorphic in $\Bb{H}$  and $2$-periodic. They were considered for the first time
by the fourth  and fifth authors \cite{rad} in the context of      Fourier pair   interpolation on the real line.

\vspace{0,05cm}

In accordance with \eqref{f2int},  for $z\in \Bb{H}$ and
$ u\in \Bb{D}$ we introduce the notation
\begin{align}\label{f11int}
    &  \hspace{-0,2cm}  \lambda (z)\! = \! \dfrac{16 \e^{\imag  \pi z }
 \theta_{2}\left({\e}^{\imag  \pi z}\right)^{4}}{\theta_{3}\left({\e}^{\imag  \pi z}\right)^{4}}\! = \!
  \lambda_{\Bb{D}}\left({\e}^{\imag  \pi z}\right) \, , \ \ \lambda_{\Bb{D}}(u)\!  :=\!
  \dfrac{16 u
\left(1  \!+ \!\sum\nolimits_{n\geqslant 1} u^{n^2 +n} \right)^{4}}{\left(1  \!+ \! 2\sum\nolimits_{n\geqslant 1} u^{ n^2  }\right)^{4}}. \hspace{-0,2cm}
\end{align}

\noindent  Then clearly, we have $\lambda_{\Bb{D}} \in {\rm{Hol}}(\Bb{D})$, $\lambda_{\Bb{D}}(0)=0$, $\lambda^{\,\prime}_{\Bb{D}}(0)=16$.
As we substitute $\lambda(z)$ for $z$ in \eqref{f6int} and apply the left-hand side identity in \eqref{f2aint},
we find
\begin{align}\label{f12int}
    &  \hspace{-0,3cm}\Delta_{n}^{S}\left(\lambda_{\Bb{D}}\left( {\rm{e}}^{i \pi z}\right) \right) \!=\!
     {\rm{e}}^{{\fo{-i n \pi z}}}\! - \!S^{{\tn{\triangle}}}_{n}\left(\!\frac{1}{\lambda(z)}\! \right), \  \  z \!\in\!\ie(\Bb{D}\setminus (-1,0])\subset \fet  \, , \  n \! \geqslant\! 1\,. \hspace{-0,25cm}
\end{align}

\noindent
In view of \eqref{f2int}, $\lambda$, $1/\lambda \in {\rm{Hol}}(\Bb{H})$ and hence the right-hand side function in \eqref{f12int} is holomorphic on $\Bb{H}$. In view of Lemma~\ref{lemfou} applied to the left-hand side function in \eqref{f12int}
we see that $\Delta_{n}^{S} (\lambda_{\Bb{D}})\in {\rm{Hol}}(\Bb{D}\setminus\{0\} )$.
But since $\Delta_{n}^{S}, \lambda_{\Bb{D}} \!\in\! {\rm{Hol}}(\Bb{D})$ with
$\lambda_{\Bb{D}}(0)=0$, the composition  $\Delta_{n}^{S} (\lambda_{\Bb{D}})$ is holomorphic in
a neighborhood of the origin, so that
in fact the function  $\Delta_{n}^{S} \!\in\! {\rm{Hol}}(\Bb{D})$ of \eqref{f6int} meets
\begin{align*}
    &  \Delta_{n}^{S} (\lambda_{\Bb{D}})\in {\rm{Hol}}(\Bb{D}) \,, \quad n \geqslant 1 \,,
\end{align*}

\noindent and, in view of the uniqueness theorem  for analytic functions (see  \cite[p.\! 78]{con}),  the relationship \eqref{f12int} can be written as
\begin{align*}
    &\hspace{-0,3cm} S^{{\tn{\triangle}}}_{n}\left({1}/{\lambda_{\Bb{D}}(u)} \right)\! =\! u^{-n} \!+ \! \Delta_{n}^{S}\left(\lambda_{\Bb{D}}\left( u\right) \right) \, , \  u\! \in \!\Bb{D}\setminus\{0\}  \, , \  \Delta_{n} (\lambda_{\Bb{D}})\!\in\! {\rm{Hol}}(\Bb{D}) \, , \   n \!\geqslant\! 1\,,\hspace{-0,1cm}
    \end{align*}

\noindent  where $S^{{\tn{\triangle}}}_{n} (0)=0$.
According to the definition of Issai Schur \cite[p.\! 34]{hur}, if $F_{n}$ is the $n$-th
 {\emph{Faber polynomial}} of $16/ \lambda_{\Bb{D}}(1/u)$ then $S^{{\tn{\triangle}}}_{n}(x)\!=\!F_{n}(16 x)\!-\!F_{n}(0)$, $n\!\geqslant\!1$.

\subsection[\hspace{-0,31cm}. \hspace{0,11cm}The generating function.]{\hspace{-0,11cm}{\bf{.}} The generating function.}\label{genfun}

\vspace{0,1cm}
The generating function of the Schwarz triangle polynomials is calculated in the next lemma.

\begin{lemsectfour}\hspace{-0,2cm}{\bf{.}}\label{lemgenerat} We have
\begin{align}
     & \hspace{-0,1cm} {\rm{i}} \pi \sum\limits_{n=1}^{\infty} S^{{\tn{\triangle}}}_{n} \left(1/z \right) {\rm{e}}^{\imag  \pi n y} = \frac{
\lambda^{\,\prime}\left(y\right)  }{z  - \lambda\left(y\right)} \ , \  \
\im\, y >  \im\, \ie (z) \ , \ \
 z\in \Bb{C} \setminus \{0\} ,\hspace{-0,1cm}
\label{f22int}\end{align}

\noindent where the series  converges absolutely and the function $\im\, \ie $ is continuously extended to $ \Bb{C}\setminus\{0\}$ in accordance with Lemma~{\rm{\ref{intlemma1}}}. 
\end{lemsectfour}

\noindent
{\bf{Proof of Lemma~\ref{lemgenerat}.}}
Let
$z\in \Bb{C} \setminus \{0\}$. Then there exists $\beta\in (0,1)$ such that  $|z|>\beta$.
We calculate
\begin{align*}
    &   \frac{1}{2 \pi {\rm{i}}} \int\nolimits_{\beta\partial  \Bb{D}}\, \zeta^{n}\zeta^{-p-1}  d \zeta=\delta_{n, p}  \ , \ \ \  n, p \in \Bb{Z}; \quad     \dfrac{1}{z  - \zeta} =   \sum\nolimits_{n\geqslant 0} \ \frac{\zeta^{n}}{z^{n+1}}  \ , \quad   \zeta\in \beta\partial  \Bb{D},
\end{align*}

\noindent and in view of \eqref{f6int}, we have
\begin{align}\label{f13int}
    &  \hspace{-0,3cm} S^{{\tn{\triangle}}}_{n} (1/z) =
   \frac{1}{2 \pi {\rm{i}}} \int_{\beta\partial  \Bb{D}}   \frac{
   {\rm{e}}^{{\fo{- n \pi {\rm{i}} \ \ie(\zeta) }}} d \zeta}{z  - \zeta} =
   \frac{1}{2 \pi {\rm{i}}} \int_{-\pi}^{\pi}   \frac{
   {\rm{e}}^{{\fo{- n \pi {\rm{i}} \ \ie\left(\beta {\rm{e}}^{{\rm{i}} \varphi}\right) }}} d \left(\beta {\rm{e}}^{{\rm{i }}\varphi}\right) }{z  - \beta
    {\rm{e}}^{{\rm{i}} \varphi}} \, , \hspace{-0,1cm}
\end{align}

\noindent  where the curve $\ie\left(\beta {\rm{e}}^{{\rm{i}} \varphi}\right)$, $-\pi <  \varphi < \pi $,
lies in $\fet$ and connects two points
$\ie\!\!\left(-\beta - {\rm{i}} 0\right) $ and $\ie\!\!\left(-\beta + {\rm{i}} 0\right)$ which according to (see \eqref{f6intthA} and \eqref{f15int})
\begin{align*}
    &  \ie (-\beta +  \imag  0 ) = 2 +  \ie (-\beta - \imag  0 ) = 1 +   \imag   \,
{\fo{\Delta}} (\beta) , \quad {\fo{\Delta}} (\beta)  > 1 ,
\end{align*}

\noindent  belong to the boundary of $ \fet$.  Making the change of variables $\zeta = \ie\left(\beta
{\rm{e}}^{{\rm{i}} \varphi}\right)$ for  $-\pi <  \varphi < \pi $ in \eqref{f13int} and  using  \eqref{f2aint}, we obtain
$\beta {\rm{e}}^{{\rm{i}} \varphi} = \lambda\left(\zeta\right)$
and
\begin{align}\label{f17int}
    &    S^{{\tn{\triangle}}}_{n} (1/z) = \frac{1}{2 \pi {\rm{i}}} \int\nolimits_{{\fo{ -1 +   \imag   {\fo{\Delta}} (\beta)  }} }^{{\fo{ 1 +   \imag    {\fo{\Delta}} (\beta) }} }  \ \frac{
   \lambda^{\,\prime}\left(\zeta\right){\rm{e}}^{{\fo{- n \pi {\rm{i}}  \zeta }}}  }{z  - \lambda\left(\zeta\right)}\, d \zeta\,, \qquad |z|> \beta  , \ \beta\in (0,1) \, ,
\end{align}

\noindent where the contour of integration $\{\ie\!\left(\beta \label{e}^{{\rm{i}} \varphi}\right)\}_{-\pi <   \varphi  <  \pi }$ can be replaced by the straight line segment $[ -1, 1 ]\!+ \!  \imag    {\fo{\Delta}} (\beta)$ connecting the two points $\pm 1 +   \imag    {\fo{\Delta}} (\beta)$, as by \eqref{f1inttheor1},\vspace{-0,2cm}
\begin{align*}
    & \lambda \Big( \big\{\ z\in \Bb{H}\ \big|\ |\re\, z| \leqslant 1 , \  \im\, z \geqslant {\fo{\Delta}} (\beta)\ \big\}\Big) \subset     \beta\Bb{D}  \ , \quad \beta\in \left(0, 1\right).
\end{align*}

\vspace{-0,2cm}
\noindent In view of \eqref{f9intthA}, Corollary~\ref{intcorol2} tells us that the integrand in
\eqref{f17int} is a holomorphic function of the variable $\zeta$ when $\im \, \zeta > \im\, \ie (z)$, and hence the periodicity of the integrand in \eqref{f17int} allows us to use Lemma~\ref{lemfou} to shift the
segment of integration $[ -1, 1 ]\!+ \!  \imag    {\fo{\Delta}} (\beta)$ to
   $[ -1, 1 ]\!+ \!  \imag    a$ for any $a >   \im\, \ie (z)$. This gives
\begin{align}\label{f17aint}\hspace{-0,2cm}
    &    S^{{\tn{\triangle}}}_{n} (1/z)\! =\!  \frac{1}{2 \pi {\rm{i}}} \int\nolimits_{{\fo{ -1 \! +  \!  \imag   a }} }^{{\fo{ 1 \! + \!   \imag    a }} }   \frac{
   \lambda^{\,\prime}\left(\zeta\right){\rm{e}}^{{\fo{- n \pi {\rm{i}}  \zeta }}}  }{z  - \lambda\left(\zeta\right)}\, d \zeta\,, \ \ a >   \im\, \ie (z)  , \
   z\in \Bb{C} \setminus \{0\}  .\hspace{-0,1cm}
\end{align}

  \noindent
  We conclude that the series\vspace{-0,2cm}
\begin{align}\label{f21int}
    &  \sum\limits_{n=1}^{\infty} S^{{\tn{\triangle}}}_{n} \left(1/z \right) {\rm{e}}^{\imag  \pi n y} =
    \frac{1}{2 \pi {\rm{i}}} \int\nolimits_{{\fo{ -1 +   \imag   a }} }^{{\fo{ 1 +   \imag    a }} }
    \  \frac{  \lambda^{\,\prime}\left(\zeta\right) d \zeta }{\left({\rm{e}}^{{\fo{  \pi {\rm{i}}(  \zeta -y) }}} -1\right)\big(z  - \lambda\left(\zeta\right)\big)} \,,
\end{align}

\vspace{-0,1cm}
\noindent converges absolutely if $\im\, y\!>\!a \!>\! \im\, \ie (z)$. By the periodicity of the integrand in \eqref{f21int}, we can  apply Lemma~\ref{lemfou} again to shift the interval of integration  horizontally  by $\re\, y$. We then obtain from the Cauchy formula applied to the rectangle with vertices $ \pm  1 + {\rm{i}} a + \re\, y$, $ \pm 1  + {\rm{i}}A+ \re\, y$ with $A > \im\, y $ that this integral equals\vspace{-0,2cm}
\begin{align*}
     &  \frac{1}{i \pi}  \frac{  \lambda^{\,\prime}\left(y\right)  }{z  - \lambda\left(y\right)} +
      \frac{1}{2 \pi {\rm{i}}} \int\limits_{{\fo{-1 +   \imag   A +\re\, y}} }^{{\fo{  1 +   \imag   A +\re\, y}} }     \frac{  \lambda^{\,\prime}\left(\zeta\right) d \zeta }{\left({\rm{e}}^{{\fo{  \pi {\rm{i}}(  \zeta -y) }}} -1\right)\big(z  - \lambda\left(\zeta\right)\big)} \ , \quad A > \im\, y \,.
     \end{align*}

\vspace{-0,2cm}
\noindent The integral expression tends to zero as  $A \to +\infty$, because
  it follows from the expressions \eqref{f19int} for $\lambda^{\,\prime}\left(\zeta\right)$ and
\eqref{f11int} for $\lambda \left(\zeta\right)$ that
$\lambda^{\,\prime}\left(\zeta\right)\to 0$ and
$\lambda \left(\zeta\right)\to 0$, as $\im\,\zeta \to +\infty$.
This proves that the identity \eqref{f22int} holds. $\square$

\vspace{-0,2cm}
\subsection[\hspace{-0,31cm}. \hspace{0,11cm}Asymptotic behavior  for the large index.]{\hspace{-0,11cm}{\bf{.}} Asymptotic behavior  for the large index.} The condition in \eqref{f22int} for the convergence of the series for the generating function of the Schwarz triangle
polynomials is sharp. Indeed, it follows from  Lemma~\ref{lemdhp} that for every $a \in \Bb{C}\setminus\{0,1\}$ the set $\{y\in \Bb{H}_{|\re|\leqslant1} \,|\,\lambda(y)= a\}$  is countable and has no limit points in $\Bb{H}$.  Then \eqref{f1inttheor2} implies
that for every $z\! \in\! \Bb{C}\!\setminus\!\{0,1\}$ there exists $\beta (z) \!\in\! (0,1)$ such
that $\im\, y \!<\! \beta (z) \im\, \ie (z)$ for all $y\!\in \! \Bb{H}_{|\re|\!\leqslant 1}\!\setminus\!{\rm{clos}}\left(\mathcal{F}_{{\tn{\square}}}\right)$
satisfying $\lambda(y)= z$. This allows us to apply the Cauchy formula in \eqref{f17aint} for the rectangle with vertices $ \pm \! 1 \!+ \!\imag a $,
$ \pm 1 \! +  \imag    \beta (z) \im\, \ie (z)$ and  obtain
\begin{align}\label{f21zint}
    &   S^{{\tn{\triangle}}}_{n} (1/z)\! =\!\!\!\!\!\!\sum\limits_{{\fo{\left\{y\!\in\! {\rm{clos}}\left(\mathcal{F}_{{\tn{\square}}}\right) \ | \ \lambda(y)\!=\! z\right\}}}}\!\!{\rm{e}}^{{\fo{- n \pi \imag  y }}}+
{\rm{O}} \left({\rm{e}}^{{\fo{ n \pi \beta (z) \im\, \ie (z) }}}\right) \, , \  n\to +\infty .
\end{align}

\noindent where{\hyperlink{r9}{${}^{\ref*{case9}}$}}\hypertarget{br9}{} by  Theorem~\ref{inttheor2}, \eqref{f6intthA} and \eqref{f7intthA},
\begin{align*}
    &  \sum\limits_{{\fo{\left\{y\!\in\! {\rm{clos}}\left(\mathcal{F}_{{\tn{\square}}}\right) \ \big| \ \lambda(y)\!=\! z\right\}}}}\!\!\!\!\!\!\!\!\!\!\!\!\!\!\!\!\!\!\!\!{\rm{e}}^{{\fo{- n \pi \imag  y }}}=
\left\{
  \begin{array}{ll}
\exp \left(- n \pi \imag \ie (z)\right)    \,, &\ \  \hbox{if}\quad z \in
 (0,1)\cup \left(\Bb{C}\setminus\Bb{R}\right) \,; \\[0,25cm]
 2 {\rm{e}}^{{\fo{    \dfrac{  \, n \pi {\fo{\Delta}} (x)}{1 + {\fo{\Delta}} (x)^{2}}  }}}\cos  \dfrac{ n \pi }{1 + {\fo{\Delta}} (x)^{2}} \  , &\ \ \hbox{if}\quad z= 1+x  \ , \quad  x > 0 \,;\\[0,5cm]
2 (-1)^{n} \exp \left( n \pi {\fo{\Delta}} (x)\right)     , &\ \  \hbox{if}\quad z = - x  \ , \qquad  x > 0 \,.
  \end{array}
\right.
\end{align*}

\subsection[\hspace{-0,31cm}. \hspace{0,11cm}Symmetry property.]{\hspace{-0,11cm}{\bf{.}} Symmetry property.} The condition for the convergence of the series in
\eqref{f22int} is at its weakest when $z=1$ as $\im\, \ie (1)=0$. In the next lemma we explain this fact by showing in
\eqref{f2intlem3} that $S^{{\tn{\triangle}}}_{n} (1 ) = {\rm{O}}  (n^{2})$, as $n\to \infty$,
and derive in \eqref{f1intlem3}(a) an important symmetric property  of the
Schwarz triangle polynomials.\vspace{-0,2cm}
\begin{lemsectfour}\hspace{-0,15cm}{\bf{.}}\label{intlem3} The following identities hold
\begin{align}\label{f1intlem3}  & {\rm{(a)}} \ \,(-1)^{n} S^{{\tn{\triangle}}}_{n} (z)=S^{{\tn{\triangle}}}_{n} (1-z)-
S^{{\tn{\triangle}}}_{n} (1) \ ,  \quad   {\rm{(b)}} \ \,   S^{{\tn{\triangle}}}_{2n} (1)=0 \ , \qquad  n \!\geqslant \!1\,,
 \\    & \sum\limits_{n=1}^{\infty} S^{{\tn{\triangle}}}_{n} \left(1 \right) u^{n-1}\! =\! 16\,\theta_{2}(u)^{4} ,   \   u \!\in \!\Bb{D} \, ;  \ \  S^{{\tn{\triangle}}}_{n}(1) \!=\!
    \left(1\! - \!(-1)^{n}\right)  r_{4}(n) , \   n \!\geqslant\! 1 \,,
\label{f2intlem3}
 \\[-0,2cm]      \label{f5intlem3}
    &  S^{{\tn{\triangle}}}_{n} (z)\! := \! \sum\limits_{k=1}^{n} s^{{\tn{\triangle}}}_{n, k} z^{k}
      , \  n\!\geqslant \!1\,,  \
        S^{{\tn{\triangle}}}_{1} (z)\! = \!16 z  , \
 S^{{\tn{\triangle}}}_{2} (z)\! =\! 16^{2} z (z - 1), \ \ z\!\in\! \Bb{C}\,,
       \\    &   \label{f6intlem3} s^{{\tn{\triangle}}}_{n, n} =16^{n}
      \ , \quad  n\geqslant 1\,,  \quad
    s^{{\tn{\triangle}}}_{n, n-1} =- 8 n 16^{n-1} \ , \quad  n\geqslant 2\,, \\    &
\label{f7intlem3}    \sum\limits_{n=1}^{\infty} s^{{\tn{\triangle}}}_{n, 1} u^{n-1} \!= \!
     16 \,\theta_{2}(u)^{4}\,\theta_{3}(-u)^{4}\,\theta_{3}(u)^{-4} \ , \quad  u \in \Bb{D}\,.
\end{align}
\end{lemsectfour}

We supply the proof of Lemma~\ref{intlem3} in Section~\ref{pintlem1}.
 By the  symmetric property \eqref{f1intlem3} and
the change of variables $z^{\,\prime}=-1/z$  on the right-hand side integral of \eqref{f9int},
 for arbitrary $x\in \Bb{R} $ and $n \geqslant 1$ we obtain from \eqref{f3intth1},  \eqref{f2int} and \eqref{f2pinttheor1} that\,{\hyperlink{r7}{${}^{\ref*{case7}}$}}\hypertarget{br7}{}
\begin{align}\label{f23int}
    & \hspace{-0,2cm}
    \eurm{M}_{n} (x)\! =\!\dfrac{1}{ 4 \pi^{2} n }\!\!\! \!\!\!  \int\limits_{\gamma (-1,1)}\!\!\!\!\!
\frac{S^{{\tn{\triangle}}}_{n} \!\left(\dfrac{1}{1\!-\!\lambda (z)}\right) d z}{( x+z)^{2}}\!=\!
\dfrac{(-1)^{n}}{ 4 \pi^{2} n }\!\!\! \!\!\!  \int\limits_{\gamma (-1,1)}\!\!\!\!\!
\frac{S^{{\tn{\triangle}}}_{n} \!\big(\lambda (z\!+\!1)\big) -S^{{\tn{\triangle}}}_{n}(1)}{( x+z)^{2}} d z\,.\hspace{-0,1cm}
\end{align}

\noindent Here, in view of \eqref{f3bint}(d),(e) and \eqref{f11int},  we have
\begin{align*}
    & \hspace{-0,2cm}\lambda (z\!+\!1)=- \dfrac{\Theta_{2}(z)^{4}}{\Theta_{4}(z)^{4}} = \lambda_{\Bb{D}}\left(-{\rm{e}}^{\imag  \pi z}\right) , \quad  z\in \Bb{H} , \hspace{-0,1cm}
\end{align*}

\noindent where $\lambda_{\Bb{D}} \in {\rm{Hol}}(\Bb{D})$ with $\lambda_{\Bb{D}}(0)=0$.
Hence for arbitrary $n \geqslant 1$ and $z\in \Bb{H}$, we obtain
\begin{align}\nonumber
     (-1)^{n} &S^{{\tn{\triangle}}}_{n} \!\left(\dfrac{1}{\lambda (-1/z)}\right) \! =\!(-1)^{n} S^{{\tn{\triangle}}}_{n} \!\left(\dfrac{1}{1\!-\!\lambda (z)}\right)\\[0.2cm]  & \! =\!S^{{\tn{\triangle}}}_{n} \!\big(\lambda (z\!+\!1)\big) -S^{{\tn{\triangle}}}_{n}(1)
  = -S^{{\tn{\triangle}}}_{n}(1)\!-\!16s^{{\tn{\triangle}}}_{n, 1}{\rm{e}}^{\imag  \pi  z} \!+\!\sum\nolimits_{k\geqslant 2} {\eurm{b}}_{n,k}{\rm{e}}^{\imag  \pi k z} ,
\label{f23aaint}\end{align}

\noindent where   we have used the following notation for the Taylor series
(see \eqref{f5intlem3})
\begin{align}\label{f23baint}
     & \sum\limits_{k\geqslant 2} {\eurm{b}}_{n,k}u^{k}:=16s^{{\tn{\triangle}}}_{n, 1}u + \sum\limits_{k=1}^{n} s^{{\tn{\triangle}}}_{n, k}  \lambda_{\Bb{D}}\left(-u\right)^{k} \in {\rm{Hol}}(\Bb{D}) \ , \quad u\in\Bb{D} \, , \quad  n \geqslant 1\,.
\end{align}

\noindent  Besides, by using \eqref{f3bint}(b) and the change of variables $z^{\,\prime}=-1/z$  in the left-hand side integral of \eqref{f9int}, we find {\hyperlink{r8}{${}^{\ref*{case8}}$}}\hypertarget{br8}{}
\begin{align}\label{f24int}
    & \eurm{H}_{0} (x)= \eurm{H}_{0} (-x) \, , \ \ x\in \Bb{R}\,, \qquad
    \eurm{H}_{0} (-1/x)\!=\eurm{H}_{0} (x) x^{2}\, , \ \ x\in \Bb{R}\setminus\{0\} \,.
\end{align}

\vspace{0,2cm}
\subsection[\hspace{-0,31cm}. \hspace{0,11cm}Characteristic behavior at the vertices.]{\hspace{-0,11cm}{\bf{.}} Characteristic behavior at the vertices of the Schwarz quadrilateral.}We note that  there are only four distinct M\"{o}bius transformations (see \cite[p.\! 47]{con})
\begin{align*}
    &  z \ , \ \ -{1}\big/{z} \ , \ \  { (z -1)}\big/{(z + 1)}  \ , \ \
    {(1+ z)}\big/{(1- z)} \ ,
\end{align*}

\noindent
 which{\hyperlink{r11}{${}^{\ref*{case11}}$}}\hypertarget{br11}{} map $\fet $ onto itself    and interchange
 the vertices according to
\begin{align*}
\begin{array}{lrclr}
z\in   \fet,   &\ \      z \to 0 \ \  & \Leftrightarrow &  \ \
  -1/z \in\fet ,   &\quad    -1/z \to \infty  \ ,\\
 z \in \fet,   & \ \      z \to 1 \ \ & \Leftrightarrow & \ \ {(1 + z)}\big/{(1- z)}\in \fet ,  &\quad
   {(1 + z)}\big/{(1- z)} \to \infty \ , \\
z\in \fet,   &   \ \    z \to -1 \ \ & \Leftrightarrow & \ \  { (z -1)}\big/{(z + 1)} \in \fet,    &  \quad    { (z -1)}\big/{(z + 1)} \to \infty\,.
\end{array}
\end{align*}

The $n$-th Schwarz triangle polynomial $ S^{{\tn{\triangle}}}_{n}$ can be fully characterized by the behavior of the function  $ S^{{\tn{\triangle}}}_{n} (1/\lambda)$ from \eqref{f12int} at the vertices
of the Schwarz quadrilateral. Note that \eqref{f12int} for arbitrary
$n \!\geqslant\! 1$ and $z\!\in\! \Bb{H}$ can also be written as
\begin{align}\label{f12aint}
    &  S^{{\tn{\triangle}}}_{n}\left(\!\frac{1}{\lambda(z)}\! \right)\! =\! S^{{\tn{\triangle}}}_{n} \!\left(\dfrac{1}{1\!-\!\lambda (-1/z)}\right)\!=\!  {\rm{e}}^{{\fo{-i n \pi z}}}\! - \!\Delta_{n}^{S} (0) - \sum\nolimits_{k\geqslant 1} {\eurm{d}}_{n,k} {\rm{e}}^{\imag  \pi k z} ,
\end{align}

\noindent where  we use the Taylor series notation
\begin{align*}
    &  \sum\nolimits_{k\geqslant 1} {\eurm{d}}_{n,k} u^{k} := \Delta_{n}^{S} (\lambda_{\Bb{D}} (u)) - \Delta_{n}^{S} (0) \in {\rm{Hol}}(\Bb{D}) \ , \quad u\in\Bb{D} \, , \quad  n \geqslant 1\,.
\end{align*}

\noindent By  \eqref{f3wcint}(b) and  the same line of argument{\hyperlink{r10}{${}^{\ref*{case10}}$}}\hypertarget{br10}{} as in Section~\ref{genfun}, we get
\begin{align}\label{f12cint}
    & 1\! +\!  \sum\nolimits_{n=1}^{\infty} \Delta_{n}^{S}(0) u^{n}\!=\! \theta_{3}(-u)^{4},  \   u \!\in\! \Bb{D} \, ;    \quad \Delta_{n}^{S}(0) \!= \!(-1)^{n}   r_{4}(n)  ,  \   n \!\geqslant\! 1 \,,
\end{align}

\noindent and, consequently, $\Delta_{n}^{S}(0)\! =\! {\rm{O}}  (n^{2})$, as $n\!\to \! \infty$.
It  follows from \eqref{f2int}, from the identities
\begin{align}\label{f2qinttheor1}
    & \frac{1}{\lambda\left(1\!-\!{1}/{z}\right)}=\lambda(z\!+\!1) \ ,  \quad  \frac{1}{1-\lambda\left(1\!-\!{1}/{z}\right)}=\lambda(z)
      \ , \quad  z \in \Bb{H},
\end{align}

\noindent  (see \cite[p.\! 111]{cha})  and from \eqref{f11int} that
\begin{align*}
    &  S^{{\tn{\triangle}}}_{n}
    \left(\!\frac{1}{\lambda\left(1\!-\!{1}/{z}\right)}\! \right) \!=\!
S^{{\tn{\triangle}}}_{n}\Big(\!\lambda_{\Bb{D}}\left(-{\rm{e}}^{{\rm{i}} \pi z}\right)\! \Big)  \ , \
S^{{\tn{\triangle}}}_{n}
\left(\!\frac{1}{1-\lambda\left(1\!-\!{1}/{z}\right)}\! \right) \!=\!
S^{{\tn{\triangle}}}_{n}\Big(\!\lambda_{\Bb{D}}\left({\rm{e}}^{{\rm{i}} \pi z}\right)\!  \Big) \,,
\end{align*}

\noindent and thus for arbitrary $n \geqslant 1$ and $z\in \Bb{H}$, we obtain, by \eqref{f23baint},
\begin{align}\label{f25int}\hspace{-0,2cm}
    S^{{\tn{\triangle}}}_{n}\left(\!\frac{1}{\lambda\left(1\!-\!{1}/{z}\right)}\! \right)&\!=\! -16 s^{{\tn{\triangle}}}_{n, 1 } {\rm{e}}^{\imag  \pi  z } +
\sum\nolimits_{k\geqslant 2}{\eurm{b}}_{n, k}{\rm{e}}^{\imag  \pi k z }  , \hspace{-0,1cm} \\
\label{f25aint}\hspace{-0,2cm}
    S^{{\tn{\triangle}}}_{n}\left(\!\frac{1}{1-\lambda\left(1\!-\!{1}/{z}\right)}\! \right)&\!=\ 16 s^{{\tn{\triangle}}}_{n, 1 } {\rm{e}}^{\imag  \pi  z } +
\sum\nolimits_{k\geqslant 2}(-1)^{k}{\eurm{b}}_{n, k}{\rm{e}}^{\imag  \pi k z }  . \hspace{-0,1cm}
\end{align}

 When $z \in \fet$ tends to one of the vertices $\alpha \in \{1, -1, 0, \infty\}$ of $\fet $, we assume that $\im\, z \geqslant 1/2$
if $\alpha= \infty$ and  $\im\, z < 1/2$ if $\alpha \in \{1, -1, 0\}$, which imply that
$2\,\im\, z \geqslant |\re\, z|$ if $\alpha= \infty$, while
$\im\, z \geqslant |\re\, z - \alpha|$ if $\alpha \in \{ 0, 1, -1\}$. Hence it follows from \eqref{f23aaint}, \eqref{f12aint}, \eqref{f25int} and \eqref{f25aint} that for every $n \geqslant 1$ there exists ${\eusm{D}}_{n}\in \Bb{R}_{>0}$ such that
\begin{align}\label{f27int}
    &\hspace{0,2cm}\begin{array}{llll}
    {\rm{(a)}} &\ \   \left|\eurm{R}_{n}^{{\tn{\triangle}}} (z)\right|\!\leqslant\! {\eusm{D}}_{n} \exp\left( - \dfrac{  \pi}{2\, \im\, z} \right) ,   &
\ \  \left|z\pm 1\right| \leqslant 1/\sqrt{2}  \, ,      &\ \    z\in  \fet\,;\\[0,3cm]  {\rm{(b)}} &\ \
\left|\eurm{R}_{n}^{{\tn{\triangle}}} (z)\!-\! \exp \left( -\imag n \pi z\right)\right| \!\leqslant\! {\eusm{D}}_{n}  \ ,      & \ \   \im\, z \geqslant 1/2  \, ,
& \ \   z\in  \fet\,;\\[0,25cm]{\rm{(c)}} &\ \
\left|\eurm{R}_{n}^{{\tn{\triangle}}} (z)\right|\!\leqslant\! {\eusm{D}}_{n}  \ ,   &     \ \  \left|z\right| \leqslant 1/\sqrt{2} \, ,      &\ \    z\in  \fet\,;
     \end{array}
\end{align}

\noindent and
\begin{align}\label{f27aint}
    &\hspace{-0,3cm}\begin{array}{llll}
    {\rm{(a)}} &\ \   \left|\eurm{R}_{n}^{{\tn{\triangle}}} (-1/z)
    \right|\!\leqslant\! {\eusm{D}}_{n} \exp\left( - \dfrac{  \pi}{2\, \im\, z} \right) ,   &
\ \  \left|z\pm 1\right| \leqslant 1/\sqrt{2}  \, ,      &\ \    z\in  \fet\,;\\[0,3cm]  {\rm{(b)}} &\ \
\left|\eurm{R}_{n}^{{\tn{\triangle}}} (-1/z)\right|\!\leqslant\! {\eusm{D}}_{n}
  \ ,      & \ \   \im\, z \geqslant 1/2  \, ,
& \ \   z\in  \fet\,;\\[0,25cm]{\rm{(c)}} &\ \
\left|\eurm{R}_{n}^{{\tn{\triangle}}} (-1/z)\!-\! \exp \left( \imag n \pi /z\right)\right| \!\leqslant\! {\eusm{D}}_{n}  \ ,   &     \ \  \left|z\right| \leqslant 1/\sqrt{2} \, ,      &\ \    z\in  \fet\,;
     \end{array}\hspace{-0,1cm}
\end{align}

\noindent where for every $n \geqslant 1$ the functions\vspace{-0,1cm}
\begin{align}\label{f2bvslem1}
    &\hspace{-0,2cm}
    \begin{array}{lrcccl}
 {\rm{(a)}} \qquad &  \eurm{R}_{n}^{{\tn{\triangle}}} (z) \!  & :=  & \!S^{{\tn{\triangle}}}_{n}
\left(\!\dfrac{1}{\lambda(z)} \!\right)\! &  = & \!
\begin{displaystyle}\sum\nolimits_{k=1}^{n} \ \dfrac{s^{{\tn{\triangle}}}_{n, k}}{\lambda(z)^{k}}\ ;
\end{displaystyle}
 \\[0,4cm]
{\rm{(b)}} \qquad &   \eurm{R}_{n}^{{\tn{\triangle}}} (-1/z) \!  & =  & \!S^{{\tn{\triangle}}}_{n}
\left(\!\dfrac{1}{1\!-\!\lambda(z)}\! \right) \! &  = & \! \begin{displaystyle}\sum\nolimits_{k=1}^{n}\ \dfrac{s^{{\tn{\triangle}}}_{n, k}}{\big(1\!-\!\lambda(z)\big)^{k}}
 \  ;
\end{displaystyle}
    \end{array}
 \hspace{-0,1cm}
\end{align}

\vspace{-0,1cm}
\noindent are  holomorphic in $\Bb{H}$ and periodic with period $2$. We observe that in this notation the formulas \eqref{f9int} and \eqref{f23int} can be written  in the form
\begin{align}\label{f28int}
    &  \eurm{H}_{n} (x) \! = \dfrac{(-1)}{4 \pi^{2} n} \!\!\! \!\!\!  \int\limits_{\gamma (-1,1)}\!\!\!\!\!
\frac{ \eurm{R}_{n}^{{\tn{\triangle}}} (z) {\rm{d}} z}{(x+z)^{2}}\, ,
   \quad
 \eurm{M}_{n} (x)\! =\!\dfrac{1}{ 4 \pi^{2} n }\!\!\! \!\!\!  \int\limits_{\gamma (-1,1)}\!\!\!\!\!
\frac{\eurm{R}_{n}^{{\tn{\triangle}}} (-1/z) {\rm{d}} z}{( x+z)^{2}} \ , \ \
n \geqslant 1 \,,
\end{align}

\noindent where the behavior of the both integrands in $\fet$ is completely described by  \eqref{f27int} and \eqref{f27aint}. In particular, it follows
from \eqref{f27int}(a) and \eqref{f27aint}(a) that
the both integrals converge absolutely for all $x\in \Bb{R} $.
Furthermore, by writing \eqref{f2qinttheor1} in the form (see \cite[p.\! 111]{cha})\vspace{-0,2cm}
\begin{align}\label{f2aqinttheor1}
    &  \dfrac{1}{\lambda (z)} +  \dfrac{1}{\lambda (z+1)} =1  \ , \quad  z \in \Bb{H} \,,
\end{align}

\noindent we conclude from \eqref{f1intlem3} and \eqref{f2bvslem1} that for any $z\in \Bb{H}$ and $n \geqslant 1$ we have
\begin{align*}
&
\eurm{R}_{2n}^{{\tn{\triangle}}} (z+1)    =    \eurm{R}_{2n}^{{\tn{\triangle}}} (z) \,,\quad
  \eurm{R}_{2n-1}^{{\tn{\triangle}}} (z+1)   =   -\eurm{R}_{2n-1}^{{\tn{\triangle}}} (z) + 2\, r_{4}(2n-1) \,.
 \end{align*}

\vspace{0.2cm}
  The following Liouville-type property holds for the Schwarz quadrilateral $\fet$, where in the notation \eqref{f2intthA},
$\fet = -1/\fet$ and
\begin{align}\label{zbssqlem1}
    &\hspace{-0,3cm}
    \partial_{\Bb{H}}
\fet\! :=\! \Bb{H}\cap\partial
\fet\! =\! \gamma(1,\infty) \cup \gamma(-1,\infty)\cup
\gamma(0,1)\cup
 \gamma(0,-1).\hspace{-0,1cm}
\end{align}

\vspace{-0,1cm}
 \newcommand{\red}{\cite[p.\! 597]{bh2}}
\begin{lemmx}[\red]\hspace{-0,18cm}{\bf{.}}
\label{bssqlem1}\hypertarget{hbssqlem1}{}
Suppose that $f\!\in\! {\rm{Hol}}(\fet) \cap C (\Bb{H}\, \cap \,{\rm{clos}}\,\fet)$ satisfies \
  $f(z)=f(z+2)$ \  and  \
$ f (-1/z) = f \big(-1/(z+2) \big) \vphantom{\dfrac{A}{A}}$\ \  for each  $z \in -1 +  \imag
  \Bb{R}_{>0}$.
Suppose also that there exist nonnegative integers $n_{\infty}$,
 $n_{0}$, and $n_{ 1}$ such that
\begin{align*}
     &   \left|f (z)\right|
     =\oh\big(\exp\left(\pi\left( n_{\infty}
+1\right) |z|\right)\big)    \ ,   &     &  \hspace{-0cm}
z \in \fet,  \, \ \  z\to \infty \, ,  \\[0,2cm]   
  &
       \left|f (z)\right|=\oh\big( \exp \left(\pi
\left(n_{0}+1\right)|z|^{-1} \right)\big)  \ , &     &
\hspace{-0cm} z\in \fet, \,  \ \  z \to 0 \, ,  \\[0,2cm] 
       &          \left|f (z)\right|=
\oh\big(\exp \left(\pi\left( n_{ 1} +1\right)|z-\sigma|^{-1}\right)\big)
 \ , &     &     \hspace{-0cm} z\in\fet, \,  \ \  z \to\sigma \, , \ \sigma\in\{1,-1\}\,.
\end{align*}
Then there exists an algebraic polynomial $P$ of
degree $\le n_{\infty} + n_{0} + n_{1}$ such that
\begin{align*}
    &  f \big( z\big) = \frac{\vphantom{\frac{a}{b}}
P\big(\lambda(z)\big)}{\vphantom{\dfrac{a}{b}}
     \lambda(z)^{\,{\fo{n_{\infty}}}} \big(1-\lambda(z)\big)^{\,{\fo{n_{0}}}} }  \ ,
\quad  z \in \fet  \,.
\end{align*}
\end{lemmx}

\vspace{-0,2cm} The next lemma shows that the function  $  \eurm{R}_{n}^{{\tn{\triangle}}} = S^{{\tn{\triangle}}}_{n} (1/\lambda)$ is uniquely determined
 by much weaker asymptotic conditions   at the vertices of the Schwarz quadrilateral than those of \eqref{f27int}.
\begin{lemsectfour}\hspace{-0,18cm}{\bf{.}}\label{bvslem1} For a given positive integer $n$ there exists a unique function $ \eurm{R}_{n}^{{\tn{\triangle}}}$
in $C (\Bb{H}\, \cap \,{\rm{clos}}\,\fet) \cap {\rm{Hol}}(\fet)$
   such that for any $\sigma\in\{1,-1\}$ we have
  \begin{align}\label{f1bvslem1}
     &
\begin{array}{lll}
 {\rm{(a)}} \ &     \eurm{R}_{n}^{{\tn{\triangle}}}(z)
\!=\!\eurm{R}_{n}^{{\tn{\triangle}}}(z\!+\!2)  ,\ \
 \eurm{R}_{n}^{{\tn{\triangle}}} (-1/z)\!=\!\eurm{R}_{n}^{{\tn{\triangle}}} \big(-1/(z\!+\!2)\big)  ,
 &
\ \   z \!\in \!-1\! +\!  \imag  \Bb{R}_{>0} \, ,
  \\[0,2cm]         {\rm{(b)}} \  &
       \eurm{R}_{n}^{{\tn{\triangle}}} (z)\! =\! {\rm{O}}(1)  \, , &
 \ \  z\! \in \!\fet,  \   z \!\to\! 0  \, ,  \\[0,2cm]
              {\rm{(c)}}  \  &         \eurm{R}_{n}^{{\tn{\triangle}}} (z) \!=\!
\oh\big(1\big)
 \ , &
 \ \ z\!\in\!\fet,  \    z \!\to \!\sigma \,,
  \\[0,2cm]          {\rm{(d)}}  \  &
        \eurm{R}_{n}^{{\tn{\triangle}}} (z) \! = \!e^{{\fo{-{\rm{i}}\pi  n  z}} }\! + \!{\rm{O}}(1)  \, , &
 \ \ z\!\in\!\fet,   \,   z \!\to\! \infty  \,.
\end{array}
\end{align}

\end{lemsectfour}

\noindent Lemma~\ref{bvslem1} follows easily from Lemma~\ref{bssqlem1}.
Indeed,  the conditions \eqref{f1bvslem1} are satisfied if \eqref{f2bvslem1}(a) holds, as follows from \eqref{f27int}.
As for the uniqueness,
suppose that the function
$F_{n}$ has all the properties of  \eqref{f1bvslem1} (with $\eurm{R}_{n}^{{\tn{\triangle}}}$ replaced by $F_n$).
Then the difference
$F_{n} -  S^{{\tn{\triangle}}}_{n} (1/\lambda)$ meets the conditions of Lemma~\ref{bssqlem1}
with $n_{\infty} = n_{0} = n_{1}=0$. By Lemma~\ref{bssqlem1},  this  difference is a constant function, which
must equal
zero, by property  \eqref{f1bvslem1}(c). This  completes the proof of  Lemma~\ref{bvslem1}.

\vspace{0,1cm}
 The symmetrized functions
 $\eurm{R}^{\pm}_{n}(z)\!:=\!\eurm{R}_{n}^{{\tn{\triangle}}} (z)\pm \eurm{R}_{n}^{{\tn{\triangle}}} (-1/z)$
 introduced in \eqref{f10int} are characterized in a similar manner.
\begin{lemsectfour}\hspace{-0,18cm}{\bf{.}}\label{bvslem2} For every positive integer $n$,  there exist unique functions $ \eurm{R}_{n}^{+}$ and $ \eurm{R}_{n}^{-}$ in the set $C (\Bb{H}\, \cap \,{\rm{clos}}\,\fet) \cap {\rm{Hol}}(\fet)$
   such that
  \begin{align*}
     &
\begin{array}{lll}
{\rm{(a)}} \ \   &      \eurm{R}_{n}^{\pm}(z)
\!=\! \eurm{R}_{n}^{\pm}(z\!+\!2)  ,
 &
\ \   z \!\in \!-1\! +\!  \imag  \Bb{R}_{>0} \, ,
  \\[0,15cm]         {\rm{(b)}} \ \  &
       \eurm{R}_{n}^{\pm} (z)= \pm\, \eurm{R}_{n}^{\pm} (-1/z)  \, , &
 \ \  z \in \fet \, ,  \\[0,15cm]
              {\rm{(c)}} \ \  &          \eurm{R}_{n}^{\pm} (z) \!=\!
\oh\big(1\big)
 \ , &
 \ \ z\!\in\!\fet,  \  \,  z \!\to \! 1 \,,
  \\[0,15cm]         {\rm{(d)}} \ \  &
         \eurm{R}_{n}^{\pm} (z) \! = \!{\rm{e}}^{{\fo{-{\rm{i}}\pi  n  z}} }\! + \!{\rm{O}}(1)  \, , &
 \ \  z\!\in\!\fet,  \   z \!\to\! \infty  \,.
\end{array}
\end{align*}
\end{lemsectfour}

On a side note, we observe that the functional properties \eqref{f2qinttheor1}  imply
that for arbitrary $n \geqslant 1$ and $z\in\Bb{H}$,  we have
\begin{align*}
    &  \eurm{R}^{+ }_{n}\left(\!1\!-\!\dfrac{1}{z}\!\right)\!=\! S^{{\tn{\triangle}}}_{n}\left(\lambda (z)\right)\!+\! S^{{\tn{\triangle}}}_{n}\left(\lambda (z\!+\!1)\right) \, , \ \
\eurm{R}^{- }_{n}\left(\!1\!-\!\dfrac{1}{z}\!\right)\!= \!S^{{\tn{\triangle}}}_{n}\left(\lambda (z)\right)\!-\! S^{{\tn{\triangle}}}_{n}\left(\lambda (z\!+\!1)\right) \,.
\end{align*}

\section[\hspace{-0,30cm}. \hspace{0,11cm}The  biorthogonal sequence]{\hspace{-0,095cm}{\bf{.}} The biorthogonal sequence}\label{bs}

This section is devoted to the study of the functions $\eurm{H}_{0}$, $\eurm{H}_{n}$,
$\eurm{M}_{n}$, for $n\in \Bb{Z}_{\neq 0}$, given in Theorem~\ref{intth1}.

\subsection[\hspace{-0,31cm}. \hspace{0,11cm}Explicit expressions.]{\hspace{-0,11cm}{\bf{.}} Explicit expressions for the biorthogonal sequence.}\label{exprfunbs} \
It can easily be derived  from
\eqref{f28int}, \eqref{f17zint} and \eqref{f3zint}  that
$\eurm{H}_{n}, \eurm{M}_{n}\! \in \!{\rm{Hol}}(\Bb{H})\!\cap\! C(\Bb{H}\!\cup \! \Bb{R})$ with{\hyperlink{r35}{${}^{\ref*{case35}}$}}\hypertarget{br35}{}
\begin{align*}
 &   \left|\eurm{H}_{n} (z)\right|+ \left|\eurm{M}_{n} (z)\right| \leqslant
    \dfrac{8\,{\rm{e}}^{{\fo{2 \pi n  }}} }{|z-1|^{2} +|z+1|^{2}}\ , \quad z \in \Bb{H}\cup \Bb{R} \, , \
   n \geqslant 1 \,,
\end{align*}

\noindent from which, in view of \eqref{f2intth1}, we obtain
\begin{align}\label{f1bs}
    & \eurm{H}_{n} \, , \  \eurm{M}_{n}\in {\rm{H}}^1_{+}(\Bb{R}) \ , \quad
 \eurm{H}_{-n} \, , \  \eurm{M}_{-n}\in {\rm{H}}^1_{-}(\Bb{R}) \ , \quad n \geqslant 1\,.
\end{align}

The main results of  this section are the following assertions, the proofs of which are deferred to
Sections~\ref{pbslem1}.
\begin{lemsectfive}\hspace{-0,18cm}{\bf{.}}\label{bslem1}
    For every  $n\in \Bb{Z}_{\neq 0}$  the functions $\eurm{H}_{n} $ and $ \eurm{M}_{n}$ have the following properties:
\begin{align*}
     & {\rm{(1)}} \ \eurm{H}_{n}, \eurm{M}_{n}\! \in\! C^{\infty} (\Bb{R}) \cap {\rm{H}}^{1}_{+} (\Bb{R}),\  \mbox{if}\ n \!\geqslant\! 1, \quad
 \eurm{H}_{n}, \ \eurm{M}_{n}\! \in\! C^{\infty} (\Bb{R}) \cap {\rm{H}}^{1}_{-} (\Bb{R}), \ \mbox{if}\ n \!\leqslant\! -1  ;
\\
     & {\rm{(2)}} \  \forall m \in \Bb{Z}_{\geqslant 0}:\,\,\eurm{H}_{n}^{(m)} (x),  \eurm{M}_{n}^{(m)} (x) ={\rm{O}} \left(|x|^{-2-m}\right), \ \ \ |x| \to +\infty  \,   ;\   \\
     & {\rm{(3)}} \ \eurm{M}_{n}(x) =  \eurm{H}_{n} (-1/x)/x^{2}, \ \ \ x\!\in\! \Bb{R}\!\setminus\!\{0\};
 \\ &
{\rm{(4)}} \  2 \!\!\sum_{{\fo{k \in \Bb{Z}}}}   \eurm{H}_{n} (x + 2k) = {\rm{e}}^{{\fo{{\rm{i}} \pi n x}}}\ ,  \qquad   \sum_{{\fo{k \in \Bb{Z}}}}  \eurm{M}_{n} (x + 2 k) = 0\ , \quad x \in \Bb{R}\,.
\end{align*}
\end{lemsectfive}
 \begin{lemsectfive}\hspace{-0,15cm}{\bf{.}}\label{bslem2}
   The function $\eurm{H}_{0} $ has the following properties:
\begin{align*}
     & {\rm{(1)}} \ \eurm{H}_{\,0}\! \in\!  C^{\infty} (\Bb{R}) ,   \
\eurm{H}_{0} (x)\! = \! \eurm{H}_{0} (-x) , \  x\! \in\!  \Bb{R}\,,
 &  &
{\rm{(2)}} \  \eurm{H}_{\,0} (x) = {\rm{O}} \left(x^{-2}\right) , \  |x| \to +\infty
   \, , \\ &
{\rm{(3)}}\  2 \!\! \sum_{{\fo{k \in \Bb{Z}}}}   \eurm{H}_{\,0} (x + 2k) = 1\ , \  \ x \in \Bb{R}\,,  &    &
{\rm{(4)}} \   \eurm{H}_{0} (-1/x)\!=\! \eurm{H}_{0} (x) x^{2}\!\! , \  \ \
x\! \in \! \Bb{R}\! \setminus\! \{0\}.
\end{align*}
\end{lemsectfive}

\noindent It can easily be obtained from Lemmas~\ref{bslem1} and~\ref{bslem2} that for every $m \in \Bb{Z}$ we have
\begin{align}\label{f2bs}
    &   \int\limits_{\Bb{R}} {\rm{e}}^{{\fo{-{\rm{i}} \pi m x}}} \eurm{H}_{n} (x) {\rm{d}} x = \delta_{m,  n} \, ,   &     &
   \int\limits_{\Bb{R}} {\rm{e}}^{{\fo{ {{\rm{i}} \pi m}/{x} }}}  \eurm{H}_{n} (x) {\rm{d}} x = 0  \, ,   &     &    n \in \Bb{Z} \, , \   \\  & \int\limits_{\Bb{R}} {\rm{e}}^{{\fo{-{\rm{i}} \pi m x}}} \eurm{M}_{n} (x) {\rm{d}} x = 0 \, ,   &     &
 \int\limits_{\Bb{R}} {\rm{e}}^{{\fo{ {{\rm{i}} \pi m}/{x} }}}  \eurm{M}_{n} (x) {\rm{d}} x = \delta_{m,  n} \,,   &     &
 n \in \Bb{Z}_{\neq 0}  \, .
\label{f3bs}\end{align}

\noindent Indeed, Lemma~\ref{bslem1}(2,3) together with Lemma~\ref{bslem2}(2,3)  give that
\begin{align*}
    \int\nolimits_{\Bb{R}} {\rm{e}}^{{\fo{-{\rm{i}} \pi m x}}} \eurm{H}_{n} (x) {\rm{d}} x &= \int\nolimits_{-1}^{1}
{\rm{e}}^{{\fo{-{\rm{i}} \pi m x}}}\left(\sum\nolimits_{{\fo{k \in \Bb{Z}}}}   \eurm{H}_{n} (x + 2k)\right) {\rm{d}} x \\    &   =
\dfrac{1}{2} \int\nolimits_{-1}^{1}{\rm{e}}^{{\fo{-i \pi m x}}}{\rm{e}}^{{\fo{{\rm{i}} \pi n x}}}{\rm{d}} x = \delta_{n,  m}  \ , \quad
m \in \Bb{Z}  \, , \ n \in \Bb{Z}_{\geqslant 0}   \, ,
\end{align*}

\noindent while \eqref{f2intth1}, \eqref{f3intth1},  Lemma~\ref{bslem1}(2),(4) and Lemma~\ref{bslem2}(1),(4)
for any $n \geqslant 1$  and $m \in \Bb{Z}$  entail
\begin{align*}
   \int\nolimits_{\Bb{R}} {\rm{e}}^{{\fo{-{\rm{i}} \pi m x}}} \eurm{H}_{-n} (x) {\rm{d}} x  & = \int\nolimits_{\Bb{R}} {\rm{e}}^{{\fo{-{\rm{i}} \pi m x}}} \eurm{H}_{n} (-x) {\rm{d}} x =  \int\nolimits_{\Bb{R}} {\rm{e}}^{{\fo{{\rm{i}} \pi m x}}} \eurm{H}_{n} (x) {\rm{d}} x =  \delta_{-n,  m} \, ,  \\
 \int\limits_{\Bb{R}} {\rm{e}}^{{\fo{ {{\rm{i}} \pi m}/{x} }}}  \eurm{H}_{0} (x) {\rm{d}} x & = \int\nolimits_{\Bb{R}} {\rm{e}}^{{\fo{-{\rm{i}} \pi m x}}} \eurm{H}_{0} (x) {\rm{d}} x=  \delta_{0,  m} \, , \ \\
\int\limits_{\Bb{R}} {\rm{e}}^{{\fo{ {{\rm{i}} \pi m}/{x} }}}  \eurm{H}_{n} (x) {\rm{d}} x &= \int\nolimits_{\Bb{R}} {\rm{e}}^{{\fo{-{\rm{i}} \pi m x}}} \eurm{M}_{n} (x) {\rm{d}} x  \\    &    = \int\nolimits_{-1}^{1}
{\rm{e}}^{{\fo{-{\rm{i}} \pi m x}}}\left(\sum\nolimits_{{\fo{k \in \Bb{Z}}}}   \eurm{M}_{n} (x + 2k)\right) {\rm{d}} x    =0\,.
\end{align*}

\noindent This concludes the verification of \eqref{f2bs}. By applying the change of the variables $x = -1/x^{\,\prime}$ in \eqref{f2bs},
the relations \eqref{f3bs} are immediate, since the functions $\eurm{H}_{n},\eurm{M}_n$ are connected via the relation \eqref{f3intth1}.
Expressed differently, the following theorem holds, which contains part of the assertions made in Theorem~\ref{intth1}.
\begin{theorem}\hspace{-0,15cm}{\bf{.}}\label{bslem3}
The system consisting of the functions
$\eurm{H}_{0}$, $\eurm{H}_{n}$, $\eurm{M}_{n} \subset L^1 (\Bb{R}) \cap C^{\infty} (\Bb{R})$, $n\in \Bb{Z}_{\neq 0}$,  is biorthogonal to the hyperbolic trigonometric system $1$, $\exp ({\rm{i}}\pi n x)$, $\exp ({\rm{i}}\pi n/ x)$ $ \subset L^{\infty} (\Bb{R})$, $n \in \Bb{Z}_{\neq 0}$.
\end{theorem}

\begin{rem}\hspace{-0,15cm}{\bf{.}}\label{rem1}{\rm{
Let  $f \in L^{1}([0,2], {\rm{d}} x)$  be $2$-periodic on $\Bb{R}$. Then clearly $f \in L^1 (\Bb{R}, (1+x^2)^{-1}{\rm{d}} x)$ and,
in view of   Lemma~\ref{bslem1}(4) and Lemma~\ref{bslem2}(3),
\begin{align*}
    &  \eurm{h}_{n}(f)=  \int\limits_{\Bb{R}}  f (t) \eurm{H}_{-n} (t) {\rm{d}} t =
 \int\limits_{0}^{2}  f (t) \sum_{{\fo{k \in \Bb{Z}}}}   \eurm{H}_{-n} (t + 2k) {\rm{d}} t =
 (1/2) \int\limits_{0}^{2}  f (t)\,{\rm{e}}^{{\fo{ -{\rm{i}} \pi n t}}} {\rm{d}} t,  \\    &
 \eurm{m}_{m}(f)=  \int\limits_{\Bb{R}}  f (t) \eurm{M}_{-m} (t) {\rm{d}} t =
 \int\limits_{0}^{2}  f (t) \sum_{{\fo{k \in \Bb{Z}}}}   \eurm{M}_{-m} (t + 2k) {\rm{d}} t = 0, \  n \in \Bb{Z}\,,\  m \in \Bb{Z}_{\neq 0}\,.
\end{align*}

\noindent It follows that the hyperbolic  Fourier series \eqref{f03int} of the $2$-periodic function $f$
is the usual Fourier series
$\sum_{n\in \Bb{Z}} \eurm{h}_{n}(f) \exp ({\rm{i}} \pi n x)$ of $f$ on the interval $[0,2]$.  Next, assume that
$g (-1/x)= f (x)$, $x\in\Bb{R}_{\neq 0}$. Then, by \eqref{f3intth1} and the above identities for $f$,
\begin{align*}
    & \eurm{m}_{n}(g)=  \int\nolimits_{\Bb{R}}  g (t) \eurm{M}_{-n} (t) {\rm{d}} t =
    \int\nolimits_{\Bb{R}}  f (t) \eurm{H}_{-n} (t) {\rm{d}} t =  (1/2) \int\nolimits_{0}^{2} g (-1/t)\,{\rm{e}}^{{\fo{ -{\rm{i}} \pi n t}}} {\rm{d}} t,
    \\    &   \eurm{h}_{n}(g)=  \int\nolimits_{\Bb{R}}  g (t) \eurm{H}_{-n} (t) {\rm{d}} t =
    \int\nolimits_{\Bb{R}}  f (t) \eurm{M}_{-n} (t) {\rm{}}d t =0, \    \ \ n \in \Bb{Z}_{\neq 0}\,,
\end{align*}

\noindent while, as follows from  Lemma~\ref{bslem2}(4) and the above equalities,
\begin{align*}
    &  \eurm{h}_{0}(g)=  \int\nolimits_{\Bb{R}}  g (t) \eurm{H}_{0} (t) {\rm{d}} t =
    \int\nolimits_{\Bb{R}}  f (t) \eurm{H}_{0} (t) {\rm{d}} t = (1/2) \int\nolimits_{0}^{2} g (-1/t) {\rm{d}} t\,.
\end{align*}

\noindent It follows that the hyperbolic  Fourier series \eqref{f03int} of $g$ equals the usual Fourier series
$\eurm{h}_{0}(g) + \sum_{n\in \Bb{Z}_{\neq 0}} \eurm{m}_{n}(g) \exp (-i \pi n/ x)$ of $g (-1/x)$ on the interval $[0,2]$ expressed in the variable $-1/x$.

}}\end{rem}

\subsection[\hspace{-0,31cm}. \hspace{0,11cm}The generating function.]{\hspace{-0,11cm}{\bf{.}} The generating function for the biorthogonal sequence.}\label{genfunbs}
In the notation of \eqref{f2bvslem1},
it follows from \eqref{f22int} and \eqref{f9intthA} that
\begin{align}
     & \hspace{-0,1cm} i \pi \sum\limits_{n=1}^{\infty} \eurm{R}_{n}^{{\tn{\triangle}}} (z)\,
{\rm{e}}^{\imag  \pi n y} \!=\! \frac{
\lambda^{\,\prime}\left(y\right)  }{\lambda(z) \! -\! \lambda\left(y\right)} \ , \quad
\im\, y \!>\! 1\! \geqslant  \! \im\,  z ,  \
 z \!\in\! \gamma (-1,1) ,\hspace{-0,1cm}
\label{f1genfunbs}\end{align}

\noindent where for arbitrary $a \in (1, \im\, y)$, in accordance with \eqref{f17aint}, we have
\begin{align}\label{f17zint}\hspace{-0,2cm}
    &    \eurm{R}_{n}^{{\tn{\triangle}}} (z)\! =\!  \frac{1}{2 \pi i} \int\nolimits_{{\fo{ -1 \! +  \!  \imag   a }} }^{{\fo{ 1 \! + \!   \imag    a }} }   \frac{
   \lambda^{\,\prime}\left(\zeta\right)e^{{\fo{- n \pi i  \zeta }}}  }{\lambda(z)  - \lambda\left(\zeta\right)}\, {\rm{d }}\zeta\,, \quad \im\, y \!>\! a \!>\! 1 \!\geqslant\! \im\, z , \
   z \!\in\! \gamma (-1,1)  .\hspace{-0,1cm}
\end{align}

\noindent In view of Corollary~\ref{intcorol1}, we have $|\lambda^{\,\prime}(\zeta)|\leqslant 9$ and
$  4 |\lambda(z)-\lambda(\zeta)|\geqslant |\lambda(z)|\, |1-2\lambda(i a)|$ for all $\zeta \in [-1+ia, 1+ia]$ and
$z \!\in\! \gamma (-1,1)$. Hence we obtain
\begin{align*}
    &  \left|\sum\nolimits_{n=1}^{\infty} \eurm{R}_{n}^{{\tn{\triangle}}} (z)\, {\rm{e}}^{\imag  \pi n y} \right| \!\leqslant\!
 \sum\nolimits_{n=1}^{\infty} \dfrac{12\, {\rm{e}}^{ \pi n (a-\im\, y)} }{  |1\!-\!2\lambda(i a)|\, |\lambda(z)|}=
 \dfrac{12\, \left( {\rm{e}}^{ \pi ((\im\, y) - a)} -1\right)^{-1} }{  |1\!-\!2\lambda(i a)|\,  |\lambda(z)|}\,,
\end{align*}

\noindent for every $z \!\in\! \gamma (-1,1)$ and  $\im\, y \!>\! a \!>\! 1$. Together with \eqref{f3zint}, this inequality allows us to apply the Lebesgue dominated convergence theorem (see \cite[p.\! 161]{nat}) in order to derive from \eqref{f1genfunbs} and \eqref{f28int} that
\begin{align}\label{f3genfunbs}
    & \hspace{-0,2cm}  \sum\limits_{n \geqslant 1}
 {\rm{i}}\pi \, n \,\eurm{H}_{n} (x)\, {\rm{e}}^{{\fo{{\rm{i}}\pi  n  y}} }  =  \frac{1}{4 \pi^{2}} \hspace{-0,2cm}    \int\limits_{ \gamma (-1,1)} \hspace{-0,2cm}\dfrac{\lambda^{\,\prime}(y)}{\lambda(y)- \lambda(z)}\frac{{\rm{d}} z}{(x+z)^{2}}  \, , \quad \  \ \ \im\, y \!>\! 1\,,\hspace{-0,1cm}
\\    &  \hspace{-0,2cm}
\sum\limits_{n \geqslant 1}
 {\rm{i}}\pi \, n \,\eurm{M}_{n} (x)\, {\rm{e}}^{{\fo{{\rm{i}}\pi  n  y}} }  =  \frac{1}{4 \pi^{2}}  \hspace{-0,2cm}   \int\limits_{ \gamma (-1, 1)}\hspace{-0,2cm} \dfrac{\lambda^{\,\prime}(y)}{1\!-\!\lambda(y)\!-\! \lambda(z)}\frac{{\rm{d }}z}{(x+z)^{2}}  \, , \  \ \  \im\, y \!>\!  1\,,\hspace{-0,1cm}
\label{f4genfunbs}
\end{align}

\noindent where
both series  converge absolutely and uniformly over all $x\in \Bb{R}$.



\vspace{-0,15cm}
\section[\hspace{-0,30cm}. \hspace{0,11cm}Partitions  of the upper half-plane]{\hspace{-0,095cm}{\bf{.}} Partitions  of the upper half-plane}
\label{bsden}$\phantom{a}$\vspace{-0,7cm}

In the Poincar\'{e} half-plane model of hyperbolic space   the (generalized) semicircles $\gamma (a,b)$, $\gamma (a,\infty)$, $a,b \in \Bb{R}$,  $a\neq b$, (see \eqref{f1schw}) are the  hyperbolic straight lines connecting ideal points $a$ and $b$, or $a$ and $\infty$, respectively.
 Moreover, given an arbitrary collection of four (three) points $\!-\!\infty\! <\! a\! <\! b\!<\! c\!<\! d\!<\!+\!\infty$ ($-\!\infty\! < \!a \!<\! b\!<\! c\! <\!+\!\infty$)
a bounded set $A \subset \Bb{H}$ is called an {\it{ideal hyperbolic quadrilateral}}
({\it{triangle}}),
 with  four (three) {\it{vertices}} $a,b,c,d$ ($a,b,c$) if $\Bb{H} \cap \partial A$ is equal to the union of the
 three (two) {\it{lower arches}} $\gamma (a,b), \gamma (b,c),
\gamma (c,d)$  ($\gamma (a,b), \gamma (b,c)$) and the {\it{roof}} (upper arch)  $\gamma (a,d)$ \hspace{0,005cm}($\gamma (a,c)$).
 We will omit the word "ideal" in the sequel,
and note that clearly, there exists a unique  bounded closed set satisfying this definition for any given collection of  vertices lying on $\Bb{R}$.

Let $\Gamma$ be the  group of M\"obius transformations (called modular)\\[0,1cm] $\phantom{a}$\hspace{1cm}$\phi_{M}(z) :=(az+b)(cz+d)^{-1}$, $M=(\begin{smallmatrix} a & b \\ c & d \end{smallmatrix})$,  $ \ a, b, c, d \in\Bb{Z}$, $ad - bc=1$,\\[0,1cm]\noindent  on $\Bb{H}$ with the superposition as a group operation  (see \cite[p.\! 11]{cha}). The theta subgroup $\Gamma_{\vartheta}$ of $\Gamma$ is defined as a collection of all $\phi_{M}\!\in\!\Gamma$ with $M$, satisfying   either
$M\equiv (\begin{smallmatrix} 1 & 0 \\ 0 & 1 \end{smallmatrix}) ({\rm{mod}}\, 2)$ or
$M\equiv (\begin{smallmatrix} 0 & 1 \\ 1 & 0 \end{smallmatrix}) ({\rm{mod}}\, 2)$, while the subgroup  $\Gamma (2)$ of $\Gamma_{\vartheta}$ as all $\phi_{M}\!\in\!\Gamma_{\vartheta}$ with $M\equiv (\begin{smallmatrix} 1 & 0 \\ 0 & 1 \end{smallmatrix}) ({\rm{mod}}\, 2)$.

\begin{defsectsix}\hspace{-0,2cm}{\bf{.}} {\rm{For $(a,b) \in \{ (1,\infty),  (-1,1)\}$, we say that the set
 $\{\phi (\gamma (a,b))\}_{\phi \in \Gamma_{\vartheta}}$ {\rm{(}}called orbit of $\gamma (a,b)$ with respect to $\Gamma_{\vartheta}${\rm{)}}
generates a partition $\sqcup_{k\geqslant 1} A_{k}=\Bb{H}$ of $\Bb{H}$  if $\cup_{k\geqslant 1}\partial A_{k}=\cup_{\phi \in \Gamma_{\vartheta}}\phi (\gamma (a,b))$ and ${\rm{int}}A_{k} \neq \emptyset$ for all
$k\in \Bb{N}$.}}
\end{defsectsix}

 In this section,  we consider two partitions of $\Bb{H}$ generated by  $\{\phi (\gamma (1,\infty))\}_{\phi \in \Gamma_{\vartheta}}$ and  $\{\phi (\gamma (-1,1))\}_{\phi \in \Gamma_{\vartheta}}$.
 Both partitions arise from the need to analyze  integral operators with   kernel of the form\\
$\phantom{a}$ \hspace{3,5cm}
$G_{\lambda}(z, \zeta):=\lambda^{\,\prime}(z)/(\lambda(z)-\lambda(\zeta))$
\\
\noindent and integration over $\zeta \in \gamma(-1,1)$. An operator of this type (but with integration over $z$) was first studied  by Chibrikova in 1956 and later reproduced in Gakhov's monograph
\cite[p.\! 513, (52.3)]{gah} from 1966 when considering the Riemann boundary value problem in a fundamental domain with respect to a Fuchsian group of linear fractional transformations.  The importance of studying such integral operators was highlighted in the
recent work of the fourth  and fifth authors \cite{rad}, where the following crucial observation was made.
Given, e.g., a bounded function $f\in {\rm{Hol}} (\Bb{H}) $,
if we integrate of $ f(\zeta)G_{\lambda}(z, \zeta)$ with respect to $\zeta$ along the semicircle
$\gamma(-1,1)$, we obtain an analytic function in the domain
$\Bb{H}\setminus (2 \Bb{Z} + \overline{\Bb{D}})$ which can be analytically extended to all of
$\Bb{H}$, while preserving relevant growth control.
 The main goal of introducing the above mentioned  partitions of the upper half-plane is to obtain explicit formulas for such an analytic extension,  see Section~\ref{contgen} below. The results of Sections~\ref{bsden} and~\ref{contgen} are expressed in the language of the even-integer continued fractions which is the suggestion of the first author. This elaborates on the approach of the second and the third authors \cite{hed} based on the
 connection  between   even-integer continued fractions and the even Gauss map, see, e.g., \cite{pio} and  \cite{lop}.

 \subsection[\hspace{-0,31cm}. \hspace{0,11cm}Even Gauss map in the upper half-plane.]{\hspace{-0,11cm}{\bf{.}} Even Gauss map in the upper half-plane.}\label{bsdentppre} \
For arbitrary $N \in \Bb{N}$ and  $n_{1}, n_{2}, ..., n_{N}\in \Bb{Z}_{\neq 0}$ let $\phi_{{\fo{\hspace{0,005cm}n_{N}, ..., n_{1}}}}$ be an even-integer continued fraction of the form
(see \cite{sho} and \eqref{f6zpf8contgen})

\begin{align}\label{f1bsden}
    &   \phi_{{\fo{\hspace{0,005cm}n_{N}, ..., n_{1}}}}(z):=
  {\fo{\dfrac{1}{2n_{N}\! -\! \dfrac{1}{ \begin{subarray}{l}
\begin{displaystyle} \vphantom{A^{A}}
2n_{N-1}\!- \!\end{displaystyle}\\[-0.4cm]
\hphantom{2n_{k-1} 11 -} \ddots\\[-0.4cm]
\hphantom{2n_{k-1} 11 - \ddots} \begin{displaystyle} - \!\end{displaystyle} \dfrac{1}{2 n_{2} \!- \! \dfrac{1}{
\vphantom{A^{A}}2 n_{ 1} -z }}\end{subarray} }}}}
  \ ,  \ z \in [-1,1]\cup \left(\Bb{C}\setminus\Bb{R}\right) \,.
  \end{align}

\vspace{0,2cm}
\noindent The theta subgroup $\Gamma_{\vartheta}$ is generated by the M\"obius transformations $z\mapsto z+2$ and $z\mapsto-1/z$ (see \cite[p.\! 112]{cha}). If we write, for $z\in \Bb{H}$, $\Bb{Z}_{\neq 0}^{k} := \left(\Bb{Z}_{\neq 0}\right)^{k}$, $k \in \Bb{N}$,
\begin{align}\label{f1zbsden}\hspace{-0,2cm}
\eufm{n}\!:=\!(n_{N}, ..., n_{1})\in \Bb{Z}_{\neq 0}^{\hspace{0,02cm}\Bb{N}_{\hspace{-0,02cm}\eurm{f}}}\!:=\!\sqcup_{k\geqslant 1} \Bb{Z}_{\neq 0}^{k} , \ \phi_{\eufm{n}}\!:= \!\phi_{{\fo{n_{N}, ..., n_{1}}}}  \, , \ \phi_{0} (z):= z ,
\end{align}

\noindent
 we can represent the theta subgroup as{\hyperlink{r17}{${}^{\ref*{case17}}$}}\hypertarget{br17}{}
\begin{align}\label{f4contgen}
    & \hspace{-0,3cm} \left\{ \, \phi (z)\,\right\}_{{\fo{\phi\! \in\! \Gamma_{\vartheta} }}} \! =\!
\left\{  \,  \phi_{\eufm{n}}(z), \  -1/\phi_{\eufm{n}}(z) , \  \phi_{\eufm{n}}(-1/z)  , \  -1/\phi_{\eufm{n}}(-1/z)\,\right\}_{\,{\fo{\eufm{n} \!\in \!\Bb{Z}_{\neq 0}^{\hspace{0,02cm}\Bb{N}_{\hspace{-0,02cm}\eurm{f}}}\cup\!\{0\} }}}\,.\hspace{-0,1cm}
\end{align}

\noindent Here, we see that
\begin{align}\label{f5contgen}
    & \left\{\begin{array}{rlll}
      -1/\phi_{{\fo{\hspace{0,005cm}n_{1}}}}(z)  &    =   &   - 2 n_{1} + z \,,   &\ \mbox{if} \ N=1 \,,    \\
       -1/\phi_{{\fo{\hspace{0,005cm}n_{N}, ..., n_{1}}}}(z)  &    =  &    - 2 n_{N} +
    \phi_{{\fo{\hspace{0,005cm}n_{N-1}, ..., n_{1}}}}(z)\,, &\ \mbox{if}\ N\geqslant 2\,,
      \end{array}\right.
\end{align}

\noindent for every $z \in [-1,1]\cup \left(\Bb{C}\setminus\Bb{R}\right)$,  $N \in \Bb{N}$ and $(n_{N}, ..., n_{1})\in \Bb{Z}_{\neq 0}^{N}$. Hence, if $A \subset \Bb{H}$
is invariant under the inversion $z\mapsto-1/z$, i.e., if $A = - 1/A:= \{-1/z \, | \, z \in A \}$, then its orbit with respect to  $\Gamma_{\vartheta}$
can be written in the form
\begin{align}\label{f6contgen}
    & \hspace{-0,3cm}  \left\{ \, \phi (A)\,\right\}_{{\fo{\phi\! \in\! \Gamma_{\vartheta} }}} \! =\!
     \left\{ \, 2n_{0} + \phi_{\eufm{n}} (A)\,\right\}_{\,
           {\fo{n_{0} }}  \, {\fo{\in}}  \,    {\fo{\Bb{Z}}}  \, , \    {\fo{\eufm{n} }}  \,{\fo{\in}}  \,   {\fo{ \Bb{Z}_{\neq 0}^{\hspace{0,02cm}\Bb{N}_{\hspace{-0,02cm}\eurm{f}}}\cup\!\{0\} }} } \ , \quad  A \subset \Bb{H} \, , \  A = - 1/A\,.\hspace{-0,1cm}
\end{align}

\noindent This suggests the introduction of the following subset of  $\Gamma_{\vartheta}$,
\begin{align}\label{f7contgen}
    &   \Gamma_{\vartheta}^{\hspace{0,015cm}{\tn{||}}}  := \left\{\,  \phi_{\eufm{n}}
      \, \right\}_{{\fo{
    \eufm{n} \!\in \!\Bb{Z}_{\neq 0}^{\hspace{0,02cm}\Bb{N}_{\hspace{-0,02cm}\eurm{f}}} }}} \,,\quad   (a_{\phi}z+b_{\phi})(c_{\phi}z+d_{\phi})^{-1}:=  \phi(z) \,,\quad  \phi \in \!\Gamma_{\vartheta}^{\hspace{0,015cm}{\tn{||}}}\,,
\end{align}

\noindent $a_{\phi}, b_{\phi}, c_{\phi}, d_{\phi} \!\in\!\Bb{Z} $,  with the basic properties  (see Sections~\ref{pf8contgen}  and~\ref{not})
\begin{align}\label{f8contgen}\hspace{-0,2cm}
    & |a_{\phi}|\!<\! |b_{\phi}|\!<\!|d_{\phi}|,\ |a_{\phi}|\!<\! |c_{\phi}|\!<\!|d_{\phi}|,\ \
     \phi \left(\Bb{H}_{|\re|\leqslant 1}\right)\! \subset\! \Bb{H}_{|\re|<1}\! \setminus\!\fet \, , \ \  \phi \!\in \!\Gamma_{\vartheta}^{\hspace{0,015cm}{\tn{||}}}  \,.
\end{align}

\vspace{0,2cm}
Let $\{x\}$ and $\lfloor{x}\rfloor$ denote the usual fractional and integer parts of $x \in \Bb{R}$, respectively, and put (see \cite[pp.\! 599, 600]{bh2})
\begin{align}\label{f0bsdentp}
    &  \mathcal{F}^{\,{\tn{||}}}_{{\tn{\square}}}:= \mathcal{F}_{{\tn{\square}}}\sqcup
(- 1 +  \mathrm{i} \hspace{0,015cm}   \Bb{R}_{>0})\sqcup( 1 +  \mathrm{i}
\hspace{0,015cm}   \Bb{R}_{>0}) \, ,  &     &    \mathcal{F}^{\infty}_{{\tn{\square}}}:=
\cup_{\,{\fo{m \in \Bb{Z}}}} \left(2m +
\mathcal{F}^{\,{\tn{||}}}_{{\tn{\square}}}\right).
\end{align}

\noindent  For arbitrary $x\! \in\! \Bb{R}$ we define its {\it{even integer part}}
$\rceil{x}\lceil_{\! 2}\,\in\! 2 \Bb{Z}$   as the average of the endpoints of the interval that the point $x$ belongs to
in the following interval partition of the real axis,\vspace{-0,1cm}
\begin{align}\label{f1defegm}
   \Bb{R} = (-1,1) \  \sqcup \   \bigsqcup\limits_{{\fo{ \ m \geqslant 1}}} \ \bigg(\Big(-2m-1, -2m+1 \Big] \   \sqcup  \ \Big[2m-1, 2m+1 \Big)\bigg),
   \end{align}

\vspace{-0,1cm}
\noindent
and we define its {\it{even fractional part}}
$\{x\}^{{\tn{\rceil\hspace{-0,015cm}\lceil}}}_{2} \in [-1,1]$ to be $\{x\}^{{\tn{\rceil\hspace{-0,015cm}\lceil}}}_{2} := x - \rceil{x}\lceil_{\!2}$. As such, they have the properties $\{x\}^{{\tn{\rceil\hspace{-0,015cm}\lceil}}}_{2}\! =\! (-1+ 2 \, \{ (1+|x|)/ 2 \})\,  {\rm{sign}} (x)$,
 $\{-x\}^{{\tn{\rceil\hspace{-0,015cm}\lceil}}}_{2} \!=\! - \{x\}^{{\tn{\rceil\hspace{-0,015cm}\lceil}}}_{2}$,
$ \rceil{x}\lceil_{\!2}=\!  2 \, \lfloor{(1\!+\!|x|)/ 2}\rfloor\,  {\rm{sign}} (x)$, $\rceil{-x}\lceil_{\!2}=\!-\rceil{x}\lceil_{\!2}$ for each $x\! \in\! \Bb{R}$, and
\begin{align}\label{f1ybsden}
    &\hspace{-0,2cm} \rceil{-2 n +  {\rm{sign}} (n)}\lceil_{\!2}\,=-2 n \,, \quad
 \{\hspace{0.07cm} -2 n +  {\rm{sign}} (n)\hspace{0.07cm}\}^{{\tn{\rceil\hspace{-0,015cm}\lceil}}}_{2}\! = \!{\rm{sign}} (n), \quad n  \in \Bb{Z} \,.\hspace{-0,1cm}
\end{align}

\begin{defsectsix}\hspace{-0,17cm}{\bf{.}}\label{contgendef1}
We define the complex analogue   $\Bb{G}_{2} : \Bb{H}_{|\re|\leqslant 1} \to \Bb{H}_{|\re|\leqslant 1}$ of the even Gauss map
$G^{{\tn{\rceil\hspace{-0,015cm}\lceil}}}_{2} : [-1,1] \to [-1,1]$,
$G^{{\tn{\rceil\hspace{-0,015cm}\lceil}}}_{2} (0) := 0$,
$G^{{\tn{\rceil\hspace{-0,015cm}\lceil}}}_{2} (x) := \{-1/x\}^{{\tn{\rceil\hspace{-0,015cm}\lceil}}}_{2}$,
$x\in [-1,1]\setminus\{0\}$, associated with the  even fractional part $\{x\}^{{\tn{\rceil\hspace{-0,015cm}\lceil}}}_{2}$,
 as follows
\begin{align}\label{f2contgendef1}
&
 \Bb{G}_{2} (z) :=
         \     \left\{\re \left(- \dfrac{1}{z}\right)\right\}^{{\tn{\rceil\hspace{-0,015cm}\lceil}}}_{2} + {\rm{i}} \,\im \left(- \dfrac{1}{z}\right)\,,
\ \  \ \ z \in \Bb{H}_{|\re|\leqslant 1} \, .
\end{align}

\noindent For $ z \in \Bb{H}_{|\re|\leqslant 1}$, we have that
\begin{align}\label{f3contgendef1}
    &  \Bb{G}_{2} (z) =  - \dfrac{1}{z} -
\left\rceil{\re\left(- \dfrac{1}{z}\right)}\right\lceil_{2}
\in
\Bb{H}_{|\re|\leqslant 1}
\ , \qquad \left\rceil{\re\left(- \dfrac{1}{z}\right)}\right\lceil_{2}
 \in
 2\, \Bb{Z}_{\neq 0} \ .
\end{align}

\end{defsectsix}

 We note that
\begin{align}\label{f9bsdentp}
     & \vspace{-0,3cm}
\Bb{G}_{2} (z) = -1/z  \ , \ z \in \Bb{H}_{|\re|\leqslant 1}
\ \ \Longleftrightarrow\ \  z \in \Bb{H}_{|\re|\leqslant 1} \cap \left(-1 \big/\Bb{H}_{|\re|< 1}\right) = \mathcal{F}^{\,{\tn{||}}}_{{\tn{\square}}},
\vspace{-0,1cm}
\end{align}

\noindent
and observe that it follows from \eqref{f3contgendef1} and \eqref{f1defegm}
that{\hyperlink{r19}{${}^{\ref*{case19}}$}}\hypertarget{br19}{}
\begin{align}\label{f1contgendef1}
    &
\begin{array}{lrcll}{\rm{(a)}} \   &
\Bb{G}_{2} (\gamma(\sigma,0)) & = & \gamma(\sigma,\infty)\,,  & \quad \sigma\in \{1,-1\}\,,\\[0,2cm]
{\rm{(b)}} \   &
\Bb{G}_{2} (\gamma(\sigma,\infty)) & = & \gamma(-\sigma,0)\,,  & \quad \sigma\in \{1,-1\}\,,\\[0,1cm]
{\rm{(c)}} \   &
  \Bb{G}_{2} \Big( \phi_{{\fo{n}}} \big(\hspace{0.07cm}\gamma(\sigma_{n}, \infty)\hspace{0.07cm}\big)\Big) & = &
\gamma(\sigma_{n}, \infty) \, ,  & \quad  \sigma_{n}:={\rm{sign}} (n) \, , \  n  \in \Bb{Z}_{\neq 0}  \,.
\end{array}
\end{align}

\noindent
Estimating from above the modulus of $z$  lying on the union of semicircles $\gamma (1,0)\cup \gamma (-1,0)$
we obtain from the identity  $\im (-1/z) = (\im\, z)/|z|^{2} $, $z\in \Bb{H}$, that
\begin{align}\label{f17contgen}
    & \hspace{-0,2cm} \im\,  \Bb{G}_{2} \left(z\right) \geqslant \dfrac{\im\, z}{(1/2) + \sqrt{(1/4) - (\im\, z)^{2}}}  \ , \quad  z \in \Bb{H}_{|\re|\leqslant 1} \setminus   \mathcal{F}^{\,{\tn{||}}}_{{\tn{\square}}} \subset
 \Bb{H}_{\im\leqslant 1/2}
\, .\hspace{-0,1cm}
\end{align}

\noindent
For arbitrary $N \in \Bb{N}$ and  $n_{1}, n_{2}, ..., n_{N}\in \Bb{Z}_{\neq 0}$ it follows from  \eqref{f5contgen} and \eqref{f8contgen} that
\begin{align}\label{f10bsdentp}
    & \left\{\begin{array}{rlll}
     \re \left(-1/\phi_{{\fo{\hspace{0,005cm}n_{1}}}}(z)\right)  &    \in   &   - 2 n_{1} + (-1,1) \,,   &\ \mbox{if} \ N=1 \,, \
z \! \in \! \Bb{H}_{|\re|< 1}  \,,    \\
    \re  \left( -1/\phi_{{\fo{\hspace{0,005cm}n_{N}, ..., n_{1}}}}(z) \right) &    \in  &    - 2 n_{N} +
    (-1,1)\,, &\ \mbox{if}\ N\geqslant 2\,, \ z \! \in \! \Bb{H}_{|\re|\leqslant 1}  \,,
      \end{array}\right.
\end{align}

\noindent and hence we see from \eqref{f3contgendef1}  and \eqref{f5contgen}
that{\hyperlink{r20}{${}^{\ref*{case20}}$}}\hypertarget{br20}{}
\begin{align}\label{f15contgen}
    & \hspace{-0,3cm}  \begin{array}{lrlllc}{\rm{(a)}} \   &
 \Bb{G}^{N}_{2} \big( \phi_{{\fo{n_{N}, ..., n_{1}}}} (z)\big)  &\! = \!  &
z  \,,   &\  \mbox{if} \ N \! \geqslant\!  1\,,   &\  \ z \! \in \! \Bb{H}_{|\re|< 1}  \,;\\[0,1cm]
{\rm{(b)}} \   &
 \Bb{G}^{k}_{2} \big( \phi_{{\fo{n_{N}, ..., n_{1}}}} (z)\big)  & \! = \!  & \phi_{{\fo{\hspace{0,005cm}n_{N-k}, ..., n_{1}}}}(z)\,, &\  \mbox{if} \ 1 \! \leqslant\!  k \! \leqslant\!  N\! -\! 1 \,,  &\  \  z \! \in \! \Bb{H}_{|\re|\leqslant 1} \,.
      \end{array}
\hspace{-0,1cm}
\end{align}

\subsection[\hspace{-0,31cm}. \hspace{0,11cm}The Schwarz partition of the upper half-plane.]{\hspace{-0,11cm}{\bf{.}} The Schwarz partition of the upper half-plane.}\label{bsdentp} \ We introduce the notation
\begin{align*} & \hspace{-0,24cm}
 \fet^{\hspace{0,01cm}\phi}\!:=\!\fet^{\hspace{0,01cm}\eufm{n}}\!:= \! \fet^{\hspace{0,05cm}n_{N}, ..., n_{1}} \!:= \! \phi_{{\fo{n_{N}, ..., n_{1}}}}\big( \fet\big), \  \  \eufm{n}\!=\!(n_{N}, ..., n_{1})\!\in\! \Bb{Z}_{\neq 0}^{N}, \ \  N \!\in\! \Bb{N}\,,
\end{align*}

\noindent where   $\phi := \phi_{\eufm{n}}$ is as in \eqref{f1bsden}.
We associate with  an arbitrary  $N$-tuple $\eufm{n}\!=\!(n_{N}, ..., n_{1})\!\in\!
\Bb{Z}_{\neq 0}^{N}$, $N \!\in\! \Bb{N}$,  and, correspondingly,  with each   $\phi \!=\! \phi_{\eufm{n}}\!\in \!\Gamma_{\vartheta}^{\hspace{0,015cm}{\tn{||}}} $,  its~"sign"  \begin{align}\label{f9contgen}
    & \sigma_{\phi}:=  \sigma_{\eufm{n}}:=\sigma_{_{{\fo{\hspace{0,005cm}n_{N}, ..., n_{1}}}}}:=\sigma_{n_{1}} := {\rm{sign}} (n_{1}) \in \{+1, -1\}  \ , \quad \phi = \phi_{\eufm{n}} \in \Gamma_{\vartheta}^{\hspace{0,015cm}{\tn{||}}} \,.
\end{align}

\noindent Depending on the sign of the transformation $ \phi = \phi_{\eufm{n}}\in \Gamma_{\vartheta}^{\hspace{0,015cm}{\tn{||}}}$, $\eufm{n}\!=\!(n_{N}, ..., n_{1}) \in  \Bb{Z}_{\neq 0}^{\hspace{0,02cm}\Bb{N}_{\hspace{-0,02cm}\eurm{f}}}$, which we apply to the
Schwarz quadrilateral $\fet $,
we add to  $\fet $ one of the rays $\pm 1 +  \imag   \Bb{R}_{>0}$,
\begin{align}\label{f10contgen}
    &    \fet^{\,{\tn{|}}\sigma_{\phi}}:=\fet^{\,{\tn{|}}\sigma_{\eufm{n}}}
    :=
    \fet^{\,{\tn{|}}\sigma_{n_{1}}}:=
    \fet \sqcup \left( \sigma_{n_{1}} \!+\! i \Bb{R}_{>0} \right) \, ,
\end{align}

\noindent
and denote the resulting image as follows
  \begin{align}
  \fetr^{\hspace{0,01cm}\phi}:=\fetr^{\hspace{0,01cm}\eufm{n}}:= \fetr^{\hspace{0,05cm}n_{N}, ..., n_{1}} :=  \phi_{{\fo{n_{N}, ..., n_{1}}}}\big( \fet^{\,{\tn{|}}\sigma_{n_{1}}}\big) \subset \Bb{H}_{|\re|<1} \setminus\fet\,.
        \label{f2bsden}
 \end{align}

\noindent For each $\eufm{n}\in  \Bb{Z}_{\neq 0}^{\hspace{0,02cm}\Bb{N}_{\hspace{-0,02cm}\eurm{f}}}$
the open set  $\fet^{\hspace{0,01cm}\eufm{n}}$ is a  hyperbolic quadrilateral
with four vertices $\{\,\phi_{\eufm{n}}(x)\, | \, x\in \{-1,0,1,\infty\}\,\} \subset [-1,1]$  whose position is  completely determined by the value of
 $\sigma_{\eufm{n}}$, as follows from the  relationships (see  Section~\ref{pf8contgen})
 \vspace{-0,2cm}
 \begin{align}\label{f3bsden}
    & \begin{array}{l}
     {\rm{(a)}}\   \phi_{\eufm{n}}(- 1) < \phi_{\eufm{n}}(0) < \phi_{\eufm{n}}(1) \ ,      \quad{\rm{(b)}}\
    \sigma_{\eufm{n}} \phi_{\eufm{n}}(\infty) < \sigma_{\eufm{n}} \phi_{\eufm{n}}(-\sigma_{\eufm{n}}) \, , \\[0,2cm]
    {\rm{(c)}}\ \   {\rm{sign}} \big(\phi_{\eufm{n}}(- 1)\big)={\rm{sign}} \big(\phi_{\eufm{n}}(0)\big)={\rm{sign}} \big(\phi_{\eufm{n}}( 1)\big)\, ,
          \quad  \eufm{n}\in  \Bb{Z}_{\neq 0}^{\hspace{0,02cm}\Bb{N}_{\hspace{-0,02cm}\eurm{f}}}\,.
      \end{array}
 \end{align}

\vspace{-0,2cm}
\noindent Moreover,  for arbitrary $\eufm{n}\in  \Bb{Z}_{\neq 0}^{\hspace{0,02cm}\Bb{N}_{\hspace{-0,02cm}\eurm{f}}}$ the properties \eqref{f3bsden} and ($\gamma_{0}(a, b)\!:= \! \gamma (a, b)$)
\begin{align}\label{12bsdenac}
    &\hspace{-0,2cm}  \gamma_{\eufm{n}}(a, b)\!:= \!\phi_{\eufm{n}} \big(\gamma (a, b) \big)\!=\!
  \gamma \big(\phi_{\eufm{n}} (a), \phi_{\eufm{n}} (b)\big)\, ,  \ \    a \in  [-1,1] \, , \  b \in  [-1,1] \cup \{\infty\} \, ,
\end{align}

\noindent $a\neq b$,
imply  that  $ \Bb{H} \cap \partial
 \fet^{\hspace{0,01cm}\eufm{n}} $ is a hyperbolic polygon which consists of the three lower arches  $\gamma_{\eufm{n}}( - \sigma_{\eufm{n}}, \infty)$, $\gamma_{\eufm{n}}(-1, 0)$,
$\gamma_{\eufm{n}} (0, 1)$,
which do not belong to  $\fetr^{\hspace{0,01cm}\eufm{n}}$,  and of the  roof
$\gamma_{\eufm{n}} (\sigma_{\eufm{n}}, \infty)$, which is in $\fetr^{\hspace{0,01cm}\eufm{n}}$(see Figure~\ref{figure:image}).  Obviously, $\fet^{\hspace{0,01cm}\eufm{n}} = {\rm{int}} (\fetr^{\hspace{0,01cm}\eufm{n}})$\, and\, $\partial
 \fet^{\hspace{0,01cm}\eufm{n}} = \partial
 \fetr^{\hspace{0,01cm}\eufm{n}}$.

\begin{lemsectsix}\hspace{-0,15cm}{\bf{.}}\label{bsdentplem1}  Let $N \!\in\! \Bb{N}$ and $\eufm{n}:=(n_{N}, ..., n_{1})\!\in\!  \Bb{Z}_{\neq 0}^{N}$ {\rm{(}}see also \eqref{f1bsden}, \eqref{f2bsden}{\rm{)}}.

\vspace{0,1cm}
\noindent
{\rm{{\bf{(a)}}}} For arbitrary $\eufm{n}, \eufm{m}\in \Bb{Z}_{\neq 0}^{\hspace{0,02cm}\Bb{N}_{\hspace{-0,02cm}\eurm{f}}}\!:=\!\cup_{k\geqslant 1} \Bb{Z}_{\neq 0}^{k} $, $\eufm{n}\neq \eufm{m}$,  we have
$\fetr^{\hspace{0,01cm}\eufm{n}} \cap \fetr^{\hspace{0,01cm}\eufm{m}} = \emptyset$.

\vspace{0,1cm}
\noindent
{\rm{{\bf{(b)}}}} The lower arches\vspace{-0,1cm}
\begin{align*}
    & \phi_{\eufm{n}} \big( \gamma (-1, 0) \big)\,, \quad \phi_{\eufm{n}} \big( \gamma (0, 1) \big)\,,  \quad\phi_{\eufm{n}} \big( \gamma(- {\rm{sign}} (n_{1}), \infty)\big)\,,
\end{align*}

\vspace{-0,1cm}
\noindent of $\fet^{\hspace{0,05cm}n_{N}, ..., n_{1}}$  are the roofs  of the respective quadrilaterals\vspace{-0,1cm}
\begin{align*}
    &  \fet^{\hspace{0,05cm}n_{N}, ..., n_{1},-1} \, , \quad  \fet^{\hspace{0,05cm}n_{N}, ..., n_{1}, \,1} \, , \quad \fet^{\hspace{0,05cm}n_{N}, ..., n_{1}+ {\rm{sign}} (n_{1}) } \, .
\end{align*}

\vspace{-0,1cm}
\noindent {\rm{{\bf{(c)}}}} The roofs of $\fet^{\hspace{0,05cm}1}$ and $\fet^{\hspace{0,05cm}-1}$ are  $\gamma (0, 1)$ and $\gamma (-1, 0)$, respectively.

\vspace{0,2cm}
\noindent {\rm{{\bf{(d)}}}} The roof of $\fet^{\hspace{0,05cm}n_{N}, ..., n_{1}}$ is the lower arch of\vspace{-0,1cm}
\begin{align*}
    &
\begin{array}{lr}
 \fet^{\hspace{0,05cm}n_{N}, ..., n_{1}-{\rm{sign}} (n_{1})}  &, \  \mbox{if} \ \ n_{1}\!\neq\! {\rm{sign}} (n_{1}) \ , \quad  N \geqslant 1 \,,      \\[0,1cm]
  \fet^{\hspace{0,05cm}n_{N}, ..., n_{2}} &,  \  \mbox{if} \ \ n_{1}\!= \! {\rm{sign}} (n_{1})\ , \quad  N \geqslant 2 \,.
\end{array}
\end{align*}

\vspace{-0,25cm}
\noindent {\rm{{\bf{(e)}}}} \ $2 \max \left\{ \ \im z \, \left| \, z \in  \fetr^{\hspace{0,05cm}n_{N}, ..., n_{1}} \,  \right\}\right. \leqslant 1
\Big/\left(1 + \sum\limits_{k=2}^{N} (2 |n_{k}| -1) \right) \leqslant 1/N$ holds, \vspace{-0,15cm} \\
\hphantom{ {\rm{{\bf{(e)}}}}   }where it is assumed that $\sum\nolimits_{k=2}^{1} := 0$.
\end{lemsectsix}

The proof of Lemma~\ref{bsdentplem1} is supplied in Section~\ref{pbsdentplem1}. In the sequel, we agree to use the notation\vspace{-0,15cm}
 \begin{align}\label{f9contgendef1}
     &
     \phi_{\eufm{n}, \eufm{m}}\!:= \!\phi_{{\fo{n_{N}, ..., n_{1},
    m_{M}, ..., m_{1}
     }}} \, , \ \
     \eufm{m}\!=\!(m_{M}, ..., m_{1})  , \ \eufm{n}\!=\!(n_{N}, ..., n_{1})\in \Bb{Z}_{\neq 0}^{\hspace{0,02cm}\Bb{N}_{\hspace{-0,02cm}\eurm{f}}}\,.
 \end{align}

 \noindent
As we combine the properties \eqref{f15contgen}(b) and \eqref{f1contgendef1}(c)
together with the definitions \eqref{f2bsden}, \eqref{f10contgen} and \eqref{f9contgen},
we find for every $N\! \in\! \Bb{N}$ and
$(n_{N}, ..., n_{1})\!\in \!\Bb{Z}_{\neq 0}^{N}$ that{\hyperlink{r21}{${}^{\ref*{case21}}$}}\hypertarget{br21}{}
\begin{align}\label{f4contgendef1}
    &
\hspace{-0,3cm}  \begin{array}{lrcccll}{\rm{(a)}} \,   &
 \Bb{G}^{N}_{2} \left(\fetr^{\hspace{0,05cm}n_{N}, ..., n_{1}}\right)& \! = \!  &
 \fet^{\,{\tn{|}}\sigma_{n_{1}}} & \! = \!  & \fet \sqcup \left( \sigma_{n_{1}} \!+\! i \Bb{R}_{>0} \right)  \,,   &\  \mbox{if} \ N \! \geqslant\!  1\,;   \\[0,25cm]
{\rm{(b)}} \,   &
 \Bb{G}^{k}_{2} \left(\fetr^{\hspace{0,05cm}n_{N}, ..., n_{1}}\!\right)& \! = \!  &\fetr^{\hspace{0,05cm}n_{N-k}, ..., n_{1}}& \! = \!  &
\phi_{{\fo{n_{N-k}, ..., n_{1}}}}\big( \fet^{\,{\tn{|}}\sigma_{n_{1}}}\!\big)
\,, &\  \mbox{if} \ 1 \! \leqslant\!  k \! \leqslant\!  N\! -\! 1 \,.
      \end{array}\hspace{-0,2cm}
\end{align}

\noindent
It follows  from Lemma~\ref{bsdentplem1}(a) and \eqref{f2bsden} that  the sets
$\fetr^{\hspace{0,01cm}\eufm{n}}$, $\eufm{n}\in \Bb{Z}_{\neq 0}^{\hspace{0,02cm}\Bb{N}_{\hspace{-0,02cm}\eurm{f}}}$, are disjoint subsets of
$\Bb{H}_{|\re|<1} \setminus\fet$, and their union equals all of
$\Bb{H}_{|\re|<1} \setminus\fet $.

\begin{lemsectsix}\hspace{-0,15cm}{\bf{.}}\label{bsdentplem2} Let  $\Bb{Z}_{\neq 0}^{\hspace{0,02cm}\Bb{N}_{\hspace{-0,02cm}\eurm{f}}}\!:=\!\cup_{k\geqslant 1} \Bb{Z}_{\neq 0}^{k}$.  We have the partition
\begin{align}\label{f3bsdentp}
    &  \Bb{H}_{|\re |\leqslant1}
     =\bigsqcup\limits_{\, {\fo{\eufm{n} \!\in \!\Bb{Z}_{\neq 0}^{\hspace{0,02cm}\Bb{N}_{\hspace{-0,02cm}\eurm{f}}} \cup\{0\}  }}}
\fetr^{\,\eufm{n}} = \mathcal{F}^{\,{\tn{||}}}_{{\tn{\square}}}\sqcup
 \bigsqcup\limits_{\, {\fo{\phi\!\in \!  \Gamma_{\vartheta}^{\hspace{0,015cm}{\tn{||}}}   }}}
 \phi \left(\fet^{\,{\tn{|}}\sigma_{\phi}}\right) ,
 \end{align}

\noindent where for $\eufm{n} =0$  we use the  notation $ \fetr^{\hspace{0,04cm}0} :=\fet^{\,{\tn{||}}}$.
\end{lemsectsix}

\noindent  The proof of Lemma~\ref{bsdentplem2} is supplied in Section~\ref{pbsdentplem1}.
It is clear that (see \eqref{f0bsdentp})
\begin{align}\label{f2bsdentp}
    &  \mathcal{F}^{\infty}_{{\tn{\square}}}
    = \mathcal{F}_{{\tn{\square}}}  \sqcup \bigsqcup\nolimits_{\,n \geqslant 1} \
 \Big(\big( \fet^{\,{\tn{|\!+\!1}}} - 2n\big) \sqcup \big( \fet^{\,{\tn{|\!-\!1}}} + 2n\big)\Big)\,.
\end{align}

\noindent By combining \eqref{f3bsdentp} and \eqref{f2bsdentp}, we arrive at the partition\vspace{-0,1cm}
\begin{multline}\label{f4bsdentp}
\Bb{H} =   \mathcal{F}_{{\tn{\square}}} \  \sqcup
   \bigsqcup\limits_{{\fo{n_{0} \in \Bb{Z}_{\neq 0}}}} \bigg(-2 n_{0}+ \Big( \fet \sqcup \big({\rm{sign}}(n_{0})  + i \Bb{R}_{> 0}\big)\Big)\bigg)
    \\[-0,1cm]  \sqcup
 \bigsqcup\limits_{
\substack{  {\fo{n_{1}, ... , n_{N} \in \Bb{Z}_{\neq 0}}} \\[0,075cm] {\fo{N \in \Bb{N}\,, \ \ n_{0} \in \Bb{Z}}} }
}\bigg( 2 n_{0} + \phi_{{\fo{\,n_{N}, ..., n_{1}}}}
\Big( \fet \sqcup \big({\rm{sign}}(n_{1})  + i \Bb{R}_{> 0}\big)\Big)\bigg)\,.
\end{multline}

\noindent which we refer to as the {\it{Schwarz partition of $\Bb{H}$}}. The following
property is obtained in Section~\ref{pbsdentpth2}.
\begin{theorem}\hspace{-0,15cm}{\bf{.}}\label{bsdentpth2}\! The set $\{\phi  (1\!+\!{\rm{i}}\Bb{R}_{>0})\}_{\phi \in\! \Gamma_{\vartheta}}$\!  generates the Schwarz
partition of $\Bb{H}$.
\end{theorem}

\vspace{-0,2cm}
\begin{rem}\hspace{-0,15cm}{\bf{.}}\label{rem51}\!\!  {\rm{
It is known  that the fundamental domain (see \cite[p.\! 15]{cha}) $F_{\Gamma(2)}$ for the subgroup $ \Gamma(2)\! \subset\!  \Gamma $, (elements written below are linear fractional mappings)
\begin{align}\nonumber
     \hspace{-0,25cm}
   \Gamma(2)   &    =
     \bigsqcup\limits_{ n_{0} \in \Bb{Z}}\left\{2n_{0} + z \right\}
   \ \   \sqcup
   \bigsqcup\limits_{
\substack{  {\fo{n_{1}, ... , n_{2 N} \in \Bb{Z}_{\neq 0}}} \\[0,075cm] {\fo{N \in \Bb{N}\,, \ \ n_{0} \in \Bb{Z}}} }
} \left\{ 2 n_{0} + \phi_{{\fo{\hspace{0,005cm} n_{2N}, n_{2N-1},..., n_{1}}}}(z)\right\}
\\[0,1cm]   & \hspace{1,63cm}
  \sqcup
\bigsqcup\limits_{
\substack{  {\fo{n_{1}, ... , n_{2 N-1} \in \Bb{Z}_{\neq 0}}} \\[0,075cm] {\fo{N \in \Bb{N}\,, \ \ n_{0} \in \Bb{Z}}} }
} \left\{2 n_{0} + \phi_{{\fo{\hspace{0,005cm} n_{2N-1}, n_{2N-2},..., n_{1}}}}(-1/z) \right\},
 \hspace{-0,1cm}
\label{f4zbsdentp}\end{align}

\vspace{-0,1cm}
\noindent
 generated by the M\"obius transformations $z\mapsto z+2$ and $z\mapsto z/(1-2z)$ (see \cite[p.\! 111]{cha}),   can be chosen by any of the four sets written in
\eqref{f1lemoneto}, for instance, $F_{\Gamma(2)} =  \fet \sqcup \gamma (-1,\infty) \sqcup \gamma (1,0)$, where
 ${\rm{int}} F_{\Gamma(2)} =  \fet$
 (cp. \cite[p.\! 115]{cha}). Then
\begin{align}\label{f4ybsdentp}
    & \Bb{H} =   \bigsqcup\limits_{{\fo{  \phi\! \in\! \Gamma(2)   }}}   \phi \big(F_{\Gamma(2)} \big) =
   \bigsqcup\limits_{{\fo{  \phi\! \in\! \Gamma(2)   }}}\Big(
   \phi \big(\fet\big)\sqcup\phi \big(\gamma (-1,\infty) \big)\sqcup \phi \big(\gamma (1,0) \big)\Big)   \, , \
\end{align}

\noindent
forms a partition of the upper half-plane corresponding to the subgroup $ \Gamma(2) $. But it follows from \eqref{f4zbsdentp}, $-1/\fet=\fet$    and  \eqref{f4ybsdentp} that
\begin{align}\label{f4xbsdentp}\hspace{-0,2cm}
\bigsqcup\limits_{{\fo{  \phi\! \in\! \Gamma(2)   }}} \!\!\!  \phi \big(\fet \big)     =
     \bigsqcup\limits_{ {\fo{n_{0} \in \Bb{Z}}}}\left(2n_{0} + \fet \right) \
     \sqcup \!\!\!
   \bigsqcup\limits_{
\substack{  {\fo{n_{1}, ... , n_{ N} \in \Bb{Z}_{\neq 0}}} \\[0,075cm] {\fo{N \in \Bb{N}\,, \ \ n_{0} \in \Bb{Z}}} }
} \!\!\!\!\!\!\left( 2 n_{0} + \phi_{{\fo{\hspace{0,005cm} n_{N}, ..., n_{1}}}}(\fet)\right) \, , \
\end{align}

\noindent and therefore the partition \eqref{f4ybsdentp} coincides with the Schwarz partition  \eqref{f4bsdentp}  on the set $\Bb{H} \setminus \{\phi  (1\!+\!{\rm{i}}\Bb{R}_{>0})\}_{\phi \in \Gamma_{\vartheta}}= \sqcup_{ \phi \in \Gamma(2)} \phi(\fet)$, because
\begin{align}\hspace{-0,3cm}\nonumber
      \bigsqcup\nolimits_{{\fo{  \phi\! \in\! \Gamma(2)   }}} \  \phi \big( \gamma (-1,\infty) \sqcup
    \gamma (1,0)\big)&  = \bigcup\nolimits_{{\fo{  \phi\! \in\! \Gamma(2)   }}}  \  \phi \big( \partial_{\Bb{H}}\fet\big)   =
    \bigcup\nolimits_{{\fo{  \phi\! \in\! \Gamma_{\vartheta}   }}}\  \phi(\partial_{\Bb{H}}\fet) \\    & =
    \bigcup\nolimits_{{\fo{  \phi\! \in\! \Gamma_{\vartheta}   }}}\ \phi  (1\!+\!{\rm{i}}\Bb{R}_{>0}),
\hspace{-0,14cm}\label{f4wbsdentp}\end{align}

\noindent as follows from $\gamma (1,\infty)=2+\gamma (-1,\infty)$, $\gamma (-1,0)=\phi_{-1}(-1/\gamma (1,0))$, \eqref{zbssqlem1},
 $-1/\partial_{\Bb{H}}\fet=\partial_{\Bb{H}}\fet$, $\left\{ \, \phi (z)\,\right\}_{{\fo{\phi\! \in\! \Gamma_{\vartheta} }}} \! =\! \left\{ \, \phi (z)\,, \ \phi (-1/z)\,\right\}_{{\fo{\phi\! \in\!
\Gamma (2) }}}$ (see \cite[p.\! 115]{cha}) and $\phi_{-1} (z)= z/(1-2z)$, $z\in \Bb{H}$, $\phi_{-1} \in \Gamma(2)$. But on the set
\begin{align}
    & \bigcup\limits_{{\fo{\phi \in \Gamma_{\vartheta}}}} \phi  \big(1+{\rm{i}}\Bb{R}_{>0}\big) = \Bb{H}\setminus
    \bigcup\limits_{{\fo{\phi \in \Gamma (2)}}} \phi  \big(\, {\rm{int}} F_{\Gamma(2)}\big)
\label{f4vbsdentp}\end{align}

 \noindent   these two partitions are
 different.
 This difference is  essential, since the sets that make up the Schwarz partition are much simpler than the sets of the partition \eqref{f4ybsdentp}. Actually, if the partition \eqref{f4ybsdentp} adds to each open hyperbolic quadrilateral  $\phi_{{\fo{\hspace{0,005cm} n_{N}, ..., n_{1}}}} (\fet)$   its two sides  $ \gamma_{{\fo{\hspace{0,005cm} n_{N}, ..., n_{1}}}}(0,1)$ and $\gamma_{{\fo{\hspace{0,005cm} n_{N}, ..., n_{1}}}} (-1,\infty)$, which can be two lower arches  or a lower arch and a roof,  depending on the sign of $n_{1}$, then  the Schwarz partition
  adds to such a quadrilateral
  its own roof $\gamma_{{\fo{\hspace{0,005cm} n_{N}, ..., n_{1}}}}  ({\rm{sign}} (n_{1}), \infty)$, where
  $n_{1}, ... , n_{ N} \in \Bb{Z}_{\neq 0}$ and $N \in \Bb{N}$ (see the note before Lemma~\ref{bsdentplem1}). Hence, the  Schwarz partition  \eqref{f4bsdentp} is an easy-to-use modification of the known partition \eqref{f4ybsdentp}
  associated with the subgroup $\Gamma(2)$. Moreover, the choice of the most convenient partition of the set
  \eqref{f4vbsdentp}
 is a separate issue{\hyperlink{r42}{${}^{\ref*{case42}}$}}\hypertarget{br42}{}, which is no longer related to the structure of  $ F_{\Gamma(2)} \cap \partial F_{\Gamma(2)}$, but depends on the mutual arrangement of the sets $\{\phi (\,{\rm{int}} F_{\Gamma(2)})\}_{\phi \in \Gamma(2) }$.
}}\end{rem}

\subsection[\hspace{-0,31cm}. \hspace{0,11cm}Even rational  partition of the upper half-plane.]{\hspace{-0,11cm}{\bf{.}} Even rational  partition of the upper half-plane.
}\label{bsdenac}
By virtue of Lemma~2 of \cite[p.\! 112]{cha}{\hyperlink{r43}{${}^{\ref*{case43}}$}}\hypertarget{br43a}{} and \eqref{f8contgen}, the relationships \eqref{f4contgen} and \eqref{f7contgen} can be expressed in the form
(the elements of the sets below are linear fractional mappings){\hyperlink{r44}{${}^{\ref*{case44}}$}}\hypertarget{br44}{}
\begin{align}\nonumber
    & \hspace{-0,2cm} \Gamma_{\vartheta} \! =\!\left\{\,
 {\fo{ \dfrac{1\cdot z + 0}{0\cdot z +1} }} \,, \
 {\fo{ \dfrac{0\cdot z - 1}{1\cdot z +0} }} \,\right\}  \\ &  \hspace{0,08cm}\bigsqcup \bigg\{\,
{\fo{ \dfrac{a z + b}{c z +d} }} \,,\,
{\fo{\dfrac{(-c) z +(-d)}{a z + b} }} \,,\,
{\fo{ \dfrac{b z +(- a)}{d z +(- c)} }} \,,\,
{\fo{ \dfrac{(-d) z + c}{b z +(- a)} }} \ \bigg| \
 {\fo{\dfrac{a z + b}{c z +d}}} \in  \Gamma_{\vartheta}^{\hspace{0,015cm}{\tn{||}}}
\, \bigg\}\, ,\hspace{-0,1cm}
\label{f1denac}  \\[0,2cm]  & \nonumber
\hspace{-0,2cm}  \Gamma_{\vartheta}^{\hspace{0,015cm}{\tn{||}}} \! =\!
\bigg\{
{\fo{\, \dfrac{a z + b}{c z +d} }} \ \bigg| \
{\fo{\begin{array}{l}
 |a|\!<\! |b|\!<\!d,\ |a|\!<\! |c|\!<\!d     \\
ab\equiv cd \equiv 0 ({\rm{mod}}\, 2)    \\
ad - bc\!=\!1 \, , \ a, b, c, d \!\in\! \Bb{Z}
\end{array}}}\, \bigg\} \\  &\hspace{0,19cm}
= \!\!\Bigg\{\,
 {\fo{\dfrac{1}{2n_{N}\! -\! \dfrac{1}{ \begin{subarray}{l}
\begin{displaystyle} \vphantom{A^{A}}
2n_{N-1}\!- \!\end{displaystyle}\\[-0.4cm]
\hphantom{2n_{k-1} 11 -} \ddots\\[-0.4cm]
\hphantom{2n_{k-1} 11 - \ddots} \begin{displaystyle} - \!\end{displaystyle} \dfrac{1}{2 n_{2} \!- \! \dfrac{1}{
\vphantom{A^{A}}2 n_{ 1} -z }}\end{subarray} }}}}
\ \Bigg| \ n_{1}, n_{2}, ..., n_{N}\in \Bb{Z}_{\neq 0} \, , \
 N\in \Bb{N} \ \Bigg\} \, ,
\label{f2denac} \end{align}

\noindent where no repetitions of linear fractional maps occur in the
listing of the sets on the  right-hand sides of \eqref{f2denac} and \eqref{f1denac}. That this is so follows
from Lem\-ma~\ref{bsdentplem1}(a) and the fact that neither of the mappings $z\mapsto z$ nor
$z\mapsto -1/z$ can be an element of the rightmost set of \eqref{f1denac}, where
the listed
four linear fractional maps are all different as the unique entry in the associated $2\times2$ matrix with
biggest absolute value (denoted by $d$) occupies four different positions. Consequently, it makes sense to define  the "order" of each element $\phi\! \in\! \Gamma_{\vartheta}$ by
\begin{align}\label{f11contgen}
    &  d(\phi )  \!= \!\left\{
\begin{array}{llll}
 0  ,      &   \mbox{if} \ \phi(z)=\phi_{0}(z)=z \, ;         &
\  1  ,       &    \mbox{if} \ \phi(z)=-1/z  \, ;      \\
N  ,       &      \mbox{if} \ \phi(z)=\phi_{\eufm{n}}(z)  \, ;    &
 \ N -1 ,       &      \mbox{if} \ \phi(z)= -1/\phi_{\eufm{n}}(z)  \, ;   \\
 N +1 ,       &     \mbox{if} \ \phi(z)= \phi_{\eufm{n}}(-1/z)  \, ;     &
\ N  ,       &     \mbox{if} \ \phi(z)= -1/\phi_{\eufm{n}}(-1/z)  \, ,   \\
\end{array}
\right.
\end{align}

\noindent where $\eufm{n}\!=\!(n_{N}, ..., n_{1})\in \Bb{Z}_{\neq 0}^{N}$, $N \in \Bb{N}$. Then the functional relations $\lambda(z+2)=\lambda(z)$, $\lambda(-1/z)=1-\lambda(z)$, $z\in \Bb{H}$, give that
(cf. \eqref{f2int}, \cite[p.\! 111]{cha})
\begin{align}\label{f14bsdenac}
    &  \lambda\big(\phi (z)\big)=
\left\{\begin{array}{rl}
  \lambda(z), &  \  \mbox{if} \   d(\phi) \in \{0\}\cup 2 \Bb{N}\,, \\[0,1cm]
1-\lambda(z),   &   \   \mbox{if} \    d(\phi) \in 2 \Bb{N}\!-\!1\,, \end{array}\right.
 \  \quad   \phi\! \in\! \Gamma_{\vartheta}  \, , \  z\in \Bb{H}\,.
\end{align}

\noindent Moreover, since the semicircle $\gamma (-1,1)$ is invariant under the inversion
$z\mapsto-1/z$, we realize that
\begin{align}\label{f7denac}
    &  \lambda \Big(\phi \big(\gamma (-1,1)\big)\Big)=
\lambda\big(\gamma (-1,1)\big) = (1/2)+ i \Bb{R}_{> 0} \, , \quad  \phi\!\in \!  \Gamma_{\vartheta}  \, , \
\end{align}

\noindent and, in the notation $\Lambda:=(0,1)\cup (\Bb{C}\setminus \Bb{R})$,
\begin{align}\label{f14zbsdenac}
    & \lambda \left(\phi \big(\fet^{\,{\tn{|}}\sigma_{\phi}}\big)\right)=
\left\{\begin{array}{rl}
   \lambda\left(\fet\cup \gamma(\sigma_{\phi}, \infty)\right)
=\Bb{R}_{<0}\cup\Lambda
, &  \  \mbox{if} \   d(\phi) \in \{0\}\cup 2 \Bb{N}\,, \\[0,1cm]
 \lambda\left(\fet\cup \gamma(-\sigma_{\phi}, 0)\right)
=\Bb{R}_{>1}\cup\Lambda
,   &   \   \mbox{if} \    d(\phi) \in 2 \Bb{N}\!-\!1\,, \end{array}\right.
\end{align}

\noindent holds for every $\phi\!\in \!  \Gamma_{\vartheta} \cup \{\phi_{0}\}$ (where $\phi_{0} (z):= z$). By Lemma~\ref{lemoneto},  \eqref{f14zbsdenac} means that $\lambda$ is one-to-one on each set of the partition \eqref{f3bsdentp}, with the exception of  $\mathcal{F}^{\,{\tn{||}}}_{{\tn{\square}}}$, where $\lambda (\gamma (1, \infty)) = \lambda (\gamma (-1, \infty))= \Bb{R}_{<0}$. This has one useful consequence, the proof of which is supplied in Section~\ref{plemdhp}.
\begin{lemsectsix}\hspace{-0,15cm}{\bf{.}}\label{lemdhp} Let us write
$\Bb{Z}_{\neq 0}^{2\Bb{N}_{\hspace{-0,02cm}\eurm{f}}}\!:=\!\cup_{k\geqslant 1} \Bb{Z}_{\neq 0}^{2k} $  and
$\Bb{Z}_{\neq 0}^{2\Bb{N}_{\hspace{-0,02cm}\eurm{f}}-1}\!:=\!\cup_{k\geqslant 1}\Bb{Z}_{\neq 0}^{2k-1}$.
Then for arbitrary $y \in \Bb{H}\cap {\rm{clos}}\left(\mathcal{F}_{{\tn{\square}}}\right)$ and
$\sigma \in\{1, -1\}$ we have
\begin{align}\label{f1lemdhp}
    &  \left\{ z\in \Bb{H}_{|\re|\leqslant 1} \ | \ \lambda(z)=\lambda(y)\right\} \\[0,15cm] \nonumber    &=
\left\{
  \begin{array}{ll}
  \left\{y\right\}\sqcup
\left\{\phi_{{\fo{\eufm{n}}}}(y)\ \left|  \ \eufm{n}\in  \Bb{Z}_{\neq 0}^{2\Bb{N}_{\hspace{-0,02cm}\eurm{f}}}  \right\}\right.\sqcup
\left\{\phi_{{\fo{\eufm{n}}}}(-1/y)\ \left| \ \eufm{n}\in  \Bb{Z}_{\neq 0}^{2\Bb{N}_{\hspace{-0,02cm}\eurm{f}}-1}  \right\}\right., &\ \hbox{if}\ \
y\in \fet;
 \\[0,15cm]
\left\{y\right\}\sqcup\left\{y-2\sigma\right\}
\sqcup
\left\{\phi_{{\fo{\eufm{n}}}}\left(y-\sigma  + \sigma_{\eufm{n}}\right)\ \left| \ \eufm{n}\in  \Bb{Z}_{\neq 0}^{2\Bb{N}_{\hspace{-0,02cm}\eurm{f}}}  \right\} \right., & \  \hbox{if}\  \ y  \in \gamma(\sigma, \infty); \\[0,15cm]
\left\{\phi_{{\fo{\eufm{n}}}}\left(  \sigma_{\eufm{n}}+\sigma-{\fo{(}}1/y{\fo{)}}
\hspace{0,05cm}\right) \ \left| \ \eufm{n}\in  \Bb{Z}_{\neq 0}^{2\Bb{N}_{\hspace{-0,02cm}\eurm{f}}-1}  \right\}\right. , & \  \hbox{if} \ \ y  \in \gamma(\sigma, 0) \, ,
  \end{array}
\right.
\end{align}

\noindent where each set is countable and has no limit points in $\Bb{H}$.
\end{lemsectsix}

  Lemma~\ref{lemdhp} can be used to characterize the sets
\begin{align}\label{f1contgen}
    & \hspace{-0,15cm} S^{\hspace{0,02cm}{\tn{||}}}_{\!{\tn{\frown}}}:=\! \left\{\,  z \!\in\! \Bb{H}_{|\re|\leqslant 1} \, \left|\, \lambda (z)\! \in \!
    (1/2)\! +\! i \Bb{R}\,\right\} \right.\! , \  S^{\infty}_{\!{\tn{\frown}}}\!:=\! \left\{\,  z \!\in \!\Bb{H} \,  \left|\, \lambda (z) \!\in \! (1/2) \!+\! i \Bb{R} \,\right\} \right.\! .\hspace{-0,1cm}
\end{align}

\noindent Indeed, the case $y\in \fet$ in \eqref{f1lemdhp}  applied   to the  semicircle  $\gamma (-1,1) \!\subset\! \fet$,  gives
\begin{align}\label{f2bsdenac}
    &\hspace{-0,25cm}  {\rm{(a)}}\  S^{\hspace{0,02cm}{\tn{||}}}_{\!{\tn{\frown}}}\!=\!  \bigsqcup\nolimits_{{\fo{\phi \in \Gamma_{\vartheta}^{\hspace{0,02cm}{\tn{||}}}\cup \{\phi_{0}\} }}} \ \phi \big(\gamma (-1,1)\big) \ , \ \  {\rm{(b)}}\
     S_{\!{\tn{\frown}}}^{\infty}=\bigsqcup\nolimits_{n \in \Bb{Z}} \  (2n + S^{\hspace{0,02cm}{\tn{||}}}_{\!{\tn{\frown}}})\,,\hspace{-0,1cm}
\end{align}

 \noindent and it follows from  $-1/\gamma (-1,1)=\gamma (-1,1)$ and  \eqref{f6contgen} that
\begin{align}\label{f1zbsdenac}
    & \hspace{-0,2cm}
     S_{\!{\tn{\frown}}}^{\infty}\!=\!\!\bigcup\limits_{{\fo{\phi \in \Gamma_{\vartheta} }}}\!\!\! \phi \big(\gamma (-1,1)\big) \ , \ \ \phi \left(  S_{\!{\tn{\frown}}}^{\infty}\right)\!= \! S_{\!{\tn{\frown}}}^{\infty}  , \  \phi \left( \Bb{H}\setminus S_{\!{\tn{\frown}}}^{\infty}\right)\!= \! \Bb{H}\setminus S_{\!{\tn{\frown}}}^{\infty}  , \
          \ \
\phi \!\in \!\Gamma_{\vartheta} \, . \hspace{-0,1cm}
\end{align}

Regarding the two triangular parts into which the open semicircle $\gamma (-1,1)$ splits the Schwarz
quadrilateral, we introduce the notation (see \eqref{f10contgen})
\begin{align}\label{f3zbsdenac}
\fetu:=\mathcal{F}_{{\tn{\square}}} \setminus\overline{\Bb{D}} \, ,  \ \
\fetu^{\,{\tn{||}}}:=\mathcal{F}^{\,{\tn{||}}}_{{\tn{\square}}} \setminus\overline{\Bb{D}} \, ,  \ \
 \fetd := \Bb{D} \cap \fet   \, , \ \
         \fetu^{\,{\tn{|}}\sigma} :=  \fet^{\,{\tn{|}}\sigma}  \setminus\overline{\Bb{D}} \, ,
\end{align}

\noindent where $\sigma\in \{+1,-1\}$. Then
\begin{align}\label{f4denac}
    & \hspace{-0,15cm}\fet  = \fetu\sqcup \gamma (-1,1) \sqcup \fetd \, , \
\fet^{\,{\tn{|}}\sigma}  = \fetu^{\,{\tn{|}}\sigma}\sqcup \gamma (-1,1) \sqcup \fetd \, , \ \ \sigma\in \{+1,-1\} \,,\hspace{-0,1cm}
\end{align}

\noindent and for each $\eufm{n}\in \Bb{Z}_{\neq 0}^{\hspace{0,02cm}\Bb{N}_{\hspace{-0,02cm}\eurm{f}}}$ we can apply $\phi_{\eufm{n}}$ to the  equalities \eqref{f4denac}
with $\sigma=\sigma_{\eufm{n}}$, to obtain
\begin{align}\label{f5bsden}
    & \hspace{-0,2cm}
    \begin{array}{lll}
\fetr^{\hspace{0,01cm}\eufm{n}} =
\fetur^{\hspace{0,01cm}\eufm{n}} \sqcup \gamma_{\eufm{n}} (-1, 1)  \sqcup \fetd^{\hspace{0,01cm}\eufm{n}} \, ,  & \ \fetur^{\hspace{0,01cm}\eufm{n}} := \phi_{\eufm{n}}  \big(\fetu^{\,{\tn{|}}\sigma_{\eufm{n}}}\big),  & \ \fetd^{\hspace{0,01cm}\eufm{n}} := \phi_{\eufm{n}} \big(\fetd  \big),       \\[0,2cm]
\get^{\hspace{0,01cm}\eufm{n}}\! =
\getu^{\hspace{0,01cm}\eufm{n}} \sqcup \gamma_{\eufm{n}} (-1, 1)  \sqcup \fetd^{\hspace{0,01cm}\eufm{n}} \, ,  & \
\getu^{\hspace{0,01cm}\eufm{n}} := \phi_{\eufm{n}}  \big(\fetu\big),  & \
\   \eufm{n}\in \Bb{Z}_{\neq 0}^{\hspace{0,02cm}\Bb{N}_{\hspace{-0,02cm}\eurm{f}}}\,.
    \end{array}
\hspace{-0,1cm}
\end{align}

\noindent Here $\fetd^{\hspace{0,01cm}\eufm{n}}$ and $\fetur^{\hspace{0,01cm}\eufm{n}}$ are  hyperbolic triangles, the set
$\{\,\phi_{\eufm{n}}(x)\, | \, x\!\in\! \{-1,0,1\}\,\}$ forms the vertices of $\fetd^{\hspace{0,01cm}\eufm{n}}$, and $\fetd^{\hspace{0,01cm}\eufm{n}}$
  is open,
    while $\gamma_{\eufm{n}} (-1, 1)$ is the roof of  $\fetd^{\hspace{0,01cm}\eufm{n}}$.  At the same time, $\{\,\phi_{\eufm{n}}(x)\, | \, x\!\in\! \{-1,1,\infty\}\,\}$  are also the vertices of $\fetur^{\hspace{0,01cm}\eufm{n}}$,   $\gamma_{\eufm{n}} (-1, 1)$ and  $\gamma_{\eufm{n}} ( - \sigma_{\eufm{n}}, \infty)$ are the lower arches of $\fetur^{\hspace{0,01cm}\eufm{n}}$   not contained in $\fetur^{\hspace{0,01cm}\eufm{n}}$,  while
 the roof $\gamma_{\eufm{n}}(\sigma_{\eufm{n}}, \infty)$ of   $\fetr^{\hspace{0,01cm}\eufm{n}}$ is contained in
 $\fetur^{\hspace{0,01cm}\eufm{n}}$.

A combination of  \eqref{f3bsdentp} and \eqref{f2bsdenac}, \eqref{f4denac}, \eqref{f5bsden} gives
\begin{align*}
    & \hspace{-0,3cm} \Bb{H}_{|\re |\leqslant 1}\setminus S^{\hspace{0,02cm}{\tn{||}}}_{\!{\tn{\frown}}} =
   \fetu^{\,{\tn{||}}} \sqcup \fetd
\sqcup\hspace{-0,3cm}
 \bigsqcup\limits_{ {\fo{n_{1}, ... , n_{N} \in \Bb{Z}_{\neq 0}}} \,, \  {\fo{N \in \Bb{N}}} }\left(\fetur^{\hspace{0,05cm}n_{N}, ..., n_{1}}
  \sqcup \fetd^{\hspace{0,05cm}n_{N}, ..., n_{1}}\right).\hspace{-0,1cm}
\end{align*}

\noindent By regrouping the subsets involved, we find that
\begin{align}\label{f15bsdenac}
    &  \Bb{H}_{|\re |\leqslant 1}\setminus S^{\hspace{0,02cm}{\tn{||}}}_{\!{\tn{\frown}}} = \fetu^{\,{\tn{||}}}  \sqcup \bigsqcup\nolimits_{{\fo{\ \eufm{n}\in \Bb{Z}_{\neq 0}^{\hspace{0,02cm}\Bb{N}_{\hspace{-0,02cm}\eurm{f}}}\cup\{0\} }}}\eusm{E}^{\eufm{n}}_{\!{\tn{\frown}}} \, , \
\end{align}

\vspace{-0,4cm}
\noindent where we use the notation
\begin{align}\label{f8bsdenac}
    &\hspace{-0,2cm}
\begin{array}{rl}
{\rm{(a)}}   &  \begin{displaystyle}
   \eusm{E}^{0}_{\!{\tn{\frown}}}\!:=\! \fetd \sqcup\!\!\!\!
 \bigsqcup\limits_{{\fo{n_{0}\!\in\! \Bb{Z}_{\neq 0}}}}\!\!\!\!\! \fetur^{\hspace{0,05cm} n_{0}} ; \qquad \eusm{E}^{\hspace{0,05cm}n_{N}, ...,\, n_{1}}_{\!{\tn{\frown}}}\!:=\!\fetd^{\hspace{0,05cm}n_{N}, ...,\, n_{1}} \sqcup \!\!\!\! \bigsqcup\limits_{{\fo{n_{0}\!\in\! \Bb{Z}_{\neq 0}}}}\!\!\!\!\!  \fetur^{\hspace{0,05cm}n_{N}, ..., \,n_{1}, \,n_{0}} ,
\end{displaystyle}
 \\[0,6cm]
{\rm{(b)}}  &\begin{displaystyle}
  \eusm{E}^{\hspace{0,05cm}n_{N}, ...,\, n_{1}}_{\!{\tn{\frown}}} =   \phi_{{\fo{n_{N}, ..., n_{1}}}}\big( \eusm{E}^{0}_{\!{\tn{\frown}}}\big),
 \
 \quad  n_{N}, ..., n_{1}\!\in\!  \Bb{Z}_{\neq 0}\,, \ \ N \in \Bb{N}\,,
\end{displaystyle}\end{array}\hspace{-0,1cm}
\end{align}

\noindent  and the  identities \eqref{f8bsdenac}(b) follow directly from \eqref{f5bsden} and \eqref{f2bsden}.
Moreover,
\begin{align*}
    &  -1\big/ \eusm{E}^{0}_{\!{\tn{\frown}}}\!=\! \fetu\sqcup\!\!\!
 \bigsqcup\limits_{{\fo{n_{0}\!\in\! \Bb{N}}}}\!\!\!\Big( \big(\fetu^{\,{\tn{|}}+1}-2 n_{0}\big) \sqcup\big(\fetu^{\,{\tn{|}}-1}+2 n_{0}\big)\Big)=
\Bb{H} \setminus\underset{{\fo{m \in \Bb{Z}}}}{\cup } \left(2m +
\overline{\Bb{D}}\right),
\end{align*}

\noindent  which leads to ($\phi_{0}(z):=z$)
\begin{align}\label{f5denac}
    &
 \eusm{E}^{\eufm{n}}_{\!{\tn{\frown}}} =   \phi_{\eufm{n}}\big(- 1/\eusm{E}^{\infty}_{\!{\tn{\frown}}}\big) \ , \quad
 \eufm{n}\in \{0\}\cup \Bb{Z}_{\neq 0}^{\hspace{0,02cm}\Bb{N}_{\hspace{-0,02cm}\eurm{f}}} \, , \quad \eusm{E}^{\infty}_{\!{\tn{\frown}}}:=
\Bb{H} \setminus\underset{{\fo{m \in \Bb{Z}}}}{\cup } \left(2m +
\overline{\Bb{D}}\right).
\end{align}

\noindent In particular, the set $\eusm{E}^{0}_{\!{\tn{\frown}}}$ is open and its boundary
consists of the  roof $\gamma(-1,1)$ and the infinitely many lower arches $\gamma (1/(2n+1), 1/(2n-1))$,
$n\!\in\! \Bb{Z}_{\neq 0}$, which accumulate at the origin, i.e.,
\begin{align}\label{f5tdenac}
 {\rm{co}} \left(\eusm{E}^{0}_{\!{\tn{\frown}}}\right)\!=\!\Bb{D}_{\im > 0} \, ,  \qquad
    \eusm{E}^{0}_{\!{\tn{\frown}}}\!=\! \Bb{D}_{\im > 0}\Big\backslash \bigsqcup\limits_{{\fo{n \in \Bb{Z}_{\neq 0}}}} \phi_{n}\big(\overline{\Bb{D}}_{\im > 0}\big) ,
\end{align}

\noindent where $ {\rm{co}} (A)$ denotes the convex hull of a given subset $A \subset \Bb{R}^{2}$.
For every $\eufm{n}\!=\!(n_{N}, ..., n_{1}) \in  \Bb{Z}_{\neq 0}^{\hspace{0,02cm}\Bb{N}_{\hspace{-0,02cm}\eurm{f}}}$ the
associated subset $\eusm{E}^{\eufm{n}}_{\!{\tn{\frown}}}$ has the similar structure, as follows from the  relationships \eqref{f5denac}.
To be specific, $ \eusm{E}^{\eufm{n}}_{\!{\tn{\frown}}}$ is open and its boundary  consists of the roof $\gamma(\phi_{\eufm{n}}(-1),\phi_{\eufm{n}}(1))$ and the infinitely many lower arches
\\
$\phantom{a}$ \hspace{2,25cm}
$\gamma \big(\phi_{\eufm{n}}(1/{\sm{(}}2n+1{\sm{)}}), \phi_{\eufm{n}}(1/{\sm{(}}2n-1{\sm{)}})\big)$,  \quad   $n\!\in\! \Bb{Z}_{\neq 0}$,
\\
\noindent
 which accumulate at the  point $\phi_{\eufm{n}} (0)$. The analogue of \eqref{f5tdenac} reads   (see \eqref{f9contgendef1})
\begin{align}\label{f5zdenac}
    &\hspace{-0,25cm}
\begin{array}{l}
 \ \, {\rm{co}} \left(\eusm{E}^{\eufm{n}}_{\!{\tn{\frown}}}\right)\!=\!
    \phi_{\eufm{n}}\big(\Bb{D}_{\im > 0}\big), \
   \eusm{E}^{\eufm{n}}_{\!{\tn{\frown}}}=
\phi_{\eufm{n}}\big(\Bb{D}_{\im > 0}\big)\big\backslash\!\!\! \bigsqcup\limits_{{\fo{n_{0} \in \Bb{Z}_{\neq 0}}}}
\phi_{{\fo{\eufm{n}, n_{0}}}}\big(\overline{\Bb{D}}_{\im > 0}\big)
 \,.
\end{array}\hspace{-0,15cm}
\end{align}

\vspace{-0,15cm}\noindent
It follows that (see \eqref{f3bsden}, \eqref{12bsdenac} and \eqref{f1schw})
\begin{align*}
 &
\hspace{-0,25cm}\begin{array}{ll}
  \phi_{\eufm{n}}\big(\Bb{D}_{\im > 0}\big) = {\rm{co}}\big(\,\gamma \left(\phi_{\eufm{n}}(-1), \phi_{\eufm{n}}(1)\right)\,\big) \, , &\quad
 -1\! \leqslant\! \phi_{\eufm{n}}(-1)\! <\! \phi_{\eufm{n}}(1) \!\leqslant\! 1   \ , \\[0,25cm]
\Bb{H}\! \cap \partial \phi_{\eufm{n}}\big(\Bb{D}_{\im > 0}\big)\! = \!\gamma_{\eufm{n}} (-1,1)  \, , &\quad
 \eufm{n}\!\in \!\Bb{Z}_{\neq 0}^{\hspace{0,02cm}\Bb{N}_{\hspace{-0,02cm}\eurm{f}}} \,.
\end{array} \hspace{-0,1cm}
\end{align*}

\noindent In particular, $\eusm{E}^{\eufm{n}}_{\!{\tn{\frown}}}$ is simply connected  for each  $\eufm{n} \in \Bb{Z}_{\neq 0}^{\hspace{0,02cm}\Bb{N}_{\hspace{-0,02cm}\eurm{f}}}\cup\!\{0\} $ (see \cite[p.\! 93]{con}).

\vspace{0,1cm}The rational number $p/q$, $p \in \Bb{Z}$, $ q \in \Bb{Z}_{\neq 0}$, ${\rm{gcd}}(p, q) =1$, is called even if
$p\,q \,\equiv 0 ({\rm{mod}}\, 2)$.
It is known
(see \cite[p.\! 303]{lop}){\hyperlink{r39}{${}^{\ref*{case39}}$}}\hypertarget{br39}{}
that   each nonzero even rational number on $(-1,1)$ can be uniquely represented in the form $\phi_{\eufm{n}} (0)$ with some  $\eufm{n}\in  \Bb{Z}_{\neq 0}^{\hspace{0,02cm}\Bb{N}_{\hspace{-0,02cm}\eurm{f}}}$, and conversely.
This suggests the introduction of the following notions.

\begin{defsectsix}\hspace{-0,17cm}{\bf{.}}\label{contgendef2}
Given $m\in \Bb{Z}$ and $\eufm{n}\in  \Bb{Z}_{\neq 0}^{\hspace{0,02cm}\Bb{N}_{\hspace{-0,02cm}\eurm{f}}}\!:=\!\sqcup_{k\geqslant 1}\! \left(\Bb{Z}_{\neq 0}\right)^{k}$, the open sets \\[0,1cm] $\phantom{a}$\hspace{3cm}$2m\!+\!\eusm{E}^{0}_{\!{\tn{\frown}}}$, \   $\eusm{E}^{\infty}_{\!{\tn{\frown}}}$\   and \   $2m+  \eusm{E}^{\eufm{n}}_{\!{\tn{\frown}}}$  \\[0,1cm]\noindent
    are called the {\it{even rational  neighborhoods}} of
 $2m$, $\infty$  and   the even rational number $2m+ \phi_{\eufm{n}} (0)\in (2m-1,2m+1)\setminus\{2m\}$,
  respectively.
\end{defsectsix}

\vspace{-0.7cm}
\begin{figure}[htbp] \centering
    \begin{tikzpicture}
    \definecolor{cv0}{rgb}{0.95,0.95,0.95}
    \definecolor{cv1}{rgb}{0.9,0.9,0.9}
       \begin{scope}[scale=1]
  \clip(-5.5,-1) rectangle (6,4.5);
   \draw (-5.25,0) circle (0.04)  node[below] {$-3$};
    \draw (-3.5,0) circle (0.04)  node[below] {$-2$};
    \draw (-1.75,0) circle (0.04) node[below] {$-1$};
     \draw (-0.6,0) circle (0.025) (-0.7,0) node[below] {${\fo{-\dfrac{1}{3}}}$};
     \draw (-0.15,0)  node[below] {$\cdot\!\!\cdot\!\!\cdot$};
    \draw (0,0) circle (0.025) (0,-0.2) node[below] {$0$};
     \draw (0.15,0)  node[below] {$\cdot\!\!\cdot\!\!\cdot$};
          \draw (0.6,0) circle (0.025) node[below] {${\fo{\dfrac{1}{3}}}$};
    \draw (1.75,0) circle (0.04) node[below] {$1$};
    \draw (3.5,0) circle (0.04)  node[below] {$2$};
       \draw (5.25,0) circle (0.04)  node[below] {$3$};
       \fill[color=cv0,draw=lightgray] (-5.25,0) arc (180:0:1.75) arc (180:0:1.75) arc (180:0:1.75) -- (5.25,4) -- (-5.25,4);
\fill[color=cv1,draw=lightgray] (-1.75,0) arc (180:0:0.575) arc (180:0:0.15) arc (180:0:0.1)
arc (180:0:0.05) arc (180:0:0.05) arc (180:0:0.1) arc (180:0:0.15) arc (180:0:0.575)  arc (0:180:1.75);
 \draw (-0.3,0) circle (0.025) (-0.4,0.1) node[above] {${\fo{-\dfrac{1}{5}}}$};
 \draw (0.3,0) circle (0.025) (0.3,0.1) node[above] {${\fo{\dfrac{1}{5}}}$};
             \draw (-5.25,0) -- (5.25,0);
        \draw (-5.25,0) arc (180:0:1.75) arc (180:0:1.75) arc (180:0:1.75);
         \draw  (-1.75,0) arc (180:0:0.575) arc (180:0:0.15) arc (180:0:0.1)  arc (180:0:0.05);
          \draw  (1.75,0) arc (0:180:0.575) arc (0:180:0.15) arc (0:180:0.1) arc (0:180:0.05);
          \draw (0,0.835) node[above] {${\nor{\eusm{E}^{0}_{\!{\tn{\frown}}}\!:=\! - 1/\eusm{E}^{\infty}_{\!{\tn{\frown}}}   }} $};
          \draw (0,2.4) node[above] {${\nor{\eusm{E}^{\infty}_{\!{\tn{\frown}}}\!:=\!
\Bb{H} \setminus\underset{{\fo{m \in \Bb{Z}}}}{\cup } \left(2m +
\overline{\Bb{D}}\,\right)}}$}; 
 \draw (0,1.75) circle (0.04); \draw (0,1.95)  node[right] {$\!i$};
 \draw[dotted] (0,0)  -- (0,1);  \draw[dotted] (0,1.3)  -- (0,2.4); \draw[dotted] (0,3.25)  -- (0,4);
 \draw[dotted] (-5.25,1.75)  -- (5.25,1.75);
 \draw[dotted] (-3.5,0)  -- (-3.5,1.75); \draw (-3.5,1.75) circle (0.04);
 \draw[dotted] (3.5,0)  -- (3.5,1.75); \draw (3.5,1.75) circle (0.04);
      \end{scope}
    \end{tikzpicture}\vspace{-0.3cm}
    \caption{\hspace{-0,2cm}{\bf{.}}  Even rational  neighborhoods $\eusm{E}^{\infty}_{\!{\tn{\frown}}}$ and
    $\eusm{E}^{0}_{\!{\tn{\frown}}}$ of $\infty$ and of $0$, respectively.}
    \label{figure:neighborhood}\vspace{0,3cm}
\end{figure}

\vspace{-0.3cm}
For each given $\eufm{n}\in \{0\} \cup\Bb{Z}_{\neq 0}^{\hspace{0,02cm}\Bb{N}_{\hspace{-0,02cm}\eurm{f}}} $, we add  to the subset $\eusm{E}^{\eufm{n}}_{\!{\tn{\frown}}}$  its corresponding roof
$\phi_{\eufm{n}} (-1,1)$,  and
obtain the roofed subsets (see \eqref{12bsdenac}),
\begin{align}\label{f16bsdenac}
    &  \widehat{\eusm{E}}^{\hspace{0,025cm}\eufm{n}}_{\!{\tn{\frown}}}\!:=\!
   \gamma_{\eufm{n}}(-1,1) \sqcup   \eusm{E}^{\hspace{0,05cm}\eufm{n}}_{\!{\tn{\frown}}}, \quad
   \widehat{\eusm{E}}^{\hspace{0,025cm}\eufm{n}}_{\!{\tn{\frown}}}\!=\!
\phi_{\eufm{n}}\big(\overline{\Bb{D}}_{\im > 0}\big)\setminus \bigsqcup\nolimits_{\, {\fo{n_{0} \in \Bb{Z}_{\neq 0}}}}
\ \phi_{{\fo{\eufm{n}, n_{0}}}}\big(\overline{\Bb{D}}_{\im > 0}\big)
    ,
\end{align}

\noindent
which in view of \eqref{f15bsdenac} and \eqref{f2bsdenac} gives us  the following partition
\begin{align}
    &  \Bb{H} = \eusm{E}^{\infty}_{\!{\tn{\frown}}}\  \sqcup
   \bigsqcup\limits_{{\fo{n_{0} \in \Bb{Z}}}} \Big(2 n_{0}+
  \widehat{\eusm{E}}^{\hspace{0,025cm}0}_{\!{\tn{\frown}}} \Big)
 \ \  \sqcup \!\!\!\!\!\!
 \bigsqcup\limits_{
  {\fo{\eufm{n}\in  \Bb{Z}_{\neq 0}^{\hspace{0,02cm}\Bb{N}_{\hspace{-0,02cm}\eurm{f}}}\,, \ n_{0} \in \Bb{Z}  }} }
\Big( 2 n_{0} + \widehat{\eusm{E}}^{\hspace{0,025cm}\eufm{n}}_{\!{\tn{\frown}}}\Big),
 \label{f6zdenac}
\end{align}

\vspace{-0,2cm}
\noindent which we refer to as the {\it{even rational  partition}} of $\Bb{H}$. Here, we
write
$\eusm{E}^{\infty}_{\!{\tn{\frown}}} = \Bb{H} \setminus \cup_{\, {\fo{m \in \Bb{Z}}}}\ (2m + \overline{\Bb{D}})$
and in accordance with \eqref{f2bsdenac}(b) and \eqref{f15bsdenac}, it follows from \eqref{f6zdenac} that (for notation, cf. \eqref{f1contgen})
\begin{align}\label{f6wdenac}
    &  \Bb{H} \big\backslash  S_{\!{\tn{\frown}}}^{\infty}= \eusm{E}^{\infty}_{\!{\tn{\frown}}}\  \sqcup
   \bigsqcup\limits_{{\fo{n_{0} \in \Bb{Z}}}} \Big(2 n_{0}+
  \eusm{E}^{0}_{\!{\tn{\frown}}} \Big)
 \ \  \sqcup \!\!\!\!\!\!
 \bigsqcup\limits_{
  {\fo{\eufm{n}\in  \Bb{Z}_{\neq 0}^{\hspace{0,02cm}\Bb{N}_{\hspace{-0,02cm}\eurm{f}}}\,, \ n_{0} \in \Bb{Z}  }} }
\Big( 2 n_{0} + \eusm{E}^{\eufm{n}}_{\!{\tn{\frown}}}\Big).
\end{align}

\vspace{-0,2cm}
\noindent The following property is obtained in Section~\ref{pbsdentpth2}.
\begin{theorem}\hspace{-0,15cm}{\bf{.}}\label{bsdentpth3} The set
$\{\phi ( \Bb{H}\cap \partial\Bb{D} )\}_{\phi \in \Gamma_{\vartheta}}$ generates the even rational
partition \eqref{f6zdenac} of $\Bb{H}$.
\end{theorem}

By intersecting the both sides of \eqref{f6zdenac} with the closed unit disk $\overline{\Bb{D}}$, we get
\begin{align}\label{f17bsdenac}
    &
\overline{ \Bb{D}}_{\im > 0}= \widehat{\eusm{E}}^{\hspace{0,025cm}0}_{\!{\tn{\frown}}}\ \  \sqcup \ \
 \bigsqcup\nolimits_{
  {\fo{\ \eufm{n}\in  \Bb{Z}_{\neq 0}^{\hspace{0,02cm}\Bb{N}_{\hspace{-0,02cm}\eurm{f}}}  }} } \ \   \widehat{\eusm{E}}^{\hspace{0,025cm}\eufm{n}}_{\!{\tn{\frown}}} \, ,
\end{align}

\noindent
and for each $\eufm{n}\in \Bb{Z}_{\neq 0}^{\hspace{0,02cm}\Bb{N}_{\hspace{-0,02cm}\eurm{f}}}$ we can apply $\phi_{\eufm{n}}$ to both sides of  \eqref{f17bsdenac}, to obtain, in view of \eqref{f9contgendef1}  and $\phi_{{\fo{\,\eufm{n}}}}(\overline{ \Bb{D}}_{\im > 0}\big)  \!=\!
   \gamma_{\eufm{n}}(-1,1) \sqcup   \phi_{{\fo{\,\eufm{n}}}}\big(\Bb{D}_{\im > 0})$,
\begin{align}\label{f18bsdenac}
    &\hspace{-0,25cm}  \phi_{{\fo{\eufm{n}}}}\big(\overline{\Bb{D}}_{\im > 0}\big)=\widehat{\eusm{E}}^{\hspace{0,025cm}\eufm{n}}_{\!{\tn{\frown}}}\ \sqcup \
 \bigsqcup\nolimits_{\
  {\fo{\eufm{m} \in \Bb{Z}_{\neq 0}}} }
\
\widehat{\eusm{E}}^{\eufm{n}, \eufm{m}}_{\!{\tn{\frown}}}
\ , \ \eufm{n}:=(n_{N}, ..., n_{1})\!\in\!  \Bb{Z}_{\neq 0}^{N} \ , \ N \in \Bb{N}.
\end{align}

\noindent This for arbitrary $N \in \Bb{N}$ shows that  (cf. \cite[p.\! 44]{bon})
\begin{align*}
\bigsqcup\limits_{  {\fo{\eufm{n} \in \Bb{Z}^{N}_{\neq 0}}} } \phi_{{\fo{\eufm{n}}}}\big(\overline{\Bb{D}}_{\im > 0}\big)  =
\bigsqcup\limits_{ \ {\fo{\eufm{n} \in \Bb{Z}^{N}_{\neq 0}}} }\Big(\widehat{\eusm{E}}^{\hspace{0,025cm}\eufm{n}}_{\!{\tn{\frown}}}\ \  \sqcup \!\!
 \bigsqcup\limits_{
 {\fo{\eufm{m} \in \Bb{Z}^{\hspace{0,02cm}\Bb{N}_{\hspace{-0,02cm}\eurm{f}}}_{\neq 0}}} }
\widehat{\eusm{E}}^{\eufm{n}, \eufm{m}}_{\!{\tn{\frown}}}
\Big)  =
\bigsqcup\limits_{
  {\fo{\eufm{m}\in \Bb{Z}^{M}_{\neq 0}}} \,,\ {\fo{M\geqslant N  }} }
\widehat{\eusm{E}}^{\hspace{0,025cm}\eufm{m}}_{\!{\tn{\frown}}} \ ,
\end{align*}

\noindent so that in view of \eqref{f11contgen},
\begin{align*}
     & \hspace{-0,25cm}\bigsqcup\limits_{
 {\fo{ d (\phi)\geqslant N  }} \,, \ {\fo{\phi\in \Gamma_{\vartheta}^{\hspace{0,015cm}{\tn{||}}}}}\cup \{\phi_{0}\}  } \hspace{-0,5cm} \phi
\left(\widehat{\eusm{E}}^{\hspace{0,025cm}0}_{\!{\tn{\frown}}}\right) = \bigsqcup\limits_{
 {\fo{ d (\phi)= N  }} \,, \ {\fo{\phi\in \Gamma_{\vartheta}^{\hspace{0,015cm}{\tn{||}}}}}\cup \{\phi_{0}\}   }
\hspace{-0,5cm}\phi \big(\overline{\Bb{D}}_{\im > 0}\big) , \ \  N\in \Bb{Z}_{\geqslant 0}\,,\hspace{-0,1cm}
\end{align*}

\noindent where for $N\!=\!0$ this splitting coincides with \eqref{f17bsdenac}.
In accordance with \eqref{f11contgen}, \eqref{f5denac} and \eqref{f6zdenac}, we can associate with each point $z\!\in\! \Bb{H}$ the corresponding {\it{even rational  height}}
\begin{align}\label{f16xcontgen}
    & \hspace{-0,5cm} \eurm{h}_{\eusm{E}} (z) :=
\left\{\begin{array}{ll}
  0  \, , \    &    \ \mbox{if} \ z \in \eusm{E}^{\infty}_{\!{\tn{\frown}}}=
\Bb{H} \setminus \cup_{\, {\fo{m \in \Bb{Z}}}}\ \left(2m +
\overline{\Bb{D}}\right) \, , \    \\[0,15cm]
   1 + d (\phi) \, , \    &
\begin{displaystyle}
 \ \mbox{if} \ z \in   \bigsqcup\nolimits_{{\ \fo{n_{0} \in \Bb{Z}}}}
\ \Big(2 n_{0}+
  \phi
\big(\widehat{\eusm{E}}^{\hspace{0,025cm}0}_{\!{\tn{\frown}}}\big) \Big)\, , \ \phi\in \Gamma_{\vartheta}^{\hspace{0,015cm}{\tn{||}}}\cup \{\phi_{0}\}\,.
\end{displaystyle}
   \end{array}\right.\hspace{-0,7cm}
\end{align}

\noindent The results established  in \cite[p.\! 44]{bon} imply the asymptotics
\begin{align*}
    &  \int\nolimits_{-1}^{1}\eurm{h}_{\eusm{E}} (x+ iy )  d x = 2 \pi^{-2}
    \log^{2}(1/ y) + {\rm{O}} \big(\log (1/ y)\big) \, ,
\qquad  0<y \to 0 \,.
\end{align*}

\vspace{0,1cm}
Finally, we notice that the inclusion $\eusm{E}^{0}_{\!{\tn{\frown}}}\! \subset\! \Bb{H}_{|\re|< 1} $, taken together with the properties \eqref{f8bsdenac}(b) and \eqref{f15contgen}, shows that{\hyperlink{r25}{${}^{\ref*{case25}}$}}\hypertarget{br25}{}
\begin{align}\label{f16contgen}
    &    \begin{array}{lrllrcll}
    {\rm{(a)}} \   &
\Bb{G}_{2}  \big(\eusm{E}^{0}_{\!{\tn{\frown}}} \big)& \! = \!  &
        \mathcal{F}^{\,{\tn{||}}}_{{\tn{\square}}} \setminus\overline{\Bb{D}} \,,
           &\Bb{G}_{2}\big(\gamma(-1,1)\big)& \! = \!  &\gamma(-1,1)\,;
          &   \\[0,3cm]
    {\rm{(b)}} \   &
\Bb{G}_{2}^{N}  \big( \eusm{E}^{\hspace{0,05cm}\eufm{n}}_{\!{\tn{\frown}}} \big)& \! = \!  &
      \eusm{E}^{0}_{\!{\tn{\frown}}} \ ,
      &\Bb{G}_{2}^{N}\big(\gamma_{\eufm{n}}(-1,1)\big)& \! = \!  &\gamma(-1,1),
        &\  \mbox{if} \ N \! \geqslant\!  1\,;  \\[0,3cm]
{\rm{(c)}} \   &
\Bb{G}_{2}^{k}  \big( \eusm{E}^{\hspace{0,05cm}\eufm{n}}_{\!{\tn{\frown}}} \big)  & \! = \!  & \eusm{E}^{\hspace{0,05cm}\eufm{n}_{k}}_{\!{\tn{\frown}}}\,,
&\Bb{G}_{2}^{k}\big(\gamma_{\eufm{n}}(-1,1)\big)& \! = \!  &
\gamma_{\eufm{n}_{k}}(-1,1),
&\  \mbox{if} \ 1 \! \leqslant\!  k \! \leqslant\!  N\! -\! 1, \, N\! \geqslant\! 2  \,.
      \end{array}\hspace{-1cm}
\end{align}

\noindent
 Here,
$\eufm{n} =(n_{N}, ..., n_{1})\!\in \!\Bb{Z}_{\neq 0}^{N}$
for a given $N\! \in\! \Bb{N}$,
 and we write
$\eufm{n}_{k} \!:= \!(n_{N-k}, ..., n_{1})$ for $k$ with $1\!\leqslant\! k\! \leqslant \!N\!-\!1$, $N\! \geqslant\! 2$. We observe that it follows from \eqref{f5int}, Lemma~\ref{lemoneto} and \eqref{f2int} that
\begin{align}\label{f3contgen}
    &
\begin{array}{lrcccl}
 {\rm{(a)}} & \,  \lambda \left( \fet^{\,{\tn{||}}} \setminus\overline{\Bb{D}}\,\right) &  =  & \lambda \left( \eusm{E}^{\infty}_{\!{\tn{\frown}}} \right) & = &    \Bb{C}_{\,\re < 1/2}\setminus\{0\}\,;
\\[0,3cm]
  {\rm{(b)}}& \,  \lambda \big(\eusm{E}^{0}_{\!{\tn{\frown}}}\big) &  =  &
\lambda \left(-1/ \eusm{E}^{\infty}_{\!{\tn{\frown}}} \right) &  = &
    \Bb{C}_{\,\re > 1/2}\setminus\{1\}
\,.
\end{array}
\end{align}



\section[\hspace{-0,30cm}. \hspace{0,11cm}Analytic continuation of the generating function]{\hspace{-0,095cm}{\bf{.}} Analytic continuation of the generating function}
\label{contgen}$\phantom{a}$ \vspace{-0,5cm}

 For arbitrary $N \!\in \!\Bb{N}$ and $\eufm{n}\!=\!(n_{N}, ..., n_{1})\!\in\! \Bb{Z}_{\neq 0}^{N}$
 let $d(\eufm{n} )\!:= \!d(\phi_{\eufm{n}})\! =\! N$ (see \eqref{f11contgen}),
\begin{align}\label{f4bsdenac}
    & \hspace{-0,3cm}  \psi_{{\fo{\eufm{n}}}} (z)\! :=\!  \psi_{{\fo{n_{N}, ..., n_{1}}}} (z)\! :=\!
{1}\big/{\phi_{{\fo{n_{1}, ..., n_{N}}}}(1/z)} \, , \ \ \ \psi_{0} (z) := z\,,
\  \  z \!\in \!\Bb{H} \,,\hspace{-0,1cm}
\end{align}

\noindent  and $d(0):= 0$.  It can easily be shown from \eqref{f1bsden} that
 \begin{align}\label{f5bsdenac}
     &
\begin{array}{l}
 {1}\big/{\phi_{{\fo{n_{1}, ..., n_{N}}}}(1/z)} = - {1}\big/{\phi_{{\fo{-n_{1}, ..., -n_{N}}}}(-1/z)} \, , \quad \ \, (n_{1}, ..., n_{N})\in \Bb{Z}_{\neq 0}^{\hspace{0,02cm}\Bb{N}_{\hspace{-0,02cm}\eurm{f}}}\,,
\\[0,15cm]
 \psi_{{\fo{\eufm{n}}}}\big(\phi_{{\fo{\eufm{n}}}} (z) \big)=
     \phi_{{\fo{\eufm{n}}}}\big(\psi_{{\fo{\eufm{n}}}} (z) \big)=z\,, \qquad
     z \in \Bb{H} \,, \quad \psi_{{\fo{\eufm{n}}}} \! \in\! \Gamma_{\vartheta}\,, \quad  \eufm{n}\in \Bb{Z}_{\neq 0}^{\hspace{0,02cm}\Bb{N}_{\hspace{-0,02cm}\eurm{f}}}\cup\{0\}\,,
\end{array}
 \end{align}

\noindent and the  identities  \eqref{f8bsdenac}(b) can be written as
\begin{align}\label{f16zcontgen}
    &   \psi{{\fo{n_{N}, ..., n_{1}}}}  \big( \eusm{E}^{\hspace{0,05cm}n_{N}, ...,\, n_{1}}_{\!{\tn{\frown}}} \big)=
      \eusm{E}^{0}_{\!{\tn{\frown}}}= - 1/\eusm{E}^{\infty}_{\!{\tn{\frown}}} \ , \quad  n_{N}, ..., n_{1}\in  \Bb{Z}_{\neq 0}\,, \ \ N \in \Bb{N}\,.
\end{align}

\noindent The formulas \eqref{f15contgen} for  arbitrary $n_{1}, n_{2}, ..., n_{N}\in \Bb{Z}_{\neq 0}$ and $N\! \in \!\Bb{N}$  take the form{\hyperlink{r22}{${}^{\ref*{case22}}$}}\hypertarget{br22}{}
\begin{align}\label{f2bsdentppre}
    & \hspace{-0,3cm}  \Bb{G}^{N}_{2} (y)\! = \! \psi{{\fo{n_{N}, ..., n_{1}}}} (y)  , \ N
\! \geqslant\! 1, \
 \Bb{G}^{k+1}_{2} (z)\! = \! \psi{{\fo{n_{N}, ..., n_{N-k}}}} (z) , \
 0 \! \leqslant\!  k \! \leqslant\!  N\! -\! 2 \,,\hspace{-0,1cm}
\end{align}

\noindent  provided that $ y\! \in \!\phi_{{\fo{n_{N}, ..., n_{1}}}} \left(\Bb{H}_{|\re|< 1}\right)$ and  $ z \!\in\! \phi_{{\fo{n_{N}, ..., n_{1}}}} \left(\Bb{H}_{|\re|\leqslant 1}\right)$.
In addition, according to \eqref{f2bsden},  \eqref{f5bsden}, \eqref{f3zbsdenac}$\vphantom{A^{A^{A}}}$ and \eqref{f5bsdenac},  we  have
\begin{align*}
    & \psi_{{\fo{\eufm{n}}}} \big(\fetr^{\hspace{0,01cm}\eufm{n}}\big) = \fet^{\,{\tn{|}}\sigma_{\eufm{n}}}
\, , \quad
\psi_{{\fo{\eufm{n}}}} \big(\fetur^{\hspace{0,01cm}\eufm{n}}\big) =
  \fetu^{\,{\tn{|}}\sigma_{\eufm{n}}} \, , \quad  \psi_{{\fo{\eufm{n}}}} \big(\fetd^{\hspace{0,01cm}\eufm{n}}\big) =
 \fetd \ , \quad  \eufm{n}\in \Bb{Z}_{\neq 0}^{\hspace{0,02cm}\Bb{N}_{\hspace{-0,02cm}\eurm{f}}} \, .
\end{align*}

\noindent Moreover, we observe that \eqref{f14bsdenac} gives
\begin{align}\label{f9bsdenac}
    &  \lambda^{\,\prime} (z) = (-1)^{d(\phi )}
      \lambda^{\,\prime}(\phi (z))\phi^{\,\prime}(z)\,,  \  \quad   \phi\! \in\! \Gamma_{\vartheta}  \, , \  z\in \Bb{H}\,.
\end{align}

We let $x\! \in\! \Bb{R} $ be fixed, and consider the following two functions of $z$,
\begin{align}\label{f1bsdenac}
     &  \hspace{-0,15cm}
\Phi^{\,\delta}_{\infty}(x;z) := \frac{1}{2\pi i}
\int\limits_{{\fo{\gamma (-1,1)}}}\dfrac{\lambda^{\,\prime}(z){\rm{d}} \zeta}{\big(\lambda(z)- \lambda(\zeta)\big)\big(x^{\delta}\zeta-(-x)^{1-\delta}\big)^{2}} \ \, , \quad \delta \in \{0,1\} \, . \hspace{-0,1cm}
\end{align}

\noindent
By  \eqref{f7denac}, \eqref{f1contgen} and \eqref{f3zint}, the functions $\Phi^{\hspace{0,05cm}0}_{\infty}(x;z)$ and $\Phi^{1}_{\infty}(x;z)$ are  holomorphic at each point of
$\Bb{H}\setminus S_{\!{\tn{\frown}}}^{\infty}$, and, in view of \eqref{f3genfunbs}
and \eqref{f4genfunbs},{\hyperlink{r24}{${}^{\ref*{case24}}$}}\hypertarget{br24}{}
\begin{align}\label{f1wbsdenac}
    & \hspace{-0,25cm}  \sum\limits_{n \geqslant 1}
  \! n \,\eurm{H}_{n} (x)\, {\rm{e}}^{{\fo{{\rm{i}}\pi  n  z}} }\!
 = \! \frac{\Phi^{\hspace{0,02cm}0}_{\infty}(x;z)}{2 \pi^{2} }\, ,
\ \ \sum\limits_{n \geqslant 1}
  \! n \,\eurm{M}_{n} (x)\, {\rm{e}}^{{\fo{{\rm{i}}\pi  n  z}} }\!  = \! \frac{\Phi^{\hspace{0,02cm}1}_{\infty}(x;z)}{2 \pi^{2} }\, ,   \ \ z\in \Bb{H}_{\,\im > 1}\,.\hspace{-0,1cm}
\end{align}

\noindent
Furthermore,  \eqref{f1zbsdenac}  and
 \eqref{f9bsdenac} imply that{\hyperlink{r23}{${}^{\ref*{case23}}$}}\hypertarget{br23}{}
\begin{align}\label{10bsdenac}
    &  (-1)^{d(\phi)} \Phi^{\hspace{0,02cm}\delta}_{\infty}(x, z) =
\left\{\begin{array}{ll}
\phi^{\,\prime} (z)\,\Phi^{\hspace{0,02cm}\delta}_{\infty}(x;\phi (z))  \, , & \quad  \mbox{if} \   d(\phi) \in  \{0\}\cup 2 \Bb{N}\,, \\[0,25cm]
 \phi^{\,\prime} (z) \, \Phi^{\hspace{0,02cm}1-\delta}_{\infty}(x;\phi (z))  \, ,   &  \quad   \mbox{if} \    d(\phi) \in 2 \Bb{N}\!-\!1\,, \end{array}\right.
\end{align}

\noindent holds for all  $z\in \Bb{H}\setminus S_{\!{\tn{\frown}}}^{\infty}$,
$\phi \in \Gamma_{\vartheta}$ and $ \delta \!\in \! \{0,1\}$.
  It follows from \eqref{f2bsdenac} that
 $\Phi^{\hspace{0,025cm}\delta}_{\infty}(x;z)$ is holomorphic $2$-periodic function on the set $\eusm{E}^{\infty}_{\!{\tn{\frown}}}= \Bb{H} \setminus \cup_{\, {\fo{m \in \Bb{Z}}}}\ (2m + \overline{\Bb{D}})$. The goal of  this section
is to prove that for every $ \delta \!\in \! \{0,1\}$ the function $\Phi^{\hspace{0,025cm}\delta}_{\infty}(x;z)$ can  be analytically extended from this set
to $\Bb{H}$. To do this, it suffices to show that such an extension is possible from the set $\eusm{E}^{\infty}_{\!{\tn{\frown}}}$ to $ \Bb{H}_{|\re |< 1}\cup \eusm{E}^{\infty}_{\!{\tn{\frown}}}$, because then the desired extension $\Phi^{\hspace{0,025cm}\delta}_{\Bb{H}}(x;z)$ on the remaining set $  \cup_{n\in\Bb{Z}_{\neq 0} } (2n + \overline{\Bb{D}}_{\im > 0})$ can be constructed from the resulting extension $\Phi^{\hspace{0,025cm}\delta}_{\,{\tn{||}}}(x;z)$ by the formula
\begin{align}\label{11bsdenac}
    &  \Phi^{\hspace{0,025cm}\delta}_{\Bb{H}}(x; z\! +\! 2n)\! := \!\Phi^{\hspace{0,025cm}\delta}_{\,{\tn{||}}}(x;z) \ , \quad  z\! \in \!
\overline{\Bb{D}}_{\im > 0}\, , \  \ \ n\! \in\! \Bb{Z}_{\neq 0} \, , \quad \delta \in \{0,1\} \, .
\end{align}

\subsection[\hspace{-0,31cm}. \hspace{0,11cm}Auxiliary lemmas.]{\hspace{-0,11cm}{\bf{.}} Auxiliary lemmas.
}\label{bsdenancont} In view of Lemma~\ref{lemoneto}, for every $x \!\in\! \Bb{R}$, $z\!\in\! \fet$, $\im\, z \!\neq \!2$, we introduce the function $\Phi^{\hspace{0,02cm}\delta}_{{\tn{\sqcap}}}(x;z)$ as the integral \eqref{f1bsdenac},  where $\gamma(-1,1)$ is replaced by the oriented contour $\Pi(-1,1)$,  which passes from $-1$ to $1$  along the polygonal contour $(-1,-1+2{\rm{i}}]\cup [-1+2{\rm{i}}, 1+2{\rm{i}}] \cup [1+2{\rm{i}}, 1)$.

\begin{lemsectseven}\hspace{-0,17cm}{\bf{.}}\label{bsdenaclem1}
    For each  $x \!\in\! \Bb{R}$ and $ \delta \!\in \! \{0,1\}$ the function $\Phi^{\hspace{0,02cm}\delta}_{\infty}(x;z)$ of $z$ can  be analytically   ex\-ten\-ded from $\eusm{E}^{\infty}_{\!{\tn{\frown}}} = \Bb{H} \setminus \cup_{\, {\fo{m \in \Bb{Z}}}}\ (2m + \overline{\Bb{D}})$  to $\eusm{E}^{\infty}_{\!{\tn{\frown}}}\sqcup \gamma (-1,1)\sqcup  \eusm{E}^{0}_{\!{\tn{\frown}}}$ such that the resulting extension
     $\Phi_{0}^{\hspace{0,02cm}\delta}(x, z)$ satisfies
    \begin{align}\label{f1bsdenaclem1}
    &  \Phi_{0}^{\hspace{0,02cm}\delta}(x, z)\! =\!
    \left\{\begin{array}{ll} \Phi^{\hspace{0,02cm}\delta}_{\infty}(x;z) \, ,    &  \quad   \ \mbox{if} \ \   z\in
 \eusm{E}^{\infty}_{\!{\tn{\frown}}}  \, ;
\\[0,2cm]    \Phi^{\hspace{0,02cm}\delta}_{\infty}(x;z)  -
\big(x^{\delta}z\!-\!(-x)^{1-\delta}\big)^{-2}
 \, ,  &  \quad \ \mbox{if} \  \  z\in  \eusm{E}^{0}_{\!{\tn{\frown}}}\,.
\end{array}\right.
\end{align}
\end{lemsectseven}

\begin{proof}\hspace{-0,15cm}{{.}}
By  transforming the contour $\gamma(-1,1)$ of integration in \eqref{f1bsdenac} to $\Pi(-1,1)$ and using
 Lemma~\ref{lemoneto},  we obtain from the residue theorem \cite[p.\! 112]{con} that
\begin{align}\label{f1bsdenaclem1p}
    &
\begin{array}{ll}
  \Phi^{\hspace{0,02cm}\delta}_{\infty}(x;z)\! =\!\Phi^{\hspace{0,02cm}\delta}_{{\tn{\sqcap}}}(x;z) \, ,   &   \quad
z\in \fetd  \sqcup  \Bb{H}_{\im >2} ,    \\[0,2cm]
   \Phi^{\hspace{0,02cm}\delta}_{\infty}(x;z)\! =\!\Phi^{\hspace{0,02cm}\delta}_{{\tn{\sqcap}}}(x;z) - \big(x^{\delta}z\!-\!(-x)^{1-\delta}\big)^{-2}\, ,   &    \quad
z\in \fetu \cap \Bb{H}_{\im <2}\,  .
\end{array}
\end{align}

\noindent This means that the function $\Phi^{\hspace{0,02cm}\delta}_{{\tn{\sqcap}}}(x;z)  -   (x^{\delta}z - (-x)^{1-\delta})^{-2}$, being holomorphic on
$\Bb{H}_{\im <2}\cap \fet $, coincides  on the set  $\Bb{H}_{\im <2} \cap \fetu  \subset  \Bb{H}_{\im <2}\cap \fet$  with the  function $\Phi^{\hspace{0,02cm}\delta}_{\infty}(x;z)$, which is  holomorphic on $\fetu  \supset  \Bb{H}_{\im <2} \cap \fetu$. By the uniqueness theorem  for analytic functions (see \cite[p.\! 78]{con}),  we find that
the function $\Phi^{\hspace{0,02cm}\delta}_{\infty}(x;z)$  can  be analytically   ex\-ten\-ded from $\fetu$  to $\mathcal{F}_{{\tn{\square}}}  =  \fetu \cup
(\Bb{H}_{\im <2}\cap \fet)$.   Moreover, for $z \in  \fetd \subset \Bb{H}_{\im <2}\cap \fet$, the resulting
extension equals the expression  $H(z):=\Phi^{\hspace{0,02cm}\delta}_{\infty}(x;z) - (x^{\delta}z-(-x)^{1-\delta})^{-2}$ since
 $\Phi^{\hspace{0,02cm}\delta}_{{\tn{\sqcap}}}(x;z)=
\Phi^{\hspace{0,02cm}\delta}_{\infty}(x;z) $ holds for all  $z\in
\fetd $, in view of \eqref{f1bsdenaclem1p}. But by  \eqref{f8bsdenac}(a) and \eqref{f6wdenac}, we see that $\fetd \!\subset \!\eusm{E}^{0}_{\!{\tn{\frown}}}\!\subset\!\Bb{H} \big\backslash  S_{\!{\tn{\frown}}}^{\infty}$, and
since $\eusm{E}^{0}_{\!{\tn{\frown}}}$ is simply connected it follows that
  the latter function $H(z)$  is actually holomorphic on $\eusm{E}^{0}_{\!{\tn{\frown}}}$.
  This proves  \eqref{f1bsdenaclem1} and completes the proof
of Lemma~\ref{bsdenaclem1}.
$\square$
\end{proof}

\begin{lemsectseven}\hspace{-0,17cm}{\bf{.}}\label{bsdenaclem2}
 Let  $x \!\in\! \Bb{R}$, $N \!\in\! \Bb{N}$, $\eufm{n}\!=\!(n_{N}, ..., n_{1})\in \Bb{Z}_{\neq 0}^{N}$, $ \delta \!\in \! \{0,1\}$ and the sets
$\fetd^{\hspace{0,01cm}\eufm{n}}$, $\getu^{\hspace{0,01cm}\eufm{n}}$  be defined as in \eqref{f5bsden}.
 The function $\Phi^{\hspace{0,02cm}\delta}_{\infty}(x;z)$ of $z$ can be  analytically  ex\-ten\-ded from $\getu^{\hspace{0,01cm}\eufm{n}}$  to $\get^{\hspace{0,01cm}\eufm{n}} =
\getu^{\hspace{0,01cm}\eufm{n}} \sqcup \gamma_{\eufm{n}} (-1, 1)  \sqcup \fetd^{\hspace{0,01cm}\eufm{n}}$ such that the resulting extension
     $\Phi_{\eufm{n}}^{\hspace{0,02cm}\delta}(x, z)$ satisfies
    \begin{align}\label{f1bsdenaclem2}
    &  \Phi_{\eufm{n}}^{\hspace{0,02cm}\delta}(x, z)\! =\!
    \left\{\begin{array}{ll} \Phi^{\hspace{0,02cm}\delta}_{\infty}(x;z) \, ,    &  \quad   \ \mbox{if} \ \   z\in
 \getu^{\hspace{0,01cm}\eufm{n}}  \, ;
\\[0,2cm]    \Phi^{\hspace{0,02cm}\delta}_{\infty}(x;z)  -
\Delta^{\delta}_{{\fo{\eufm{n}}}} (x;z)
 \, ,  &  \quad \ \mbox{if} \  \  z\in   \fetd^{\hspace{0,01cm}\eufm{n}}\,.
\end{array}\right.
\end{align}

\noindent where the holomorphic at each $z\!\in\! \Bb{H}$ function  $\Delta^{\delta}_{{\fo{\eufm{n}}}}(x;z)$  is defined as
\begin{align}\label{f2bsdenaclem2}
   & \hspace{-0,4cm} \Delta^{\delta}_{{\fo{\eufm{n}}}} (x;z)\!:= \! \left\{\begin{array}{ll}
\hphantom{-} \psi^{\,\prime}_{{\fo{\eufm{n}}}} (z)\big(x^{\delta}\psi_{\eufm{n}} (z)-(-x)^{1-\delta}\big)^{-2} \, ,
&\quad  N \in 2 \Bb{N}\,,   \\[0,2cm]
- \psi^{\,\prime}_{{\fo{\eufm{n}}}} (z)\big(x^{1-\delta}\psi_{\eufm{n}} (z)-(-x)^{\delta}\big)^{-2} \, ,
&  \quad   N \!\in \!2 \Bb{N}\!-\!1\,.   \end{array}
    \right.
\end{align}

\end{lemsectseven}

\begin{proof}\hspace{-0,15cm}{{.}}
We put $\delta_{N}:= \delta$ if $N$ is even, while $\delta_{N}:= 1-\delta$ if $N$ is odd. Let the function $z \mapsto \Phi_{0}^{\hspace{0,02cm}\delta_{N}}(x, z) \in
{\rm{Hol}} (\eusm{E}^{\infty}_{\!{\tn{\frown}}}\sqcup \gamma (-1,1)\sqcup  \eusm{E}^{0}_{\!{\tn{\frown}}})$  be given as in  Lemma~\ref{bsdenaclem1} with $\delta = \delta_{N}$.
By \eqref{f4denac}, \eqref{f5bsden} and \eqref{f5bsdenac}, we have
$\psi_{\eufm{n}}(\get^{\hspace{0,01cm}\eufm{n}}) = \fet$, and
$\fet \subset \eusm{E}^{\infty}_{\!{\tn{\frown}}}\sqcup \gamma (-1,1)\sqcup  \eusm{E}^{0}_{\!{\tn{\frown}}}$, in view of \eqref{f8bsdenac}(a).
It now follows that $z \mapsto \Phi_{0}^{\hspace{0,02cm}\delta_{N}}(x, \psi_{\eufm{n}} (z))$ is in
${\rm{Hol}} (\get^{\hspace{0,01cm}\eufm{n}})$. By   \eqref{10bsdenac},
we obtain from \eqref{f1bsdenaclem1} that, for $z\in \getu^{\hspace{0,01cm}\eufm{n}}$,
\begin{align*}
     &\hspace{-0,25cm} \Phi_{0}^{\hspace{0,02cm}\delta_{N}}(x, \psi_{\eufm{n}}(z))\!=\!
      \left\{\begin{array}{ll}
\Phi^{\hspace{0,02cm}\delta}_{\infty}(x;\psi_{\eufm{n}}(z))=
\psi_{\eufm{n}}^{\, \prime}(z)^{-1} \Phi^{\hspace{0,02cm}\delta}_{\infty}(x;z)\, ,
&\quad  N \in 2 \Bb{N}\,,   \\[0,1cm]
\Phi^{\hspace{0,02cm}1-\delta}_{\infty}(x;\psi_{\eufm{n}}(z))=
-\psi_{\eufm{n}}^{\, \prime}(z)^{-1} \Phi^{\hspace{0,02cm}\delta}_{\infty}(x;z) \, ,
&  \quad   N \!\in \!2 \Bb{N}\!-\!1\,.   \end{array}
    \right.\hspace{-0,2cm}
\end{align*}
\noindent
On the other hand, if  $z\in \fetd^{\hspace{0,01cm}\eufm{n}}$ then
\begin{align*}
     &\hspace{-0,25cm}\Phi_{0}^{\hspace{0,02cm}\delta_{N}}(x, \psi_{\eufm{n}}(z))\!=\!
      \left\{\begin{array}{lll}
\psi_{\eufm{n}}^{\, \prime}(z)^{-1} \Phi^{\hspace{0,02cm}\delta}_{\infty}(x;z)\!-\!(x^{\delta}\psi_{\eufm{n}}(z)\!-\!
      (-x)^{1-\delta})^{-2}\! ,
&\,  N \in 2 \Bb{N}\,,   \\[0,1cm]
\!-\psi_{\eufm{n}}^{\, \prime}(z)^{-1}\Phi^{\hspace{0,02cm}\delta}_{\infty}(x;z)\!-\!
(x^{1-\delta}\psi_{\eufm{n}}(z)\!-\!(-x)^{\delta})^{-2}
  \! ,
&  \,   N \!\in \!2 \Bb{N}\!-\!1\,.   \end{array}
    \right.\hspace{-0,2cm}
\end{align*}

  \noindent Hence the function $\Phi_{\eufm{n}}^{\hspace{0,02cm}\delta}(x, z):=(-1)^{N}\psi_{\eufm{n}}^{\, \prime}(z)\Phi_{0}^{\hspace{0,02cm}\delta_{N}}(x, \psi_{\eufm{n}}(z))$   is in
${\rm{Hol}} (\get^{\hspace{0,01cm}\eufm{n}})$, where $\get^{\hspace{0,01cm}\eufm{n}} =\getu^{\hspace{0,01cm}\eufm{n}} \sqcup \gamma_{\eufm{n}} (-1, 1)  \sqcup \fetd^{\hspace{0,01cm}\eufm{n}}$ and $
\Phi_{\eufm{n}}^{\hspace{0,02cm}\delta}(x, z)=\Phi^{\hspace{0,02cm}\delta}_{\infty}(x;z)$ for all $z\in \getu^{\hspace{0,01cm}\eufm{n}}$, while
\begin{align*}
     & \Phi_{\eufm{n}}^{\hspace{0,02cm}\delta}(x, z) = \Phi^{\hspace{0,02cm}\delta}_{\infty}(x;z) - (-1)^{N}(x^{\delta_{N}}\psi_{\eufm{n}}(z)\!-\!
      (-x)^{1-\delta_{N}})^{-2} \ , \quad z \in \fetd^{\hspace{0,01cm}\eufm{n}}\,.
\end{align*}

\noindent This shows that as a function of $z$, $\Phi^{\hspace{0,02cm}\delta}_{\infty}(x;z)$  can  be  analytically  ex\-ten\-ded from $\getu^{\hspace{0,01cm}\eufm{n}}$  to $\get^{\hspace{0,01cm}\eufm{n}}$, and the formulas \eqref{f1bsdenaclem2}, \eqref{f2bsdenaclem2} hold as well. Lemma~\ref{bsdenaclem2} follows.$\square$
\end{proof}

The desired analytical extension $\Phi^{\,\delta}_{\,{\tn{||}}}(x;z)$ of the function $\Phi^{\,\delta}_{\infty}(x;z)$ from the set $\eusm{E}^{\infty}_{\!{\tn{\frown}}}$ to  $ \Bb{H}_{|\re |< 1} \cup \eusm{E}^{\infty}_{\!{\tn{\frown}}}$ can  be easily constructed by combining  the results of  Lemmas~\ref{bsdenaclem1} and~\ref{bsdenaclem2}.
\begin{lemsectseven}\label{bsdenacth1}
    Let $x\in\Bb{R}$,  $ \delta \!\in \! \{0,1\}$   and $\Phi^{\,\delta}_{\infty}(x;z)$ be given by \eqref{f1bsdenac}. The function $z \mapsto \Phi^{\,\delta}_{\infty}(x;z)$ extends analytically
 from the set $\eusm{E}^{\infty}_{\!{\tn{\frown}}}$ to $\Bb{H}_{|\re |< 1} \cup\eusm{E}^{\infty}_{\!{\tn{\frown}}}$ such that the resulting extension  $\Phi^{\,\delta}_{{\tn{||}}}(x;z)$ on the set
\begin{align}\label{f2bsdenacth1}
    &  \left(\Bb{H}_{|\re |< 1} \cup\eusm{E}^{\infty}_{\!{\tn{\frown}}}\right)\big\backslash  S^{\infty}_{\!{\tn{\frown}}}= \eusm{E}^{\infty}_{\!{\tn{\frown}}}\  \sqcup
    \eusm{E}^{0}_{\!{\tn{\frown}}}\ \  \sqcup \ \
 \bigsqcup\nolimits_{
  {\fo{\ \eufm{n}\in  \Bb{Z}_{\neq 0}^{\hspace{0,02cm}\Bb{N}_{\hspace{-0,02cm}\eurm{f}}}  }} } \ \   \eusm{E}^{\eufm{n}}_{\!{\tn{\frown}}}
\end{align}

\vspace{-0,35cm}
\noindent
 satisfies
\begin{align}     &   \label{f4bsdenacth1} \Phi^{\,\delta}_{{\tn{||}}}(x;z)=\Phi^{\,\delta}_{\infty}(x; z) \, , \    &   &    \hspace{-0,5cm}   z\in \eusm{E}^{\infty}_{\!{\tn{\frown}}} \, , \   \\    &\label{f5bsdenacth1}
\Phi^{\,\delta}_{{\tn{||}}}(x;z)=\Phi^{\,\delta}_{\infty}(x; z)- \Delta^{\delta}_{0}(x; z) \, , \    &  & \hspace{-0,5cm}      z\in  \eusm{E}_{\!{\tn{\frown}}}^{\hspace{0,05cm}0}  \, , \
 \\[-0,3cm]    &  \Phi^{\,\delta}_{{\tn{||}}}(x;z)=\Phi^{\,\delta}_{\infty}(x; z) - \Delta^{\delta}_{0}(x; z)- \sum\limits_{k=0}^{N-1} \Delta^{\delta}_{n_{N}, n_{N-1}, ..., n_{N-k}}(x; z)
 \ ,   &    &
\label{f5zbsdenacth1}\end{align}

\noindent when $z\in \eusm{E}_{\!{\tn{\frown}}}^{\hspace{0,05cm}n_{N}, ...,\, n_{1}}$  for arbitrary $N \!\in\! \Bb{N}$ and $ \eufm{n} =(n_{N}, ..., n_{1})\!\in\!  \Bb{Z}_{\neq 0}^{N}$. Here, for every  $ \eufm{n}\!\in\!  \Bb{Z}_{\neq 0}^{N}$,
$\Delta^{\delta}_{{\fo{\eufm{n}}}} (x;z)$   is defined as in \eqref{f2bsdenaclem2} and
\begin{align}\label{f3bsdenacth1}
    & \Delta^{\delta}_{0}(x; z) := (x^{\delta}z - (-x)^{1-\delta})^{-2} \ , \quad  x\in\Bb{R}\,,  \  \delta \!\in \! \{0,1\} \,, \ z \in \Bb{H}\,.
\end{align}
\end{lemsectseven}

 {\emph{Proof.}} As follows from $\eusm{E}^{\infty}_{\!{\tn{\frown}}} = \Bb{H} \setminus \cup_{\, {\fo{m \in \Bb{Z}}}}\ (2m + \overline{\Bb{D}})$  and \eqref{f2bsdenac}, we have $\eusm{E}^{\infty}_{\!{\tn{\frown}}}\cap  S_{\!{\tn{\frown}}}^{\infty}=\emptyset$, and hence \eqref{f2bsdenacth1} is immediate from \eqref{f15bsdenac}.
Since  $\Phi^{\,\delta}_{\infty}(x; z)$ is holomorphic at each point of
$\Bb{H}\setminus S_{\!{\tn{\frown}}}^{\infty}$ we conclude from \eqref{f2bsdenacth1} that $\Phi^{\hspace{0,02cm}\delta}_{\,{\tn{||}}}(x;z)$  is holomorphic on $ \eusm{E}^{\eufm{n}}_{\!{\tn{\frown}}}$  for every $\eufm{n} \!\in \!\Bb{Z}_{\neq 0}^{\hspace{0,02cm}\Bb{N}_{\hspace{-0,02cm}\eurm{f}}} \cup\{0\} $.

What remains to prove is that $\Phi^{\hspace{0,02cm}\delta}_{\,{\tn{||}}}(x;z)$ extends analytically  across
$\gamma_{\eufm{n}}(-1,1)$ for each $\eufm{n} \!\in \!\Bb{Z}_{\neq 0}^{\hspace{0,02cm}\Bb{N}_{\hspace{-0,02cm}\eurm{f}}} \cup\{0\} $.
 In the case when $\eufm{n}=0$, this fact follows from
Lemma~\ref{bsdenaclem1} because,  in accordance with \eqref{f1bsdenaclem1}, $\Phi^{\hspace{0,02cm}\delta}_{\,{\tn{||}}}(x;z)=\Phi_{0}^{\hspace{0,02cm}\delta}(x, z) $ for all
$z \in \eusm{E}^{\infty}_{\!{\tn{\frown}}}\sqcup  \eusm{E}^{0}_{\!{\tn{\frown}}}$.
If $N \in \Bb{N}$ and $\eufm{n}:=(n_{N}, ..., n_{1})\!\in\!  \Bb{Z}_{\neq 0}^{N}$ then $\gamma_{\eufm{n}} (-1, 1)$ is the roof of  $\fetd^{\hspace{0,01cm}\eufm{n}}$ and is the lower arch of $\getu^{\hspace{0,01cm}\eufm{n}}$, by virtue of \eqref{f5bsden}.
According to  \eqref{f8bsdenac}(a), $\fetd^{\hspace{0,01cm}\eufm{n}}\subset \eusm{E}^{\eufm{n}}_{\!{\tn{\frown}}}$
and $\getu^{\hspace{0,05cm}n_{N}, ...,\, n_{1}}\subset \fetur^{\hspace{0,05cm}n_{N}, ...,\, n_{1}}\subset\eusm{E}^{\hspace{0,05cm}n_{N}, ...,\, n_{2}}_{\!{\tn{\frown}}}$, if $N \geqslant 2$, while $\getu^{\hspace{0,05cm}n_{1}}\subset \fetur^{\hspace{0,05cm}n_{1}}\subset\eusm{E}^{0}_{\!{\tn{\frown}}}$, if $N = 1$.
   Hence, in this case,  the required property follows from Lemma~\ref{bsdenaclem2}, because for each
$z \in \getu^{\hspace{0,01cm}\eufm{n}}  \sqcup \fetd^{\hspace{0,01cm}\eufm{n}}$,
the difference
$\Phi^{\hspace{0,02cm}\delta}_{\,{\tn{||}}}(x;z)-\Phi_{\eufm{n}}^{\hspace{0,02cm}\delta}(x, z)$   equals the function
\begin{align*} &
  - \Delta^{\delta}_{0}(x; z)- \sum\nolimits_{k=0}^{N-2} \Delta^{\delta}_{n_{N}, n_{N-1}, ..., n_{N-k}}(x; z)\,, \quad \sum\nolimits_{k=0}^{-1} := 0\,,
\end{align*}

\noindent which is
holomorphic on $\Bb{H}$, in view of \eqref{f2bsdenaclem2} and \eqref{f3bsdenacth1}.
Lemma~\ref{bsdenacth1} follows. $\square$

\vspace{0,25cm}\subsection[\hspace{-0,31cm}. \hspace{0,11cm}Main result.]{\hspace{-0,11cm}{\bf{.}} Main result.
}\label{bsdenanmain}
By  \eqref{f16bsdenac}, we obtain from \eqref{f17bsdenac} and \eqref{f18bsdenac} that
\begin{align}\label{f21bsdenac}
    &\hspace{-0,25cm}  \Bb{D}_{\,\im > 0}\setminus S^{\hspace{0,02cm}{\tn{||}}}_{\!{\tn{\frown}}}= \eusm{E}^{0}_{\!{\tn{\frown}}}\ \  \sqcup \ \
 \bigsqcup\limits_{
  {\fo{\ \eufm{m}\in  \Bb{Z}_{\neq 0}^{\hspace{0,02cm}\Bb{N}_{\hspace{-0,02cm}\eurm{f}}}  }} } \ \   \eusm{E}^{\eufm{m}}_{\!{\tn{\frown}}} \, , \\    &
\hspace{-0,25cm}  \phi_{{\fo{\eufm{n}}}}\big(\Bb{D}_{\,\im > 0}\big)\setminus S^{\hspace{0,02cm}{\tn{||}}}_{\!{\tn{\frown}}}=\eusm{E}^{\hspace{0,05cm}n_{N}, ...,\, n_{1}}_{\!{\tn{\frown}}}\ \sqcup \!\!\!\!\!\!
 \bigsqcup\limits_{\
  {\fo{m_{M}, ...,\, m_{1} \in \Bb{Z}_{\neq 0}\,, \ M \in \Bb{N}  }} }
\!\!\!\!\!\! \eusm{E}^{\hspace{0,05cm}n_{N}, ...,\, n_{1}, \,m_{M}, ...,\, m_{1}}_{\!{\tn{\frown}}}
\ ,
\label{f22bsdenac}\end{align}

\noindent where $N \in \Bb{N}$ and $\eufm{n}:=(n_{N}, ..., n_{1})\!\in\!  \Bb{Z}_{\neq 0}^{N}$. Consequently, for each $x\in\Bb{R}$  and  $ \delta \!\in \! \{0,1\}$, the formula for
$\Phi^{\,\delta}_{{\tn{||}}}(x;z)$ in Lemma~\ref{bsdenacth1}
can be written in the form
\begin{align*}
    &\hspace{-0,3cm}  \Phi^{\,\delta}_{{\tn{||}}}(x;z)\!=\!\Phi^{\,\delta}_{\infty}(x; z)\!-\!
\!\!\!\!\!\sum\limits_{{\fo{\eufm{n} \!\in \!\Bb{Z}_{\neq 0}^{\hspace{0,02cm}\Bb{N}_{\hspace{-0,02cm}\eurm{f}}} \cup\{0\}}}}\Delta^{\delta}_{{{\eufm{n}}}} (x;z)\, \chi_{{\fo{\phi_{{\fo{\eufm{n}}}}\left(\Bb{D}_{\,\im > 0}\right)}}}(z),
\quad z\in \Bb{H}_{|\re |\leqslant 1}\setminus S^{\hspace{0,02cm}{\tn{||}}}_{\!{\tn{\frown}}},
\hspace{-0,1cm}
\end{align*}

\noindent where, for every $z\in \Bb{H}_{|\re |\leqslant 1}\setminus S^{\hspace{0,02cm}{\tn{||}}}_{\!{\tn{\frown}}}$, the sum
has only finitely many nonzero  terms, the number of which equals (see \eqref{f16xcontgen}, \eqref{f21bsdenac} and \eqref{f22bsdenac}){\hyperlink{r5}{${}^{\ref*{case5}}$}}\hypertarget{br5}{}
\begin{align}\label{f24bsdenac}
&  \sum\limits_{{\fo{\eufm{n} \!\in \!\Bb{Z}_{\neq 0}^{\hspace{0,02cm}\Bb{N}_{\hspace{-0,02cm}\eurm{f}}} \cup\{0\}}}}\chi_{{\fo{\phi_{{\fo{\eufm{n}}}}\left(\Bb{D}_{\,\im > 0}\right)}}}(z)= \eurm{h}_{\eusm{E}} (z) \ , \quad z\in \Bb{H}_{|\re |\leqslant 1}\setminus S^{\hspace{0,02cm}{\tn{||}}}_{\!{\tn{\frown}}}\, .
\end{align}

Let   $N \!\in\! \Bb{Z}_{\geqslant 2}$,  $ \eufm{n} =(n_{N}, ..., n_{1})\!\in\!  \Bb{Z}_{\neq 0}^{N}$, $ \delta \!\in \! \{0,1\}$, and put
$\delta_{2n}:=\delta$, $\delta_{2n+1}:=1-\delta$,  $n\in \Bb{Z}_{\geqslant 0}$.
The equalities of sets in \eqref{f16zcontgen} and in \eqref{f8bsdenac}(b) lead to
\begin{align*}
     \psi{{\fo{n_{N}, ..., n_{N-k}}}}\big(\eusm{E}^{\hspace{0,05cm}n_{N}, ...,\, n_{1}}_{\!{\tn{\frown}}} \big)&=
\psi{{\fo{n_{N}, ..., n_{N-k}}}}\big(\phi{{\fo{n_{N}, ..., n_{N-k}}}}\big(\phi{{\fo{n_{N-k-1}, ..., n_{1}}}}\big(\eusm{E}^{0}_{\!{\tn{\frown}}}\big)\big)\big) \\    &    = \phi{{\fo{n_{N-k-1}, ..., n_{1}}}}\big(\eusm{E}^{0}_{\!{\tn{\frown}}}\big) = \eusm{E}^{\hspace{0,05cm}n_{N-k-1}, ...,\, n_{1}}_{\!{\tn{\frown}}} \ , \quad  0\! \leqslant\! k\! \leqslant\! N\!-\!2\,,
\end{align*}

\noindent and hence, in view of \eqref{f2bsdentppre}, we find that
\begin{align}\label{f25bsdenac}
    & \hspace{-0,25cm}
 \begin{array}{lll}
\psi_{{\fo{n_{N}, ..., n_{N-k}}}} (z) \!=\! \Bb{G}^{k+1}_{2} (z) ,  &  \  \ \ 0  \!\leqslant \! k  \!\leqslant\!  N\! -\! 1 \,,     & \ z\!\in\!  \eusm{E}^{n_{N}, ...,\, n_{1}}_{\!{\tn{\frown}}}  \, ,
 \\[0,3cm]    \Bb{G}^{k+1}_{2} \left(\eusm{E}^{n_{N}, ...,\, n_{1}}_{\!{\tn{\frown}}}\right) \!=\!  \eusm{E}^{\hspace{0,05cm}n_{N-k-1}, ...,\, n_{1}}_{\!{\tn{\frown}}} \, ,  & \  \ \
\Bb{G}^{N}_{2} \left(\eusm{E}^{n_{N}, ...,\, n_{1}}_{\!{\tn{\frown}}}\right) \!=\! \eusm{E}^{0}_{\!{\tn{\frown}}}  \, ,    &   \  0  \!\leqslant \! k  \!\leqslant\!  N\! -\! 2 \,.
 \end{array}\hspace{-0,1cm}
\end{align}

\noindent By applying the identity
\begin{align*}
    &  \left|\phi^{\,\prime} (z)\right| = \dfrac{\im  \,\phi (z)}{\im\, z} \, , \qquad  \phi\in \Gamma_{\vartheta} \ , \quad  z \in \Bb{H}\,,
\end{align*}

\noindent we see from \eqref{f25bsdenac} and \eqref{f5bsdenac} that
\begin{align}\label{f27bsdenac}
    &
      \left|\psi^{\,\prime}_{{\fo{n_{N}, ..., n_{N-k}}}} (z)\right|       =
\dfrac{\im\,\Bb{G}^{k+1}_{2} (z) }{\im\,z}  \ ,      \quad      0  \leqslant  k  \leqslant\!  N - 1 \,, \quad  z\in  \eusm{E}^{\hspace{0,05cm}n_{N}, ...,\, n_{1}}_{\!{\tn{\frown}}}\,.
     \end{align}

\noindent In view of \eqref{f25bsdenac},  for every $0\leqslant k \leqslant N-1$, $x\!\in\!\Bb{R}$ and $z\in \eusm{E}_{\!{\tn{\frown}}}^{\hspace{0,05cm}n_{N}, ...,\, n_{1}} $ we can write the function   $\Delta^{\delta}_{n_{N}, n_{N-1}, ..., n_{N-k}}$  from \eqref{f5zbsdenacth1} in the form (see \eqref{f2bsdenaclem2})
\begin{align}\label{f31bsdenac}
     & \Delta^{\delta}_{n_{N}, n_{N-1}, ..., n_{N-k}}(x,z) =
     \dfrac{(-1)^{k+1}\psi_{n_{N},...,n_{N-k}}^{\,\prime} (z)}{\left(
x^{1-\delta_{k}   }\Bb{G}^{k+1}_{2} (z) -
(-x)^{\delta_{k} }\right)^{2}} \ , \
\end{align}

\noindent where  $ \delta \!\in \! \{0,1\}$ and the properties \eqref{f25bsdenac}, \eqref{f27bsdenac} hold. On the other hand, in the  formula \eqref{f5zbsdenacth1}  the values of the function $\Phi^{\,\delta}_{\infty}(x; z)$ for $z\in \eusm{E}_{\!{\tn{\frown}}}^{\hspace{0,05cm}n_{N}, ...,\, n_{1}} $ can be expressed with the help of \eqref{10bsdenac} through its values on $\eusm{E}_{\!{\tn{\frown}}}^{\hspace{0,05cm}0}$:
\begin{align}\label{f32bsdenac}
    & (-1)^{N} \Phi^{\,\delta}_{\infty}(x; z)\big/ \psi^{\,\prime}_{{\fo{n_{N}, ...,  n_{1}}}} (z)= \Phi^{\,\delta_{N}}_{\infty}(x; \psi_{n_{N},...,n_{1}} (z)) =
\Phi^{\,\delta_{N}}_{\infty}(x; \Bb{G}^{N}_{2} (z)) ,
\end{align}

\noindent because $\psi_{n_{N},...,n_{1}} \in \Gamma_{\vartheta}$, by virtue of \eqref{f5bsdenac}, and
$\psi_{n_{N},...,n_{1}} (z)=\Bb{G}^{N}_{2} (z) \in      \eusm{E}^{0}_{\!{\tn{\frown}}}$ for all $ z\in  \eusm{E}^{\hspace{0,05cm}n_{N}, ...,\, n_{1}}_{\!{\tn{\frown}}}$. Moreover, the first equality in \eqref{f1bsdenaclem1p} actually holds  for all   $ z\in  \eusm{E}^{0}_{\!{\tn{\frown}}}$. Indeed,   by \eqref{f3contgen}(b) we have  $\lambda \big(\eusm{E}^{0}_{\!{\tn{\frown}}}\big)\!  = \!     \Bb{C}_{\,\re > 1/2}~\setminus~\{1\}$, while  $\lambda\left(\fet \setminus \Bb{D}\right)\!=\!\Bb{C}_{\re\leqslant 1/2}  \!\setminus\! \Bb{R}_{\leqslant 0}$, as follows from
\eqref{f5int}(c). It follows that $\lambda(z)- \lambda(\zeta)\neq 0$ holds for all $\zeta \in \fet \setminus \Bb{D}$
and $ z\in \eusm{E}^{0}_{\!{\tn{\frown}}}$. Since $\fet \setminus \Bb{D}$ is simply connected and $\gamma(-1,1)\subset \fet \setminus \Bb{D}$, $\Pi(-1,1)\subset \fet \setminus \Bb{D}$,
 we can transform the contour $\gamma(-1,1)$ of integration in \eqref{f1bsdenac} to $\Pi(-1,1)$ for each
$ z\in  \eusm{E}^{0}_{\!{\tn{\frown}}}$  to get $ \Phi^{\hspace{0,02cm}\delta}_{\infty}(x;z)\! =\!\Phi^{\hspace{0,02cm}\delta}_{{\tn{\sqcap}}}(x;z)$ and hence
\begin{align}\label{f33bsdenac}
    &  \Phi^{\hspace{0,02cm}\delta}_{\infty}(x;\Bb{G}^{N}_{2} (z))\! =\!\Phi^{\hspace{0,02cm}\delta}_{{\tn{\sqcap}}}(x;\Bb{G}^{N}_{2} (z)),
\ z\in  \eusm{E}^{\hspace{0,05cm}n_{N}, ...,\, n_{1}}_{\!{\tn{\frown}}} \, , \  x \in \Bb{R}, \ \delta \!\in \! \{0,1\}\,.
\end{align}

\noindent A similar change of the contour $\gamma(-1,1)$ can be made in \eqref{f5bsdenacth1} and  in  \eqref{f4bsdenacth1} after application   \eqref{10bsdenac} to  \eqref{f4bsdenacth1}
with $\phi (z)=-1/z$.
Combining this with \eqref{f33bsdenac}, \eqref{f32bsdenac},  \eqref{f31bsdenac},  \eqref{f25bsdenac}, \eqref{f27bsdenac} and Lemma~\ref{bsdenacth1} gives the  main result of this section.

    Let
\begin{align*} \hspace{-0,2cm} \Phi^{\,\delta}_{\infty}(x;z) & := I_{{\fo{\gamma (-1,1)}}}^{\delta} (x;z) , \ \, z \!\in \!\Bb{H}\setminus S_{\!{\tn{\frown}}}^{\infty} \, ; \ \
\Phi^{\hspace{0,02cm}\delta}_{{\tn{\sqcap}}}(x;z) := I_{{\fo{\Pi (-1,1)}}}^{\delta} (x;z) , \ \, z\! \in \! \eusm{E}_{\!{\tn{\frown}}}^{\hspace{0,05cm}0} \,,
    \\[0,1cm]       \hspace{-0,25cm}
I^{\,\delta}_{\Gamma}(x;z) & :={{ \dfrac{1}{2\pi i}}}
\int\limits_{{\fo{\Gamma}}} {{\dfrac{\lambda^{\,\prime}(z)\,{\rm{d}} \zeta}{\big(\lambda(z)- \lambda(\zeta)\big)\big(x^{\delta}\zeta-(-x)^{1-\delta}\big)^{\!2}}}}  \ , \quad   x\in\Bb{R}  \ , \quad  \delta \in  \{0,1\} \, .
 \end{align*}

 \vspace{-0,2cm}
\noindent Here the contour $\Pi (-1,1)= (-1,-1+\!2{\rm{i}}]\cup [-1\!+\!2{\rm{i}}, 1+\!2{\rm{i}}] \cup [1+\!2{\rm{i}}, 1)$ passes from $-1$ to $1$,  $\lambda \big(\eusm{E}^{0}_{\!{\tn{\frown}}}\big)\! \subset \!  \Bb{C}_{\,\re > 1/2}$, $\lambda\left(\Pi(-1,1)\right)\!\subset\!\Bb{C}_{\re< 1/2}$ and the convergence of both integrals is absolute, in view of \eqref{f3zint} and \eqref{f2auxevagen}.

\vspace{0.15cm}
 \begin{theorem}\label{bsdenacth2} For each $x\in\Bb{R}$  and  $\delta \in  \{0,1\}$ the function $\Phi^{\,\delta}_{\infty}(x;z)$ of the va\-ri\-ab\-le $z$ can  be analytically extended
 from the set $\eusm{E}^{\infty}_{\!{\tn{\frown}}}= \Bb{H} \setminus \cup_{\, {\fo{m \in \Bb{Z}}}}\ (2m + \overline{\Bb{D}})$ to $\Bb{H}_{|\re |< 1} \cup\eusm{E}^{\infty}_{\!{\tn{\frown}}}$ such that the resulting extension  $\Phi^{\,\delta}_{{\tn{||}}}(x;z)$ on the set {\rm{(}}see \eqref{f6wdenac}{\rm{)}}
\begin{align}\label{f2bsdenacth2}
    &  \hspace{-0,3cm} \left(\Bb{H}_{|\re |< 1} \cup\eusm{E}^{\infty}_{\!{\tn{\frown}}}\right)\big\backslash  S^{\hspace{0,02cm}{\tn{||}}}_{\!{\tn{\frown}}}\!=\! \eusm{E}^{\infty}_{\!{\tn{\frown}}}  \sqcup
    \eusm{E}^{0}_{\!{\tn{\frown}}}\ \sqcup \!\!\!\!\!\!
 \bigsqcup\limits_{
 {\fo{\,n_{1}, ... , n_{N} \in \Bb{Z}_{\neq 0}}} \,, \,  {\fo{N \in \Bb{N}} }} \,   \eusm{E}_{\!{\tn{\frown}}}^{\hspace{0,05cm}n_{N}, ...,\, n_{1}} \hspace{-0,15cm}
\end{align}

\noindent for   $\delta_{2n}\!:=\!\delta$, $\delta_{2n+1}\!:=\!1\!-\!\delta$,  $n\!\in\! \Bb{Z}_{\geqslant 0}$, and arbitrary $N \!\in\! \Bb{N}$, $n_{N}, ..., n_{1}\!\in\!  \Bb{Z}_{\neq 0}$
 satisfies
\begin{align}\label{f1bsdenacth2}
    &  \hspace{-0,25cm}  \Phi^{\,\delta}_{{\tn{||}}}(x;z)=
\left\{
\begin{array}{ll}
 - z^{-2}
 \Phi^{\hspace{0,02cm}1-\delta}_{{\tn{\sqcap}}}(x;-1/z),   &  \hspace{-2,2cm} \mbox{if}  \ z\in \eusm{E}^{\infty}_{\!{\tn{\frown}}} \,,
\\[0,2cm]
- {{\dfrac{1}{(x^{\delta}z - (-x)^{1-\delta})^{2}} }} + \Phi^{\hspace{0,02cm}\delta}_{{\tn{\sqcap}}}(x;z)
     \ ,   &  \hspace{-2,2cm}  \mbox{if}  \ z\in\eusm{E}_{\!{\tn{\frown}}}^{\hspace{0,05cm}0}\,,
\\[0,4cm]
  - {{\dfrac{1}{(x^{\delta}z - (-x)^{1-\delta})^{2}}}}+
\sum\limits_{k=0}^{N-1} {{\dfrac{(-1)^{k}\psi_{{\fo{n_{N},...,n_{N-k}}}}^{\,\prime} (z)}{\left(
x^{1-\delta_{k}  }\Bb{G}^{k+1}_{2} (z) -
(-x)^{\delta_{k} }\right)^{2}}}}
\\[0,7cm]
+(-1)^{N}\psi^{\,\prime}_{{\fo{n_{N}, ...,  n_{1}}}} (z)\,
\Phi^{\hspace{0,02cm}\delta_{N}}_{{\tn{\sqcap}}}\left(x;\Bb{G}^{N}_{2} (z) \right)
  ,   &     \hspace{-2,2cm}  \mbox{if}  \  z\in\eusm{E}_{\!{\tn{\frown}}}^{\hspace{0,05cm}n_{N}, ...,\, n_{1}},
\end{array}
\right. \hspace{-0,3cm}
\end{align}

\noindent where the sets in \eqref{f2bsdenacth2} and $\psi_{{\fo{n_{N}, ...,  n_{1}}}} \! \in\! \Gamma_{\vartheta}$ in \eqref{f1bsdenacth2} are defined as in \eqref{f1zbsdenac}, \eqref{f8bsdenac}, \eqref{f5zdenac},  and  as in \eqref{f4bsdenac}, correspondingly. Here, in \eqref{f1bsdenacth2} we have
\begin{align*}
     &\vspace{-0,25cm} \psi_{{\fo{n_{N}, ..., n_{N-k}}}} (z) \!=\! \Bb{G}^{k+1}_{2} (z)\,  ,   \quad z\!\in\!  \eusm{E}^{n_{N}, ...,\, n_{1}}_{\!{\tn{\frown}}}  ,  \ \ 0 \! \leqslant\!  k \! \leqslant\!  N \!-\! 1\,,\vspace{-0,1cm}
\end{align*}

\noindent $\Bb{G}^{N}_{2} \left(\eusm{E}^{n_{N}, ...,\, n_{1}}_{\!{\tn{\frown}}}\right) \!=\! \eusm{E}^{0}_{\!{\tn{\frown}}}$ and
$-1/\eusm{E}^{\infty}_{\!{\tn{\frown}}}=
 \eusm{E}_{\!{\tn{\frown}}}^{\hspace{0,05cm}0}$ {\rm{(}}see also \eqref{f25bsdenac} and \eqref{f27bsdenac}{\rm{)}}.
 \end{theorem}

 We observe that by the definition \eqref{f3contgendef1} of $ \Bb{G}_{2} (z)$ and the identity
$\im (-1/z) = (\im\, z)/|z|^{2} $, $z\!\in \!\Bb{H}$, we have $\im\,  \Bb{G}_{2} (z)\! = \!\im (-1/z)$ for every
$ z \!\in\! \Bb{H}_{|\re|\leqslant 1}$ and thus
\begin{align}\label{f29bsdenac}
    & \im\,  \Bb{G}_{2} (z)\! =\! \im (-1/z)\! \geqslant\! \im\, z  \, , \ \  z\!\in \!\Bb{D}_{\im>0} \, ; \ \
\im\,  (-1/z) \! \leqslant \!\dfrac{1}{\im\, z}  \, , \ \  z\!\in\! \Bb{H} \,.
\end{align}

\noindent Since for every $  z\!\in\!  \eusm{E}^{\hspace{0,05cm}n_{N}, ...,\, n_{1}}_{\!{\tn{\frown}}}$ we have
$\Bb{G}_{2}^{k} (z) \in \Bb{D}_{\im>0}$, $0 \leqslant k \leqslant N$, then
\begin{align}\label{f30bsdenac}
    & \begin{array}{l}\hspace{-0,25cm}
      \im\, z \!\leqslant \!\im\,\Bb{G}_{2}^{k} (z) \! \leqslant \! \im\,\Bb{G}_{2}^{k+1} (z)\!=\!\im\, \left( - 1/ \Bb{G}_{2}^{k} (z) \right)\!  \leqslant\! 1  \ , \quad  0 \!\leqslant\! k \leqslant\! N\!-\!1 \, , \\[0,3cm]\hspace{-0,25cm}
      \im\left( - 1/ \Bb{G}_{2}^{N} (z) \right)\!\leqslant \!1/\im\, z \, , \quad  z\!\in\!  \eusm{E}^{\hspace{0,05cm}n_{N}, ...,\, n_{1}}_{\!{\tn{\frown}}} \, , \  n_{N}, ..., n_{1}\!\in\!  \Bb{Z}_{\neq 0}  , \ N \!\in\! \Bb{N}.
      \end{array}
\end{align}


\section[\hspace{-0,30cm}. \hspace{0,11cm}Evaluation of biorthogonal functions for large index]{\hspace{-0,095cm}{\bf{.}} Evaluation of biorthogonal functions for large index}
\label{evagen}

\ The relations{\hyperlink{r28}{${}^{\ref*{case28}}$}}\hypertarget{br28}{}
\begin{align*}
    &
\begin{array}{l}
  \Theta_{2} (z)^{4}=(\sigma- z)^{-2} \Theta_{4} \left({1}/{(\sigma- z)}\right)^{4}, \ \ z \in \Bb{H}\,;
  \\[0,15cm]
\im (\sigma- z)^{-1} \geqslant 1/2 \, , \  |z|> 1 \, , \  \sigma \re\, z \in [0,1] \, , \  \im \, z \leqslant 1 \, , \ \sigma \in \left\{1, -1\right\},
\end{array}
\end{align*}
\noindent
 and  the Landen equations
\eqref{f3d1int} imply that{\hyperlink{r1}{${}^{\ref*{case1}}$}}\hypertarget{br1a}{}
\begin{align}\label{f1auxevagen}
    & \hspace{-0,25cm} \left|\Theta_{2} (z)\right|^{4} \im^{2} z\! \leqslant\! \theta_{3}({\rm{e}}^{-\pi/2})^{4}
\! \leqslant\!  5 , \     z\! \in \! \eusm{E}^{\infty}_{\!{\tn{\frown}}} \, ; \ \
 \lambda (2z)\!=\! \left(\dfrac{1\!-\!\sqrt{1\!-\!\lambda (z)}}{1\!+\!\sqrt{1\!-\!\lambda (z)}}\right)^{2} \!\!\! \!,  \,  z\!\in\! \fet\,, \hspace{-0,1cm}
\end{align}

\noindent where the value of $\theta_{3}({\rm{e}}^{-\pi/2})$
is taken
from \cite[p.\! 325]{ber1} and the principal branch of the square root is used.
   From \eqref{f2aqinttheor1}  and \eqref{f0apinttheor1}(b) we thus obtain{\hyperlink{r26}{${}^{\ref*{case26}}$}}\hypertarget{br26}{}
\begin{align}\label{f2auxevagen}
    &   \dfrac{1}{1-2 \lambda (2{\rm{i}})} =
\dfrac{11 + 8\sqrt{2}}{21}  \, , \quad  \left|\lambda (\pm 1 \!+\! {\rm{i}} t)\right| \geqslant \dfrac{ 17 - 12\sqrt{2}}{16} {\rm{e}}^{\pi /t} \ , \quad  t  \in (0,2]\,.
\end{align}

\subsection[\hspace{-0,31cm}. \hspace{0,11cm}Evaluation of the generating function.]{\hspace{-0,11cm}{\bf{.}} Evaluation of the generating function.} \
Using the identity $\lambda^{\,\prime}(z)= \pi i \Theta_{4}\left(z \right)^{4} \lambda (z)$, $z\in \Bb{H}$,
written in \eqref{f19int}, and the fact that all three values of the variable $z$ used in the formula \eqref{f1bsdenacth2} for the function
$\Phi^{\,\delta}_{{\tn{\sqcap}}}(x;z)$ belong to $\eusm{E}_{\!{\tn{\frown}}}^{\hspace{0,05cm}0}$, this function  can be estimated  as
\begin{align*}
    &  \left|\Phi^{\,\delta}_{{\tn{\sqcap}}}(x;z)\right|\!\leqslant\! \left|\Theta_{4}\left(z \right)\right|^{4}I_{\delta} (z) \, , \
I_{\delta} (z):= \!\!\!\!\!\!\!\int\limits_{{\fo{\Pi (-1,1)}}}\!\!\!\!\!\dfrac{
\left(\left|\lambda \left(z \right)\right| \big/\big|\lambda\left(z \right)\!-\! \lambda(\zeta)\big|\right)|{\rm{d}} \zeta|}{2\big|x^{\delta}\zeta-(-x)^{1-\delta }\big|^{2} }  \ , \quad  z\!\in\!\eusm{E}_{\!{\tn{\frown}}}^{\hspace{0,05cm}0} \, .
\end{align*}

For arbitrary $z\in\Bb{D}_{\im>0}$ and $x\in \Bb{R}$ we obviously have
\begin{align*}
    &  \left|z+x\right|\geqslant \im z  \, , \  \left|x z-1\right|= \left|z\right|
\left|x- (1/z)\right|\geqslant \left|z\right| \im (-1/z)  =  \im z /\left|z\right|\geqslant \im z \,,
\end{align*}

\noindent and hence
\begin{align}\label{f1evagen}
    &  \left|x^{\delta}z - (-x)^{1-\delta}\right| \geqslant \left(\min \left\{1, |z|^{-1}\right\}\right) \im z   \ , \quad  z\in\Bb{H} \, , \
\delta \!\in \! \{0,1\}  \, , \  x\!\in\!\Bb{R}\,.
\end{align}

\noindent Since the sets $\eusm{E}_{\!{\tn{\frown}}}^{\hspace{0,05cm}0}$ and $\eusm{E}_{\!{\tn{\frown}}}^{\hspace{0,05cm}n_{N}, ...,\, n_{1}}$ in \eqref{f1bsdenacth2} are the subsets of $\Bb{D}$ and
 for any $z$ from these sets we have
 $  \Bb{G}^{k}_{2} (z) \in \Bb{D} $, $0 \leqslant k \leqslant N$ ($N=0$, if $ z\in \eusm{E}_{\!{\tn{\frown}}}^{\hspace{0,05cm}0}$), in view of \eqref{f25bsdenac},
\eqref{f17bsdenac} and \eqref{f16bsdenac},  we deduce from \eqref{f1evagen} that
\begin{align*}
    & \left|\dfrac{1}{(x^{\delta}\Bb{G}^{k}_{2} (z) - (-x)^{1-\delta})^{2}}\right|\leqslant\dfrac{1}{\im^{2}\Bb{G}^{k}_{2} (z) } \, , \ 0 \leqslant k \leqslant d(\eufm{n})  \ , \quad
z\in  \eusm{E}_{\!{\tn{\frown}}}^{\eufm{n}} \, , \  \eufm{n}\in \{0\}\cup \Bb{Z}_{\neq 0}^{\hspace{0,02cm}\Bb{N}_{\hspace{-0,02cm}\eurm{f}}}\,,
\end{align*}

\noindent  where $\Bb{G}^{0}_{2} (z):=z$.
Together with \eqref{f27bsdenac}, this gives  the following estimate for the sum in \eqref{f1bsdenacth2}
\begin{align*}
    & \left|  {{\dfrac{1}{(x^{\delta}z - (-x)^{1-\delta})^{2}}}}-
\sum\limits_{k=0}^{N-1} {{\dfrac{(-1)^{k}\psi_{{\fo{n_{N},...,n_{N-k}}}}^{\,\prime} (z)}{\left(
x^{1-\delta_{k}  }\Bb{G}^{k+1}_{2} (z) -
(-x)^{\delta_{k} }\right)^{2}}}}\right|\leqslant \sum\limits_{k=0}^{N}\dfrac{\im \,\Bb{G}^{k}_{2} (z)}{(\im\,z)\, \im^{2} \Bb{G}^{k}_{2} (z)}\ .
\end{align*}

\noindent As the roof $\gamma_{\eufm{n}}(-1,1)$ of $\eusm{E}^{\eufm{n}}_{\!{\tn{\frown}}}$ is also the roof of
$\fet^{\hspace{0,05cm}n_{N}, ..., n_{1}}$ then, in accordance with Lemma~\ref{bsdentplem1}(e),
the property $z\in\eusm{E}^{\eufm{n}}_{\!{\tn{\frown}}}$ implies that $\im \,z \leqslant 1/(2N) $, i.e.,  $2 N \leqslant 1/( \im \,z)$ and, by  \eqref{f30bsdenac},
we obtain for the latter sum the estimate
\begin{align}\label{f2zevagen}
    &  \sum\limits_{k=0}^{N}\dfrac{1}{(\im\,z)\, \im \,\Bb{G}^{k}_{2} (z)} \leqslant \dfrac{N+1}{(\im\,z)^{2}}\leqslant \dfrac{1}{(\im\,z)^{3}}\, , \quad  z\!\in\!  \eusm{E}^{\eufm{n}}_{\!{\tn{\frown}}} \, , \  \eufm{n}\!\in\!  \Bb{Z}_{\neq 0}^{N}\, , \ N\in \Bb{N} \, .
\end{align}

\noindent As a result, we obtain  the following estimates  for the functions from \eqref{f1bsdenacth2}:
\begin{align}\label{f2evagen}
    &\hspace{-0,3cm}   \left|\Phi^{\,\delta}_{{\tn{||}}}(x;z)\right|\!\leqslant\!
\left\{
  \begin{array}{ll}
  \dfrac{\left|\Theta_{4}\left(-1/z \right)\right|^{4}}{ |z|^{2}}\, I_{\delta} (-1/z), &\ \  \hbox{if}\ \, z\in \eusm{E}^{\infty}_{\!{\tn{\frown}}} \,;\\
 \left|\Theta_{4}\left(z \right)\right|^{4} \,I_{\delta} (z) + \dfrac{1}{(\im\,z)^{2}}, &\ \   \hbox{if}\ \, z\in\eusm{E}_{\!{\tn{\frown}}}^{\hspace{0,05cm}0}\,;\\
 \dfrac{\left|\Theta_{4}\big(\Bb{G}^{N}_{2} (z) \big)\right|^{4}}{ \left(\im\,z\right)\big/
\im\,\Bb{G}^{N}_{2} (z)
}\, I_{\delta} \big(\Bb{G}^{N}_{2} (z)\big)+
\dfrac{\sum\limits_{k=0}^{N}\dfrac{1}{\im \,\Bb{G}^{k}_{2} (z)}}{\im\,z }
, &\ \  \hbox{if}\ \, z\in\eusm{E}^{\eufm{n}}_{\!{\tn{\frown}}}  \,.
  \end{array}
\right.\hspace{-0,25cm}
\end{align}

\noindent Here,   $N \!\in\! \Bb{N}$, $ \eufm{n} =(n_{N}, ..., n_{1})\!\in\!  \Bb{Z}_{\neq 0}^{N}$,
  $x \!\in\! \Bb{R}$, $ \delta \!\in \! \{0,1\}$  and the equality \eqref{f27bsdenac} was used.

Since for arbitrary $\zeta \in \Pi (-1,1)$, we have $|\zeta| \leqslant \sqrt{5}$,  \eqref{f1evagen} gives that
\begin{align*}
    &  \dfrac{ I_{\delta} (z)}{(5/2)}\!\leqslant \!\!\!\!\!\!\!\!\int\limits_{{\fo{\Pi (-1,1)}}}\!\!\!\! \!\dfrac{
\left|\lambda \left(z \right)\right| |{\rm{d}} \zeta|}{\big|\lambda\left(z \right)- \lambda(\zeta)\big|
(\im\, \zeta)^{2} } \!= \!\int\limits_{-1}^{1} \!\dfrac{\left|\lambda \left(z \right)\right| {\rm{d}} t}{4 \big|\lambda\left(z \right)- \lambda(2 {\rm{i}} \!+\! t)\big| }\! + \!\int\limits_{0}^{2}\! \dfrac{2 \left|\lambda \left(z \right)\right| {\rm{d}} t}{t^{2} \big|\lambda\left(z \right)- \lambda(1 \!+\! {\rm{i}} t)\big| }\,.
\end{align*}

\noindent If we denote by  $I^{-}_{\delta} (z)$ and $I^{|}_{\delta} (z)$ the first and the second integrals in the right-hand side of the above equality, correspondingly,  we can apply  \eqref{f1intcorol1}(c) and \eqref{f1intcorol1}(d) to get $(2/5) I_{\delta} (z)\!\leqslant \! I^{-}_{\delta} (z) +  I^{|}_{\delta} (z) $  and
\begin{align}\label{f4evagen}
    &
 I^{-}_{\delta} (z)\leqslant \dfrac{2}{1-2 \lambda (2{\rm{i}})} \, , \
I^{|}_{\delta} (z)\leqslant \int\limits_{0}^{2}\! \dfrac{2 \sqrt{2}\left|\lambda \left(z \right)\right| {\rm{d}} t}{t^{2} \left(\big|\lambda\left(z \right)\big|+ \big|\lambda(1 \!+\! {\rm{i}}\, t)\big|\right) } \ , \quad  z\in\eusm{E}_{\!{\tn{\frown}}}^{\hspace{0,05cm}0}  \,.
\end{align}

\noindent Let $a := \big|\lambda\left(z \right)\big|$. Then $\lambda \big(\eusm{E}^{0}_{\!{\tn{\frown}}}\big)\! \subset \!  \Bb{C}_{\,\re > 1/2}$
implies $a > 1/2$ and using \eqref{f2auxevagen} we get{\hyperlink{r27}{${}^{\ref*{case27}}$}}\hypertarget{br27}{}
\begin{align*}
    &  \dfrac{I^{|}_{\delta} (z)}{2 a\sqrt{2}}\leqslant \int\limits_{0}^{2}\! \dfrac{ {\rm{d}} t}{t^{2} \left(a+ \dfrac{ 17 - 12\sqrt{2}}{16} \,{\rm{e}}^{\pi /t}\right) }
= \dfrac{1}{\pi a}  \log  \left( 1 + \dfrac{16\, {\rm{e}}^{-\pi /2}a}{17 - 12\sqrt{2}}\right) \, , \
\end{align*}

\noindent from which\vspace{-0,15cm}
\begin{align*}
    &  I^{|}_{\delta} (z) \leqslant \dfrac{2 \sqrt{2}}{\pi} \log  \left( 1 + \dfrac{16\, {\rm{e}}^{-\pi /2}}{17 - 12\sqrt{2}}\big|\lambda\left(z \right)\big|\right)\ , \quad  z\in\eusm{E}_{\!{\tn{\frown}}}^{\hspace{0,05cm}0}  \,,
\end{align*}

\noindent and hence,
we deduce from \eqref{f4evagen} and \eqref{f2auxevagen} that
\begin{align}\label{f5evagen}
    & I_{\delta} (z)\!\leqslant \! \dfrac{55 + 40\sqrt{2}}{21}  + \dfrac{5 \sqrt{2}}{\pi} \log  \left( 1 + \dfrac{16\, {\rm{e}}^{-\pi /2}}{17 - 12\sqrt{2}}\big|\lambda\left(z \right)\big|\right) \ , \quad  z\in\eusm{E}_{\!{\tn{\frown}}}^{\hspace{0,05cm}0}  \,   \,.
\end{align}

\vspace{0,1cm}
 To obtain the final estimates of $\Phi^{\,\delta}_{{\tn{||}}}$, we need to deal with the variable $z$
lying in $\eusm{E}^{\infty}_{\!{\tn{\frown}}}= - 1/ \eusm{E}_{\!{\tn{\frown}}}^{\hspace{0,05cm}0}$ but not in $\eusm{E}_{\!{\tn{\frown}}}^{\hspace{0,05cm}0}$. Hence, we transform \eqref{f5evagen}, taking into account that due to $2$-periodicity of $\lambda$ we have $1-\lambda(z) = \lambda(- 1/z) = \lambda(\Bb{G}_{2} (z))$, where $\Bb{G}_{2} (z)\in \mathcal{F}^{\,{\tn{||}}}_{{\tn{\square}}} \setminus\overline{\Bb{D}}$ for all $z \in \eusm{E}_{\!{\tn{\frown}}}^{\hspace{0,05cm}0}$, in view of \eqref{f16contgen}(a). Namely,
\begin{align*}
    &  I_{\delta} (z)\!\leqslant \!
 \dfrac{55 + 40\sqrt{2}}{21}  + \dfrac{5 \sqrt{2}}{\pi} \log\left(1+ \dfrac{16 \,{\rm{e}}^{-\pi /2}}{17 - 12\sqrt{2}}\right)
+    \dfrac{5 \sqrt{2}}{\pi}\log \left(1 + \big|\lambda\left(- 1/z \right)\big|\right),
\end{align*}

\noindent which after numerical calculations can be written as{\hyperlink{r29}{${}^{\ref*{case29}}$}}\hypertarget{br29}{}
\begin{align}\label{f6evagen}
    &  \hspace{-0,2cm} I_{\delta} (z)\leqslant 16 + (9/4) \log \Big( 1 \!+ \!\left|\lambda \big(\Bb{G}_{2} (z)\big)\right|\Big), \ \
\Bb{G}_{2} (z)\!\in\! \mathcal{F}^{\,{\tn{||}}}_{{\tn{\square}}} \setminus\overline{\Bb{D}}\!\subset \! \eusm{E}^{\infty}_{\!{\tn{\frown}}},
\ \   z\in\eusm{E}_{\!{\tn{\frown}}}^{\hspace{0,05cm}0}  \,.\hspace{-0,1cm}
\end{align}

\noindent
But for arbitrary $x+\imag y \in  \mathcal{F}^{\,{\tn{||}}}_{{\tn{\square}}} \setminus\overline{\Bb{D}}$ by \eqref{f1inttheor1}, \eqref{f2aqinttheor1} and
\eqref{f0apinttheor1}(b) we have
\begin{align*}
     & \left|\lambda(x+\imag  y)\right|\leqslant\left|\lambda(1+\imag  y)\right| =
     \lambda(\imag  y)\big/\lambda(\imag / y) \leqslant  16\, {\rm{e}}^{{\fo{-  \pi y + \pi / y }}} \,,
\end{align*}

\noindent from which
\begin{align*}
     & \log \left(1 + \left|\lambda(x+\imag  y)\right|\right)\leqslant
     \log \left(1 + 16\, {\rm{e}}^{{\fo{-  \pi y + \pi / y }}}\right) \leqslant \pi \left(1
+  {1}/{y}\right)    ,
\end{align*}

\noindent  and hence{\hyperlink{r29}{${}^{\ref*{case29}}$}}\hypertarget{br29}{}
\begin{align}\label{f12evagen}
    &  \hspace{-0,2cm} I_{\delta} (z)\leqslant \dfrac{147 \pi}{20}
\left(1
+  \dfrac{1}{\im\,\Bb{G}_{2} (z) }\right), \ \
\Bb{G}_{2} (z)\!\in\! \mathcal{F}^{\,{\tn{||}}}_{{\tn{\square}}} \setminus\overline{\Bb{D}}\!\subset \! \eusm{E}^{\infty}_{\!{\tn{\frown}}},
\ \   z\in\eusm{E}_{\!{\tn{\frown}}}^{\hspace{0,05cm}0}  \,.\hspace{-0,1cm}
\end{align}

By applying for the three expressions in \eqref{f2evagen} the obvious consequence
$ \im \, z=\im \, \Bb{G}_{2} (-1/z) $, $z \in \eusm{E}^{\infty}_{\!{\tn{\frown}}}$,  of \eqref{f2contgendef1}, $ (147/20) \theta_{3}({\rm{e}}^{-\pi/2})^{4} <30 $,
the left-hand side inequality of \eqref{f1auxevagen}, \eqref{f3bint}(a),(c) and \eqref{f12evagen}, we obtain in Section~\ref{pevagenlem1} the following properties.

\begin{lemsecteight}\hspace{-0,17cm}{\bf{.}}\label{evagenlem1}
Let $x\in\Bb{R}$,  $\delta \in  \{0,1\}$,   $z\in \Bb{H}_{|\re |< 1} \cup\eusm{E}^{\infty}_{\!{\tn{\frown}}}$ and the function $\Phi^{\,\delta}_{{\tn{||}}}(x;z)$  be defined as in
Theorem~{\rm{\ref{bsdenacth2}}}.  Then\vspace{-0,1cm}
\begin{align}\label{f13evagen}
    &   \hspace{-0,23cm}\left|\Phi^{\,\delta}_{{\tn{||}}}(x;z)\right|\leqslant
\left\{
  \begin{array}{ll}
\dfrac{30 \pi\left( 1 + \dfrac{1}{\im\, z}\right)}{\im^{2} z} , &\quad \hbox{if}\ \ z\in \eusm{E}^{\infty}_{\!{\tn{\frown}}} \,;\\[0,15cm]
\dfrac{1}{\im^{2}\,z} + \dfrac{30 \pi \, \left( 1 + \dfrac{1}{\im\, \Bb{G}_{2} (z)}\right) }{(\im\,z)\, \im\,\Bb{G}_{2} (z)}
, &\quad  \hbox{if}\ \ z\in\eusm{E}_{\!{\tn{\frown}}}^{\hspace{0,05cm}0}\,;\\[0,25cm]
\dfrac{\sum\limits_{k=0}^{N}\dfrac{1}{\im \,\Bb{G}^{k}_{2} (z)}}{\im\,z }
+ \dfrac{30 \pi\,\left( 1 + \dfrac{1}{\im\, \Bb{G}^{N+1}_{2} (z)}\right)}{(\im\,z)\,\im\,\Bb{G}_{2}^{N+1} (z)}
, &\quad \hbox{if}\ \ z\in\eusm{E}^{\eufm{n}}_{\!{\tn{\frown}}}  \,;
  \end{array}
\right.\hspace{-0,1cm}
\end{align}

\noindent $\Bb{G}_{2} \big(\eusm{E}^{0}_{\!{\tn{\frown}}}\big) \! = \! \Bb{G}_{2}^{N+1}\big(\eusm{E}^{\eufm{n}}_{\!{\tn{\frown}}}\big) \! = \!
        \mathcal{F}^{\,{\tn{||}}}_{{\tn{\square}}} \setminus\overline{\Bb{D}}$,
             for arbitrary   $N \!\in\! \Bb{N}$ and $ \eufm{n}\! = \! (n_{N}, ..., n_{1})\!\in\!  \Bb{Z}_{\neq 0}^{N}$.
\end{lemsecteight}

By \eqref{f29bsdenac}, $\im\,\Bb{G}_{2} (z)\geqslant \im\,z$ for every
$z\in \eusm{E}_{\!{\tn{\frown}}}^{\hspace{0,05cm}0}\subset \Bb{D}_{\im>0}$. At the same time, for any $\eufm{n}\in \Bb{Z}_{\neq 0}^{\hspace{0,02cm}\Bb{N}_{\hspace{-0,02cm}\eurm{f}}}$ and $z\in\eusm{E}^{\eufm{n}}_{\!{\tn{\frown}}} $ we have $\Bb{G}_{2}^{N} (z)\in \eusm{E}_{\!{\tn{\frown}}}^{\hspace{0,05cm}0}\subset \Bb{D}_{\im>0}$ and
$\im\,\Bb{G}_{2}^{N} (z)\geqslant \im\,z$, in view of \eqref{f16contgen}(b) and \eqref{f30bsdenac}, correspondingly. Applying \eqref{f29bsdenac} once more,
we get $\im\,\Bb{G}_{2}^{N+1} (z)\geqslant \im\,z$.
  Since $\Phi^{\,\delta}_{{\tn{||}}}$ of the variable $z$ is continuous on $ \Bb{H}_{|\re|\leqslant 1}$ we derive from   \eqref{f13evagen} and \eqref{f2zevagen} the next assertion.{\hyperlink{r31}{${}^{\ref*{case31}}$}}\hypertarget{br31}{}

\begin{corsecteight}\hspace{-0,17cm}{\bf{.}}\label{evagencor1}
Let $x\in\Bb{R}$,  $\delta \in  \{0,1\}$,   $z\in \Bb{H}_{|\re |< 1} \cup\eusm{E}^{\infty}_{\!{\tn{\frown}}}$ and the function $\Phi^{\,\delta}_{{\tn{||}}}(x;z)$  be defined as in
Theorem~{\rm{\ref{bsdenacth2}}}.  Then
\begin{align}\label{f14evagen}
    &  \hspace{-0,23cm}\left|\Phi^{\,\delta}_{{\tn{||}}}(x;z)\right|\leqslant
\left\{
\begin{array}{ll}
\dfrac{20 \pi^{2}}{\im^{3} z} \ , \   & \mbox{if}\ \, \im \, z\!  \leqslant \! 1,
  \\[0,3cm]
  \dfrac{20 \pi^{2}}{\im^{2} z} \ , \
 & \mbox{if}\ \, \im \, z\!  \geqslant\!  1,
\end{array}
\right.
  \  \quad   \ z \! \in \!  \Bb{H}_{|\re|\leqslant 1}, \
x \!\in\! \Bb{R} , \  \delta \!\in \! \{0,1\} \,.
\end{align}
\end{corsecteight}

\subsection[\hspace{-0,31cm}. \hspace{0,11cm}Main inequalities.]{\hspace{-0,11cm}{\bf{.}} Main inequalities.}\label{mainevagen} \
It follows from \eqref{f1wbsdenac}  that for arbitrary $n  \geqslant 1$ and  $x\in \Bb{R}$ we have
\begin{align*}
    &
   4 \pi^{2} n   \eurm{H}_{n} (x)  =\! \!\! \int\limits_{-1+2\imag }^{1+2\imag } \! \!\! {\rm{e}}^{{\fo{ - {\rm{i}}\pi  n  z}} } \Phi^{\hspace{0,02cm}0}_{\infty}(x;z)  {\rm{d}} z , \ \ \
   4 \pi^{2} n   \eurm{M}_{n} (x)  = \! \!\!\int\limits_{-1+2\imag }^{1+2\imag }\! \!\!   {\rm{e}}^{{\fo{ - {\rm{i}}\pi n  z}} } \Phi^{\hspace{0,02cm}1}_{\infty}(x;z)  {\rm{d}} z\, .
\end{align*}

\noindent  By Theorem~\ref{bsdenacth2},  for every  $ \delta \!\in \! \{0,1\}$ the function $\Phi^{\hspace{0,02cm}\delta}_{\infty}$ can be analytically extended to $\Bb{H}_{|\re |< 1} \cup\eusm{E}^{\infty}_{\!{\tn{\frown}}}$ and in a second step, by \eqref{11bsdenac}, to $2$-periodic holomorphic
function $\Phi^{\hspace{0,025cm}\delta}_{\Bb{H}}(x; z)$ on $\Bb{H}$, which  equals $\Phi^{\,\delta}_{{\tn{||}}}(x;z)$ on $\Bb{H}_{|\re |\leqslant 1}$. As a consequence, by
 Lemma~\ref{lemfou}, we obtain from the above formulas that
\begin{align*}
    &
 4 \pi^{2} n   \eurm{H}_{n} (x)  =  \! \!\!\int\limits_{-1+\imag /n}^{1+\imag /n}  \! \!\! {\rm{e}}^{{\fo{ - {\rm{i}}\pi  n  z}} } \Phi^{\,0}_{{\tn{||}}}(x;z)  {\rm{d}} z , \ \ \
   4 \pi^{2} n   \eurm{M}_{n} (x)  =  \! \!\!\int\limits_{-1+\imag /n}^{1+\imag /n} \! \!\!   {\rm{e}}^{{\fo{ - {\rm{i}}\pi n  z}} } \Phi^{\,1}_{{\tn{||}}}(x;z)  {\rm{d}} z \,,
\end{align*}

\noindent from which it follows that
\begin{align}\label{f0evagen}
    &   \hspace{-0,3cm}
  \left\{\begin{array}{c}
          \left| \eurm{H}_{n} (x)\right|  \\[0,1cm]
         \left| \eurm{M}_{n} (x)\right|
         \end{array}
  \right\}
 \!\leqslant\! \dfrac{ {\rm{e}}^{{\fo{\pi }} }}{4 \pi^{2} n}
\int\limits_{-1}^{1} \left|  \Phi^{\,
 {\tn{\left\{\begin{array}{c}
          0  \\
         1
         \end{array}
  \right\} }}
}_{\,{\tn{||}}} \left(x; \dfrac{\imag }{n}\!+\!t\right)  \right|{\rm{d}} t , ,\hspace{-0,15cm}
\end{align}

\noindent for  $n  \geqslant 1$ and  $x\in \Bb{R}$. Since $\im\, (t+ \imag /n) = 1/n \leqslant 1$ we can apply
\eqref{f14evagen} to estimate the integral in the right-hand side of \eqref{f0evagen} as follows
\begin{align*}
    &  \left| \eurm{H}_{n} (x)\right| \ , \
         \left| \eurm{M}_{n} (x)\right|\leqslant 20 \pi^{2} n^{3}\dfrac{ {\rm{e}}^{{\fo{\pi }}} }{ 2\pi^{2} n} =  10 \,{\rm{e}}^{{\fo{\pi }}} n^{2}< \dfrac{\pi^{6} n^{2}}{4}\, , \quad x\in \Bb{R} \,.
\end{align*}

\noindent Here, we used{\hyperlink{r32}{${}^{\ref*{case32}}$}}\hypertarget{br32}{} that $40 \exp (\pi) < \pi^{6}$.
Using  \eqref{f2intth1} and \eqref{f3intth1}, we obtain{\hyperlink{r33}{${}^{\ref*{case33}}$}}\hypertarget{br33}{} \begin{align}\label{f2fevagen}
     &
\left| \eurm{H}_{n} (x)\right| ,
         \left| \eurm{M}_{n} (x)\right|\leqslant \min \left\{\,\dfrac{\pi^{6} n^{2}}{4}\, , \ \
         \dfrac{\pi^{6}n^{2}}{2(1+ x^{2})}\,\right\}\, , \quad x\in \Bb{R} \, , \ n \in \Bb{Z}_{\neq 0}\,.
\end{align}

\noindent The corresponding estimates of $\eurm{H}_{0}$ can be  derived from the explicit integral formula written after \vspace{-0,2cm} \eqref{f2intcor1}:{\hyperlink{r34}{${}^{\ref*{case34}}$}}\hypertarget{br34}{}
\begin{align}\label{f3fevagen}
    &  \left| \eurm{H}_{\hspace{0,02cm}0} (x) \right|\leqslant \min  \left\{\dfrac{3}{2} \ , \
\dfrac{3}{1 + x^{2}} \right\}\ , \quad  x \in \Bb{R} \ .
\end{align}

\vspace{-0,1cm}
 The first immediate consequence of the obtained estimates and \eqref{f1wbsdenac} is the possibility of expressing  the functions $\Phi^{\hspace{0,02cm}\delta}_{\Bb{H}}$, $\delta \!\in \! \{0,1\}$, for any $x \in \Bb{R}$ in the form
\begin{align}
    &    \sum\limits_{n \geqslant 1}
  \! n \,\eurm{H}_{n} (x)\, {\rm{e}}^{{\fo{{\rm{i}}\pi  n  z}} }\!
 = \! \frac{\Phi^{\hspace{0,02cm}0}_{\Bb{H}}(x;z)}{2 \pi^{2} }\, ,
\ \ \sum\limits_{n \geqslant 1}
  \! n \,\eurm{M}_{n} (x)\, {\rm{e}}^{{\fo{{\rm{i}}\pi  n  z}} }\!  = \! \frac{\Phi^{\hspace{0,02cm}1}_{\Bb{H}}(x;z)}{2 \pi^{2} }\, ,   \ \ z\in \Bb{H}\,,\hspace{-0,1cm}
\label{f6fevagen}\end{align}

\noindent because, by  \eqref{f2fevagen} and \eqref{f3fevagen}, both series on the left-hand sides of the equalities in \eqref{f1wbsdenac} turn out to be holomorphic on $\Bb{H}$. It follows from \eqref{f1bsdenac}, the identity
 $\lambda^{\,\prime}(z)= \pi {\rm{i}} \Theta_{4}\left(z \right)^{4} \lambda (z)$, $z\in \Bb{H}$
(see \eqref{f19int}), and the asymptotics $\lambda (\imag t) \to 0$ as $t \to +\infty$, that for all $\delta \in \{0,1\}$ and
$x\in \Bb{R}$, the function
\begin{align}\label{f8fevagen}
     &  \hspace{-0,15cm}
\Psi^{\,\delta}_{\infty}(x;z) := \frac{1}{2\pi \imag }
\int\limits_{{\fo{\gamma (-1,1)}}}\big(x^{\delta}\zeta-(-x)^{1-\delta}\big)^{-2}\log \left(1-\dfrac{\lambda (z)}{\lambda (\zeta)}\right) {\rm{d}} \zeta \ \, , \hspace{-0,1cm}
\end{align}

\vspace{-0,15cm}
\noindent of the variable $z\in \eusm{E}^{\infty}_{\!{\tn{\frown}}}$ is the unique primitive of $\Phi^{\,\delta}_{\infty}$  which has limit $0$ as $z \in {\rm{i}} \Bb{R}_{>1}, z \to \infty$. Then we derive from \eqref{f1wbsdenac} that
\begin{align*}
    & \hspace{-0,25cm}  \sum\limits_{n \geqslant 1}
  \!\eurm{H}_{n} (x)\, {\rm{e}}^{{\fo{{\rm{i}}\pi  n  z}} }\!
 = \! \frac{\imag  \Psi^{\hspace{0,02cm}0}_{\infty}(x;z)}{2 \pi  }\, ,
\ \ \sum\limits_{n \geqslant 1}
  \! \eurm{M}_{n} (x)\, {\rm{e}}^{{\fo{{\rm{i}}\pi  n  z}} }\!  = \! \frac{\imag  \Psi^{\hspace{0,02cm}1}_{\infty}(x;z)}{2 \pi }\, ,   \ \ z\in \Bb{H}_{\,\im > 1}\,,\hspace{-0,1cm}
\end{align*}

\noindent and since both series here for each $x\in \Bb{R}$ are holomorphic on $\Bb{H}$, in view of \eqref{f2fevagen} and \eqref{f3fevagen}, we conclude that for any $\delta \in \{0,1\}$  and
$x\in \Bb{R}$ the function $\Psi^{\hspace{0,02cm}\delta}_{\infty}$ has analytic
 extension from
$\Bb{H}_{\,\im > 1}$  to $\Bb{H}$ and the resulting extension $\Psi^{\hspace{0,02cm}\delta}_{\Bb{H}}$
satisfies
\begin{align}\label{f9fevagen}
    & \hspace{-0,25cm}  \sum\limits_{n \geqslant 1}
  \!\eurm{H}_{n} (x)\, {\rm{e}}^{{\fo{{\rm{i}}\pi  n  z}} }\!
 = \! \frac{\imag  \Psi^{\hspace{0,02cm}0}_{\Bb{H}}(x;z)}{2 \pi  }\, ,
\ \ \sum\limits_{n \geqslant 1}
  \! \eurm{M}_{n} (x)\, {\rm{e}}^{{\fo{{\rm{i}}\pi  n  z}} }\!  = \! \frac{\imag  \Psi^{\hspace{0,02cm}1}_{\Bb{H}}(x;z)}{2 \pi }\, ,  \  z\in \Bb{H} \,.\hspace{-0,1cm}
\end{align}

\noindent It follows from \eqref{f6fevagen} and \eqref{f9fevagen} that\vspace{-0,05cm}
\begin{align}\label{f13fevagen}
   ({\partial}/{\partial z})\,   \Psi^{\,\delta}_{\Bb{H}}(x;z) =  \Phi^{\,\delta}_{\Bb{H}}(x;z) \ , \quad  x\in \Bb{R} \, , \  z\in \Bb{H} \, , \  \delta \in \{0,1\}\,.
\end{align}

\vspace{-0,05cm}
\noindent
When we integrate by parts, \eqref{f8fevagen}  becomes  (see \eqref{f3zint})
\begin{align*}
  &   \Psi^{\,\delta}_{\infty}(x;z) =
   -\frac{1}{2\pi \imag }
\int\limits_{{\fo{\gamma (-1,1)}}}\!\!\!\!\!\log \left(1-\dfrac{\lambda (z)}{
\lambda (\zeta)}\right)\, {\rm{d}} \dfrac{\zeta^{\delta}}{x^{\delta}\zeta-(-x)^{1-\delta}} \\    &     =
\frac{1}{2\pi \imag }
\int\limits_{{\fo{\gamma (-1,1)}}} \!\!\!\!\! \dfrac{\lambda (z)\lambda^{\,\prime} (\zeta)}{\lambda (\zeta)\left(\lambda (\zeta)-\lambda (z)\right)}
\dfrac{\zeta^{\delta}{\rm{d}}\zeta}{\big(x^{\delta}\zeta-(-x)^{1-\delta}\big)} \, , \  \
\delta \in \{0,1\} \, , \  x\in \Bb{R}\setminus\{0\}\,,
\end{align*}

\vspace{-0,1cm}
\noindent and hence, in view of $\lambda^{\,\prime} (z)\!= \!\imag  \pi \lambda (z)
       \Theta_{4}(z)^{4}$, $z\in \Bb{H}$ (see \eqref{f19int}, \eqref{f2int}), we find
\begin{align}\label{f10fevagen}
      & \Psi^{\,\delta}_{\infty}(x;z)= \frac{1}{2 }\!\!\!\!\!\int\limits_{{\fo{\gamma (-1,1)}}}\!\!\!
\dfrac{\lambda (z)\Theta_{4}(\zeta)^{4}\, \zeta^{\delta}{\rm{d}}\zeta}{
\big(\lambda(\zeta)-\lambda(z)\big)\big(x^{\delta}\zeta-(-x)^{1-\delta}\big)} \, , \   \  \delta\! \in\! \{0,1\}\,.
\end{align}

\vspace{-0,1cm}
\noindent In comparison with \eqref{f8fevagen},
the formula \eqref{f10fevagen} determines the values of $\Psi^{\,\delta}_{\infty}(x;z)$
for arbitrary $ x\!\in\! \Bb{R}$ and
$z\in \Bb{H}\setminus S_{\!{\tn{\frown}}}^{\infty}$, and, it can be easily calculated that{\hyperlink{r40}{${}^{\ref*{case40}}$}}\hypertarget{br40}{}\vspace{-0,05cm}
\begin{align*}
      &    ({\partial}/{\partial z})\,   \Psi^{\,\delta}_{\infty}(x;z) =  \Phi^{\,\delta}_{\infty}(x;z) \ , \quad z\in \Bb{H}\setminus S_{\!{\tn{\frown}}}^{\infty} \, , \  x\in \Bb{R}\setminus\{0\} \, , \  \delta \in \{0,1\} \,.
\end{align*}

\vspace{-0,05cm}
\noindent This relationship enables us to apply the reasons as in the proof of Lemma~\ref{bsdenaclem1} to obtain for arbitrary $\delta \in \{0,1\}$ and $x\in \Bb{R}$ that{\hyperlink{r41}{${}^{\ref*{case41}}$}}\hypertarget{br41}{}\vspace{-0,1cm}
\begin{align}\label{f12fevagen}
      &  \Psi^{\hspace{0,02cm}\delta}_{\Bb{H}}(x;z) = \Psi^{\,\delta}_{\infty}(x;z)
        + \dfrac{z^{\delta}}{x^{\delta} z-(-x)^{1-\delta}} \ , \quad
        z\in  \eusm{E}^{0}_{\!{\tn{\frown}}} \subset \Bb{H}\setminus S_{\!{\tn{\frown}}}^{\infty}\,.
\end{align}



\section[\hspace{-0,30cm}. \hspace{0,11cm}Interpolation formula for the Klein-Gordon equation]{\hspace{-0,095cm}{\bf{.}} Interpolation formula for the Klein-Gordon equation}\label{mres}

\subsection[\hspace{-0,31cm}. \hspace{0,11cm}Hyperbolic  Fourier series in the upper half-plane.]{\hspace{-0,11cm}{\bf{.}} Hyperbolic  Fourier series in $\Bb{H}$.}\label{hypfouupp} \
In view of \eqref{11bsdenac}, \eqref{f5bsdenacth1} and \eqref{f1bsdenaclem1}, for arbitrary fixed $x \!\in\! \Bb{R}$ and $ \delta \!\in \! \{0,1\}$ we have $\Phi^{\hspace{0,025cm}\delta}_{\Bb{H}}(x; z) = \Phi^{\hspace{0,025cm}\delta}_{0}(x; z)$
for any $z\in \eusm{E}^{\infty}_{\!{\tn{\frown}}}\sqcup  \eusm{E}^{0}_{\!{\tn{\frown}}}$, where
$-1/\eusm{E}^{\infty}_{\!{\tn{\frown}}} = \eusm{E}^{0}_{\!{\tn{\frown}}}$. Then  for each $z\in \eusm{E}^{\infty}_{\!{\tn{\frown}}}$ it follows from \eqref{10bsdenac}, \eqref{f1bsdenaclem1}, \eqref{f4bsdenacth1} and \eqref{11bsdenac} that
\begin{align*}
    & \Phi^{\hspace{0,025cm}\delta}_{\Bb{H}}(x; -1/z) \!=\! \Phi^{\hspace{0,025cm}\delta}_{0}(x; -1/z) \!=\! \Phi^{\hspace{0,02cm}\delta}_{\infty}(x;-1/z)  -
\big(x^{\delta}(-1/z)\!-\!(-x)^{1-\delta}\big)^{-2} \\    &    =\! -z^{2} \Phi^{1-\delta}_{\infty}(x;z)
-z^{2}\big(x^{\delta}\!+\!z (-x)^{1-\delta}\big)^{-2} \! = \!
-z^{2} \Phi^{1-\delta}_{\Bb{H}}(x;z)
-z^{2}\big(x^{1-\delta}z\!-\! (-x)^{\delta}\big)^{-2}\! \! ,
\end{align*}

\noindent and since the functions appearing on the two sides of this identity are holomorphic in $\Bb{H}$ we obtain, by  the uniqueness theorem  for analytic functions (see
\cite[p.\! 78, Theorem 3.7(c)]{con}), that
\begin{align}\label{f1mres}
    & \hspace{-0,2cm}\Phi^{1-\delta}_{\Bb{H}}(x;z)\! +\! z^{-2}\Phi^{\hspace{0,025cm}\delta}_{\Bb{H}}(x; -1/z)\! =\!
- \big(x^{1-\delta}z\!-\! (-x)^{\delta}\big)^{-2}\!\! , \ z\!\in \!\Bb{H}, \ x\!\in\! \Bb{R}, \ \delta \!\in\! \{0,1\}.\hspace{-0,1cm}
\end{align}

\noindent In the case of $\delta =1$, we integrate \eqref{f1mres} with respect to the variable $z$ and obtain, by
\eqref{f13fevagen},\vspace{-0,4cm}
\begin{align*}
    &  \Psi^{0}_{\Bb{H}}(x;z) + \Psi^{\hspace{0,025cm}1}_{\Bb{H}}(x; -1/z) =
 \big(z+ x\big)^{-1} + \eta (x) \,,
\end{align*}

\noindent where, for $x\in \Bb{R}\setminus\{0\}$, we use \eqref{f10fevagen}, \eqref{f12fevagen} and $\lim_{t \to +\infty}\lambda ({\rm{i}}/t) = 1$ to see that
\begin{align*}
     \eta (x) & = \lim_{t \to +\infty}    \Psi^{\hspace{0,02cm}1}_{\Bb{H}}(x;\imag /t)=\!
- \frac{1}{2 }\!\!\!\!\!\int\limits_{{\fo{\gamma (-1,1)}}}\!\!\!\!\!
    \dfrac{\Theta_{3}(\zeta)^{4}\zeta {\rm{d}} \zeta}{x \zeta -1}   =- \frac{1}{2 }\!\!\!\!\!\int\limits_{{\fo{\gamma (-1,1)}}}\!\!\!\!\!
    \dfrac{-\Theta_{3}(-1/\zeta)^{4} {\rm{d}} \zeta}{\zeta^{2}\big(x  -1/\zeta\big)}
    \\    &    = -\frac{1}{2}\!\!\!\!\!
\int\limits_{{\fo{\gamma (-1,1)}}}\!\!\!\!\!\dfrac{\Theta_{3}(\zeta)^{4} {\rm{d}} \zeta}{
\zeta+x}  \ ,
\end{align*}

\noindent where we have applied \eqref{f2int} and \eqref{f3bint}(b). Thus,
\begin{align}\label{f2mres}
    &  \Psi^{0}_{\Bb{H}}(x;z) + \Psi^{\hspace{0,025cm}1}_{\Bb{H}}\left(x; -\dfrac{1}{z}\right) =
\dfrac{1}{z+x}  -  \frac{1}{2}\!\!\!\!\!
\int\limits_{{\fo{\gamma (-1,1)}}}\!\!\!\!\!\dfrac{\Theta_{3}(\zeta)^{4} {\rm{d}} \zeta}{
\zeta+x} \ , \quad  z\!\in \!\Bb{H}, \ x\!\in\! \Bb{R}\, .
\end{align}

\vspace{-0,15cm}
\noindent By substituting  the identities \eqref{f6fevagen} and \eqref{f9fevagen} into \eqref{f1mres}(with $\delta =1$) and
in \eqref{f2mres}, respectively, we see that for arbitrary $z\in\Bb{H}$ and $x, t\in \Bb{R}$ the following identities hold:
\begin{align}
    &
 \dfrac{1}{2 \pi \imag  (x-z)^{2} }= \sum\limits_{n \geqslant 1}
   \imag  \pi n \,\eurm{H}_{-n} (x)\, {\rm{e}}^{{\fo{{\rm{i}}\pi  n  z}} } + \dfrac{1}{z^{2}}\sum\limits_{n \geqslant 1}
   \imag  \pi n \,\eurm{M}_{-n} (x)\, {\rm{e}}^{\fo{{- \dfrac{{\rm{i}}\pi  n }{z} }} } \, ,
\label{f5fevagen}\end{align}

\vspace{-0,15cm}
\noindent and\vspace{-0,15cm}
\begin{align}
    &\hspace{-0,2cm}
 \dfrac{1}{2 \pi \imag }\bigg(\dfrac{1}{t-z}  -  \frac{1}{2}\!\!\!\!\!
\int\limits_{{\fo{\gamma (-1,1)}}}\!\!\!\!\!\dfrac{\Theta_{3}(\zeta)^{4} d \zeta}{
t- \zeta} \bigg) = \sum\limits_{n \geqslant 1}
  \eurm{H}_{-n} (t)\, {\rm{e}}^{{\fo{{\rm{i}}\pi  n  z}} } + \sum\limits_{n \geqslant 1}
   \eurm{M}_{-n} (t)\, {\rm{e}}^{\fo{{- \dfrac{{\rm{i}}\pi  n }{z} }} } \, .
 \hspace{-0,1cm}
\label{f3mres}\end{align}

\vspace{-0,2cm}
\noindent By multiplying the latter equality  by a function $F \in {H}^1_{+} (\Bb{R})$ we  integrate it with
respect to $t\in \Bb{R}$ and apply well-known property of functions in ${H}^1 (\Bb{H})$ (see \cite[p.\! 116]{ko}), together with the estimate \eqref{f3yint} and
the  identity (see \eqref{f9int})
\begin{align*}
  2 \pi \imag  \int\limits_{\Bb{R}} F (t)  \eurm{H}_{0} (t) {\rm{d}} t   &= \int\limits_{\Bb{R}} F (t)
\bigg(\ \frac{1}{2}\!\!\!
\int\limits_{{\fo{\gamma (-1,1)}}}\!\!\!\!\!\dfrac{\Theta_{3}(\zeta)^{4} {\rm{d}} \zeta}{
t- \zeta}   - \frac{1}{2}\!\!\!\!\!
\int\limits_{{\fo{\gamma (-1,1)}}}\!\!\!\!\!\dfrac{\Theta_{3}(\zeta)^{4} {\rm{d}} \zeta}{
t+ \zeta} \ \bigg){\rm{d}} t \\    &   =
  \int\limits_{\Bb{R}} F (t)
\bigg(\ \frac{1}{2}\!\!\!
\int\limits_{{\fo{\gamma (-1,1)}}}\!\!\!\!\!\dfrac{\Theta_{3}(\zeta)^{4} {\rm{d}} \zeta}{
t- \zeta}    \ \bigg){\rm{d}} t
 \, ,
\end{align*}

\vspace{-0,15cm}
\noindent to obtain that  $F$ enjoys the representation  (see \eqref{f03int})
\begin{align}
     &\hspace{-0,2cm} F (z) =  \eurm{h}_{0}(F) + \sum\limits_{n \geqslant 1}
   \eurm{h}_{n}(F)\, {\rm{e}}^{{\fo{{\rm{i}}\pi  n  z}} }+ \sum\limits_{n \geqslant 1}
   \eurm{m}_{n}(F) \, \exp \left({\fo{{- \dfrac{{\rm{i}}\pi  n }{z} }} }\right), \ \  z \in \Bb{H}\,.
\label{f3zmres}\end{align}

\noindent
We now extend this representation to a larger class of functions.
According to \cite[p.\! 60, Corollary 3.4]{gar}, $\Phi (\imag  (1-z)/(1+z)) \in {H}^1 (\Bb{D})$ if and only if $\Phi(z)/(z+\imag )^{2} \in {H}^1 (\Bb{H})$. Since then $F_{\varepsilon} (z):= \Phi(z)/(\imag + \varepsilon z)^{2} \in {H}^1 (\Bb{H})$ for any $\varepsilon >0$ we can expand $F_{\varepsilon}$ as in \eqref{f3zmres}. We then apply the estimates \eqref{f2fevagen}, \eqref{f3fevagen} together with the
Lebesgue dominated convergence theorem \cite[p.\! 161]{nat} as $0 < \varepsilon \to 0$ to get
the following representation for functions in a weighted  Hardy class ${H}^1 (\Bb{H})$ (cf. Theorem~3.1 in \cite[p.\! 14]{bon}).

\begin{theorem}\hspace{-0,15cm}{\bf{.}}\label{mresth2}
Let $F(z)/(z+\imag )^{2} \in {H}^1 (\Bb{H})$. Then
\begin{align*}
    & F (z) =  \eurm{h}_{0}(F) + \sum\limits_{n \geqslant 1}
   \eurm{h}_{n}(F)\, {\rm{e}}^{{\fo{{\rm{i}}\pi  n  z}} } + \sum\limits_{n \geqslant 1}
   \eurm{m}_{n}(F) \, {\rm{e}}^{\fo{{- \dfrac{{\rm{i}}\pi  n }{z} }} }  \ , \quad  z \in \Bb{H}\,,
\end{align*}

\vspace{-0,25cm}
\noindent where
\begin{align*}
    &  \eurm{h}_{n}(F)
\!:= \! \int\limits_{\Bb{R}}\!  F (t) \eurm{H}_{-n} (t) {\rm{d}} t\, , \ n\!\in\! \Bb{Z}_{\geqslant 0} \, , \ \ \eurm{m}_{n}(F)\! :=\!  \int\limits_{\Bb{R}}\!  F (t) \eurm{M}_{-n} (t) {\rm{d}} t \ , \ n\!\in\! \Bb{N} \, ,
\end{align*}

\noindent and $ |\eurm{h}_{n}(F)|, |\eurm{m}_{n}(F)| \leqslant (\pi^{6}/4)\, n^{2} \left\|F_{\Bb{D}}\right\|_{H^1_{+}}$, $n \geqslant 1$, $F_{\Bb{D}}(z):= F(z)/(z+\imag )^{2}$, $z\in \Bb{H}$.
\end{theorem}

In view of Jordan's lemma, for any $x, y > 0$ the function $z\mapsto\exp (\imag  x z- \imag  y/z)$ is bounded and holomorphic in
$\Bb{H}$, and we can calculate that
\begin{align}\nonumber
    \int\limits_{\Bb{R}} \dfrac{{\rm{e}}^{{\fo{{\rm{i}} x t- {\rm{i}} y/t}}} {\rm{d}} t}{t-z} &  := \lim\limits_{A\to+\infty} \int\limits_{-A}^{A} \dfrac{{\rm{e}}^{{\fo{{\rm{i}} x t- {\rm{i}} y/t}}} {\rm{d}} t}{t-z} \\[0,2cm]  &  =2 \pi {\rm{i}} \, {\rm{e}}^{{\fo{{\rm{i}} x z- {\rm{i}} y/z}}} \chi_{\Bb{H}}(z)  \, , \ \
      z \in \Bb{C}\setminus \Bb{R}  \ , \ x > 0\ , \  y \geqslant 0 \, .
\label{f4zmres}\end{align}

\noindent By multiplying \eqref{f3mres} by this function and after integrating over $t\in \Bb{R}$, we obtain as above,
\begin{align*}
    &\hspace{-0,2cm} {\rm{e}}^{{\fo{{\rm{i}} x z- {\rm{i}} y /z}}} =  \eusm{R}_{0}(x, -y) + \sum\limits_{n \geqslant 1}
   \eusm{R}_{n}(x, -y)\, {\rm{e}}^{{\fo{{\rm{i}}\pi  n  z}} } + \sum\limits_{n \geqslant 1}
    \eusm{R}_{n}(y, -x) \, {\rm{e}}^{\fo{{- \dfrac{{\rm{i}}\pi  n }{z} }} } , \end{align*}

\noindent for all $z \in \Bb{H}$ and $x,y \geqslant 0$, where
\begin{align}\label{f5mres}
    &   \eusm{R}_{n}(x, y)
\!:= \! \int\nolimits_{\Bb{R}}\  e{\rm{}}^{{\fo{{\rm{i}} x t+ {\rm{i }}y/t}}} \eurm{H}_{-n} (t) {\rm{d}} t\, , \quad n\!\in\! \Bb{Z}_{\geqslant 0} \, , \quad x,y \in \Bb{R}\,.
\end{align}

\noindent Here,  the change of variable $t^{\,\prime}= - 1/t$ and the symmetry property \eqref{f3intth1}
entail that, for arbitrary $x,y \in \Bb{R}$,
\begin{align}\label{f6mres}
    &  \eusm{R}_{n}(y, -x)
\!= \! \int\nolimits_{\Bb{R}}\,  {\rm{e}}^{{\fo{i{\rm{}} y t- {\rm{{\rm{i}}}} x/t}}} \eurm{H}_{-n} (t) {\rm{d}} t = \int\nolimits_{\Bb{R}}\,  {\rm{e}}^{{\fo{{\rm{i}} x t- {\rm{i}} y/t}}} \eurm{M}_{-n} (t) {\rm{d}} t  \ , \quad  n\in \Bb{N} \,,
\end{align}

\noindent while, as a consequence of \eqref{f24int}, we have
\begin{align}\label{f6zmres}
  \eusm{R}_{0}(x, -y)= \eusm{R}_{0}(y, -x) =\eusm{R}_{0}(-x, y) \ , \quad x,y \in \Bb{R}\,.
\end{align}

\subsection[\hspace{-0,31cm}. \hspace{0,11cm}Conjugate hyperbolic  Fourier series.]{\hspace{-0,11cm}{\bf{.}} Conjugate hyperbolic  Fourier series.}\label{conhypfouser}
For arbitrary $\varphi\in L^1 (\Bb{R})$, the series in the right-hand side of the equality \eqref{f06int}
is called {\emph{conjugate hyperbolic  Fourier series}} of $\varphi$, where  the coefficients are defined as in
\eqref{f05bint}.
Taking account the estimates \eqref{f2fevagen}, the result  \eqref{f06int}
can be improved as follows.
\begin{theorem}\hspace{-0,15cm}{\bf{.}}\label{mresth1}
Let  $\varphi\in L^1 (\Bb{R})$, the numbers $\{\eurm{h}_{n}^{\star}(\varphi)\}_{n \in \Bb{Z}}$\,,
$\{\eurm{m}_{n}^{\star}(\varphi)\}_{n \in \Bb{Z}_{\neq 0}}$ be defined as in \eqref{f05aint}, and
\begin{align}\label{f1mresth1}
    & \sum\nolimits_{n \in \Bb{Z}\setminus \{0\}}n^{2}\left(\left|\eurm{h}_{n}^{\star}(\varphi)\right|+ \left|\eurm{m}_{n}^{\star}(\varphi)\right|\right) < \infty .
\end{align}
\noindent  Then $\varphi$ can be expanded  into the conjugate hyperbolic  Fourier series \eqref{f06int},
which  converges absolutely and uniformly over all  $x\in \Bb{R}$.
\end{theorem}

We now apply Theorem~\ref{mresth1} to obtain the conjugate hyperbolic Fourier series expansion of the Poisson kernel{\hyperlink{r36}{${}^{\ref*{case36}}$}}\hypertarget{br36}{}
\begin{align}\nonumber
  &    \dfrac{1}{ \pi } \dfrac{y}{(t- x)^{2} + y^{2}} =  \eurm{H}_{0}(t)\! +\!
 \sum\limits_{n\geqslant 1 }\!
  \Big(  {\rm{e}}^{{\fo{ -\pi n \,(  y+ {\rm{i}}x) }}}
\eurm{H}_{n}(t)
\! +\!
{\rm{e}}^{{\fo{-\dfrac{ \pi n  }{  y+{\rm{i}} x}}}}
\eurm{M}_{n}(t)\Big) \\    &
+ \sum\limits_{n\geqslant 1 }\!
  \Big(  {\rm{e}}^{{\fo{- \pi n \,(  y - {\rm{i}}   x) }}}
\eurm{H}_{-n}(t)
\! +\!
{\rm{e}}^{{\fo{-\dfrac{ \pi n  }{  y - {\rm{i}}  x}}}}
\eurm{M}_{-n}(t)\Big) \ , \quad   \ t, x\in \Bb{R}, \ y > 0 \,.
\label{f1conhypfouser}\end{align}

\noindent This allows us to expand the harmonic extension to the upper half-plane of a given
$f \in L^1 (\Bb{R}, (1+x^2)^{-1}d x)$  given by convolution with the Poisson kernel in the form
\begin{align*}
    &  \dfrac{1}{\pi}\int\limits_{\Bb{R}} \dfrac{y}{(t- x)^{2} + y^{2}} f (t) {\rm{d}} t  =
  \eurm{h}_{0}(f)\! +\!
 \sum\limits_{n\geqslant 1 }\!
  \left(  {\rm{e}}^{{\fo{- \pi n \, (  y - {\rm{i}}   x) }}}
\eurm{h}_{n}(f)
\! +\!
{\rm{e}}^{{\fo{-\dfrac{  \pi n }{  y -{\rm{i}}   x}}}}
\eurm{m}_{n}(f)\right) \\    &
+
\sum\limits_{n\geqslant 1 }\!
  \left(  {\rm{e}}^{{\fo{ -\pi n\, (   y+ {\rm{i}} x) }}}
\eurm{h}_{-n}(f)
\! +\!
{\rm{e}}^{{\fo{-\dfrac{ \pi n }{  y+{\rm{i}}x}}}}
\eurm{m}_{-n}(f)\right)
 \ , \quad   \  x\in \Bb{R}, \ y > 0 \,.
\end{align*}

\noindent This formula can be understood as the regularization of the hyperbolic Fourier series
\eqref{f03int} of $f$ found by considering the harmonic extensions of the basis functions.

\subsection[\hspace{-0,31cm}. \hspace{0,11cm}Density in Hardy classes.]{\hspace{-0,11cm}{\bf{.}} Density in Hardy classes.}

By manipulations  similar to those employed in the proof of Corollary 3.3 in \cite[p.\! 59]{gar} we obtain in Section~\ref{pintreslem3} the following property of the Hardy class ${\rm{H}}^1_{+}(\Bb{R})$.
\begin{lemsectnine}\hspace{-0,15cm}{\bf{.}}\label{intreslem3}
   The linear subspace ${\rm{H}}^1_{+}(\Bb{R}) \cap \eurm{S} (\Bb{R})$ is  dense in ${\rm{H}}^1_{+}(\Bb{R})$.
\end{lemsectnine}

\noindent In view of \eqref{f06int}, every $\varphi \in {\rm{H}}^1_{+}(\Bb{R}) \cap \eurm{S} (\Bb{R})$,  can be
expanded in an absolutely convergent conjugate hyperbolic  Fourier series, all whose coefficients
with non-positive indexes are zero, as follows from \eqref{f05aint} and the known properties of the functions
from the class ${\rm{H}}^1_{+}(\Bb{R})$ (see \cite[p.\! 88, 2(iv)]{gar}). It now follows from \eqref{f1bs} and  \eqref{f2intth1} that the following property holds.
\begin{theorem}\hspace{-0,15cm}{\bf{.}}\label{bsdenth1}
   The function system  $\left\{\eurm{H}_{ n}(x)\right\}_{n\geqslant 1}\!\cup\! \left\{\eurm{M}_{ n}(x)\right\}_{n\geqslant 1 }\!\subset\!{\rm{H}}^1_{+ } (\Bb{R})$ is complete in ${\rm{H}}^1_{+} (\Bb{R})$, while the functions
$\left\{\eurm{H}_{- n}(x)\right\}_{n\geqslant 1}\!\cup\! \left\{\eurm{M}_{- n}(x)\right\}_{n\geqslant 1 }\!\subset\!{\rm{H}}^1_{- } (\Bb{R})$ form a complete system in ${\rm{H}}^1_{-} (\Bb{R})$.
\end{theorem}

\subsection[\hspace{-0,31cm}. \hspace{0,11cm}Interpolating functions and interpolation formula.]{\hspace{-0,11cm}{\bf{.}} Interpolating functions and interpolation formula.}\label{intfor}\
We apply  Jordan's lemma and the residue theorem in the same way as
for the proof of \eqref{f4zmres},  and obtain
\begin{align}\label{f1intfor}
    &   \int\limits_{\Bb{R}} \dfrac{{\rm{e}}^{{\fo{{\rm{i}} x t- {\rm{i}} y/t}}} {\rm{d}} t}{(t-z)^{2}} = -2 \pi\left(x+ \dfrac{y}{z^{2}}\right){\rm{e}}^{{\fo{{\rm{i}} x z- {\rm{i}} y/z}}}\chi_{\Bb{H}}(z) \ , \ \  z \in \Bb{C}\setminus \Bb{R},
\end{align}

\noindent for all $x> 0$, $y \geqslant 0$.  We proceed and apply this identity together with \eqref{f4zmres}
to the formulas \eqref{f9int} written in the form
\begin{align}\nonumber 
    &   \eurm{H}_{0} (t)\!=\dfrac{1}{4 \pi\imag} \!\!\! \!\!\! \int\limits_{\gamma (-1,1)} \!\!\!\! \Theta_{3}\left(z\right)^{4} \left(\dfrac{1}{t-z} - \dfrac{1}{t+z}\right)   {\rm{d}} z \, ,  &     &
    t \in \Bb{R}\,,  \\  &
     \eurm{H}_{-n} (t)= \eurm{H}_{n} (-t) = \dfrac{(-1)}{4 \pi^{2} n} \!\!\! \!\!\!  \int\limits_{\gamma (-1,1)}\!\!\!\!\!
\dfrac{S^{{\tn{\triangle}}}_{n} \!\left({1}/{\lambda (z)}\right){\rm{ d}} z}{(t-z)^{2}} \ ,   &     &     n\geqslant 1  \, , \     t \in \Bb{R}\,,
\label{f4intfor}\end{align}

\noindent and obtain from \eqref{f5mres}, by using \eqref{f3yint}, \eqref{f2bvslem1}(a) and \eqref{f27int}(a), that
the integral representations
\begin{align}\nonumber
    &  \eusm{R}_{n}(x, -y)=  \int\limits_{\Bb{R}}\!  {\rm{e}}^{{\fo{{\rm{i}} x t- {\rm{i}} y/t}}} \eurm{H}_{-n} (t) {\rm{d}} t= \dfrac{1}{2 \pi n} \!\!\! \!\!\!  \int\limits_{\gamma (-1,1)}\!\!\!\!\!
\left(x+ \dfrac{y}{z^{2}}\right){\rm{e}}^{{\fo{{\rm{i}} x z- {\rm{i}} y/z}}}
S^{{\tn{\triangle}}}_{n} \!\left(\dfrac{1}{\lambda (z)}\right) {\rm{d}} z \, ,  \\    &
  \eusm{R}_{0}(x, -y)= \! \int\limits_{\Bb{R}}\!  {\rm{e}}^{{\fo{{\rm{i}} x t- {\rm{i}} y/t}}} \eurm{H}_{0} (t) {\rm{d}} t= \dfrac{1}{2} \!\!\! \!\!\! \int\limits_{\gamma (-1,1)} \!\!\!\! \Theta_{3}\left(z\right)^{4} {\rm{e}}^{{\fo{{\rm{i}} x z- {\rm{i}} y/z}}}    {\rm{d}} z
\, , \ \ n\geqslant 1\,,
\label{f2intfor}\end{align}

\noindent hold for all $x, y \geqslant 0$,   because $\eusm{R}_{n} \in C (\Bb{R}^{2})$ for all $ n\geqslant 0$, as we see from    the estimates \eqref{f2fevagen} and \eqref{f3fevagen}.  In addition to the symmetry property
\eqref{f6zmres}, we observe that by substituting \eqref{f4intfor} into \eqref{f5mres}, while taking into account \eqref{f1intfor}  (interchange in orders of integration is justified by \eqref{f2bvslem1}(a) and \eqref{f27int}(a)), we arrive at
\begin{align}\label{f6intfor}
       \eusm{R}_{n}(-x, y)&=  \int\limits_{\Bb{R}}\!  {\rm{e}}^{{\fo{-{\rm{i}} x t+ {\rm{i}} y/t}}} \eurm{H}_{-n} (t) {\rm{d}} t = \int\limits_{\Bb{R}}\!  {\rm{e}}^{{\fo{{\rm{i}} x t- {\rm{i}} y/t}}} \eurm{H}_{n} (t) {\rm{d}} t  \\  & =
     \dfrac{(-1)}{4 \pi^{2} n} \!\!\! \!\!\!  \int\limits_{\gamma (-1,1)}\!\!\!\!\!
S^{{\tn{\triangle}}}_{n} \!\left(\dfrac{1}{\lambda (z)}\right)
 \int\limits_{\Bb{R}}\!  \dfrac{{\rm{e}}^{{\fo{{\rm{i}} x t- {\rm{i}} y/t}}}  {\rm{d}} t}{(t+z)^{2}}
 \ {\rm{ d}}  z = 0 \ , \ \ \  x, y > 0\,, \ n\geqslant 1 \,,
\nonumber \end{align}

\noindent and hence, since $\eusm{R}_{n}$ is continuous on $\Bb{R}^{2}$ for each $n\geqslant 1$, we find that
\begin{align}\label{f5intfor}
 & \eusm{R}_{n}(-x, y)=0 \,, \quad  x, y \geqslant 0\,, \ n\geqslant 1 \,.
\end{align}

\noindent This property also follows  directly  from $\eurm{H}_{n} \in {\rm{H}}^1_{+}(\Bb{R})$ (see \eqref{f1bs}) and the equality \eqref{f6intfor}. In Section~\ref{pintreslem3} we prove the following assertion.

\begin{theorem}\hspace{-0,15cm}{\bf{.}}\label{intforth1}
   Let $\{\eusm{R}_{\hspace{0.025cm}n}\}_{n \geqslant 0} $ be given by \eqref{f5mres}. Then $\eusm{R}_{\hspace{0.025cm}n}\in C (\Bb{R}^{2})$, and, in addition, for each $ n\geqslant 0$,
the restriction of the function  $\eusm{R}_{\hspace{0.025cm}n}$ to the quadrant
$\Bb{R}_{\geqslant 0}\times \Bb{R}_{\leqslant 0}$ extends to all of $\Bb{C}^{2}$ as an entire function of
two variables.  At the same time, we have
\begin{gather}\label{f1intforth1}
    \begin{array}{lllll}\hspace{-0,5cm}
 {\rm{{(a)}}}&  \ \  \eusm{R}_{\hspace{0.025cm}0} (\pi m, 0) = \delta_{0 m}\ ,&  \ \  \eusm{R}_{\hspace{0.025cm}0} (0, - \pi m) = \delta_{0 m}\ , & \ m \geqslant 0 \ , & \
  \\[0.2cm]
\hspace{-0,5cm} {\rm{{(b)}}}&  \ \  \eusm{R}_{\hspace{0.025cm}n} (\pi m, 0) = \delta_{n m}\ ,&   \ \  \eusm{R}_{\hspace{0.025cm} n} (0,-\pi m) = 0 \ , &  \ m \geqslant 0\ , & \ n \geqslant 1 \ .
 \end{array}
 \end{gather}

\noindent The function $\eusm{R}_{\hspace{0.025cm}0}$ satisfies
\begin{align}\label{f3intforth1}
   &  \eusm{R}_{0}(x, -y)= \eusm{R}_{0}(y, -x) =\eusm{R}_{0}(-x, y) , \
  \ \   \left|\eusm{R}_{0}(x, y)\right|\leqslant 3 \pi
    \, ,   &     &     x,y \in \Bb{R}\,; \\  &
    \left|\eusm{R}_{0}(x, -y)\right|\leqslant
 \pi^{7} K_{0} \left(\sqrt{2 \pi (x+y+1)}\right) ,   &     &    x, y \geqslant 0\,,
\label{f5intforth1}\end{align}

\noindent while for any $n \geqslant 1$ we have
\begin{align} &  \label{f4intforth1}
\eusm{R}_{\hspace{0.025cm}n} (-x, y)= 0\, , \ \ \ \ x, y \geqslant 0 \, ; \ \
 \ \ \ \ \ \ \ \  \left|\eusm{R}_{n}(x, y)\right|\leqslant  \pi^{7} n^{2}/2
    \, ,  &     &   x,y \in \Bb{R}\,;
\\
     &   \left|\eusm{R}_{n}(x, -y)\right|\leqslant  \pi^{12} e^{{\fo{2 \pi n  }}}
     \dfrac{(x+y)}{\sqrt{x+y+1}}\, K_{1} \left(2\sqrt{\pi (x+y+1)}\right) ,  &     &  x, y \geqslant 0\,,
\label{f2intforth1}\end{align}

\noindent where $K_{0} (x) =\int_{1}^{\infty}(t^{2}-1)^{-1/2} \exp(-x t)\,{\rm{d}} t$,  $K_{1} (x) =\int_{0}^{\infty} \exp(-x \sqrt{t^{2}+1})\,{\rm{d}} t$,
$x~>~0$, are the modified Hankel function satisfying  $K_{j} (x)\, x^{1/2} \, \exp (x) \to \sqrt{\pi/2}$ as $x\to+\infty$ for each $j \in \{0, 1\}$.
\end{theorem}

For  $\varphi\in L^{1} (\Bb{R})$, let  $U_{\varphi}$ be defined as in
\eqref{f0int}, i.e.,
\begin{align}\label{f7intfor}
    &  U_{\varphi} (x, y) = \int\limits_{\Bb{R}}  {\rm{e}}^{{\fo{{\rm{i}}  x t + {\rm{i}}  y / t}}} \varphi (t) {\rm{d}} t  \ , \quad x, y \in \Bb{R} \,.
\end{align}

\noindent
Then $u=U_\varphi$ solves the Klein-Gordon equation $u_{x y} + u = 0$ in the sense of distribution theory on $\Bb{R}^2$.
We recall that $u$ is  a {\emph{solution}} of  the Klein-Gordon equation on a given open subset
$G\! \subset\! \Bb{R}^{2}$ in the sense of distribution theory if $u \!\in\! \mathcal{D}^{\, \prime} (G)$ and the equality
$u_{x y} + u = 0$ holds in the sense of  an equality for the linear functionals $u_{x y},  u\!\in\! \mathcal{D}^{\, \prime} (G)$  on test functions   $C_{0}^{\infty}(G)$, the compactly supported
$C^\infty$-smooth  functions mapping $G$ to $\Bb{C}$  (cf. \cite[pp.\! 14, 34]{hor}). Given that our primary interest is in solutions of
the form $U_{\varphi}$, which are continuous  on $\Bb{R}^{2}$, it is more convenient to use
 alternative definition which  is equivalent to the definition above for continuous
 solutions of the Klein-Gordon equation.

\begin{defsectnine}\hspace{-0,15cm}{\bf{.}}\label{intfordef1}
Let $G$ be  an open convex subset of $\Bb{R}^{2}$. We say that $U$  is a {\it{continuous solution}}
 of  the Klein-Gordon equation
on $G$ if $U \in C \left(G\right) $ and
\begin{gather}\label{f1intfordef1}
   U (b, d)\! - \!U (b, c )\! -\! U (a, d)\! +\! U (a, c)\! +\! \!\int\nolimits_{c}^{\,d}\int\nolimits_{a}^{\,b} U (t, s) \, {\rm{d}} t {\rm{d}} s=0 \ ,
\end{gather}

\noindent for all $[a, b]\times [c, d] \subset G$,
where $-\infty < a < b < +\infty$ and  $-\infty < c < d < +\infty$.
\end{defsectnine}

We readily verify that \eqref{f1intfordef1} holds for $U=U_{\varphi}$ and $G = \Bb{R}^{2}$,  and, consequently,
$U_{\varphi}$  is a continuous solution of  the Klein-Gordon equation on $\Bb{R}^{2}$.

Next, let us assume that $\varphi$  meets the condition \eqref{f1mresth1} of  Theorem~\ref{mresth1}.  In view of the identities \eqref{f07int}, this is the same as requiring that \eqref{f08int} holds. Then the conjugate hyperbolic  Fourier series \eqref{f06int} of $\varphi$ can be substituted for  $\varphi$ into the integral in \eqref{f7intfor} to get the equality \eqref{f09int} by taking into account the identities \eqref{f011int}. As a consequence, we obtain the following interpolation formula for the solutions of  the Klein-Gordon equation of the
type \eqref{f7intfor}.
\begin{theorem}\hspace{-0,2cm}{\bf{.}}\label{intforth2} Suppose $\varphi \in L^{1}(\Bb{R})$
and that the continuous  solution $U_{\varphi}$ of the Klein-Gordon equation on $\Bb{R}^{2}$ is given by
\eqref{f7intfor}. If $U_{\varphi}$ satisfies \eqref{f08int} then
\begin{multline*}\hspace{-0,3cm}
  U_{\varphi} (x, y)\! = \!  U_{\varphi} (0, 0)\, \eusm{R}_{\hspace{0.025cm}0} (x, y)
    + \! \sum\limits_{n \geqslant 1} \!\Big[  U_{\varphi} (\pi n, 0)  \eusm{R}_{\hspace{0.025cm}n} (  x,  y)   + U_{\varphi} (0, -\pi n) \eusm{R}_{\hspace{0.025cm}n} ( -y,  -x)      \Big]\\[0,2cm]  \! + \!
   \sum\limits_{n \geqslant 1} \Big[\,  U_{\varphi} (-\pi n, 0)\,  \eusm{R}_{\hspace{0.025cm}n} (- x,  -y)   + U_{\varphi} (0, \pi n)\, \eusm{R}_{\hspace{0.025cm}n} ( y,  x) \Big] \,, \quad (x, y) \in  \Bb{R}^{2} \,, \hspace{-0,15cm}
   \end{multline*}
\noindent
where  the sequence $\{\eusm{R}_{\hspace{0.025cm}n}\}_{n \geqslant 0} $ is given by \eqref{f5mres}, and the series  converges absolutely and uniformly over all
$(x,y)\in \Bb{R}^2$.
\end{theorem}


\section[\hspace{-0,30cm}. \hspace{0,11cm}Extension of Theorem~\ref{intth0}]{\hspace{-0,095cm}{\bf{.}} Extension of Theorem~\ref{intth0}}\label{refa}

In this section, we derive the following extension of Theorem~\ref{intth0}. We make an effort to present a direct proof which
does not rely on ergodic theory, and which also covers the instance of Theorem~\ref{intth0} (when $N=M=0$ in \eqref{f1refath1} below). This makes the approach direct, but
we should mention that there are shortcuts available if we were to abandon this aim.
\begin{theorem}\hspace{-0,2cm}{\bf{.}}\label{refath1}
  Let $ \varphi  \in L^{1} (\Bb{R})$, $N, M \in \Bb{Z}_{\geqslant 0}$, $\delta_{0}:=1$, and  $\delta_{k}\in \left\{0,1\right\}$, $k \in \Bb{N}$. 
  Assume that\vspace{-0,2cm}
\begin{align}\label{f1refath1}
    &
    \begin{array}{l} {\rm{(a)}} \ \
\begin{displaystyle}
   \int\limits_{\Bb{R}} \varphi (t)  \dfrac{{\rm{e}}^{{\fo{{\rm{i}} \pi n t }}} -
   \begin{displaystyle}
     \sum\limits_{k=0}^{N-1}
   \end{displaystyle}
\dfrac{1}{k\, !}\left({\rm{i}} \pi n t \right)^{k}}{\vphantom{A^{A^{A}}}({\rm{i}} t)^{N-1+\delta_{N}}}  {\rm{d}} t  =   0\,,
\end{displaystyle}
      \\[0,7cm]  {\rm{(b)}} \ \
  \begin{displaystyle}
\int\limits_{\Bb{R}} \varphi (t) \left({\rm{e}}^{{\fo{ \dfrac{{\rm{i}} \pi n}{t}}}} - \sum\limits_{k=0}^{M-1}
\dfrac{1}{k\, !}\left( \dfrac{{\rm{i}} \pi n}{t}\right)^{k}\right)\left(\dfrac{t}{{\rm{i}}}\right)^{M-1+\delta_{M}} {\rm{d}} t=0\,,
\end{displaystyle}
    \end{array}
      \quad  n\in \Bb{Z} \, ,
\end{align}

\vspace{-0,1cm}
\noindent where $\sum_{k=0}^{-1}:=0$.   Then $ \varphi  = 0$.
\end{theorem}

In \cite{hed2}, the second and the third authors sharpened Theorem~\ref{intth0}
as follows:
\begin{align*}
     &   \varphi \! \in  \!L^{1} (\Bb{R}) \, , \
     \int\limits_{\Bb{R}} \! \!  {\rm{e}}^{{\fo{{\rm{i}} \pi n x}}}
    \varphi (x) {\rm{d}} x   \!= \!\int\limits_{\Bb{R}} \! \!  {\rm{e}}^{{\fo{{-\rm{i}}
     \pi n / x}}} \varphi (x) {\rm{d}} x \!= \!0 \, , \  n \!\geqslant \! 0
      \ \Rightarrow \  \varphi \! \in  \! {\rm{H}}^1_{+}(\Bb{R})\,.
\end{align*}

\noindent The analogue of this assertion for the biorthogonal system $\eurm{H}_{0}$, $\eurm{H}_{n}$, $\eurm{M}_{n} $, $n\in \Bb{Z}_{\neq 0}$, can be formulated as the following theorem.

\begin{theorem}\hspace{-0,2cm}{\bf{.}}\label{refath2} Let $ \varphi  \in L^{1} (\Bb{R})$.
Suppose that
\begin{align*}
      &    \int\limits_{\Bb{R}}  \varphi (t) \eurm{H}_{n} (t) {\rm{d}} t=
  \int\limits_{\Bb{R}}  \varphi (t) \eurm{M}_{n} (t) {\rm{d}} t=0 \ , \quad n\in \Bb{N} \,.
\end{align*}

\noindent Then  $\varphi \in  {\rm{H}}^1_{+}(\Bb{R})$.
\end{theorem}

\vspace{0.25cm}
\noindent
{\bf{Proof of Theorem~\ref{refath2}.}} We apply \eqref{f5fevagen} written in the form
\begin{align*}
      &   \dfrac{1}{2 \pi \imag  (x+z)^{2} }= \sum\limits_{n \geqslant 1}
   \imag  \pi n \,\eurm{H}_{n} (x)\, {\rm{e}}^{{\fo{{\rm{i}}\pi  n  z}} } + \dfrac{1}{z^{2}}\sum\limits_{n \geqslant 1}
   \imag  \pi n \,\eurm{M}_{n} (x)\, {\rm{e}}^{\fo{{- \dfrac{{\rm{i}}\pi  n }{z} }} }
     \ , \quad  x\in \Bb{R} \, , \  z\in \Bb{H}\,,
\end{align*}

\noindent
multiply it by the function $\varphi (x)$ and integrate over $x\in \Bb{R}$, to obtain from the given assumptions that
\begin{align*}
      &  0= \int\limits_{\Bb{R}} \dfrac{ \varphi (x ) {\rm{d}} x}{(x+z)^{2}} = \dfrac{{\rm{d}}}{{\rm{d}} z}\, F (z)  \, , \quad  F (z) := \int\limits_{\Bb{R}} \dfrac{ \varphi (x ) {\rm{d}} x}{x+z} \ , \qquad  z\in \Bb{H}\,.
\end{align*}

\noindent  It follows that the function $F$, which is holomorphic in $\Bb{H}$, must be constant.
Since $ \varphi  \in L^{1} (\Bb{R})$ automatically implies that $F (i t) \to 0$ as $t \to +\infty$, we find that
$F (z)=0$ for all $z\in \Bb{H}$. By the known characterizations of the space ${\rm{H}}^1_{+}(\Bb{R})$
(see \cite[p.\! 88, 2(ii)]{gar}), we obtain that  $\varphi \in  {\rm{H}}^1_{+}(\Bb{R})$. Theorem~\ref{refath2} follows. $\square$

\vspace{0.25cm}
\subsection[\hspace{-0,31cm}. \hspace{0,11cm}Proof of Theorem~\ref{refath1}.]{\hspace{-0,11cm}{\bf{.}} Proof of Theorem~\ref{refath1}.}\label{prtha}\
 We observe that an application of the finite difference
$\Delta^{r}[f](x)=\sum_{k=0}^{r}\binom{r}{k} (-1)^{r-k} f (x+k)$ over $n$ of order $r=N$ and $r=M$ to \eqref{f1refath1}(a) and to \eqref{f1refath1}(b), respectively, gives (see \cite[p.\! 29]{mil})
\begin{align*}
    &
    \begin{array}{l} {\rm{(a)}} \ \
\begin{displaystyle}
   \int\limits_{\Bb{R}} \varphi (t) {\rm{e}}^{{\fo{{\rm{i}} \pi n t }}} \dfrac{\left({\rm{e}}^{{\fo{{\rm{i}} \pi  t }}} -1\right)^{N}
  }{\vphantom{A^{A^{A}}}({\rm{i}} t)^{N-1+\delta_{N}}} \, {\rm{d}} t  =   0\,,
\end{displaystyle}
      \\[0,7cm]  {\rm{(b)}} \ \
  \begin{displaystyle}
\int\limits_{\Bb{R}} \varphi (t) {\rm{e}}^{{\fo{ \dfrac{{\rm{i}} \pi n}{t}}}}\left({\rm{e}}^{{\fo{ \dfrac{{\rm{i}} \pi }{t}}}} - 1\right)^{M} \left(\dfrac{t}{{\rm{i}}}\right)^{M-1+\delta_{M}} {\rm{d}} t=0\,,
\end{displaystyle}
    \end{array}
      \quad  n\in \Bb{Z} \,,
\end{align*}

\noindent or, in the notation $\varphi^{J} (t):= \varphi (-1/t)/t^{2}$, $t \in \Bb{R}_{\neq 0}$,
where $\varphi^{J}\in L^{1} (\Bb{R})$ by a change-of-variables, the conditions may be written as follows:
\begin{align*}
    &\hspace{-0,3cm}
      \int\limits_{\Bb{R}} \varphi (t) {\rm{e}}^{{\fo{{\rm{i}} \pi n t }}} \dfrac{\left({\rm{e}}^{{\fo{{\rm{i}} \pi  t }}} -1\right)^{N}
  }{\vphantom{A^{A^{A}}} t^{N-1+\delta_{N}}}\,  {\rm{d}} t  =
\int\limits_{\Bb{R}} \varphi^{J}  (t) {\rm{e}}^{{\fo{ {\rm{i}} \pi n t  }}}
 \dfrac{\left(1- {\rm{e}}^{{\fo{-{\rm{i}} \pi  t }}}\right)^{M}
  }{\vphantom{A^{A^{A}}} t^{M-1+\delta_{M}}} \, {\rm{d}} t=0\,,
      \ \   n\in \Bb{Z} \,.\hspace{-0,1cm}
\end{align*}

\noindent From this we obtain that
\begin{align}\label{f4refath1}
    &
      \begin{array}{l} {\rm{(a)}} \ \
 \dfrac{\varphi (t)}{ t^{N-1+\delta_{N}}}+
 \begin{displaystyle}
  \sum\limits_{k \in \Bb{Z}_{\neq 0}}
 \end{displaystyle}
 \dfrac{\varphi (t+2k)}{ (t+2k)^{N-1+\delta_{N}}} = 0 \, , \
      \\[0,7cm]  {\rm{(b)}} \ \
  \dfrac{\varphi^{J} (t)}{ t^{M-1+\delta_{M}}}+  \begin{displaystyle}
  \sum\limits_{k \in \Bb{Z}_{\neq 0}}
 \end{displaystyle} \dfrac{\varphi^{J} (t+2k)}{ (t+2k)^{M-1+\delta_{M}}} = 0 \, , \
    \end{array} \end{align}

\noindent for almost all $t \in [-1,1]$ because the system $\{\exp (i \pi n x)\}_{n \in \Bb{Z}}$ is complete
in $L^{\infty} ([-1,1])$. Since the right-hand side series in \eqref{f4refath1}  represent  integrable functions on
$[-1,1]$ we conclude that the  left-hand side functions in \eqref{f4refath1}, being obviously integrable on $\Bb{R}\setminus [-1,1]$,  are also integrable on
 $[-1,1]$, and hence
\begin{align*}
    &  t^{-p}\varphi (t)  \, , \   t^{-q}\varphi^{J} (t) \in
    L^{1} (\Bb{R}) \ , \ \ \   p\!:=\! N\!-\!1\!+\!\delta_{N}\!\geqslant\! 0 \, , \  q\!:= \!M\!-\!1\!+\!\delta_{M}\!\geqslant\! 0\,.
\end{align*}

\noindent If for arbitrary $\sigma \in \{1, -1\}$ we introduce
\begin{align}\label{f5zrefath1}
    &  F_{\sigma} (t):=  t^{-p}\varphi (t)+ \sigma t^{-q}\varphi^{J} (t) \ , \quad  t \in \Bb{R}_{\neq 0} \ , \quad   F_{\sigma} \in  L^{1} (\Bb{R})   \, , \
\end{align}

\noindent then
\begin{align*}
    &   F_{\sigma} (-1/t)=
(-1)^{p} t^{p+2}\varphi^{J} (t)+
(-1)^{q}     t^{q+2}    \sigma\varphi (t)  \, , \
\end{align*}

\noindent from which
\begin{align*}
    & F_{\sigma} (-1/t)/t^{p+q+2 }=(-1)^{p} t^{-q}\varphi^{J} (t)+
(-1)^{q}     t^{-p}    \sigma\varphi (t)=\sigma (-1)^{q} F_{\sigma(-1)^{p+q}}(t) \, ,
\end{align*}

\noindent and thus
\begin{align*}
    & \hspace{-0,3cm} F_{\sigma}(t)=\sigma (-1)^{p} F_{\sigma(-1)^{p+q}} (-1/t)/t^{p+q+2 } \, , \quad \sigma \in \{1, -1\}\,, \ t \in \Bb{R}_{\neq 0}\,. \hspace{-0,1cm}
\end{align*}

\noindent Then it follows from \eqref{f4refath1} that\vspace{-0,3cm}
\begin{align*}
    & - \! F_{\sigma}(t) \! =\!\! \sum\limits_{k \in \Bb{Z}_{\neq 0}}\!\!F_{\sigma}(2k\!+\!t) \!= \!\!
    \sum\limits_{k \in \Bb{Z}_{\neq 0}}\!\!F_{\sigma}(-2k\!+\!t ) \!= \!
   \sigma (-1)^{q}\! \sum\limits_{k \in \Bb{Z}_{\neq 0}}\!\!
   \dfrac{F_{\sigma(-1)^{p+q}} \left(\dfrac{1}{2k\!-\!t}\right)}{(2k\!-\!t)^{p+q+2 }}\ ,
\end{align*}

\noindent i.e.,
\begin{align}\label{f7refath1}
    &\hspace{-0,25cm}  F_{\sigma}(t) \! =\!  \sigma (-1)^{q+1} {\mathbf{T}}_{p+q +1} \left[F_{\sigma(-1)^{p+q}}\right] (t)
    \, , \ \  t \!\in\! [-1,1]\setminus\{0\} \,, \ \sigma\! \in \!\{1, -1\} \,,\hspace{-0,1cm}
\end{align}

\noindent where it can be easily seen that
   for arbitrary  $f \in L_1 \big([-1,1]\big)$ in the space $L_1 \big([-1,1]\big)$  there exists
   the limit\vspace{-0,3cm}
    \begin{align}\label{f8refath1}
        &   {\mathbf{T}}_{p+q +1} [f] (x) \ := \ \lim\limits_{n \to +\infty}  \sum_{\substack{{\fo{n=-N}}\\[0.05cm] {\fo{n\neq 0}}}}^{{\fo{N}}} \dfrac{ f \big({1}/{(2n- x)}\big)}{(2n-x)^{p+q +2}} \in L_1 \big([-1,1]\big) \,.
    \end{align}

\noindent In view of $|2k-x|\geqslant 1$, $k \in \Bb{Z}_{\neq 0}$, $x \in [-1,1]$, the operator \eqref{f8refath1}
for arbitrary Borel set $A \subset (-1,1)$  possesses the following essential property
\begin{align*}
    & \int\limits_{A}\left| {\mathbf{T}}_{p+q +1} [f] (x) \right|\,{\rm{d}} x \leqslant\!\!\! \sum\limits_{n \in \Bb{Z}_{\neq 0}}
  \int\limits_{A}  \dfrac{ \left|f \big({1}/{(2n- x)}\big)\right|}{|2n-x|^{p+q +2}}\,{\rm{d}} x\leqslant\!\!\!\sum\limits_{n \in \Bb{Z}_{\neq 0}}  \int\limits_{A}  \dfrac{ \left|f \big({1}/{(2n- x)}\big)\right|}{(2n-x)^{2}}\,{\rm{d}} x \\    &
 =\sum\limits_{n \in \Bb{Z}_{\neq 0}} \ \  \int\limits_{1/(2n-A)} \!\!\!\!\!\!|f (x)| \,{\rm{d}} x =
 \int\limits_{{\fo{\sqcup_{{\fo{\, n \in \Bb{Z}_{\neq 0}}}} \ 1/(2n-A)}}} \!\!\!\!\!\!|f (x)| \,{\rm{d}} x \, , \
\end{align*}

\noindent or, in the notation \eqref{f1bsden},
\begin{align}\label{f9refath1}
    & \hspace{-0,25cm} \int\limits_{A}\left| {\mathbf{T}}_{p+q +1} [f] (x) \right|\,{\rm{d}} x \leqslant \!\!\!\!\! \int\limits_{\omega_{1} (A)}\!\!\!\!\!|f (x)| \,{\rm{d}} x  \, , \  A \subset (-1,1)\,, \ \omega_{1} (A) :=
    \underset{n \in \Bb{Z}_{\neq 0}}{\sqcup} \phi_{n} (A),\hspace{-0,1cm}
\end{align}

\noindent because the sets $1/(2n-(-1,1))$, $n \in \Bb{Z}_{\neq 0}$, are disjoint.  Since $\omega_{1} (A) \subset (-1,1)$ we can successively apply  \eqref{f9refath1} to get
\begin{align}\label{f10refath1}
    &  \begin{array}{l}
    \begin{displaystyle}
      \int\limits_{A}\left| {\mathbf{T}}^{N}_{p+q +1} [f] (x) \right|\,{\rm{d}} x \leqslant \!\!\!\!\! \int\limits_{{\fo{\omega_{N} (A)}}}\!\!\!\!\!|f (x)| \,{\rm{d}} x  \, , \  A \subset (-1,1)\,,  \end{displaystyle}
      \\[0,7cm]
        \omega_{N} (A) :=   \underset{{\fo{n_{1}, ..., n_{N} \in \Bb{Z}_{\neq 0}}}}{\sqcup} \phi_{n_{1}, ..., n_{N}} (A)\, , \  \ N \in \Bb{N}\,,
       \end{array}
\end{align}

\noindent where for the Lebesgue measure $m$ on the real line and a  Borel set $A \subset (-1,1)$  after the similar manipulations  we obtain
\begin{align}\label{f11refath1}
    &  m \big(\omega_{N} (A)\big)=  \int\limits_{A} {\mathbf{T}}^{N}_{1} [1] (x) \,{\rm{d}} x \, , \qquad  A \subset (-1,1)\,,\  \ N \in \Bb{N}\,.
\end{align}

\noindent Here, in the notations \eqref{f1pf8contgen},  \eqref{f3pf8contgen} and \eqref{f0preresauxlem1},
for any $x\!\in\! [-1,1]$ we have
\begin{align}\label{f12refath1}
    & \hspace{-0,25cm} {\mathbf{T}}^{N}_{1} [1] (x) =\hspace{-0,5cm} \sum_{{\fo{n_{1}, n_{2}, ... n_{N}\! \in\! \Bb{Z}_{\neq 0}}}} \frac{1}{\left(p_{N}^{\eufm{n}}  x + q_{N}^{\eufm{n}}\right)^{2}} \ , \quad   \eufm{n}\!=\!(n_{N}, ..., n_{1})\!\in\! \Bb{Z}_{\neq 0}^{N} \, , \ N \geqslant 1\,,\hspace{-0,15cm}
\end{align}

\noindent because, in view of $\phi^{\,\prime}_{{\fo{n_{1}, n_{2}, ... n_{N-1}}}}(\phi_{n_{N}} (x))\phi^{\, \prime}_{n_{N}} (x)= \phi^{\,\prime}_{{\fo{n_{1}, n_{2}, ... n_{N}}}}(x)$, $N \geqslant 2\, $,
\begin{align*}
    &  {\mathbf{T}}^{N}_{1} [1] (x) =\hspace{-0,5cm} \sum_{{\fo{n_{1}, n_{2}, ... n_{N}\! \in\! \Bb{Z}_{\neq 0}}}} \hspace{-0,75cm}    \phi^{\,\prime}_{{\fo{n_{1}, n_{2}, ... n_{N}}}}(x) \, , \quad \phi_{{\fo{n_{1}, n_{2}, ... n_{N}}}}(x) =
\dfrac{z p^{\eufm{n}}_{N-1} + q^{\eufm{n}}_{N-1}}{z p^{\eufm{n}}_{N} + q^{\eufm{n}}_{N}}\, , \ N \geqslant 1 \, .
\end{align*}

\noindent But according to \eqref{f7pf8contgen}, $\left| p^{\eufm{n}}_{N}\right| < \left| q^{\eufm{n}}_{N}\right|$, and hence, we obtain from
\begin{align*}
      &   \left|p_{N}^{\eufm{n}}  x + q_{N}^{\eufm{n}}\right| \geqslant  \left|q_{N}^{\eufm{n}}\right| - \left|x\right| \left|p_{N}^{\eufm{n}}\right| \geqslant (1-|x|) |q_{N}^{\eufm{n}}|
\end{align*}

\noindent
and \eqref{f12refath1} that
\begin{align*}
    &  {\bf{T}}_{2}^{N} [1] (x) \leqslant \dfrac{{\bf{T}}_{2}^{N} [1] (0)}{(1-|x|)^{2}}  \ , \quad  x \in [-1,1] \, ,
\end{align*}

\noindent and, consequently, \eqref{f11refath1} implies that
\begin{align}\label{f14refath1}
    &   m \big(\omega_{N} ([\alpha, \beta])\big)\leqslant  \dfrac{2 {\bf{T}}_{2}^{N} [1] (0)}{
    \min \left\{(1+ \alpha)^{2}, \left(1-\beta\right)^{2}\right\}    } \, , \ \
    -1\! <\! \alpha\! < \!\beta\! <\! 1\,,\  \ N\! \in\! \Bb{N}\,.
\end{align}

\noindent Iterating \eqref{f7refath1} gives
\begin{align*}
    &\hspace{-0,25cm}  F_{\sigma}(t) \! =\!   {\mathbf{T}}^{2N}_{p+q +1} \left[F_{\sigma}\right] (t)
    \, , \ \  t \!\in\! [-1,1]\setminus\{0\} \,, \ \sigma\! \in \!\{1, -1\} \,, N \in \Bb{N}\,,\hspace{-0,1cm}
\end{align*}

\noindent which together with   \eqref{f10refath1} lead us to the conclusion that
\begin{align}\label{f16refath1}
    &  \hspace{-0,3cm} \int\limits_{\alpha}^{\beta}\left| F_{\sigma} (x) \right|\,{\rm{d}} x \leqslant \hspace{-0,25cm}\int\limits_{{\fo{\omega_{2 N} ([\alpha, \beta])}}} \hspace{-0,45cm}|F_{\sigma} (x)| \,{\rm{d}} x
    \,, \ \sigma\! \in \!\{1, -1\}     \, , \ \
        -1\! <\! \alpha\! < \!\beta\! <\! 1\,,\  \ N\! \in\! \Bb{N}\,.
\hspace{-0,1cm} \end{align}

\noindent But in view of \eqref{f1arefalem1} below and \eqref{f14refath1} we have $\limlw\nolimits_{N \to \infty} m \big(\omega_{2 N} ([\alpha, \beta])\big) =0$, which by $F_{\sigma} \in  L^{1} ([-1,1])$ and \eqref{f16refath1} yields that
  $ F_{\sigma} (x) = 0$ for almost all $x\in [-1,1]$ and  for any $\sigma\! \in \!\{1, -1\}$.
 Definition  \eqref{f5zrefath1} of the functions $F_{\sigma}$ gives that $\varphi (x)=0$ and $\varphi (-1/x)=0$
 for almost all $x\in [-1,1]$. Thus,  $\varphi (x)=0$  for almost all $x\in \Bb{R}$ and Theorem~\ref{refath1} follows. $\square$

\vspace{0.25cm}
We supply the following auxiliary lemma which was referred to in the proof of Theorem~\ref{refath1}.
\begin{lemsectten}\hspace{-0,2cm}{\bf{.}}\label{refalem1}
Let ${\bf{T}}_{1}$ be defined as in \eqref{f1bsfir}. Then
 \begin{align}\label{f1arefalem1}
    &   \limlw\limits_{N \to \infty} {\bf{T}}_{1}^{2 N} [1] (0) = 0 \ .
 \end{align}
 \end{lemsectten}

\noindent
{\bf{Proof of Lemma~\ref{refalem1}.}} We first obtain  that
\begin{align}\label{f1refalem1}
    &   \limlw\limits_{N \to \infty} {\bf{T}}_{1}^{N} [1] (0) = 0 \ .
 \end{align}

\noindent
 Assume that there exists $\varepsilon > 0$ such that
\begin{align*}
    & {\bf{T}}_{1}^{N} [1] (0) \geqslant \varepsilon > 0  \ \quad  \ \mbox{for all} \ N \geqslant 1 \ .
\end{align*}

\noindent Since ${\bf{T}}_{1}^{N} [1] (x)$ is  nondecreasing on $ [0,1]$
(see \cite[p.\! 1715, Proposition 3.7.2]{hed1}) then
\begin{align*}
    & {\bf{T}}_{1}^{N} [1] (x) \geqslant \varepsilon > 0  \ \quad  \ \mbox{for all} \ N \geqslant 1 \, , \  \ x \in [0,1] \ .
\end{align*}

\noindent For arbitrary $x \in [0,1]$ and $N \geqslant 1$  we obviously have
\begin{align}\nonumber
      {\bf{T}}_{1}^{N} [1] (x) &=  \sum_{n \geqslant  1} \left[\frac{ {\bf{T}}_{1}^{N-1} [1]\left(\dfrac{1}{2 n- x }\right) }{( 2 n- x)^{2}} +
 \frac{{\bf{T}}_{1}^{N-1} [1] \left(\dfrac{1}{2 n+ x }\right) }{( 2 n+ x)^{2}}\right]  \\    &   >
\frac{ {\bf{T}}_{1}^{N-1} [1]\left(\dfrac{1}{2 - x }\right) }{( 2 - x)^{2}} +
 \frac{{\bf{T}}_{1}^{N-1} [1] \left(\dfrac{1}{2 + x }\right) }{( 2 + x)^{2}} \ .
\label{f2refalem1}\end{align}

\noindent
 and, in view of the inequality
\begin{align*}
    &  \dfrac{1}{2-x} \geqslant x  \ \Leftrightarrow \  1 \geqslant 2 x - x^{2}  \ \Leftrightarrow \  (x-1)^{2} \geqslant 0 \, , \quad x \in [0,1]\,,
\end{align*}

\noindent and of nondecreasing property of ${\bf{T}}_{1}^{N-1} [1] (x)$ on $ [0,1]$, we get
\begin{align*}
    &  {\bf{T}}_{1}^{N-1} [1]\left(\dfrac{1}{2 - x }\right) \geqslant {\bf{T}}_{1}^{N-1} [1](x)\, , \quad x \in [0,1] \ .
\end{align*}

\noindent As a consequence, we derive from \eqref{f2refalem1} that
\begin{align*}
    &\hspace{-0,3cm}  {\bf{T}}_{1}^{N} [1] (x) \geqslant \dfrac{{\bf{T}}_{1}^{N-1} [1](x)}{( 2 - x)^{2}}  + \frac{{\bf{T}}_{1}^{N-1} [1] \left(\dfrac{1}{2 + x }\right) }{( 2 + x)^{2}} \geqslant  \dfrac{{\bf{T}}_{1}^{N-1} [1](x)}{( 2 - x)^{2}}  + \dfrac{\varepsilon}{9} \ , \ \ \  N \geqslant 1 \, .
\end{align*}

\noindent Iterating this procedure for $N \geqslant 4$, we obtain
\begin{align*}
    &   {\bf{T}}_{1}^{N} [1] (x) \geqslant \dfrac{\varepsilon}{9} + \dfrac{1}{( 2 - x)^{2}} {\bf{T}}_{1}^{N-1} [1](x)   \geqslant \dfrac{\varepsilon}{9} + \dfrac{1}{( 2 - x)^{2}} \left[\dfrac{\varepsilon}{9} +  \dfrac{1}{( 2 - x)^{2}} {\bf{T}}_{1}^{N-2} [1](x)\right]  \\ & =
\dfrac{\varepsilon}{9}\left(1+  \dfrac{1}{( 2 - x)^{2}}\right) + \dfrac{1}{( 2 - x)^{4}} {\bf{T}}_{1}^{N-2} [1](x) \\    &   \geqslant \dfrac{\varepsilon}{9}\left(1+  \dfrac{1}{( 2 - x)^{2}}\right) + \dfrac{1}{( 2 - x)^{4}} \left[\dfrac{\varepsilon}{9} + \dfrac{1}{( 2 - x)^{2}} {\bf{T}}_{1}^{N-3} [1](x)\right]  \\ & =
\dfrac{\varepsilon}{9}\left(1+  \dfrac{1}{( 2 - x)^{2}} + \dfrac{1}{( 2 - x)^{4}}\right) + \dfrac{1}{( 2 - x)^{6}} {\bf{T}}_{1}^{N-3} [1](x)\geqslant \ldots  \\    &   \geqslant \dfrac{\varepsilon}{9}\left(1+  \dfrac{1}{( 2 - x)^{2}} + \dfrac{1}{( 2 - x)^{4}}+ \ldots + \dfrac{1}{( 2 - x)^{2 (m-1)}}\right) + \dfrac{1}{( 2 - x)^{2m}} {\bf{T}}_{1}^{N-m} [1](x)\\    & \geqslant \ldots\geqslant   \dfrac{\varepsilon}{9} \sum_{k=0}^{N-1} \dfrac{1}{( 2 - x)^{2k}} +  \dfrac{1}{( 2 - x)^{2N}} \ ,
\end{align*}

\noindent i.e.,
\begin{align*}
    &   {\bf{T}}_{1}^{N} [1] (x) \geqslant \dfrac{\varepsilon}{9} \sum_{k=0}^{N-1} \dfrac{1}{( 2 - x)^{2k}} +  \dfrac{1}{( 2 - x)^{2N}} \ , \qquad x\in [0,1] \, .
\end{align*}

\noindent Hence,
\begin{align}\nonumber
     \int_{0}^{1} {\bf{T}}_{1}^{N} [1] (x) \,{\rm{d}} x &\geqslant \dfrac{\varepsilon}{9} \sum_{k=0}^{N-1} \int_{0}^{1}\dfrac{{\rm{d}} x}{( 2 - x)^{2k}} +  \int_{0}^{1}\dfrac{{\rm{d}} x}{( 2 - x)^{2N}} \\    &    = \dfrac{\varepsilon}{9} \left( 1 + \sum_{k=1}^{N-1}\dfrac{1- \dfrac{1}{2^{2k-1}}}{2k-1}\right) +
\dfrac{1- \dfrac{1}{2^{2N-1}}}{2N-1} \ .
\label{f3refalem1}\end{align}

\noindent In view of \eqref{f10refath1}, for any $N \geqslant 1$
the set  $\omega_{N} ((-1,1))$ contains at least all irrational point
of $(-1,1)$. Therefore  its Lebesgue measure is equal to $2$.  Then, by virtue of \eqref{f11refath1},
\begin{align*}
    &  \int_{-1}^{1} {\bf{T}}_{1}^{N} [1] (x) \,{\rm{d}} x =m \Big(
    \omega_{N} \big((-1,1)\big)\Big) = 2 \, , \quad N \geqslant 1 \ , \
\end{align*}

\noindent and since ${\bf{T}}_{1}^{N} [1] (x)$ is  even on $ [0,1]$
(see \cite[p.\! 1715, Proposition 3.7.2]{hed1}) we derive from \eqref{f3refalem1} that
\begin{align*}
     1 &= \dfrac{1}{2}\int_{-1}^{1} {\bf{T}}_{1}^{N} [1] (x) \,{\rm{d}} x = \int_{0}^{1} {\bf{T}}_{1}^{N} [1] (x) \,{\rm{d}} x  \\  & \geqslant \dfrac{\varepsilon}{9}
\sum_{k=1}^{N-1}\dfrac{1}{2k-1} - \dfrac{\varepsilon}{9}\sum_{k=1}^{N-1} \dfrac{1}{2^{2k-1}}\cdot \dfrac{1}{2k-1} \geqslant   -\dfrac{\varepsilon}{9} + \dfrac{\varepsilon}{9}\sum_{k=1}^{N-1}\dfrac{1}{2k-1} \to + \infty \ ,
\end{align*}

\noindent as $N \to \infty$. This contradiction proves \eqref{f1refalem1}.

\vspace{0.15cm} Next, to establish \eqref{f1arefalem1}, we appeal  to the formula
\eqref{f12refath1} with $N\geqslant 2$. In view of \eqref{f2preresauxlem1a}, for arbitrary $x \in (-1,1)$ we have
\begin{align}\label{f4refalem1}  & \hspace{-0,25cm}
\left\{
\begin{array}{ll}
   y_{k}= 2 n_{k} y_{k-1}-y_{k-2} \, , \   \  1 \leqslant k \leqslant N \,;   &
   y_{k} := p_{k}^{\eufm{n}}  x + q_{k}^{\eufm{n}}  \ ,  \ \   -1 \leqslant k \leqslant N \, ;    \\
 a:=  y_{-1} =p_{-1}^{\eufm{n}}  x + q_{-1}^{\eufm{n}}=x \in (-1,1) \, ,   &    b:= y_{0} =p_{0}^{\eufm{n}}  x + q_{0}^{\eufm{n}}=1 \,,
\end{array}
\right. \hspace{-0,1cm}
\end{align}

\noindent where $ \eufm{n}\!=\!(n_{N}, ..., n_{1})\!\in\! \Bb{Z}_{\neq 0}^{N}$.
These relationships mean that we can apply conclusion \eqref{f2preresauxprlem2} to the finite collection of numbers $\{y_{k}\}_{k = -1}^{Q}$ with $Q=N$, according to which,
\begin{align*}
     &   1 = \left|p_{0}^{\eufm{n}}  x + q_{0}^{\eufm{n}}\right| <
     \left|p_{1}^{\eufm{n}}  x + q_{1}^{\eufm{n}}\right|<\ldots<\left|p_{N}^{\eufm{n}}  x + q_{N}^{\eufm{n}}\right| \ , \quad  x \in (-1,1)\,.
\end{align*}

\noindent In particular,
\begin{align*}
      & \hspace{-0,25cm}  \dfrac{p_{N-2}^{\eufm{n}}  x + q_{N-2}^{\eufm{n}}}{p_{N-1}^{\eufm{n}}  x + q_{N-1}^{\eufm{n}}} \in (-1,1) \, , \  x \in (-1,1) \, , \  \eufm{n}\!=\!(n_{N}, ..., n_{1})\!\in\! \Bb{Z}_{\neq 0}^{N} \, , \ N\geqslant 2\,.\hspace{-0,1cm}
\end{align*}

\noindent Then we derive from \eqref{f12refath1} and \eqref{f4refalem1} that
\begin{align*}
      & {\mathbf{T}}^{N}_{1} [1] (x) =\hspace{-0,35cm} \sum_{{\fo{n_{1}, n_{2}, ... n_{N}\! \in\! \Bb{Z}_{\neq 0}}}}\hspace{-0,15cm} \left(p_{N}^{\eufm{n}}  x + q_{N}^{\eufm{n}}\right)^{-2}
       =\hspace{-0,35cm} \sum_{{\fo{n_{1}, n_{2}, ... n_{N-1}\! \in\! \Bb{Z}_{\neq 0}}}}\hspace{-0,15cm} \left(p_{N-1}^{\eufm{n}}  x + q_{N-1}^{\eufm{n}}\right)^{-2}
 \times  \\    & \times   \sum_{{\fo{ n_{N}\! \in\! \Bb{Z}_{\neq 0}}}}
  \left(2n_{N}-\dfrac{p_{N-2}^{\eufm{n}}  x + q_{N-2}^{\eufm{n}}}{p_{N-1}^{\eufm{n}}  x + q_{N-1}^{\eufm{n}}}\right)^{-2}\leqslant {\mathbf{T}}^{N-1}_{1} [1] (x)\max_{{\fo{y\in [-1, 1]}}}\sum_{{\fo{n\in\! \Bb{Z}_{\neq 0}}}}\frac{1}{(2n + y)^{2}} \\    &
 = {\mathbf{T}}^{N-1}_{1} [1] (x)\max_{{\fo{y\in [-1, 1]}}}
  \left(\frac{\pi^{2}}{4\sin^{2} (\pi y/2)} - \frac{1}{y^{2}}\right)=
\left(\frac{\pi^{2}}{4} -1\right){\mathbf{T}}^{N-1}_{1} [1] (x)
\,,
\end{align*}

\noindent for all $x \in (-1,1) $. Hence,
\begin{align}\label{f7refalem1}
      &   0 < {\mathbf{T}}^{N}_{1} [1] (0)\leqslant\left(\frac{\pi^{2}}{4} -1\right){\mathbf{T}}^{N-1}_{1} [1] (0) \, , \  N\geqslant 2\,.
\end{align}

\noindent In view of \eqref{f1refalem1}, there exists an increasing  sequence of positive integers $\{N_{k}\}_{k\geqslant 1}$ such that ${\bf{T}}_{1}^{N_{k}} [1] (0) \to 0$ as $k\to+\infty$. By keeping each even $N_{k}$ and replacing each odd $N_{k}$ by $N_{k}+1$ we obtain a new sequence
of even integers $\{2 n_{k}\}_{k\geqslant 1}$ with the property that ${\bf{T}}_{1}^{2 n_{k}} [1] (0) \to 0$ as $k\to+\infty$, as follows from \eqref{f7refalem1}. This obtains \eqref{f1arefalem1} and establishes the validity of Lemma~\ref{refalem1}. $\square$

\newpage

\section[\hspace{-0,15cm}. \hspace{-0,02cm}Proofs]{\hspace{-0,095cm}{\bf{.}} Proofs}\label{prgen}

\subsection[\hspace{-0,15cm}. \hspace{-0,04cm}Proofs for Section~\ref{inttrian}]{\hspace{-0,11cm}{\bf{.}} Proofs for Section~\ref{inttrian}}

\subsubsection[\hspace{-0,2cm}. Proofs of Lemma~\ref{inttrianlem1} and \eqref{f0apinttheor1}(b).]{\hspace{-0,11cm}{\bf{.}} Proofs of Lemma~\ref{inttrianlem1} and \eqref{f0apinttheor1}(b).}\vspace{-0,25cm}
\label{pinttrianlem1} \ The statement of Lemma~\ref{inttrianlem1} is immediate from the relationship
\begin{align*}
    &   \big| \Arg  (1\! -\! tz)\! -\!  \Arg  (1\! -\! t\!  +\!
      t z)\big| \! <\!  \pi  \, ,
\end{align*}

\noindent  proved in \cite[p.\! 608, (3.15)]{bh2} and from the identities
\begin{align*}
    &  |1\!-\!t\!+\!tz|^{2}\!-\!|1\!-\!tz|^{2} =t(2x\!-\!1)(2\!-\!t) \, , \ \
\im\,\dfrac{1\!-\!tz}{1\!-\!t\!+\!tz}\!=\!- \dfrac{y t(2\!-\!t)}{(1\!-t\!+\!tx)^{2} \!+ \!t^{2}y^{2}}\ ,
\end{align*}

\noindent where $t \in (0,1)$ and  $z=x+{\rm{i}} y\!\in \! (0,1)\cup \left(\Bb{C}\setminus\Bb{R}\right)$.

\vspace{0.25cm}
  To prove \eqref{f0apinttheor1}(b) we use \eqref{f19int}, \eqref{f2int} and \eqref{f3cint}, to obtain, for any $t > 0$,
\begin{align*}
    &  \frac{d}{d t} {\rm{e}}^{{\fo{  \pi t }}} \lambda ({\rm{i}}t) =
 {\rm{e}}^{{\fo{  \pi t }}}\left( \pi \lambda ({\rm{i}}t) + {\rm{i}} \lambda^{\,\prime}  ({\rm{i}}t)\right)=
 {\rm{e}}^{{\fo{  \pi t }}}\left(  \pi \lambda ({\rm{i}}t) + {\rm{i}} \lambda^{\,\prime}  ({\rm{i}}t)\right)  \\    &   =
 {\rm{e}}^{{\fo{  \pi t }}}
\left(\pi \lambda ({\rm{i}}t) \!-\! \pi \,\lambda ({\rm{i}}t) \,\left(1 \!-\!\lambda ({\rm{i}}t)\right)\,\Theta_{3}\left({\rm{i}}t\right)^{4}\right)=
\pi \lambda ({\rm{i}}t)\, {\rm{e}}^{{\fo{  \pi t }}}\left(1\!- \!\left(1\! -\!\lambda ({\rm{i}}t)\right)\,\Theta_{3}\left({\rm{i}}t\right)^{4}\right)  \\    &   =
\pi \lambda ({\rm{i}}t)\,  {\rm{e}}^{{\fo{  \pi t }}}\left(1- \Theta_{4}\left({\rm{i}}t\right)^{4}\right) \, , \
\end{align*}

\noindent from which
\begin{align*}
    & \frac{d}{d t} {\rm{e}}^{{\fo{  \pi t }}} \lambda ({\rm{i}}t) =\frac{d}{d t} \dfrac{16 \,
 \theta_{2}\left({\rm{e}}^{-  \pi t}\right)^{4}}{\theta_{3}\left({\rm{e}}^{-  \pi t}\right)^{4}} =
  \dfrac{16 \pi\,
 \theta_{2}\left({\rm{e}}^{-  \pi t}\right)^{4}\left(1-\theta_{4}
\left({\rm{e}}^{-  \pi t}\right)^{4}\right)}{\theta_{3}\left({\rm{e}}^{-  \pi t}\right)^{4}} > 0 \ , \quad  t> 0,
\end{align*}

\noindent because
\begin{align*}
    & \theta_{4} \left({\rm{e}}^{-  \pi t}\right) \!  =\!
1\!-
2\sum\limits_{k\geqslant 0} \left({\rm{e}}^{-  \pi (2k+1)^2 t}
\!-\! {\rm{e}}^{-  \pi (2k+2)^2 t}\right) < 1 \ , \quad  t> 0.
\end{align*}

\noindent Inequalities \eqref{f0apinttheor1}(b) is now immediate from $\lim_{0 < t \to 0} {\rm{e}}^{{\fo{  \pi t }}} \lambda ({\rm{i}}t) =1$ and from
\begin{align*}
    &  \lim_{t \to+\infty} {\rm{e}}^{{\fo{  \pi t }}} \lambda ({\rm{i}}t) =\lim_{t \to+\infty}  \dfrac{16\,
 \theta_{2}\left({\rm{e}}^{-  \pi t}\right)^{4}}{\theta_{3}\left({\rm{e}}^{-  \pi t}\right)^{4}} =16\,.
\end{align*}

\vspace{0.25cm}
\subsubsection[\hspace{-0,2cm}. Proofs of \eqref{f4int} and \eqref{f7int}.]{\hspace{-0,11cm}{\bf{.}} Proofs of \eqref{f4int} and \eqref{f7int}.}
\label{pf4aint}  \ For the integrand in the formula for $\eurm{H}_{0}$ we have
 \begin{align*}
    & {\rm{I}}_{0} := \dfrac{\he  (1/2 - {\rm{i}} t)}{  \big(t^{2}+1/4\big)  \he  (1/2 + {\rm{i}} t) \Big(
x^{2}\he  (1/2 + {\rm{i}} t)^{2}+ \he  (1/2 - {\rm{i}} t)^{2} \Big)} \\    &   =
  \dfrac{\dfrac{\he  (1/2 - {\rm{i}} t)}{ \he  (1/2 + {\rm{i}} t)}}{  \big(1/2 + {\rm{i}} t\big)\big(1/2 - {\rm{i}} t\big)  \he  (1/2 + {\rm{i}} t)^{2} \left(
x^{2}+\dfrac{\he  (1/2 - {\rm{i}} t)^{2}}{\he  (1/2 + {\rm{i}} t)^{2}}  \right)}  \,,
 \end{align*}

\noindent where by \eqref{f1intthA} and \eqref{f3int},
 written for $z= 1/2 + {\rm{i}} t \in (0,1)\cup \left(\Bb{C}\setminus\Bb{R}\right)$, $t \in \Bb{R}$,
\begin{align}\label{f1pf4aint}
    &    \frac{\het (1/2 - {\rm{i}} t)}{\het (1/2 + {\rm{i}} t)} =
    \dfrac{\ie (1/2 + {\rm{i}} t)}{\imag}
      \, , \\    &    \label{f2pf4aint}
       \frac{\het (1/2 - {\rm{i}} t)^{2}}{\het (1/2 + {\rm{i}} t)^{2}}=  \, -   \ie (1/2 + {\rm{i}} t)^{2}  \, , \\    &
     \dfrac{1}{\big(1/2 + {\rm{i}} t\big)\big(1/2 - {\rm{i}} t\big)  \he  (1/2 + {\rm{i}} t)^{2} } =  \imag  \, \pi\, \ie^{\,\prime} (1/2 + {\rm{i}} t)  \, , \
\label{f3pf4aint}\end{align}

\noindent we obtain
\begin{align*}
    &   {\rm{I}}_{0} =  \  \pi\,  \frac{\ie (1/2 + {\rm{i}} t)\ie^{\,\prime} (1/2 + {\rm{i}} t) \, \he  (1/2 + {\rm{i}} t)}{ x^{2}- \ie (1/2 + {\rm{i}} t)^{2}  } \ .
\end{align*}

\noindent Hence,
\begin{align*}
      \eurm{H}_{0} (x)& =\dfrac{1}{2 \pi^{2 }}\int\limits_{ - \infty}^{ +  \infty}
 \frac{\he  (1/2 - {\rm{i}} t) {\rm{d}} t}{  \big(t^{2}+1/4\big)  \he  (1/2 + {\rm{i}} t) \Big(
x^{2}\he  (1/2 + {\rm{i}} t)^{2}+ \he  (1/2 - {\rm{i}} t)^{2} \Big)}  \\    &   =
\dfrac{1}{2 \pi}\int\limits_{ - \infty}^{ +  \infty} \frac{\ie (1/2 + {\rm{i}} t)\ie^{\,\prime} (1/2 + {\rm{i}} t) \, \he  (1/2 + {\rm{i}} t) {\rm{d}} t }{ x^{2}- \ie (1/2 + {\rm{i}} t)^{2}  }\\    &   =
\dfrac{1}{2 \pi \imag}\int\limits_{ - \infty}^{ +  \infty} \frac{\ie (1/2 + {\rm{i}} t)\ie^{\,\prime} (1/2 + {\rm{i}} t) \, \he  (1/2 + {\rm{i}} t) {\rm{d}} (1/2 + {\rm{i}} t) }{ x^{2}- \ie (1/2 + {\rm{i}} t)^{2}  }\\    &   =
\dfrac{1}{2 \pi \imag}
\int\limits_{ 1/2 - {\rm{i}} \infty}^{1/2 + {\rm{i}} \infty}
 \frac{\ie (y)\ie^{\,\prime} (y) \, \he  (y)}{ x^{2}- \ie (y)^{2}  } {\rm{d}} y \, , \
\end{align*}

\noindent which proves \eqref{f4int}.

Similarly, by using \eqref{f1pf4aint} and \eqref{f3pf4aint}, we get
\begin{align*}
    &  \dfrac{1}{ \big(t^{2}+1/4\big) \Big(x \he  (1/2 + {\rm{i}} t) +{\rm{i}} \he  (1/2 - {\rm{i}} t)\Big)^{2}} =\\    &   =
   \dfrac{1}{ \big(1/2 + {\rm{i}} t\big)\big(1/2 - {\rm{i}} t\big)  \he  (1/2 + {\rm{i}} t)^{2}
   \left(x  +{\rm{i}} \dfrac{\he  (1/2 - {\rm{i}} t)}{\he  (1/2 + {\rm{i}} t)}\right)^{2}}\\    &   =
   \dfrac{\imag  \, \pi\, \ie^{\,\prime} (1/2 + {\rm{i}} t)}{ \left(x  +\ie (1/2 + {\rm{i}} t)\right)^{2}} \, , \
\end{align*}

\noindent and therefore
\begin{align*}
     \eurm{H}_{n} (x)  & =  \dfrac{(-1)}{4 \pi^{3} n}
  \int\limits_{ - \infty}^{ +  \infty}
\dfrac{  S^{{\tn{\triangle}}}_{n} \left(\dfrac{1}{1/2 + {\rm{i}} t}\right) {\rm{d}} t}{ \big(t^{2}+1/4\big) \Big(x \he  (1/2 + {\rm{i}} t) +{\rm{i}} \he  (1/2 - {\rm{i}} t)\Big)^{2} } \\    &   =
\dfrac{(-1)}{4 \pi^{3} n}
  \int\limits_{ - \infty}^{ +  \infty} \ \dfrac{\imag  \, \pi\, \ie^{\,\prime} (1/2 + {\rm{i}} t) S^{{\tn{\triangle}}}_{n} \left(\dfrac{1}{1/2 + {\rm{i}} t}\right) {\rm{d}} t}{\left(x  +\ie (1/2 + {\rm{i}} t)\right)^{2}} \\    &=
\dfrac{(-1)}{4 \pi^{2} n}
  \int\limits_{ - \infty}^{ +  \infty} \ \dfrac{ \ie^{\,\prime} (1/2 + {\rm{i}} t) S^{{\tn{\triangle}}}_{n} \left(\dfrac{1}{1/2 + {\rm{i}} t}\right) {\rm{d}} (1/2 + {\rm{i}} t)}{\left(x  +\ie (1/2 + {\rm{i}} t)\right)^{2}} \\    &   =
  \dfrac{(-1)}{4 \pi^{2} n} \int\limits_{ 1/2 - {\rm{i}} \infty}^{1/2 + {\rm{i}} \infty} \
\frac{S^{{\tn{\triangle}}}_{n}\! \left({1}\big/{y}\right)\ie^{\,\prime} (y) {\rm{d}} y}{(x+\ie (y))^{2} } \, , \
\end{align*}

\noindent which completes the proof of  \eqref{f7int}.

\subsubsection[\hspace{-0,2cm}. Proofs of Theorem~\ref{inttheor1} and Corollary~\ref{intcorol1}.]{\hspace{-0,11cm}{\bf{.}} Proofs of Theorem~\ref{inttheor1} and Corollary~\ref{intcorol1}.}
\label{pinttheor1} $\phantom{a}$

\vspace{0.25cm} \ We first prove \eqref{f9cint}. By using \eqref{f19int}, for arbitrary $x + {\rm{i}} y\in \Bb{H} $  we get
\begin{align*}
    &   \dfrac{{\rm{d}}}{{\rm{d}} x} \left|\lambda ({\rm{i}} y + x)\right|^{2} =  \dfrac{{\rm{d}}}{{\rm{d}} x}\lambda ({\rm{i}} y + x) \lambda ({\rm{i}} y - x)  =
    \lambda^{\,\prime} ({\rm{i}} y + x) \lambda ({\rm{i}} y - x)  \\    &    - \lambda ({\rm{i}} y + x) \lambda^{\,\prime} ({\rm{i}} y - x) =
    {\rm{i}} \pi \left|\lambda ({\rm{i}} y + x)\right|^{2}
      \big(1-\lambda ({\rm{i}} y + x)\big)  \Theta_{3}\left({\rm{i}} y + x\right)^{4} \\    &   -
      {\rm{i}} \pi \left|\lambda ({\rm{i}} y + x)\right|^{2}\big(1-\lambda ({\rm{i}} y - x)\big)  \Theta_{3}\left({\rm{i}} y - x\right)^{4} \, , \
\end{align*}

\noindent where, in accordance with \eqref{f2int} and \eqref{f3bint}(g),
\begin{align*}
    & \left(1\!- \!\lambda (z)\right)\Theta_{3}(z)^{4}=\Theta_{4}(z)^{4} = \Theta_{3}(z-1)^{4}
      \, , \ \   z \!\in\!  \Bb{H} \,,
\end{align*}

\noindent and therefore
\begin{align*}
    & 2\dfrac{{\rm{d}}}{{\rm{d}} x}\log \left|\lambda ({\rm{i}} y\! + \!x)\right|\!=\!  {\rm{i}} \pi \left(\Theta_{4}({\rm{i}} y + x)^{4}- \Theta_{4}({\rm{i}} y - x)^{4}\right) =    - 2 \pi \im \Theta_{3}({\rm{i}} y + x-1)^{4} \, , \
\end{align*}

\noindent which proves the second equality in \eqref{f9cint}. Similarly,
\begin{align*}
    &   \dfrac{{\rm{d}}}{{\rm{d}} x} \left|\lambda ({\rm{i}} y + x)\right|^{2}  \left|1-\lambda ({\rm{i}} y + x)\right|^{2}  = \\ & =
 \dfrac{{\rm{d}}}{{\rm{d}} x}\lambda ({\rm{i}} y + x)\left(1-\lambda ({\rm{i}} y + x)\right)
\lambda ({\rm{i}} y - x)\left(1-\lambda ({\rm{i}} y - x)\right) = \\ & =
{\rm{i}} \pi \left|\lambda_{2} ({\rm{i}} y + x)\right|^{2}
\left[\left(1-2\lambda ({\rm{i}} y + x)\right)\Theta_{3}({\rm{i}} y + x)^{4}-\left(1-2\lambda ({\rm{i}} y - x)\right)\Theta_{3}({\rm{i}} y - x)^{4}\right],
\end{align*}

\noindent where, in accordance with \eqref{f2int} and \eqref{f19int},
\begin{align*}
    & (1-2\lambda(z))\Theta_{3}^{4}(z)=
\Theta_{3}^{2}(z)-2\Theta_{2}^{2}(z)
= \Theta_{4}^{4}(z)-\Theta_{2}^{4}(z)=2\Theta_{4}^{4}(z)-\Theta_{3}^{4}(z)
\end{align*}

\noindent for arbitrary $ z \!\in\!  \Bb{H}$, and hence, by \eqref{f3bint}(g), we obtain
\begin{align*}
      2\dfrac{{\rm{d}}}{{\rm{d}} x}\log \left|\lambda_{2} ({\rm{i}} y\! + \!x)\right| &\!=\!
    -2\pi \im \left(
2\Theta_{4}^{4}({\rm{i}} y + x)-\Theta_{3}^{4}({\rm{i}} y + x)\right) \\    &   \!=\!2\pi \im \left(
\Theta_{3}^{4}({\rm{i}} y + x)-2\Theta_{3}^{4}({\rm{i}} y + x-1)\right) \, , \
\end{align*}

\noindent which completes the proof of the first equality in \eqref{f9cint}.

\vspace{0,1cm}
Prove now \eqref{f1inttheor1} and \eqref{f4inttheor1}.
For $x \in (0,1)$ it follows from  $x\!+\!{\rm{i}} y\in \fet $ and \eqref{f9bint} that
$ \im \Theta_{3}^{4}({\rm{i}} y + x)> 0$ and $ \im \Theta_{3}^{4}({\rm{i}} y + x-1)< 0$ because $x-1 \in (-1,0)$.
These two inequalities together with \eqref{f9cint} yield the validity of \eqref{f1inttheor1} and \eqref{f4inttheor1} for $x > 0$.
If  $x \in (-1,0)$ in \eqref{f1inttheor1} and in \eqref{f4inttheor1},  we similarly get their validity by using inequalities
$ \im \Theta_{3}^{4}({\rm{i}} y + x)< 0$ and $ \im \Theta_{3}^{4}({\rm{i}} y + x-1) =  \im \Theta_{3}^{4}({\rm{i}} y + x+1) > 0$
which follow from  \eqref{f9bint}, in view of $x + 1 \in (0,1) $. The proof of \eqref{f1inttheor1} and \eqref{f4inttheor1} is completed.

To prove \eqref{f3inttheor1} and \eqref{f6inttheor1}, observe that
 by using  the relationship
\begin{align}\label{f2pinttheor1}
    &  \lambda (z +1) = \lambda (z -1) =\frac{\lambda (z)}{\lambda (z) -1}  \ , \quad  z \in \Bb{H},
\end{align}

\noindent  (see \cite[p.\! 111]{cha}), we obtain from  \eqref{f1inttheor1},  \eqref{f4inttheor1} and \eqref{f6apinttheor1} that
\begin{align}\label{f1pinttheor1}
    &
    \begin{array}{l}
\begin{displaystyle}
     \max_{x \in [0,1]} \left|\lambda ({\rm{i}} y + x)\right| \ =\,\left|\lambda ({\rm{i}} y \pm 1)\right|\ =
\frac{\lambda ({\rm{i}} y)}{1-\lambda ({\rm{i}} y) }  \ ,
\end{displaystyle}
  \\[0,5cm]
  \begin{displaystyle}
\max_{x \in [0,1]} \left|\lambda_{2} ({\rm{i}} y + x)\right|  =\left|\lambda_{2} ({\rm{i}} y \pm 1)\right|= \dfrac{\lambda ({\rm{i}} y )}{\big(1-\lambda ({\rm{i}} y )\big)^{2}}  \ , \quad  y \geqslant  1/2\,,
\end{displaystyle}
      \end{array}
     \end{align}

\noindent
where, by virtue of \eqref{f6apinttheor1},  the both functions of $y$ in the right-hand sides of \eqref{f1pinttheor1}  strictly decrease on the interval $(0, +\infty)$ from $+\infty $ to $0$.
This proves \eqref{f3inttheor1} and \eqref{f6inttheor1}.

The property \eqref{f2inttheor1} is immediate from \eqref{f16inttrian}. Furthermore, \eqref{f1intcorol1}(a) is a simple consequence of  \eqref{f2inttheor1} and \eqref{f3inttheor1}.

Prove now \eqref{f1intcorol1}(c).
For arbitrary $z \in \eusm{E}^{0}_{\!{\tn{\frown}}} $ and $x \in [-1,1],$ $ y> 1$, it follows from \eqref{f1intcorol1}(a)  and \eqref{f3contgen}(b) that
 $\lambda (x\! + \!{\rm{i}} y)\in \Bb{C}_{\re \leqslant \lambda ({\rm{i}} y)}$,
$\lambda(z)\! \in \! \Bb{C}_{\re \geqslant 1/2}$
and
 \begin{align}\label{f7pinttheor1}
    &  \left|\lambda(z)-\lambda(x+{\rm{i}}y)\right|\geqslant \left|\dfrac{1}{2}- \lambda({\rm{i}}y)\right|\geqslant
    \left|\dfrac{1}{2}- \lambda({\rm{i}}y)\right|\dfrac{\left|\lambda(z)\right|}{2}  \, , \
 \end{align}

\noindent provided that $\left|\lambda(z)\right| \leqslant 2$. But if $\left|\lambda(z)\right| > 2$
then by \eqref{f3inttheor1} and $\lambda({\rm{i}}y) \in (0, 1/2)$ we get $|\lambda(x+{\rm{i}}y)| \leqslant \lambda({\rm{i}}y)/(1-\lambda({\rm{i}}y)) < 1$ and therefore
 \begin{align*}
    &  \left|\lambda(z)-\lambda(x+{\rm{i}}y)\right|\geqslant \left|\lambda(z)\right| - \left|\lambda(x+{\rm{i}}y)\right| \geqslant
    \dfrac{\left|\lambda(z)\right|}{2} \geqslant \dfrac{\left|\lambda(z)\right|}{2}  \left|\dfrac{1}{2}- \lambda({\rm{i}}y)\right|  \, ,
 \end{align*}

 \noindent which together with \eqref{f7pinttheor1} completes the proof of \eqref{f1intcorol1}(c). Next, by
 \eqref{f19int}, \eqref{f6inttheor1} and \cite[p.\! 325, (i)]{ber1}, for $x \in [-1,1]$ and $y> 1$ we obtain
 \begin{align*}
    & \left|\lambda^{\,\prime}\left(x + {\rm{i}} y \right)\right| = \pi \left|\lambda_{2}(x + {\rm{i}} y)\right| \left|\theta_{3}\left({\rm{e}}^{\imag     \pi (x + {\rm{i}} y)}\right)\right|^{4}\leqslant
\pi   \dfrac{\lambda ({\rm{i}} y )}{\big(1-\lambda ({\rm{i}} y )\big)^{2}} \left|\theta_{3}\left({\rm{e}}^{-   \pi  y}\right)\right|^{4}\leqslant \\    &
\leqslant  2 \pi \left|\theta_{3}\left({\rm{e}}^{-   \pi  }\right)\right|^{4} =
2 \pi \dfrac{ \pi}{\Gamma(3/4)^{4}}< 8,7537585 < 9\,,
 \end{align*}

\noindent
which proves \eqref{f1intcorol1}(b).  Finally, to prove  \eqref{f1intcorol1}(d) observe that for $ 1/2 + \eta := \lambda(z)$ we have $ u + {\rm{i}} v:= \eta \in \Bb{C}_{\re > 0}$, in view of  \eqref{f3contgen}(b) , and for $a := -\lambda(1+ {\rm{i}} t)$
we have $a > 0$, by virtue of Lemma~\ref{lemoneto}. Then
\begin{align*}
      2\left|\lambda(z)-\lambda(1+{\rm{i}}t)\right|^{2}& = 2 \left| 1/2  + u + a+ {\rm{i}} v\right|^{2} =
2 \left(1/2  + u + a\right)^{2} + 2v^{2} \, , \\
\left(\left|\lambda(z)\right|+\left|\lambda(1+{\rm{i}}t)\right|\right)^{2} &= (u+ 1/2)^{2} + v^{2}+ a^{2} + 2 a
 \sqrt{(u+ 1/2)^{2} + v^{2}}\\  & \leqslant (u+ 1/2)^{2} + v^{2}+ a^{2} +
(u+ 1/2)^{2} + v^{2}+ a^{2} \\    &   =
2(u+ 1/2)^{2} + 2v^{2}+ 2a^{2} = 2 \left(1/2  + u + a\right)^{2} + 2v^{2} \\    &   -
4 a (u+ 1/2) < 2\left|\lambda(z)-\lambda(1+{\rm{i}}t)\right|^{2} \, ,
\end{align*}

\noindent which proves  \eqref{f1intcorol1}(d) and
completes the proof of Theorem~\ref{inttheor1} and Corollary~\ref{intcorol1}.

\vspace{0.25cm}
\subsubsection[\hspace{-0,2cm}. Proofs of Theorem~\ref{inttheor2} and Corollary~\ref{intcorol2}.]{\hspace{-0,11cm}{\bf{.}} Proofs of Theorem~\ref{inttheor2} and Corollary~\ref{intcorol2}.}
\label{pinttheor2} $\phantom{a}$

\vspace{0.25cm} \  For $n \in \Bb{Z}_{\neq 0}$   and $z=a+{\rm{i}} b$  with $a\in [-1,1]$, $b>0$ we have
\begin{align*}
     & \im  \dfrac{1}{2n -z } = \dfrac{b}{(2n -a)^{2} +  b^{2}}\leqslant
\dfrac{b}{1 +  b^{2}}<b \ . \
\end{align*}

\noindent Besides that,
\begin{align*}
    & \im  \dfrac{1}{2n + {1}/{z} } =  \dfrac{b}{(2n a + 1)^{2} + 4 n^{2}b^{2}} =\frac{b}{1 + 4n^{2} \left( a^{2} + \dfrac{a}{n} +b^{2}\right)}\ , \end{align*}

\noindent which is strictly less than $b$ if $a=0$, but if $a\!\neq\! 0$ then $\sigma_{a}\!:=\! {\rm{sign}} (a)\! \in \!\{1,-1\}$ and
\begin{align*}
     &
a^{2} + \dfrac{a}{n} +b^{2}=
|a|\left(1+\dfrac{\sigma_{a}}{ n} \right) +\left|z-\frac{\sigma_{a}}{2}\right|^{2}-\frac{1}{4} \geqslant 0 \ , \
\end{align*}

\noindent provided that $ z \!\in\! {\rm{clos}}\left(\mathcal{F}_{{\tn{\square}}}\right)$ and the equality here is strict if $|n|\geqslant 2$ or $n \cdot\re \, z > 0$, or $ z \!\in\! \mathcal{F}_{{\tn{\square}}}$. Thus, the following properties hold.
\begin{align}\label{f1pinttheor2}
\left\{\begin{array}{l}
   z \in \Bb{H}_{{\fo{\left|\re\right|\leqslant 1}}} \\[0,3cm]
 n \in \Bb{Z}_{\neq 0}
\end{array}\right.
&  \ \Rightarrow \ &  &
      \left\{
        \begin{array}{l}
        \im    {\fo{\dfrac{1}{2n -z }}} <  \im \, z \ ,    \\[0,3cm]
          {\fo{\dfrac{1}{2n -z }}}\in \Bb{D}\cap\Bb{H} \subset \Bb{H}_{{\fo{\left|\re\right|\leqslant 1}}} \ ,
        \end{array}
      \right.
      \\[0,3cm]
\left\{\begin{array}{l}
  z \!\in\! \Bb{H}\cap {\rm{clos}}\left(\mathcal{F}_{{\tn{\square}}}\right)
  \\[0,3cm]
 n \in \Bb{Z}_{\neq 0}
\end{array}\right.
& \ \Rightarrow \ &  &
 \left\{
        \begin{array}{l}
       \im  {\fo{\dfrac{1}{2n + (1/z) }}} \leqslant  \im \, z \ ,   \\[0,3cm]
          {\fo{\dfrac{1}{2n + (1/z) }}}\in \Bb{D}\cap\Bb{H}  \subset \Bb{H}_{{\fo{\left|\re\right|\leqslant 1}}} \ , \\[0,2cm]
 \im  {\fo{\dfrac{1}{2n + (1/z) }}} <  \im \, z \ , \ \mbox{if} \
\left\{\begin{array}{l}
|n|\!\geqslant\! 2\,, \ \mbox{or} \\[0,1cm]
 n\cdot \re \, z\! >\! 0\,,
  \\[0,1cm]
\ \mbox{or} \ z\in  \fet\,.
\end{array}\right.
        \end{array}
      \right.
\label{f2pinttheor2}\end{align}

\noindent According to the definition \eqref{f1bsden}, for $N\geqslant 1$ and
$\eufm{n}=(n_{N}, n_{N-1},\ldots, n_{1})\in \Bb{Z}_{\neq 0}^{N}$ we have
\begin{align}\label{f3pinttheor2}
     & \phi_{\eufm{n}}(y)= \phi_{n_{N}, n_{N-1},\ldots, n_{1}} (y) = \phi_{n_{N}} \Big(\phi_{n_{N-1}} \big(\ldots\big( \phi_{n_{1}} (y) \big)\ldots\big)\Big) , \quad  y\in \Bb{H},
\end{align}

\noindent where
\begin{align}\label{f4pinttheor2}
     & \ \phi_{n} (z)=\frac{1}{2n-z} \ , \quad
\phi_{n} (-1/z)=\frac{1}{2n+1/z} \ ,\quad n \in \Bb{Z}_{\neq 0} \, , \ z\in \Bb{H}  .
\end{align}

\noindent Applying successively \eqref{f1pinttheor2} and \eqref{f2pinttheor2} to \eqref{f3pinttheor2} we obtain
\begin{align}\label{f7pinttheor2}
     &\vspace{-0,4cm} \begin{array}{lrll}
 {\rm{(a)}}  &      \im\, \phi_{{\fo{\eufm{n}}}}(y)\! < \! \im\, y , &
 \ \ \eufm{n}\in \cup_{N\geqslant 1} \Bb{Z}_{\neq 0}^{N}  ,   &    \ \  y \in \Bb{H}_{{\fo{\left|\re\right|\leqslant 1}}} \ ,
 \\[0,1cm]{\rm{(b)}}  &
\im\, \phi_{{\fo{\eufm{n}}}}(-1/z)\! < \! \im\, z  , & \ \  \eufm{n}\in \cup_{N\geqslant 2} \Bb{Z}_{\neq 0}^{N}  ,  &   \ \ z \in \Bb{H}\cap {\rm{clos}}\left(\mathcal{F}_{{\tn{\square}}}\right) ,
\\[0,1cm]{\rm{(c)}}  &
\im\, \phi_{{\fo{\eufm{n}}}}(-1/z) \!< \! \im\, z  , & \ \  \eufm{n}\in \Bb{Z}\setminus\{-1,0,1\}  ,  &   \ \  z \in \Bb{H}\cap {\rm{clos}}\left(\mathcal{F}_{{\tn{\square}}}\right) ,
\\[0,1cm]{\rm{(d)}}  &
\im\, \phi_{{\fo{\eufm{n}}}}(-1/z) \!< \! \im\, z  , & \ \  \eufm{n}\!= \! {\rm{sign}} (\re\, z)\!\in\!\{-1,1\}  ,  &   \ \  z \in \Bb{H}\cap {\rm{clos}}\left(\mathcal{F}_{{\tn{\square}}}\right) ,
\\[0,1cm]{\rm{(e)}}  & \im\, \phi_{{\fo{\eufm{n}}}}(-1/\zeta) \!<\!  \im\, \zeta  , & \ \  \eufm{n}\in \cup_{N\geqslant 1} \Bb{Z}_{\neq 0}^{N} ,  &   \ \ \zeta \in \fet\,,
       \end{array}
\vspace{-0,1cm}\end{align}

\noindent where \eqref{f7pinttheor2}(b) with $\eufm{n}=(n_{N}, n_{N-1},\ldots, n_{1})$ follows from \eqref{f7pinttheor2}(a) with  $\eufm{n}=(n_{N}, n_{N-1},\ldots, n_{2})$, applied to $y =  \phi_{n_{1}}(-1/z)$
which by \eqref{f4pinttheor2} and \eqref{f2pinttheor2}  satisfies
$y \in \Bb{H}_{{\fo{\left|\re\right|\leqslant 1}}}$ and
$\im\,y \leqslant \im\,z$.
Introduce the notation
\begin{align}\label{f8pinttheor2}&
\begin{array}{rcccl}
  \lambda_{{\tn{\square}}}^{(-1)}(a) &   := & \{\,z \in \Bb{H}\cap {\rm{clos}}\left(\mathcal{F}_{{\tn{\square}}}\right) \,  & \mid  & \, \lambda (z)=a\}, \\
  \lambda_{\,\sqcup}^{(-1)}(a)&   := &\{\,z \in \Bb{H}_{|\re|\leqslant 1}\setminus {\rm{clos}}\left(\mathcal{F}_{{\tn{\square}}}\right) \,
& \mid  & \, \lambda (z)=a\},
\end{array}\qquad a \in \Bb{C}\setminus\{0,1\} \,.
\end{align}

\noindent In view of Lemma~\ref{lemoneto}, Theorem~\ref{intthA} and \eqref{f1intlemma1},
\begin{align}\label{f9pinttheor2}
     &
\begin{array}{ll}
  \lambda_{{\tn{\square}}}^{(-1)}(a)=\left\{\ie (a)\right\} \subset \fet\,,  & \quad a\in (0,1)\cup \left(\Bb{C}\setminus\Bb{R}\right) , \\
\lambda_{{\tn{\square}}}^{(-1)}(a)=\left\{\ie (a+ {\rm{i}} 0), \ie (a- {\rm{i}} 0)\right\}
\subset \partial\fet\,, & \quad a\in \Bb{R}_{<0}\cup \Bb{R}_{>1} ,
\end{array}
\end{align}

\noindent where by \eqref{f6intthA}, \eqref{f2intlemma1} and  \eqref{f4intthA}, \eqref{f5intthA},
\begin{align}\label{f10pinttheor2}
     & \hspace{-0,3cm}
\begin{array}{lll}
\ie (a\! \pm \!{\rm{i}}  0) \!\in\! (\pm 1) \!+\! {\rm{i}} \Bb{R}_{>0} \, , & \ \ \
 \ie (a\! -\! {\rm{i}}  0)\! = \!\ie (a \!+ \!{\rm{i}}  0)\! -\!2 \,,  & \  \ \
a\!\in \!\Bb{R}_{<0} \, ,  \\
\ie (a\! \pm\! {\rm{i}}  0)\! \in\! \gamma(0, \pm 1)\ ,  & \  \ \
\ie (a\! -\! {\rm{i}}  0) =\dfrac{1}{-2 \!+\! 1/ \ie (a\! +\! {\rm{i}}  0)} \ ,
 & \  \ \  a\!\in\! \Bb{R}_{> 1} \,.
\end{array}
\end{align}

\vspace{-0,2cm}
\noindent Furthermore, by Lemma~\ref{intlemma1},
\begin{align}\label{f10zpinttheor2}
     & \im\, \ie (a)  = \im\,\ie (a\! +\! {\rm{i}}  0)= \im\,\ie (a\! -\! {\rm{i}}  0) \ , \ a\!\in\! \Bb{R}_{<0}\cup \Bb{R}_{> 1}\,.
\end{align}

\noindent It has been established in Lemma~\ref{lemdhp} that for each $a \in \Bb{C}\setminus\{0,1\} = \lambda (\Bb{H} \cap{\rm{clos}} (\mathcal{F}_{{\tn{\square}}}))$ the set $\lambda_{\,\sqcup}^{(-1)}(a)$ is countable and cannot have the limit points in $\Bb{H}$. That's why to prove the statement of Theorem~\ref{inttheor2} it suffices to show that the imaginary part of each number from   $\lambda_{\,\sqcup}^{(-1)}(a)$ is strictly less
than $\im\, \ie (a)$, i.e.,
\begin{align}\label{f5pinttheor2}
    & \im\, z < \im\, \ie (a) \ , \quad  z\in \lambda_{\,\sqcup}^{(-1)}(a) \ , \quad  a \in \Bb{C}\setminus\{0,1\}\,,
\end{align}

\noindent which obviously coincides with the property \eqref{f1inttheor2}.

Assume that the number $a$ in Theorem~\ref{inttheor2} belongs to $(0,1)\cup \left(\Bb{C}\setminus\Bb{R}\right)$.
In  the notation $y:= \ie (a)\in \fet$, it follows from  \eqref{f9pinttheor2} that $\lambda_{{\tn{\square}}}^{(-1)}(a)=\left\{y\right\}$ and therefore \eqref{f1lemdhp} yields
 \begin{align}\label{f6pinttheor2}  &
   \lambda_{\sqcup}^{(-1)}(a) =
    \left\{\phi_{{\fo{\eufm{n}}}}(y)\ \left|  \ \eufm{n}\in  \Bb{Z}_{0}^{2\Bb{N}}  \right\}\right.\cup
\left\{\phi_{{\fo{\eufm{n}}}}(-1/y)\ \left| \ \eufm{n}\in  \Bb{Z}_{0}^{2\Bb{N}-1}  \right\}\right. ,
\end{align}

\noindent which satisfies \eqref{f5pinttheor2}, in view of \eqref{f7pinttheor2}(a) and \eqref{f7pinttheor2}(e).
This proves  Theorem~\ref{inttheor2} for  $a\in (0,1)\cup \left(\Bb{C}\setminus\Bb{R}\right)$.

Let $a$  in Theorem~\ref{inttheor2} belong to $\Bb{R}_{<0}$ and $y := \ie (a + {\rm{i}}  0) \in \gamma (1, \infty)$.
By virtue of \eqref{f9pinttheor2} and \eqref{f10pinttheor2}, $\lambda_{{\tn{\square}}}^{(-1)}(a)= \{y, y-2\}$
and we deduce from \eqref{f1lemdhp} that
\begin{align}\label{f5apinttheor2}
    &  \lambda_{\sqcup}^{(-1)}(a) =\left\{\phi_{{\fo{\eufm{n}}}}\left(y-1  + \sigma_{\eufm{n}}\right)\ \left| \ \eufm{n}\in  \Bb{Z}_{0}^{2\Bb{N}}  \right\}\right.,
\end{align}

\noindent which satisfies \eqref{f5pinttheor2}, in view of \eqref{f7pinttheor2}(a), $\im\, (y-1  + \sigma_{\eufm{n}})=\im\, y$ and  $y-1 + \sigma_{\eufm{n}}\in \sigma_{\eufm{n}}+  {\rm{i}} \Bb{R}_{>0} \subset
\Bb{H} \cap{\rm{clos}} (\mathcal{F}_{{\tn{\square}}})\subset \Bb{H}_{{\fo{\left|\re\right|\leqslant 1}}}$. This proves  Theorem~\ref{inttheor2} for $a\!\in\!\Bb{R}_{<0}$.

Let finally  $a$  in Theorem~\ref{inttheor2} belong to $\Bb{R}_{> 1}$ and $y := \ie (a + {\rm{i}}  0) \in \gamma (1, 0)$.
Then $-1/y \in \gamma (-1, \infty)$,
\begin{align}\label{f5zapinttheor2}
    & \sigma\!+\!1\!-\!{\fo{(}}1/y{\fo{)}}\in \sigma+  {\rm{i}} \Bb{R}_{>0} \subset
\Bb{H} \cap{\rm{clos}} (\mathcal{F}_{{\tn{\square}}}) \ , \quad  \sigma \in \{1,-1\} \, , \
\end{align}

\noindent  and   \eqref{f9pinttheor2},  \eqref{f10pinttheor2} and \eqref{f10zpinttheor2} imply that $\lambda_{{\tn{\square}}}^{(-1)}(a)=\{y, 1/(-2+1/y)\}$ and
\begin{align}\label{f5wapinttheor2} & \dfrac{1}{-2 + 1/y}\in \gamma (-1, 0) \ ,  &  &  \im \dfrac{1}{-2 + 1/y}= \im\, y \ , \  \\
    &  y= \phi_{n_{1}}(\sigma_{n_{1}} + 1-1/y)\Big|_{n_{1}=1}   \ ,  &  &
    \dfrac{1}{-2 + 1/y} = \phi_{n_{1}}(\sigma_{n_{1}} + 1-1/y)\Big|_{n_{1}=-1}
\nonumber \end{align}
\noindent (see \eqref{f9contgen}). In accordance with \eqref{f1lemdhp}, we get
\begin{align}\nonumber
    \hspace{-0,1cm}  \lambda_{\sqcup}^{(-1)}(a)\! &    =\!\left\{\phi_{{\fo{\eufm{n}}}}\!\left(  \sigma_{\eufm{n}}\!+\!1\!-\!{\fo{(}}1/y{\fo{)}}
\hspace{0,05cm}\right) \ \left| \ \eufm{n}\!\in\!  \Bb{Z}_{0}^{2\Bb{N}+1}  \right\}\right.\cup\\    &\,\cup
\big\{\phi_{n_{1}}\!\!\left(  \sigma_{n_{1}}\!+\!1\!-\!{\fo{(}}1/y{\fo{)}}
\hspace{0,05cm}\right)\big\}_{n_{1}\!\in \,\Bb{Z}\setminus\{-1, 0, 1\} },
\label{f5tapinttheor2} \end{align}

\noindent because $y, 1/(-2+1/y)\in \Bb{H} \cap{\rm{clos}} (\mathcal{F}_{{\tn{\square}}})$ and, in view of
\eqref{f4pinttheor2}, \eqref{f5wapinttheor2},
\eqref{f5zapinttheor2} and \eqref{f1pinttheor2},
\begin{align}\nonumber
     & \left\{\phi_{n_{1}}\!\!\left(  \sigma_{n_{1}}\!+\!1\!-\!{\fo{(}}1/y{\fo{)}}
\hspace{0,05cm}\right)\right\}_{n_{1}\in \Bb{Z}\setminus\{ 0\} } \\    &   =\!
\left\{y\right\}\!\cup \!\left\{\dfrac{1}{-2\! +\! 1/y}\right\}
\!\cup\!
\left\{\phi_{n_{1}}\!\!\left(  \sigma_{n_{1}}\!+\!1\!-\!{\fo{(}}1/y{\fo{)}}
\hspace{0,05cm}\right)\right\}_{n_{1}\in \Bb{Z}\setminus\{-1, 0, 1\} }\subset \Bb{H}_{{\fo{\left|\re\right|\leqslant 1}}} \,.
\label{f5yapinttheor2}\end{align}

\noindent Here $\phi_{n_{1}}(\sigma_{n_{1}}\! +\! 1\!-\!(1/y))=\phi_{\widetilde{n}_{1}}(-\!1/y)$,
$\widetilde{n}_{1}= n_{1}-(\sigma_{n_{1}}\! +\! 1)/2\in \Bb{Z}\setminus\{0,-1\} $ for every $n_{1}\in \Bb{Z}\setminus\{-1, 0, 1\}$, where $\widetilde{n}_{1}=1$
if and only if $n_{1}=1= {\rm{sign}}(\re y)$. Applying \eqref{f7pinttheor2}(c) and \eqref{f7pinttheor2}(d) to
 $\phi_{\widetilde{n}_{1}}(-\!1/y)$ with $y \in \gamma (1, 0) \subset \Bb{H}\cap {\rm{clos}}\left(\mathcal{F}_{{\tn{\square}}}\right)$,  we obtain
\begin{align}\label{f5xapinttheor2}
     & \im\,\phi_{n_{1}}(\sigma_{n_{1}}\! +\! 1\!-\!(1/y))< \im\,y \ , \quad n_{1}\in \Bb{Z}\setminus\{-1, 0, 1\} , \quad y\in \gamma (1, 0).
\end{align}

\noindent At the same time, if  $N\geqslant 1$ and
$\eufm{n}=(n_{2N+1}, n_{2N},\ldots, n_{2}, n_{1})\in \left(\Bb{Z}_{\neq 0}\right)^{2N+1}$  then \eqref{f3pinttheor2} means that in \eqref{f5tapinttheor2} we have
\begin{align*}
     \phi_{\eufm{n}}\!\left(  \sigma_{\eufm{n}}\!+\!1\!-\!{\fo{(}}1/y{\fo{)}}
\hspace{0,05cm}\right)&= \phi_{n_{2N+1}, n_{2N},\ldots, n_{2}, n_{1}} \!\left(  \sigma_{n_{1}}\!+\!1\!-\!{\fo{(}}1/y{\fo{)}}
\hspace{0,05cm}\right)  \\    &   =
    \phi_{n_{2N+1}, n_{2N},\ldots, n_{2}} \big(\phi_{n_{1}}\!\!\left(  \sigma_{n_{1}}\!+\!1\!-\!{\fo{(}}1/y{\fo{)}}
\hspace{0,05cm}\right) \big)  \, , \
\end{align*}

\noindent where, by \eqref{f5yapinttheor2}, $\phi_{n_{1}}\!\!\left(  \sigma_{n_{1}}\!+\!1\!-\!{\fo{(}}1/y{\fo{)}}
\hspace{0,05cm}\right)\in \Bb{H}_{{\fo{\left|\re\right|\leqslant 1}}} $.  Hence, in accordance with \eqref{f7pinttheor2}(a), \eqref{f1pinttheor2}  and \eqref{f5yapinttheor2}, \eqref{f5wapinttheor2}, \eqref{f5xapinttheor2},
\begin{align*}
    &  \im\, \phi_{\eufm{n}}\!\left(  \sigma_{\eufm{n}}\!+\!1\!-\!{\fo{(}}1/y{\fo{)}}
\hspace{0,05cm}\right) <  \im\,\phi_{n_{1}}\!\!\left(  \sigma_{n_{1}}\!+\!1\!-\!{\fo{(}}1/y{\fo{)}}
\hspace{0,05cm}\right) \leqslant \im\, y\,.
\end{align*}

\noindent This inequality together with \eqref{f5xapinttheor2} and \eqref{f5tapinttheor2} proves
\eqref{f5pinttheor2} for each $a \in \Bb{R}_{> 1}$. Thus,  Theorem~\ref{inttheor2} holds in
this case and its proof is completed.

If we assume the contrary in Corollary~\ref{intcorol2} we obtain the contradiction with  \eqref{f1inttheor2} and therefore Corollary~\ref{intcorol2} holds as well.

\vspace{0,15cm}

\subsection[\hspace{-0,15cm}. \hspace{-0,04cm}Proofs for Section~\ref{intprel}]{\hspace{-0,11cm}{\bf{.}} Proofs for Section~\ref{intprel}}

\subsubsection[\hspace{-0,2cm}. Proofs of Lemmas~\ref{intlem1} and~\ref{intlem3}.]{\hspace{-0,11cm}{\bf{.}} Proofs of Lemmas~\ref{intlem1} and~\ref{intlem3}.}
\label{pintlem1} We first prove Lemma~\ref{intlem1}.
Let $x \in \Bb{C}\setminus\{0, 1\}$ and $P_n (x) = \sum_{k=0}^{n} p_{k} x^{k}$ satisfy the conditions of Lemma~\ref{intlem1}. Then
\begin{align*}
    &   P_{n} \left(\frac{1}{x}\right) \pm  P_{n} \left(\frac{1}{1-x}\right) =
    \sum_{k=0}^{n} p_{k}\left( \dfrac{1}{x^{k}} \pm  \dfrac{1}{(1-x)^{k}}\right) =
     x^{-n} (1-x)^{-n} R_{n}^{\pm} (x) \,,
      \\    &
        R_{n}^{\pm} (x) :=   \sum\nolimits_{k=0}^{n}\  p_{k}\left( (1-x)^{n}x^{n-k} \pm  x^{n}(1-x)^{n-k}\right)  \, , \
\end{align*}

\noindent where $R_{n}^{\pm} (1 - x) = \pm R_{n}^{\pm} ( x)$, and therefore, there exist the real numbers
$q_{k}$, $0 \leqslant k \leqslant n$, such that
\begin{align*}
    & R_{n}^{+} (1/2 + x) =  \sum_{0 \leqslant k \leqslant n} \!\!\!
    q_{2 k} \,x^{2 k}  \, , \
    R_{n}^{-} (1/2 + x) =  \sum_{0 \leqslant k \leqslant n-1 }\!\!\!
    q_{2 k+1}  \,x^{2 k+1}  \, , \
\end{align*}

\vspace{-0,1cm}
\noindent where $k \in \Bb{Z}_{\geqslant 0}$. Then there exist  the real numbers
$r_{k}^{+}$, $0 \leqslant k \leqslant n$, satisfying
\begin{align*}
       R_{n}^{+} ( x)& =  \sum_{0 \leqslant k \leqslant n} \!\!\!
    q_{2 k} \,(x-1/2)^{2 k}=  \sum_{0 \leqslant k \leqslant n} \!\!\!
   2^{-2k} q_{2 k} \,(2x-1)^{2 k} \\    &   =  \sum_{0 \leqslant k \leqslant n} \!\!\!
   2^{-2k} q_{2 k} \,(4x^{2}-4 x+1)^{ k} =  \sum_{0 \leqslant k \leqslant n} \!\!\!
   2^{-2k} q_{2 k} \,\big(1 - 4 x (1-x)\big)^{ k}\\    &   =
   \sum_{0 \leqslant k \leqslant n} r^{+}_{n-k} \big(x (1-x)\big)^{k} \, , \
\end{align*}

\vspace{-0,1cm}
\noindent and $r_{k}^{-}$, $1 \leqslant k \leqslant n$, such that
\begin{align*}
     &  R_{n}^{-} ( x)\! =\! (x\!-\!1/2)\!\!\! \sum_{0 \leqslant k \leqslant n-1} \!\!\!
    q_{2 k+1} \,(x\!-\!1/2)^{2 k}\!=\! 2^{-1} (2x\!-\!1) \!\!\!\sum_{0 \leqslant k \leqslant n-1} \!\!\!
   2^{-2k} q_{2 k+1} \,(2x\!-\!1)^{2 k} \\    &  = \! 2^{-1} (2x\!-\!1) \!\!\!
   \sum_{0 \leqslant k \leqslant n-1} \!\!\!
   2^{-2k} q_{2 k+1} \,\big(1 \!-\! 4 x (1\!-\!x)\big)^{ k}\! =\!(1\!-\!2x)
   \sum_{0 \leqslant k \leqslant n-1} r^{-}_{n-k} \big(x (1-x)\big)^{k} \, .
\end{align*}

\vspace{-0,2cm}
\noindent Hence,
\begin{align*}
    &    P_{n} \left(\frac{1}{x}\right)\! + \! P_{n} \left(\frac{1}{1-x}\right) = \eusm{R}^{+}\!\!\left[P_n\right] \left(\frac{1}{x (1-x)}\right) ,   &     &
  \!\!   \eusm{R}^{+}\!\!\left[P_n\right](x)\!:= \! \sum_{0 \leqslant k \leqslant n} r^{+}_{k} x^{k}, \  \\    &   P_{n} \left(\frac{1}{x}\right)\! - \! P_{n} \left(\frac{1}{1-x}\right) =(1\!-\!2x) \eusm{R}^{-}\!\!\left[P_n\right] \left(\frac{1}{x (1-x)}\right) ,   &     &
 \!\!   \eusm{R}^{-}\!\!\left[P_n\right](x)\!:= \!  \sum_{1 \leqslant k \leqslant n} r^{-}_{k} x^{k},
\end{align*}

\noindent where obviously, $ \eusm{R}^{-}\!\left[P_n\right]  (0)\!=\!0 $ and $ \eusm{R}^{+}\!\!\left[P_n\right](0)
\!= \!r^{+}_{0}\! =\! 2p_{0}\! = \!2 P (0)$. Lemma~\ref{intlem1} is proved.


\vspace{0.25cm}
Next we prove Lemma~\ref{intlem3}.
For arbitrary $z\! \in\! \Bb{D}\!\setminus\!\{0\} $ we deduce from \eqref{f1aint} that
\begin{align*}
    &    e^{{\fo{ n \pi \ \dfrac{\he(1-z)}{\he(z)}}}} =
\dfrac{16^{n}}{z^{n}} \  \exp\left(-n z
\dfrac{\dfrac{1}{2} +   \dfrac{2}{\pi} \sum\limits_{n=1}^{\infty} \dfrac{\Gamma (n+3/2)^{2}}{(n+1) !^{2}} \left[ \sum\limits_{k=1}^{n+1} \dfrac{1}{(2k -1)k}\right] z^{n}}{ 1 + \dfrac{1}{\pi}\sum\limits_{n = 1}^{\infty}
\dfrac{\Gamma(n+1/2)^{2}}{(n !)^{2}} z^{n}}\right) \\    &   =
\dfrac{16^{n}}{z^{n}}\sum\limits_{k\geqslant 0}\dfrac{( -nz)^{k}}{k !}
\left(\dfrac{\dfrac{1}{2} +   \dfrac{2}{\pi} \sum\limits_{n=1}^{\infty} \dfrac{\Gamma (n+3/2)^{2}}{(n+1) !^{2}} \left[ \sum\limits_{k=1}^{n+1} \dfrac{1}{(2k -1)k}\right] z^{n}}{ 1 + \dfrac{1}{\pi}\sum\limits_{n = 1}^{\infty}
\dfrac{\Gamma(n+1/2)^{2} }{(n !)^{2}} z^{n}}\right)^{k}= \dfrac{16^{n}}{z^{n}} \\    &
+\dfrac{16^{n}}{z^{n}}(- n z) \dfrac{\dfrac{1}{2} +   \dfrac{2}{\pi} \sum\limits_{n=1}^{\infty} \dfrac{\Gamma (n+3/2)^{2}}{(n+1) !^{2}} \left[ \sum\limits_{k=1}^{n+1} \dfrac{1}{(2k -1)k}\right] z^{n}}{ 1 + \dfrac{1}{\pi}\sum\limits_{n = 1}^{\infty}
\dfrac{\Gamma(n+1/2)^{2} }{(n !)^{2}} z^{n}}  \\    &
+ \dfrac{16^{n}}{z^{n}}\sum\limits_{k\geqslant 2}\dfrac{( -nz)^{k}}{k !}
\left(\dfrac{\dfrac{1}{2} +   \dfrac{2}{\pi} \sum\limits_{n=1}^{\infty} \dfrac{\Gamma (n+3/2)^{2}}{(n+1) !^{2}} \left[ \sum\limits_{k=1}^{n+1} \dfrac{1}{(2k -1)k}\right] z^{n}}{ 1 + \dfrac{1}{\pi}\sum\limits_{n = 1}^{\infty}
\dfrac{\Gamma(n+1/2)^{2} }{(n !)^{2}} z^{n}}\right)^{k}  \\    &   = \dfrac{16^{n}}{z^{n}}\left(1-\dfrac{n z}{2}\right)
+ \dfrac{16^{n}}{z^{n}}\sum\limits_{k=2}^{\infty} \beta_{n, k}z^{k} = \dfrac{16^{n}}{z^{n}}
 - 8 n\dfrac{16^{n-1}}{z^{n-1}} +
 16^{n}\sum\limits_{k=2}^{n-1}\dfrac{ \beta_{n, k}}{z^{n-k}} \\    &
 +  16^{n}\sum\limits_{k=n}^{\infty} \beta_{n, k}z^{k-n} \, , \
\end{align*}

\noindent where $\beta_{n, k}$, $n\geqslant 1$, $k\geqslant 2$, are the real numbers such that
 $\sum_{k=2}^{\infty} \beta_{n, k}z^{k}\in {\rm{Hol}}(\Bb{D})$ for each $n\geqslant 1$.
 Hence, in the notation $S^{{\tn{\triangle}}}_{n} (z) =  \sum_{k=1}^{n} s^{{\tn{\triangle}}}_{n, k} z^{k} $, $n\geqslant 1$,  of \eqref{f5intlem3},
it follows that
\begin{align*}  &   S^{{\tn{\triangle}}}_{1} (z) = 16 z  \ , \quad
 S^{{\tn{\triangle}}}_{2} (z) = 16^{2} z (z - 1) \, , \   \  S^{{\tn{\triangle}}}_{n} (z) = 16^{n} z^{n}- 8 n 16^{n-1}z^{n-1} \\    &   +  \sum\limits_{k=2}^{n-1} s^{{\tn{\triangle}}}_{n, n-k} z^{n-k} \ ,   \quad  s^{{\tn{\triangle}}}_{n, n} =16^{n}
    \ , \quad
    s^{{\tn{\triangle}}}_{n, n-1} =- 8 n 16^{n-1} \ , \quad  n\geqslant 3\,,
\end{align*}

\noindent which proves \eqref{f5intlem3} and \eqref{f6intlem3}.

According to   Lemma~{\rm{\ref{intlemma1}}},
$\im\, \ie (1)\!=\!0 $ and therefore
\eqref{f22int}   for $z\!=\!1$   gives
\begin{align*}
    &   \sum\limits_{n=1}^{\infty} S^{{\tn{\triangle}}}_{n} \left(1 \right) {\rm{e}}^{\imag  \pi n y} = \frac{
\lambda^{\,\prime}\left(y\right)  }{i \pi\left(1  - \lambda\left(y\right)\right)} \ , \quad   \im\, y >  \im\, \ie (1)=0  \, , \
\end{align*}

\noindent where, in view of  \eqref{f19int}, \eqref{f2int} and \eqref{f3cint},  we have
\begin{align*}
    &    \frac{\lambda^{\,\prime}\left(y\right)  }{i \pi\left(1  - \lambda\left(y\right)\right)} =
\dfrac{\Theta_{2}(y)^{4} \Theta_{4}(y)^{4}\Theta_{3}(y)^{-4}}{\Theta_{4}(y)^{4}\Theta_{3}(y)^{-4}} =  \Theta_{2}(y)^{4} =  16 {\rm{e}}^{\imag  \pi y }
 \theta_{2}\left({\rm{e}}^{\imag  \pi y}\right)^{4}.
\end{align*}

\noindent Together with \eqref{f3wcint}(c) this proves \eqref{f2intlem3} and  \eqref{f1intlem3}(b).

To prove \eqref{f7intlem3}, observe that
$\im\, \ie (1/z) \to 0$ as $z\in\Bb{C}\!\setminus\!\Bb{R},  z \to 0$ (see \cite[p.\! 609, (4.6)]{bh2}), which yields, in view of Lemma~\ref{intlemma1},  that $ \im\, \ie (1/z) \to 0$ as $z\in\Bb{C},  z \to 0$. So that for every $a > 0$ one can find $\delta_{a} > 0$ such that
$0 < \im\, \ie (1/z) < a$ for all $ 0< |z| < \delta_{a}$ and
\eqref{f17aint} can be written as follows
\begin{align}\label{f4pintlem3}
    &  \hspace{-0,2cm}
\frac{ S^{{\tn{\triangle}}}_{n} (z)}{z}\! =\!
\sum\limits_{k=1}^{n} s^{{\tn{\triangle}}}_{n, k} z^{k-1}\! =\!  \frac{1}{2 \pi i} \int\nolimits_{{\fo{ -1 \! +  \!  \imag   a }} }^{{\fo{ 1 \! + \!   \imag    a }} }   \frac{
   \lambda^{\,\prime}\left(\zeta\right){\rm{e}}^{{\fo{- n \pi {\rm{i}}  \zeta }}}  }{1  - z\lambda\left(\zeta\right)}\, d \zeta\,, \ \
   z\in \left(\delta_{a} \Bb{D}\right) \setminus \{0\}\,.
\end{align}

\noindent By choosing a positive number $\eta_{a}< \delta_{a}$ less than
$(1/2) \min_{\zeta \in [-1+{\rm{i}} a, 1+{\rm{i}} a]} 1/|\lambda(\zeta)|$  we deduce from
\eqref{f4pintlem3} that
\begin{align*}
     &|S^{{\tn{\triangle}}}_{n} (z)|/|z|\leqslant (2/\pi)\big(\max\nolimits_{\zeta \in [-1+{\rm{i}} a, 1+{\rm{i}} a]} |\lambda^{\, \prime}(\zeta)|\big)\exp (n \pi a)  \ , \  z\in \eta_{a}  \Bb{D} \,.
\end{align*}

\noindent Therefore for every $\im\, y > a$ the series in \eqref{f22int}, written  in the form,
\begin{align*}
    &  {\rm{i}} \pi \sum\nolimits_{n=1}^{\infty}\left( \sum\nolimits_{k=1}^{n} s^{{\tn{\triangle}}}_{n, k} z^{k-1}\right) {\rm{e}}^{\imag  \pi n y} = {
\lambda^{\,\prime}\left(y\right)  }\big/{\left(1 - z \lambda\left(y\right)\right)}\,, \quad
   z\in \left(\delta_{a} \Bb{D}\right) \setminus \{0\} \, ,
\end{align*}

\noindent  converges uniformly over all $z\in \eta_{a}  \Bb{D}$ and we
can take the limit as $z \to 0$  to get, in view of  \eqref{f19int} and arbitrariness of $a>0$, the following equivalent form of \eqref{f7intlem3}
\begin{align}\label{f5pintlem3}
    &   \sum\nolimits_{n=1}^{\infty}s^{{\tn{\triangle}}}_{n, 1}  {\rm{e}}^{\imag  \pi n y} = \lambda^{\,\prime}\left(y\right)\big/({\rm{i}} \pi)=\Theta_{2}(y)^{4} \Theta_{4}(y)^{4}\Theta_{3}(y)^{-4} \,
 , \ \im\, y > 0 \ .
\end{align}

 It follows from the functional equation
\begin{align*}
     & \ie( z) - {\rm{sign}} ({\rm{Im}}\, z) =\ie\left(
{z}\big/{(z-1)}\right)  \ , \quad
    z \in \Bb{C}\setminus \Bb{R}\,,
\end{align*}

\noindent (see \cite[p.\! 48, (A.14n)]{bh1}) that
\begin{align}\label{f1pintlem3}
    &  \exp \left(- {\rm{i}} \pi n \ie\left({z}\big/{(z\!-\!1)}\right)\right)\!=\!(-1)^{n} \exp\left(- {\rm{i}} \pi n \ie( z)\right) , \
    z \!\in \!\Bb{C}\!\setminus \!\Bb{R} , \ n \!\geqslant\! 1 .
\end{align}

\noindent By using \eqref{f6int},
\begin{align*}
    &
\begin{array}{rclcll}
\exp\left(- {\rm{i}} \pi n \ie( z)\right)  & = &  S^{{\tn{\triangle}}}_{n} (1/z) &  +  & \Delta_{n}^{S}(z)\ ,  & \quad
    z \in \Bb{D}\setminus \{0\}  ,  \\[0,2cm]
\exp \left(-{\rm{ i}} \pi n \ie\left({z}\big/{(z\!-\!1)}\right)\right)& = & S^{{\tn{\triangle}}}_{n} (1-1/z)&  +  &\Delta_{n}^{S}\left(z\big/(z-1)\right)\ ,  & \quad
    z \in \Bb{C}_{\re\, < 1/2}\,  ,
\end{array}
\end{align*}

\noindent we obtain\vspace{-0,2cm}
\begin{align}\label{f2pintlem3}
     & S^{{\tn{\triangle}}}_{n} (1-1/z)=(-1)^{n}S^{{\tn{\triangle}}}_{n} (1/z)\! + \!(-1)^{n}\Delta_{n}^{S}(z)-\Delta_{n}^{S}\left(z\big/(z-1)\right) ,
\end{align}

\vspace{-0,1cm}
\noindent for every $z \in \left(\Bb{D}\setminus \{0\}\right)\cap\Bb{C}_{\re\, < 1/2}$. Since $S^{{\tn{\triangle}}}_{n}(0)=0$ and $z /(z-1)\!\in\! {\rm{Hol}}(\Bb{D})$ we have
$\Delta_{n}^{S}(z/(z-1))\!\in\! {\rm{Hol}}(\Bb{D}/2)$ and therefore \eqref{f2pintlem3} , \eqref{f12cint} and \eqref{f2intlem3} yield
\begin{align*}
     & S^{{\tn{\triangle}}}_{n} (1-1/z)-(-1)^{n}S^{{\tn{\triangle}}}_{n} (1/z)
=\left((-1)^{n} -1\right)\Delta_{n}^{S}(0)=\left(1-(-1)^{n} \right) r_{4}(n) =S^{{\tn{\triangle}}}_{n} (1) \,
\end{align*}

\noindent for every $z\in \left(\Bb{D}\setminus \{0\}\right)\big/2$,  which proves \eqref{f1intlem3}(a) and completes the proof of Lemma~\ref{intlem3}.

\vspace{0,2cm}
\subsection[\hspace{-0,15cm}. \hspace{-0,04cm}Proofs for Section~\ref{bs}]{\hspace{-0,11cm}{\bf{.}} Proofs for Section~\ref{bs}} $\phantom{a}$\vspace{-0,1cm}

\subsubsection[\hspace{-0,2cm}. Proofs of Lemmas~\ref{bslem1} and~\ref{bslem2}.]{\hspace{-0,11cm}{\bf{.}} Proofs of Lemmas~\ref{bslem1} and~\ref{bslem2}.}
\label{pbslem1}

We first prove  Lemma~\ref{bslem1}.
The properties Lemma~\ref{bslem1}(1),(2),(3) are immediate from \eqref{f28int}, \eqref{f27int}, \eqref{f27aint}
\eqref{f1bs} and  \eqref{f3intth1}.
To prove Lemma~\ref{bslem1}(4) for positive integer $n$ we change the contour $\gamma (-1,1)$ of integration in \eqref{f28int} to
the contour which passes from $-1$ to $-1+{\rm{i}} $  along  $[-1, -1 + {\rm{i}}]$, from $-1 + {\rm{i}}$ to $1+{\rm{i}} $  along
 $[-1 + {\rm{i}}, 1 + {\rm{i}}]$ and from $1 + {\rm{i}}$ to $1$  along  $[1 + {\rm{i}}, 1 ]$.  By using the periodicity of $ \eurm{R}_{n}^{{\tn{\triangle}}} (z)$
and $ \eurm{R}_{n}^{{\tn{\triangle}}} (-1/z)$,  we obtain
\begin{align*}
    & - 4 \pi^{2} n \eurm{H}_{n} (x) \! =  \!\int\nolimits_{-1}^{1}\!
\frac{ \eurm{R}_{n}^{{\tn{\triangle}}} ( {\rm{i}} + t) {\rm{d}} t}{(x+ {\rm{i}} + t)^{2}}\!+\!
{\rm{i}}\!\int\nolimits_{0}^{1}\!\eurm{R}_{n}^{{\tn{\triangle}}} (-1\!+\!{\rm{i}} t)
\!\left[\dfrac{1}{(x\!-\!1\!+\!{\rm{i}}t)^{2}}\!-\! \dfrac{1}{(x\!+\!1\!+\!{\rm{i}}t)^{2}} \right] \!{\rm{d}}  t, \\    &
 4 \pi^{2} n \eurm{M}_{n} (x) \! =  \!\int\nolimits_{-1}^{1}\!
\frac{ \eurm{R}_{n}^{{\tn{\triangle}}}\left(\dfrac{-1}{ {\rm{i}} \!+\! t}\right)
 {\rm{d}}  t}{(x+ {\rm{i}} + t)^{2}}\!+\!
{\rm{i}}\!\int\nolimits_{0}^{1}\!\eurm{R}_{n}^{{\tn{\triangle}}}\left(\dfrac{1}{1\!-\!{\rm{i}} t}\right)
\!\left[\dfrac{1}{(x\!-\!1\!+\!{\rm{i}}t)^{2}}\!-\! \dfrac{1}{(x\!+\!1\!+\!{\rm{i}}t)^{2}} \right] \!{\rm{d}}  t,
\end{align*}

\noindent and therefore, for any positive integer $N \geqslant 2$ it follows that
\begin{align*}\hspace{-0,1cm}
    - 4 \pi^{2} n \!\!\!\sum_{k=-N}^{N} \eurm{H}_{n} (x+2k)  &\! =  \! \int\nolimits_{-1}^{1}\!\eurm{R}_{n}^{{\tn{\triangle}}} ( {\rm{i}} + t) \left(\sum_{k=-N}^{N}\frac{1 }{(x+2k+ {\rm{i}} + t)^{2}}\right) {\rm{d}}  t
 \\ \hspace{-0,2cm}   &
\!+\!
{\rm{i}}\!\!\int\nolimits_{0}^{1}\!\!\!\eurm{R}_{n}^{{\tn{\triangle}}} (-1\!+\!{\rm{i}} t)
\!\left[\dfrac{1}{(x\!-\!2N\!-\!1\!+\!{\rm{i}}t)^{2}}\!-\! \dfrac{1}{(x\!+\!2N\!+\!1\!+\!{\rm{i}}t)^{2}} \right] \!{\rm{d}}  t, \\ \hspace{-0,1cm}
 4 \pi^{2} n \!\!\!\sum_{k=-N}^{N} \eurm{M}_{n} (x+2k)& \! =  \!\int\nolimits_{-1}^{1}\!\eurm{R}_{n}^{{\tn{\triangle}}}\left(-\dfrac{1}{ {\rm{i}} \!+\! t}\right)
\left(\sum_{k=-N}^{N}\frac{ 1 }{(x+2k+ {\rm{i}} + t)^{2}}\right) {\rm{d}}  t  \\  \hspace{-0,2cm}   &   \!+\!
{\rm{i}}\!\!\int\nolimits_{0}^{1}\!\!\!\eurm{R}_{n}^{{\tn{\triangle}}}\left(\dfrac{1}{1\!-\!{\rm{i}} t}\right)
\!\left[\dfrac{1}{(x\!-2N\!\!-\!1\!+\!{\rm{i}}t)^{2}}\!-\! \dfrac{1}{(x\!+\!2N\!+\!1\!+\!{\rm{i}}t)^{2}} \right] \!{\rm{d}}  t .
\end{align*}

\noindent Letting $N \to \infty$ and using the identity
(see \cite[p.\! 44, 1.422.4]{gra}){\hyperlink{r15}{${}^{\ref*{case15}}$}}\hypertarget{br15}{}
\begin{align*}
    &  \sum\limits_{k \in \Bb{Z}} \frac{1}{(2 k + z)^{2}} =
    \frac{\pi^{2}}{4 \sin^{2} \dfrac{\pi z}{2}}
    =  - \pi^{2} \sum_{m \geqslant 1}m {\rm{e}}^{{\fo{ {\rm{i}} \pi m z}}}
    \ ,  \qquad z \in \Bb{H}     \, , \
\end{align*}

\noindent we derive from the first of the latter two equalities and \eqref{f12aint} that
\begin{align*}
    &  - 4 \pi^{2} n \sum_{{\fo{k \in \Bb{Z}}}}   \eurm{H}_{n} (x + 2k)=
\frac{\pi^{2}}{4}\int\limits_{-1}^{1}\!
\frac{\eurm{R}_{n}^{{\tn{\triangle}}} ( {\rm{i}} + t)\, {\rm{d}}  t}{
\sin^{2} \dfrac{\pi (x+{\rm{i}} + t)}{2}} = - \pi^{2} \times
  \\    &
\times \! \sum_{m \geqslant 1}m  {\rm{e}}^{ -  \pi m } {\rm{e}}^{ {\rm{i}} \pi m x}
    \int\limits_{-1}^{1} {\rm{e}}^{ {\rm{i}} \pi m t}
\bigg(\! {\rm{e}}^{ n \pi } {\rm{e}}^{-{\rm{i}} n \pi t}\! + \!\Delta_{n}^{S} (0) + \sum\limits_{k\geqslant 1} {\eurm{d}}_{n,k} {\rm{e}}^{-  \pi k } {\rm{e}}^{\imag  \pi k t}\!\bigg)
{\rm{d}}  t \\    &   =- 2 \pi^{2} n {\rm{e}}^{ {\fo{{\rm{i}} \pi n x}}} \, , \
\end{align*}

\noindent and from the second one  and \eqref{f23aaint} that
\begin{align*}
    &  4 \pi^{2} n (-1)^{n} \sum_{{\fo{k \in \Bb{Z}}}}  \eurm{M}_{n} (x+2k) \! =  \! \frac{(-1)^{n}\pi^{2}}{4} \int\limits_{-1}^{1}\!
   \dfrac{ \eurm{R}_{n}^{{\tn{\triangle}}}\left(-\dfrac{1}{ {\rm{i}}+ t}\right)\, {\rm{d}}  t}{
\sin^{2} \dfrac{\pi (x+{\rm{i}} + t)}{2}}   =- \pi^{2}\times
\\    &
\times\! \sum_{m \geqslant 1}m  {\rm{e}}^{ -  \pi m } {\rm{e}}^{ {\rm{i}} \pi m x}
    \int\limits_{-1}^{1} {\rm{e}}^{ {\rm{i}} \pi m t}
\bigg(\! -S^{{\tn{\triangle}}}_{n}(1)\!-\!16s^{{\tn{\triangle}}}_{n, 1}
{\rm{e}}^{- \pi  }
{\rm{e}}^{\imag  \pi  t} \!+\!\sum\limits_{k\geqslant 2} {\eurm{b}}_{n,k}{\rm{e}}^{-  \pi k }{\rm{e}}^{\imag  \pi k t}\!\bigg)
{\rm{d}}  t \\    &   =0\,.
\end{align*}

\noindent This completes the proof of  Lemma~\ref{bslem1}(4) for positive integer $n$. But according to  \eqref{f2intth1}, for any $n \geqslant 1$ we have
\begin{align*}
     & \sum_{{\fo{k \in \Bb{Z}}}}   \eurm{H}_{-n} (x + 2k)=
     \sum_{{\fo{k \in \Bb{Z}}}}   \eurm{H}_{n} (-x - 2k) =
     \sum_{{\fo{k \in \Bb{Z}}}}   \eurm{H}_{n} (-x + 2k) = \frac{{\rm{e}}^{ {\fo{-{\rm{i}} \pi n x}}}}{2} \ , \ x\! \in \! \Bb{R} \,,
\end{align*}

\noindent  and, by \eqref{f3intth1},
\begin{align*}
     & \eurm{M}_{-n}(x) :=  \eurm{H}_{-n} (-1/x)/x^{2}= \eurm{H}_{n} (1/x)/x^{2}=\eurm{M}_{n}(-x) \ , \ x\! \in \! \Bb{R}\! \setminus\! \{0\} \ , \
\end{align*}

\noindent from which $\eurm{M}_{-n}(x) =\eurm{M}_{n}(-x)$ for all $x\! \in \! \Bb{R}$, in view of the continuity of $\eurm{M}_{n}$. Then
\begin{align*}
     & \sum_{{\fo{k \in \Bb{Z}}}}   \eurm{M}_{-n} (x + 2k)=
     \sum_{{\fo{k \in \Bb{Z}}}}   \eurm{M}_{n} (-x - 2k) =
     \sum_{{\fo{k \in \Bb{Z}}}}   \eurm{M}_{n} (-x + 2k) =0 \ , \ x\! \in \! \Bb{R} \,.
\end{align*}

 \noindent This proves  Lemma~\ref{bslem1}(4) and completes the proof of  Lemma~\ref{bslem1}.

\vspace{0.25cm}
Next, we prove  Lemma~\ref{bslem2}.
The properties Lemma~\ref{bslem2}(1),(2),(4) are simple consequences of \eqref{f24int} and   \eqref{f9int}.
To prove Lemma~\ref{bslem2}(3) we change the contour $\gamma (-1,1)$ of integration in \eqref{f9int}
similar to that of in Section~\ref{pbslem1} and by the periodicity of $\Theta_{3}$  we get
\begin{align*}
    &  2 \pi\imag  \eurm{H}_{0} (x)\!=\!
\int\limits_{-1}^{1}\!
 \frac{({\rm{i}}+t) \, \Theta_{3}\left({\rm{i}}+t\right)^{4}}{  x^{2} - ({\rm{i}}+t)^{2}} {\rm{d}} t \\    &   +
\dfrac{{\rm{i}}}{2}\int\limits_{0}^{1} \! \Theta_{3}\left(-1+{\rm{i}}t\right)^{4}
\left[\frac{1}{  x\! - \!(-1\!+\!{\rm{i}}t)} \!- \!\frac{1}{  x\! + \!(-1\!+\!{\rm{i}}t)} \!- \!
\frac{1}{  x \!-\! (1\!+\!{\rm{i}}t)} \!+\! \frac{1}{  x \!+\! (1\!+\!{\rm{i}}t)}\right]\! {\rm{d}} t,
\end{align*}

\noindent \noindent and therefore, for any positive integer $N \geqslant 2$ this yields
\begin{align*}
    & 2 \pi\imag \sum_{k=-N}^{N} \eurm{H}_{0} (x+2k)\!=\!
\int\limits_{-1}^{1}\!\Theta_{3}\left({\rm{i}}+t\right)^{4}
  \left(\sum_{k=-N}^{N}\frac{({\rm{i}}+t) \, }{  (x+2k)^{2} - ({\rm{i}}+t)^{2}} \right){\rm{d}} t \\    &   +
\dfrac{{\rm{i}}}{2}\int\limits_{0}^{1} \! \Theta_{3}\left(-1+{\rm{i}}t\right)^{4}
\left[\frac{1}{  x\! + 2N + \!1\!-\!{\rm{i}}t}  \!- \!
\frac{1}{  x\!-\!2N \!-\! 1\!-\!{\rm{i}}t} \right]\! {\rm{d}} t \\    &   +
\dfrac{{\rm{i}}}{2}\int\limits_{0}^{1} \! \Theta_{3}\left(-1+{\rm{i}}t\right)^{4}
\left[ \frac{1}{  x \!+\!2N\!+\! 1\!+\!{\rm{i}}t} \!- \!\frac{1}{  x\! -\!2N \!-1\!+\!{\rm{i}}t}\right]\! {\rm{d}} t\,.
\end{align*}

\noindent Letting $N \to \infty$ and using the identity{\hyperlink{r16}{${}^{\ref*{case16}}$}}\hypertarget{br16}{}
\begin{align*}
       \sum\limits_{ {\fo{ k \in \Bb{Z} }} } \dfrac{ z}{(2 k + x)^{2} - z^{2}}&=
- \dfrac{\pi}{4}\left(\cot \dfrac{\pi(z -  x)}{2}  + \cot \dfrac{\pi(z +  x)}{2}\right)
 \\    &    =
   \dfrac{\pi {\rm{i}}}{2} +  {\rm{i}} \pi \sum\limits_{ {\fo{ m  \geqslant 1  }} }{\rm{e}}^{{\rm{i}} \pi m z} \cos \pi m x
    \ ,  \qquad z \in \Bb{H}  \, , \  x \in \Bb{R}    \, , \
\end{align*}

\noindent we derive from the latter  equality that
\begin{align*}
    &  2 \pi\imag \sum_{{\fo{k \in \Bb{Z}}}}   \eurm{H}_{0} (x + 2k)=
- \dfrac{\pi}{4}\int\limits_{-1}^{1}\!\Theta_{3}\left({\rm{i}}+t\right)^{4}
  \left(\cot \dfrac{\pi({\rm{i}}+t -  x)}{2}  + \cot \dfrac{\pi({\rm{i}}+t +  x)}{2} \right){\rm{d}} t\\    &   =
\int\limits_{-1}^{1}\!
\left(1  \!+ \! 2\sum\limits_{n\geqslant 1} {\rm{e}}^{- \pi n^{2} } {\rm{e}}^{{\rm{i}} \pi n^{2} t} \right)^{\!4}
\left(
 \dfrac{\pi {\rm{i}}}{2} +  {\rm{i}} \pi \sum\limits_{ {\fo{ m  \geqslant 1  }} }{\rm{e}}^{- \pi m }{\rm{e}}^{{\rm{i}} \pi m t} \cos \pi m x
\right) {\rm{d}} t={\rm{i}} \pi \, .
\end{align*}

\noindent This proves  Lemma~\ref{bslem2}(3) and completes the proof of  Lemma~\ref{bslem2}.


\vspace{0,2cm}
\subsection[\hspace{-0,15cm}. \hspace{-0,04cm}Proofs for Section~\ref{bsden}]{\hspace{-0,11cm}{\bf{.}} Proofs for Section~\ref{bsden}} $\phantom{a}$\vspace{-0,1cm}

\subsubsection[\hspace{-0,2cm}. Proofs of \eqref{f8contgen} and \eqref{f3bsden}.]{\hspace{-0,11cm}{\bf{.}} Proofs of \eqref{f8contgen} and \eqref{f3bsden}.}
\label{pf8contgen} \ We first prove \eqref{f8contgen}. In the notation
\begin{align}\label{f1pf8contgen}
     & I := \begin{pmatrix} 1 & 0 \\ 0 & 1 \\ \end{pmatrix}  \, , \ \
 T := \begin{pmatrix} 1 & 1 \\ 0 & 1 \\ \end{pmatrix}  \, , \ \  S := \begin{pmatrix} 0 & 1 \\ -1 & 0 \\ \end{pmatrix} \, ,\ \
\phi_{{\nor{(\begin{smallmatrix} a & b \\ c & d \end{smallmatrix})}}}(z) := \dfrac{a  z + b}{c  z + d} \,  ,
\end{align}

\noindent where $ z\in \Bb{H} $ and  $a, b, c, d \!\in\! \Bb{Z}$ such that $ad - bc\! \neq\! 0$, we obviously have
\begin{align}\label{f2pf8contgen}
     &
\begin{array}{l}
\begin{displaystyle}
S^{2} =  - I\, , \quad T^{ -2 n}=
  \begin{pmatrix} 0 & -2 n \\ 0 &  1 \\ \end{pmatrix} \, , \quad  S T^{ -2 n}=
  \begin{pmatrix} 0 & 1 \\ -1 &  2 n \\ \end{pmatrix} \, ,
\end{displaystyle}\\[0,4cm]
 \begin{displaystyle}
\phi_{{\fo{ T^{{\fo{ -2 n}}}}}}(z)= z-2n  \in \Bb{H}\, , \quad
\phi_{{\fo{S T^{ {\fo{-2 n}}}}}}(z)= \dfrac{1}{2n -z}  \in \Bb{H}
 \, , \quad n\in \Bb{Z}  \, .
\end{displaystyle}
\end{array}
\end{align}

 Let $\phi \!\in \!\Gamma_{\vartheta}^{\hspace{0,015cm}{\tn{||}}}$. Then, by the definitions \eqref{f1bsden} and  \eqref{f1zbsden}, we have
\begin{align}\label{f3pf8contgen}
     & \phi(z)\!= \!\phi_{\eufm{n}}(z)\!= \!\phi_{{\fo{n_{N}, ..., n_{1}}}}(z)=
 \phi_{ {\fo{ S T^{{\fo{-2 n_{N}}}}  S T^{{\fo{-2 n_{N-1}}}}\ldots S T^{{\fo{-2 n_{1}}}}   }}}(z) \ ,
\end{align}

\noindent for some $ \eufm{n}\!=\!(n_{N}, ..., n_{1})\in \Bb{Z}_{\neq 0}^{\hspace{0,02cm}\Bb{N}_{\hspace{-0,02cm}\eurm{f}}}\!:=\!\cup_{k\geqslant 1} \Bb{Z}_{\neq 0}^{k}$, $N \in \Bb{N}$ and every $z\in \Bb{H} $.

To prove the right-hand side inclusion of \eqref{f8contgen}, which by \eqref{f7contgen} can be written
in the form\vspace{-0,3cm}
\begin{align}\label{f4pf8contgen}
    &   \phi_{\eufm{n}} \left(\Bb{H}_{|\re|\leqslant 1}\right)\! \subset\! \Bb{H}_{|\re|<1}\! \setminus\!\fet \, , \ \  \eufm{n} \!\in \! \Bb{Z}_{\neq 0}^{\hspace{0,02cm}\Bb{N}_{\hspace{-0,02cm}\eurm{f}}} \, , \
\end{align}

\noindent observe that $\Bb{H}_{|\re|\leqslant 1}- 2 n \subset \Bb{H}_{|\re|\geqslant 1}$ for any
$n \in \Bb{Z}_{\neq 0}$, and therefore it follows from $- 1/ \Bb{H}_{|\re|\geqslant 1} = \Bb{H}_{|\re|<1}\! \setminus\!\fet$ that
\begin{align}\label{f5pf8contgen}
    &\hspace{-0,2cm} \phi_{{\fo{S T^{ {\fo{-2 n}}}}}} \left(\Bb{H}_{|\re|\leqslant 1}\right)\subset
      \Bb{H}_{|\re|<1}\! \setminus\!\fet \subset  \Bb{H}_{|\re| \leqslant 1} \, , \ \  \eufm{n} \!\in \! \Bb{Z}\!\setminus\!\{0\} \,.
\end{align}

\noindent Successive applications of \eqref{f5pf8contgen} to \eqref{f3pf8contgen} prove \eqref{f4pf8contgen}
and the right-hand side inclusion of \eqref{f8contgen} as well.

\vspace{0,1cm}
For  $ \eufm{n}\!=\!(n_{N}, ..., n_{1})\!\in\! \Bb{Z}_{\neq 0}^{N}$ and $N\! \in \!\Bb{N}$ from \eqref{f3pf8contgen} introduce the matrices
\begin{align}
    &  \label{f0preresauxlem1}
         & \hspace{-0,4cm}\left\{   \begin{array}{ll}
\begin{displaystyle}  \begin{pmatrix}  p^{\eufm{n}}_{-1} & q^{\eufm{n}}_{-1} \\ p^{\eufm{n}}_{0} & q^{\eufm{n}}_{0}\\ \end{pmatrix} := \begin{pmatrix} 1 & 0 \\ 0 &  1 \\ \end{pmatrix}
   \ , \ \end{displaystyle}
    \\[0.5cm]
\begin{displaystyle}   \begin{pmatrix} p^{\eufm{n}}_{k-1} & q^{\eufm{n}}_{k-1}  \\ p^{\eufm{n}}_{k} & q^{\eufm{n}}_{k} \\ \end{pmatrix}\! :=  \!\begin{pmatrix}0 & 1 \\ -1 &  2 n_{k} \\ \end{pmatrix} \begin{pmatrix}0 & 1 \\ -1 &  2 n_{k-1} \\ \end{pmatrix}\ldots \begin{pmatrix}0 & 1 \\ -1 &  2 n_{1} \\ \end{pmatrix} ,  \
 1 \!\leqslant\! k\! \leqslant\! N \,  , \end{displaystyle}   \end{array} \right.
\end{align}

\noindent whose elements obviously satisfy
\begin{align}\label{f2apreresauxlem1a}
p^{\eufm{n}}_{k-1} q^{\eufm{n}}_{k} - p^{\eufm{n}}_{k} q^{\eufm{n}}_{k-1} = 1  \ , \  \quad\quad    0 \leqslant k \leqslant N \ ,
 \end{align}

\noindent and\vspace{-0,2cm}
\begin{align}\label{f2preresauxlem1a}
     &  \left\{  \begin{array}{lll}
        p^{\eufm{n}}_{k} = 2 n_{k} p^{\eufm{n}}_{k-1} - p^{\eufm{n}}_{k-2} \ , \   &       \ p^{\eufm{n}}_{0} = 0  \ ,   &   \   p^{\eufm{n}}_{-1} =  1\  ,    \\
         q^{\eufm{n}}_{k} = 2 n_{k} q^{\eufm{n}}_{k-1} -\, q^{\eufm{n}}_{k-2} \ , \   &        \ q^{\eufm{n}}_{0} = 1  \ ,   &   \ q^{\eufm{n}}_{-1} =  0  \  ,
              \end{array} \right.  \quad\quad    1 \leqslant k \leqslant N \,.
\end{align}

\noindent In the notation \eqref{f7contgen} and \eqref{f3pf8contgen}, we have (see \cite[p.\! 63]{bh1})
\begin{align}\label{f6zpf8contgen}
     & \phi_{\eufm{n}}(z)\!= \! \dfrac{z p^{\eufm{n}}_{N-1} + q^{\eufm{n}}_{N-1}}{z p^{\eufm{n}}_{N} + q^{\eufm{n}}_{N}} \ , \
     \eufm{n}\!=\!(n_{N}, ..., n_{1})\!\in\! \Bb{Z}_{\neq 0}^{N} \ , \ N\! \in \!\Bb{N} \, , \ z \in \Bb{H} \,,
\end{align}

\noindent and ($\phi\!= \!\phi_{\eufm{n}}$)
\begin{align}\label{f6pf8contgen}
    &  \left|a_{\phi}\right|=\left|p^{\eufm{n}}_{N-1}\right| \, , \
    \left|b_{\phi}\right|=\left|q^{\eufm{n}}_{N-1}\right| \, , \
    \left|c_{\phi}\right|=\left|p^{\eufm{n}}_{N}\right| \, , \
    \left|d_{\phi}\right|=\left|q^{\eufm{n}}_{N}\right| \, .
\end{align}

\noindent
We prove that
\begin{align}\label{f1preresauxlem2}
    &
      {\rm{(a)}} \ \ \big|q^{\eufm{n}}_{k}\big| > \big|p^{\eufm{n}}_{k}\big| > \big|p^{\eufm{n}}_{k-1}\big| \ ,  \quad   {\rm{(b)}} \  \ \big|q^{\eufm{n}}_{k}\big| > \big|q^{\eufm{n}}_{k-1}\big| > \big|p^{\eufm{n}}_{k-1}\big| \ , \quad 1 \leqslant k \leqslant N \,  .
  \end{align}

For arbitrary positive integer $Q$, real numbers $a, b$ satisfying $|a| < |b|$ and nonzero integers $\{m_{k}\}_{k = 1}^{Q}\subset \Bb{Z}_{\neq 0}$,  we define the following collection
$\{y_{k}\}_{k = -1}^{Q}$ of real numbers
\begin{gather}\label{f1preresauxprlem2}
  y_{k} = 2 m_{k} y_{k-1} -\, y_{k-2} \ , \ \ y_{0} = b    \ , \ y_{-1} =  a  \ ,  \ \quad  1 \leqslant k \leqslant Q \, , \ \ |a| < |b| \ .
\end{gather}

\noindent We state that
\begin{align}\label{f2preresauxprlem2}
    &   0 < |b| = |y_{0}|< |y_{1}| < \ldots   < |y_{Q}| \ .
\end{align}

\noindent We prove $|y_{k-1}|< |y_{k}|$, $ 1 \leqslant k \leqslant Q$, by induction on $k$. For $k=1$ such an inequality holds because   $| y_{1}| = |2 m_{1} b -\, a | \geqslant 2\cdot |m_{1} |\cdot |b| - |a| > 2\cdot |m_{1} |\cdot |b| - |b| = ( 2\cdot |m_{1} |-1) |b| \geqslant  |b|$, and hence,  $|b| =| y_{0}| <| y_{1}|$. This also completely proves \eqref{f2preresauxprlem2} for $Q=1$. If $Q\geqslant 2$
we assume that  $| y_{k-2}| < | y_{k-1}|  $ for all $ 2 \leqslant k \leqslant P$  for some $ 2\leqslant P \leqslant Q$. Then
$| y_{P }| = |2 m_{P } y_{P -1} -\, y_{P -2}| \geqslant 2\cdot |m_{P } |\cdot | y_{P -1}| - |y_{P -2}| > 2\cdot |m_{P } |\cdot | y_{P -1}| - |y_{P -1}|  =
    ( 2\cdot |m_{P } |-1) | y_{P -1}|  \geqslant | y_{P -1}|$, i.e. $| y_{P -1}| < | y_{P }|$. By induction, we conclude that \eqref{f2preresauxprlem2} is true.

For $k=1$ \eqref{f1preresauxlem2} holds because   $p^{\eufm{n}}_{0} = 0$, $q^{\eufm{n}}_{0}=1$, $p^{\eufm{n}}_{1}= -1$ and $q^{\eufm{n}}_{1}= 2 n_{1}$,  in view of \eqref{f0preresauxlem1}.

It remains to examine  in \eqref{f1preresauxlem2} the case $2 \leqslant k \leqslant N$.
By setting first $Q:=N-1$, $a:=0$, $b:=-1$, $y_{k}:=p^{\eufm{n}}_{k+1}$, $-1 \leqslant k \leqslant N-1$,
  $m_{k}:=n_{k+1}$, $1 \leq k \leqslant N-1$, in \eqref{f1preresauxprlem2} and then
  $Q:=N$, $a:=0$, $b:=1$, $y_{k}:=q^{\eufm{n}}_{k}$, $-1 \leq k\leqslant N$, $m_{k}:=n_{k}$, $1 \leqslant k \leqslant N$,
   we deduce from \eqref{f2preresauxlem1a} and \eqref{f2preresauxprlem2} that \vspace{-0,1cm}
\begin{align}\label{f6preresauxprlem2}
    &  |q^{\eufm{n}}_{k}|\! >\! |q^{\eufm{n}}_{k-1}|\! \geqslant\! 1 \, , \  1 \!\leqslant\! k \!\leqslant\! N \, ;    \ \
     |p^{\eufm{n}}_{k}|\! > \!|p^{\eufm{n}}_{k-1}|\! \geqslant\! 1  \ , \ 2 \!\leqslant\! k \!\leqslant\! N \ ,  \ \     | p^{\eufm{n}}_{1}|\! >\! |p^{\eufm{n}}_{0}|  \, ,
\end{align}

\vspace{-0,1cm}
\noindent where the last inequality follows from $p^{\eufm{n}}_{0} = 0$ and $p^{\eufm{n}}_{1}= -1$.

For positive integer  $2 \leqslant r \leqslant N$ let  $\eufm{n}(r\!):= (-n_{1}, -n_{2}, \ldots, -n_{r})\!\in\! \Bb{Z}_{\neq 0}^{r}$. Introducing  by the formulas \eqref{f0preresauxlem1} matrices  corresponding to $\eufm{n}(r\!)$, we get \vspace{-0,1cm}
\begin{align*}
   &  A_{k}^{\eufm{n}(r)}:=
 \begin{pmatrix} p_{k-1}^{\eufm{n}(r)} & q_{k-1}^{\eufm{n}(r)}  \\ p_{k}^{\eufm{n}(r)} & q_{k}^{\eufm{n}(r)} \\ \end{pmatrix} =  \begin{pmatrix}0 & 1 \\ -1 &  -2 n_{r-k+1} \\ \end{pmatrix}
 \begin{pmatrix}0 & 1 \\ -1 &  -2 n_{r-k+2} \\ \end{pmatrix}\ldots \begin{pmatrix}0 & 1 \\ -1 &  -2 n_{r} \\ \end{pmatrix}
   \end{align*}

\vspace{-0,1cm}
\noindent for every $1\leqslant k \leqslant r $,  where, in view of \eqref{f2preresauxlem1a} and \eqref{f6preresauxprlem2},\vspace{-0,1cm}
\begin{align}&  \big| p_{k}^{\eufm{n}(r)}\big| > \big|p_{k-1}^{\eufm{n}(r)}  \big| \, ,
\qquad  \big|q_{k}^{\eufm{n}(r)}\big| > \big|q_{k-1}^{\eufm{n}(r)}\big| \, ,
\qquad  1\leqslant k \leqslant r \ .
\label{f8preresauxprlem2}\end{align}

\vspace{-0,1cm}
\noindent Applying the transpose operation to $A_{r}^{\eufm{n}(r)}$,  we obtain\vspace{-0,1cm}
\begin{align*}
    & (-1)^{r}\! \begin{pmatrix} p_{r-1}^{\eufm{n}(r)} & q_{r-1}^{\eufm{n}(r)}  \\ p_{r}^{\eufm{n}(r)} & q_{r}^{\eufm{n}(r)} \\ \end{pmatrix} \!=\!
\begin{pmatrix}0 & -1 \\ 1 &  2 n_{1} \\ \end{pmatrix}
 \begin{pmatrix}0 & -1 \\ 1 &  2 n_{2} \\ \end{pmatrix}\ldots \begin{pmatrix}0 & -1 \\ 1 &  2 n_{r} \\ \end{pmatrix} \! =\!
\begin{pmatrix} p^{\eufm{n}}_{r-1} & q^{\eufm{n}}_{r-1}  \\ p^{\eufm{n}}_{r} & q^{\eufm{n}}_{r} \\ \end{pmatrix}^{\!\!{\rm{T}}}\!\! \!=\!
\begin{pmatrix} p^{\eufm{n}}_{r-1} &  p^{\eufm{n}}_{r} \\ q^{\eufm{n}}_{r-1} & q^{\eufm{n}}_{r} \\ \end{pmatrix} , \end{align*}

\vspace{-0,1cm}
\noindent from which,\vspace{-0,1cm}
\begin{align}\label{f8pf8contgen}
    &
\begin{array}{rclrcl}
  (-1)^{r} p_{r-1}^{\eufm{n}(r)}  &    =   &    p^{\eufm{n}}_{r-1}  \, ,   &   \quad (-1)^{r} p_{r}^{\eufm{n}(r)}
  &    =   &    q^{\eufm{n}}_{r-1}  \, ,  \\[0,3cm]
(-1)^{r} q_{r-1}^{\eufm{n}(r)}&    =   &  p^{\eufm{n}}_{r}  \, ,   &   \quad
(-1)^{r} q_{r}^{\eufm{n}(r)}&    =   &  q^{\eufm{n}}_{r}\, ,
\end{array}
\end{align}

\vspace{-0,1cm}
\noindent and therefore, by arbitrariness of  $2 \!\leqslant\! r\! \leqslant\! N$, we deduce from \eqref{f8preresauxprlem2} that
\vspace{-0,1cm}
\begin{align}\label{f7pf8contgen}
    &
\left| p^{\eufm{n}}_{k}\right| < \left| q^{\eufm{n}}_{k}\right| \ , \quad  1 \leqslant k \leqslant N \, .
\end{align}

\vspace{-0,1cm}
\noindent Together with \eqref{f6preresauxprlem2} and \eqref{f6pf8contgen} this proves the required
inequalities $|a_{\phi}|\!<\! |b_{\phi}|\!<\!|d_{\phi}|$,$ |a_{\phi}|\!<\! |c_{\phi}|\!<\!|d_{\phi}|$ in the left-hand side of \eqref{f8contgen} (cp. \cite[p.\! 4, Lemma 2]{boc}).

Besides that, the relationships \eqref{f2preresauxlem1a} written for ${\eufm{n}}(r)$ and $k=r$ give\vspace{-0,1cm}
\begin{align*}
    &  \begin{array}{rccccclr}
     q^{\eufm{n}}_{r-1} \!  & \!  = \!  &   (-1)^{r}  p^{\eufm{n}(r)}_{r}  &  \! =\!   &(-1)^{r}\left( \!-\!2 n_{1} p^{\eufm{n}(r)}_{r-1} \!-\! p^{\eufm{n}(r)}_{r-2}\right)  & \!  = \!  & -\!2 n_{1} p^{\eufm{n}}_{r-1} \!- \! p^{\eufm{n}(r)}_{r-2} \, ,   &  |p^{\eufm{n}(r)}_{r-2} | < |p^{\eufm{n}}_{r-1} |,     \\[0,3cm]
    q^{\eufm{n}}_{r}  &  \! = \!  &  (-1)^{r}      q^{\eufm{n}(r)}_{r}  &  \! =  \! & (-1)^{r} \left(\!-\!2 n_{1} q^{\eufm{n}(r)}_{r-1} \!-\! q^{\eufm{n}(r)}_{r-2}\right)  &  \! = \!  & -\!2 n_{1} p^{\eufm{n}}_{r}\! -\! q^{\eufm{n}(r)}_{r-2} \, ,  &  |q^{\eufm{n}(r)}_{r-2} | < |p^{\eufm{n}}_{r} |,
              \end{array}
\end{align*}

\vspace{-0,1cm}
\noindent where the inequalities \eqref{f8preresauxprlem2} have been used. Since $r$ is arbitrary integer satisfying $2 \!\leqslant\! r\! \leqslant\! N$, we get\vspace{-0,1cm}
\begin{align}\label{f9preresauxprlem2}
    &   \left|  q^{\eufm{n}}_{k} + 2 n_{1} p^{\eufm{n}}_{k} \right| \leqslant \left| p^{\eufm{n}}_{k}\right| -1 \ , \quad  1 \leqslant k \leqslant N \, .
\end{align}

\vspace{0.25cm}
Next, we prove \eqref{f3bsden}. Let  $N\! \in \!\Bb{N}$ and $ \eufm{n}\!=\!(n_{N}, ..., n_{1})\!\in\! \Bb{Z}_{\neq 0}^{N}$. Then in the notation \eqref{f0preresauxlem1} we have
\begin{align}\label{f1preresauxcor1}
    &  \begin{array}{ll}
        {\rm{(a)}} \ \  {\rm{sign}} \, (q^{\eufm{n}}_{k}) \cdot {\rm{sign}} \, (q^{\eufm{n}}_{k-1}) =
\sigma_{k}       \ ,  &   \ \
{\rm{(b)}} \ \     {\rm{sign}} \, (p^{\eufm{n}}_{k}) \cdot {\rm{sign}} \, (q^{\eufm{n}}_{k}) = -
\sigma_{1}
     \ ,  \\ \phantom{{\rm{(a)}} \ \  }
       \sigma_{k} := {\rm{sign}} \, (n_{k}) \ ,   & \ \   \phantom{{\rm{(b)}} \ \  }   1  \leqslant k \leqslant N \, .
       \end{array}
\end{align}

\noindent Indeed, \eqref{f6preresauxprlem2} and \eqref{f2preresauxlem1a} yield that
$q^{\eufm{n}}_{k}/q^{\eufm{n}}_{k-1} \in (2n_{k}-1, 2n_{k} +1)$ for each $1  \leqslant k \leqslant N$ and hence,  \eqref{f1preresauxcor1}(a) holds. To prove \eqref{f1preresauxcor1}(b) assume to the contrary that there exists $1 \leqslant k \leqslant N $ such that $  {\rm{sign}} \, (p^{\eufm{n}}_{k}) \cdot {\rm{sign}} \, (q^{\eufm{n}}_{k}) = \sigma_{1}$. Then $|  q^{\eufm{n}}_{k} + 2 n_{1} p^{\eufm{n}}_{k} | = |q^{\eufm{n}}_{k} | + 2 |n_{1}| | p^{\eufm{n}}_{k} | \geqslant 2  | p^{\eufm{n}}_{k} | +1 > | p^{\eufm{n}}_{k} | - 1$ which contradicts
\eqref{f9preresauxprlem2}. This contradiction proves \eqref{f1preresauxcor1}(b). Observe, that
 \eqref{f1preresauxcor1} for $N \geqslant 2$ yields
\begin{align}\label{f0pf3bsden}
    & p^{\eufm{n}}_{0}  =0 \ , \quad p^{\eufm{n}}_{1}  =-1 \ , \quad  {\rm{sign}} \, (p^{\eufm{n}}_{k}) \cdot {\rm{sign}} \, (p^{\eufm{n}}_{k-1}) =
\sigma_{k}       \ , \quad 2  \leqslant k \leqslant N \, .
\end{align}

\noindent
Then
 we deduce from \eqref{f6zpf8contgen} and \eqref{f1preresauxcor1} that (in accordance with \eqref{f9contgen}, we have $\sigma_{\eufm{n}} = \sigma_{1} $)
\begin{align}\label{f1pf3bsden}
    & \hspace{-0,2cm}
\begin{array}{lcccclcl}
  \phi_{\eufm{n}}(-\sigma_{\eufm{n}})  &    =  &
    \dfrac{  q^{\eufm{n}}_{ N-1}   - \sigma_{1}p^{\eufm{n}}_{ N-1}  }{ q^{\eufm{n}}_{ N} -  \sigma_{1}p^{\eufm{n}}_{ N} }&    =  &     \sigma_{N} \, \dfrac{ |q^{\eufm{n}}_{ N-1}|+|p^{\eufm{n}}_{ N-1}|}{|q^{\eufm{n}}_{ N}|+  |p^{\eufm{n}}_{ N}|} \ ,   &    \phi_{\eufm{n}}(\infty) &    =  &    \dfrac{p^{\eufm{n}}_{N-1}}{p^{\eufm{n}}_{N}}\ ,
         \\[0.5cm]
  \phi_{\eufm{n}}(\sigma_{\eufm{n}}) &    =  &
    \dfrac{  q^{\eufm{n}}_{ N-1}   + \sigma_{1}p^{\eufm{n}}_{ N-1}  }{ q^{\eufm{n}}_{ N} +  \sigma_{1}p^{\eufm{n}}_{ N} }&    =  &     \sigma_{N} \, \dfrac{ |q^{\eufm{n}}_{ N-1}|-|p^{\eufm{n}}_{ N-1}|}{|q^{\eufm{n}}_{ N}|-  |p^{\eufm{n}}_{ N}|} \ ,   &    \phi_{\eufm{n}}(\, 0\, ) &    =  &    \dfrac{q^{\eufm{n}}_{N-1}}{q^{\eufm{n}}_{N}} \ .
\end{array}
\hspace{-0,1cm}\end{align}

\noindent It follows from \eqref{f7pf8contgen} that $q^{\eufm{n}}_{N} (q^{\eufm{n}}_{N} \pm p^{\eufm{n}}_{N}) >0$
 and therefore we obtain
the  inequalities $\phi_{\eufm{n}}(- 1)\! < \!\phi_{\eufm{n}}(0) \!< \!\phi_{\eufm{n}}(1) $ of \eqref{f3bsden}(a) , because \eqref{f2apreresauxlem1a} gives
\begin{align}\label{f2pf3bsden}
    & \hspace{-0,2cm} \phi_{\eufm{n}}(0)\!- \!\phi_{\eufm{n}}(- 1) \! =\! {1}\big/\left(q^{\eufm{n}}_{N} (q^{\eufm{n}}_{N} \!-\! p^{\eufm{n}}_{N})\right)  , \
\phi_{\eufm{n}}(1)\! -\!  \phi_{\eufm{n}}(0) \!=\! {1}\big/\left(q^{\eufm{n}}_{N} (q^{\eufm{n}}_{N}\! + \! p^{\eufm{n}}_{N})\right) .
\end{align}

\noindent
At the same time, for $N=1$ we have $\sigma_{\eufm{n}} \phi_{\eufm{n}}(-\sigma_{\eufm{n}}) =
\sigma_{1} \phi_{\eufm{n}}(-\sigma_{\eufm{n}}) > 0 =  \phi_{\eufm{n}}(\infty)$, as follows from \eqref{f1pf3bsden}
and \eqref{f2preresauxlem1a}. This proves \eqref{f3bsden}(b) for  $N=1$. But if
$N \geqslant 2$ then  we deduce from \eqref{f2apreresauxlem1a} that the sign of the number
$ \phi_{\eufm{n}}(-\sigma_{\eufm{n}}) -   \phi_{\eufm{n}}(\infty) = - 1/ (p^{\eufm{n}}_{ N} ( q^{\eufm{n}}_{ N} -  \sigma_{1}p^{\eufm{n}}_{ N}) )$ is equal to $ {\rm{sign}} (p^{\eufm{n}}_{ N}q^{\eufm{n}}_{ N}) = - \sigma_{1}$,
in view of \eqref{f7pf8contgen}  and \eqref{f1preresauxcor1}(b). This proves \eqref{f3bsden}(b) for $N \geqslant 2$.
 Since \eqref{f3bsden}(c) is immediate from \eqref{f1preresauxlem2}(a) and \eqref{f1pf3bsden},
  the proof of \eqref{f3bsden} is completed. Hence,
  for $N\! \in \!\Bb{N}$ and $ \eufm{n}\!=\!(n_{N}, ..., n_{1})\!\in\! \Bb{Z}_{\neq 0}^{N}$,
the hyperbolic quadrilateral  $\fet^{\hspace{0,01cm}\eufm{n}}$  has the shape shown in Figure~\ref{figure:image}, provided that $n_{1}\geqslant 1$ (in view of \eqref{f1preresauxcor1}(b), $n_{1}\geqslant 1$  yields $p^{\eufm{n}}_{N} \, q^{\eufm{n}}_{N} \leqslant -1$).

\begin{figure}
		\centering
\begin{tikzpicture}
\definecolor{cv4}{rgb}{0.95,0.95,0.95}
 \begin{scope}[scale=1]
\clip(-7,-2.3) rectangle (5,4.75);
\draw (-6,0) circle (0.04) (-6.15,0) node[above] {$-1$};
\draw (4,0) circle (0.04) (4,0) node[above] {$1$};
\draw (-5,0) circle (0.04)  node[below] {$\dfrac{\vphantom{A^{A}} p^{\eufm{n}}_{N-1}}{p^{\eufm{n}}_{N}}  $};
\draw (-5,-1)  node[below] {$\| $};
\draw (-5,-1.3)  node[below] {$\phi_{{\fo{n_{N}, ..., n_{1}}}}(\infty)$};
\draw (-2,0) circle (0.04)  node[below] {$\dfrac{\vphantom{A^{A}} q^{\eufm{n}}_{N-1} - p^{\eufm{n}}_{N-1} }{q^{\eufm{n}}_{N} - p^{\eufm{n}}_{N}}$};
\draw (-2,-1)  node[below] {$\| $};
\draw (-2,-1.3)  node[below] {$\phi_{{\fo{n_{N}, ..., n_{1}}}}(-1)$};
\draw (0.4,0)circle (0.04)  node[below] {$\dfrac{\vphantom{A^{A}} q^{\eufm{n}}_{N-1}}{q^{\eufm{n}}_{N}} $};
\draw (0.4,-1)  node[below] {$\| $};
\draw (0.4,-1.3)  node[below] {$\phi_{{\fo{n_{N}, ..., n_{1}}}}(0)$};
\draw (3,0) circle (0.04)  node[below] {$ \dfrac{\vphantom{A^{A}} q^{\eufm{n}}_{N-1} + p^{\eufm{n}}_{N-1} }{q^{\eufm{n}}_{N}+p^{\eufm{n}}_{N}}$};
\draw (3,-1)  node[below] {$\| $};
\draw (3,-1.3)  node[below] {$\phi_{{\fo{n_{N}, ..., n_{1}}}}(1)$};
\draw[fill=cv4] (-5,0) arc (180:0:1.5)  arc (180:0:1.2)
(0.4,0) arc (180:0:1.3) (3,0) arc (0:180:4);
\draw[densely dotted] (-2,0) arc (180:0:2.5);
 \draw (-1,2.6) node[above]{$\mathcal{F}_{{\tn{\square}}}^{\hspace{0,05cm}n_{N}, ..., n_{1}}
 \!:= \! \phi_{{\fo{n_{N}, ..., n_{1}}}}\big( \mathcal{F}_{{\tn{\square}}}\big)
 $};
\draw[draw=lightgray] (1.2,2.4) -- (1.2,4);
\draw (1.15,2.5) --(1.2,2.4) -- (1.25,2.5);
\draw (1.1,4) node[above] {${{\gamma_{{\fo{n_{N}, ..., n_{1}}}}(-1, 1)}}$};
\draw[draw=lightgray] (-2.95,3.5) -- (-2.95,4);
\draw (-2.9,3.6) --(-2.95,3.5) -- (-3,3.6);
\draw (-2.9,4) node[above] {${{\gamma_{{\fo{n_{N}, ..., n_{1}}}}(1, \infty)}}$};
\draw (-4.75,2.6) node[above]{$n_{1}\!\geqslant\! 1$};
\draw[draw=lightgray] (-5.3,2.6) -- (-4.2,2.6) -- (-4.2,3.15) -- (-5.3,3.15) -- (-5.3,2.6);
\draw (-6,0) -- (4,0);
\end{scope}
\end{tikzpicture}
\caption{\hspace{-0,2cm}{\bf{.}} Image $ \fet^{\hspace{0,05cm}n_{N}, ..., n_{1}}$ of $\fet$ by $\phi_{\eufm{n}}(z)\!= \! \dfrac{z p^{\eufm{n}}_{N-1} + q^{\eufm{n}}_{N-1}}{z p^{\eufm{n}}_{N} + q^{\eufm{n}}_{N}}$, $
     \eufm{n}\!=\!(n_{N}, ..., n_{1})\!\in\! \Bb{Z}_{\neq 0}^{N}$, $ N\! \in \!\Bb{N}$.}
\label{figure:image}
\vspace{-0,5cm}	\end{figure}

\subsubsection[\hspace{-0,2cm}. Proofs of Lemmas~\ref{bsdentplem1} and~\ref{bsdentplem2}.]{\hspace{-0,11cm}{\bf{.}} Proofs of Lemmas~\ref{bsdentplem1} and~\ref{bsdentplem2}.}
\label{pbsdentplem1} \ We first prove Lemma~\ref{bsdentplem1}.

{\rm{(a)}} Assume the contrary, i.e., there exist $N, M \!\in \!\Bb{N}$, $N\! \geqslant\! M$, $\eufm{n}\!:=\!(n_{N}, ..., n_{1})\!\in \! \Bb{Z}_{\neq 0}^{N}$,  $\eufm{m}\!:=\!(m_{N}, ..., m_{1})\!\in \! \left(\Bb{Z}_{\neq 0}\right)^{M}$, $\eufm{n} \neq \eufm{m}$,  $z\!\in\! \fet \sqcup \left( \sigma_{n_{1}}\! +\! {\rm{i}} \Bb{R}_{>0} \right)$, $y\!\in \!\fet \sqcup \left( \sigma_{m_{1}}\! +\! {\rm{i}} \Bb{R}_{>0} \right)$ such that
$\phi_{{\fo{n_{N}, ..., n_{1}}}}\big( z\big) = \phi_{{\fo{m_{M}, ..., m_{1}}}}\big( y\big)$, which
by \eqref{f3pf8contgen} can be written as follows (see \eqref{f9contgen})
\begin{align}\label{f1pbsdentplem1}\hspace{-0,2cm}
    &   \phi_{ {\fo{ S T^{{\fo{-2 n_{N}}}}  S T^{{\fo{-2 n_{N-1}}}}\ldots S T^{{\fo{-2 n_{1}}}}   }}}\!(z)\!=\!
\phi_{ {\fo{ S T^{{\fo{-2 m_{M}}}}  S T^{{\fo{-2 m_{M-1}}}}\ldots S T^{{\fo{-2 m_{1}}}}   }}}\!(y).\hspace{-0,1cm}
\end{align}

\noindent Without loss of generality one can consider in \eqref{f1pbsdentplem1} that  $n_{N} \neq m_{M}$, because otherwise, we can apply $\phi_{ T^{{{2 n_{\!N}}}}S }$ to the both parts of \eqref{f1pbsdentplem1}
to get the similar relationship for  $\eufm{n}$ and $\eufm{m}$ of lower dimension and to get $\eufm{n} = \eufm{m}$, if $M=N$, or
\begin{align}\label{f6pbsdentplem1}
    &\hspace{-0,2cm} \Bb{H}_{|\re|<1}\! \setminus\!\fet \ni  \phi_{ {\fo{  S T^{{\fo{-2 n_{N-M}}}}\ldots S T^{{\fo{-2 n_{1}}}}   }}}(z)= y \in \fet^{\,{\tn{|}}\sigma_{m_{1}}}\, , \   \
 \mbox{if} \ N-M \geqslant 1\,,
\end{align}

\noindent where the left-hand side inclusion follows from \eqref{f8contgen}. But \eqref{f6pbsdentplem1} cannot hold since  $\fet^{\,{\tn{|}}\sigma_{m_{1}}}\cap (\Bb{H}_{|\re|<1}\! \setminus\!\fet)=\emptyset$. So that we assume everywhere below that  $n_{N} \neq m_{M}$.

Let  $N=M=1$ in \eqref{f1pbsdentplem1}. Then $z-y = 2n_{1}-2m_{1}$, where $n_{1} \neq m_{1}$. Since
$\re\, z, \re\, y \in [-1,1]$ then $n_{1}-m_{1}$ can take only two values $1$ or $-1$. If
$n_{1}-m_{1} = \sigma \in \{1, -1\}$ then there exists $T > 0$ such that $z = \sigma + {\rm{i}} T$ and
$y = -\sigma + {\rm{i}} T$, which yields $\sigma_{n_{1}}= \sigma$ and  $\sigma_{m_{1}}= -\sigma$, by virtue of
\eqref{f9contgen} and \eqref{f10contgen}. But then $|n_{1} - m_{1}| \geqslant 2$ which contradicts
$|n_{1} - m_{1}| = 1$. Therefore $N \geqslant 2$, $M \geqslant 1$ and it follows from \eqref{f1pbsdentplem1}
that
\begin{align*}  &
\begin{array}{l}
 \phi_{ {\fo{  S T^{{\fo{-2 n_{N-1}}}}\ldots S T^{{\fo{-2 n_{1}}}}   }}}(z)=
\phi_{ {\fo{   T^{{\fo{-2( m_{1}- n_{N})}}}   }}}(y) \, , \   \ \mbox{if} \ M=1\,,  \\
  \phi_{ {\fo{ S T^{{\fo{-2( m_{2}- n_{N})}}}  S T^{{\fo{-2 n_{N-1}}}}\ldots S T^{{\fo{-2 n_{1}}}}   }}}(z)=\phi_{ {\fo{   T^{{\fo{-2 m_{1}}}}   }}}(y)\, , \   \ \mbox{if} \ M=2\,,  \\
      \phi_{ {\fo{S T^{{\fo{2 m_{2}}}}\ldots T^{{\fo{2 m_{M-1}}}} S T^{{\fo{-2( m_{M}- n_{N})}}}  S T^{{\fo{-2 n_{N-1}}}}\ldots S T^{{\fo{-2 n_{1}}}}   }}}(z)=\phi_{ {\fo{   T^{{\fo{-2 m_{1}}}}   }}}(y)\, ,
\end{array}
\end{align*}

\noindent if $M \geqslant 3$. But none of these equations can hold because their left-hand sides belong to
$ \Bb{H}_{|\re|<1}\! \setminus\!\fet$, in view of \eqref{f8contgen}, while their right-hand sides belong to $\Bb{H}_{|\re|\geqslant 1}$.
Lemma~\ref{bsdentplem1}{\rm{(a)}} is proved.

\vspace{0,1cm}
{\rm{(d)}} Let $\eufm{n}\in \Bb{Z}_{\neq 0}^{N}$ be such that $N \geqslant 2$ if  $n_{1}\!=\! \sigma_{n_{1}} = \sigma_{\eufm{n}}$ (see \eqref{f9contgen}).
As  noted before Lemma~\ref{bsdentplem1},
$\phi_{\eufm{n}} \big(\gamma (\sigma_{\eufm{n}}, \infty) \big) \subset  \fetr^{\hspace{0,01cm}\eufm{n}}$ is the  roof of $\fetr^{\hspace{0,01cm}\eufm{n}}$. But
$\phi_{\eufm{n}} \big(\gamma (\sigma_{\eufm{n}}, \infty) \big)$  is also a part of the boundary of\vspace{-0,2cm}
\begin{align*}
     \phi_{{\fo{n_{N}, ..., n_{1}}}}\left(\fet\! +\! 2\sigma_{\eufm{n}}\right) &\! =\!
 \phi_{{\fo{n_{N}, ..., n_{2}}}}\left(\dfrac{1}{2n_{1}\! - \!\fet\! - \!2\sigma_{\eufm{n}}   }\right) \\[0,2cm]
\left|- 1/\fet\!= \!\fet \, \Rightarrow \,  \dfrac{1}{2n_{2}\! +\! 1/\fet}\!= \!
\dfrac{1}{2n_{2} \!-\!\fet}\right|
&   \!=\!
\left\{
  \begin{array}{ll}
  \phi_{{\fo{n_{N}, ..., n_{2}, n_{1} -\sigma_{\eufm{n}}}}}\left(\fet\right), & \hbox{if} \ \,  n_{1}\!\neq \! \sigma_{\eufm{n}} ,\\
 \phi_{{\fo{n_{N}, ..., n_{2}}}}\left(\fet\right), & \hbox{if} \ \,  n_{1}\!=\! \sigma_{\eufm{n}} ,
  \end{array}
\right.
\end{align*}

\noindent  and since\vspace{-0,2cm}
\begin{align*}
     \phi_{{\fo{n_{N}, ..., n_{1}}}} \left(\gamma (\sigma_{\eufm{n}},\infty)\right)&
= \phi_{{\fo{n_{N}, ..., n_{2}}}}\left(\dfrac{1}{2n_{1}\! - \!\gamma (-\sigma_{\eufm{n}},\infty)\! - \! 2\sigma_{\eufm{n}} }\right)  \\[0,2cm]  \left| \dfrac{1}{2n_{2}\! +\! 1/\gamma (-\sigma_{\eufm{n}},{\fo{\infty}})}\!= \!
\dfrac{1}{2n_{2} \!-\!\gamma (\sigma_{\eufm{n}},0)}\right| & \!=\!
\left\{
  \begin{array}{ll}
  \phi_{{\fo{n_{N}, ..., n_{2}, n_{1} \!-\!\sigma_{\eufm{n}} }}}\!\left(\gamma (-\sigma_{\eufm{n}},{\fo{\infty}})\right), & \hbox{if} \, \,  n_{1}\!\neq\! \sigma_{\eufm{n}} ,\\
 \phi_{{\fo{n_{N}, ..., n_{2}}}}\!\left(\gamma (\sigma_{\eufm{n}},0)\right), & \hbox{if} \, \,
 n_{1}\!= \! \sigma_{\eufm{n}} ,
  \end{array}
\right.
\end{align*}\!

\noindent $\sigma_{\hspace{0,005cm}n_{N}, ..., n_{2}, n_{1} \!-\sigma_{\eufm{n}}  }\! =\! {\rm{sign}} (n_{1} \!-\sigma_{\eufm{n}} )\! =\!{\rm{sign}}(n_{1})\!=\! \sigma_{\eufm{n}}$ for $n_{1}\!\neq\! \sigma_{\eufm{n}}$,
we get that
\begin{align*}
    &  \phi_{\eufm{n}} \big(\gamma (\sigma_{\eufm{n}}, \infty) \big)=
\left\{
  \begin{array}{ll}
  \gamma_{{\fo{\hspace{0,005cm}n_{N}, ..., n_{2}, n_{1} \!-\sigma_{\eufm{n}}  }}} (-\sigma_{{\fo{\hspace{0,005cm}n_{N}, ..., n_{2}, n_{1} \!-\sigma_{\eufm{n}}}}}, \infty), & \hbox{if} \, \,  n_{1}\!\neq\! \sigma_{\eufm{n}}\, ,
\\[0,2cm]
 \gamma_{{\fo{\hspace{0,005cm}n_{N}, ..., n_{2} }} } (\sigma_{\eufm{n}},0), \ \
 \sigma_{\eufm{n}} \in \{1,-1\}\,,
& \hbox{if} \, \,
 n_{1}\!= \! \sigma_{\eufm{n}}\, .
  \end{array}
\right.
\end{align*}

\noindent According to what was  stated before Lemma~\ref{bsdentplem1}, this means that
$\phi_{\eufm{n}} \big(\gamma (\sigma_{\eufm{n}}, \infty) \big)$
is
the lower arc of $ \fetr^{\hspace{0,05cm}n_{N}, ..., n_{1}\!-\sigma_{n_{1}}}\!\!$, if $n_{1}\!\neq\! \sigma_{\eufm{n}}$. At the
same time, $\phi_{\eufm{n}} \big(\gamma (\sigma_{\eufm{n}}, \infty) \big)$ for $n_{1}\!= \! \sigma_{\eufm{n}}$  is
the lower arc of $\fetr^{\hspace{0,05cm}n_{N}, ..., n_{2}}$. Lemma~\ref{bsdentplem1}{\rm{(d)}} is proved. Lemma~\ref{bsdentplem1}{\rm{(b)}} and {\rm{(c)}} follow by  similar arguments.{\hyperlink{r18}{${}^{\ref*{case18}}$}}\hypertarget{br18}{}

\vspace{0,1cm}
{\rm{(e)}} As  noted before Lemma~\ref{bsdentplem1},  $\phi_{\eufm{n}} \big(\gamma (\sigma_{\eufm{n}}, \infty) \big)$ is the  roof of $\fetr^{\hspace{0,01cm}\eufm{n}}$ and since
$\phi_{\eufm{n}} \big(\gamma (\sigma_{\eufm{n}}, \infty) \big) =\gamma
\big(\phi_{\eufm{n}}(\sigma_{\eufm{n}}), \phi_{\eufm{n}}(\infty) \big)$,
  by virtue of \eqref{12bsdenac}, we get
\begin{align}\label{f2pbsdentplem1}
    &  2 \max \left\{ \ \im z \, \left| \, z \in  \fetr^{\hspace{0,05cm}n_{N}, ..., n_{1}} \,  \right\}\right.=
\left|\phi_{\eufm{n}}(\infty)-\phi_{\eufm{n}}(\sigma_{\eufm{n}})\right| \,,
\end{align}

\noindent where in accordance with \eqref{f1pf3bsden}, \eqref{f2apreresauxlem1a}, \eqref{f1preresauxcor1} and
\eqref{f1preresauxlem2},
\begin{align}\label{f3pbsdentplem1}
    & \hspace{-0,3cm} |\phi_{\eufm{n}}(\infty)\!-\!\phi_{\eufm{n}}(\sigma_{\eufm{n}})|\!= \! \left|\dfrac{p^{\eufm{n}}_{N-1}}{p^{\eufm{n}}_{N}}\!- \!\dfrac{  q^{\eufm{n}}_{ N-1}  \! + \!\sigma_{1}p^{\eufm{n}}_{ N-1}  }{
 q^{\eufm{n}}_{ N} \!+  \!\sigma_{1}p^{\eufm{n}}_{ N} }\right|\!=\!
\dfrac{1}{\left|p^{\eufm{n}}_{N}\right| \left(\left|q^{\eufm{n}}_{ N}\right|\! -\!\left|p^{\eufm{n}}_{ N}\right|\right)} \leqslant \dfrac{1}{\left|p^{\eufm{n}}_{N}\right|}\ .\hspace{-0,1cm}
\end{align}

\noindent We prove now
\begin{align}\label{f4pbsdentplem1}
    &  \left|p^{\eufm{n}}_{k}\right| \geqslant 1+ \sum\nolimits_{m=2}^{k}\, (2 |n_{m}| -1)\geqslant k   \ , \quad    0 \leqslant k \leqslant N \,,
\end{align}

\noindent by induction on $k$, where  $\sum_{m=2}^{0} :=\sum_{m=2}^{1} := 0$ and $\{p^{\eufm{n}}_{k}\}_{k=-1}^{N}$ are defined as in \eqref{f2preresauxlem1a}. In view of \eqref{f2preresauxlem1a}, $p^{\eufm{n}}_{1}=-1$, and therefore \eqref{f4pbsdentplem1} holds for $k=1$. In addition, \eqref{f4pbsdentplem1} is proved for $N = 1$. Assume that $N \geqslant 2$, $2 \leqslant r \leqslant N$ and   \eqref{f2preresauxlem1a} is true for $k=r-1$.
In view of \eqref{f6preresauxprlem2}, we have
\begin{align*}
    & 0 = |p^{\eufm{n}}_{0}| <|p^{\eufm{n}}_{1}| < |p^{\eufm{n}}_{2}|<\ldots < |p^{\eufm{n}}_{r}|  \ .
\end{align*}

\noindent  Then $|p^{\eufm{n}}_{r-2}|\leq |p^{\eufm{n}}_{r-1}|-1$ and by
 \eqref{f2preresauxlem1a} we get
\begin{align*}
    | p^{\eufm{n}}_{r} | &= |2 n_{r} p^{\eufm{n}}_{r-1} - p^{\eufm{n}}_{r-2}| \geqslant 2 |n_{r}| |p^{\eufm{n}}_{r-1}| - |p^{\eufm{n}}_{r-2}| \geqslant 1 + ( 2 |n_{r}|-1) |p^{\eufm{n}}_{r-1}| \\    &
\geqslant 1 + ( 2 |n_{r}|-1) \left(1 + \sum\limits_{m=1}^{r-1} (2 |n_{m}| -1)\right) \geqslant  1 + \sum\limits_{m=1}^{r} (2 |n_{m}| -1)  ,
\end{align*}

\noindent which proves the validity of \eqref{f4pbsdentplem1}  for $k=r$. By induction, we conclude that \eqref{f4pbsdentplem1} is true for all $0 \leqslant k \leqslant N$. Combining \eqref{f4pbsdentplem1} for $k=N$,
\eqref{f3pbsdentplem1} and \eqref{f2pbsdentplem1} gives Lemma~\ref{bsdentplem1}{\rm{(e)}} and completes the proof of Lemma~\ref{bsdentplem1}.


\vspace{0.25cm}
Next we prove Lemma~\ref{bsdentplem2}.
Since $\fetr^{\hspace{0,01cm}\eufm{n}} \subset \Bb{H}_{|\re|<1} \setminus\fet$ for every $\eufm{n}\in \Bb{Z}_{\neq 0}^{\hspace{0,02cm}\Bb{N}_{\hspace{-0,02cm}\eurm{f}}}$ it suffices to prove that for any point in
$\Bb{H}_{|\re|<1} \setminus\fet $ there exists $\eufm{n}\in \Bb{Z}_{\neq 0}^{\hspace{0,02cm}\Bb{N}_{\hspace{-0,02cm}\eurm{f}}}$  such that
$\fetr^{\hspace{0,01cm}\eufm{n}}$ contains this point.

 Let $z\in \Bb{H}_{|\re|< 1} \setminus \fet$, $N \in \Bb{N}$ and assume that
$\Bb{G}^{k}_{2}(z)\in \Bb{H}_{|\re|< 1} \setminus \fet$ for each $1\leqslant k \leqslant N$.
We denote by $Q = Q(z)$ the number $(1/2) + ((1/4) - (\im\, z)^{2})^{1/2}> 1$. Then by \eqref{f17contgen},
$Q^{N} \im\, z \leqslant 1/2$, because $\Bb{H}_{|\re|\leqslant 1}\cap \Bb{H}_{\im > 1/2}\subset \mathcal{F}^{\,{\tn{||}}}_{{\tn{\square}}}$. Therefore for arbitrary $z\in \Bb{H}_{|\re|< 1} \setminus \fet$
there exists a minimal positive integer $\eurm{r} (z)$ such that
\begin{align}\label{f5contgendef1}
    & \eurm{r} (z):= \min \left\{\,k\in \Bb{N}\,|\, \Bb{G}^{k}_{2}(z)\in  \mathcal{F}^{\,{\tn{||}}}_{{\tn{\square}}}\,\right\} \ , \quad  z\in \Bb{H}_{|\re|< 1} \setminus \fet\,,
\end{align}

\noindent where, in accordance with \eqref{f3contgendef1},
\begin{multline}
\Bb{G}_{2}^{k} (z)\! = \!2 j_{k} (z)\! -\! \dfrac{1}{\Bb{G}_{2}^{k-1} (z)} =
\left\{\re \left(- \dfrac{1}{\Bb{G}_{2}^{k-1} (z)}\right)\right\}^{\!{\tn{\rceil\hspace{-0,015cm}\lceil}}}_{\!2} + {\rm{i}}\, \im \left(- \dfrac{1}{\Bb{G}_{2}^{k-1} (z)}\right),
\\         j_{k}(z)\!:=\!
- \dfrac{1}{2}\left\rceil{\re\left(- \dfrac{1}{\Bb{G}_{2}^{k-1} (z)}\right)}\right\lceil_{2}\!\! \in\!
 \Bb{Z}_{\neq 0} \, , \quad  1\!\leqslant\! k\!\leqslant\! {\eurm{r}} (z).
\label{f5bsdentp}\end{multline}

\noindent  Since
$\phi_{j_{k}(z)}(\Bb{G}_{2}^{k} (z))=\Bb{G}_{2}^{k-1} (z)$ for each $1\!\leqslant\! k\!\leqslant\! {\eurm{r}} (z)$
we deduce that
\begin{align}\label{f6bsdentp}
    &   \phi_{j_{1}(z), ..., j_{{\eurm{r}}(z)}(z)}(\Bb{G}_{2}^{{\eurm{r}}(z)} (z))= z \ , \quad  \Bb{G}_{2}^{{\eurm{r}}(z)} (z)\in
\mathcal{F}^{\,{\tn{||}}}_{{\tn{\square}}}\ , \quad  z\in \Bb{H}_{|\re|< 1} \setminus \fet\ ,
\end{align}

\noindent and if $\Bb{G}_{2}^{{\eurm{r}}(z)} (z) \in \gamma(\sigma,\infty)$ for some $\sigma\in \{1,-1\}$, then
\eqref{f5bsdentp} yields
\begin{align*}
    & \left\{\re \left(- {1}\big/{\Bb{G}_{2}^{k-1} (z)}\right)\right\}^{\!{\tn{\rceil\hspace{-0,015cm}\lceil}}}_{\!2}= \sigma\ , \quad  \sigma\in \{1,-1\} \, .
\end{align*}

\noindent In view of \eqref{f1ybsden}, this means that $\re (- {1}\big/{\Bb{G}_{2}^{k-1} (z)})
\in - \sigma (2\Bb{N}-1)$ and therefore $\rceil{\re (- {1}\big/{\Bb{G}_{2}^{k-1} (z)})}\lceil_{ 2}\,
\in - 2 \sigma \Bb{N}$. Then ${\rm{sign}} (j_{{\eurm{r}}(z)}(z)) = \sigma $, by virtue of \eqref{f5bsdentp}.
Hence, in addition to \eqref{f6bsdentp}, we have
\begin{align}\label{f7bsdentp}
    &  \Bb{G}_{2}^{{\eurm{r}}(z)} (z) \in \gamma(\sigma,\infty) \, , \  \sigma\in \{1,-1\} \ \Rightarrow \
\sigma_{{\eurm{r}}}(z):={\rm{sign}} (j_{{\eurm{r}}(z)}(z)) = \sigma\,.
\end{align}

\noindent Thus,\vspace{-0,3cm}
\begin{align}\label{f8bsdentp}
    &  z \in   \phi_{j_{1}(z), ..., j_{{\eurm{r}}(z)}(z)}\left( \fet^{\,{\tn{|}}\,\sigma_{{\eurm{r}}}(z)} \right) = \fetr^{\hspace{0,05cm}j_{1}(z), ..., j_{{\eurm{r}}(z)}(z)} \,,
\end{align}

\noindent which completes the proof of Lemma~\ref{bsdentplem2}.

\subsubsection[\hspace{-0,2cm}. Proofs of Theorems~\ref{bsdentpth2} and~\ref{bsdentpth3}.]{\hspace{-0,11cm}{\bf{.}} Proofs of Theorems~\ref{bsdentpth2} and~\ref{bsdentpth3}.}
\label{pbsdentpth2} $\phantom{a}$


\vspace{0.25cm} We first prove Theorem~\ref{bsdentpth2}. Let $L := \{\phi  (\gamma(1, \infty))\}_{\phi \in \Gamma_{\vartheta}}$.
In accordance with  \eqref{zbssqlem1},
\begin{align}\label{f2pbsdentpth2}
    & \partial_{\,\Bb{H}}\fet\! :=\! \Bb{H}\cap\partial\fet\! =\! \gamma(1, \infty) \sqcup \gamma(-1, \infty)\sqcup \gamma(-1, 0) \sqcup \gamma(0, 1),
\end{align}

\noindent
and since $-1/  \left(\partial_{\Bb{H}}\fet\right)= \partial_{\Bb{H}}\fet $ we deduce from
\eqref{f6contgen} that the union $S$ of all boundary points of the sets composing the  Schwarz
partition \eqref{f4bsdentp} of $\Bb{H}$, satisfies
\begin{align*}
    &  S  \! =\!
 \bigcup\limits_{n \in \Bb{Z}} \ \left(2 n + \big\{\phi_{\eufm{n}}
\big(\partial_{\,\Bb{H}}\fet\big)
\big\}_{{\fo{\eufm{n} \in \{0\}\cup\Bb{Z}_{\neq 0}^{\hspace{0,02cm}\Bb{N}_{\hspace{-0,02cm}\eurm{f}}}}}}\right)\! =\!
 \left\{ \, \phi \big(\partial_{\,\Bb{H}}\fet\big) \,\right\}_{{\fo{\phi\! \in\! \Gamma_{\vartheta} }}}\supset L\,.
\end{align*}

\noindent  Inverse inclusion $S \subset L$ is immediate from the invariance of $L$ under any transform $\phi\! \in\! \Gamma_{\vartheta}$ and obvious properties $\gamma(-1, \infty)= \gamma(1, \infty)-2\in L$,
$\gamma(-1, 0) =-1/\gamma(1, \infty)\in L$ and  $\gamma(1, 0) =-1/\gamma(-1, \infty)\in L$. This completes the proof of
Theorem~\ref{bsdentpth2}.

\vspace{0.2cm} Next, we prove Theorem~\ref{bsdentpth3}. It follows from \eqref{f5tdenac} and \eqref{f5zdenac}  that
the union  of all boundary points of the sets composing the even rational
partition \eqref{f6wdenac} of $\Bb{H}$ is equal to the union $E$ of the roofs $\gamma_{\eufm{n}}(-1,1)$
of $\eusm{E}^{\eufm{n}}_{\!{\tn{\frown}}}$ over all $\eufm{n}\in \Bb{Z}_{\neq 0}^{\hspace{0,02cm}\Bb{N}_{\hspace{-0,02cm}\eurm{f}}}\cup\{0\}$ and their shifts on any even integer. But
 by $-1/\gamma (-1,1)=\gamma (-1,1)$ and \eqref{f6contgen} this yields that $E$ coincides with the orbit
$\{\phi ( \gamma (-1,1) )\}_{\phi \in \Gamma_{\vartheta}}$ of $\gamma (-1,1)$ with respect to $\Gamma_{\vartheta}$. Theorem~\ref{bsdentpth3} is proved.

\subsubsection[\hspace{-0,2cm}. Proof of Lemma~\ref{lemdhp}.]{\hspace{-0,11cm}{\bf{.}} Proof of Lemma~\ref{lemdhp}.}
\label{plemdhp} \
According to the partition \eqref{f3bsdentp}, we have
\begin{align}\label{f0plemdhp} &
     \big\{ z\in   \Bb{H}_{|\re|\leqslant 1} \ | \ \lambda(z)=\lambda(y)\big\}\\  & \hspace{1cm} =
     \left. \left\{ z\in  \mathcal{F}^{\,{\tn{||}}}_{{\tn{\square}}} \ \right| \ \lambda(z)=\lambda(y)\right\}\sqcup     \bigsqcup\nolimits_{\, {\fo{\phi\!\in \!  \Gamma_{\vartheta}^{\hspace{0,015cm}{\tn{||}}}   }}}
     \left.\left\{ z\in \phi \left(\fet^{\,{\tn{|}}\sigma_{\phi}}\right) \ \right| \ \lambda(z)=\lambda(y)\right\} .\nonumber
\end{align}

\noindent Let $\sigma\in \{1,-1\}$ be arbitrarily prescribed and fixed number. It follows from
 Lemma~\ref{lemoneto} that
\begin{align}\label{f1plemdhp}
    &  \left.\left\{ z\in  \mathcal{F}^{\,{\tn{||}}}_{{\tn{\square}}} \ \right| \ \lambda(z)=\lambda(y)\right\}=
\left\{
  \begin{array}{ll}
    y\,, & \ \  \hbox{if} \ \ y \in \fet\,, \\
   \left\{y, y - 2\sigma\right\}\,, & \ \   \hbox{if} \ \ y \in \gamma (\sigma, \infty)\,,\\
    \emptyset , &  \ \  \hbox{if} \ \ y \in \gamma (\sigma, 0)\,.
  \end{array}
\right.
\end{align}

\noindent Fix now an arbitrary
$\phi\!\in \!  \Gamma_{\vartheta}^{\hspace{0,015cm}{\tn{||}}}$, which, in the notation \eqref{f9contgen} and \eqref{f2bsden}, is associated with $N\in \Bb{N}$ and $(n_{N}, ..., n_{1})\in  \Bb{Z}_{\neq 0}^{N}$ such that
$\phi= \phi_{n_{N}, ..., n_{1}}$ and $\sigma_{\phi}= \sigma_{n_{1}}={\rm{sign}} (n_{1}) \in \{1, -1\}$, where $\phi_{n_{N}, ..., n_{1}}$ is  the M{\"o}bius transformation defined as in \eqref{f1bsden}. Since $\phi_{n_{N}, ..., n_{1}}$ is injective on $\Bb{H}$, we get
\begin{align}\label{f2plemdhp}
      \big\{ z\in \phi_{n_{N}, ..., n_{1}} \left(\fet^{\,{\tn{|}}\sigma_{n_{1}}}\right) \   &   \big| \ \lambda(z)=\lambda(y)\big\} \\   &   =
     \left.\left\{  \phi_{n_{N}, ..., n_{1}} \left(\eta\right) \ \right| \ \eta\in\fet^{\,{\tn{|}}\sigma_{n_{1}}} \ , \
     \lambda \left(\phi_{n_{N}, ..., n_{1}} \left(\eta\right)\right)
     =\lambda(y)\right\}  ,
\nonumber \end{align}

\noindent where, in view of  \eqref{f14bsdenac},
\begin{align*}
    & \lambda\left(\phi_{n_{N}, ..., n_{1}} (\eta) \right)=
\left\{\begin{array}{rl}
  \lambda(\eta), &  \  \mbox{if} \   N \in 2 \Bb{N}\,,     \\[0,1cm]
\lambda(-1/\eta),   &   \   \mbox{if} \    N \in 2 \Bb{N}\!-\!1\,,
 \end{array}\right. \  \  \eta \in
\fet \cup \gamma\left(\sigma_{n_{1}}, \infty\right)\,.
\end{align*}

\noindent By \eqref{f2int}, the equation $\lambda(y)=\lambda(-1/\eta)=1-\lambda(\eta)$ is equivalent to $\lambda(\eta)=\lambda(-1/y)$, and  therefore \eqref{f2plemdhp} yields
\begin{align}\label{f3plemdhp} &  E_{n_{N}, ..., n_{1}} (y) := \left.\left\{ z\in \phi_{n_{N}, ..., n_{1}} \left(\fet^{\,{\tn{|}}\sigma_{n_{1}}}\right) \ \right| \ \lambda(z)=\lambda(y)\right\} \\  & =
 \left\{ \phi_{n_{N}, ..., n_{1}} \left(\eta\right) \ \left| \  \eta \in
\fet \cup \gamma\left(\sigma_{n_{1}}, \infty\right), \
\begin{array}{ll}
  \lambda(\eta)=\lambda(y), &  \  \mbox{if} \   N \in 2 \Bb{N}\,,    \\[0,1cm]
\lambda(\eta)=\lambda(-1/y),   &   \   \mbox{if} \    N \in 2 \Bb{N}\!-\!1\,,  \end{array} \right.
\right\}\,.\nonumber
\end{align}

 If $y \in \fet $, then $-1/y \in \fet $ and, in view of Lemma~\ref{lemoneto}, the solutions  of the both equations in the right-hand side of the latter equality are unique. Hence,
\begin{align}
    &   E_{n_{N}, ..., n_{1}} (y) =
\left\{\begin{array}{ll}
  \phi_{n_{N}, ..., n_{1}} (y), &  \  \mbox{if} \   N \in 2 \Bb{N}\,,   \\[0,1cm]
 \phi_{n_{N}, ..., n_{1}} (-1/y),   &   \   \mbox{if} \    N \in 2 \Bb{N}\!-\!1\,,
 \end{array}\right.   \  \quad  y \in \fet\,.
\label{f4plemdhp}\end{align}

 If $y\in \gamma(\sigma, \infty)$ then $\lambda(y)\in\Bb{R}_{<0}$, while  $\lambda(-1/\eta)\in\Bb{R}_{>1}\cup  (0,1)\cup \left(\Bb{C}\setminus\Bb{R}\right)$, by Lemma~\ref{lemoneto} and Theorem~\ref{intthA}.  So that there is no solutions  of the  equation in the right-hand side of
\eqref{f3plemdhp} for odd $N$. But for even $N$ such an equation has the unique solution $\eta=y$, if $\sigma=\sigma_{n_{1}}$, and $\eta=y-2\sigma$, if $\sigma=-\sigma_{n_{1}}$. Or, what is the same   $\eta=y-\sigma+\sigma_{n_{1}}$. Thus,
\begin{align}
    &   E_{n_{N}, ..., n_{1}} (y)\! =\!
\left\{\begin{array}{ll}
   \phi_{n_{N}, ..., n_{1}} (y\!-\!\sigma\!+\!\sigma_{n_{1}}), &  \  \mbox{if} \   N \!\in\! 2 \Bb{N}\,,   \\[0,1cm]
\emptyset,   &   \   \mbox{if} \    N\! \in\! 2 \Bb{N}\!-\!1\,,
 \end{array}\right.   \  \  y\! \in\! \gamma(\sigma, \infty)\,.
\label{f5plemdhp}
\end{align}

 If $y\in \gamma(\sigma, 0)$ then $\lambda(y)\in\Bb{R}_{>1}$, while  $\lambda(\eta)\in\Bb{R}_{<0}\cup  (0,1)\cup \left(\Bb{C}\setminus\Bb{R}\right)$, by Lemma~\ref{lemoneto} and Theorem~\ref{intthA}.  Therefore there is no solutions  of the  equation in the right-hand side of
\eqref{f3plemdhp} for even $N$. But for odd $N$ such an equation has the unique solution $\eta =-1/y $, if $\sigma=-\sigma_{n_{1}}$, and
$\eta =-1/y +2\sigma $, if $\sigma=\sigma_{n_{1}}$. I.e.,  $\eta =-1/y +\sigma+\sigma_{n_{1}}$, and
\begin{align}
    &   E_{n_{N}, ..., n_{1}} (y)\! =\!
\left\{\begin{array}{ll}
    \emptyset , &  \  \mbox{if} \   N \in 2 \Bb{N}\,,   \\[0,1cm]
\phi_{n_{N}, ..., n_{1}} (-1/y\! +\!\sigma\!+\!\sigma_{n_{1}}),   &   \   \mbox{if} \    N \in 2 \Bb{N}\!-\!1\,,
 \end{array}\right.   \  \  y\! \in\! \gamma(\sigma, 0)\,.
\label{f6plemdhp}
\end{align}

\noindent The relationships \eqref{f0plemdhp}, \eqref{f1plemdhp} and  \eqref{f3plemdhp}--\eqref{f6plemdhp} prove \eqref{f1lemdhp}. It follows from  Lemma~{\rm{\ref{bsdentplem1}(e)}}  that each set in \eqref{f1lemdhp} is countable and has no limit points in $\Bb{H}$. This completes the proof of Lemma~\ref{lemdhp}.


\vspace{0,25cm}
\subsection[\hspace{-0,15cm}. \hspace{-0,04cm}Proofs for Section~\ref{evagen}]{\hspace{-0,11cm}{\bf{.}} Proofs for Section~\ref{evagen}}$\phantom{a}$\vspace{-0,1cm}

\subsubsection[\hspace{-0,2cm}. Proof of Lemma~\ref{evagenlem1}.]{\hspace{-0,11cm}{\bf{.}} Proof of Lemma~\ref{evagenlem1}.}
\label{pevagenlem1} \ If
$z\in \eusm{E}^{\infty}_{\!{\tn{\frown}}}$  in \eqref{f2evagen}, then \eqref{f2contgendef1} yields
$ \im \, z=\im \, \Bb{G}_{2} (-1/z) $, $z \in \eusm{E}^{\infty}_{\!{\tn{\frown}}}$, and we deduce from
\eqref{f3bint}(c),  \eqref{f1auxevagen} and \eqref{f12evagen} that
\begin{align*}
    & \dfrac{\left|\Theta_{4}\left(-1/z \right)\right|^{4}}{ |z|^{2}}\, I_{\delta} (-1/z) \!\leqslant \!
\dfrac{\left|\Theta_{2}\left(z \right)\right|^{4} \im^{2} z }{\im^{2} z}\dfrac{147 \pi}{20}
\left(1
\!+ \! \dfrac{1}{\im\, z }\right)\! \leqslant  \ \dfrac{30 \pi}{\im^{2} z}
\left(1\!+ \! \dfrac{1}{\im\, z }\right),
\end{align*}

\noindent because by \cite[p.\! 325]{ber1},
$(147/20)\theta_{3}(e^{-\pi/2})^{4}=(147/20)(\pi/2) (3 + 2 \sqrt{2})\Gamma  (3/4)^{-4}
 < 30$.{\hyperlink{r30}{${}^{\ref*{case30}}$}}\hypertarget{br30}{}
  This proves the first inequality of \eqref{f13evagen}. If in \eqref{f2evagen} we have $z\in \eusm{E}_{\!{\tn{\frown}}}^{\hspace{0,05cm}0}$ then \eqref{f3bint}(a) implies
\begin{align*}
    &  \left|\Theta_{4}\left(z \right)\right|^{4} =
\dfrac{\left|\Theta_{2}\left(-1/z \right)\right|^{4} }{|z|^{2}}=
\dfrac{\left|\Theta_{2}\left(-1/z \right)\right|^{4} \im\, (-1/z)}{\im\, z}
= \dfrac{\left|\Theta_{2}\left(\Bb{G}_{2} (z) \right)\right|^{4} \im\, \Bb{G}_{2} (z)}{\im\, z} \ ,
\end{align*}

\noindent
and we derive from \eqref{f1auxevagen} and \eqref{f12evagen} that
\begin{align*}
    &  \left|\Theta_{4}\left(z \right)\right|^{4} \,I_{\delta} (z)
\leqslant \dfrac{\theta_{3}(e^{-\pi/2})^{4} I_{\delta} (z)}{\left(\im\, z\right)\im\, \Bb{G}_{2} (z)}   \leqslant
\dfrac{30 \pi \big(1+1/\im\, \Bb{G}_{2} (z)\big) }{\left(\im\, z\right)\im\, \Bb{G}_{2} (z)} \ ,
\end{align*}

\noindent which proves the second inequality of \eqref{f13evagen}.
 Finally, if $z\in \eusm{E}^{\eufm{n}}_{\!{\tn{\frown}}} $ in \eqref{f2evagen},  then \eqref{f3bint}(a) yields
\begin{align*}
    &
\im\,\Bb{G}^{N}_{2} (z)\left|\Theta_{4}\big(\Bb{G}^{N}_{2} (z) \big)\right|^{4} =\im\,\Bb{G}^{N}_{2} (z) \left|\Bb{G}^{N}_{2} (z)\right|^{-2}\left|\Theta_{2}\big(-1/\Bb{G}^{N}_{2} (z) \big)\right|^{4}
 \\ & =
\im\,\left(-1/\Bb{G}^{N}_{2} (z)\right)\left|\Theta_{2}\left(-1/\Bb{G}^{N}_{2} (z) \right)\right|^{4}=
\im\,\Bb{G}^{N+1}_{2} (z)\left|\Theta_{2}\left(\Bb{G}^{N+1}_{2} (z)\right)\right|^{4}
\end{align*}

\noindent and we conclude, by \eqref{f1auxevagen} and \eqref{f12evagen}, that the third inequality of \eqref{f13evagen} holds because
\begin{align*}
    &  \dfrac{\left|\Theta_{4}\big(\Bb{G}^{N}_{2} (z) \big)\right|^{4}\, I_{\delta} \big(\Bb{G}^{N}_{2} (z)\big)}{ \left(\im\,z\right)\big/
\im\,\Bb{G}^{N}_{2} (z)}
\leqslant \dfrac{\theta_{3}(e^{-\pi/2})^{4} I_{\delta} \big(\Bb{G}^{N}_{2} (z)\big)}{\left(\im\, z\right)\im\,\Bb{G}^{N+1}_{2} (z)}   \leqslant
\dfrac{30 \pi \big(1+1/\im\, \Bb{G}^{N+1}_{2} (z)\big) }{\left(\im\, z\right)\im\, \Bb{G}^{N+1}_{2} (z)} \ .
\end{align*}


\vspace{0,25cm}
\subsection[\hspace{-0,15cm}. \hspace{-0,04cm}Proofs for Section~\ref{mres}]{\hspace{-0,11cm}{\bf{.}} Proofs for Section~\ref{mres}}$\phantom{a}$\vspace{-0,1cm}

\subsubsection[\hspace{-0,2cm}. Proofs of Lemma~\ref{intreslem3} and Theorem~\ref{intforth1}.]{\hspace{-0,11cm}{\bf{.}} Proofs of Lemma~\ref{intreslem3} and Theorem~\ref{intforth1}.}
\label{pintreslem3} \ To prove Lemma~\ref{intreslem3} for arbitrary $\varepsilon >0$ we introduce the function
\begin{align*}
    &\omega_{\varepsilon} (x) :=  \dfrac{\chi_{(0, \varepsilon)}(x)}{ \kappa \cdot \varepsilon}
 {\rm{e}}^{{\fo{- \dfrac{\varepsilon^{2}}{x(\varepsilon-x)} }}} \in C^{\infty}(\Bb{R}), \quad  x \in \Bb{R} \, , \qquad \int\limits_{0}^{\varepsilon}\omega_{\varepsilon} (x) {\rm{d}} x = 1 \,,
\end{align*}

\noindent where (see \cite[p.\! 658, (1.17), (1.18)]{bak})
 $\max_{t\in [0,1]}\exp (- t^{-1}(1-t)^{-1} ) = {1}/{{\rm{e}}^{4}}$,
\begin{gather*}\hspace{-0.2cm}
  \kappa := \int\nolimits_{0}^{1}
  {\rm{e}}^{{\fo{- \dfrac{1}{t(1-t)} }}  } d t = 2\left[K_{1} (2) - K_{0} (2)\right]/{\rm{e}}^{2} \in \left({0.3838}/{  {\rm{e}}^{4} } \ , \ {0.38382}/{  {\rm{e}}^{4} }\right).
\end{gather*}

\noindent For an arbitrary  $\varphi\in {\rm{H}}^{1}_{+}(\Bb{R})$ let
\begin{align}\label{f0pintreslem3}
    &  \varphi_{\varepsilon} (x) =   \int\nolimits_{\Bb{R}} \omega_{\varepsilon} (t-x) \varphi (t) {\rm{d}} t \ , \quad  x\in \Bb{R} \, , \  \varepsilon >0\,.
\end{align}

 \noindent Then for any $n \in \Bb{Z}_{\geqslant 0}$ the relationships
\begin{align*}  &    \varphi_{\varepsilon}^{(n)} (x) = (-1)^{n}\int\limits_{\Bb{R}} \omega_{\varepsilon}^{(n)} (t-x) \varphi (t) {\rm{d}} t \, , \   \left|  \varphi_{\varepsilon}^{(n)} (x)\right| \leqslant \max\limits_{x \in [0, \varepsilon]} \left|\omega_{\varepsilon}^{(n)} (x)\right| \int\limits_{\Bb{R}} \left|\varphi (t)\right| {\rm{d}} t\,,
   \\    &  \int\limits_{\Bb{R}} \left|  \varphi_{\varepsilon}^{(n)} (x)\right| {\rm{d}} x \leqslant \int\limits_{\Bb{R}}  \int\limits_{\Bb{R}} \left|\omega_{\varepsilon}^{(n)} (t-x)\right| \left|\varphi (t)\right| {\rm{d}} t {\rm{d}} x =\int\limits_{\Bb{R}} \left|\omega_{\varepsilon}^{(n)} (x)\right| {\rm{d}} x \int\limits_{\Bb{R}} \left|\varphi (t)\right| {\rm{d}} t\,,
\end{align*}

\noindent yield that \vspace{-0,2cm}
\begin{align}\label{f1pintreslem3}
    & \varphi^{(n)}_{\,\varepsilon} \in C^{\infty}(\Bb{R}) \cap L_{1}(\Bb{R}) \cap L_{\infty}(\Bb{R})  \ , \quad  n  \in \Bb{Z}_{\geqslant 0} \,.
\end{align}

\noindent Since  by \cite[p.\! 88, 2(iv)]{gar} we have
$\int_{\Bb{R}}\varphi (t) \exp (i x t)  {\rm{d}} t = 0$ for every $x\geqslant 0$ then the identity
\begin{align*}
    &  \int\nolimits_{\Bb{R}} {\rm{e}}^{{\fo{{\rm{i}} y  x}}}    \varphi_{\varepsilon} (x) {\rm{d}} x =  \int\nolimits_{\Bb{R}}
    {\rm{e}}^{{\fo{{\rm{i}} y  x}}} \varphi (x) {\rm{d}} x
    \int\nolimits_{\Bb{R}} {\rm{e}}^{{\fo{-{\rm{i}} y  t}}} \omega_{\varepsilon} (t) {\rm{d}} t = 0 \, , \quad y > 0 \,,
\end{align*}

\noindent implies\vspace{-0,3cm}
\begin{align}\label{f2pintreslem3}
    &   \varphi_{\varepsilon}\in  {\rm{H}}^{1}_{+}(\Bb{R})  \, , \
\end{align}

\noindent by virtue of \eqref{f1pintreslem3} and \cite[p.\! 88, 2(iv)]{gar}. For arbitrary $\delta> 0$ we introduce
the Fourier transform of $\omega_{\delta}$ as follows
\begin{align}\label{f3pintreslem3}
    &   h_{\delta} (x)   :=  \int\nolimits_{{\fo{0}}}^{{\fo{\delta}}}  {\rm{e}}^{{\fo{{\rm{i}} t  x}}} \omega_{\delta} (t) {\rm{d}} t \in {\rm{H}}^{\infty}_{+}(\Bb{R}) \ , \quad  \left| h_{\delta} (z)  \right| \leqslant 1 \,, \ z \in \Bb{H}\cup\Bb{R} \, ,
\end{align}

\noindent where ${\rm{H}}^{\infty}_{+}(\Bb{R})$ denotes the class  of all nontangential limits on  $\Bb{R}$  of the uniformly bounded and holomorphic on $\Bb{H}$ functions. Integration by parts gives
\begin{align*}
    & h_{\delta}^{(n)} (x)   =  \int\nolimits_{{\fo{0}}}^{{\fo{\delta}}}  {\rm{e}}^{{\fo{{\rm{i}} t  x}}} ({\rm{i}} t)^{n}\omega_{\delta} (t) {\rm{d}} t  =
      \dfrac{i^{m}}{x^{m}} \int\nolimits_{{\fo{0}}}^{{\fo{\delta}}}{\rm{e}}^{{\fo{{\rm{i}} t  x}}} \big(({\rm{i}} t)^{n}\omega_{\delta} (t)\big)^{(m)}{\rm{d}} t
       \, , \quad n, m \in \Bb{Z}_{\geqslant 0}  \,,
\end{align*}

\noindent which proves that $h_{\delta} \in \eurm{S} (\Bb{R})$ and we conclude, by taking into account \eqref{f1pintreslem3} and \eqref{f2pintreslem3}, that
\begin{align}\label{f4pintreslem3}
    &  h_{\delta} \cdot \varphi_{\varepsilon} \in {\rm{H}}^1_{+}(\Bb{R}) \cap \eurm{S} (\Bb{R})  \ , \quad  \varepsilon, \delta > 0\,.
\end{align}

\noindent It follows from the inequality $|{\rm{e}}^{z}-1|< 7|z|/4$, $|z|< 1$ (see \cite[p.\! 70, 4.2.38]{abr}) that for arbitrary $\rho > 0$ the following inequality holds
\begin{align}\label{f5pintreslem3}
    & \hspace{-0,4cm}  \left| h_{\delta} (x) \!-\!1 \right|\!\leqslant\! \int\limits_{{\fo{0}}}^{{\fo{\delta}}} \left| {\rm{e}}^{{\fo{{\rm{i}} t  x}}}\!-\!1\right| \omega_{\delta} (t) {\rm{d}} t \!\leqslant\!
    \dfrac{7 \delta }{4 \rho} \int\limits_{{\fo{0}}}^{{\fo{\delta}}}\omega_{\delta} (t) {\rm{d}} t =   \dfrac{7 \delta }{4 \rho} \, , \  \ \delta \leqslant \rho \, , \  |x| \leqslant \dfrac{1}{\rho} \,. \hspace{-0,1cm}
\end{align}

\noindent We introduce the notation
\begin{align*}
    &
    \left\{\begin{array}{l}
       L (a, b) := \left\{\ x \in \Bb{R} \ | \ a \leqslant |x| \leqslant b \ \right\} = [-b,-a]  \cup [a,b]  \,,  \\[0,25cm]
       G (a, b) := \Bb{R}\setminus   L (a, b) = (-a, a) \cup \left\{\ x \in \Bb{R} \ | \ |x| > b \ \right\}\,,
    \end{array}\right.\quad 0 < a < b < +\infty\,.
\end{align*}

\noindent
Now, for arbitrary $\varepsilon \in (0,1/2)$ and  $\delta = \varepsilon^{2} < \rho = \varepsilon$  it follows from \eqref{f3pintreslem3} that
\begin{align*}
    &   \int\limits_{\Bb{R}}
       \left|\varphi (x)\! -\! h_{\delta}(x)  \varphi_{\,\varepsilon} (x)\right| {\rm{d}} x\!  \leqslant
 \hspace{-0,5cm}  \int\limits_{L(\rho,  1/\rho)} \hspace{-0,4cm}       \left|\varphi (x)\! - \!h_{\delta} (x) \varphi_{\,\varepsilon} (x)\right| {\rm{d}} x\! +   \hspace{-0,5cm}
    \int\limits_{G(\rho,  1/\rho)} \hspace{-0,4cm}    \left| \varphi (x)\right| {\rm{d}} x \!+ \hspace{-0,5cm}
     \int\limits_{G(\rho,  1/\rho)} \hspace{-0,4cm}    \left|  \varphi_{\,\varepsilon} (x)\right| {\rm{d}} x\,,
\end{align*}

\noindent where
\begin{align*}
    & \int\limits_{G(\rho,  1/\rho)} \hspace{-0,4cm}    \left|  \varphi_{\,\varepsilon} (x)\right| {\rm{d}} x
    \leqslant   \int\limits_{\Bb{R}}   \int\limits_{\Bb{R}} \left|\omega_{\varepsilon} (t)\right| \left|\varphi (t+x)\right| \chi_{G(\rho,  1/\rho)}(x) \chi_{[0,\varepsilon]}(t){\rm{d}} t {\rm{d}} x = \\ & =
    \int\limits_{\Bb{R}}   \int\limits_{\Bb{R}} \left|\omega_{\varepsilon} (t)\right| \left|\varphi (x)\right| \chi_{G(\rho,  1/\rho)}(x-t) \chi_{[0,\varepsilon]}(t) {\rm{d}} x {\rm{d}} t \leqslant \hspace{-0,4cm}\int\limits_{G(\rho + \varepsilon,  ( 1/\rho) - \varepsilon)} \hspace{-0,4cm}    \left|  \varphi (x)\right| {\rm{d}} x \, , \
    \end{align*}

\noindent and, in view of \eqref{f5pintreslem3},
\begin{align*}
    &  \hspace{-0,5cm}  \int\limits_{L(\rho,  1/\rho)} \hspace{-0,4cm}       \left|\varphi (x) - h_{\delta} (x)  \varphi_{\,\varepsilon} (x)\right| {\rm{d}} x =  \hspace{-0,5cm}  \int\limits_{L(\rho,  1/\rho)} \hspace{-0,4cm}
     \Big|\varphi (x) -\varphi_{\,\varepsilon} (x) +
        \left(1-  h_{\delta} (x)\right) \varphi_{\,\varepsilon} (x)\Big| {\rm{d}} x \\    & \leqslant
      \hspace{-0,5cm}  \int\limits_{L(\rho,  1/\rho)} \hspace{-0,4cm}   \left|\varphi (x) -\varphi_{\,\varepsilon} (x)\right|{\rm{d}} x +
     \dfrac{7 \delta }{4 \rho} \int\limits_{{\fo{\Bb{R}}}}   \left|  \varphi (x)\right| {\rm{d}} x
        \ , \
\end{align*}

\noindent while
\begin{align*}
    &   \int\limits_{L(\rho,  1/\rho)} \hspace{-0,4cm}\!   \left|\varphi (x)\! -\!\varphi_{\,\varepsilon} (x)\right|{\rm{d}} x\!\leqslant\!\!
     \int\limits_{{\fo{\Bb{R}}}} \!\! \left|\varphi (x)\! -\!\varphi_{\,\varepsilon} (x)\right|{\rm{d}} x \!= \! \! \int\limits_{{\fo{\Bb{R}}}}\!\left|\int\limits_{0}^{\varepsilon}\!\omega_{\varepsilon} (t)\left(\varphi (x)\! - \! \varphi (t+x)\right) {\rm{d}} t\right|{\rm{d}} x
     \\    &   \leqslant
\int\limits_{0}^{\varepsilon}\omega_{\varepsilon} (t)     \int\limits_{{\fo{\Bb{R}}}}\left|\varphi (x) - \varphi (t+x)\right|{\rm{d}} x {\rm{d}} t\leqslant
\max\limits_{{\fo{t \in [0, \varepsilon]}}}  \int\limits_{{\fo{\Bb{R}}}}\left|\varphi (x) - \varphi (t+x)\right|{\rm{d}} x \ .
\end{align*}

\noindent So that
\begin{align*}
    &   \int\limits_{\Bb{R}}
       \left|\varphi (x)\! -\! h_{\varepsilon^{2}}(x)  \varphi_{\,\varepsilon} (x)\right| {\rm{d}} x\!  \\    &    \leqslant
   2\hspace{-0,4cm}\int\limits_{G(2 \varepsilon,   1/(2\varepsilon))} \hspace{-0,4cm}    \left|  \varphi (x)\right| {\rm{d}} x +
     \dfrac{7 \varepsilon }{4} \int\limits_{{\fo{\Bb{R}}}}   \left|  \varphi (x)\right| {\rm{d}} x +
     \max\limits_{{\fo{t \in [0, \varepsilon]}}}  \int\limits_{{\fo{\Bb{R}}}}\left|\varphi (x) - \varphi (t+x)\right|{\rm{d}} x \, , \
\end{align*}

\noindent where the right-hand side tends to zero as $\varepsilon \to 0$ because $\varphi \in L^{1}(\Bb{R})$
and the translation is continuous in $ L^{1}(\Bb{R})$ (see  \cite[p.\! 16]{gar}). This together with \eqref{f4pintreslem3}  completes the proof of Lemma~\ref{intreslem3}.

\vspace{0,2cm}To prove Theorem~\ref{intforth1} we notice that in accordance with \eqref{f2intfor}, for each $n \geqslant 1$ the functions $\eusm{R}_{0}$ and $\eusm{R}_{n}$ on the quadrant $\Bb{R}_{\geqslant 0}\times \Bb{R}_{\leqslant 0}$ coincide, correspondingly,
with the functions
\begin{align*}
      &  \dfrac{1}{2} \!\!\! \!\!\! \int\limits_{\gamma (-1,1)} \!\!\!\! \Theta_{3}\left(z\right)^{4} {\rm{e}}^{{\fo{{\rm{i}} x z+ {\rm{i}} y/z}}}   {\rm{ d}} z=
 \dfrac{1}{2{\rm{i}}} \int\limits_{0}^{\pi} \Theta_{3}\left({\rm{e}}^{{\rm{i}}\varphi}\right)^{4} {\rm{e}}^{{\fo{{\rm{i}} x {\rm{ e}}^{{\rm{i}}\varphi}+ {\rm{i}} y{\rm{e}}^{-{\rm{i}}\varphi} + {\rm{i}}\varphi  }}}    {\rm{d}} \varphi
      \ , \  \\
    &  \dfrac{1}{2 \pi n} \!\!\! \!\!\!  \int\limits_{\gamma (-1,1)}\!\!\!\!\!
\left(x- \dfrac{y}{z^{2}}\right){\rm{e}}^{{\fo{{\rm{i}} x z+ {\rm{i}} y/z}}}
S^{{\tn{\triangle}}}_{n} \!\left(\dfrac{1}{\lambda (z)}\right) {\rm{d}} z \\  & =
\dfrac{1}{2 \pi {\rm{i}} n}   \int\limits_{0}^{\pi}
\left(x- y{\rm{e}}^{-2{\rm{i}}\varphi}\right){\rm{e}}^{{\fo{{\rm{i}} x {\rm{e}}^{{\rm{i}}\varphi}+ {\rm{i}} y{\rm{e}}^{-{\rm{i}}\varphi} + {\rm{i}}\varphi  }}}
\left(\sum\limits_{k=1}^{n}\dfrac{s^{{\tn{\triangle}}}_{n, k}}{\lambda ({\rm{e}}^{{\rm{i}}\varphi})^{k}}\right)    {\rm{d}} \varphi\ , \
\end{align*}

\noindent which both are entire functions of two variables $x, y \in \Bb{C}$
because \eqref{f3yint} and \eqref{f3zint} can be written as
\begin{align}\label{f1pintforth1}
    &
    \dfrac{1}{\left|\lambda \left({\rm{e}}^{{\rm{i}} \varphi}\right)\right| } \leqslant 64 \exp \left(-\dfrac{\pi }{\sin  \varphi}\right) , \ \ \
\left|\Theta_{3} \left({\rm{e}}^{{\rm{i}} \varphi}\right)\right|^{4}  \leqslant
 \dfrac{10}{\sin  \varphi}
 \exp \left(-\dfrac{\pi }{2 \sin  \varphi}\right) ,
\end{align}

\noindent for all $0 < \varphi < \pi$. This proves the first assertion of Theorem~\ref{intforth1}.

By setting $y=0$ in  \eqref{f5mres} and in \eqref{f6mres},
 we obtain, in view of \eqref{f2intth1}, \eqref{f3intth1} and \eqref{f6zmres},
 \begin{align*}
     & \eusm{R}_{n}(x, 0)
\!= \! \int\limits_{\Bb{R}}\!  {\rm{e}}^{{\fo{{\rm{i}} x t}}} \eurm{H}_{-n} (t) {\rm{d}} t \!= \!
\int\limits_{\Bb{R}}\!  {\rm{e}}^{{\fo{-{\rm{i}} x t}}} \eurm{H}_{n} (t) {\rm{d}} t \, ,
\  n \geqslant 0 \, ; \  \ \ \eusm{R}_{0}(0, -x) =\eusm{R}_{0}(x, 0)\,;
 \\  &
 \eusm{R}_{n}(0, -x)
\!= \! \int\limits_{\Bb{R}}\!  {\rm{e}}^{{\fo{{\rm{i}} x t}}} \eurm{M}_{-n} (t){\rm{ d}} t = \int\limits_{\Bb{R}}\!  {\rm{e}}^{{\fo{-{\rm{i}} x t}}} \eurm{M}_{n} (t) {\rm{d}} t  \ , \ \
n \geqslant  1 \,, \ \ x \in \Bb{R}\,,
 \end{align*}

\noindent  from which and \eqref{f02int} we derive \eqref{f1intforth1}.

The properties \eqref{f3intforth1} and \eqref{f4intforth1} are immediate
of  \eqref{f6zmres}, \eqref{f3fevagen}, \eqref{f5intfor} and \eqref{f2fevagen}, respectively.

To prove \eqref{f5intforth1} and \eqref{f2intforth1} observe that the following estimates have been derived  after \eqref{f17zint} for $a=2$,
\begin{align}\label{f2pintforth1}
     &\hspace{-0,25cm}
     \left|S^{{\tn{\triangle}}}_{n} \!\left(\dfrac{1}{\lambda (z)}\right)\right| = \left|\eurm{R}_{n}^{{\tn{\triangle}}} (z)\right|\leqslant
     \dfrac{9\, {\rm{e}}^{{\fo{2 \pi n  }}}  }{\pi  |\lambda(z)|}\dfrac{11 + 8\sqrt{2}}{21}\leqslant\dfrac{10 \,{\rm{e}}^{{\fo{2 \pi n  }}}  }{\pi  |\lambda(z)|} \ , \quad  z \in \gamma (-1,1)\,,
\end{align}

\noindent where $n \geqslant 1$ and the value of $1-2 \lambda (2{\rm{i}})$ has been used from \eqref{f2auxevagen}. Then, by using the estimates
\eqref{f3zint}, \eqref{f3yint}, for  the parametrization $\gamma (-1,1) \ni z = (t+{\rm{i}})/(t-{\rm{i}})$, $t > 0$,  we derive from \eqref{f2intfor} and \eqref{f2pintforth1} that
\begin{align*}
     & \hspace{-0,2cm}\left| \eurm{R}_{n}^{{\tn{\triangle}}} \left(\dfrac{t+{\rm{i}}}{t-{\rm{i}}}\right)\right| \!\leqslant\! \dfrac{640 {\rm{e}}^{{\fo{2 \pi n  }}}}{\pi} {\rm{e}}^{{\fo{-\dfrac{\pi }{2}\left(t + \dfrac{1}{t}\right)}}} \, , \
     \left|\Theta_{3} \left(\dfrac{t+{\rm{i}}}{t-{\rm{i}}}\right)\right|^{4}  \leqslant 5 \left(t + \dfrac{1}{t}\right)
{\rm{e}}^{{\fo{-\dfrac{\pi }{4}\left(t + \dfrac{1}{t}\right)}}}  \ ,  \\  &
 \eusm{R}_{n}(x, -y)=\dfrac{1}{ {\rm{i}}\pi n}\int\limits_{0}^{\infty}
\left(x+ y \dfrac{(t-{\rm{i}})^{2}}{(t+{\rm{i}})^{2}}\right)
\eurm{R}_{n}^{{\tn{\triangle}}}\left(\dfrac{t+{\rm{i}}}{t-{\rm{i}}}\right)
{\rm{e}}^{{\fo{{\rm{i}} x \dfrac{t+{\rm{i}}}{t-{\rm{i}}}  - {\rm{i}} y \dfrac{t-{\rm{i}}}{t+{\rm{i}}}  }}}\dfrac{{\rm{d}} t}{(t-{\rm{i}})^{2}} \ ,  \\  &
\eusm{R}_{0}(x, -y)=\dfrac{1}{ {\rm{i}}}\int\limits_{0}^{\infty}
\Theta_{3} \left(\dfrac{t+{\rm{i}}}{t-{\rm{i}}}\right)^{4}{\rm{e}}^{{\fo{{\rm{i}} x \dfrac{t+{\rm{i}}}{t-{\rm{i}}}  - {\rm{i}} y \dfrac{t-{\rm{i}}}{t+{\rm{i}}}  }}}\dfrac{{\rm{d}} t}{(t-{\rm{i}})^{2}}  \, , \quad x, y \geqslant 0 \, , \ t > 0 \, , \ n \geqslant 1 \,.
\end{align*}

\noindent So that
\begin{align}
     &
\begin{array}{l}
 \begin{displaystyle}
  \left|\eusm{R}_{n}(x, -y)\right| \leqslant\dfrac{640 {\rm{e}}^{{\fo{2 \pi n  }}}}{ \pi^{2} n}
     (x+y)\int\limits_{0}^{\infty}
{\rm{e}}^{{\fo{-\dfrac{\pi }{2}\left(t + \dfrac{1}{t}\right) -\dfrac{ 2 t (x+y) }{t^{2} +1} }}}\dfrac{{\rm{d}} t}{t^{2} +1} \ , \
 \end{displaystyle}
    \\
 \begin{displaystyle}
\left|\eusm{R}_{0}(x, -y)\right| \leqslant 5\int\limits_{0}^{\infty}
{\rm{e}}^{{\fo{-\dfrac{\pi }{4}\left(t + \dfrac{1}{t}\right) -\dfrac{ 2 t (x+y) }{t^{2} +1} }}}\dfrac{{\rm{d}} t}{t}\ , \quad x, y \geqslant 0 \, , \ n \geqslant 1 \,,
 \end{displaystyle}
\end{array}
\label{f3pintforth1}\end{align}

\noindent where for arbitrary $a > 0$ we have{\hyperlink{r37}{${}^{\ref*{case37}}$}}\hypertarget{br37}{}
\begin{align}\label{f4pintforth1}
   \begin{array}{l}
 \begin{displaystyle}
\int\limits_{0}^{\infty}
{\rm{e}}^{{\fo{-\dfrac{\pi}{2}\left(t + \dfrac{1}{t}\right) -\dfrac{ 2 a t  }{t^{2} +1} }}}\dfrac{{\rm{d}} t}{t^{2} +1} \leqslant
 \pi^{8} \dfrac{K_{1} \left(2\sqrt{\pi (a+1)}\right)}{\sqrt{a+1}}
\ , \
 \end{displaystyle}
    \\
 \begin{displaystyle}
\int\limits_{0}^{\infty}
{\rm{e}}^{{\fo{-\dfrac{\pi}{4}\left(t + \dfrac{1}{t}\right) -\dfrac{ 2 a t  }{t^{2} +1} }}}\dfrac{{\rm{d}} t}{t}\leqslant
 \pi^{5}
K_{0} \left(\sqrt{2\pi (a+1)}\right) \,.
 \end{displaystyle}
\end{array}
\end{align}

\noindent Taking into account that $640/(\pi^{2} n) <  \pi^{4}$ for all $n \geqslant 1$,
 we derive from
\eqref{f3pintforth1} and \eqref{f4pintforth1} that
\begin{align}
     &
\begin{array}{l}
 \begin{displaystyle}
  \left|\eusm{R}_{n}(x, -y)\right| \leqslant
\pi^{12} {\rm{e}}^{{\fo{2 \pi n  }}}
     \dfrac{(x+y)}{\sqrt{x+y+1}} K_{1} \left(2\sqrt{\pi (x+y+1)}\right) \ , \
 \end{displaystyle}
    \\[0,5cm]
 \begin{displaystyle}
\left|\eusm{R}_{0}(x, -y)\right| \leqslant \pi^{7} K_{0} \left(\sqrt{2 \pi (x+y+1)}\right)\ , \quad x, y \geqslant 0 \, , \ n \geqslant 1 \,,
 \end{displaystyle}
\end{array}
\label{f5pintforth1}\end{align}

\noindent which proves \eqref{f5intforth1}, \eqref{f2intforth1}  and completes the proof of Theorem~\ref{intforth1}.

\addcontentsline{toc}{section}{References}

\vspace{0,3cm}

\vspace{0,5cm}\noindent {\bf{Remark to the References.}} The results of this preprint are based on several assertions which may be found  in \cite{abr}, \cite{and}, \cite{bh2}, \cite{bh1}, \cite[p.\! 665]{bak}, \cite{ber1}, \cite{con},
\cite{erd1}, \cite{gar}, \cite{gra}, \cite[p.\! 1715]{hed1}, \cite{ko}, \cite{lav}, \cite{nat}, \cite{olv},
\cite{rud},  \cite{vlad} and \cite{whi}.
The   other cited  books and articles are listed for completeness.

For example, the monograph \cite{law} is cited on the page \pageref{f3d1int}
in connection with the Landen trans\-for\-mation equations
\eqref{f3d1int}.
However, these equations have self-contained proofs
on the pages \pageref{case38}--\pageref{case1}.
Similarly, in connection with the reference to \cite[p.\! 303]{lop} on the page \pageref{contgendef2} we supply a self-contained proof of \cite[p.\! 303, Proposition 2]{lop}
on the pages \pageref{case39}--\pageref{case25}. Finally, the references to Lemma~2 of \cite[p.\! 112]{cha}
on the pages \pageref{f4contgen} and \pageref{bsdenac} are complemented by a self-contained proof of this lemma on the pages \pageref{case17}--\pageref{case18}.

Moreover, in \cite[p.\! 1715]{hed1} we use only  Proposition 3.7.2 which  establishes the elementary property
 of the operator ${\bf{T}}_{1}$ that it  preserves the properties of functions on $(-1,1)$ to be even and convex. Also,  we need \cite[p.\! 665]{bak}  to write explicitly the value of  the integral of the simplest test function. At last, when using a property of the theta functions needed in the preprint, we refer to   \cite{bh2} and \cite{bh1} because all the basic properties find simple proofs there with the help of  elementary integrals, Liouville's theorem, Morera's theorem, Riemann's theorem about removable singularities,   and  elementary properties of the Schwarz triangle function $\ie$. All the other facts used in the preprint without proofs are from the fundamentals of real and complex  analysis found in, e.g., \cite{abr}, \cite{and},  \cite{ber1}, \cite{con},
\cite{erd1}, \cite{gar}, \cite{gra}, \cite{ko}, \cite{lav}, \cite{nat}, \cite{olv},
\cite{rud},  \cite{vlad} and \cite{whi}.


\newpage

\renewcommand{\thesection}{A}

\section[\hspace{-0,2cm}. \hspace{0,03cm}Supplementary notes]{\hspace{-0,095cm}{\bf{.}} Supplementary notes}\label{detpro}
\setcounter{equation}{0}

%
%
%
%
%
%
\vspace{-0.15cm}
\subsection[\hspace{-0,25cm}. \hspace{0,05cm}Notes for Section~\ref{int}]{\hspace{-0,11cm}{\bf{.}} Notes for Section~\ref{int}}$\phantom{a}$

\vspace{0.1cm}
\begin{subequations}
\begin{clash}{\rm{\hypertarget{r012}{}\label{case012}\hspace{0,00cm}{\hyperlink{br012}{$\uparrow$}}
Let $M \in \Bb{N}$ and
\begin{align*}
     & \mathcal{S}_{M} (x) :=   a_{0}  + \sum_{
     n \in \Bb{Z}\setminus \{0\}, \, |n| \leqslant M}\left(a_{n} {\rm{e}}^{{\fo{\imag \pi n x}}}\! +\!
 b_{n} {\rm{e}}^{{\fo{-\imag \pi n/ x}}} \right)\, ,
    \ \ x\!\in\! \Bb{R} \,.
\end{align*}

\noindent It follows from the two equalities preceding \eqref{f05aint} that for arbitrary $\varphi \in \eurm{S} (\Bb{R})$ and $n \in \Bb{Z}\setminus \{0\}$ we have
\begin{align*}
    & \left| a_{n} \int\limits_{\Bb{R}}\! \varphi (x)\, {\rm{ e}}^{{\fo{{\rm{i}}\pi  n x }}} {\rm{d}} x\right| \leq
    \dfrac{|a_{n} |}{(\pi\,  |n|)^{N+2}}  \int\nolimits_{\Bb{R}}\left| \varphi^{(N+2)} (x)\right| {\rm{d}} x \, , \\   & \left|b_{n}  \int\limits_{\Bb{R}}\! \varphi (x) \,{\rm{e}}^{{\fo{  \dfrac{{\rm{i}} \pi n }{x} }}}  {\rm{d}} x\right| \leq   \dfrac{|b_{n} |}{(\pi\,  |n|)^{N+2}}   \int\nolimits_{\Bb{R}} \left|\left(x^{2}\dfrac{{\rm{d}}}{{\rm{d}} x}\! +\!2x \right)^{\!N+2}\!\!\varphi (x)\right|  {\rm{d}} x \, .
\end{align*}

\noindent Thus, the conditions $a_{n}, b_{n}\!=\! {\rm{O}} (n^{N}), \ n\!\to\! \infty,$ in \eqref{f04int} imply that the sequence
\begin{align*}
    & \int\limits_{\Bb{R}}\! \varphi (x) \mathcal{S}_{M} (x) {\rm{d}} x \ , \quad  M \in \Bb{N} \, ,
\end{align*}

\noindent converges and
\begin{align*}
    &  \mathcal{S}_{\infty}\big( \varphi\big) :=\lim_{ M\to \,\infty}
    \int\limits_{\Bb{R}}\! \varphi (x) \mathcal{S}_{M} (x) {\rm{d}} x \,, \quad \varphi \in \eurm{S} (\Bb{R})\,,
\end{align*}

\noindent satisfies
\begin{align*}
     \left|\mathcal{S}_{\infty}\big( \varphi\big)\right| \leq & |a_{0} |  \int_{\Bb{R}}\!\left| \varphi (x)\right|   {\rm{d}} x+
    \left(\sum_{n \in \Bb{Z}\setminus \{0\} }
  \dfrac{|a_{n} |}{(\pi\,  |n|)^{N+2}} \right) \int\nolimits_{\Bb{R}}\left| \varphi^{(N+2)} (x)\right| {\rm{d}} x \, +
  \\  &
  \left(\sum_{n \in \Bb{Z}\setminus \{0\} } \dfrac{|b_{n} |}{(\pi\,  |n|)^{N+2}}  \right) \int\nolimits_{\Bb{R}} \left|\left(x^{2}\dfrac{{\rm{d}}}{{\rm{d}} x}\! +\!2x \right)^{\!N+2}\!\!\varphi (x)\right|  {\rm{d}} x \,.
\end{align*}

\noindent This means that $\mathcal{S}_{\infty}$
defines a continuous linear functional  on $\eurm{S} (\Bb{R})$ and therefore $\mathcal{S}_{\infty}\in \eurm{S}^{\,\prime} (\Bb{R})$ (see
\cite[p.\! 77]{vlad}). In conclusion,  the series
\begin{align*}
    &   a_{0} + \sum\nolimits_{n \in \Bb{Z}\setminus \{0\}} \!\! \left(a_{n} {\rm{e}}^{{\fo{\imag \pi n x}}}\! +\!
 b_{n} {\rm{e}}^{{\fo{-\imag \pi n/ x}}} \right)
\end{align*}

\noindent converges to $\mathcal{S}_{\infty}\in \eurm{S}^{\,\prime} (\Bb{R})$ in the space $\eurm{S}^{\,\prime} (\Bb{R})$.
}}\end{clash}
\end{subequations}

\vspace{0.1cm}
\begin{subequations}
\begin{clash}{\rm{\hypertarget{r12}{}\label{case12}\hspace{0,00cm}{\hyperlink{br12}{$\uparrow$}}
 \ We prove \eqref{f2intcor1}.
 It follows from \eqref{f3dint}, written in the form
 \begin{align*}
    &  \lambda \left(\dfrac{1+z}{1-z}\right)=\dfrac{1}{2} +
\dfrac{i\left(1-2\lambda (z)\right)}{4  \sqrt{\lambda (z)\left(1-\lambda (z)\right)}} \ , \quad  z \in \fet\,,
 \end{align*}

\noindent and  \eqref{f2aint},  that
 \begin{align*}
    & \lambda \left(\dfrac{1+\ie(z)}{1-\ie(z)}\right)=\dfrac{1}{2} +
\dfrac{i\left(1-2\lambda (\ie(z))\right)}{4  \sqrt{\lambda (\ie(z))\left(1-\lambda (\ie(z))\right)}}=
 \dfrac{1}{2} +
\dfrac{i\left(1-2 z\right)}{4  \sqrt{z\left(1-z\right)}}
 \ , \quad  z \in \Lambda\,,
 \end{align*}

\noindent where $\Lambda:=(0,1)\cup \left(\Bb{C}\setminus\Bb{R}\right)$,  and then, by using
\begin{align*}\tag{\ref{f1case11lem1}}
    &   \phi_{{\nor{(\begin{smallmatrix} 1 & 1 \\ -1& 1\end{smallmatrix})}}}\left(\fet \right) =\fet \ , \quad   \phi_{{\nor{(\begin{smallmatrix} 1 & 1 \\ -1& 1 \end{smallmatrix})}}}\left(z\right) := \dfrac{1+ z }{1- z} \,,
\end{align*}

\noindent we can apply to the latter equality the function $\ie$ to get
\begin{align}\label{f1case12}
    & \fet \ni \dfrac{1+\ie(z)}{1-\ie(z)}= \ie\left(\dfrac{1}{2} +
\dfrac{i\left(1-2 z\right)}{4  \sqrt{z\left(1-z\right)}}\right)\ , \quad  z \in \Lambda\,.
\end{align}

\noindent Applying to \eqref{f1case12} the function $\Theta_{3}^{2}$ and using the squared equality
\eqref{f14inttrian}(b), for any $ z \in \Lambda$    we obtain
\begin{align*}
    &  \he \left(\dfrac{1}{2} +
\dfrac{i\left(1-2 z\right)}{4  \sqrt{z\left(1-z\right)}}\right) =
\Theta_{3}\left( \ie\left(\dfrac{1}{2} +
\dfrac{i\left(1-2 z\right)}{4  \sqrt{z\left(1-z\right)}}\right)\right)^{2}=
\Theta_{3}\left(\dfrac{1+\ie(z)}{1-\ie(z)}\right)^{2}
\end{align*}

\noindent where, by virtue of \eqref{f3bint} and \eqref{f3d1int}(c),
\begin{align*}
    &  \Theta_{3} \left(\dfrac{1+z}{1-z}\right)^{2} = \Theta_{3} \left(\dfrac{2}{1-z} -1\right)^{2}=\Theta_{4}
 \left(\dfrac{2}{1-z}\right)^{2}= \Theta_{3} \left(\dfrac{1}{1-z}\right)\Theta_{4} \left(\dfrac{1}{1-z}\right) \, , \\  &   \Theta_{3} \left(-\dfrac{1}{z-1}\right) = \sqrt{\dfrac{z-1}{i}} \Theta_{3} (z-1) =
\sqrt{\dfrac{z-1}{i}} \Theta_{4} (z)  \, , \\    &
\Theta_{4} \left(-\dfrac{1}{z-1}\right) = \sqrt{\dfrac{z-1}{i}} \Theta_{2} (z-1) =
{\rm{e}}^{-\imag \pi/4} \sqrt{\dfrac{z-1}{i}} \Theta_{2} (z)  \, , \\    &
\Theta_{3} \left(\dfrac{1+z}{1-z}\right)^{2} = {\rm{e}}^{- 3\imag \pi/4} (z-1) \Theta_{4} (z) \Theta_{2} (z) \ , \quad z\in \Bb{H}\,,
\end{align*}

\noindent and therefore
\begin{multline}\label{f2case12}
      \he \left(\dfrac{1}{2} +
\dfrac{i\left(1-2 z\right)}{4  \sqrt{z\left(1-z\right)}}\right)  \\  =
{\rm{e}}^{{\fo{- 3\imag \pi/4}}} \big(\ie(z)-1\big) \Theta_{4} \big(\ie(z)\big) \Theta_{2} \big(\ie(z)\big)\ , \quad  z \in \Lambda\,.
\end{multline}

\noindent But according to \eqref{f1intthA}, \eqref{f14inttrian}(c) and \eqref{f14inttrian}(a),
\begin{align*}
    & \ie(z)=    \dfrac{\imag\, \het (1-z)}{\het (z)}\,, \quad
     \Theta_{4}\big(\ie (z)\big)\!=\! (1\!-\!z)^{1/4}\he  (z)^{1/2}\,, \quad
     \Theta_{2}\big(\ie (z)\big)\!= \!z^{1/4} \he  (z)^{1/2}\,,
\end{align*}

\noindent and, substituting these expressions in \eqref{f2case12}, we get
\begin{align*}
    &   \he  \left(\dfrac{1}{2} +
\dfrac{i\left(1-2z\right)}{4  \sqrt{z(1-z)}}\right)={\rm{e}}^{- 3\imag \pi/4} \left( \dfrac{\imag\, \het (1-z)}{\het (z)}-1\right)
\sqrt[4]{z (1\!-\!z)}\he(z) \\    & =  {\rm{e}}^{- 3\imag \pi/4} \left(i  \he (1-z)-  \he(z)\right)\sqrt[4]{z (1\!-\!z)}=
{\rm{e}}^{ -\imag \pi/4} \left(  \he (1-z)+i  \he(z)\right)\sqrt[4]{z (1\!-\!z)}\,,
\end{align*}

\noindent from which
\begin{align}\label{f3case12}
    & \vspace{-0,3cm} \he  \left(\dfrac{1}{2}\! +\!
\dfrac{i\left(1-2z\right)}{4  \sqrt{z(1-z)}}\right)\!=\!{\rm{e}}^{ -\imag \pi/4}
 \big( \he (1-z)\!+\!i  \he(z)\big)\sqrt[4]{z (1\!-\!z)}\,,\quad  z \in \Lambda\,.
\end{align}

Since $(0,1)\subset \Lambda$, we can set $z=\tau\in (0,1)$ in \eqref{f3case12} and consider the equation
\begin{align*}
    &  t = \dfrac{\left(1-2\tau\right)}{4  \sqrt{\tau(1-\tau)}   } \ \Rightarrow \   \\    &
    16 t^{2} = \dfrac{\left(1-2\tau\right)^{2}}{\tau(1-\tau)} = \dfrac{1-4 \tau + 4\tau^{2}}{\tau(1-\tau)}=
    \dfrac{1}{\tau(1-\tau)} - 4 \ \Rightarrow \ \dfrac{1}{\tau(1-\tau)} = 4 (4 t^{2} +1) \\    &
     \ \Rightarrow \   \dfrac{1}{\sqrt{\tau(1-\tau)}}=2 \sqrt{4 t^{2} +1}  \, , \
\end{align*}

\noindent and therefore
\begin{align*}
    &  t = \dfrac{\left(1-2\tau\right)}{4  \sqrt{\tau(1-\tau)}   } \ \Rightarrow \
     t =  \dfrac{\left(1-2\tau\right)}{2   } \sqrt{4 t^{2} +1} \ \Rightarrow \
     \tau = \dfrac{1}{2} - \dfrac{t}{\sqrt{4 t^{2} +1}} \ .
\end{align*}

\noindent Since
\begin{align*}
      \dfrac{d}{d \tau} \dfrac{\left(1-2\tau\right)}{  \sqrt{\tau(1-\tau)}   }    &   =
     \dfrac{-2  \sqrt{\tau (1-\tau )}  - \left(1-2\tau \right) \dfrac{1-2\tau }{2\sqrt{\tau (1-\tau )}}  }{\tau (1-\tau )} = \\ & =
- \dfrac{4\tau (1-\tau ) + \left(1-2\tau \right)^{2} }{2\tau (1-\tau )\sqrt{\tau (1-\tau )}}=
- \dfrac{4\tau  - 4\tau ^{2} + 1 - 4\tau  + 4\tau^{2}}{2\tau (1-\tau )\sqrt{\tau (1-\tau )}} = \\ & =
- \dfrac{1 }{2\tau (1-\tau )\sqrt{\tau (1-\tau )}}< 0 \ , \quad  \tau \in (0,1)\,,
\end{align*}

\noindent and
\begin{align*}
    &  \lim\limits_{0<\tau\to 0} \dfrac{\left(1-2\tau\right)}{  \sqrt{\tau(1-\tau)}   }  = +\infty  \, ,   \quad    \lim\limits_{1>\tau\to 1}\dfrac{\left(1-2\tau\right)}{  \sqrt{\tau(1-\tau)}   }  = -\infty  \, ,
\end{align*}

\noindent we can replace $z$ in \eqref{f3case12} by
\begin{align*}
    &  z= \dfrac{1}{2} - \dfrac{t}{\sqrt{4 t^{2} +1}} \ \Rightarrow \  1-z =\dfrac{1}{2} + \dfrac{t}{\sqrt{4 t^{2} +1}} \, , \\    &
    z (1-z) = \dfrac{1}{4} - \dfrac{t^{2}}{4 t^{2} +1} = \dfrac{1}{4 \left(4 t^{2} +1\right)} \ , \quad
    \sqrt[4]{z (1-z)} = \dfrac{1}{\sqrt{2} \sqrt[4]{4 t^{2} +1}}
\end{align*}

\noindent and to get
\begin{align*}
    &  \he  \left(\dfrac{1}{2}\! +\!i t
\right)\!=\!{\rm{e}}^{ -\imag \pi/4}
\dfrac{ \he \left(\dfrac{1}{2} + \dfrac{t}{\sqrt{4 t^{2} +1}} \right)\!+\!i
 \he\left(\dfrac{1}{2} - \dfrac{t}{\sqrt{4 t^{2} +1}} \right)}{\sqrt{2} \sqrt[4]{4 t^{2} +1}}
\end{align*}

\noindent where
\begin{align*}
    & {\rm{e}}^{ \imag \pi/4} \sqrt{2} = \dfrac{1+i}{\sqrt{2}} \sqrt{2}= 1+i\,,
\end{align*}

\noindent and hence,
\begin{align}\label{f4case12}
    &  \he  \left(\dfrac{1}{2} +
it\right)=\dfrac{\he \left(\dfrac{1}{2} + \dfrac{ t}{\sqrt{4 t^{2} +1}}\right)+i
\he\left(\dfrac{1}{2} - \dfrac{ t}{\sqrt{4 t^{2} +1}}\right)}{(1+ i)\,\sqrt[4]{ 4 t^{2} +1}} \ , \quad t \in \Bb{R}\,,
\end{align}

\noindent which proves \eqref{f2intcor1}.

}}\end{clash}
\end{subequations}

\vspace{0.1cm}
\begin{subequations}
\begin{clash}{\rm{\hypertarget{r14}{}\label{case14}\hspace{0,00cm}{\hyperlink{br14}{$\uparrow$}}
 \ We prove the new expressions for the integrals of $ \eurm{H}_{n} (x)$, $n\geqslant 0$. Obviously
 \begin{align*}
    &   \eurm{H}_{0} (x)=\dfrac{1}{2 \pi^{2 }}\int\limits_{ - \infty}^{ +  \infty}
 \frac{\he  (1/2 + i t)\he  (1/2 - i t) d t}{  \big(t^{2}+1/4\big)   \Big(
\he  (1/2 - i t)^{2}+  x^{2}\he  (1/2 + i t)^{2}\Big)} \\  &
=\dfrac{1}{2 \pi^{2 }}\int\limits_{\Bb{R}} \frac{
\dfrac{\he  (1/2 - i t)}{\he  (1/2 + i t)} d t}{  \big(t^{2}+1/4\big)   \left(
\dfrac{\he  (1/2 - i t)^{2}}{\he  (1/2 + i t)^{2}}+  x^{2}\right)}\
,  \\    &
\eurm{H}_{n} (x)  =  \dfrac{1}{4 \pi^{3} n}
  \int\limits_{ - \infty}^{ +  \infty}
\dfrac{  S^{{\tn{\triangle}}}_{n} \left(\dfrac{1}{1/2 + i t}\right) d t}{ \big(t^{2}+1/4\big)
\Big( \he  (1/2 - i t) -i  x \he  (1/2 + i t)\Big)^{2} }  \, , \quad  n \geqslant 1\,.
 \end{align*}

\noindent But \eqref{f4case12} written in the form
\begin{align*}
    &  \he  \left(\dfrac{1}{2} +it\right)(1+ i)\,\sqrt[4]{ 4 t^{2} +1} =\he \left(\dfrac{1}{2} + \dfrac{ t}{\sqrt{4 t^{2} +1}}\right)+i \he\left(\dfrac{1}{2} - \dfrac{ t}{\sqrt{4 t^{2} +1}}\right)\ , \\    &
     \he  \left(\dfrac{1}{2} -it\right)(1+ i)\,\sqrt[4]{ 4 t^{2} +1} =\he \left(\dfrac{1}{2} - \dfrac{ t}{\sqrt{4 t^{2} +1}}\right)+i \he\left(\dfrac{1}{2} + \dfrac{ t}{\sqrt{4 t^{2} +1}}\right)\ ,
\end{align*}

\noindent yield
\begin{align*}
     & \dfrac{\he  \left(\dfrac{1}{2} -it\right)}{\he  \left(\dfrac{1}{2} +it\right)}=\dfrac{\he \left(\dfrac{1}{2} - \dfrac{ t}{\sqrt{4 t^{2} +1}}\right)+i \he\left(\dfrac{1}{2} + \dfrac{ t}{\sqrt{4 t^{2} +1}}\right)}{\he \left(\dfrac{1}{2} + \dfrac{ t}{\sqrt{4 t^{2} +1}}\right)+i \he\left(\dfrac{1}{2} - \dfrac{ t}{\sqrt{4 t^{2} +1}}\right)}  = \\ & =
i \dfrac{\he \left(\dfrac{1}{2} + \dfrac{ t}{\sqrt{4 t^{2} +1}}\right)-i \he\left(\dfrac{1}{2} - \dfrac{ t}{\sqrt{4 t^{2} +1}}\right)}{\he \left(\dfrac{1}{2} + \dfrac{ t}{\sqrt{4 t^{2} +1}}\right)+i \he\left(\dfrac{1}{2} - \dfrac{ t}{\sqrt{4 t^{2} +1}}\right)}
\ , \
\end{align*}

\noindent and
\begin{align*}
    & \left[\he  (1/2 - i t) -i  x \he  (1/2 + i t)\right] (1+ i)\sqrt[4]{4 t^{2} +1} =
     \he \left(\dfrac{1}{2}\! -\! \dfrac{ t}{\sqrt{4 t^{2} \!+\!1}}\right)
     \\ &
  \!+\!i
\he\left(\dfrac{1}{2}\! +\! \dfrac{ t}{\sqrt{4 t^{2}\! +\!1}}\right)
\!-\! i x \left(\he \left(\dfrac{1}{2}\! + \!\dfrac{ t}{\sqrt{4 t^{2}\! +\!1}}\right)\!+\!i
\he\left(\dfrac{1}{2}\! -\! \dfrac{ t}{\sqrt{4 t^{2}\! +\!1}}\right)\right) \\    &   =
(1+x)\he \left(\dfrac{1}{2} - \dfrac{ t}{\sqrt{4 t^{2} +1}}\right)+i (1-x)
\he\left(\dfrac{1}{2} + \dfrac{ t}{\sqrt{4 t^{2} +1}}\right) \, ,
\end{align*}

\noindent i.e., for any $t, x \in \Bb{R}$ we have
\begin{align*}
     \he  (1/2 - i t) - & i  x \he  (1/2 +  i t)\! \\    &   = \!\dfrac{
    (1\!+\!x)\he \left(\dfrac{1}{2}\! -\! \dfrac{ t}{\sqrt{4 t^{2}\! +\!1}}\right)\!+\!i (1\!-\!x)
\he\left(\dfrac{1}{2}\! + \!\dfrac{ t}{\sqrt{4 t^{2}\! +\!1}}\right)
    }{(1+ i)\sqrt{2}\sqrt[4]{ t^{2} +1/4}} \, .
\end{align*}

\noindent Substituting these expressions in the above integral for  $\eurm{H}_{0}$ we get
\begin{align*}
      \eurm{H}_{0} (x) &=
\dfrac{i}{2 \pi^{2 }}\int\limits_{\Bb{R}} \dfrac{\dfrac{\he \left({\fo{\dfrac{1}{2} + \dfrac{ t}{\sqrt{4 t^{2} +1}}}}\right)-i \he\left({\fo{\dfrac{1}{2} - \dfrac{ t}{\sqrt{4 t^{2} +1}}}}\right)}{\he \left({\fo{\dfrac{1}{2} + \dfrac{ t}{\sqrt{4 t^{2} +1}}}}\right)+i \he\left({\fo{\dfrac{1}{2} - \dfrac{ t}{\sqrt{4 t^{2} +1}}}}\right)}
 \dfrac{d t}{t^{2}+1/4}}{   x^{2}-
\left(\dfrac{\he \left({\fo{\dfrac{1}{2} + \dfrac{ t}{\sqrt{4 t^{2} +1}}}}\right)-i \he\left({\fo{\dfrac{1}{2} - \dfrac{ t}{\sqrt{4 t^{2} +1}}}}\right)}{\he \left({\fo{\dfrac{1}{2} + \dfrac{ t}{\sqrt{4 t^{2} +1}}}}\right)+i \he\left({\fo{\dfrac{1}{2} - \dfrac{ t}{\sqrt{4 t^{2} +1}}}}\right)}\right)^{\!\!\!\!2} }    \\  &
=\dfrac{i}{2 \pi^{2 }}\int\limits_{\Bb{R}}
\dfrac{\dfrac{1\!-\!i\eurm{y}(t)}{1\!+\!i\eurm{y}(t)}
}{x^{2}- \left(\dfrac{1\!-\!i\eurm{y}(t)}{1\!+\!i\eurm{y}(t)}\right)^{2}}\dfrac{d t}{t^{2}+1/4} \ , \quad
\eurm{y}(t)\! :=\!
\dfrac{ \he\left(\dfrac{1}{2}\! -\! \dfrac{ t}{\sqrt{4 t^{2}\! +\!1}}\right)}{\he \left(\dfrac{1}{2}\! + \!\dfrac{ t}{\sqrt{4 t^{2} \!+\!1}}\right)} \ ,
\end{align*}

\noindent and doing the same  for the above integral for  $\eurm{H}_{n}$ we obtain, taking account of
$((1+ i)\sqrt{2})^{2}= 4 i$,
\begin{align*}
    &  \eurm{H}_{n} (x)  & =  \dfrac{i}{ \pi^{3} n}\int\limits_{\Bb{R}} \dfrac{  S^{{\tn{\triangle}}}_{n} \left(\dfrac{1}{1/2 + i t}\right) \dfrac{d t}{\sqrt{t^{2}+1/4}}}{
\left((1\!+\!x)\he \left(\dfrac{1}{2} \!- \!\dfrac{ t}{\sqrt{4 t^{2}\! +\!1}}\right)\!+\!i (1\!-\!x)
\he\left(\dfrac{1}{2} \!+\! \dfrac{ t}{\sqrt{4 t^{2} \!+\!1}}\right)\right)^{2} } \, \ \
\end{align*}

\noindent for all $x\in \Bb{R}$ and $n \geqslant 1$.

}}\end{clash}
\end{subequations}

\vspace{0.2cm}
\subsection[\hspace{-0,25cm}. \hspace{0,05cm}Notes for Section~\ref{inttrian}]{\hspace{-0,11cm}{\bf{.}} Notes for Section~\ref{inttrian}}

\begin{subequations}
\begin{clash}{\rm{\hypertarget{r13}{}\label{case13}\hspace{0,00cm}{\hyperlink{br13}{$\uparrow$}}
 \  We prove \eqref{f14inttrian}.
 It follows actually  from  \cite[p.\! 598, (1.25); p.\! 599, (1.32)]{bh2} that
\begin{align}\label{f1case13}
    &  \hspace{-0,2cm}
    \begin{array}{ll}
  {\rm{(a)}}\	  \Theta_{2}\big(\ie (z)\big)^{4}\!= \!z \he  (z)^{2}\!  ,  &
 \qquad
     {\rm{(b)}}\	   \Theta_{3}\big(\ie (z)\big)^{2}\!= \!\he  (z)\, ,
     \\[0,1cm]
 {\rm{(c)}}\	  \Theta_{4}\big(\ie (z)\big)^{4}\!=\! (1\!-\!z)\he  (z)^{2}\!  , &
\phantom{{\rm{(c)}}a}
 \qquad z\in (0,1)\cup \left(\Bb{C}\setminus\Bb{R}\right).
    \end{array}\  \  \hspace{-0,1cm}
\end{align}

\noindent Since (see \cite[p.\! 598, (1.26)]{bh2})
\begin{align}\label{f0case13}
    &  \Theta_{3}(z)\Theta_{4}(z)\Theta_{2}(z)\neq 0 \ , \quad  z \in \Bb{H}\,,
\end{align}

\noindent for the principal branches of the quadratic and  of the fourth roots it makes possible to consider three
functions
\begin{align*}
    & \omega_{2}(z) :=\dfrac{z^{1/4} \he  (z)^{1/2}}{\Theta_{2}\big(\ie (z)\big)} \, , \
    \omega_{4}(z) :=\dfrac{(1-z)^{1/4} \he  (z)^{1/2}}{\Theta_{4}\big(\ie (z)\big)} \, , \
    \omega_{3}(z) :=\dfrac{ \he  (z)^{1/2}}{\Theta_{3}\big(\ie (z)\big)} \,,
\end{align*}

\noindent which are holomorphic in $\Lambda:=(0,1)\cup \left(\Bb{C}\setminus\Bb{R}\right)$ and in accordance with
\eqref{f1case13},
\begin{align}\label{f2case13}
    &  \omega_{2}(z),  \omega_{4}(z) \in \left\{ \exp\left(\dfrac{ \pi i k }{2}\right) \right\}_{k=0}^{3} \ , \quad  \omega_{3}(z) \in \left\{-1, 1\right\} \ , \quad  z \in \Lambda \,.
\end{align}

\noindent Therefore
\begin{align*}
    &  \Lambda = \underset{0\leqslant k \leqslant 3}{\sqcup }  \Lambda_{k}^{m} \, , \  m \in \{2, 4\} \ , \quad
\Lambda = \underset{0\leqslant k \leqslant 1}{\sqcup }  \Lambda_{k}^{3} \, ,
     \\    &
\Lambda_{k}^{m} := \left\{ z \in \Lambda \,\left|\,  \omega_{m}(z) = \exp\left(\dfrac{ \pi i k }{2}\right)\right\} , \quad 0\leqslant k \leqslant 3 \, , \  m \in \{2, 4\} \right. ,  \\    &
\Lambda_{k}^{3} := \left\{ z \in \Lambda \,\left|\,  \omega_{3}(z) = (-1)^{k}\right\} \right., \
k \in \left\{0, 1\right\},
\end{align*}

\noindent Here each of the set $\Lambda_{k}^{m}$ is relatively closed in $\Lambda$, i.e.,
$K \cap \Lambda_{k}^{m}$ is closed for every compact subset $K$ of $\Lambda$,  because
$\omega_{2}$, $\omega_{3}$, $\omega_{4} \in {\rm{Hol}} (\Lambda)$.
 In view of the Baire Category Theorem \cite[p.\! 159]{roy}, for each $m \in \{2,3,4\}$ there exists
a number $q_{m}$ such that  $q_{3} \in \{0,1\}$,   $q_{m} \in \{0,1,2,3\}$, if $m \in \{2, 4\}$, and
the set $\Lambda_{q_{m}}^{m}$ contains some neighborhood of at least one interior point of $\Lambda$.
Applying the uniqueness theorem  for analytic functions (see \cite[p.\! 78]{con}) we obtain
$\Lambda_{q_{m}}^{m}= \Lambda$ for all  $m \in \{2,3,4\}$ and therefore
\begin{align}\label{f3case13}
    &  \omega_{m}(z) = \exp\left(\dfrac{ \pi i q_{m} }{2}\right) \, , \  m \in \{2, 4\} \ , \quad
    \omega_{3}(z) = (-1)^{q_{3}}  \ , \quad  z\in \Lambda \,.
\end{align}

\noindent By setting here $z \in (0,1)$ we deduce  that $q_{m}=0$ for each  $m \in \{2,3,4\}$,
because,  in view of \eqref{f0pinttheor1}, $\ie ((0,1)) = i \Bb{R}$, and
 $\Theta_{m}( i \Bb{R}) \subset \Bb{R}_{>0}$, $\he ((0,1))\subset \Bb{R}_{>0}$, $z, 1-z \in (0,1)$ for arbitrary
 $z \in (0,1)$.  Therefore all equalities \eqref{f14inttrian} hold and the proof of
 \eqref{f14inttrian} is completed.

 }}\end{clash}
\end{subequations}

\begin{subequations}
\begin{clash}{\rm{\hypertarget{r38}{}\label{case38}\hspace{0,00cm}{\hyperlink{br38}{$\uparrow$}}
 \ We prove \eqref{f3d1int}. In view of the uniqueness theorem  for analytic functions (see  \cite[p.\! 78]{con}), it is enough to prove that
\begin{align}
    &
    \begin{array}{ll}
     \Psi_{1}(z)\! :=\!\dfrac{\Theta_{3}( z)^{2}\!-\! \Theta_{4}( z)^{2}}{2 \Theta_{2}(2 z)^{2}}\! =\! 1 \, ,   &   \  \Psi_{2}(z) \!:=\!\dfrac{\Theta_{3}( z)^{2}\!+ \!\Theta_{4}( z)^{2}}{2 \Theta_{3}(2 z)^{2}}\! =\! 1 \, , \ \\[0,5cm]
     \Psi_{3}(z)\! :=\!  \dfrac{\Theta_{4}(2 z)^{2}}{\Theta_{3}( z)\Theta_{4}( z)}\! =\! 1 \, ,   &   \  z\in \fet\,,
    \end{array}\label{f1case38}
\end{align}
\noindent because $ \Psi_{j}\in {\rm{Hol}}(\Bb{H})$, $1\leqslant j\leqslant 3$, as follows from \eqref{f0case13}.

We apply the approach suggested in
\cite[p.\! 24]{bh1}. According to this approach, it is necessary to introduce  the following three  functions
\begin{align}
    &
    \Phi_{j} (z)\! :=\!\Psi_{j}\big(\ie (z)\big) \ , \
    1\leqslant j\leqslant 3 \ , \  \  z\in \Lambda:=(0,1)\cup \left(\Bb{C}\setminus\Bb{R}\right) ,
    \label{f2case38}
\end{align}

\noindent which are holomorphic in $\Lambda$ and $\ie (\Lambda)=\fet$. And
then  to study the values of
$\{\Phi_{j}\}_{j=1}^{3}$ on the  two sides of the cuts along $\Bb{R}_{< 0}$
and along $\Bb{R}_{> 1}$, and their behavior near the points $\{0,1,\infty\}$.

We first look on their values on the  two sides of the cut along $\Bb{R}_{< 0}$.
Since by \cite[p.\! 19, (6.8)]{bh2} we have
\begin{align*}
    &   \Theta_4(z+1)= \Theta_3(z) \, , \  \Theta_3(z+1) = \Theta_4(z) \, , \  \Theta_2(z+1)= {\rm{e}}^{\imag \pi/4}\Theta_2(z) \ , \quad  z\in \Bb{H} \, , \
\end{align*}

\noindent which yield (see \cite[p.\! 56, (A.18a)]{bh2})
\begin{align*}
    &  \begin{array}{llll}
      \Theta_4(z+2)= \Theta_4(z) \, , &\  \Theta_3(z+2) = \Theta_3(z) \, , \
         &\ \Theta_2(z+2)= i\Theta_2(z)  \, , \   &    \\
        \Theta_4(z+4)= \Theta_4(z) \, , &\  \Theta_3(z+4) = \Theta_3(z) \, , \
            &\ \Theta_2(z+4)= -\Theta_2(z) \, , \    & z\in \Bb{H} \, , \
       \end{array}
\end{align*}

\noindent then
\begin{align*}
      \Phi_{1} (2+z)& =\dfrac{\Theta_{3}\big(2 +  z\big)^{2}\!-\! \Theta_{4}\big(2 +  z\big)^{2}}{2 \Theta_{2}\big(4 +  2z\big)^{2}}   =
      \dfrac{\Theta_{3}\big(  z\big)^{2}\!-\! \Theta_{4}\big( z\big)^{2}}{2 \Theta_{2}\big(  2z\big)^{2}}= \Phi_{1} (z) \, , \   \\
     \Phi_{2} (2+z) & =\dfrac{\Theta_{3}\big(2 + z\big)^{2}\!+\! \Theta_{4}\big(2 +  z\big)^{2}}{2 \Theta_{3}\big(4 +  2z\big)^{2}}
     =  \dfrac{\Theta_{3}\big(z\big)^{2}\!+\! \Theta_{4}\big( z\big)^{2}}{2 \Theta_{3}\big(  2z\big)^{2}}=\Phi_{2} (z) \, , \   \\
     \Phi_{3} (2+z) & = \dfrac{\Theta_{4}\big(4 +  2z\big)^{2}}{\Theta_{3}\big(2 +  z\big)\Theta_{4}\big(2 + z\big)}
     =\dfrac{\Theta_{4}\big(  2z\big)^{2}}{\Theta_{3}\big(  z\big)\Theta_{4}\big(  z\big)}= \Phi_{3} (z) \, , \   \quad z\in \Bb{H}\,,
\end{align*}

\noindent and we derive from \eqref{f4intthA},
\begin{align*}
    &  \ie (-x +  \imag  0 ) =
    2 +  \ie (-x - \imag  0 ) \ , \quad  x>0\,,
\end{align*}

\noindent that for every $x>0$ we have
\begin{align} \label{f3acase38}
    &
     \Phi_{1} \big(-x +  \imag  0  \big) = \Phi_{1} \big(-x - \imag  0  \big)  \, , \ \
    \Phi_{2} \big(-x +  \imag  0  \big) = \Phi_{2} \big(-x - \imag  0  \big)  \, ,  \\    &
     \Phi_{3} \big(-x +  \imag  0  \big) = \Phi_{3} \big(-x - \imag  0  \big)  \,.
\label{f3bcase38}\end{align}

\noindent
 Applying to these relationships the Morera theorem   (see \cite[p.\! 96]{lav})  we obtain that
 \begin{align}\label{f3ccase38}
     &  \Phi_{j}\in {\rm{Hol}}\big(
     \Bb{C}\setminus\left(\{0\}\cup [1, +\infty\right)\big),  \ \
     1\leqslant j\leqslant 3,
 \end{align}

\noindent and therefore $0$ could be a point of an isolated singularity  for all functions  $\{\Phi_{j}\}_{j=1}^{3}$ (see \cite[p.\! 103]{con}).
 We prove that actually $0$ is a point of a removable singularity  for
these functions. To prove this, we  apply the Riemann theorem about removable singularities (see \cite[p.\! 103]{con}), according to which $0$ is a point of a removable singularity of $\Phi_{j}$, where
$1\leqslant j\leqslant 3$,
if and only if
$\lim_{z\to 0} z\Phi_{j}(z)=0$. For any $1\leqslant j\leqslant 3$ the latter equality obviously follows from more strong property, $\sup_{2 z\in \Bb{D}\setminus\{0\} } |\Phi_{j}(z)| = \sup_{2 z\in \Bb{D}\setminus (-1,0] } |\Phi_{j}(z)| < +\infty$, and this in turn is equivalent to
$\limsp_{z\in \Bb{D}\setminus (-1,0] \, , \ z\to 0}|\Phi_{j}(z)|< +\infty$,
because $\Phi_{j}\in {\rm{Hol}} (\Bb{D}\setminus\{0\})$ and therefore $\Phi_{j}$ is continuous on $\Bb{D}\setminus\{0\}$.
 According to \cite[p.\! 609, (4.2), (4.5), (4.6)]{bh2},
\begin{align*}
    & (0,1)\cup \left(\Bb{C}\setminus\Bb{R}\right)\ni z \to 0 \      \Rightarrow      \  \fet \ni \ie (z)\to \infty \      \Rightarrow      \ \im \ \ie (z)\to +\infty   \, ,
\end{align*}

\noindent and since
\begin{align*}
    &  \Theta_{3}(z)^{2}-\Theta_{4}(z)^{2}= \left(1  \!+ \! 2\sum\limits_{n\geqslant 1} {\rm{e}}^{\imag  \pi n^2 z} \!-\!1\!-\! 2\sum\limits_{n\geqslant 1}(-1)^{n}{\rm{e}}^{\imag  \pi n^2 z} \right)\times \\   &  \times \left(1  \!+ \! 2\sum\limits_{n\geqslant 1} {\rm{e}}^{\imag  \pi n^2 z} \!+\!1\!+\! 2\sum\limits_{n\geqslant 1}(-1)^{n}{\rm{e}}^{\imag  \pi n^2 z} \right)  \\ & =
    \left(4{\rm{e}}^{\imag  \pi  z}+ 4\sum\limits_{n\geqslant 1} {\rm{e}}^{\imag  \pi (2n+1)^2 z} \right)
    \left(2 +4\sum\limits_{n\geqslant 1}{\rm{e}}^{4   \pi \imag n^2 z}\right) = 8 {\rm{e}}^{\imag  \pi  z}+ {\rm{O}} \left({\rm{e}}^{ 4 \pi \imag z}\right) , \\    &
    \Theta_{3}(z) = 1 +  {\rm{O}} \left({\rm{e}}^{  \imag \pi z}\right)  , \
    \Theta_{4}(z) = 1 +  {\rm{O}} \left({\rm{e}}^{  \imag \pi z}\right) , \
    \Theta_{2}(z) = 2 {\rm{e}}^{\imag  \pi  z/4} +  {\rm{O}} \left({\rm{e}}^{ 5 \pi\imag  z/4}\right)  ,
\end{align*}

\noindent as $\im \, z\to +\infty  $,  we obtain the existence of the following limits
\begin{align*}  &
 \lim_{{{\Bb{D}\setminus (-1,0]  \ni z \to 0}}} \Phi_{1} (z) =
    \lim_{{{\Bb{D}\setminus (-1,0]  \ni z \to 0}}}
    \dfrac{\Theta_{3}\big(2 \ie (z)\big)^{2}\!-\! \Theta_{4}\big(2 \ie (z)\big)^{2}}{2 \Theta_{2}\big(2 \ie (z)\big)^{2}}\\ & =
    \lim_{{{\fet \ni z \to \infty}}}  \dfrac{\Theta_{3}( z)^{2}\!-\! \Theta_{4}( z)^{2}}{2 \Theta_{2}(2 z)^{2}} =
    \lim_{{{\im \, z\to +\infty }}} \dfrac{ 8 {\rm{e}}^{\imag  \pi  z}+ {\rm{O}} \left({\rm{e}}^{ 4 \pi \imag z}\right)  }{2\left( 2 {\rm{e}}^{\imag  \pi  z/2} +  {\rm{O}} \left({\rm{e}}^{ 5 \pi\imag  z/2}\right)\right)^{2}}
    =1\,,
 \\[0,3cm]
 &
 \lim_{{{\Bb{D}\setminus (-1,0]  \ni z \to 0}}} \Phi_{2} (z) =
    \lim_{{{\Bb{D}\setminus (-1,0]  \ni z \to 0}}}
    \dfrac{\Theta_{3}\big(2 \ie (z)\big)^{2}\!+\! \Theta_{4}\big(2 \ie (z)\big)^{2}}{2 \Theta_{3}\big(2 \ie (z)\big)^{2}}\\ & =
    \lim_{{{\fet \ni z \to \infty}}}  \dfrac{\Theta_{3}( z)^{2}\!+\! \Theta_{4}( z)^{2}}{2 \Theta_{3}(2 z)^{2}} =     \lim_{{{\im \, z\to +\infty }}} \dfrac{ \left(1 +  {\rm{O}} \left({\rm{e}}^{  \imag \pi z}\right)\right)^{2}
    + \left(1 +  {\rm{O}} \left({\rm{e}}^{  \imag \pi z}\right)\right)^{2}}{2\left(1 +  {\rm{O}} \left({\rm{e}}^{  \imag \pi z}\right)\right)^{2}}
    =1\,,
 \\[0,3cm]
    &  \lim_{{{\Bb{D}\setminus (-1,0]  \ni z \to 0}}} \Phi_{3} (z) =
    \lim_{{{\Bb{D}\setminus (-1,0]  \ni z \to 0}}} \dfrac{\Theta_{4}\big(2 \ie (z)\big)^{2}}{\Theta_{3}\big(\ie (z)\big)\Theta_{4}\big(\ie (z)\big)}\\ & =
    \lim_{{{\fet \ni z \to \infty}}}  \dfrac{\Theta_{4}(2 z)^{2}}{\Theta_{3}( z)\Theta_{4}( z)}   =
    \lim_{{{\im \, z\to +\infty }}} \dfrac{\left(1 +  {\rm{O}} \left({\rm{e}}^{ 2 \pi\imag  z}\right)\right)^{2}}{\left(1 +  {\rm{O}} \left({\rm{e}}^{  \imag \pi z}\right)\right)
    \left(1 +  {\rm{O}} \left({\rm{e}}^{  \imag \pi z}\right)\right)}
    =1\,.
\end{align*}

\noindent Thus,
 \begin{align}\label{f4case38}
     &  \Phi_{j}\in {\rm{Hol}}\big(
     \Bb{C}\setminus [1, +\infty)\big),  \ \ \Phi_{j} (0) =1  \, , \   \ , \quad
     1\leqslant j\leqslant 3,
 \end{align}

We now  look on the values of $\{\Phi_{j}\}_{j=1}^{3}$ on the  two sides of the cut along $\Bb{R}_{> 1}$.
For arbitrary $z\in \Bb{H}$ and $\sigma \in \{1, -1\}$, by using the identities,
\begin{align*}
     &
     \begin{array}{lll}
      \Theta_{3} (z)   \!=\!\imag\, \dfrac{ \Theta_{3}\left(\dfrac{z}{1\!-\!2z}\right)}{(2z\!-\!1)^{1/2} } \,   \,  , & \
  \Theta_{4} (z)  \!=\!\, \,\dfrac{\Theta_{4}\left(\dfrac{z}{1\!-\!2z}\right)}{(2z\!-\!1)^{1/2} } \, \, , &\
   \Theta_{2} (z)   \!=\!\imag\, \dfrac{ \Theta_{2}\left(\dfrac{z}{1\!-\!2z}\right)}{(2z\!-\!1)^{1/2} } \, ,
             \\
              \Theta_{3} (z)^{2}   \!=\!\,\!-\!\, \dfrac{ \Theta_{3}\left(\dfrac{z}{1\!-\!2z}\right)^{2}}{2z\!-\!1 } \,   \,  , & \
  \Theta_{4} (z)^{2}  \!=\!\, \,\dfrac{\Theta_{4}\left(\dfrac{z}{1\!-\!2z}\right)^{2}}{2z\!-\!1 } \, \, , & \
   \Theta_{2} (z)^{2}   \!=\! \dfrac{ \Theta_{2}\left(\dfrac{z}{1\!-\!2z}\right)^{2}}{1\!-\!2z } \, ,
             \\
   \Theta_{3} (z)^{2}\!=\!  \, \imag \, \dfrac{\Theta_{2} \left( \dfrac{1}{\sigma \!-\!z}\right)^{2}}{z\!-\!\sigma}   \, \, , &\
    \Theta_{4} (z)^{2}\!=\!  \, \imag \, \dfrac{\Theta_{3} \left( \dfrac{1}{\sigma \!-\!z}\right)^{2}}{z\!-\!\sigma}
   \, \, , &\
    \Theta_{2} (z)^{2}\!=\!   \dfrac{\Theta_{4} \left(
    \dfrac{1}{\sigma \!-\!z}\right)^{2}}{\sigma \!-\!z} \ , \
     \end{array}
\end{align*}

\noindent (see   \cite[p.\! 57, (A.18b)(b), (A.18b)(c); p.58, (A.18e)(c)]{bh1}) we obtain, for every $z\in \Bb{H}$,
\begin{align*}
     \Psi_{1}(z) & \! =\!\dfrac{\Theta_{3}( z)^{2}\!-\! \Theta_{4}( z)^{2}}{2 \Theta_{2}(2 z)^{2}}\! =\!
\dfrac{\,\!-\!\, \dfrac{ \Theta_{3}\left(\dfrac{z}{1\!-\!2z}\right)^{2}}{(2z\!-\!1) }\!-\!\dfrac{\Theta_{4}\left(\dfrac{z}{1\!-\!2z}\right)^{2}}{(2z\!-\!1) } }{\, \!-\! \, 2\dfrac{\Theta_{4} \left( \dfrac{1}{1 \!-\!2z}\right)^{2}}{2z\!-\!1}} \\ & \!=\!
\dfrac{\,\!-\!\,  \Theta_{3}\left(\dfrac{z}{1\!-\!2z}\right)^{2}\!-\!\Theta_{4}\left(\dfrac{z}{1\!-\!2z}\right)^{2} }{\, \!-\! \, 2\Theta_{4} \left(1 \!-\!1+ \dfrac{1}{1 \!-\!2z}\right)^{2}} \!=\!
\dfrac{\,\!-\!\,  \Theta_{3}\left(\dfrac{z}{1\!-\!2z}\right)^{2}\!-\!\Theta_{4}\left(\dfrac{z}{1\!-\!2z}\right)^{2} }{\, \!-\! \, 2\Theta_{3} \left( \!-\!1+ \dfrac{1}{1 \!-\!2z}\right)^{2}} \\ & \!=\!
\dfrac{ \Theta_{3}\left(\dfrac{z}{1\!-\!2z}\right)^{2}+\Theta_{4}\left(\dfrac{z}{1\!-\!2z}\right)^{2}}{2\Theta_{3} \left( \dfrac{2z}{1 \!-\!2z}\right)^{2}}\!=\! \Psi_{2}\left(\dfrac{z}{1\!-\!2z}\right)\, ; \
       \\[0,4cm]
     \Psi_{2}(z) & \! =\!\dfrac{\Theta_{3}( z)^{2}\!+\! \Theta_{4}( z)^{2}}{2 \Theta_{3}(2 z)^{2}}\! =\!
\dfrac{\,\!-\!\, \dfrac{ \Theta_{3}\left(\dfrac{z}{1\!-\!2z}\right)^{2}}{(2z\!-\!1) }\!+\!\dfrac{\Theta_{4}\left(\dfrac{z}{1\!-\!2z}\right)^{2}}{(2z\!-\!1) } }{\,  2\imag \, \dfrac{\Theta_{2} \left( \dfrac{1}{1 \!-\!2z}\right)^{2}}{2z\!-\!1}}  \\ & \!=\!
\dfrac{\,\!-\!\,  \Theta_{3}\left(\dfrac{z}{1\!-\!2z}\right)^{2}\!+\!\Theta_{4}\left(\dfrac{z}{1\!-\!2z}\right)^{2} }{\,  2\imag \,\Theta_{2} \left(1 \!-\!1+ \dfrac{1}{1 \!-\!2z}\right)^{2}} \!=\!
\dfrac{\,\!-\!\,  \Theta_{3}\left(\dfrac{z}{1\!-\!2z}\right)^{2}\!+\!\Theta_{4}\left(\dfrac{z}{1\!-\!2z}\right)^{2} }{2\imag \cdot\,\imag \,\Theta_{2} \left( \!-\!1+ \dfrac{1}{1 \!-\!2z}\right)^{2}} \\ & \!=\!
\dfrac{ \Theta_{3}\left(\dfrac{z}{1\!-\!2z}\right)^{2}\!-\!\Theta_{4}\left(\dfrac{z}{1\!-\!2z}\right)^{2}}{2\Theta_{2} \left( \dfrac{2z}{1 \!-\!2z}\right)^{2}} \!=\! \Psi_{1}\left(\dfrac{z}{1\!-\!2z}\right)\, ; \
        \\[0,4cm]
      \Psi_{3}(z) & \! =\!
   \dfrac{\Theta_{4} (2z)^{2}}{\Theta_{3} (z) \Theta_{4} (z)}\!=\!
   \dfrac{i  \, \dfrac{\Theta_{3} \left( \dfrac{1}{1\!-\!2z}\right)^{2}}{2z\!-\!1} }{
  i\, \dfrac{ \Theta_{3}\left(\dfrac{z}{1\!-\!2z}\right)}{(2z\!-\!1)^{1/2} }
  \dfrac{\Theta_{4}\left(\dfrac{z}{1\!-\!2z}\right)}{(2z\!-\!1)^{1/2} } }\!=\!
  \dfrac{\Theta_{3} \left( \dfrac{1}{1\!-\!2z}\right)^{2}}{\Theta_{3}\left(\dfrac{z}{1\!-\!2z}\right)\Theta_{4}\left(\dfrac{z}{1\!-\!2z}\right)} \\ & \!=\!  \dfrac{\Theta_{3} \left( 1\!-\!1+ \dfrac{1}{1\!-\!2z}\right)^{2}}{\Theta_{3}\left(\dfrac{z}{1\!-\!2z}\right)\Theta_{4}\left(\dfrac{z}{1\!-\!2z}\right)} \!=\!\dfrac{\Theta_{4} \left( \!-\!1+ \dfrac{1}{1\!-\!2z}\right)^{2}}{\Theta_{3}\left(\dfrac{z}{1\!-\!2z}\right)\Theta_{4}\left(\dfrac{z}{1\!-\!2z}\right)}
    \\ & \!=\!  \dfrac{\Theta_{4} \left( \dfrac{2z}{1\!-\!2z}\right)^{2}}{\Theta_{3}\left(\dfrac{z}{1\!-\!2z}\right)
   \Theta_{4}\left(\dfrac{z}{1\!-\!2z}\right)}\!=\! \Psi_{3}\left(\dfrac{z}{1\!-\!2z}\right) \ , \
   \end{align*}

\noindent and derive from \eqref{f5intthA},
\begin{align*}
    & \hspace{-0,2cm} \ie (1 \!+\! x \!- \!\imag  0 )\! =\!
\dfrac{\ie (  1\!+ \!x \!+ \! \imag  0 )}{ 1\!- \!2 \ie (  1\!+ \!x\! +\!
  \imag  0 ) }\,  , \
 2\ie (1\! +\! x \!- \!\imag  0 )\! =\!
\dfrac{2\ie (  1\!+ \!x\! + \! \imag  0 )}{ 1\!-\! 2 \ie (  1\!+\! x \!+\!
  \imag  0 ) } \,, \
 x\!>\!0\,,
\end{align*}

\noindent that
\begin{align} \label{f4acase38}
    &\hspace{-0,2cm} {\rm{(1)}} \ \,  \Phi_{1}\big(1 \!+ \!x\! - \!\imag  0\big)\!=\!\Phi_{2}\big(1 \!+\! x \!+\! \imag  0\big)  \, , \
    \Phi_{2}\big(1\! +\! x\! - \!\imag  0\big)\!=\!\Phi_{1}\big(1\! + \!x \!+ \!\imag  0\big)  \, , \ x>0,    \\    &\hspace{-0,2cm} {\rm{(2)}} \ \,
    \Phi_{3}\big(1\! +\! x\! -\! \imag  0\big)\!=\!\Phi_{3}\big(1 \!+ \!x \!+ \!\imag  0\big)
      \, , \ x>0\,.
\label{f4bcase38}\end{align}

\noindent
 Applying the Morera theorem   (see \cite[p.\! 96]{lav}) to these relationships   we obtain
 \begin{align}\label{f5case38}
     & \Phi_{3}\in {\rm{Hol}}\big(
     \Bb{C}\setminus\{1\}\big), \ \  \Phi_{0}\in {\rm{Hol}}\big(
     \Bb{C}\setminus\{0\}\big),
 \end{align}

\noindent where the function
\begin{align}\label{f6case38}
    &   \Phi_{0}(z):=\left\{\begin{array}{ll}
                        \Phi_{1}\big(1+z^{2}\big) \, , \   &  z\in \Bb{H} \, , \  \\
                        \Phi_{2}\big(1+z^{2}\big) \, , \   &  z\in -\Bb{H} \, , \
                       \end{array}
   \right.
\end{align}

\noindent is holomorphic in $\Bb{C}\setminus  \Bb{R}$ as follows from \eqref{f4case38}, theorem about analyticity of the composition of two holomorphic functions (see \cite[p.\! 34]{con}),  and the fact that $1+z^{2}$ maps
conformally  $\pm\Bb{H}$ onto $ \Bb{C}\setminus [1, +\infty)$. The latter property is the consequence of the fact that the function $\pm i \sqrt{1-\zeta}$ maps
conformally $\Bb{C}\setminus [1, +\infty)$ onto $\pm \Bb{H}$ (see \cite[p.\! 46]{con}).
Therefore the inverse mapping $1+z^{2}$ maps  $\pm\Bb{H} $ one-to-one onto $\Bb{C}\setminus [1, +\infty)$.
Then $1+z^{2}= 1+(x+ \imag  y)^{2}=1+x^{2}- y^{2}+2 \imag  x y $, $z := x+ \imag y$, $x, y \in \Bb{R}$ and \eqref{f6case38}, for arbitrary $x>0$ imply
\begin{align*}
    &  \left\{\begin{array}{l}
                \Phi_{0}\big( x + \imag  0\big)=\Phi_{1}\big(1+x^{2}+ \imag  0\big) \, , \\
                \Phi_{0}\big( x - \imag  0\big)=\Phi_{2}\big(1+x^{2}- \imag  0\big) \, ,
              \end{array}
    \right. \ \
    \left\{\begin{array}{l}
               \Phi_{0}\big( -x + \imag  0\big)=\Phi_{1}\big(1+x^{2}- \imag  0\big) \, ,  \\
                \Phi_{0}\big( -x - \imag  0\big)=\Phi_{2}\big(1+x^{2}+ \imag  0\big) \, ,
              \end{array}
    \right.
\end{align*}

\noindent which by \eqref{f4acase38} means that
\begin{align*}
    &   \Phi_{0}\big( x + \imag  0\big)=\Phi_{1}\big(1+x^{2}+ \imag  0\big) = \Phi_{2}\big(1+x^{2}- \imag  0\big) =\Phi_{0}\big( x - \imag  0\big) \, ,   &     & \ \  x > 0,  \\    &
    \Phi_{0}\big( -x + \imag  0\big)=\Phi_{1}\big(1+x^{2}- \imag  0\big) =\Phi_{2}\big(1+x^{2}+ \imag  0\big)=
    \Phi_{0}\big(-x - \imag  0\big) \, ,  &     & \ \  x > 0,
\end{align*}

\noindent
and hence the Morera theorem   can be applied to $\Phi_{0}\in {\rm{Hol}}(
     \Bb{C}\setminus\Bb{R})$ to get $\Phi_{0}\in {\rm{Hol}}(
     \Bb{C}\setminus\{0\})$. The right-hand side inclusion of \eqref{f5case38} is completely proved.

We prove now that $1$ and  $0$ are the points of a removable singularity  for $\Phi_{3}$ and $\Phi_{0}$, respectively.   According to \cite[p.\! 609, (4.2), (4.4)]{bh2},
\begin{align}\label{f7case38}
    &\hspace{-0,25cm} (0,1)\cup \left(\Bb{C}\setminus\Bb{R}\right)\!\ni \!z \!\to \!1 \   \Rightarrow   \  \fet\! \ni\! \ie (z)\!\to \!0   \   \Rightarrow   \  \im  \big(-1/ \ie (z)\big)\!\to\! +\!\infty   \, ,
\hspace{-0,1cm}\end{align}

\noindent because in view of $-1/\fet=\fet \subset \Bb{H}_{\re <1}$, it follows from  $\fet \ni \ie (z)\to 0$ that $-1/\ie (z)\in \fet$ and $|1/\ie (z)|\to +\infty$, which means that $ \im  (-1/ \ie (z)) \geqslant |1/\ie (z)| - 1 \to +\infty$. Then the asymptotic equalities (see \eqref{f3bint}(a),(b),(c)),
\begin{align*}
    &  \Theta_3(z)  =\dfrac{\Theta_3(-1/z)}{(z/\imag )^{1/2}} \ , \
     \Theta_2(z)  =\dfrac{\Theta_4(-1/z)}{(z/\imag )^{1/2}} \ , \
      \Theta_4 (z)  =\dfrac{\Theta_2(-1/z)}{(z/\imag )^{1/2}} \ , \  \\  &
    \Theta_3(-1/z)  =1 +  {\rm{O}} \left({\rm{e}}^{\imag \pi ( -1/ z)}\right)  \, , \
\Theta_4(-1/z) =1+ {\rm{O}}\left({\rm{e}}^{ \imag \pi ( -1/ z)}\right) \, , \\  &
   \Theta_2(-1/z)=  2 {\rm{e}}^{  \imag\pi (-1/ z) /4}\left( 1 +  {\rm{O}} \left({\rm{e}}^{  \imag  \pi ( -1/ z)}\right)\right)\, ,
\end{align*}

\noindent as  $z\in \fet$, $z\to 0$ (and hence, $\im (-1/z) \to +\infty$) yield
\begin{align*}
    &  \lim_{{\fo{\fet\! \ni\! z\! \to \!0}}}  \Psi_{1}(z)  \! =\!\lim_{{\fo{\fet\! \ni\! z\! \to \!0}}} \dfrac{\Theta_{3}( z)^{2}\!-\! \Theta_{4}( z)^{2}}{2 \Theta_{2}(2 z)^{2}}\! =\!\lim_{{\fo{\fet\! \ni\! z\! \to \!0}}}
    \dfrac{\dfrac{\Theta_3(-1/z)^{2}}{(z/\imag )}\!-\! \dfrac{\Theta_2(-1/z)^{2}}{(z/\imag )}  }{2
    \dfrac{\Theta_4(-1/(2z))^{2}}{(2 z/\imag )}    } \\
    &     =
  \lim_{{\fo{\fet\! \ni\! z\! \to \!0}}}  \dfrac{\Theta_3(-1/z)^{2}\!-\!\Theta_2(-1/z)^{2}}{\Theta_4(-1/(2z))^{2}}
   \\    &
  =
  \lim_{{{\im (-1/z)\to +\infty }}}  \dfrac{\left(1 +  {\rm{O}} \left({\rm{e}}^{  \imag
     \pi ( -1/ z)}\right)\right)^{2}- 4 {\rm{e}}^{  \imag\pi (-1/ z) /2}
\left( 1 +  {\rm{O}} \left({\rm{e}}^{  \imag
     \pi ( -1/ z)}\right)\right)^{2}}{\left(1 +  {\rm{O}} \left({\rm{e}}^{  \imag
     \pi ( -1/(2 z))}\right)\right)^{2}} = 1  \, ,
     \\
      &  \lim_{{\fo{\fet\! \ni\! z\! \to \!0}}}  \Psi_{2}(z)  \! =\!\lim_{{\fo{\fet\! \ni\! z\! \to \!0}}} \dfrac{\Theta_{3}( z)^{2}\!+\! \Theta_{4}( z)^{2}}{2 \Theta_{3}(2 z)^{2}}\! =\!\lim_{{\fo{\fet\! \ni\! z\! \to \!0}}}
    \dfrac{\dfrac{\Theta_3(-1/z)^{2}}{(z/\imag )}\!+\! \dfrac{\Theta_2(-1/z)^{2}}{(z/\imag )}  }{2
    \dfrac{\Theta_3(-1/(2z))^{2}}{(2 z/\imag )}    } \\ & =
  \lim_{{\fo{\fet\! \ni\! z\! \to \!0}}}  \dfrac{\Theta_3(-1/z)^{2}\!+\!\Theta_2(-1/z)^{2}}{\Theta_3(-1/(2z)^{2}}
   \\    &
  =
  \lim_{{{\im (-1/z)\to +\infty }}}  \dfrac{\left(1 +  {\rm{O}} \left({\rm{e}}^{  \imag
     \pi ( -1/ z)}\right)\right)^{2}+ 4 {\rm{e}}^{  \imag\pi (-1/ z) /2}
\left( 1 +  {\rm{O}} \left({\rm{e}}^{  \imag
     \pi ( -1/ z)}\right)\right)^{2}}{\left(1 +  {\rm{O}} \left({\rm{e}}^{  \imag
     \pi ( -1/(2 z))}\right)\right)^{2}} = 1  \, , \    \\
      &  \lim_{{\fo{\fet\! \ni\! z\! \to \!0}}}  \Psi_{3}(z)  \! =\!
   \lim_{{\fo{\fet\! \ni\! z\! \to \!0}}}   \dfrac{\Theta_{4}(2 z)^{2}}{\Theta_{3}( z)\Theta_{4}( z)}
   \! =\! \lim_{{\fo{\fet\! \ni\! z\! \to \!0}}}
   \dfrac{\dfrac{\Theta_2(-1/(2z))^{2}}{(2z/\imag )}}{
   \dfrac{\Theta_3(-1/z)}{(z/\imag )^{1/2}}
   \dfrac{\Theta_2(-1/z)}{(z/\imag )^{1/2}}
   }
    \\ & = \lim_{{\fo{\fet\! \ni\! z\! \to \!0}}}
 \dfrac{\Theta_2(-1/(2z))^{2}}{2
   \Theta_3(-1/z)\Theta_2(-1/z)
   }
   \\ & =
       \lim_{{{\im (-1/z)\to +\infty }}}
      \dfrac{4 {\rm{e}}^{  \imag\pi (-1/ 2z) /2}
\left( 1 +  {\rm{O}} \left({\rm{e}}^{  \imag
     \pi ( -1/ (2z))}\right)\right)^{2}}{2\left(1 +  {\rm{O}} \left({\rm{e}}^{  \imag
     \pi ( -1/ z)}\right)\right)
          2 {\rm{e}}^{  \imag\pi (-1/ z) / 4}
\left( 1 +  {\rm{O}} \left({\rm{e}}^{  \imag
     \pi ( -1/ z)}\right)\right)}=1\,.
\end{align*}

\noindent Hence, by \eqref{f7case38}, we have, for $\Lambda := (0,1)\cup \left(\Bb{C}\setminus\Bb{R}\right)$,
\begin{align}\label{f8case38}
    & \lim_{{\fo{\Lambda\! \ni\! z\! \to \!1}}} \Phi_{j} (z) =
    \lim_{{\fo{\Lambda\! \ni\! z\! \to \!1}}} \Psi_{j}\big(\ie (z)\big)=
    \lim_{{\fo{\fet\! \ni\! z\! \to \!0}}} \Psi_{j}\big(z\big)
     =1 \, , \ 1\leqslant j\leqslant 3.
\end{align}

\noindent In view of the continuity of $\Phi_{0}$ on $\Bb{C}\setminus\{0\}$, as follows from
\eqref{f5case38}, the limit of $\Phi_{0}$ as $\Bb{C}\setminus\{0\}\ni z \to 0$ exists if and only if there exists the limit
of $\Phi_{0}$ as $\Bb{C}\setminus\Bb{R}\ni z \to 0$, and we derive from \eqref{f8case38},
 $\Bb{C}\setminus\Bb{R}=(-\Bb{H}) \sqcup\Bb{H}$ and $(1+z^{2}): \pm \Bb{H} \to \Lambda$,
 the existence of the following limit
\begin{align*}
    &  \lim_{{\fo{\Bb{C}\!\setminus\!\{0\}\! \ni\! z\! \to \!0}}} \Phi_{0} (z) = \lim_{{\fo{\Bb{C}\!\setminus\!\Bb{R}\! \ni\! z\! \to \!0}}} \Phi_{0} (z) \\    &    =\lim_{{\fo{\Bb{C}\!\setminus\!\Bb{R}\! \ni\! z\! \to \!0}}}
   \big[ \Phi_{1}\big(1+z^{2}\big)\chi_{\Bb{H}} (z) + \Phi_{2}\big(1+z^{2}\big)\chi_{-\Bb{H}} (z) \big]
=  \lim\limits_{{\fo{\Lambda\! \ni\! z\! \to \!1}}} \left\{\begin{array}{l}\Phi_{1} (z) \\ \Phi_{2} (z) \end{array}
  \right\}=1.
\end{align*}

\noindent Together with \eqref{f8case38} for $j=3$   and \eqref{f5case38}  this relationship yields that
$1$ and  $0$ are the points of a removable singularity  for $\Phi_{3}$ and $\Phi_{0}$, respectively, and so
$\Phi_{3}$ and $\Phi_{0}$ are entire functions, i.e.,
 \begin{align}\label{f9case38}
     & \Phi_{3}, \   \Phi_{0}\in {\rm{Hol}}\big(\Bb{C}\big) \, , \quad \Phi_{0}(0)= \Phi_{3}(1)=1.
 \end{align}

We prove now  that the  modulus  of $\Phi_{3}$ and $\Phi_{0}$ are uniformly bounded on $\Bb{C}$ by establishing the existence of the finite limits of  $\Phi_{3}(z)$ and $\Phi_{0}(z)$ as $\Bb{C}\ni z \to \infty$.

Let $\Lambda\ni z\! \to\! \infty$ approaching from  one of the half-planes $\sigma\! :=\! {\rm{sign}} (\im z)\! \in \!\{1, -1\}$. Then $\fet \ni \ie (z) \to \sigma$ (see \cite[p.\! 609, (4.2), (4.3)]{bh2})
and, in view of \eqref{f13inttrian}, ${\rm{sign}} (\re \,\ie (z))= \sigma$. The latter equality yields $\fet \ni \ie (z)-\sigma \! \to\! 0$ (see \cite[p.\! 24, item 2]{bh1})  and by
manipulations  similar to those employed in the proof of \eqref{f7case38}, for arbitrary $\sigma\! \in \!\{1, -1\}$ we get
\begin{align}\label{f10case38}
    &\hspace{-0,25cm} \sigma\, \Bb{H}\ni z\! \to\! \infty \ \Rightarrow \
    \fet \ni \ie (z)-\sigma \! \to\! 0\ \Rightarrow \
    \im  \big(-1/\left( \ie (z)-\sigma\right)\big)\!\to\! +\!\infty   \,.
\hspace{-0,1cm}
\end{align}

\noindent By using the relationships (see \cite[p.\! 58, (A.18e), (A.18d)]{bh1},
\begin{align*}
     & \Theta_{3} (z)^{2}\!=\!  \, \imag \, \dfrac{\Theta_{2} \left( \dfrac{1}{\sigma \!-\!z}\right)^{2}}{z\!-\!\sigma}   \, \, , \
    \Theta_{4} (z)^{2}\!=\!  \, \imag \, \dfrac{\Theta_{3} \left( \dfrac{1}{\sigma \!-\!z}\right)^{2}}{z\!-\!\sigma}
   \, \, , \
    \Theta_{2} (z)^{2}\!=\!   \dfrac{\Theta_{4} \left(
    \dfrac{1}{\sigma \!-\!z}\right)^{2}}{\sigma \!-\!z} \ ,  \\  &
  \Theta_{3} (z)^{2}\!=\!  \, \imag \,\dfrac{\Theta_{3}\left(-\dfrac{1}{z}\right)^{2}}{z}
    \, \, , \hspace{0,4cm}
    \Theta_{4} (z)^{2}\!=\!  \, \imag \,\dfrac{\Theta_{2}\left(-\dfrac{1}{z}\right)^{2}}{z}  \, \, , \hspace{0,4cm}
    \Theta_{2} (z)^{2}\!=\!  \, \imag \,\dfrac{\Theta_{4}\left(-\dfrac{1}{z}\right)^{2}}{z}   \, \, ,
    \\  &
      \Theta_{3} (z)\Theta_{4} (z)=
   \, \imag \, \dfrac{\Theta_{2} \left( \dfrac{1}{\sigma \!-\!z}\right)\Theta_{3} \left( \dfrac{1}{\sigma \!-\!z}\right)}{z\!-\!\sigma}  \ , \quad  z \in \Bb{H}\,, \  \ \sigma \in \{1,-1\}\,,
\end{align*}

\noindent we obtain, for arbitrary $\sigma\! \in \!\{1, -1\}$,  the existence of the following limits
\begin{align*}
    &   \lim_{{{\sigma \Bb{H}\ni z \to \infty}}} \Phi_{1} (z) =\!\!\!\!\!
    \lim_{{{\sigma \Bb{H}\ni z \to \infty}}}\Psi_{1}\big(\ie (z)\big)
   =\!\!\!\!\!
    \lim_{{{\fet \ni z-\sigma  \to 0}}}\Psi_{1}\big(z\big) =\!\!\!\!\!
    \lim_{{{ \fet \ni z-\sigma  \to 0}}}  \dfrac{\Theta_{3}( z)^{2}\!-\! \Theta_{4}( z)^{2}}{2 \Theta_{2}(2 z)^{2}}\\[0,2cm] & =
   \lim_{{{ \fet \ni z-\sigma  \to 0}}}  \dfrac{\Theta_{3}( z)^{2}\!-\! \Theta_{4}( z)^{2}}{
   2 \Theta_{2}\big(2(z \!-\!\sigma) + 2 \sigma\big)^{2}} =
   \lim_{{{ \fet \ni z-\sigma  \to 0}}}  \dfrac{\Theta_{3}( z)^{2}\!-\! \Theta_{4}( z)^{2}}{
   -2\Theta_{2}\big(2(z \!-\!\sigma)\big) ^{2}}
       \\[0,2cm]     &   =
    \lim_{{{ \fet \ni z-\sigma  \to 0}}}
     \dfrac{\imag \, \dfrac{\Theta_{2} \left(- \dfrac{1}{z \!-\!\sigma}\right)^{2}}{z\!-\!\sigma}\!-\!
     \imag \, \dfrac{\Theta_{3} \left(- \dfrac{1}{z \!-\!\sigma}\right)^{2}}{z\!-\!\sigma}}{- 2\imag \,
     \dfrac{\Theta_{4} \left(- \dfrac{1}{2(z \!-\!\sigma)}\right)^{2}}{2(z\!-\!\sigma)}      }=
     \lim_{{{\fet \ni z \to \infty}}} \dfrac{\Theta_{3}(z)^{2} - \Theta_{2}(z)^{2}}{\Theta_{4}(z/2)^{2} } \\[0,2cm]     &   =  \lim_{{{\im \, z\to +\infty }}} \dfrac{\left( 1 +  {\rm{O}} \left({\rm{e}}^{  \imag \pi z}\right)\right)^{2}
    -\left(2 {\rm{e}}^{\imag  \pi  z/4} +  {\rm{O}} \left({\rm{e}}^{ 5 \pi\imag  z/4}\right)\right)^{2}
    }{ \left(1 +  {\rm{O}} \left({\rm{e}}^{  \imag \pi z/2}\right)\right)^{2}}=1\,;
\end{align*}
\begin{align*}
    &   \lim_{{{\sigma \Bb{H}\ni z \to \infty}}} \Phi_{2} (z) =\!\!\!\!\!
    \lim_{{{\sigma \Bb{H}\ni z \to \infty}}}\Psi_{2}\big(\ie (z)\big)
   =\!\!\!\!\!
    \lim_{{{\fet \ni z-\sigma  \to 0}}}\Psi_{2}\big(z\big) =\!\!\!\!\!
    \lim_{{{ \fet \ni z-\sigma  \to 0}}}  \dfrac{\Theta_{3}( z)^{2}\!+\! \Theta_{4}( z)^{2}}{2 \Theta_{3}(2 z)^{2}}\\[0,2cm]  & = \lim_{{{ \fet \ni z-\sigma  \to 0}}}  \dfrac{\Theta_{3}( z)^{2}\!+\! \Theta_{4}( z)^{2}}{2 \Theta_{3}\big(2(z \!-\!\sigma) + 2 \sigma\big)^{2}}=
    \lim_{{{ \fet \ni z-\sigma  \to 0}}}  \dfrac{\Theta_{3}( z)^{2}\!+\! \Theta_{4}( z)^{2}}{
    2 \Theta_{3}\big(2(z \!-\!\sigma)\big)^{2}}  \\    &   =
    \lim_{{{ \fet \ni z-\sigma  \to 0}}}
     \dfrac{\imag \, \dfrac{\Theta_{2} \left(- \dfrac{1}{z \!-\!\sigma}\right)^{2}}{z\!-\!\sigma}\!+\!
     \imag \, \dfrac{\Theta_{3} \left(- \dfrac{1}{z \!-\!\sigma}\right)^{2}}{z\!-\!\sigma}}{2\imag \,
     \dfrac{\Theta_{3} \left(- \dfrac{1}{2(z \!-\!\sigma)}\right)^{2}}{2(z\!-\!\sigma)}      }=
     \lim_{{{\fet \ni z \to \infty}}} \dfrac{\Theta_{2}(z)^{2} + \Theta_{3}(z)^{2}}{\Theta_{3}(z/2)^{2} } \\[0,2cm]     &
     =\lim_{{{\im \, z\to +\infty }}}
     \dfrac{\left(2 {\rm{e}}^{\imag  \pi  z/4} +  {\rm{O}} \left({\rm{e}}^{ 5 \pi\imag  z/4}\right)\right)^{2}+\left( 1 +  {\rm{O}} \left({\rm{e}}^{  \imag \pi z}\right)\right)^{2}
        }{ \left(1 +  {\rm{O}} \left({\rm{e}}^{  \imag \pi z/2}\right)\right)^{2}}=1\,;
\end{align*}
\begin{align*}
    &   \lim_{{{\sigma \Bb{H}\ni z \to \infty}}} \Phi_{3} (z) =\!\!\!\!\!
    \lim_{{{\sigma \Bb{H}\ni z \to \infty}}}\Psi_{3}\big(\ie (z)\big)
   =\!\!\!\!\!
    \lim_{{{\fet \ni z-\sigma  \to 0}}}\Psi_{3}\big(z\big) =\!\!\!\!\!
    \lim_{{{ \fet \ni z-\sigma  \to 0}}}  \dfrac{\Theta_{4} (2z)^{2}}{\Theta_{3} (z) \Theta_{4} (z)}\\[0,2cm]  & =
     \lim_{{{ \fet \ni z-\sigma  \to 0}}}  \dfrac{\Theta_{4} \big(2(z \!-\!\sigma) + 2 \sigma\big)^{2}}{\Theta_{3} (z) \Theta_{4} (z)}=
     \lim_{{{ \fet \ni z-\sigma  \to 0}}}  \dfrac{\Theta_{4} \big(2(z \!-\!\sigma)  \big)^{2}}{\Theta_{3} (z) \Theta_{4} (z)} \\    &   =
     \lim_{{{ \fet \ni z-\sigma  \to 0}}}\dfrac{\imag \,
     \dfrac{\Theta_{2} \left(- \dfrac{1}{2(z \!-\!\sigma)}\right)^{2}}{2(z\!-\!\sigma)} }{  \imag \, \dfrac{\Theta_{2} \left( \dfrac{1}{\sigma \!-\!z}\right)\Theta_{3} \left( \dfrac{1}{\sigma \!-\!z}\right)}{z\!-\!\sigma}  }=\lim_{{{\fet \ni z \to \infty}}} \dfrac{\Theta_{2}(z/2)^{2}}{
     2\Theta_{2}(z)  \Theta_{3}(z) } \\[0,2cm]     &=
     \lim_{{{\im \, z\to +\infty }}} \dfrac{\left(2 {\rm{e}}^{\imag  \pi  z/8} +  {\rm{O}} \left({\rm{e}}^{ 5 \pi\imag  z/8}\right)\right)^{2}}{2 \left(2 {\rm{e}}^{\imag  \pi  z/4} +  {\rm{O}} \left({\rm{e}}^{ 5 \pi\imag  z/4}\right)\right)\left(1 +  {\rm{O}} \left({\rm{e}}^{  \imag \pi z/2}\right)\right)}=1\,.
\end{align*}

\noindent Hence, for arbitrary $\sigma\! \in \!\{1, -1\}$ there exist the limits
\begin{align}\label{f11case38}
    &   \lim_{{{\sigma \Bb{H}\ni z \to \infty}}} \Phi_{j} (z) = 1 \ , \quad  1 \leqslant j \leqslant 3 \,.
\end{align}

 Since by \eqref{f9case38} we have $\Phi_{3}\in {\rm{Hol}}(\Bb{C})$ then
 it follows from \eqref{f11case38} with $j=3$ that
 $ \Phi_{3} (z)  \to 1$ as
$\Bb{H}\ni z \to \infty$ and we obtain that the entire function $\Phi_{3}$ is bounded on $\Bb{C}$ while $\Phi_{3}(1)=1$, also
according to \eqref{f9case38}. The Liouville theorem \cite[p.\! 77]{con} yields $\Phi_{3}(z) = 1$ for every $z\in \Bb{C}$  and, in particular, $\Phi_{3}(z) = 1$ for every $z\in \Lambda$. Then $\ie (\Lambda)=\fet$
and \eqref{f2case38} imply the validity of \eqref{f1case38} for $\Psi_{3}$ and completes the proof of
\eqref{f3d1int}(c).

Since $1+z^{2} \in \Bb{R}$ if and only if $z \in (\imag \Bb{R}) \cup \Bb{R}$ and, in  accordance with \eqref{f9case38}, we have $\Phi_{0}\in {\rm{Hol}}(\Bb{C})$,
then  the limit of $ \Phi_{0} (z)$ as
$\Bb{C}\ni z \to \infty$ exists if and only if there exists the limit of  $ \Phi_{0} (z)$ as
$\Bb{C}\setminus( (\imag \Bb{R}) \cup \Bb{R}) \ni z \to \infty$. But in the latter case $1+z^{2} \in \Bb{C}\setminus \Bb{R}$ and $1+z^{2} \to \infty$. In view of the definition \eqref{f6case38} of
$ \Phi_{0}$ and the properties \eqref{f11case38}, we conclude that $ \Phi_{0} (z)  \to 1$ as
$\Bb{C}\ni z \to \infty$ and therefore the entire function $\Phi_{0}$ is bounded on $\Bb{C}$ while $\Phi_{0}(1)=1$,
according to \eqref{f9case38}. The Liouville theorem \cite[p.\! 77]{con} gives $\Phi_{0}(z)=1$ for every $z\in \Bb{C}$, which by \eqref{f6case38} yield that
$ \Phi_{1}\big(1+z^{2}\big)=1$ for every $z\in \Bb{H}$ and
$ \Phi_{2}\big(1+z^{2}\big)=1$ for every $z\in -\Bb{H}$. This implies that
$\Phi_{1}(z) = \Phi_{2}(z) = 1$ for every $z\in \Lambda$ because $1+z^{2}$ maps
conformally  $\pm\Bb{H}$ onto $\Lambda:= \Bb{C}\setminus [1, +\infty)$.  Then $\ie (\Lambda)=\fet$
and \eqref{f2case38} yield the validity of \eqref{f1case38} for $\Psi_{1}$ and $\Psi_{2}$ which completes the proof of \eqref{f3d1int}(a) and \eqref{f3d1int}(b). The  Landen trans\-for\-mation equations \eqref{f3d1int} have been completely proved.
 }}\end{clash}
\end{subequations}

\begin{subequations}
\begin{clash}{\rm{\hypertarget{r1}{}\label{case1}\hspace{0,00cm}{\hyperlink{br1}{$\uparrow$}}\hspace{0,15cm}{\hyperlink{br1a}{$\uparrow$}}
 \
 Let $z\in \fet $. Since $\lambda (\fet) = 1- \lambda (\fet) = (0,1)\cup \left(\Bb{C}\setminus\Bb{R}\right)$
 it is possible to use the principal branch of the square root for $\lambda (z)$ and $1-\lambda (z)$ as well.
 Then by \eqref{f2int} and \eqref{f3cint} we get
 \begin{align}\label{f1case1}
    &  \sqrt{\lambda (z)}= \dfrac{\Theta_{2}(z)^{2}}{\Theta_{3}(z)^{2}} \, , \
    \sqrt{1-\lambda (z)}= \dfrac{\Theta_{4}(z)^{2}}{\Theta_{3}(z)^{2}}  \ , \quad z\in \fet\,.
 \end{align}

 \noindent Then, by virtue of \eqref{f3d1int}(b), we get
 \begin{align*}
    &  2 \Theta_{3}(2 z)^{2}  = \Theta_{3}( z)^{2} + \Theta_{4}( z)^{2}=
     \Theta_{3}( z)^{2} \left(1+ \dfrac{\Theta_{4}( z)^{2}}{\Theta_{3}( z)^{2}}\right)=
     \Theta_{3}( z)^{2}\left(1 + \sqrt{1-\lambda (z)}\right) ,
   \end{align*}

\noindent from which
\begin{align}\label{f2case1}
    &    \Theta_{3}( z)^{2}=\dfrac{2 \Theta_{3}(2 z)^{2}}{1 + \sqrt{1-\lambda (z)}}  \ , \quad z\in \fet\,.
\end{align}

\noindent Substituting here \eqref{f14inttrian}(b) written in the form
\begin{align}\label{f3case1}
    &   \Theta_{3}\big(z\big)^{2}\!= \!\he  \big(\lambda(z)\big)\ , \quad z\in \fet\,,
\end{align}

\noindent we obtain
\begin{align}\label{f4case1}
    & \he \left(\lambda (z)\right) = \dfrac{2 \he \left(\lambda (2z)\right)}{1 + \sqrt{1-\lambda (z)}}
 \ , \quad z\in \fet\,.
\end{align}

\noindent But if  $2 z \in \fet $ and  $ z \in \fet $ then  it follows from \eqref{f3d1int}(a), \eqref{f3d1int}(b) and \eqref{f1case1} that
\begin{align*}
   \sqrt{\lambda (2z)}  & =  \dfrac{\Theta_{2}(2 z)^{2}}{\Theta_{3}(2 z)^{2}}=
 \dfrac{\Theta_{3}( z)^{2}- \Theta_{4}( z)^{2}}{\Theta_{3}( z)^{2}+ \Theta_{4}( z)^{2}}=
\dfrac{ \dfrac{\Theta_{3}( z)^{2}}{\Theta_{4}( z)^{2}} -1}{ \dfrac{\Theta_{3}( z)^{2}}{\Theta_{4}( z)^{2}} +1}=
\dfrac{\dfrac{1}{ \sqrt{1-\lambda (z)}}-1}{\dfrac{1}{ \sqrt{1-\lambda (z)}}+1} \\    &   =
\dfrac{1-\sqrt{1-\lambda (z)}}{1+\sqrt{1-\lambda (z)}} \, , \
\end{align*}

\noindent and after squaring this identity the constraint $2 z \in \fet $ can be dropped,
 \begin{align}\label{f5case1}
    &  \lambda (2z)= \left(\dfrac{1-\sqrt{1-\lambda (z)}}{1+\sqrt{1-\lambda (z)}}\right)^{2} \ , \quad z\in \fet\,,
 \end{align}

 \noindent and we deduce from \eqref{f4case1},
\begin{align}\label{f6case1}
    & \he \left(\lambda (z)\right) = \dfrac{2 }{1 + \sqrt{1-\lambda (z)}} \
    \he \left(\left(\dfrac{1-\sqrt{1-\lambda (z)}}{1+\sqrt{1-\lambda (z)}}\right)^{2}\right)
 \ , \quad z\in \fet\, .
\end{align}

\noindent Or, what is the same,
\begin{align}\label{f7case1}
    & \he \left(z\right) = \dfrac{2 }{1 + \sqrt{1-z}}
 \   \he \left(\left(\dfrac{1-\sqrt{1-z}}{1+\sqrt{1-z}}\right)^{2}\right)
 \ , \quad z\in (0,1)\cup \left(\Bb{C}\setminus\Bb{R}\right)\, ,
\end{align}

 \noindent
which coincides with the quadratic   transformation (3.1.10) with $a=b=1$ of   \cite[p.\! 128]{and} for the hypergeometric function $\he$.
}}\end{clash}
\end{subequations}

\begin{subequations}
\begin{clash}{\rm{\hypertarget{r2}{}\label{case2}\hspace{0,00cm}{\hyperlink{br2}{$\uparrow$}}
 \  The Landen relationships \eqref{f3d1int}(a)--(c),
\begin{align}\label{f1case2}
    &  \begin{array}{rcl}
      2 \Theta_{2}(2 z)^{2}   &   =   &   \Theta_{3}( z)^{2}- \Theta_{4}( z)^{2} \ , \  \\
2 \Theta_{3}(2 z)^{2}   &   =   &   \Theta_{3}( z)^{2}+ \Theta_{4}( z)^{2} \ , \  \\
\Theta_{4}(2 z)^{2}  &   =   &  \Theta_{3}( z)\Theta_{4}( z) \,,
\end{array} \  \quad  z\in \Bb{H}\,,
\end{align}

\noindent can obviously written as follows
\begin{align}\label{f2case2}
    &  \begin{array}{rcl}
      \Theta_{3}( z)^{2} &   =   & \Theta_{3}(2 z)^{2} + \Theta_{2}(2 z)^{2} \ , \  \\
\Theta_{4}( z)^{2} &   =   & \Theta_{3}(2 z)^{2} - \Theta_{2}(2 z)^{2} \ ,  \\
\Theta_{4}(2 z)^{2}  &   =   &  \Theta_{3}( z)\Theta_{4}( z) \,,
\end{array}\  \quad  z\in \Bb{H}\,,
\end{align}

\noindent from which, by \eqref{f3cint}, we get
\begin{align*}
    & \Theta_{2}( z)^{4} =  \left(\Theta_{3}(2 z)^{2} + \Theta_{2}(2 z)^{2}\right)^{2}- \left(\Theta_{3}(2 z)^{2} - \Theta_{2}(2 z)^{2}\right)^{2} = 4 \Theta_{3}(2 z)^{2}\Theta_{2}(2 z)^{2} \, , \
\end{align*}

\noindent i.e.,
\begin{align}\label{f3case2}
    &  \Theta_{2}( z)^{4} =4 \Theta_{3}(2 z)^{2}\Theta_{2}(2 z)^{2} \ , \quad  z\in \Bb{H}\,.
\end{align}

\noindent Therefore
\begin{align*}
    &  \lambda (z) = \dfrac{\Theta_{2}( z)^{4}}{\Theta_{3}( z)^{2}} =
    \dfrac{4 \Theta_{3}(2 z)^{2}\Theta_{2}(2 z)^{2}}{\Theta_{3}(2 z)^{4} + \Theta_{2}(2 z)^{4} + 2\Theta_{3}(2 z)^{2}\Theta_{2}(2 z)^{2}}
\end{align*}

\noindent and then, in view of \eqref{f3bint} and \eqref{f3cint},
\begin{align*}
    &  \Theta_2 (z-1)^{2} = i \Theta_2 (z)^{2} \ , \quad   \Theta_3(z-1)^{2} = \Theta_4 (z)^{2} \ , \quad
    \dfrac{\lambda (z)}{1-\lambda (z)} = \dfrac{\Theta_2 (z)^{4}}{\Theta_4 (z)^{4}}\,,  \quad  z \in \Bb{H}\,,
\end{align*}

\noindent we obtain
\begin{align*}
      \lambda \left(\dfrac{z-1}{2}\right)& = \dfrac{\Theta_{2}\left(\dfrac{z-1}{2}\right)^{4}}{ \Theta_{3}\left(\dfrac{z-1}{2}\right)^{4}} =  \dfrac{4 \Theta_{3}(z-1)^{2}\Theta_{2}(z-1)^{2}}{\Theta_{3}(z-1)^{4} + \Theta_{2}(z-1)^{4} +
2\Theta_{3}(z-1)^{2}\Theta_{2}(z-1)^{2}} = \\ & =
  \dfrac{4 i \Theta_{4}(z)^{2}\Theta_{2}(z)^{2}}{\Theta_{4}(z)^{4} - \Theta_{2}(z)^{4} +
2i \Theta_{4}(z)^{2}\Theta_{2}(z)^{2}} \, , \
\end{align*}

\noindent from which, taking account of \eqref{f2int}, for arbitrary $z\in \fet$ we get
\begin{align*}
    &  \lambda \left(\dfrac{1+z}{1-z}\right)= \lambda \left(\dfrac{2}{1-z} -1\right)= \dfrac{\lambda \left(\dfrac{2}{1-z}\right)}{\lambda \left(\dfrac{2}{1-z}\right) -1} =
     \dfrac{\lambda \left(\dfrac{z-1}{2}\right)-1}{\lambda \left(\dfrac{z-1}{2}\right)}= \\ & = 1 -
\dfrac{1}{\lambda \left(\dfrac{z-1}{2}\right)}  =
1-  \dfrac{\Theta_{4}(z)^{4} - \Theta_{2}(z)^{4} +
2i \Theta_{4}(z)^{2}\Theta_{2}(z)^{2}}{4 i \Theta_{4}(z)^{2}\Theta_{2}(z)^{2}} = \\ & =
\dfrac{1}{2} - \dfrac{1}{4 i} \left(\dfrac{\Theta_{4}(z)^{2}}{\Theta_{2}(z)^{2}} - \dfrac{\Theta_{2}(z)^{2}}{\Theta_{4}(z)^{2}}\right)=
\dfrac{1}{2} +  \dfrac{1}{4 i} \left( \sqrt{\dfrac{\lambda (z)}{1-\lambda (z)}} -
\sqrt{\dfrac{1-\lambda (z)}{\lambda (z)}}\right) = \\ & =
\dfrac{1}{2} +  \dfrac{2\lambda (z)-1}{4 i \sqrt{(1-\lambda (z))\lambda (z)}}=
\dfrac{1}{2} +  i\dfrac{1-2\lambda (z)}{4  \sqrt{(1-\lambda (z))\lambda (z)}}  \, ,
\end{align*}

\noindent i.e., in the notation \eqref{f3eint},
\begin{align}\label{f4case2}
    &  \lambda \left(\dfrac{1+z}{1-z}\right)=\dfrac{1}{2} +
\dfrac{i\lambda_{1} (z)}{4  \sqrt{\lambda_{2} (z)}}   \ , \quad      z\in \fet\,,
\end{align}

\noindent which proves the validity of the left-hand side equality in \eqref{f3dint}. The right-hand side equality in \eqref{f3dint} is obtained from \eqref{f4case2} by replacing $z$ by $-1/z$ and by using \eqref{f2int}.

}}\end{clash}
\end{subequations}

\begin{subequations}
\begin{clash}{\rm{\hypertarget{r4}{}\label{case4}\hspace{0,00cm}{\hyperlink{br4}{$\uparrow$}}
 \ The right-hand side equality \eqref{f3dint} written for $z=it$
\begin{align*}
     & \dfrac{z-1}{z+1}=
     \dfrac{it-1}{it+1}=\dfrac{-t-i}{-t+i}=
\dfrac{t+i}{t-i}
      \ , \  \\  & \lambda \left(\dfrac{z-1}{z+1}\right)=\dfrac{1}{2} -
\dfrac{i\lambda_{1} (z)}{4  \sqrt{\lambda_{2} (z)}}\ \Rightarrow\  \\  &
\lambda \left(\dfrac{t+i}{t-i}\right)=\dfrac{1}{2} -
\dfrac{i\lambda_{1} (it)}{4  \sqrt{\lambda_{2} (it)}} \ , \quad t > 0,
\end{align*}

\noindent where $\lambda (i t)\in (0, 1)$ for all $t> 0$, makes it possible to write
\begin{align*}
     &  \left|\lambda \left(\dfrac{t+i}{t-i}\right)\right|^{2}=
\dfrac{1}{4} + \dfrac{\left(1-2\lambda (it)\right)^{2}}{16\lambda (it)\left(1-\lambda (it)\right) } \\  &
=\dfrac{4\lambda (it)\left(1-\lambda (it)\right)+\left(1-2\lambda (it)\right)^{2}}{16\lambda (it)\left(1-\lambda (it)\right) } \\  & =
\dfrac{4\lambda (it)-4\lambda (it)^{2}+1-4\lambda (it)+4\lambda (it)^{2}
}{16\lambda (it)\left(1-\lambda (it)\right) }=
\dfrac{1}{16\lambda (it)\left(1-\lambda (it)\right) } \ , \
\end{align*}

\noindent i.e., by \eqref{f2int},
\begin{align}\label{f1case4}
     &  \left|\lambda \left(\dfrac{t+i}{t-i}\right)\right|^{2}=
     \dfrac{1}{16\lambda (it)\left(1-\lambda (it)\right) }=
      \dfrac{1}{16\lambda (i t)\lambda (i/t) } \ , \quad t > 0\ , \
\end{align}

\noindent that proves equality in \eqref{f3zint}.

By using \eqref{f0apinttheor1}(b) written in the forms,
\begin{align*}
     & 0 \!< \! \lambda (it)\! <\! 16 {\rm{e}}^{{\fo{-  \pi t }}} , \ \
     0 \!< \! \lambda (i/t)\! <\! 16 {\rm{e}}^{{\fo{-  \pi/ t }}} , \ \
 0 \!<\! t \!<\! \infty  ,
\end{align*}

\noindent we derive from \eqref{f1case4} that
\begin{align*}
     & 16 \left|\lambda \left(\dfrac{t+i}{t-i}\right)\right|^{2}\geqslant
    \dfrac{1}{16\cdot {\rm{e}}^{{\fo{-  \pi t }}}\cdot 16 {\rm{e}}^{{\fo{-  \pi/ t }}} }=\frac{1}{16^{2}}\exp \pi \left(t
     + \dfrac{1}{t}\right) \ , \
\end{align*}

\noindent from which we obtain the inequality  in \eqref{f3zint},
\begin{align}\label{f2case4}
     &  64 \left|\lambda \left(\dfrac{t+i}{t-i}\right)\right|\geqslant
  \exp \dfrac{\pi }{2}\left(t + \dfrac{1}{t}\right)   \ , \quad t > 0\ .
\end{align}

}}\end{clash}
\end{subequations}

\begin{subequations}
\begin{clash}{\rm{\hypertarget{r3}{}\label{case3}\hspace{0,00cm}{\hyperlink{br3}{$\uparrow$}}
 \   The Landen relationships \eqref{f3d1int}(c) together with
\eqref{f3bint}(g) and  (A.18e) of \cite[p.\! 58]{bh1} for any $z\in \Bb{H}$ give
\begin{align*}
      \Theta_{3} \left(\dfrac{1}{1-z}\right)^{2}&= i (1-z)\Theta_{4}(z)^{2} \ , \
\Theta_{4} \left(\dfrac{1}{1-z}\right)^{2}=  (1-z)\Theta_{2}(z)^{2} \ ,  \\  \Theta_{3} \left(\dfrac{1+z}{1-z}\right)^{4}& = \Theta_{3} \left(\dfrac{2}{1-z} -1\right)^{4}=\Theta_{4}
 \left(\dfrac{2}{1-z}\right)^{4}= \Theta_{3}
 \left(\dfrac{1}{1-z}\right)^{2}\Theta_{4}
 \left(\dfrac{1}{1-z}\right)^{2} \\  & =
 i (1-z)^{2}\Theta_{2}(z)^{2}\Theta_{4}(z)^{2} \ , \
\end{align*}

\noindent which by \eqref{f3bint}(a), (c),
\begin{align*}
     & \Theta_{2}(-1/z)\Theta_{4}(-1/z)=(z/i)\Theta_{2}(z)\Theta_{4}(z) \ \Rightarrow \  \\  &
\Theta_{2}(-1/z)^{2}\Theta_{4}(-1/z)^{2}=-z^{2}\Theta_{2}(z)^{2}\Theta_{4}(z)^{2}
\ ,
\end{align*}

\noindent lead to
\begin{align*}
     &i (1-z)^{2}\Theta_{2}(z)^{2}\Theta_{4}(z)^{2}= -i (1-1/z)^{2}\Theta_{2}(-1/z)^{2}\Theta_{4}(-1/z)^{2} \ , \
\end{align*}

\noindent and therefore
\begin{align}\label{f1case3} &
    \Theta_{3} \left(\dfrac{1+z}{1-z}\right)^{4}=
\left\{
  \begin{array}{l}
   i (1-z)^{2}\Theta_{2}(z)^{2}\Theta_{4}(z)^{2} \ , \  \\[0,5cm]
   -i (1-1/z)^{2}\Theta_{2}(-1/z)^{2}\Theta_{4}(-1/z)^{2} \ ,
  \end{array}
\right. \quad z\in \Bb{H}\,.\end{align}

\noindent Making here the change of variables $z=-1/z^{\,\prime}$, we obtain
\begin{align}&
\Theta_{3} \left(\dfrac{z-1}{z+1}\right)^{4}=
\left\{  \begin{array}{l}
-i (1+z)^{2}\Theta_{2}(z)^{2}\Theta_{4}(z)^{2}  \ , \  \\[0,5cm]
  i (1+1/z)^{2}\Theta_{2}(-1/z)^{2}\Theta_{4}(-1/z)^{2}    \,,
  \end{array}\right.\quad z\in \Bb{H}\,.
\label{f2case3}\end{align}

\noindent When $z = i t$, $t> 0$,
\begin{align*}
     & \dfrac{z-1}{z+1} = \dfrac{i t-1}{i t+1}=\dfrac{ t+i}{ t-i} \ , \
\end{align*}

\noindent and \eqref{f2case3}  yields
\begin{align*}
     & \Theta_{3} \left(\dfrac{t+i}{t-i}\right)^{4}  =
\left\{  \begin{array}{l}
4 i (t-i)^{2} e^{-\pi t /2} \theta_{2}\left(e^{-\pi t }\right)^{2}\theta_{4}\left(e^{-\pi t }\right)^{2}
\ , \  \\[0,5cm]
 4 i \dfrac{(t-i)^{2}}{t^{2}}  e^{-\pi  /(2t)} \theta_{2}\left(e^{-\pi/ t }\right)^{2}\theta_{4}\left(e^{-\pi /t }\right)^{2}
  \ ,
  \end{array}\right.\quad t> 0\,.
\end{align*}

\noindent Hence,
\begin{align*}
     & \Phi(t):= \left|\Theta_{3} \left(\dfrac{t+i}{t-i}\right)\right|^{4}  =
\left\{  \begin{array}{ll}
4  (1+ t^{2}) e^{-\pi t /2} \theta_{2}\left(e^{-\pi  t}\right)^{2}\theta_{4}\left(e^{-\pi t }\right)^{2} \ , \  & \mbox{if} \ \ t \geqslant 1 \ , \
 \\[0,5cm]
 4  \dfrac{1+ t^{2}}{t^{2}}  e^{-\pi  /(2t)} \theta_{2}\left(e^{-\pi/ t }\right)^{2}\theta_{4}\left(e^{-\pi  / t }\right)^{2}
 \ , \  & \mbox{if} \ \ 0 <  t \leqslant 1 \ , \
  \end{array}\right.
\end{align*}

\noindent Here
\begin{align*}
    &  \theta_{4} (u)   \! :=  \!1  \!+ \! 2\sum\limits_{n\geqslant 1}
(-1)^{n}u^{n^2} = 1  \!- \! 2\sum\limits_{n\geqslant 1}\left(u^{n^2}-u^{(n+1)^2}\right) < 1 \, , \quad u \in (0,1),
\end{align*}

\noindent while, by \cite[p.\! 325, (xii)]{ber1},
\begin{align*}
    &  \theta_{2}\left(u\right) =
\!1 \!+\! \sum\limits_{n\geqslant 1} u^{n^2 +n} \leqslant \theta_{2}\left(e^{-\pi}\right) \, , \  0 < u < e^{-\pi} \, , \\  &
\psi(q)=\sum\limits_{n\geqslant 0} q^{n(n+1)/2}= \theta_{2}\left(\sqrt{q}\right)  \ , \quad
\theta_{2}\left(e^{-\pi}\right)= \psi\left(e^{-2\pi}\right) =  \dfrac{\pi^{1/4}e^{\pi/4}}{2^{5/4}\Gamma (3/4)} \ , \
\end{align*}

\noindent and therefore
\begin{align*}
     & \Phi(t)= \left|\Theta_{3} \left(\dfrac{t+i}{t-i}\right)\right|^{4}  \leqslant
\left\{  \begin{array}{ll}
4  (1+ t^{2}) e^{-\pi t /2}  \dfrac{\pi^{1/2}e^{\pi/2}}{2^{5/2}\Gamma (3/4)^{2}}  \ , \  & \mbox{if} \ \ t \geqslant 1 \ , \
 \\[0,5cm]
 4  \dfrac{1+ t^{2}}{t^{2}}  e^{-\pi  /(2t)} \dfrac{\pi^{1/2}e^{\pi/2}}{2^{5/2}\Gamma (3/4)^{2}}
 \ , \  & \mbox{if} \ \ 0 <  t \leqslant 1 \ .
  \end{array}\right.
\end{align*}

\noindent But
\begin{align*}
    &  \Phi(1/t)= \left|\Theta_{3} \left(\dfrac{1/t+i}{1/t-i}\right)\right|^{4} =
\left|\Theta_{3} \left(\dfrac{1+it}{1-it}\right)\right|^{4}  =
\left|\Theta_{3} \left(\dfrac{i-t}{i+ t}\right)\right|^{4}\\ & =
\left|\theta_{3} \left(e^{{\fo{i \pi \dfrac{i-t}{i+ t}}}}\right)\right|^{4} =
\left|\theta_{3} \left(e^{{\fo{-i \pi \dfrac{-i-t}{-i+ t}}}}\right)\right|^{4} =
\left|\theta_{3} \left(e^{{\fo{i \pi \dfrac{t+i}{t-i}  }}}\right)\right|^{4}\\ &=\left|\Theta_{3} \left(\dfrac{t+i}{t-i}\right)\right|^{4}= \Phi(t) \ , \
\end{align*}

\noindent and so
\begin{align*}
    &   \Phi(t)= \sqrt{ \Phi(t) \Phi(1/t)} \leqslant  \dfrac{\pi^{1/2}e^{\pi/2}}{2^{1/2}\Gamma (3/4)^{2}} \left(t + \dfrac{1}{t}\right)
e^{{\fo{-\dfrac{\pi }{4}\left(t + \dfrac{1}{t}\right)  }}} \ , \quad t > 0,
\end{align*}

\noindent where
\begin{align*}
    & \Gamma\left(\dfrac{3}{4}\right)= 1,2254167024\ldots   \ , \quad  \Gamma\left(\dfrac{3}{4}\right)^{2} = 1,5016460\ldots   \, , \\  &   \sqrt{2}=1,41421356\ldots \, , \  \sqrt{\pi}=1,772453850\ldots \, , \quad
e^{{\fo{\pi/2}}}= 4,810477380\ldots \, , \\    &
 \sqrt{\pi} e^{{\fo{\pi/2}}}=8,526349152518913\ldots \, ,\quad
 \sqrt{2}\Gamma\left(\dfrac{3}{4}\right)^{2}=2,1236481390831\ldots \, , \\    &
\dfrac{\pi^{1/2}e^{\pi/2}}{2^{1/2}\Gamma (3/4)^{2}}=4,0149537937108354094530647660142\ldots <5 \ , \
\end{align*}

\noindent which completes the proof of \eqref{f3yint}.

}}\end{clash}
\end{subequations}

\vspace{0.2cm}
\subsection[\hspace{-0,25cm}. \hspace{0,05cm}Notes for Section~\ref{intprel}]{\hspace{-0,11cm}{\bf{.}} Notes for Section~\ref{intprel}}

\vspace{0.2cm}
\begin{subequations}
\begin{clash}{\rm{\hypertarget{r711}{}\label{case711}\hspace{0,00cm}{\hyperlink{br711}{$\uparrow$}}
 \, We prove the statements of Lemma~\ref{lemfou}.

Suppose for the moment the assertion of Lemma~\ref{lemfou} holds with $a\!=\!0$.
Then we would apply that statement to the new function $g (z)\!:=\! f_{0} (ia \!+ \!z)$ and derive
that Lemma~\ref{lemfou} holds in fact for any $a \!\in\! \Bb{R}$. Hence it suffices to obtain \eqref{f1lemfou}
in the case $a=0$ only. Then our assumptions are that the functions $f_{0}$  and $f_{1} (z)\!:= \!\varphi (\exp ({\rm{i}}\pi z))$ are both holomorphic in $\Bb{H}$ and $2$-periodic.

For $z= \beta \exp (i b)\in \Bb{D} \setminus (-1,0]$, $\beta \in (0,1)$, $-\pi< b < \pi$, we have
\begin{align}\label{f2lemfou}
    &    \dfrac{\Log z }{{\rm{i}} \pi}=
\dfrac{\Log \left(\beta {\rm{e}}^{{\rm{i}}b }\right) }{{\rm{i}} \pi}=\frac{{\rm{i}}}{\pi} \ln \frac{1}{\beta} +\frac{b}{\pi}  \in  (-1,1) +{\rm{i}} \Bb{R}_{>0} \subset \Bb{H} \,,
\end{align}

\noindent and since $\Log  \in {\rm{Hol}}\left(\Bb{C}\setminus \Bb{R}_{\leqslant 0}\right)$ we obtain that
$\Phi^{\delta}\in {\rm{Hol}}( \Bb{D} \setminus (-1,0])$ for every $\delta\in \{0,1\}$, where
\begin{align*}
    & \Phi^{\delta} (z) := f_{\delta} \left(\dfrac{\Log z }{i \pi}\right)
    \, ,  \quad
 \Phi^{1} (z) =   \varphi(z)    \, , \ \quad  z\in \Bb{D} \setminus (-1,0]  \, , \  \delta\in \{0,1\}  \,.
\end{align*}

\noindent  Here, if we write $\Phi^{\delta}_{\pm}  (-x  ):=
 \lim_{z\in \pm \Bb{H}, \, z \to -x} \  \Phi^{\delta}  (z)$ for each $\delta\in \{0,1\}$, we find
\begin{align*}
     & \Phi^{\delta}_{+} (-x) = f_{\delta} \left(\frac{i}{\pi} \ln \frac{1}{x} +1\right)=
      f_{\delta} \left(\frac{i}{\pi} \ln \frac{1}{x} -1\right) = \Phi^{\delta}_{-} (-x) \ , \ x \in (0,1) \,,
\end{align*}

\noindent and consequently, $\Phi^{\delta}$ extends continuously across $(-1,0)$.
From Morera's theorem (see \cite[p.\! 96]{lav})  we conclude that $\Phi^{\delta} \in {\rm{Hol}} (\Bb{D}\setminus\{0\})$. For $\delta =1$ this completes the proof of the second assertion of Lemma~\ref{lemfou}.
 As for $\delta = 0$, we obtain that $\Phi^{0}$ admits the  Laurent expansion (see \cite[p.\! 107]{con}) $\Phi^{0} (z) =
\sum\nolimits_{n \in \Bb{Z}} \ a_{n}  z^{n}$,  $ z \in \Bb{D}\setminus\{0\} $, which is absolutely
convergent in $\Bb{D}\setminus\{0\}$.
 The change of variables  $z =
 {\rm{e}}^{{\rm{i}} \pi \zeta} $, $\zeta = x + {\rm{i}} y$, $x \in (-1,1)$, $y > 0$, gives, by using
 \eqref{f2lemfou}, that $\Log \exp (i \pi \zeta) = {\rm{i}} \pi \zeta$, and, moreover, that \begin{align*}
    &  f_{0} (\zeta) = \sum\nolimits_{n \in \Bb{Z}} \ a_{n}   {\rm{e}}^{{\rm{i}} \pi n \zeta} \ ,\ \  \
    \im\,\zeta> 0  \ ,  \  \re\,\zeta   \in (-1,1) \,.
\end{align*}

\noindent Both $f_0$ and the Fourier series on the right-hand side are continuous on
$\Bb{H}$ and periodic with period $2$. Hence they must be equal for all $\zeta\in\Bb{H}$  and, consequently,
\begin{align*}
    & \int_{z}^{z+2} f_{0} (\zeta) d \zeta = 2  a_{0}\,,
\end{align*}

\noindent for arbitrary $z \in  \Bb{H}$. This proves the first assertion of  Lemma~\ref{lemfou} and
completes the proof of  Lemma~\ref{lemfou}.
}}\end{clash}
\end{subequations}

\vspace{0.2cm}
\begin{subequations}
\begin{clash}{\rm{\hypertarget{r7}{}\label{case7}\hspace{0,00cm}{\hyperlink{br7}{$\uparrow$}}
  \ We have, in view of \eqref{f9int},
\begin{align*}
    &  - 4 \pi^{2} n \eurm{M}_{n} (x) = - 4 \pi^{2} n \eurm{H}_{n} (-1/x)/x^{2} \! =  \!\!\! \!\!\!  \int\limits_{\gamma (-1,1)}\!\!\!\!\!
\frac{S^{{\tn{\triangle}}}_{n} \!\left(\dfrac{1}{\lambda (z)}\right) d z}{(z x-1)^{2}}=
 \!\!\! \!\!\!  \int\limits_{\gamma (-1,1)}\!\!\!\!\!
\frac{S^{{\tn{\triangle}}}_{n} \!\left(\dfrac{1}{\lambda (z)}\right) }{( x-1/z)^{2}}\,
\dfrac{d z}{z^{2}} \\    & \left|z= - \dfrac{1}{z^{\,\prime}}\,, \ d z = \dfrac{d z^{\,\prime} }{\left(z^{\,\prime}\right)^{2}}\right|  =-
\!\!\! \!\!\!  \int\limits_{\gamma (-1,1)}\!\!\!\!\!
\frac{S^{{\tn{\triangle}}}_{n} \!\left(\dfrac{1}{\lambda (-1/z)}\right) d z}{( x+z)^{2}}=-
\!\!\! \!\!\!  \int\limits_{\gamma (-1,1)}\!\!\!\!\!
\frac{S^{{\tn{\triangle}}}_{n} \!\left(\dfrac{1}{1-\lambda (z)}\right) d z}{( x+z)^{2}} \, , \
\end{align*}

\noindent from which
\begin{align*}
    &  \eurm{M}_{n} (x) =\dfrac{1}{ 4 \pi^{2} n }\!\!\! \!\!\!  \int\limits_{\gamma (-1,1)}\!\!\!\!\!
\frac{S^{{\tn{\triangle}}}_{n} \!\left(\dfrac{1}{1-\lambda (z)}\right) d z}{( x+z)^{2}} \,,
\end{align*}

\noindent and therefore by   \eqref{f1intlem3} and \eqref{f2pinttheor1},
\begin{align*}
    &  (-1)^{n}4 \pi^{2} n \eurm{M}_{n} (x) = \!\!\! \!\!\!  \int\limits_{\gamma (-1,1)}\!\!\!\!\!
\frac{S^{{\tn{\triangle}}}_{n} \!\left(\dfrac{1}{1-\lambda (z)}\right) d z}{( x+z)^{2}}=
\!\!\! \!\!\!  \int\limits_{\gamma (-1,1)}\!\!\!\!\!
\frac{S^{{\tn{\triangle}}}_{n} \!\left(1-\dfrac{1}{1-\lambda (z)}\right) -S^{{\tn{\triangle}}}_{n}(1)
}{( x+z)^{2}} d z \\    &   =
\!\!\! \!\!\!  \int\limits_{\gamma (-1,1)}\!\!\!\!\!
\frac{S^{{\tn{\triangle}}}_{n} \!\left(\dfrac{\lambda (z)}{\lambda (z)-1}\right) -S^{{\tn{\triangle}}}_{n}(1)}{( x+z)^{2}} d z=
\!\!\! \!\!\!  \int\limits_{\gamma (-1,1)}\!\!\!\!\!
\frac{S^{{\tn{\triangle}}}_{n} \!\left(\lambda (z+1)\right) -S^{{\tn{\triangle}}}_{n}(1)}{( x+z)^{2}} d z
\end{align*}

\noindent from which
\begin{align*}
    & \eurm{M}_{n} (x) =
\dfrac{(-1)^{n}}{ 4 \pi^{2} n }\!\!\! \!\!\!  \int\limits_{\gamma (-1,1)}\!\!\!\!\!
\frac{S^{{\tn{\triangle}}}_{n} \!\left(\lambda (z+1)\right) -S^{{\tn{\triangle}}}_{n}(1)}{( x+z)^{2}} d z \ .
\end{align*}

}}\end{clash}
\end{subequations}

\begin{subequations}
\begin{clash}{\rm{\hypertarget{r9}{}\label{case9}\hspace{0,00cm}{\hyperlink{br9}{$\uparrow$}}
  \ According with Theorem~\ref{inttheor2}, \eqref{f6intthA} and \eqref{f7intthA},
\begin{align}\nonumber
    &  \left\{y\!\in\! {\rm{clos}}\left(\mathcal{F}_{{\tn{\square}}}\right) \ \big| \ \lambda(y)\!=\! z\right\}
     \\[0,5cm]  \nonumber  &   =\left\{  \begin{array}{ll}
 \ie (z)    \,, &\ \  \hbox{if}\quad z \in
 (0,1)\cup \left(\Bb{C}\setminus\Bb{R}\right) \,; \\[0,25cm]
  \left\{\ie (z+ i 0), \ie (z- i 0)\right\} \  , &\ \ \hbox{if}\quad z= 1+x  \ , \quad  x > 0 \,,\\[0,25cm]
 \left\{\ie (z+ i 0), \ie (z- i 0)\right\}  , &\ \  \hbox{if}\quad z = - x  \ , \qquad  x > 0 \,,
  \end{array}\right.
  \\[0,5cm]  &   \left\{  \begin{array}{cll}
  =   &   \ie (z)    \,, &\ \  \hbox{if}\quad z \in
 (0,1)\cup \left(\Bb{C}\setminus\Bb{R}\right) \,; \\[0,25cm]
  \stackrel{{\fo{\eqref{f7intthA}}}}{\vphantom{A}=}   &
  \left\{\dfrac{-1 +  \imag   \,  {\fo{\Delta}} (x)}{1 + {\fo{\Delta}} (x)^{2}},
  \dfrac{1 +  \imag   \,  {\fo{\Delta}} (x)}{1 + {\fo{\Delta}} (x)^{2}}\right\} \  , &\ \ \hbox{if}\quad z= 1+x  \ , \quad  x > 0 \,,\\[0,5cm]
   \stackrel{{\fo{\eqref{f6intthA}}}}{\vphantom{A}=}   &
 \left\{-1 +  \imag \, {\fo{\Delta}} (x),   1 +  \imag   \,  {\fo{\Delta}} (x)\right\}  ,
 &\ \  \hbox{if}\quad z = - x  \ , \qquad  x > 0 \,.
  \end{array}\right.
\label{f1case9}\end{align}

\noindent Therefore
\begin{align*}
    &  \sum\limits_{{\fo{\left\{y\!\in\! {\rm{clos}}\left(\mathcal{F}_{{\tn{\square}}}\right) \ \big| \ \lambda(y)\!=\! z\right\}}}}\!\!\!\!\!\!\!\!\!\!\!\!\!\!\!\!\!\!\!\!e^{{\fo{- n \pi i  y }}}
\end{align*}

\noindent is equal to
\begin{align*}
    &  e^{{\fo{- n \pi i \ie (z)}}} \, , \
\end{align*}

\noindent if $z \in
 (0,1)\cup \left(\Bb{C}\setminus\Bb{R}\right)$,

 \begin{align*}
    &   e^{{\fo{   - n \pi i \dfrac{-1 +  \imag   \,  {\fo{\Delta}} (x)}{1 + {\fo{\Delta}} (x)^{2}}  }}} +  e^{{\fo{  - n \pi i \dfrac{1 +  \imag   \,  {\fo{\Delta}} (x)}{1 + {\fo{\Delta}} (x)^{2}}  }}}=
  e^{{\fo{ \dfrac{n \pi \,  {\fo{\Delta}} (x)}{1 + {\fo{\Delta}} (x)^{2}}   }}}
  \left(e^{{\fo{     \dfrac{n \pi i}{1 + {\fo{\Delta}} (x)^{2}}  }}} +
  e^{{\fo{  -   \dfrac{n \pi i}{1 + {\fo{\Delta}} (x)^{2}}  }}} \right) \\    &   =
 2 \, e^{{\fo{ \dfrac{n \pi \,  {\fo{\Delta}} (x)}{1 + {\fo{\Delta}} (x)^{2}}   }}}  \cos
 \dfrac{n \pi }{1 + {\fo{\Delta}} (x)^{2}} \, , \
 \end{align*}

 \vspace{0.25cm}
\noindent if $ z= 1+x$, $x > 0$, and

\begin{align*}
    &  e^{{\fo{- n \pi i \left( -1 +  \imag \, {\fo{\Delta}} (x)\right)   }}} +
    e^{{\fo{- n \pi i \left(1 +  \imag \, {\fo{\Delta}} (x) \right)  }}}=
    2 (-1)^{n} e^{{\fo{n \pi {\fo{\Delta}} (x)}}}  \, , \
\end{align*}

\vspace{0.25cm}
\noindent if $ z= -x$, $x > 0$. Thus,
\begin{align*}
    &  \sum\limits_{{\fo{\left\{y\!\in\! {\rm{clos}}\left(\mathcal{F}_{{\tn{\square}}}\right) \ \big| \ \lambda(y)\!=\! z\right\}}}}\!\!\!\!\!\!\!\!\!\!\!\!\!\!\!\!\!\!\!\!e^{{\fo{- n \pi i  y }}}=
\left\{
  \begin{array}{ll}
\exp \left(- n \pi i \ie (z)\right)    \,, &\ \  \hbox{if}\quad z \in
 (0,1)\cup \left(\Bb{C}\setminus\Bb{R}\right) \,; \\[0,25cm]
 2 e^{{\fo{    \dfrac{  \, n \pi {\fo{\Delta}} (x)}{1 + {\fo{\Delta}} (x)^{2}}  }}}\cos  \dfrac{ n \pi }{1 + {\fo{\Delta}} (x)^{2}} \  , &\ \ \hbox{if}\quad z= 1+x  \ , \quad  x > 0 \,,\\[0,5cm]
2 (-1)^{n} \exp \left( n \pi {\fo{\Delta}} (x)\right)     , &\ \  \hbox{if}\quad z = - x  \ , \qquad  x > 0 \,,
  \end{array}
\right.
\end{align*}

\vspace{0.25cm}\noindent  where
\begin{align*}\tag{\ref{f15int}}
    & \hspace{-0,2cm}{\fo{\Delta}} (x)\! := \!   \dfrac{\he   \left( {1}\big/{(1\!  +\!  x)}\right)}{\he
 \left(  {x}\big/{(1\!  +\!  x)}\right)} \ ,
    \quad
    \left\{\begin{array}{l}
    {\fo{\Delta}} (0)\! =\!  +\! \infty \ ,\\[0,2cm]
    {\fo{\Delta}} (+\infty)\!  = \! 0 \ ,
    \end{array}\right.
    \ \  \ \dfrac{{{\diff}} {\fo{\Delta}} (x)}{{{\diff}} x}\!  <\!  0, \ \
     {\fo{\Delta}} (x){\fo{\Delta}} (1/x)\!  = \!1,\hspace{-0,1cm}
\end{align*}

\noindent what has been asserted after \eqref{f21zint}.

}}\end{clash}
\end{subequations}

\begin{subequations}
\begin{clash}{\rm{\hypertarget{r8}{}\label{case8}\hspace{0,00cm}{\hyperlink{br8}{$\uparrow$}}
  \ It follows from  \eqref{f9int}
\begin{align*}
    &  2 \pi\imag \eurm{H}_{0} (1/x)/x^{2}\!=\!\!\! \!\!\! \int\limits_{\gamma (-1,1)} \!\!\!\!\!\! \frac{z \, \Theta_{3}\left(z\right)^{4}}{  1 - z^{2}x^{2}} d z=
    \!\!\! \!\!\! \int\limits_{\gamma (-1,1)} \!\!\!\! \!\!
    \frac{z \, \Theta_{3}\left(z\right)^{4}}{  z^{-2} - x^{2}} \dfrac{d z}{z^{2}} =-
    \!\!\! \!\!\! \int\limits_{\gamma (-1,1)} \!\!\!\!\!\!
    \frac{(-1/z)\Theta_{3}\left(-1/z\right)^{4}}{  z^{2} - x^{2}} d z \, , \
\end{align*}

\noindent where by  \eqref{f3bint}(b),
\begin{align*}
    &  \Theta_3(-1/z)^{4} = -z^{2}\Theta_3(z)^{4} \, , \
\end{align*}

\noindent and therefore
\begin{align*}
    &  2 \pi\imag \eurm{H}_{0} (-1/x)/x^{2}\!=-\!\!\! \!\!\! \int\limits_{\gamma (-1,1)} \!\!\!\!\!\!
    \frac{(-1/z)\left(-z^{2}\Theta_3(z)^{4}\right)}{  z^{2} - x^{2}} d z =
    \!\!\! \!\!\! \int\limits_{\gamma (-1,1)} \!\!\!\!\!\!
    \frac{z\Theta_3(z)^{4}}{  x^{2} - z^{2}} d z = 2 \pi\imag \eurm{H}_{0} (x) \,.
\end{align*}

\noindent Thus,
\begin{align*}
    &  \eurm{H}_{0} (x)= \eurm{H}_{0} (-x) \, , \  x\in \Bb{R}\,, \quad
    \eurm{H}_{0} (-1/x)\!=\eurm{H}_{0} (x) x^{2}\, , \  x\in \Bb{R}\setminus\{0\} \,.
\end{align*}
}}\end{clash}
\end{subequations}

\begin{subequations}
\begin{clash}{\rm{\hypertarget{r11}{}\label{case11}\hspace{0,00cm}{\hyperlink{br11}{$\uparrow$}}
  \ We prove the next assertion.

\begin{lemma}\hspace{-0,18cm}{\bf{.}}\label{case11lem1}
Let $a,b,c,d \in \Bb{C}$, $ad - b c \neq 0$ and $\phi_{{\nor{(\begin{smallmatrix} a & b \\ c & d \end{smallmatrix})}}}$ be the M\"{o}bius transformation {\rm{\cite[p.\! 47, Definition 3.5]{con}}} such that
\begin{align}\label{f1case11lem1}
    &   \phi_{{\nor{(\begin{smallmatrix} a & b \\ c & d \end{smallmatrix})}}}\left(\fet \right) =\fet \ , \quad   \phi_{{\nor{(\begin{smallmatrix} a & b \\ c & d \end{smallmatrix})}}}\left(z\right) := \dfrac{a z + b}{c z+ d} \ .
\end{align}

\vspace{0.25cm} Then
\begin{align}\label{f2case11lem1}
    & \phi_{{\nor{(\begin{smallmatrix} a & b \\ c & d \end{smallmatrix})}}}(z ) \in
\left\{ \ z \ , \ \ -\dfrac{1}{z} \ , \ \  \dfrac{ z -1}{z + 1}  \ , \ \ -\dfrac{z + 1}{ z -1} \  \right\} \,.
\end{align}
\end{lemma}

{\emph{Proof.}}
   The boundary of the image of $\fet $ under the M\"{o}bius transformation
\begin{align*}
    &  \phi_{{\nor{(\begin{smallmatrix} a & b \\ c & d \end{smallmatrix})}}}\left(z\right) = \dfrac{a z + b}{c z+ d} \ ,
\end{align*}

\noindent according to the known property of the M\"{o}bius transformations to map circles onto circles
\cite[p.\! 49, Theorem 3.14]{con} and to the  known open mapping theorem \cite[p.\! 99, Theorem 7.5]{con}  of nonconstant analytic functions to map open set to open set, consists of four circles and therefore four vertices of $\fet$
are transformed into vertices, i.e.,
\begin{align*}
    &   \phi_{{\nor{(\begin{smallmatrix} a & b \\ c & d \end{smallmatrix})}}}\left(\left\{\infty, 0, 1, -1\right\}\right) = \left\{\infty, 0, 1, -1\right\} \,.
\end{align*}

\noindent Since $-1/\fet = \fet$ then if $ \phi_{{\nor{(\begin{smallmatrix} a & b \\ c & d \end{smallmatrix})}}}\left(\fet \right) =\fet$ then
\begin{align}\label{f3fabpolquadlem2}
    &   -1/ \phi_{{\nor{(\begin{smallmatrix} a & b \\ c & d \end{smallmatrix})}}}(z) \ , \quad
  -1/ \phi_{{\nor{(\begin{smallmatrix} a & b \\ c & d \end{smallmatrix})}}}(-1/z) \ , \quad-1/ \phi_{{\nor{(\begin{smallmatrix} a & b \\ c & d \end{smallmatrix})}}}(-1/z) \ ,
\end{align}

\noindent also map $\fet $ onto itself. Therefore it suffices to consider two cases
\begin{align*}
    &  1 ) \ \  f (\infty) = \infty \ , \quad  2) \ \  f (\infty) = 1 \ , \quad  f (z):= \phi_{{\nor{(\begin{smallmatrix} a & b \\ c & d \end{smallmatrix})}}}\left(z\right) = \dfrac{a z + b}{c z+ d}  \,.
\end{align*}

\vspace{0.25cm} In the first case we get $c =0$ and therefore one may assume that $f (z) = a z + b$. Then
$f (0) = b\in \{0, -1, 1\}$, i.e. it should be
\begin{align*}
    &  \{f(1), f(-1)\}  = \{0, -1, 1\}\setminus\{b\}   \ , \quad  f (z) = a z + b \ , \quad  b\in \{0, -1, 1\}\,.
\end{align*}

\noindent we get three cases  $f (z) = a z $, $f (z) = a z -1$, $f (z) = a z +1$, where correspondingly
\begin{align*}
    & b=0  \ \Rightarrow \   \{f(1), f(-1)\}  = \left\{a, -a\right\} =  \{1, -1\}
 \ \Rightarrow \  a \in  \{1, -1\}
  \ ,  \\    &
b=-1  \ \Rightarrow \   \{f(1), f(-1)\}  = \left\{a-1, -a-1\right\} =  \{0, 1\}
 \ \Rightarrow \  \left\{a, -a\right\} =  \{1, 2\}  \\    &  \Rightarrow \  \mbox{impossible }
   ,  \\    &
b=1  \ \Rightarrow \  \{f(1), f(-1)\}  = \left\{a+1, -a+1\right\} =  \{0, -1\}
 \ \Rightarrow \  \left\{a, -a\right\} =  \{-1, -2\}  \\    &  \Rightarrow \  \mbox{impossible }
 ,
\end{align*}

\noindent and since $f (z) = - z$ is not acceptable because $-\fet \neq \fet$, we obtain $f (z)=z$.

\vspace{0.25cm} For the second case we have $a=c$, i.e.
\begin{align*}
    &   f (z) = \dfrac{ z + b}{z + d} \,,
\end{align*}

\noindent and it follows from  $\infty \in f ( \{1, -1, 0\})$ that
 \begin{align*}
    &  d  \in \{1, -1, 0\} \,.
 \end{align*}

\noindent and therefore it should be
\begin{align*}
    & f \left( \{0, -1, 1\}\setminus\{-d\}   \right) = \{-1, 0\}  \ , \quad f (z) = \dfrac{ z + b}{z + d} \ , \quad   d  \in \{1, -1, 0\} \,.
\end{align*}

\noindent We get three cases
\begin{align*}
    &  d = 0  \ \Rightarrow \  f \left( \{ -1, 1\}   \right) =\left\{ 1 +b, 1-b\right\} = \{-1, 0\} \ \Rightarrow \  \left\{ b, -b\right\} = \{-2, -1\} \\    &  \Rightarrow \  \mbox{impossible }
 ,  \\    &
d = 1  \ \Rightarrow \  f \left( \{ 1, 0\}   \right) =\left\{ \dfrac{ 1 + b}{1 + 1}, \dfrac{  b}{ 1}\right\} =
\left\{\dfrac{ 1 + b}{2}, b\right\}= \{-1, 0\}  \\    & \Rightarrow \  \left\{1+b, 2 b\right\} = \{-2, 0\}  \ \Rightarrow \ b = -1\ ,  \\    &
d = -1  \ \Rightarrow \  f \left( \{ -1, 0\}   \right) =\left\{ \dfrac{ -1 + b}{-1 - 1}, \dfrac{  b}{- 1}\right\} = \left\{\dfrac{1-b}{2}, -b\right\}= \{-1, 0\}  \\    &  \ \Rightarrow \ b=1  \ \Rightarrow \
f (z) = \dfrac{ z + 1}{z -1}  = 1 + \dfrac{ 2}{z -1} \ \Rightarrow \   f (\Bb{H}) \subset - \Bb{H} \Rightarrow \  \mbox{impossible } .
\end{align*}

\noindent Thus, $f (z)= (z-1)/(z+1)$ and taking account of \eqref{f3fabpolquadlem2} we complete the proof.

}}\end{clash}
\end{subequations}

\begin{subequations}
\begin{clash}{\rm{\hypertarget{r10}{}\label{case10}\hspace{0,00cm}{\hyperlink{br10}{$\uparrow$}}
  \
 To calculate $\Delta_{n}^{S}(0)$ in \eqref{f6int}, written as
\begin{align*}
    &  e^{{\fo{- n \pi i \ie(z)  }}} \!=\! S^{{\tn{\triangle}}}_{n} (1/z)\! + \!\Delta_{n}^{S}(z)\ , \ z \!\in \! \Bb{D}\setminus \{0\} \ , \
\end{align*}

\noindent we use the  formula
\begin{align*}
    &  \frac{1}{2 \pi i} \int\nolimits_{\beta\partial  \Bb{D}}\, \zeta^{-1}\zeta^{p}  d \zeta=\delta_{0, p} \ , \quad p \in \Bb{Z} \ , \quad \beta\in (0,1)  \, , \
\end{align*}

\noindent and similarly to \eqref{f17int}, by applying Lemma~\ref{lemfou} to the periodic integrands,  obtain
\begin{align*}
    &  \Delta_{n}^{S}(0) = \frac{1}{2 \pi i} \int_{\beta\partial  \Bb{D}}   \frac{
   e^{{\fo{- n \pi i \ \ie(\zeta) }}} d \zeta}{ \zeta} =\frac{1}{2 \pi i} \int_{-\pi}^{\pi}   \frac{
   e^{{\fo{- n \pi i \ \ie\left(\beta e^{i \varphi}\right) }}} d \left(\beta e^{i \varphi}\right) }{   \beta e^{i \varphi}} = \\ & =
 \frac{1}{2 \pi i} \int\nolimits_{{\fo{ -1 +   \imag   {\fo{\Delta}} (\beta)  }} }^{{\fo{ 1 +   \imag    {\fo{\Delta}} (\beta) }} }  \ \frac{
   \lambda^{\,\prime}\left(\zeta\right)e^{{\fo{- n \pi i  \zeta }}}  }{ \lambda\left(\zeta\right)}\, d \zeta
= \frac{1}{2 \pi i} \int\nolimits_{{\fo{ -1 \! +  \!  \imag   a }} }^{{\fo{ 1 \! + \!   \imag    a }} }   \frac{
   \lambda^{\,\prime}\left(\zeta\right)e^{{\fo{- n \pi i  \zeta }}}  }{ \lambda\left(\zeta\right)}\, d \zeta \ , \end{align*}

\noindent where  ${\fo{\Delta}} (\beta)  > 1 $, while $a$ is arbitrary positive real number.  Then for any
$\im\, y > a$ we can  apply Lemma~\ref{lemfou} once more for arbitrary $A > \im\, y$ to get similarly to the transform of \eqref{f21int},
 \begin{align*}  &
\sum\limits_{n=1}^{\infty} \Delta_{n}^{S}(0) {\rm{e}}^{\imag  \pi n y} =
    \frac{1}{2 \pi i} \int\nolimits_{{\fo{ -1 +   \imag   a }} }^{{\fo{ 1 +   \imag    a }} }
    \  \frac{  \lambda^{\,\prime}\left(\zeta\right) d \zeta }{\left(e^{{\fo{  \pi i(  \zeta -y) }}} -1\right) \lambda\left(\zeta\right)} = \\ & =
\frac{1}{i \pi}  \frac{  \lambda^{\,\prime}\left(y\right)  }{ \lambda\left(y\right)} +
      \frac{1}{2 \pi i} \int\limits_{{\fo{-1 +   \imag   A +\re\, y}} }^{{\fo{  1 +   \imag   A +\re\, y}} }
     \frac{  \lambda^{\,\prime}\left(\zeta\right) d \zeta }{\left(e^{{\fo{  \pi i(  \zeta -y) }}} -1\right)
 \lambda\left(\zeta\right)} \, , \
\end{align*}

\noindent from which  by letting $A \to +\infty$ it follows from
\begin{align*}
 \frac{1}{i \pi}  \frac{  \lambda^{\,\prime}\left(y\right)  }{ \lambda\left(y\right)}&=   \frac{1}{i \pi}  \frac{   \imag \, \pi \,\lambda (y) \,\left(1 -\lambda (y)\right)\,\Theta_{3}\left(y\right)^{4}  }{ \lambda\left(y\right)}  = \left(1 -\lambda (y)\right)\Theta_{3}\left(y\right)^{4} = \\ & =
\left(1 -\dfrac{\Theta_{2}\left(y\right)^{4}}{\Theta_{3}\left(y\right)^{4}}\right)\Theta_{3}\left(y\right)^{4}=
\Theta_{4}\left(y\right)^{4}   \ , \quad  y\in \Bb{H},
\end{align*}

\noindent and $\Theta_{4}\left(y\right) \to 1$, as $0 < \im\, y \to +\infty$, that
\begin{multline*}
   \lim\limits_{ A \to +\infty} \frac{1}{2 \pi i}\!\!\!\int\limits_{{\fo{-1 +   \imag   A +\re\, y}} }^{{\fo{  1 +   \imag   A +\re\, y}} }
    \!\!\! \frac{  \lambda^{\,\prime}\left(\zeta\right) d \zeta }{\left(e^{{\fo{  \pi i(  \zeta -y) }}} -1\right)
 \lambda\left(\zeta\right)} \\  =
 - \lim\limits_{ A \to +\infty}\frac{1}{2 }\!\!\! \int\limits_{{\fo{-1 +   \imag   A +\re\, y}} }^{{\fo{  1 +   \imag   A +\re\, y}} }\!\!\!\frac{  \Theta_{4}\left(\zeta\right)^{4} d \zeta }{1- e^{{\fo{  \pi i(  \zeta -y) }}}  }=-1,
\end{multline*}

\noindent and therefore
\begin{align*}
    & \sum\limits_{n=1}^{\infty} \Delta_{n}^{S}(0) {\rm{e}}^{\imag  \pi n y} =\frac{1}{i \pi}  \frac{  \lambda^{\,\prime}\left(y\right)  }{ \lambda\left(y\right)} -1=
\Theta_{4}\left(y\right)^{4} -1 ,
\end{align*}

\noindent i.e.,
\begin{align*}
    &  1+ \sum\limits_{n=1}^{\infty} \Delta_{n}^{S}(0) u^{n} = \left(1  \!+ \! 2\sum\nolimits_{n\geqslant 1}
(-1)^{n} u^{n^{2}} \right)^{4} = \theta_{3}(-u)^{4},  \   u \in \Bb{D} \,,
\end{align*}

\noindent which together with \eqref{f3wcint}(b) proves \eqref{f12cint}.

}}\end{clash}
\end{subequations}
\subsection[\hspace{-0,25cm}. \hspace{0,05cm}Notes for Section~\ref{bs}]{\hspace{-0,11cm}{\bf{.}} Notes for Section~\ref{bs}}
$\phantom{a}$

\vspace{0,2cm}
\begin{subequations}
\begin{clash}{\rm{\hypertarget{r35}{}\label{case35}\hspace{0,00cm}{\hyperlink{br35}{$\uparrow$}}
  \ In accordance with \eqref{f17zint}  for $a=2$ we have
\begin{align}\label{f3case35}\hspace{-0,2cm}
    &    \eurm{R}_{n}^{{\tn{\triangle}}} (z)\! =\!  \frac{1}{2 \pi i} \int\nolimits_{{\fo{ -1 \! +  \!  2 \imag    }} }^{{\fo{ 1 \! + \! 2  \imag    }} }   \frac{
   \lambda^{\,\prime}\left(\zeta\right)e^{{\fo{- n \pi i  \zeta }}}  }{\lambda(z)  - \lambda\left(\zeta\right)}\, d \zeta\,, \quad
       z \!\in\! \gamma (-1,1)  ,\hspace{-0,1cm}
\end{align}

\noindent where $|\lambda^{\,\prime}(\zeta)|\leqslant 9$ and
$  4 |\lambda(z)-\lambda(\zeta)|\geqslant |\lambda(z)|\, |1-2\lambda(2 i )|$ for all $\zeta \in [-1+2i, 1+2i]$ and
$z \!\in\! \gamma (-1,1)$, in view of Corollary~\ref{intcorol1}.
Applying \eqref{f2auxevagen}, we get
\begin{align}\label{f4case35}
     & \left|\eurm{R}_{n}^{{\tn{\triangle}}} (z)\right|\leqslant
     \dfrac{9 e^{{\fo{2 \pi n  }}}  }{\pi  |\lambda(z)|}\dfrac{11 + 8\sqrt{2}}{21}\leqslant\dfrac{10 e^{{\fo{2 \pi n  }}}  }{\pi  |\lambda(z)|} \ , \quad  z \in \gamma (-1,1)\,.
\end{align}

 For the parametrization $\gamma (-1,1) \ni z = (t+i)/(t-i)$, $t \in \Bb{R}_{> 0}$, taking into account
\begin{align*}
    &  - \dfrac{1}{z} = - \dfrac{t-i}{t+i} = \dfrac{i-t}{i+ t} = \dfrac{1+it}{1-it}= \dfrac{1/t+i}{1/t-i} \, ,
\end{align*}

\noindent and \eqref{f3zint},
\begin{align*}\tag{\ref{f3zint}}
  &    64  \left|\lambda \left(\dfrac{t+i}{t-i}\right)\right|  \geqslant \exp \dfrac{\pi }{2}\left(t + \dfrac{1}{t}\right) \ , \quad  t > 0  \ \Rightarrow \       \\    &
 64  \left|\lambda \left(-1/\left(\dfrac{t+i}{t-i}\right)\right)\right| =
 64  \left|\lambda \left(\dfrac{1/t+i}{1/t-i}\right)\right| \geqslant \exp \dfrac{\pi }{2}\left(t + \dfrac{1}{t}\right) \ , \quad  t > 0 \, , \
\end{align*}

\noindent we obtain
\begin{align}\label{f1case35}
    & \hspace{-0,2cm}\left| \eurm{R}_{n}^{{\tn{\triangle}}} (z)\right| , \   \left| \eurm{R}_{n}^{{\tn{\triangle}}} (-1/z)\right|\!\leqslant\! \dfrac{640 e^{{\fo{2 \pi n  }}}}{\pi} e^{{\fo{-\dfrac{\pi }{2}\left(t + \dfrac{1}{t}\right)}}} \, , \
    z \!= \!\dfrac{t\!+\!i}{t\!-\!i} \ , \  t\! >\! 0\,.\hspace{-0,1cm}
\end{align}

\noindent
Then we deduce from \eqref{f28int} and
\begin{align*}
    & z=\dfrac{t+i}{t-i}=1+ \dfrac{2i}{t-i}  \ , \quad  d z = \dfrac{2i d t}{(t-i)^{2}} \, , \
\end{align*}

\noindent
 that for $z = x + i y\in \Bb{H}$
\begin{align}\label{f7case35}
    &- 4 \pi^{2} n  \eurm{H}_{n} (z)  =   \int\limits_{0}^{\infty}
\frac{2 i \eurm{R}_{n}^{{\tn{\triangle}}} \left(\dfrac{t+i}{t-i}\right) d t}{(t-i)^{2}\left(z+\dfrac{t+i}{t-i}\right)^{2}}  =
\int\limits_{0}^{\infty}
\frac{2 i \eurm{R}_{n}^{{\tn{\triangle}}} \left(\dfrac{t+i}{t-i}\right) d t}{\left(z(t-i)+t+i\right)^{2}}\ ,
\\    &
4 \pi^{2} n  \eurm{M}_{n} (z)  =   \int\limits_{0}^{\infty}
\frac{2 i \eurm{R}_{n}^{{\tn{\triangle}}} \left(\dfrac{1/t+i}{1/t-i}\right) d t}{(t-i)^{2}\left(z+\dfrac{t+i}{t-i}\right)^{2}}=   \int\limits_{0}^{\infty}
\frac{2 i \eurm{R}_{n}^{{\tn{\triangle}}} \left(\dfrac{1/t+i}{1/t-i}\right) d t}{\left(z(t-i)+t+i\right)^{2}} \, , \
\label{f8case35}\end{align}

\noindent where
\begin{align*}
    & z(t-i)+t+i = (x + i y)(t-i)+t+i = x t - i x + i y t + y +t+i  \\    &   =
     x t + y +t + i (y t - x + 1) = t (x+1) + y + i \big(t y - (x-1)\big)  \, , \\    &
     \left| z(t-i)+t+i\right|^{2} = \left(t (x+1) + y\right)^{2} + \left(t y - (x-1)\right)^{2} \\    &    =
     2 t y \big((x+1) - (x-1)\big) + (1 + t^{2})y^{2} + t^{2} (x+1)^{2} + (x-1)^{2}  \\    &    =
    (1 + t^{2})y^{2} + t^{2} (x+1)^{2} + (x-1)^{2} + 4 t y\,.
\end{align*}

\noindent Thus,
\begin{align*}
    &     \left| z(t-i)+t+i\right|^{2} =(1 + t^{2})y^{2} + t^{2} (x+1)^{2} + (x-1)^{2} + 4 t y \\    &    \geqslant t^{2} (y^{2} + (x+1)^{2}) + y^{2} + (x-1)^{2} = |z-1|^{2} + t^{2}|z+1|^{2} \\    &    \geqslant
    \dfrac{|z-1|^{2} +|z+1|^{2}}{1 + \dfrac{1}{t^{2}}} \, , \
\end{align*}

\noindent from which
\begin{align}\label{f5case35}
    & \hspace{-0,2cm} n  \left|\eurm{H}_{n} (z)\right|,  n  \left|\eurm{M}_{n} (z)\right| \leqslant
   \dfrac{320 \pi^{-3} e^{{\fo{2 \pi n  }}} }{|z\!-\!1|^{2}\! +\!|z\!+\!1|^{2}}
    \int\limits_{0}^{\infty}
  \left(1\! + \!\dfrac{1}{t^{2}}\right)  e^{{\fo{-\dfrac{\pi }{2}\left(t + \dfrac{1}{t}\right)}}}d t,\hspace{-0,1cm}
\end{align}

\noindent where
\begin{align*}
    & \int\limits_{0}^{\infty}
  \left(1\! + \!\dfrac{1}{t^{2}}\right)  e^{{\fo{-\dfrac{\pi }{2}\left(t + \dfrac{1}{t}\right)}}}d t=
  2\int\limits_{1}^{\infty}
  \left(1\! + \!\dfrac{1}{t^{2}}\right)  e^{{\fo{-\dfrac{\pi }{2}\left(t + \dfrac{1}{t}\right)}}}d t  \boxed{=} \\    &
  = \left|2 t^{\,\prime} = t + \dfrac{1}{t} \, , \  2 d t^{\,\prime} = \left(1-\dfrac{1}{t^{2}}\right) d t \, , \  \dfrac{1\! + \!\dfrac{1}{t^{2}}}{1\! - \!\dfrac{1}{t^{2}}} = \dfrac{t + \dfrac{1}{t}}{t - \dfrac{1}{t}}\right| \\  &   \left|t^{2} - 2 t  t^{\,\prime} + 1 = 0 \, , \  t = t^{\,\prime} + \sqrt{\left(t^{\,\prime}\right)^{2} -1} \, , \  \dfrac{1}{t} = t^{\,\prime} - \sqrt{\left(t^{\,\prime}\right)^{2} -1}\right| \\    &     \boxed{=}
 \  2\int\limits_{1}^{\infty}
  \dfrac{1\! + \!\dfrac{1}{t^{2}}}{1\! - \!\dfrac{1}{t^{2}}}   e^{{\fo{-\dfrac{\pi }{2}\left(t + \dfrac{1}{t}\right)}}}\left(1-\dfrac{1}{t^{2}}\right)d t =
    2\int\limits_{1}^{\infty}
  \dfrac{t + \dfrac{1}{t}}{t - \dfrac{1}{t}}   e^{{\fo{-\dfrac{\pi }{2}\left(t + \dfrac{1}{t}\right)}}}\left(1-\dfrac{1}{t^{2}}\right)d t \\    &   =
   4\int\limits_{1}^{\infty}\dfrac{t}{\sqrt{t^{2}-1} } e^{-\pi t} d t =
   2\int\limits_{1}^{\infty}\dfrac{1}{\sqrt{t-1} } e^{-\pi\sqrt{ t}} d t = 2\int\limits_{0}^{\infty}\dfrac{1}{\sqrt{t} } e^{-\pi\sqrt{ t +1}} d t \\    &   =
  4 \int\limits_{0}^{\infty}e^{-\pi\sqrt{ t^{2} +1}} d t\leqslant 4
  \int\limits_{0}^{\infty}e^{-\pi t} d t = \dfrac{4}{\pi}
\end{align*}

\noindent where
\begin{align*}
    &  \int\limits_{0}^{\infty}e^{-\pi\sqrt{ t^{2} +1}} d t= \sum\limits_{n\geqslant 0} \int\limits_{n}^{n+1}e^{-\pi\sqrt{ t^{2} +1}} d t\leqslant\sum\limits_{n\geqslant 0} e^{-\pi\sqrt{ n^{2} +1}} = \\ & =e^{-\pi} + \sum\limits_{n\geqslant 1}e^{-\pi n} = \dfrac{1}{e^{\pi}} + \dfrac{1}{e^{\pi} +1} = \\ & =
    0,04321391 + 0,04142383216636282681=
    0,08463775043
\end{align*}

\noindent But actually \cite[p.\! 82, (19)]{erd}, \cite[p.\! 376, 9.6.27]{abr}
\begin{align*}
    &  K_{0} (z) = \int\limits_{1}^{\infty}\dfrac{1}{\sqrt{t^{2}-1} } e^{-z t} d t  \ \Rightarrow \
    \int\limits_{1}^{\infty}\dfrac{t}{\sqrt{t^{2}-1} } e^{-z t} d t = - K_{0}^{\,\prime} (z)  \ \Rightarrow \   \\    &     \int\limits_{0}^{\infty}e^{-\pi\sqrt{ t^{2} +1}} d t=  \int\limits_{1}^{\infty}\dfrac{t}{\sqrt{t^{2}-1} } e^{-\pi t} d t =- K_{0}^{\,\prime} (\pi)= K_{1} (\pi) \,.
\end{align*}

\noindent So that
\begin{align*}
    &  n  \left|\eurm{H}_{n} (z)\right|,  n  \left|\eurm{M}_{n} (z)\right| \leqslant \dfrac{320}{ \pi^{3}}
    \dfrac{ e^{{\fo{2 \pi n  }}} }{|z\!-\!1|^{2}\! +\!|z\!+\!1|^{2}}
    4 K_{1} (\pi) \ , \  \\  &
\dfrac{1280 K_{1} (\pi)}{ \pi^{3}} < \dfrac{1280 }{ \pi^{3}}  \left(\dfrac{1}{e^{\pi}} + \dfrac{1}{e^{\pi} +1} \right)<4 \ , \
\end{align*}

\noindent and hence,
\begin{align}\label{f6case35}
     & n  \left|\eurm{H}_{n} (z)\right|,  n  \left|\eurm{M}_{n} (z)\right| \leqslant
    \dfrac{4 e^{{\fo{2 \pi n  }}} }{|z\!-\!1|^{2}\! +\!|z\!+\!1|^{2}} \ , \
    z \in \Bb{H}\cup \Bb{R}\,.
\end{align}

}}\end{clash}
\end{subequations}

\begin{subequations}
\begin{clash}{\rm{\hypertarget{r15}{}\label{case15}\hspace{0,00cm}{\hyperlink{br15}{$\uparrow$}}
  \
  Using  \cite[p.\! 44, 1.422.4]{gra}, we obtain the validity of the left-hand  side equality in
\begin{align*}
    &  \sum\limits_{k \in \Bb{Z}} \frac{1}{(2 k + z)^{2}} =  \frac{\pi^{2}}{4 \sin^{2} \dfrac{\pi z}{2}}
    =  - \pi^{2} \sum_{m \geqslant 1}m e^{{\fo{ i \pi m z}}}
    \ ,  \qquad z \in \Bb{H}     \, , \
\end{align*}

\noindent while the validity of the right-hand  side equality here follows from the following identities
\begin{align*}
    & \sin^{2} \dfrac{\pi z}{2} = \left(\dfrac{e^{{\fo{\dfrac{i \pi z}{2}}}} - e^{{\fo{{\nor{-}}\dfrac{i \pi z}{2}}}}}{2 i}\right)^{2} = -
\dfrac{e^{{\fo{i \pi z}}}}{4} \left(1 - e^{{\fo{- i \pi z}}}\right)^{2} \ , \ \\    &
\dfrac{1}{\sin^{2} \dfrac{\pi z}{2}} = - \dfrac{4 e^{{\fo{- i \pi z}}}}{ \left(1 - e^{{\fo{- i \pi z}}}\right)^{2}} \ , \\    &
\dfrac{1}{1-u} = \sum_{m \geqslant 0} u^{m} \ , \  \dfrac{1}{(1-u)^{2}} = \sum_{m \geqslant 1} m u^{m-1}  \ , \ u \in \Bb{D}, \\    &
\dfrac{1}{\sin^{2} \dfrac{\pi z}{2}} = - \dfrac{4 e^{{\fo{ i \pi z}}}}{ \left(1 - e^{{\fo{ i \pi z}}}\right)^{2}} =
- 4 e^{{\fo{ i \pi z}}} \sum_{m \geqslant 1} m e^{{\fo{ i \pi ( m -1) z}}} = - 4 \sum_{m \geqslant 1}m e^{{\fo{ i \pi m z}}} \ .\end{align*}

}}\end{clash}
\end{subequations}

\vspace{0,2cm}
\begin{subequations}
\begin{clash}{\rm{\hypertarget{r16}{}\label{case16}\hspace{0,00cm}{\hyperlink{br16}{$\uparrow$}}
  \ For any $z\in \Bb{C}\setminus \Bb{R}$ it follows from \cite[p.\! 44, 1.421.3]{gra} that
\begin{align*}
     &\frac{\pi}{2}  \cot \dfrac{\pi z}{2}=
 \frac{1}{x}+ \sum\limits_{k \geqslant 1}\left(\dfrac{1}{z+ 2k} + \dfrac{1}{z- 2k}\right)
= \lim\limits_{N\to +\infty} \sum\limits_{ {\fo{-N \leqslant  k \leqslant N }} } \frac{1}{2 k +z }  \, ,
\end{align*}

\noindent and if $x\in\Bb{R}$ then
\begin{align*}  &
     \sum\limits_{ {\fo{ k \in \Bb{Z} }} } \dfrac{2  z}{(2 k + x)^{2} - z^{2}} =     \sum\limits_{ {\fo{ k \in \Bb{Z} }} }
     \left(\frac{1}{2  k + x -  z } -  \frac{1}{2 k + x +  z }\right)   \\ & =
\lim\limits_{N\to +\infty} \sum\limits_{ {\fo{-N \leqslant  k \leqslant N }} } \left(\frac{1}{2  k + x -  z } -  \frac{1}{2 k + x +  z }\right)
\\    &   =
\frac{\pi}{2} \cot \dfrac{\pi(x -  z)}{2}  - \frac{\pi}{2}\cot \dfrac{\pi(x +  z)}{2} =
-\frac{\pi}{2} \cot \dfrac{\pi(z -  x)}{2}  - \frac{\pi}{2}\cot \dfrac{\pi(x +  z)}{2}
 \ ,     \end{align*}

\noindent where for arbitrary $z \in \Bb{H}$ we have
\begin{align*}
 &   \cot\dfrac{\pi z}{2}
=i  \dfrac{e^{i \pi  z/2}+ e^{-i \pi  z/2}}{e^{i \pi  z/2}- e^{-i \pi  z/2}}=
i  \dfrac{e^{i \pi  z}+ 1}{e^{i \pi  z}- 1} =
- i  \dfrac{1+ e^{i \pi  z}}{1- e^{i \pi  z}}
= - i - 2 i \sum\limits_{ {\fo{ m  \geqslant 1  }} }   e^{i \pi m z} \ ,  \\  &    \dfrac{4}{\pi}\sum\limits_{ {\fo{ k \in \Bb{Z} }} } \dfrac{ z}{(2 k + x)^{2} - z^{2}} =
 - \cot \dfrac{\pi(z -  x)}{2}  - \cot \dfrac{\pi(z +  x)}{2}  \\    &    =  2 i +
 2 i \sum\limits_{ {\fo{ m  \geqslant 1  }} }  \left( e^{i \pi m z} e^{- i \pi m x}  +  e^{i \pi m z} e^{ i \pi m x}\right) =  2 i + 4 i \sum\limits_{ {\fo{ m  \geqslant 1  }} }e^{i \pi m z} \cos \pi m x \ , \
\end{align*}

\noindent which proves the identity
\begin{align*}
       \sum\limits_{ {\fo{ k \in \Bb{Z} }} } \dfrac{ z}{(2 k + x)^{2} - z^{2}}&=
- \dfrac{\pi}{4}\left(\cot \dfrac{\pi(z -  x)}{2}  + \cot \dfrac{\pi(z +  x)}{2}\right)
 \\    &    =
   \dfrac{\pi i}{2} +  i \pi \sum\limits_{ {\fo{ m  \geqslant 1  }} }e^{i \pi m z} \cos \pi m x
    \ ,  \qquad z \in \Bb{H}  \, , \  x \in \Bb{R}    .
\end{align*}

}}\end{clash}
\end{subequations}


\subsection[\hspace{-0,25cm}. \hspace{0,05cm}Notes for Section~\ref{bsden}]{\hspace{-0,11cm}{\bf{.}} Notes for Section~\ref{bsden}}

\vspace{0,5cm}
\vspace{0,2cm}
\begin{subequations}
\begin{clash}{\rm{\hypertarget{r17}{}\label{case17}\hspace{0,00cm}{\hyperlink{br17}{$\uparrow$}}
  \  Let $\phi\! \in\! \Gamma_{\vartheta} $. Then by Lemma 2 of
  \cite[p.\! 112]{cha}{\hyperlink{r43}{${}^{\ref*{case43}}$}}\hypertarget{br43b}{} the transform $\phi$ can be represented as a superposition  of a finite number of the degrees of the transformations $z+2$ and $-1/z$ (uniqueness of such representation has not been proved by Chandrasekharan in \cite{cha}).
In the notation of \eqref{f1pf8contgen} and in view of \eqref{f2pf8contgen}, this means that either $\phi(z)\in \{z, -1/z\}= \{\phi_{0}(z), \phi_{0}(-1/z)\}=\{\phi_{0}(z), \  -1/\phi_{0}(z) , \  \phi_{0}(-1/z)  , \  -1/\phi_{0}(-1/z)\}$ or there exist
    $\alpha, \beta \in \{0, 1\}$, $N \in \Bb{N}$ and $n_{1}, \ldots, n_{N} \in \Bb{Z}_{\neq 0}  $ such that (see \cite[p.\! 63]{bh1})
\begin{align*}
     & \phi(z)\!= \!
 \phi_{ {\fo{S^{{\fo{\,\alpha}}} S T^{{\fo{-2 n_{1}}}} \ldots S T^{{\fo{-2 n_{N-1}}}} S T^{{\fo{-2 n_{N}}}}
        S^{{\fo{\,\beta}}}     }}}(z)= \phi_{ {\fo{S^{{\fo{\,\alpha}}}}} }
\left(\phi_{{\fo{n_{N}, ..., n_{1}}}} \left(
 \phi_{ {\fo{S^{{\fo{\,\beta}}}}} } (z)\right)\right)\! ,
\end{align*}

\noindent where \eqref{f3pf8contgen} is used for $\eufm{n}\!:=\!(n_{N}, ..., n_{1})\in \Bb{Z}_{\neq 0}^{\hspace{0,02cm}\Bb{N}_{\hspace{-0,02cm}\eurm{f}}}\!:=\!\cup_{k\geqslant 1} \Bb{Z}_{\neq 0}^{k}$ and $z\in \Bb{H}$. In view of \eqref{f7contgen} and
\eqref{f1zbsden},
\begin{align*}
     & \phi_{\eufm{n}}\!:= \!\phi_{{\fo{n_{N}, ..., n_{1}}}}  \in \Gamma_{\vartheta}^{\hspace{0,015cm}{\tn{||}}} \, ,
\end{align*}

\noindent and by setting in the above equalities to $(\alpha, \beta)$ each of four possible values
$(0,0), (1,0), (0,1), (1,1)$ we obtain   that
\begin{align*}
     & \phi(z)\in \left\{\phi_{\eufm{n}}(z), \  -1/\phi_{\eufm{n}}(z) , \  \phi_{\eufm{n}}(-1/z)  , \  -1/\phi_{\eufm{n}}(-1/z)\right\} \, , \
\end{align*}

\noindent which completes the proof of \eqref{f4contgen}.
}}\end{clash}
\end{subequations}

\vspace{0,2cm}
\begin{subequations}
\begin{clash}{\rm{\hypertarget{r43}{}\label{case43}\hspace{0,00cm}{\hyperlink{br43a}{$\uparrow$}}\hspace{0,15cm}{\hyperlink{br43b}{$\uparrow$}}
  \     Lemma 2 of \cite[p.\! 112]{cha} can be sharpened as follows.
\begin{theorem}\hspace{-0,2cm}{\bf{.}}\label{case17th1}
   For arbitrary $A \in   {\rm{SL}}_{2} (\vartheta, \Bb{Z})\setminus \{I, - I, S , - S\}$ there exists a unique collection of numbers $\sigma \in \{1, -1\}$,
    $\alpha, \beta \in \{0, 1\}$, $N \in \Bb{N}$ and $n_{1}, \ldots, n_{N} \in \Bb{Z}_{\neq 0}  $ such that
    \begin{align}\label{f1case17th1}
        &  A  =  \sigma \cdot S^{{\fo{\,\alpha}}} \Big(S T^{{\fo{-2 n_{1}}}} \ldots S T^{{\fo{-2 n_{N-1}}}} S T^{{\fo{-2 n_{N}}}}\Big)
        S^{{\fo{\,\beta}}}  ,
    \end{align}

\noindent where
\begin{align*}  &
     {\rm{SL}}_{2} (\vartheta, \Bb{Z})\! =\! \left\{ \, {\fo{\begin{pmatrix} a & b \\ c & d \\ \end{pmatrix}}}\! \in\! {\rm{SL}}_{2} (\Bb{Z}) \ \left| \ {\fo{\begin{pmatrix} a & b \\ c & d \\\end{pmatrix}}}\! \equiv\! {\fo{ \begin{pmatrix} 1 & 0 \\ 0 & 1 \\\end{pmatrix}}} ({\rm{mod}}\, 2)\  \ \mbox{or} \
     {\fo{\begin{pmatrix} a & b \\ c & d\\ \end{pmatrix}}}\! \equiv\!
    {\fo{ \begin{pmatrix} 0 & 1 \\ 1 & 0 \\ \end{pmatrix}}}  ({\rm{mod}}\, 2)
\, \right\}\right. \! ,\\[0.3cm]
    &  T  =  {\fo{\begin{pmatrix} 1 & 1 \\ 0 & 1  \\ \end{pmatrix}}} \not\in {\rm{SL}}_{2} (\vartheta, \Bb{Z}), \ T^{\pm 2}  = {\fo{\begin{pmatrix} 1 & \pm  2 \\ 0 & 1  \\ \end{pmatrix}}}\in
     {\rm{SL}}_{2} (\vartheta, \Bb{Z}),
     \   S  =  {\fo{\begin{pmatrix} 0 & 1 \\ -1 & 0 \\ \end{pmatrix}}}\in {\rm{SL}}_{2} (\vartheta, \Bb{Z}),
\end{align*}

\vspace{0.25cm}
\noindent
$S^{2} =  - I$,  $I:= (\begin{smallmatrix} 1 & 0 \\ 0 & 1 \end{smallmatrix}) $,
and
${\rm{SL}}_{2} (\Bb{Z})$ denotes the set of all $2\times2$ matrices $(\begin{smallmatrix} a & b \\ c & d \end{smallmatrix})$ with integer coefficients  $a, b, c, d \!\in\! \Bb{Z}$ satisfying  $ad - bc\!=\!1$.
\end{theorem}

The fact that $ {\rm{SL}}_{2} (\vartheta, \Bb{Z}) $ is invariant under multiplication of matrices follows from the possibility of a termwise multiplying the congruences  (see \cite[p.\! 107, Theorem 5.2(b)]{apo}) and $ (\begin{smallmatrix} 0 & 1 \\ 1 & 0 \end{smallmatrix}) (\begin{smallmatrix} 0 & 1 \\ 1 & 0 \end{smallmatrix}) =(\begin{smallmatrix} 1 & 0 \\ 0 & 1 \end{smallmatrix})$,
  while  belonging to $ {\rm{SL}}_{2} (\vartheta, \Bb{Z}) $ of each its inverse matrix is a consequence of   the  general formula
\begin{gather}\label{f4preres}
   M = \begin{pmatrix} a & b \\ c & d \\ \end{pmatrix} \in \mathrm{SL}_2 (\Bb{Z}) \ \ \Rightarrow \ \
M^{-1} = \begin{pmatrix} d & -b \\ -c & a \\ \end{pmatrix} \in \mathrm{SL}_2 (\Bb{Z}) \ .
\end{gather}

\noindent As well as  ${\rm{SL}}_{2} (\Bb{Z})$, the set  $ {\rm{SL}}_{2} (\vartheta, \Bb{Z}) $ is also invariant with respect to the transpose operation
\begin{gather}\label{f5apreres}
  {\rm{SL}}_{2} (\vartheta, \Bb{Z}) \ni  M = \begin{pmatrix} a & b \\ c & d \\ \end{pmatrix}  \ \  \ \Rightarrow \   \ \
M^{\ast} = \begin{pmatrix} a & c \\ b & d  \\ \end{pmatrix} \in  {\rm{SL}}_{2} (\vartheta, \Bb{Z}) \ .
\end{gather}
\noindent
Everywhere below we use the notation $\{\pm I, \pm S\} := \{I, - I, S , - S\}$.

Before proving Theorem~\ref{case17th1}, we first establish several special
properties of matrices from the set $ {\rm{SL}}_{2} (\vartheta, \Bb{Z})$.
Each column and each row of any matrix from the set $ {\rm{SL}}_{2} (\vartheta, \Bb{Z})$ contains two numbers of different parity and therefore they cannot be equal to each other. In fact,  a much stronger property of these matrices takes place.
\begin{lemma}\hspace{-0,2cm}{\bf{.}}\label{case17lem1}  Let   $ A = (\begin{smallmatrix} a & b \\ c & d \end{smallmatrix}) \in  {\rm{SL}}_{2} (\vartheta, \Bb{Z})\setminus \{I, - I, S , - S\}$. Then the value of $\min\{|a|, |b|, |c|, |d|\} $ is attained on only one element.
\end{lemma}
\begin{proof}\hspace{-0,15cm}{\bf{.}}
    Suppose to the contrary that $\min\{|a|, |b|, |c|, |d|\} $ is attained on more than one element. Since the pairs $(a, b)$, $(a, c)$, $(c, d)$ and $(b, d)$ consist of numbers with different parity then the only one of two following cases  can take place
\begin{align}\label{f1preresbaspropprlem1}
    &
 \begin{array}{l} {\rm{(1)}}\quad
|a| = |d| = \min\{|a|, |b|, |c|, |d|\} < \min\{ |b|, |c|\} \ , \  \\
 {\rm{(2)}}\quad   |b| = |c| = \min\{|a|, |b|, |c|, |d|\} < \min\{|a|,  |d|\} \ .
 \end{array} \end{align}

\noindent If \eqref{f1preresbaspropprlem1}(1) holds then
 \begin{align}\label{f3preresbaspropprlem1}
    & |a| = |d| = \min\{|a|, |b|, |c|, |d|\} \leqslant  \min\{ |b|, |c|\} -1 \ , \
\end{align}

\noindent from which $ |a d|\leqslant \left( \min\{ |b|, |c|\} -1\right)^{2} \leqslant \left(\min\{ |b|, |c|\}\right)^{2} -2 \min\{ |b|, |c|\} + 1 \leq |b c| + 1 - 2 \min\{ |b|, |c|\}$ and therefore $2 \min\{ |b|, |c|\}-1 \leqslant |b c|- |a d|\leqslant |b c  -  ad| =1$, i.e.,\vspace{-0,2cm}
\begin{align}\label{f4preresbaspropprlem1}
    &  \min\{ |b|, |c|\} \leqslant 1 \ .
\end{align}

\noindent Substituting this inequality in \eqref{f3preresbaspropprlem1} we get $a=d=0$, while \eqref{f4preresbaspropprlem1} and $ad- bc =1$ mean that
$b c = -1$ and  $|b|= |c| = 1$. These conditions are satisfied only when $A=\pm S$ that contradicts  $A \not\in\{\pm I, \pm S\}$. Thus,  \eqref{f1preresbaspropprlem1} cannot be valid.

If \eqref{f1preresbaspropprlem1}(2) holds then the matrix $ A S=  (\begin{smallmatrix}  -b & a \\ - d & c \\ \end{smallmatrix}) $ satisfies the condition  \eqref{f1preresbaspropprlem1}(1) and using the above reasoning we conclude that either $A S = S$ or $A S = -S$, i.e.,
either $A  = I$ or $A  = -I$. This also contradicts  $A\not\in\{\pm I, \pm S\}$ and finishes the proof of Lemma~\ref{case17lem1}. $\square$
\end{proof}

We say that  $ A = (\begin{smallmatrix} a & b \\ c & d \end{smallmatrix}) \in  {\rm{SL}}_{2} (\vartheta, \Bb{Z})$ is {\emph{modulo isotonic}} if
\begin{align}\label{f1auxdef1}
    &  |a|     <    |b| <      |d|  \quad \mbox{ and}  \quad |a|     <    |c| <      |d| \ .
\end{align}

\noindent
Lemma~\ref{case17lem1} can be essentially sharpened as follows.

\begin{lemma}\hspace{-0,18cm}{\bf{.}}\label{case17lem2}  Let   $ A = (\begin{smallmatrix} a & b \\ c & d \end{smallmatrix}) \in  {\rm{SL}}_{2} (\vartheta, \Bb{Z})\setminus \{\pm I, \pm S\}$.  Then among four matrices
\begin{align}\label{f1case17lem2}
    & A:= {\fo{\begin{pmatrix} a & b \\ c & d \\ \end{pmatrix}}} \ , \
A S= {\fo{ \begin{pmatrix} -b & a \\ - d & c \\ \end{pmatrix} }}\ , \
S A={\fo{\begin{pmatrix} c & d \\ -a & -b \\ \end{pmatrix}}} \ , \
S A S={\fo{\begin{pmatrix} -d & c \\ b & -a \\ \end{pmatrix}}} \ ,
\end{align}

\noindent each of which belongs to ${\rm{SL}}_{2} (\vartheta, \Bb{Z})\setminus \{\pm I, \pm S\}$, there exists exactly one modulo isotonic matrix.
\end{lemma}
\begin{proof}\hspace{-0,15cm}{\bf{.}} Four matrices $\{\pm I, \pm S\}$ form the group with the matrix multiplication as a group operation,  and therefore one of the matrices from \eqref{f1case17lem2} can belong to $\{\pm I, \pm S\}$ if and only if $\{\pm I, \pm S\}$ contains the matrix $A$.
  It  readily follows from Lemma~\ref{case17lem1} and \eqref{f1case17lem2} that  there exists a unique ordered pair $(\alpha, \beta)$ of numbers $\alpha, \beta \in \{0, 1\}$ such that the  matrix $S^{\alpha} A S^{\beta}$ at the intersection of its first row and first column  contains exactly that  element of $A$ which has the smallest absolute value. By designating the elements of the transformed matrix $S^{\alpha} A S^{\beta}$ by the same letters, we have
\begin{align}\label{f2preresbaspropprlem1}
    & B :=  S^{\alpha} A S^{\beta} = \begin{pmatrix}a& b \\ c & d \\ \end{pmatrix} \ , \  |a| \leqslant |b| -1  \ , \  |a| \leqslant |c| -1 \ , \  |a| \leqslant |d|-1 \ .
\end{align}

\noindent
To prove that $B$ is modulo isotonic, it is sufficient to show that $|b| < |d|$ and $|c| < |d|$.

Assume first that\vspace{-0.3cm}
\begin{align}\label{f5preresbaspropprlem1}
    & |d| \leqslant |b| \ .
\end{align}

\noindent  By virtue of the different parity of $|b|$ and $|d|$, we have  $|d| \leqslant  |b| -1 $ and therefore
$|a d| \leqslant |a | |b| - |a| \leqslant ( |c| -1) |b| - |a| = |b | |c| - |b| - |a|$, i.e., $|a| + | b|\leqslant |b | |c| -|a d| $ and
$1 = |bc - ad|  \geqslant   |b | |c| - |a| | d| \geqslant  |a| + |b| \geqslant     2  |a| + 1$, which yields $a =0$.
Then $ad-bc = 1$ implies $bc = -1$ and therefore $|c| =|b| =1$. Our assumption $|d| \leqslant |b| -1 = 0$ gives $d=0$.
Finally, we obtain $a=d=0$ and   $|c| =|b| =1$. This holds if and only if  $B =\sigma \cdot S$ for some $\sigma \in \{1, -1\}$.
By solving the equation $ S^{\alpha} A S^{\beta} =\sigma \cdot S$  we get $A =(-1)^{\alpha +\beta}\sigma \cdot S^{\alpha +\beta+1} \in
\{I, - I, S , - S\}$ which contradicts  $A\not\in\{\pm I, \pm S\}$. This contradiction  proves $|b| < |d|$.

If we assume that  $ |d| \leqslant |c|$ then the conjugate matrix
$B^{*} = (\begin{smallmatrix} a & c \\ b & d \end{smallmatrix})$ also belongs to ${\rm{SL}}_{2} (\vartheta, \Bb{Z})$ and satisfies the conditions
\eqref{f2preresbaspropprlem1} and \eqref{f5preresbaspropprlem1}. By application of the above  reasoning we conclude that $B^{*} =\sigma \cdot S$ for some $\sigma \in \{1, -1\}$. Then  $B =(-\sigma) \cdot S$ and similarly to the above $A =(-1)^{\alpha +\beta+1}\sigma \cdot S^{\alpha +\beta+1} \in
\{I, - I, S , - S\}$, which contradicts  $A\not\in\{\pm I, \pm S\}$. This contradiction  proves  $|c| < |d|$ and  that $B$  is modulo isotonic. But if one of the matrices  in \eqref{f1case17lem2} is modulo isotonic then no other one in \eqref{f1case17lem2} can have this property because all other matrices contain the unique element of $A$ with the smallest absolute value at the wrong place. Lemma~\ref{case17lem2} is  proved. $\square$
\end{proof}

The next assertion is immediate from Lemma~\ref{case17lem2}.
\begin{corollary}\hspace{-0,15cm}{\bf{.}}\label{case17cor1}
 Let $A = (\begin{smallmatrix} a & b \\ c & d \end{smallmatrix}) \in  {\rm{SL}}_{2} (\vartheta, \Bb{Z})$ satisfy $|a|    \leqslant    |b| \leqslant      |d| $ and
$|a|     \leqslant   |c| \leqslant      |d|$.  Then $|a|     <    |b| <      |d|$ and $|a|     <    |c| <      |d|$, i.e., $A$ is modulo isotonic.
\end{corollary}

The following result shows that any  modulo isotonic matrix from the set ${\rm{SL}}_{2} (\vartheta, \Bb{Z})$ is uniquely determined by its second column.
\begin{lemma}\hspace{-0,15cm}{\bf{.}}\label{case17lem3}
  Let two matrices $A_{1}\! = \!(\begin{smallmatrix} a_{1} & b_{1} \\ c_{1} & d_{1} \end{smallmatrix})$ and   $A_{2}\! = \!(\begin{smallmatrix} a_{2} & b_{2} \\ c_{2} & d_{2} \end{smallmatrix})$ belong to ${\rm{SL}}_{2} (\vartheta, \Bb{Z})$ and be  modulo isotonic. Suppose that $b_{1}\! = \!b_{2}$ and $d_{1}\! = \! d_{2}$.
    Then $A_{1} \! = \! A_{2}$.
\end{lemma}
\begin{proof}\hspace{-0,15cm}{\bf{.}}
Assume that two different modulo isotonic matrices in ${\rm{SL}}_{2} (\vartheta, \Bb{Z})$ satisfy
\begin{align}\label{f0preresbaspropprlem2}
    &  A_{1}= \begin{pmatrix} a_{1} & b \\ c_{1} & d \\ \end{pmatrix}  \ , \
     A_{2}= \begin{pmatrix} a_{2} & b \\ c_{2} & d \\ \end{pmatrix}  \ ; \quad {\rm{(1)}} \ \  ( a_{1} -a_{2})^{2} +     ( c_{1} -c_{2})^{2} > 0 \ ,
\end{align}

\noindent In view of \eqref{f1auxdef1}, $|d| > |b| > 0$ and it follows from $a_{1}d - bc_{1} = 1$, $a_{2}d - bc_{2} = 1$ that
\begin{align}\label{f1preresbaspropprlem2}
    & {\rm{(1)}} \ \,  (a_{1} - a_{2}) d =  (c_{1} - c_{2}) b\ ; \quad
    {\rm{(2)}} \ \,  {\rm{gcd}}(d, b) = 1 \ ; \quad {\rm{(3)}} \ \,   d, b \in \Bb{Z}\setminus\{0\}  \    .
\end{align}

\noindent By virtue of \eqref{f1preresbaspropprlem2}(1) and \eqref{f1preresbaspropprlem2}(3), the property \eqref{f0preresbaspropprlem2}(1) yields that
\begin{align}\label{f2preresbaspropprlem2}
  a_{1} - a_{2} \neq 0  \ , \ c_{1} - c_{2} \neq  0 \ .
\end{align}

\noindent But \eqref{f1preresbaspropprlem2}(1) and \eqref{f1preresbaspropprlem2}(2) implies that  $d$ divides $(c_{1} - c_{2}) b$  and since ${\rm{gcd}}(d, b) = 1$   we get that $d$ divides $(c_{1} - c_{2})$. Hence, taking account of \eqref{f2preresbaspropprlem2} and \eqref{f1preresbaspropprlem2}(3), there exists $m \in \Bb{Z}_{\neq 0}$ such that $c_{1} -  c_{2} = m d$. Substituting this in  \eqref{f1preresbaspropprlem2}(1) we obtain $(a_{1} - a_{2}) d = m d b$, where we can divide by nonzero number $d$ and derive from \eqref{f1auxdef1} that
\begin{align}&\nonumber
         {\rm{(1)}} \ \   a_{1} =  a_{2} + m b  \ , \ |a_{1}|, |a_{2}| \leqslant |b|-1 \ ; \quad
    {\rm{(2)}} \ \     c_{1} =  c_{2} + m d \ , \ |c_{1}|, |c_{2}| \leqslant |d|-1 \ ; \\[0.2cm]  &
    {\rm{(3)}} \ \  m \in \Bb{Z}\setminus\{0\} \ .
\label{f3preresbaspropprlem2}\end{align}

\noindent It follows from \eqref{f3preresbaspropprlem2}(1) that $|m b| \leq 2|b|-2$ and since  by \eqref{f1preresbaspropprlem2}(3) $b$ is nonzero we conclude that the condition \eqref{f3preresbaspropprlem2}(3) can be replaced by
\begin{align}\label{f4preresbaspropprlem2}
   m \in \{1,-1\} \ .
\end{align}

\noindent According to the definition ${\rm{SL}}_{2} (\vartheta, \Bb{Z})$, we have only two possibilities
\begin{align}\label{f5preresbaspropprlem2}
    &
    \begin{array}{l}
   {\rm{(1)}} \ \   \begin{pmatrix} a_{2} & b \\ c_{2} & d \\\end{pmatrix}\! \equiv\! \begin{pmatrix} a_{1} & b \\ c_{1} & d \\\end{pmatrix}\! \equiv\! \begin{pmatrix} 1 & 0 \\ 0 & 1 \\\end{pmatrix} ({\rm{mod}}\, 2) \quad \mbox{or} \\[0,5cm]
     {\rm{(2)}} \ \ \begin{pmatrix} a_{2} & b \\ c_{2} & d\\ \end{pmatrix}\! \equiv\!  \begin{pmatrix} a_{1} & b \\ c_{1} & d\\ \end{pmatrix}\! \equiv\! \begin{pmatrix} 0 & 1 \\ 1 & 0 \\ \end{pmatrix}  ({\rm{mod}}\, 2) \ .
    \end{array}
\end{align}

Assume that \eqref{f5preresbaspropprlem2}(1) holds.  But   \eqref{f3preresbaspropprlem2}(2)  and  \eqref{f4preresbaspropprlem2} yield  $c_{1} \equiv c_{2} + m d \equiv 1 ({\rm{mod}}\, 2)$ which contradicts  \eqref{f5preresbaspropprlem2}(1).

Suppose now that \eqref{f5preresbaspropprlem2}(2) holds.  Then   \eqref{f3preresbaspropprlem2}(1) and  \eqref{f4preresbaspropprlem2}  yield  $a_{1} \equiv a_{2} + m b \equiv 1 ({\rm{mod}}\, 2)$ which contradicts  \eqref{f5preresbaspropprlem2}(2).

\noindent
These contradictions together with \eqref{f0preresbaspropprlem2} complete the proof of Lemma~\ref{case17lem3}.
 $\square$ \end{proof}

\begin{lemma}\hspace{-0,15cm}{\bf{.}}\label{case17lem4}
  Let $B = (\begin{smallmatrix} a & b \\ c & d \end{smallmatrix}) \in  {\rm{SL}}_{2} (\vartheta, \Bb{Z})$
 be modulo isotonic {\rm{(}}see \eqref{f1auxdef1}{\rm{)}}. Then there exists a unique collection of numbers
    $\gamma \in \{0, 1\}$, $N \in \Bb{N}$ and $n_{1}, \ldots, n_{N} \in \Bb{Z}_{\neq 0}  $ such that
\begin{align}\label{f1case17lem4}
    &  (-1)^{\gamma} B\! =\!
    S T^{{\fo{-2 n_{N}}}}  S T^{{\fo{-2 n_{N-1}}}}\ldots S T^{{\fo{-2 n_{1}}}}.
\end{align}
\end{lemma}
\begin{proof}\hspace{-0,15cm}{\bf{.}}
 By \eqref{f1auxdef1} and the definition of ${\rm{SL}}_{2} (\vartheta, \Bb{Z})$, $b$ and $d$ are nonzero integers of different parity, $b/d \in (-1,1)\setminus\{0\} $ and it follows from  $ad - bc\!=\!1$ that ${\rm{gcd}}(b, d)= 1$. By applying Lemma~\ref{case39lem1} to the rational number $b/d$ we obtain the existence of
the unique $N\! \in \!\Bb{N}$ and $(n_{N}, ..., n_{1})\!\in\! \Bb{Z}_{\neq 0}^{N}$ such that
\begin{align*}
    & \dfrac{b}{d} = \phi_{{\fo{\hspace{0,005cm}n_{N}, ..., n_{1}}}}(0):=
  {\fo{\dfrac{1}{2n_{N}\! -\! \dfrac{1}{ \begin{subarray}{l}
\begin{displaystyle} \vphantom{A^{A}}
2n_{N-1}\!- \!\end{displaystyle}\\[-0.4cm]
\hphantom{2n_{k-1} 11 -} \ddots\\[-0.4cm]
\hphantom{2n_{k-1} 11 - \ddots} \begin{displaystyle} - \!\end{displaystyle} \dfrac{1}{2 n_{2} \!- \! \dfrac{1}{
\vphantom{A^{A}}2 n_{ 1}  }}\end{subarray} }}}}
  \ ,  \quad (n_{N}, ..., n_{1})\!\in\! \Bb{Z}_{\neq 0}^{N} \ , \ N\! \in \!\Bb{N} \,,
\end{align*}

\vspace{-0,2cm}
\noindent where
\begin{align*}
    &  \phi_{\eufm{n}}(z)\!= \! \dfrac{z p^{\eufm{n}}_{N-1} + q^{\eufm{n}}_{N-1}}{z p^{\eufm{n}}_{N} + q^{\eufm{n}}_{N}}=
    {\fo{\dfrac{1}{2n_{N}\! -\! \dfrac{1}{ \begin{subarray}{l}
\begin{displaystyle} \vphantom{A^{A}}
2n_{N-1}\!- \!\end{displaystyle}\\[-0.4cm]
\hphantom{2n_{k-1} 11 -} \ddots\\[-0.4cm]
\hphantom{2n_{k-1} 11 - \ddots} \begin{displaystyle} - \!\end{displaystyle} \dfrac{1}{2 n_{2} \!- \! \dfrac{1}{
\vphantom{A^{A}}2 n_{ 1} -z }}\end{subarray} }}}} \, , \   \ z \in [-1,1]\cup (\Bb{C}\setminus\Bb{R})\,,
\end{align*}

\noindent in view of \eqref{f6zpf8contgen}  and \eqref{f1bsden}. Therefore there exists a unique $\gamma \in \{0, 1\}$ such that $b = (-1)^{\gamma} q^{\eufm{n}}_{N-1}$ and $d = (-1)^{\gamma} q^{\eufm{n}}_{N}$.
By virtue of \eqref{f0preresauxlem1}, \eqref{f1preresauxlem2} and \eqref{f7pf8contgen}, two modulo
isotonic matrices
\begin{align*}
    &  \begin{pmatrix} p^{\eufm{n}}_{N-1} & q^{\eufm{n}}_{N-1}  \\ p^{\eufm{n}}_{N} & q^{\eufm{n}}_{N} \\ \end{pmatrix}\! =  \!\begin{pmatrix}0 & 1 \\ -1 &  2 n_{N} \\ \end{pmatrix} \begin{pmatrix}0 & 1 \\ -1 &  2 n_{N-1} \\ \end{pmatrix}\ldots \begin{pmatrix}0 & 1 \\ -1 &  2 n_{1} \\ \end{pmatrix}
\end{align*}

\noindent and $(-1)^{\gamma} B$ belong to ${\rm{SL}}_{2} (\vartheta, \Bb{Z})$ and have the same second column. In accordance with Lemma~\ref{case17lem3} and \eqref{f2pf8contgen}, this means that \eqref{f1case17lem4} holds and that Lemma~\ref{case17lem4}  is proved.
$\square$ \end{proof}

 \vspace{0.15cm}
 \noindent
{\bf{Proof of Theorem~\ref{case17th1}.}} Let $ A \in  {\rm{SL}}_{2} (\vartheta, \Bb{Z})\setminus \{\pm I, \pm S\}$. According to Lemma~\ref{case17lem2},  there exists a unique ordered pair $(\alpha, \beta)$ of numbers
$\alpha, \beta \in \{0, 1\}$ such that the  matrix $B:= (\begin{smallmatrix} a & b \\ c & d \end{smallmatrix}) :=S^{\alpha} A S^{\beta}$ is modulo isotonic. Then, obviously,
\begin{align}\label{f3case17th1}
    &   A = (-1)^{\alpha + \beta} S^{\alpha} B S^{\beta} \ , \quad  \alpha, \beta \in \{0, 1\} \,.
\end{align}

\noindent
\noindent By using \eqref{f1case17lem4} and \eqref{f3case17th1}, we obtain
\begin{align*}
    &   A = (-1)^{\alpha + \beta+ \gamma} S^{\alpha} \big(S T^{{\fo{-2 n_{N}}}}  S T^{{\fo{-2 n_{N-1}}}}\ldots S T^{{\fo{-2 n_{1}}}}\big) S^{\beta} \,.
\end{align*}

\noindent This proves \eqref{f1case17th1} for $\{1, 0\} \ni \sigma \equiv  \alpha + \beta+ \gamma ({\rm{mod}}\, 2)$ and the uniqueness of such representation as well. Theorem~\ref{case17th1} is proved.  $\square$

}}\end{clash}
\end{subequations}


\vspace{0,2cm}
\begin{subequations}
\begin{clash}{\rm{\hypertarget{r18}{}\label{case18}\hspace{0,00cm}{\hyperlink{br18}{$\uparrow$}}
  \ We prove Lemma~\ref{bsdentplem1}{\rm{(b)}}.  The lower arc $\phi_{\eufm{n}} \big(
\gamma (- \sigma_{n_{1}}, \infty )\big)$  of $\fetr^{\hspace{0,05cm}n_{N}, ..., n_{1}}$  is  a part of the boundary of\vspace{-0,2cm}
\begin{align*}
     \phi_{{\fo{n_{N}, ..., n_{1}}}}\left(\fet\! -\! 2\sigma_{n_{1}}\right) &\! =\!
 \phi_{{\fo{n_{N}, ..., n_{2}}}}\left(\dfrac{1}{2n_{1}\! - \!\fet\! + \!2\sigma_{n_{1}}   }\right)
\!=\!
  \phi_{{\fo{n_{N}, ..., n_{2}, n_{1} +\sigma_{n_{1}}   }}}\left(\fet\right),
\end{align*}

\vspace{-0,2cm}
\noindent and also \vspace{-0,2cm}
\begin{align*}
    &  \phi_{{\fo{n_{N}, ..., n_{1}}}} \left(\gamma (-\sigma_{n_{1}},\infty)\right)=
 \phi_{{\fo{n_{N}, ..., n_{2}}}}\left(\dfrac{1}{2n_{1} - \gamma (\sigma_{n_{1}},\infty) + 2\sigma_{n_{1}}}\right)  \\[0,15cm] & =
  \phi_{{\fo{n_{N}, ..., n_{2}, n_{1} +\sigma_{n_{1}} }}}\left(\gamma (\sigma_{n_{1}},\infty)\right)=
 \gamma_{{\fo{\hspace{0,005cm}n_{N}, ..., n_{2}, n_{1}  +\sigma_{n_{1}}  }}} (\sigma_{{\fo{\hspace{0,005cm}n_{N}, ..., n_{2}, n_{1} +\sigma_{n_{1}}  }}}, \infty).
\end{align*}

\noindent  As stated before Lemma~\ref{bsdentplem1}, this means that
$\phi_{\eufm{n}} \big( - \sigma_{n_{1}}+ i \Bb{R}_{>0} \big)$ is the roof of $\fetr^{\hspace{0,05cm}n_{N}, ..., n_{1}+ \sigma_{n_{1}} }$.

Let $\sigma \in \{1, -1\}$. Then the lower arc $\phi_{\eufm{n}} \big(
\gamma ( \sigma, 0 )\big)$  of $\fetr^{\hspace{0,05cm}n_{N}, ..., n_{1}}$  is  a part of the boundary of\vspace{-0,3cm}
\begin{align*}
    &  \phi_{{\fo{n_{N}, ..., n_{1}}}}\left(-\dfrac{1}{\fet \!-  \!2 \sigma }\right)  \!= \!
 \phi_{{\fo{n_{N}, ..., n_{2}}}}\left(\dfrac{1}{2n_{1}  \!- \! \dfrac{1}{2 \sigma  \!- \!\fet}}\right) \!= \!
  \phi_{{\fo{n_{N}, ..., n_{2}, n_{1}, \sigma }}}\left(\fet\right),
\end{align*}

\vspace{-0,2cm}
\noindent and \vspace{-0,2cm}
\begin{align*}
    &  \phi_{{\fo{n_{N}, ..., n_{1}}}} \left(\gamma (\sigma,0)\right)
 \!= \!
\phi_{{\fo{n_{N}, ..., n_{1}}}}\left( -\dfrac{1}{\gamma (-\sigma,\infty) } \right)
 \!=  \!
\phi_{{\fo{n_{N}, ..., n_{1}}}}\left( -\dfrac{1}{\gamma (\sigma,\infty) - 2\sigma} \right) \\    &
 \!= \!
 \phi_{{\fo{n_{N}, ..., n_{2}}}}\left(\dfrac{1}{2n_{1}  \!- \! \dfrac{1}{2 \sigma  \!- \!\gamma (\sigma,\infty)}}\right)
  \! = \!
 \phi_{{\fo{n_{N}, ..., n_{2}, n_{1}, \sigma}}}\big(\gamma (\sigma,\infty)\big) \,.
\end{align*}

\noindent According to what was  stated before Lemma~\ref{bsdentplem1}, this means that $\phi_{{\fo{n_{N}, ..., n_{1}}}} \left(\gamma (\sigma,0)\right)$ is the roof of $\fetr^{\hspace{0,05cm}n_{N}, ..., n_{1},\sigma}$ for every
$\sigma \in \{1, -1\}$. Lemma~\ref{bsdentplem1}{\rm{(b)}} is proved. Lemma~\ref{bsdentplem1}{\rm{(c)}} follows from the obvious fact that $\gamma (0, 1)$ and $\gamma (-1, 0)$, correspondingly, are the roofs of
\begin{align*}
    &  \fetr^{\hspace{0,05cm}1}= \dfrac{1}{2 - \fet^{\,{\tn{|\!+\!1}}}} = - \dfrac{1}{\fet^{\,{\tn{|\!+\!1}}} - 2}
 \ \ \mbox{and }\  \ \fetr^{\hspace{0,05cm}-1}= \dfrac{1}{-2 - \fet^{\,{\tn{|\!-\!1}}}}= - \dfrac{1}{\fet^{\,{\tn{|\!-\!1}}} + 2}\ .
\end{align*}

}}\end{clash}
\end{subequations}


\vspace{0,2cm}
\begin{subequations}
\begin{clash}{\rm{\hypertarget{r19}{}\label{case19}\hspace{0,00cm}{\hyperlink{br19}{$\uparrow$}}
  \ We prove \eqref{f1contgendef1}. (a) We have $z\in \gamma(\sigma,0) = - 1/\gamma(-\sigma,\infty) $ if and only if there exists $T > 0$ such that
\begin{align*}
    &  z = - \dfrac{1}{- \sigma + i T}   \ \Rightarrow \  - \dfrac{1}{z} =  - \sigma + i T\, .
\end{align*}

\noindent By using \eqref{f1ybsden}, according to which
\begin{align*}
    &  \left\rceil{\re\left(- \dfrac{1}{z}\right)}\right\lceil_{2} = \left\rceil{
- \sigma}\right\lceil_{2} = - 2 \sigma\,,
\end{align*}

\noindent we obtain by \eqref{f3contgendef1},
\begin{align*}
    &  \Bb{G}_{2} (z) = - \dfrac{1}{z} -
\left\rceil{\re\left(- \dfrac{1}{z}\right)}\right\lceil_{2} =- \sigma + i T +  2 \sigma = \sigma + i T \in \gamma(\sigma,\infty) \,.
\end{align*}

\noindent Since $T> 0$ is arbitrary, \eqref{f1contgendef1}(a) is proved.

\vspace{0,1cm}(b) Since $\gamma(\sigma,\infty) \subset \mathcal{F}^{\,{\tn{||}}}_{{\tn{\square}}}$, \eqref{f1contgendef1}(b) is immediate from \eqref{f9bsdentp} and evident relationship
$-1/\gamma(\sigma,\infty)= \gamma(-\sigma,0)$.

\vspace{0,1cm}(c) In view of \eqref{f5contgen},
\begin{align*}
    &  - \dfrac{1}{\phi_{{\fo{n}}} \big(\hspace{0.07cm}\gamma(\sigma_{n}, \infty)\hspace{0.07cm}\big)} = -2n +
\gamma(\sigma_{n}, \infty) = -2n + \sigma_{n} + {\rm{i}}\, \Bb{R}_{>0}\, .
\end{align*}

\noindent By \eqref{f2contgendef1} and \eqref{f1ybsden},
\begin{align*}
    & \Bb{G}_{2} \Big( \phi_{{\fo{n}}} \big(\hspace{0.07cm}\gamma(\sigma_{n}, \infty)\hspace{0.07cm}\big)\Big) =
\left\{\re \left(- \dfrac{1}{\phi_{{\fo{n}}} \big(\hspace{0.07cm}\gamma(\sigma_{n}, \infty)\hspace{0.07cm}\big)}\right)\right\}^{{\tn{\rceil\hspace{-0,015cm}\lceil}}}_{2} + {\rm{i}} \,\im \left(- \dfrac{1}{\phi_{{\fo{n}}} \big(\hspace{0.07cm}\gamma(\sigma_{n}, \infty)\hspace{0.07cm}\big)}\right) \\    &   =
\left\{-2n + \sigma_{n}\right\}^{{\tn{\rceil\hspace{-0,015cm}\lceil}}}_{2} + {\rm{i}}\, \Bb{R}_{>0}=
\sigma_{n}+ {\rm{i}}\, \Bb{R}_{>0}=\gamma(\sigma_{n}, \infty)  \, , \
\end{align*}

\noindent which proves \eqref{f1contgendef1}(c) and completes the proof of \eqref{f1contgendef1}.

}}\end{clash}
\end{subequations}


\vspace{0,2cm}
\begin{subequations}
\begin{clash}{\rm{\hypertarget{r20}{}\label{case20}\hspace{0,00cm}{\hyperlink{br20}{$\uparrow$}}
  \ We prove \eqref{f15contgen}. It follows from
\begin{align*}\tag{\ref{f5contgen}}
    & \left\{\begin{array}{rlll}
      -1/\phi_{{\fo{\hspace{0,005cm}n_{1}}}}(z)  &    =   &   - 2 n_{1} + z \,,   &\ \mbox{if} \ N=1 \,,    \\
       -1/\phi_{{\fo{\hspace{0,005cm}n_{N}, ..., n_{1}}}}(z)  &    =  &    - 2 n_{N} +
    \phi_{{\fo{\hspace{0,005cm}n_{N-1}, ..., n_{1}}}}(z)\,, &\ \mbox{if}\ N\geqslant 2\,,
      \end{array}\right.
\end{align*}

\noindent and  from
\begin{align*}\tag{\ref{f8contgen}}\hspace{-0,2cm}
    &     \phi \left(\Bb{H}_{|\re|\leqslant 1}\right)\! \subset\! \Bb{H}_{|\re|<1}\! \setminus\!\fet \, , \ \  \phi \!\in \!\Gamma_{\vartheta}^{\hspace{0,015cm}{\tn{||}}}  \,,
\end{align*}

\noindent
that \eqref{f10bsdentp} holds because
\begin{align*}
    &      \re \left(-1/\phi_{{\fo{\hspace{0,005cm}n_{1}}}}(z)\right) =\re \left( - 2 n_{1} + z \right) \in     - 2 n_{1} + (-1,1) \,,      \ \mbox{if} \ N=1 \,, \
z \! \in \! \Bb{H}_{|\re|< 1}  \,,    \\  &
    \re  \left( -1/\phi_{{\fo{\hspace{0,005cm}n_{N-k+1}, ..., n_{1}}}}(z) \right) = \re
\left( - 2 n_{N-k+1} +
    \phi_{{\fo{\hspace{0,005cm}n_{N-k}, ..., n_{1}}}}(z)\right)
  \\    &  \in      - 2 n_{N-k+1} +
    (-1,1)\,,       \ \mbox{if}\ N\geqslant 2\,, \ z \! \in \! \Bb{H}_{|\re|\leqslant 1} \ \Rightarrow \  \phi_{{\fo{\hspace{0,005cm}n_{N-k}, ..., n_{1}}}}(z) \! \in \! \Bb{H}_{|\re|< 1} \,.
\end{align*}

\noindent Therefore
\begin{align*}
      \Bb{G}_{2} \big( \phi_{{\fo{n_{1}}}} (z)\big) &=
-\dfrac{1}{\phi_{{\fo{\hspace{0,005cm}n_{1}}}}(z)}
 -
\left\rceil{\re\left(-\dfrac{1}{\phi_{{\fo{\hspace{0,005cm}n_{1}}}}(z)}\right)}\right\lceil_{2} \\    &     =
 - 2 n_{1} + z  -
\left\rceil{- 2 n_{1} + \re\,z}\right\lceil_{2}    = z    \ , \quad  z \! \in \! \Bb{H}_{|\re|< 1}  \,,
\end{align*}

\noindent and
\begin{align*}
    &  \Bb{G}_{2} \big( \phi_{{\fo{n_{N-k+1}, ..., n_{1}}}} (z)\big) = - \dfrac{1}{ \phi_{{\fo{n_{N-k+1}, ..., n_{1}}}} (z)}-
\left\rceil{\re\left(-\dfrac{1}{\phi_{{\fo{n_{N-k+1}, ..., n_{1}}}} (z)}\right)}\right\lceil_{2}
 \\    &   =
  - 2 n_{N-k+1} +
    \phi_{{\fo{\hspace{0,005cm}n_{N-k}, ..., n_{1}}}}(z)-
\left\rceil{ - 2 n_{N-k+1} +  \re\, \phi_{{\fo{\hspace{0,005cm}n_{N-k}, ..., n_{1}}}}(z) }\right\lceil_{2} \\    &   =  \phi_{{\fo{\hspace{0,005cm}n_{N-k}, ..., n_{1}}}}(z) \ , \quad  \phi_{{\fo{\hspace{0,005cm}n_{N-k}, ..., n_{1}}}}(z)\in  \Bb{H}_{|\re|<1}\,,
\end{align*}

\noindent what was to be proved in \eqref{f15contgen}.

}}\end{clash}
\end{subequations}


\vspace{0,2cm}
\begin{subequations}
\begin{clash}{\rm{\hypertarget{r21}{}\label{case21}\hspace{0,00cm}{\hyperlink{br21}{$\uparrow$}}
  \ For $1 \! \leqslant\!  k \! \leqslant\!  N\! -\! 1$ in \eqref{f4contgendef1}(b) we have
\begin{align*}
    &   \Bb{G}^{k}_{2} \left(\fetr^{\hspace{0,05cm}n_{N}, ..., n_{1}}\right)=
 \Bb{G}^{k}_{2} \left(\phi_{{\fo{n_{N}, ..., n_{1}}}}\big( \fet^{\,{\tn{|}}\sigma_{n_{1}}}\big)\right) \\    &
=\phi_{{\fo{n_{N-k}, ..., n_{1}}}}\big( \fet^{\,{\tn{|}}\sigma_{n_{1}}}\big)=
\fetr^{\hspace{0,05cm}n_{N-k}, ..., n_{1}} \, , \
\end{align*}

\noindent by virtue of
\begin{align*}\tag{\ref{f15contgen}(b)}
    &  \Bb{G}^{k}_{2} \big( \phi_{{\fo{n_{N}, ..., n_{1}}}} (z)\big)  \! = \!  \phi_{{\fo{\hspace{0,005cm}n_{N-k}, ..., n_{1}}}}(z)\,, \  \mbox{if} \ 1 \! \leqslant\!  k \! \leqslant\!  N\! -\! 1 \,,  &\  \  z \! \in \! \Bb{H}_{|\re|\leqslant 1} \,.
\end{align*}

\noindent And
\begin{align*}
    &  \Bb{G}_{2} \left(\fetr^{\hspace{0,05cm} n_{1}}\right)=
 \Bb{G}_{2} \left(\phi_{{\fo{ n_{1}}}}\big( \fet^{\,{\tn{|}}\sigma_{n_{1}}}\big)\right)=
 \Bb{G}_{2} \left(\phi_{{\fo{ n_{1}}}}\big( \fet\big)\right)\bigcup
  \Bb{G}_{2} \Big( \phi_{{\fo{n_{1}}}} \big(\hspace{0.07cm}\gamma(\sigma_{n_{1}}, \infty)\hspace{0.07cm}\big)\Big)
 \\    &   =
\fet \bigcup \gamma(\sigma_{n_{1}}, \infty)= \fet^{\,{\tn{|}}\sigma_{n_{1}}}\, , \
\end{align*}

\noindent by virtue \eqref{f15contgen}(a) since $\fet \subset \Bb{H}_{|\re|<1} $, and by virtue of
\eqref{f1contgendef1}(c).

}}\end{clash}
\end{subequations}

\vspace{0,2cm}
\begin{subequations}
\begin{clash}{\rm{\hypertarget{r42}{}\label{case42}\hspace{0,00cm}{\hyperlink{br42}{$\uparrow$}}
  \  \ It follows from \eqref{f4bsdentp} and Theorem~\ref{bsdentpth2} that
\begin{multline}
    \bigcup\limits_{\phi \in \Gamma_{\vartheta}} \phi  (1+{\rm{i}}\Bb{R}_{>0})  =
    \bigsqcup\limits_{{\fo{n_{0} \in \Bb{Z}_{\neq 0}}}} \Big(-2 n_{0}+  \big({\rm{sign}}(n_{0})  + i \Bb{R}_{> 0}\big)\Big) \\
    \sqcup
 \bigsqcup\limits_{
\substack{  {\fo{n_{1}, ... , n_{N} \in \Bb{Z}_{\neq 0}}} \\[0,075cm] {\fo{N \in \Bb{N}\,, \ \ n_{0} \in \Bb{Z}}} }
}\Big( 2 n_{0} + \phi_{{\fo{\,n_{N}, ..., n_{1}}}}
  \big({\rm{sign}}(n_{1})  + i \Bb{R}_{> 0}\big)\Big) \, , \
\label{f1case42}\end{multline}

\noindent while \eqref{f4wbsdentp} and \eqref{f4xbsdentp} for $A=
  \gamma (-1,\infty)\sqcup\gamma (1,0)$ give
 \begin{multline}
    \bigcup\limits_{\phi \in \Gamma_{\vartheta}} \phi  (1+{\rm{i}}\Bb{R}_{>0})  =
    \bigsqcup\limits_{{\fo{n_{0} \in \Bb{Z}_{\neq 0}}}} \Big(-2 n_{0}+  \big(\gamma (-1,\infty)\sqcup\gamma (1,0)\big)\Big) \\
    \sqcup
 \bigsqcup\limits_{
\substack{  {\fo{n_{1}, ... , n_{N} \in \Bb{Z}_{\neq 0}}} \\[0,075cm] {\fo{N \in \Bb{N}\,, \ \ n_{0} \in \Bb{Z}}} }
}\Big( 2 n_{0} + \phi_{{\fo{\,n_{N}, ..., n_{1}}}}
  \big(\gamma (-1,\infty)\sqcup\gamma (1,0)\big)\Big) \, , \
\label{f2case42}\end{multline}

\noindent We prove directly that the right-hand side set in \eqref{f1case42}    equals to  the right-hand side set in \eqref{f2case42}. Everywhere below we use the notation
\begin{align*}
    &  \sigma_{n} := {\rm{sign}} (n)  \ , \quad  n \in \Bb{Z}_{\neq 0} \,,
\end{align*}

\noindent and the following identities
\begin{align}\label{f3case42}
    &
    \begin{array}{l}
  {\rm{(1)}} \ \     \gamma (0, 1) =   \phi_{1} \big(\gamma (1,\infty)\big) \ , \\[0,5cm]
  {\rm{(2)}} \ \  \phi_{n} \big(\gamma (-1,\infty)\big) =
    \left\{
    \begin{array}{ll}
 \phi_{n +1} \big(\gamma (\sigma_{n},\infty)\big)\,,  &   \
\ \mbox{if} \ n \in \Bb{Z}_{\geqslant 1} \,;     \\[0,3cm]
\phi_{n } \big(\gamma (\sigma_{n},\infty)\big),
          &   \    \ \mbox{if} \ n \in \Bb{Z}_{\leqslant -1}\,,
           \end{array}
    \right.\\[0,75cm]
  {\rm{(3)}} \ \   2n +  \gamma (-1,\infty) =
  \left\{
    \begin{array}{ll}
- 2 (-n) +  \gamma( \sigma_{- n}, \infty)\,,  &   \
\ \mbox{if} \ n \in \Bb{Z}_{\geqslant 1} \,;     \\[0,3cm]
- 2 (1-n ) +  \gamma( \sigma_{1- n }, \infty)\,,
          &   \    \ \mbox{if} \ n \in \Bb{Z}_{\leqslant 0}\,,
           \end{array}
    \right.
    \end{array}
\end{align}

\noindent which hold because
\begin{align*}
    &   \phi_{1} \left(\gamma (1,\infty)\right) =  \dfrac{1}{2 - \left(1 + i \Bb{R}\right)} = \dfrac{1}{1 - i \Bb{R}} =
    \gamma (0, 1) \, , \
\end{align*}

\noindent and
\begin{align*}
    &  \phi_{n} \left(\gamma (-1,\infty)\right) =
    \left\{
    \begin{array}{ll}
             \dfrac{1}{2n - \left(-1 + i \Bb{R}\right)} = \dfrac{1}{2n + 2 -1 - i \Bb{R}}  &\\[0,3cm]       =
          \dfrac{1}{2(n+1) - \left( \sigma_{n} + i \Bb{R}\right)}  =
       \phi_{n +1} \left(\gamma (\sigma_{n},\infty)\right),
        &   \  \ \mbox{if} \ n \in\Bb{Z}_{\geqslant 1}\,;     \\[0,3cm]
      \dfrac{1}{2n - \left(-1 + i \Bb{R}\right)} =
        \dfrac{1}{2n -\left( \sigma_{n} + i \Bb{R}\right)}  & \\[0,4cm]
        =        \phi_{n } \left(\gamma (\sigma_{n},\infty)\right),
          &   \            \ \mbox{if} \ n  \in\Bb{Z}_{\leqslant -1}\,.
           \end{array}
    \right.
\end{align*}

\noindent Applying first identity \eqref{f3case42}(1) and then identities \eqref{f3case42}(2),(3) to the set on the right-hand side of the equality \eqref{f2case42}, we obtain that it equals to
\begin{align*}
    &   \bigsqcup\limits_{ n_{0} \in \Bb{Z}}\big(2n_{0} + \gamma (-1,\infty)\big)
     \sqcup
   \bigsqcup\limits_{
\substack{  {\fo{n_{1}, ... , n_{ N} \in \Bb{Z}_{\neq 0}}} \\[0,075cm] {\fo{N \in \Bb{N}\,, \ \ n_{0} \in \Bb{Z}}} }
} \Big( 2 n_{0} + \phi_{{\fo{\hspace{0,005cm} n_{N}, ..., n_{1}}}}\big(\gamma (-1,\infty)\big)\Big) \\[0,3cm]    &
   \sqcup \bigsqcup\limits_{ n_{0} \in \Bb{Z}}\Big(2n_{0} + \phi_{1}\big(\gamma (1,\infty)\big) \Big)
     \sqcup
   \bigsqcup\limits_{
\substack{  {\fo{n_{2}, ... , n_{ N} \in \Bb{Z}_{\neq 0}}} \\[0,075cm] {\fo{N \in \Bb{Z}_{\geqslant 2}\,, \ \ n_{0} \in \Bb{Z}}} } }
\Big( 2 n_{0} + \phi_{{\fo{\hspace{0,005cm} n_{N}, ..., n_{2}, 1}}}\big(\gamma (1,\infty)\big)\Big) \\[0,3cm]    &
= \bigsqcup\limits_{ n_{0} \in \Bb{Z}_{\geqslant 1}}\big(-2(-n_{0}) + \gamma (\sigma_{-n_{0}} ,\infty)\big)
\sqcup
 \bigsqcup\limits_{ n_{0} \in \Bb{Z}_{\leqslant 0}}\big(-2(1-n_{0}) + \gamma (\sigma_{1-n_{0}} ,\infty)\big)
 \\    &     \sqcup \bigsqcup\limits_{ n_{0} \in \Bb{Z}} \bigg[
 \Big(2n_{0} + \phi_{1}\big(\gamma (1,\infty)\big) \Big)  \sqcup
   \bigsqcup\limits_{n_{1}\in \Bb{Z}_{\geqslant 2}}  \Big(2n_{0} +  \phi_{n_{1} } \big(\gamma (\sigma_{n_{1}},\infty)\big) \Big)
 \\    & \sqcup
\bigsqcup\limits_{n_{1}\in \Bb{Z}_{\leqslant -1} }  \Big(2n_{0} +  \phi_{n_{1} } \big(\gamma (\sigma_{n_{1}},\infty)\big)\Big)
  \bigg]  \\    &
  \sqcup \bigsqcup\limits_{
\substack{  {\fo{n_{2}, ... , n_{ N} \in \Bb{Z}_{\neq 0}}} \\[0,075cm] {\fo{N \in \Bb{Z}_{\geqslant 2}\,, \ \ n_{0} \in \Bb{Z}}} } }  \bigg[
   \Big( 2 n_{0} + \phi_{{\fo{\hspace{0,005cm} n_{N}, ..., n_{2}, 1}}}\big(\gamma (1,\infty)\big)\Big)
\\    &  \sqcup   \bigsqcup\limits_{n_{1} \in \Bb{Z}_{\geqslant 2}}  \Big(2n_{0} +  \phi_{{\fo{\hspace{0,005cm} n_{N}, ..., n_{1}}}} \big(\gamma (\sigma_{n_{1}},\infty)\big) \Big)
 \sqcup
 \bigsqcup\limits_{n_{1}\in \Bb{Z}_{\leqslant -1} }\! \!\!\! \Big(2n_{0} +   \phi_{{\fo{\hspace{0,005cm} n_{N}, ..., n_{1}}}}  \big(\gamma (\sigma_{n_{1}},\infty)\big)\Big)
  \bigg]  \\    &
    =
  \bigsqcup\limits_{ n \in \Bb{Z}_{\neq 0}}\big(-2 n + \gamma (\sigma_{n} ,\infty)\big)
  \sqcup
  \bigsqcup\limits_{
\substack{  {\fo{n_{1} \in \Bb{Z}_{\neq 0}}} \\[0,075cm] {\fo{ n_{0} \in \Bb{Z}}} } }
 \Big(2n_{0} +  \phi_{n_{1} } \big(\gamma (\sigma_{n_{1}},\infty)\big)\Big) \\    &
 \sqcup
 \bigsqcup\limits_{
\substack{  {\fo{n_{1}, ... , n_{ N} \in \Bb{Z}_{\neq 0}}} \\[0,075cm] {\fo{N \in \Bb{Z}_{\geqslant 2}\,, \ \ n_{0} \in \Bb{Z}}} } } \Big(2n_{0} +   \phi_{{\fo{\hspace{0,005cm} n_{N}, ..., n_{1}}}}  \big(\gamma (\sigma_{n_{1}},\infty)\big)\Big) \\    &
=
\bigsqcup\limits_{ n_{0} \in \Bb{Z}_{\neq 0}}\big(-2 n_{0} +  (\sigma_{n_{0}} + i\Bb{R}_{>0})\big)
  \sqcup
    \bigsqcup\limits_{
\substack{  {\fo{n_{1}, ... , n_{ N} \in \Bb{Z}_{\neq 0}}} \\[0,075cm] {\fo{N \in \Bb{N},, \ \ n_{0} \in \Bb{Z}}} } }
    \! \!\!\!\Big(2n_{0} +   \phi_{{\fo{\hspace{0,005cm} n_{N}, ..., n_{1}}}}  \big(\gamma (\sigma_{n_{1}},\infty)\big)\Big),
\end{align*}

\noindent where the latter set coincides with the right-hand set in \eqref{f1case42}. The proof of the equality between the right-hand sets in \eqref{f1case42} and  in \eqref{f2case42} is finished.
So that these two sets represent two different partitions of the set
\begin{align*}
    &\bigcup\limits_{{\fo{\phi \in \Gamma_{\vartheta}}}} \phi  (1+{\rm{i}}\Bb{R}_{>0})  =  \bigcup\limits_{{\fo{  \phi\! \in\! \Gamma(2)   }}}   \phi \big(\Bb{H} \cap \partial F_{\Gamma(2)} \big)\,.
\end{align*}

}}\end{clash}
\end{subequations}

\vspace{0,2cm}
\begin{subequations}
\begin{clash}{\rm{\hypertarget{r44}{}\label{case44}\hspace{0,00cm}{\hyperlink{br44}{$\uparrow$}}
  \  \ We prove \eqref{f2denac}. The set in the right-hand side of \eqref{f2denac} equals to $ \Gamma_{\vartheta}^{\hspace{0,015cm}{\tn{||}}}$, in view of the definition \eqref{f7contgen}. The relationship \eqref{f4contgen} yields that  $ \Gamma_{\vartheta}^{\hspace{0,015cm}{\tn{||}}} \subset
\Gamma_{\vartheta} $ and since
\begin{align}\label{f2case44}  \Gamma_{\vartheta}= \bigg\{
{\fo{\, \dfrac{a z + b}{c z +d} }} \ \bigg| \
{\fo{\begin{array}{l}
ab\equiv cd \equiv 0 ({\rm{mod}}\, 2)    \\[0,1cm]
ad - bc\!=\!1 \, , \ a, b, c, d \!\in\! \Bb{Z}
\end{array}}}\, \bigg\} \,,
\end{align}

\noindent we deduce from \eqref{f8contgen} that
\begin{align}\label{f1case44}
    & \Gamma_{\vartheta}^{\hspace{0,015cm}{\tn{||}}} \subset \bigg\{
{\fo{\, \dfrac{a z + b}{c z +d} }} \ \bigg| \
{\fo{\begin{array}{l}
 |a|\!<\! |b|\!<\!d,\ |a|\!<\! |c|\!<\!d     \\
ab\equiv cd \equiv 0 ({\rm{mod}}\, 2)    \\
ad - bc\!=\!1 \, , \ a, b, c, d \!\in\! \Bb{Z}
\end{array}}}\, \bigg\}  \, .
\end{align}

 \noindent
 The inverse  inclusion in \eqref{f1case44} follows from  \eqref{f2case44}, Lemma~\ref{case17lem4}, \eqref{f3pf8contgen} and \eqref{f7contgen}. The equalities \eqref{f2denac} are proved.
}}\end{clash}
\end{subequations}

\vspace{0,2cm}
\begin{subequations}
\begin{clash}{\rm{\hypertarget{r39}{}\label{case39}\hspace{0,00cm}{\hyperlink{br39}{$\uparrow$}}
  \  \ We show that this fact can be easily deduced from the properties
of the even integer part $\rceil{x}\lceil_{\! 2}\,\in\! 2 \Bb{Z}$ and of the
even fractional part
$\{x\}^{{\tn{\rceil\hspace{-0,015cm}\lceil}}}_{2} \in [-1,1]$
of the real number  $x\! \in\! \Bb{R}$. Moreover, our further reasoning does not depend on how these two functions are defined on the odd integers.

\vspace{0.15cm}
We first prove that  $\phi_{\eufm{n}} (0)$ for arbitrary $\eufm{n}\in  \Bb{Z}_{\neq 0}^{\hspace{0,02cm}\Bb{N}_{\hspace{-0,02cm}\eurm{f}}}$ is an even rational number lying in $(-1,1)\setminus\{0\}$.
  Observe that
\begin{align*}\tag{\ref{f6zpf8contgen}}
     & \phi_{\eufm{n}}(z)\!= \! \dfrac{z p^{\eufm{n}}_{N-1} + q^{\eufm{n}}_{N-1}}{z p^{\eufm{n}}_{N} + q^{\eufm{n}}_{N}} \ , \
     \eufm{n}\!=\!(n_{N}, ..., n_{1})\!\in\! \Bb{Z}_{\neq 0}^{N} \ , \ N\! \in \!\Bb{N} \, , \ z \in \Bb{H} \,,
\end{align*}

\noindent and
\begin{align*}\tag{\ref{f1preresauxlem2}}
    &
      {\rm{(a)}} \ \ \big|q^{\eufm{n}}_{k}\big| > \big|p^{\eufm{n}}_{k}\big| > \big|p^{\eufm{n}}_{k-1}\big| \ ,  \quad   {\rm{(b)}} \  \ \big|q^{\eufm{n}}_{k}\big| > \big|q^{\eufm{n}}_{k-1}\big| > \big|p^{\eufm{n}}_{k-1}\big| \ , \quad 1 \leqslant k \leqslant N \,  .
  \end{align*}

\noindent yield that
\begin{align}\label{f1case39}  &
\phi_{\eufm{n}} \in {\rm{Hol}}\big(\,[-1,1]\cup \left(\Bb{C}\setminus\Bb{R}\right) \big)
\, , \quad
     \eufm{n}\!=\!(n_{N}, ..., n_{1})\!\in\! \Bb{Z}_{\neq 0}^{N} \ , \ N\! \in \!\Bb{N} \,.
\end{align}

\noindent   So that, by setting $z=0$ in \eqref{f5contgen} and \eqref{f6zpf8contgen}, we obtain
\begin{align}\label{f2case39}
    & \left\{\begin{array}{rlll}
      -1/\phi_{{\fo{\hspace{0,005cm}n_{1}}}}(0)  &    =   &   - 2 n_{1}  \,,   &\ \mbox{if} \ N=1 \,,    \\
       -1/\phi_{{\fo{\hspace{0,005cm}n_{N}, ..., n_{1}}}}(0)  &    =  &    - 2 n_{N} +
    \phi_{{\fo{\hspace{0,005cm}n_{N-1}, ..., n_{1}}}}(0)\,, &\ \mbox{if}\ N\geqslant 2\,,
      \end{array}\right.
\end{align}

\noindent and, in view of \eqref{f1preresauxlem2},
\begin{align}\label{f3case39}
\phi_{\eufm{n}}(0)\!= \! \dfrac{ q^{\eufm{n}}_{N-1}}{ q^{\eufm{n}}_{N}} \in
\left(-1,1\right)\setminus\{0\}
\, , \quad
     \eufm{n}\!=\!(n_{N}, ..., n_{1})\!\in\! \Bb{Z}_{\neq 0}^{N} \ , \ N\! \in \!\Bb{N} \,.
\end{align}

\noindent  It follows from \eqref{f3case39} that  $\phi_{\eufm{n}} (0)$ is a rational number lying in $(-1,1)\setminus\{0\}$. Furthermore,
\begin{align*}\tag{\ref{f2apreresauxlem1a}}
p^{\eufm{n}}_{k-1} q^{\eufm{n}}_{k} - p^{\eufm{n}}_{k} q^{\eufm{n}}_{k-1} = 1  \ , \  \quad\quad    0 \leqslant k \leqslant N \ ,
 \end{align*}

\noindent for $k=N$ implies that
\begin{align}\label{f4case39}
    & {\rm{gcd}} \left(q^{\eufm{n}}_{N-1},  q^{\eufm{n}}_{N}\right) =1 \, , \
\end{align}

\noindent while \eqref{f2preresauxlem1a} written in the form
\begin{align*}\tag{\ref{f2preresauxlem1a}}
     &      q^{\eufm{n}}_{k} = 2 n_{k} q^{\eufm{n}}_{k-1} -\, q^{\eufm{n}}_{k-2} \ , \         \ q^{\eufm{n}}_{0} = 1  \ ,      \ q^{\eufm{n}}_{-1} =  0  \  ,
                \quad    1 \leqslant k \leqslant N \,,
\end{align*}

\noindent means that
\begin{align*}
    &   q_{2n-1}  \equiv 0 ({\rm{mod}} \, 2) \ , \quad   q_{2n}  \equiv 1 ({\rm{mod}} \, 2)
 \ , \quad       0 \leqslant n < (N+1)/2 \ ,
\end{align*}

\noindent and therefore $q^{\eufm{n}}_{N-1}  q^{\eufm{n}}_{N} \equiv 0 ({\rm{mod}} \, 2) $. This property together with \eqref{f4case39} proves that    $\phi_{\eufm{n}} (0)$ is an even rational number lying in $(-1,1)\setminus\{0\}$
for arbitrary $\eufm{n}\in  \Bb{Z}_{\neq 0}^{\hspace{0,02cm}\Bb{N}_{\hspace{-0,02cm}\eurm{f}}}$.

\vspace{0.15cm}
By applying the even integer and the even fractional parts  to \eqref{f2case39}, we obtain, taking into account of \eqref{f3case39},
\begin{align*}
    &
    \left\{\begin{array}{rllrlll}G^{{\tn{\rceil\hspace{-0,015cm}\lceil}}}_{2}\left( \phi_{{\fo{\hspace{0,005cm}n_{1}}}}(0)\right)&    =   &   0  \,,   &\
     \left\rceil{ 1/\phi_{{\fo{\hspace{0,005cm}n_{1}}}}(0)}\right\lceil_{2}  &    =   &    2 n_{1}  \,,   &\ \mbox{if} \ N=1 \,,    \\[0,2cm]
     G^{{\tn{\rceil\hspace{-0,015cm}\lceil}}}_{2}\left(\phi_{{\fo{\hspace{0,005cm}n_{N}, ..., n_{1}}}}(0)\right)&    =   &   \phi_{{\fo{\hspace{0,005cm}n_{N-1}, ..., n_{1}}}}(0)  \,,   &\
       \left\rceil{1/\phi_{{\fo{\hspace{0,005cm}n_{N}, ..., n_{1}}}}(0)}\right\lceil_{2}  &    =  &     2 n_{N} \,, &\ \mbox{if}\ N\geqslant 2\,.
      \end{array}\right.
\end{align*}

\noindent Hence, we can easily calculate the coefficients of the continued fraction
\begin{align*}
    &  \phi_{{\fo{\hspace{0,005cm}n_{N}, ..., n_{1}}}}(0):=
  {\fo{\dfrac{1}{2n_{N}\! -\! \dfrac{1}{ \begin{subarray}{l}
\begin{displaystyle} \vphantom{A^{A}}
2n_{N-1}\!- \!\end{displaystyle}\\[-0.4cm]
\hphantom{2n_{k-1} 11 -} \ddots\\[-0.4cm]
\hphantom{2n_{k-1} 11 - \ddots} \begin{displaystyle} - \!\end{displaystyle} \dfrac{1}{2 n_{2} \!- \! \dfrac{1}{
\vphantom{A^{A}}2 n_{ 1}  }}\end{subarray} }}}}
  \ ,  \quad (n_{N}, ..., n_{1})\!\in\! \Bb{Z}_{\neq 0}^{N} \ , \ N\! \in \!\Bb{N} \,,
\end{align*}

\vspace{-0,2cm}
\noindent by the formulas,
\begin{align}\label{f5case39}
    &  2n_{N-k}=\big\rceil 1\big/ \big(G^{{\tn{\rceil\hspace{-0,015cm}\lceil}}}_{2}\big)^{k} \big( \phi_{{\fo{n_{N}, ..., n_{1}}}} (0)\big)\big\lceil_{\! 2}  \ ,  \quad   0\leqslant k \leqslant N-1\,,
\end{align}

\noindent and also,
\begin{align}\label{f6case39}
     & \left\{
     \begin{array}{ll}
  \big(G^{{\tn{\rceil\hspace{-0,015cm}\lceil}}}_{2}\big)^{k}\! \big( \phi_{{\fo{n_{N}, ..., n_{1}}}} (0)\big)   \! = \!   \phi_{{\fo{\hspace{0,005cm}n_{N-k}, ..., n_{1}}}}(0) \, ,   &   \  0\!\leqslant\! k \!\leqslant \!N\!-\!1 \,,  \\[0,1cm]
     \big(G^{{\tn{\rceil\hspace{-0,015cm}\lceil}}}_{2}\big)^{N} \!\big( \phi_{{\fo{n_{N}, ..., n_{1}}}} (0)\big)   \! = \!0  \, ,   &   \ (n_{N}, ..., n_{1})\!\in\! \Bb{Z}_{\neq 0}^{N} \ , \ N\! \in \!\Bb{N} \,.
     \end{array}
     \right.
\end{align}

\noindent  It follows from \eqref{f5case39} that
\begin{align}\label{f7case39}
    &  \phi_{\eufm{n}}(0)=\phi_{\eufm{m}}(0) \, , \  \eufm{n}, \eufm{m}\in  \Bb{Z}_{\neq 0}^{\hspace{0,02cm}\Bb{N}_{\hspace{-0,02cm}\eurm{f}}} \ \ \Longleftrightarrow \ \
     \eufm{n} = \eufm{m} \,.
\end{align}

\vspace{0.25cm}
 We  prove now that the following analogue of \eqref{f6case39} for an arbitrary even rational number $p/q\in (-1,1)$, $p, q \in \Bb{Z}_{\neq 0}$, ${\rm{gcd}}(p, q) =1$, holds
\begin{align}\nonumber
    &  \dfrac{p}{q}\!\in \!(-1,1)\!\setminus\!\{0\} \ \mbox{is even}
     \Rightarrow  \
     \left\{ \begin{array}{ll} G^{{\tn{\rceil\hspace{-0,015cm}\lceil}}}_{2} \left(\dfrac{p}{q}\right)\!=\!\dfrac{r}{p}\!\in \!(-1,1)\!\setminus\!\{0\}\ \mbox{is even} \,,   & \ \mbox{if}\ |p|\!\geqslant\! 2 ;   \\[0,3cm]  G^{{\tn{\rceil\hspace{-0,015cm}\lceil}}}_{2} \left(\dfrac{p}{q}\right)=
      0\,,   & \ \mbox{if}\ |p| = 1 \,;  \end{array} \right. \\   &
       {\rm{(a)}} \  {\rm{gcd}}(r, p)={\rm{gcd}}(p, q)= 1  ;    \      {\rm{(b)}} \
     r  \equiv q ({\rm{mod}}\, 2) , \  {\rm{(c)}} \   1 \leqslant  |r| \leqslant  |p| -1 ,
\label{f8case39}\end{align}

\noindent where $r\in \Bb{Z}_{\neq 0}$. Since for $|p| = 1 $ the integer $q$   is even, then $p/q = \phi_{n_{1}}(0)$
with $n_{1} :=(q/2)\,{\rm{sign}} (p) \in \Bb{Z}_{\neq 0}$ and $G^{{\tn{\rceil\hspace{-0,015cm}\lceil}}}_{2}\left( \phi_{{\fo{\hspace{0,005cm}n_{1}}}}(0)\right)=0$, as follows from \eqref{f6case39} for $N=1$. It remains to prove
\eqref{f8case39} for $|p|\!\geqslant\! 2$.
 Since $q$ and $p$ are coprime then $q$ belongs to the set
\begin{align*}
    & \big(  - |p| ,   |p| \big) \  \sqcup \bigsqcup_{{\fo{m \in \Bb{Z}_{\neq 0} }}} \Big( 2 m |p| - |p| , \  2 m |p| + |p|\Big) =
    \Bb{R} \setminus   \bigsqcup_{{\fo{m \in \Bb{Z}}}} \Big\{\big(2m+1\big) \cdot p\Big\} \ ,
\end{align*}

\noindent and does not lie on the interval $(  - |p| ,   |p| )$ because $|q| > |p|$. Thus, there exists a unique $m_{1} \in \Bb{Z}_{\neq 0}$ such that\vspace{-0.3cm}
\begin{align*}
    &  q \in \Big( 2 m_{1} |p\,| - |p\,| ,  2 m_{1} |p\,| + |p\,| \Big) \cap \Bb{Z} = \big\{2 m_{1} |p\,| \big\}  \  \sqcup
    \bigsqcup_{ \substack{ {\fo{k = -  |p\,|+1}} \\[0.1cm] {\fo{k \neq 0}} } }^{{\fo{|p\,| -1}}} \big\{2 m_{1} |p\,| - k \big\}  \ ,
\end{align*}

\noindent and since $q \neq 2 m_{1} |p\,|$,   there also exists  a unique $k_{1} \in \Bb{Z}_{\neq 0}$ satisfying
\begin{align*}
    &  q = 2 m_{1} |p\,| - k_{1}  \ , \quad  -  |p\,|+1 \leq  k_{1} \leq |p\,| -1  \ , \  k_{1} \neq 0 \ .
\end{align*}

\noindent By setting $n := m_{1} \cdot {\rm{sign}} (p)$ and $r := k_{1}$ we obtain the existence of the integers $n$ and $r$ satisfying
\begin{align}\label{f9case39}
    &  q = 2 n p - r  \ , \quad  n, r \in \Bb{Z}_{\neq 0}  \, , \   1 \leqslant  |r| \leqslant  |p| -1 \,.
\end{align}

\noindent Here, the property \eqref{f8case39}(b) is immediate, while \eqref{f8case39}(a)
follows from   \cite[p.\! 5, Lemma 1.5]{jon}, which states that
\begin{align}\label{f0bpreres}
    &  a, c   \in   \Bb{Z} \ , \  \ b,  d   \in    \Bb{Z}_{\neq 0} \ ,  \quad
     a= d b - c \ \Longrightarrow \ {\rm{gcd}}(a, b)= {\rm{gcd}}(c, b) .
\end{align}

\noindent  Finally,
$-q/p = -2 n  + r/p$ and $r/p \in (-1,1)\setminus\{0\}$ yield $G^{{\tn{\rceil\hspace{-0,015cm}\lceil}}}_{2} ({p}/{q})\!=\!{r}/{p}$ and completes the proof of \eqref{f8case39}.

\vspace{0.25cm}
We prove that  an arbitrary even rational number $p/q$, $p, q \in \Bb{Z}_{\neq 0}$, ${\rm{gcd}}(p, q) =1$, lying on the set $(-1,1)$, can be represented as $\phi_{{\fo{\hspace{0,005cm}n_{N}, ..., n_{1}}}}(0)$ with  $(n_{N}, ..., n_{1})\!\in\! \Bb{Z}_{\neq 0}^{N}$, $N\! \in \!\Bb{N}$,  by employing induction on $|p|$.  It has already been written above that for $|p| = 1 $ the integer $q$   is even, and therefore  $p/q = \phi_{n_{1}}(0)$ with $n_{1} :=(q/2)\,{\rm{sign}} (p) \in \Bb{Z}_{\neq 0}$. Assume that $P \in \Bb{Z}_{\geqslant 2}$ and this statement is proved for all nonzero even rational numbers from $(-1,1)$ the modulus of whose nominators is less than or equal $P-1$. Let $p/q \in (-1,1)$, $p, q \in \Bb{Z}_{\neq 0}$, ${\rm{gcd}}(p, q) =1$ and $|p|=P$. Invoking the induction hypothesis on the even rational number $r/p$ satisfying \eqref{f9case39}, we obtain the existence  of
$N  \in \Bb{Z}_{\geqslant 2}$ and $n_{1}, ..., n_{N-1}\in \Bb{Z}_{\neq 0}$ such that $r/p =\phi_{{\fo{\hspace{0,005cm}n_{N-1}, ..., n_{1}}}}(0)$. But then \eqref{f8case39}  and  \eqref{f9case39}
yield
\begin{align*}
    &  \dfrac{p}{q} =\dfrac{1}{\dfrac{q}{p}} \!=\!
    \dfrac{1}{2n \!-\! \dfrac{r}{p} }\! =\!\dfrac{1}{2n\! -\!\phi_{{\fo{\hspace{0,005cm}n_{N-1}, ..., n_{1}}}}(0) } \!  =\!
 \phi_{{\fo{\hspace{0,005cm}n_{N}, n_{N-1}, ..., n_{1}}}}(0)\ , \ n_{N}\!:= \!n \!\in \!\Bb{Z}_{\neq 0}\,.
 \end{align*}

 \noindent The induction is complete. In view of \eqref{f5case39} and \eqref{f6case39}, we have proved the following assertion.
 \begin{lemma}\hspace{-0,2cm}{\bf{.}}\label{case39lem1}
    Each nonzero even rational number $p/q$ in $(-1,1)$ with nonzero coprime integers $p$ and $q$
  can be uniquely represented in the form $\phi_{{\fo{n_{N}, ..., n_{1}}}} (0)$ with  $(n_{N}, ..., n_{1})\!\in\! \Bb{Z}_{\neq 0}^{N}$, $N\! \in \!\Bb{N}$, satisfying
  \begin{align}\label{f10case39}
    &  2n_{N-k}=\Bigg\rceil\dfrac{ 1}{\big(G^{{\tn{\rceil\hspace{-0,015cm}\lceil }}}_{2}\big)^{k} \big(
     p/q\big)} \Bigg\lceil_{\! 2}  \ , \ \ 0\leqslant k \leqslant N-1\,, \ \ \big(G^{{\tn{\rceil\hspace{-0,015cm}\lceil}}}_{2}\big)^{N} \big(p/q\big)=0\,.
\end{align}

\noindent Conversely, $\phi_{\eufm{n}} (0)$ for arbitrary $\eufm{n}\in  \Bb{Z}_{\neq 0}^{\hspace{0,02cm}\Bb{N}_{\hspace{-0,02cm}\eurm{f}}}\!:=\!\sqcup_{k\geqslant 1}\! \left(\Bb{Z}_{\neq 0}\right)^{k}$ is an even rational number lying in $(-1,1)\setminus\{0\}$.
For different $\eufm{n}$ and $\eufm{m}$ from  $\Bb{Z}_{\neq 0}^{\hspace{0,02cm}\Bb{N}_{\hspace{-0,02cm}\eurm{f}}}$
the numbers $\phi_{\eufm{n}} (0)$ and $\phi_{\eufm{m }} (0)$ are different.
 \end{lemma}

}}\end{clash}
\end{subequations}

\vspace{0,2cm}
\begin{subequations}
\begin{clash}{\rm{\hypertarget{r25}{}\label{case25}\hspace{0,00cm}{\hyperlink{br25}{$\uparrow$}}
  \  \ We prove \eqref{f16contgen}(a). Since $\gamma(-1,1)\subset \mathcal{F}^{\,{\tn{||}}}_{{\tn{\square}}}$
we can apply \eqref{f9bsdentp} to get $\Bb{G}_{2}\big(\gamma(-1,1)\big)= - 1/\gamma(-1,1)= \gamma(-1,1)$.
To prove the left-hand side equality
\begin{align*}\tag{\ref{f16contgen}{\rm{(a)}}}
    &  \Bb{G}_{2}  \big(\eusm{E}^{0}_{\!{\tn{\frown}}} \big) \! = \!
        \mathcal{F}^{\,{\tn{||}}}_{{\tn{\square}}}\setminus\gamma(-1,1) \,,
\end{align*}

\noindent  of  \eqref{f16contgen}(a), we use \eqref{f8bsdenac}(a),
\begin{align*}
    &  \eusm{E}^{0}_{\!{\tn{\frown}}}\!:=\! \fetd \sqcup\!\!\!\!
 \bigsqcup\limits_{{\fo{n_{0}\!\in\! \Bb{Z}_{\neq 0}}}}\!\!\!\!\! \fetur^{\hspace{0,05cm} n_{0}} \ \Rightarrow \
 \Bb{G}_{2}\left(\eusm{E}^{0}_{\!{\tn{\frown}}}\right)\!:=\!  \Bb{G}_{2}\left(\fetd \right)\sqcup\!\!\!\!
 \bigsqcup\limits_{{\fo{n_{0}\!\in\! \Bb{Z}_{\neq 0}}}}\!\!\!\!\!  \Bb{G}_{2}\left(\fetur^{\hspace{0,05cm} n_{0}}\right) \, , \
\end{align*}

\noindent where by \eqref{f3zbsdenac} $\fetd = \Bb{D} \cap \fet \subset \mathcal{F}^{\,{\tn{||}}}_{{\tn{\square}}}$ we apply \eqref{f9bsdentp} to get
\begin{align*}
    & \Bb{G}_{2}\left(\fetd \right)= - 1/\fetd \, , \
\end{align*}

\noindent while for every $n_{0}\!\in\! \Bb{Z}_{\neq 0}$ we can apply to (see \eqref{f5bsden})
\begin{align*}
    &  \fetur^{\hspace{0,01cm}n_{0}} = \phi_{n_{0}}  \big(\fetu\sqcup\gamma(\sigma_{n_{0}}, \infty)\big)=
\phi_{n_{0}}  \big(\fetu\big)\sqcup\phi_{n_{0}} \big(\gamma(\sigma_{n_{0}}, \infty)\big)
  \ , \quad  \fetu \subset \Bb{H}_{|\re|< 1} \, , \
\end{align*}

\noindent the equality \eqref{f15contgen}(a)
\begin{align*}\tag{\ref{f15contgen}{\rm{(a)}}}
    &  \Bb{G}_{2} \big( \phi_{{\fo{n_{0}}}} (z)\big) = z   \ , \quad  \ z \! \in \! \Bb{H}_{|\re|< 1}  \,,
\end{align*}

\noindent  and the equality \eqref{f1contgendef1}(c)
\begin{align*}\tag{\ref{f1contgendef1}{\rm{(c)}}}
    &   \Bb{G}_{2} \Big( \phi_{{\fo{n}}} \big(\hspace{0.07cm}\gamma(\sigma_{n}, \infty)\hspace{0.07cm}\big)\Big)  = \gamma(\sigma_{n}, \infty)\ , \quad  n\!\in\! \Bb{Z}_{\neq 0}\, ,
\end{align*}

\noindent to obtain
\begin{align*}
    &   \Bb{G}_{2}\left(\fetur^{\hspace{0,05cm} n_{0}}\right)= \fetu \sqcup\gamma(\sigma_{n_{0}}, \infty)\ , \quad  n_{0}\!\in\! \Bb{Z}_{\neq 0}\,.
\end{align*}

\noindent Finally,
\begin{align*}
       \Bb{G}_{2}\left(\eusm{E}^{0}_{\!{\tn{\frown}}}\right)&= \fetu \sqcup  \left(- 1/\fetd\right) \sqcup
\gamma(1, \infty)\sqcup
\gamma(-1, \infty) \\    &= \fetu\sqcup
\gamma(1, \infty)\sqcup
\gamma(-1, \infty)     =
\mathcal{F}^{\,{\tn{||}}}_{{\tn{\square}}} \setminus\overline{\Bb{D}}\,.
\end{align*}
\noindent
This finishes the proof of \eqref{f16contgen}(a).

}}\end{clash}
\end{subequations}


\vspace{0,2cm}
\begin{subequations}
\begin{clash}{\rm{\hypertarget{r22}{}\label{case22}\hspace{0,00cm}{\hyperlink{br22}{$\uparrow$}}
  \ We prove \eqref{f2bsdentppre}. Since
  $z \in \phi_{{\fo{n_{N}, ..., n_{1}}}} \left(\Bb{H}_{|\re|\leqslant 1}\right)$ then there exists
  $z_{0} \in \Bb{H}_{|\re|\leqslant 1}$ such that
  \begin{align}\label{f1case22}
    & z =  \phi_{{\fo{n_{N}, ..., n_{1}}}} (z_{0}) \stackrel{{\fo{\eqref{f8contgen}}}}{\vphantom{A}\in}  \Bb{H}_{|\re| < 1}\! \setminus\!\fet \, , \quad  z_{0}\in \Bb{H}_{|\re|\leqslant 1}\,,
  \end{align}

\noindent and by \eqref{f5bsdenac},
\begin{align}\label{f2case22}
    &\hspace{-0,3cm}  z_{0} \!= \!   \psi_{{\fo{n_{N}, ..., n_{1}}}} (z)\!=\!
\dfrac{1}{\phi_{{\fo{n_{1}, ..., n_{N}}}}(1/z)} \!= \!
\left\{
  \begin{array}{ll}
  2 n_{1} \! - \! \phi_{{\fo{n_{2}, ..., n_{N}}}}(1/z), & \hbox{if} \  N  \!\geqslant  \!2 ; \\
   2 n_{1}  \!-  \!\dfrac{1}{z}, & \hbox{if} \ N \!= \!1 \,.
  \end{array}
\right.\hspace{-0,1cm}
\end{align}

 If $ N  \!\geqslant  \!2$ the application of \eqref{f15contgen}(b) to \eqref{f1case22} gives
\begin{align*}
    & \Bb{G}^{k+1}_{2}(z)= \Bb{G}^{k+1}_{2} \big( \phi_{{\fo{n_{N}, ..., n_{1}}}} (z_{0})\big)   \! = \!   \phi_{{\fo{\hspace{0,005cm}n_{N-k-1}, ..., n_{1}}}}(z_{0})\,, \quad  \mbox{if} \ \ \ 0 \! \leqslant\!  k \! \leqslant\!  N\! -\! 2 \,,
\end{align*}

\noindent where, by virtue of \eqref{f2case22} and \eqref{f5bsdenac},
\begin{align*}
     \phi_{{\fo{\hspace{0,005cm}n_{N-k-1}, ..., n_{1}}}}(z_{0})&=
  \phi_{{\fo{\hspace{0,005cm}n_{N-k-1}, ..., n_{1}}}}\big(\psi_{{\fo{n_{N}, ..., n_{1}}}} (z)\big) \\    &   =
  \psi_{{\fo{n_{N}, ..., n_{N-k}}}}\left(\phi_{{\fo{n_{N}, ..., n_{N-k}}}}\Big(\phi_{{\fo{\hspace{0,005cm}n_{N-k-1}, ..., n_{1}}}}\big(\psi_{{\fo{n_{N}, ..., n_{1}}}} (z)\big)\Big)\right)\\    &   =
 \psi_{{\fo{n_{N}, ..., n_{N-k}}}}\Big(\phi_{{\fo{\hspace{0,005cm}n_{N}, ..., n_{1}}}}\big(\psi_{{\fo{n_{N}, ..., n_{1}}}} (z)\big)\Big) =  \psi_{{\fo{n_{N}, ..., n_{N-k}}}} (z  )\, , \
\end{align*}

\noindent and therefore
\begin{align}\label{f3case22}
    &  \Bb{G}^{k+1}_{2}(z)= \psi_{{\fo{n_{N}, ..., n_{N-k}}}} (z  )\,, \quad  \mbox{if} \ \ \ 0 \! \leqslant\!  k \! \leqslant\!  N\! -\! 2 \,,
\end{align}

\noindent which proves the right-hand side equalities in \eqref{f2bsdentppre}.

Then for $y =  \phi_{{\fo{n_{N}, ..., n_{1}}}} (y_{0}) $, $y_{0}\in \Bb{H}_{|\re|< 1}$ and $N \geqslant 2$ we have
\begin{align*}
    &\Bb{G}^{N}_{2}(y)=\Bb{G}_{2} \left(\Bb{G}^{N-1}_{2}(y)\right)=
\Bb{G}_{2} \left(\Bb{G}^{N-1}_{2}\big( \phi_{{\fo{n_{N}, ..., n_{1}}}} (y_{0})\big)\right)
= \Bb{G}_{2} \left(\phi_{{\fo{ n_{1}}}} (y_{0})\right) \\    &   =
\Bb{G}_{2} \left(\dfrac{1}{2 n_{1} -y_{0}}\right) = y_{0} -2 n_{1} -
\left\rceil{\re\left(y_{0} -2 n_{1}\right)}\right\lceil_{2} =  y_{0} \!= \!   \psi_{{\fo{n_{N}, ..., n_{1}}}} (y) \, , \
\end{align*}

\noindent
while, if $ N  \!=  \!1$, then similarly $\Bb{G}^{N}_{2}(y)= \Bb{G}_{2} \left(\phi_{{\fo{ n_{1}}}} (y_{0})\right) =  y_{0} \!= \!   \psi_{{\fo{n_{1}}}} (y)$. This proves the left-hand side equalities in \eqref{f2bsdentppre} and finishes the proof of \eqref{f2bsdentppre}.

}}\end{clash}
\end{subequations}


\subsection[\hspace{-0,25cm}. \hspace{0,05cm}Notes for Section~\ref{contgen}]{\hspace{-0,11cm}{\bf{.}} Notes for Section~\ref{contgen}}$\phantom{a}$


\vspace{0,2cm}
\begin{subequations}
\begin{clash}{\rm{\hypertarget{r24}{}\label{case24}\hspace{0,00cm}{\hyperlink{br24}{$\uparrow$}}
  \  We prove \eqref{f1wbsdenac}. In view of
\begin{align*} \tag{\ref{f3genfunbs}}
    & \hspace{-0,2cm}  \sum\limits_{n \geqslant 1}
  \, n \,\eurm{H}_{n} (x)\, {\rm{e}}^{{\fo{{\rm{i}}\pi  n  y}} }  =  \frac{1}{4 \pi^{3} {\rm{i}}} \hspace{-0,2cm}    \int\limits_{ \gamma (-1,1)} \hspace{-0,2cm}\dfrac{\lambda^{\,\prime}(y)}{\lambda(y)- \lambda(z)}\frac{d z}{(x+z)^{2}}  \, , \quad \  \ \ \im\, y \!>\! 1\,,\hspace{-0,1cm}
\\    &  \hspace{-0,2cm}
\sum\limits_{n \geqslant 1}
  \, n \,\eurm{M}_{n} (x)\, {\rm{e}}^{{\fo{{\rm{i}}\pi  n  y}} }  =  \frac{1}{4 \pi^{3} {\rm{i}}}  \hspace{-0,2cm}   \int\limits_{ \gamma (-1, 1)}\hspace{-0,2cm} \dfrac{\lambda^{\,\prime}(y)}{1\!-\!\lambda(y)\!-\! \lambda(z)}\frac{d z}{(x+z)^{2}}  \, , \  \ \  \im\, y \!>\!  1\,,\hspace{-0,1cm}
\tag{\ref{f4genfunbs}}
\end{align*}

and \eqref{f1bsdenac}
\begin{align*}
     &  \hspace{-0,15cm}
\Phi^{0}_{\infty}(x;y) := \frac{1}{2\pi i}
\int\limits_{{\fo{\gamma (-1,1)}}}\dfrac{\lambda^{\,\prime}(y)d z}{\big(\lambda(y)- \lambda(z)\big)\big(z+x\big)^{2}} \ \, ,  \\    &
\Phi^{1}_{\infty}(x;y) := \frac{1}{2\pi i}
\int\limits_{{\fo{\gamma (-1,1)}}}\dfrac{\lambda^{\,\prime}(y)d z}{\big(\lambda(y)- \lambda(z)\big)\big(x z-1\big)^{2}} \ \, ,
 \hspace{-0,1cm}
\end{align*}

\noindent we get
\begin{align}\label{f1case24}
    &   \sum\limits_{n \geqslant 1}
  \, n \,\eurm{H}_{n} (x)\, {\rm{e}}^{{\fo{{\rm{i}}\pi  n  y}} }  =  \frac{\Phi^{0}_{\infty}(x;y)}{2 \pi^{2} }\, , \quad \  \ \ \im\, y \!>\! 1\,,
\end{align}

\noindent and since
\begin{align*}
    &  \int\limits_{ \gamma (-1, 1)}\hspace{-0,2cm} \dfrac{\lambda^{\,\prime}(y)}{1\!-\!\lambda(y)\!-\! \lambda(z)}\frac{d z}{(x+z)^{2}} =
\int\limits_{ \gamma (-1, 1)}\hspace{-0,2cm} \dfrac{\lambda^{\,\prime}(y)}{\lambda(-1/z)\!-\!\lambda(y) }\frac{d z}{(x+z)^{2}}
\\    & =\left|z= - \dfrac{1}{z^{\,\prime}}\,, \ d z = \dfrac{d z^{\,\prime} }{\left(z^{\,\prime}\right)^{2}} \, , \  - \dfrac{1}{\gamma (-1,1)} = \gamma (1,-1)
\right|   \\    &
=-\int\limits_{ \gamma (-1, 1)}\hspace{-0,2cm} \dfrac{\lambda^{\,\prime}(y)}{\lambda(z)\!-\!\lambda(y) }\frac{d z}{z^{2}(x-1/z)^{2}} = \int\limits_{ \gamma (-1, 1)}\hspace{-0,2cm} \dfrac{\lambda^{\,\prime}(y)}{\lambda(y)\!-\!\lambda(z) }\frac{d z}{(xz -1)^{2}} = 2\pi i \Phi^{1}_{\infty}(x;y) \, , \
\end{align*}

\noindent we obtain
\begin{align}\label{f2case24}
    &   \sum\limits_{n \geqslant 1}
  \, n \,\eurm{M}_{n} (x)\, {\rm{e}}^{{\fo{{\rm{i}}\pi  n  y}} }  =  \frac{\Phi^{1}_{\infty}(x;y)}{2 \pi^{2} }\, , \quad \  \ \ \im\, y \!>\! 1\,.
\end{align}

\noindent The equalities \eqref{f1case24} and \eqref{f2case24} finishes the proof of \eqref{f1wbsdenac}.

}}\end{clash}
\end{subequations}


\vspace{0,2cm}
\begin{subequations}
\begin{clash}{\rm{\hypertarget{r23}{}\label{case23}\hspace{0,00cm}{\hyperlink{br23}{$\uparrow$}}
  \  \ We prove \eqref{10bsdenac}. We introduce the {\it{parity indicator}} $\eurm{p} : \Bb{Z} \mapsto \{0, 1\}$ such that $\eurm{p}_{n}  = 0$ if the integer $n$ is even and $\eurm{p}_{n}= 1$ if $n$ is odd. Let $\phi \in \Gamma_{\vartheta}$ and $\psi \in \Gamma_{\vartheta}$ be inverse to $\phi$, i.e., $\phi(\psi(z))=\psi(\phi(z))=z$ for all $z\in \Bb{H}$. Denote $\alpha := \eurm{p}_{d(\phi)}\in \{0,1\}$. Then by \eqref{f14bsdenac} and \eqref{f9bsdenac},
\begin{align*}
    &  \lambda\big(\phi (z)\big)=\alpha + (-1)^{\alpha} \lambda(z) \ , \quad
\lambda^{\,\prime} (z) = (-1)^{\alpha}
      \lambda^{\,\prime}(\phi (z))\phi^{\,\prime}(z)\,,  \  \quad    z\in \Bb{H}\,.
\end{align*}

\noindent Then
\begin{align*}
       \Phi_{\infty}^{\,\delta}(x;\phi(z))&\! :=\! \frac{1}{2\pi i}   \int\limits_{
\gamma(1,-1)} \dfrac{\lambda^{\,\prime}(\phi(z))d \zeta}{\big(\lambda(\phi(z))\!-\! \lambda(\zeta)\big)\big(x^{\delta}\zeta-(-x)^{1-\delta}\big)^{2}}   \\    &
  \\  &   =
\frac{(-1)^{\alpha}}{2\pi i\phi^{\,\prime}(z)}\int\limits_{
\gamma(1,-1)} \dfrac{\lambda^{\,\prime}\big(z \big)d \zeta}{\big(
\alpha + (-1)^{\alpha} \lambda\big(z \big)\!-\! \lambda(\zeta)\big)\big(x^{\delta}\zeta-(-x)^{1-\delta}\big)^{2}} \, , \
\end{align*}

\noindent from which
\begin{align*}
    &  (-1)^{\alpha}  \Phi^{\,\delta}_{\infty}(x;\phi(z))\phi^{\,\prime}(z)=
\frac{1}{2\pi i}\int\limits_{
\gamma(1,-1)} \dfrac{\lambda^{\,\prime}\big(z \big)d \zeta}{\big(
\alpha + (-1)^{\alpha} \lambda\big(z \big)\!-\! \lambda(\zeta)\big)\big(x^{\delta}\zeta-(-x)^{1-\delta}\big)^{2}}
\end{align*}

\noindent where the right-hand side is $ \Phi^{\,\delta}_{\infty}(x;z)$ if $\alpha =0$  and \eqref{10bsdenac} is proved for that case. But if $\alpha =1$ then
\begin{align*}
       -\Phi^{\,\delta}_{\infty}(x;\phi(z))\phi^{\,\prime}(z)&=
\frac{1}{2\pi i}\int\limits_{
\gamma(1,-1)} \dfrac{\lambda^{\,\prime}\big(z \big)d \zeta}{\big(
1 - \lambda\big(z \big)\!-\! \lambda(\zeta)\big)\big(x^{\delta}\zeta-(-x)^{1-\delta}\big)^{2}}  \\    &   =
\frac{1}{2\pi i}\int\limits_{
\gamma(1,-1)} \dfrac{\lambda^{\,\prime}\big(z \big)d \zeta}{\big(
\lambda(-1/\zeta) - \lambda\big(z \big)\big)\big(x^{\delta}\zeta-(-x)^{1-\delta}\big)^{2}}\\    &   =
\left| -1/\zeta : \gamma(1,-1) \mapsto \gamma(-1,1), \ d (-1/\zeta) = 1/\zeta^{2} \right|\\    &   =
-\frac{1}{2\pi i}\int\limits_{
\gamma(1,-1)} \dfrac{\lambda^{\,\prime}\big(z \big)d \zeta}{\big(
\lambda(\zeta) - \lambda(z )\big)\zeta^{2}\left(-\dfrac{x^{\delta}}{\zeta}-(-x)^{1-\delta}\right)^{2}}\\    &   =\frac{1}{2\pi i}\int\limits_{
\gamma(1,-1)} \dfrac{\lambda^{\,\prime}\big(z \big)d \zeta}{\big(
\lambda(z) - \lambda(\zeta )\big)\left(x^{\delta} + (-x)^{1-\delta} \zeta\right)^{2}} \\    &   =
\frac{1}{2\pi i}\int\limits_{
\gamma(1,-1)} \dfrac{\lambda^{\,\prime}\big(z \big)d \zeta}{\big(
\lambda(z) - \lambda(\zeta )\big)\big(x^{1-\delta}\zeta-(-x)^{\delta}\big)^{2}} = \Phi^{1-\delta}_{\infty}(x;z)
 \, , \
\end{align*}

\noindent because
\begin{align*}
    & \left(x^{\delta} + (-x)^{1-\delta} \zeta\right)^{2}=
\left\{
  \begin{array}{ll}
 \left(1 - x \zeta\right)^{2} =  \left( x \zeta -1\right)^{2}, & \hbox{if} \ \ \delta = 0\,, \\
 \left(x + \zeta\right)^{2}, & \hbox{if} \ \ \delta = 1\,,
  \end{array} \ \ \ = \big(x^{1-\delta}\zeta-(-x)^{\delta}\big)^{2}\,.
\right.
\end{align*}

\noindent This finished the proof of  \eqref{10bsdenac}.

}}\end{clash}
\end{subequations}


\vspace{0,5cm}
\begin{subequations}
\begin{clash}{\rm{\hypertarget{r5}{}\label{case5}\hspace{0,00cm}{\hyperlink{br5}{$\uparrow$}}
\ We prove \eqref{f24bsdenac}.  If $z \in \eusm{E}^{\infty}_{\!{\tn{\frown}}} = \Bb{H} \setminus \cup_{\, {\fo{m \in \Bb{Z}}}}\ (2m + \overline{\Bb{D}})$ then $z\not\in S^{\infty}_{\!{\tn{\frown}}}$ and
\begin{align*}
    &  z\not\in\phi_{{\fo{\eufm{n}}}}\left(\Bb{D}_{\im > 0}\right) \ , \quad  \eufm{n} \!\in \!\Bb{Z}_{\neq 0}^{\hspace{0,02cm}\Bb{N}_{\hspace{-0,02cm}\eurm{f}}} \cup\{0\} \,.
\end{align*}

\noindent Therefore, in this case,
\begin{align*}
    &  \chi_{{\fo{\phi_{{\fo{\eufm{n}}}}\left(\Bb{D}_{\im > 0}\right)}}}(z)=0 \ , \quad  \eufm{n} \!\in \!\Bb{Z}_{\neq 0}^{\hspace{0,02cm}\Bb{N}_{\hspace{-0,02cm}\eurm{f}}} \cup\{0\} \,,
\end{align*}

\noindent  while $\eurm{h}_{\eusm{E}} (z)=0$, as follows from \eqref{f16xcontgen}. So that \eqref{f24bsdenac} is proved for $z \in \eusm{E}^{\infty}_{\!{\tn{\frown}}}$.

Let
\begin{align*}
    &  z \in \Bb{D}_{\im > 0}\setminus S^{\hspace{0,02cm}{\tn{||}}}_{\!{\tn{\frown}}}= \eusm{E}^{0}_{\!{\tn{\frown}}}\ \  \sqcup \ \
 \bigsqcup\nolimits_{
  {\fo{\ \eufm{n}\in  \Bb{Z}_{\neq 0}^{\hspace{0,02cm}\Bb{N}_{\hspace{-0,02cm}\eurm{f}}}  }} } \ \   \eusm{E}^{\eufm{n}}_{\!{\tn{\frown}}} \, .
\end{align*}

\noindent In view of  \eqref{f16xcontgen}, written for $z\in \Bb{D}_{\im > 0}\setminus S^{\hspace{0,02cm}{\tn{||}}}_{\!{\tn{\frown}}}$,  $z \not\in\eusm{E}^{\infty}_{\!{\tn{\frown}}}$ in the form,
\begin{align*}\tag{\ref{f16xcontgen}}
    & \hspace{-0,5cm} \eurm{h}_{\eusm{E}} (z) =
   1 + d \left(\phi_{{\fo{\eufm{n}}}}\right) \, , \
 \ \mbox{if} \ z \in
  \phi_{{\fo{\eufm{n}}}}
\big(\eusm{E}^{0}_{\!{\tn{\frown}}}\big)= \eusm{E}^{\eufm{n}}_{\!{\tn{\frown}}}\, , \  \eufm{n} \!\in \!\Bb{Z}_{\neq 0}^{\hspace{0,02cm}\Bb{N}_{\hspace{-0,02cm}\eurm{f}}} \cup\{0\}\,,\hspace{-0,7cm}
\end{align*}

\noindent we have  $\eurm{h}_{\eusm{E}} (z) = 1$, if $z\in \eusm{E}^{0}_{\!{\tn{\frown}}}$. Together with
($\phi_{0} (z)=z$),
\begin{align*}
    & \chi_{{\fo{\phi_{0}\left(\Bb{D}_{\im > 0}\right)}}}(z)=1 \ ,\quad
 \chi_{{\fo{\phi_{{\fo{\eufm{n}}}}\left(\Bb{D}_{\im > 0}\right)}}}(z)=0 \ , \quad  \eufm{n} \!\in \!\Bb{Z}_{\neq 0}^{\hspace{0,02cm}\Bb{N}_{\hspace{-0,02cm}\eurm{f}}} \,,
\end{align*}

\noindent this yields the validity  of \eqref{f24bsdenac} for $z \in \eusm{E}^{0}_{\!{\tn{\frown}}}$,
\begin{align*}
    &  \sum\limits_{{\fo{\eufm{n} \!\in \!\Bb{Z}_{\neq 0}^{\hspace{0,02cm}\Bb{N}_{\hspace{-0,02cm}\eurm{f}}} \cup\{0\}}}}\chi_{{\fo{\phi_{{\fo{\eufm{n}}}}\left(\Bb{D}_{\im > 0}\right)}}}(z)= \eurm{h}_{\eusm{E}} (z) \ , \quad z\in \eusm{E}^{0}_{\!{\tn{\frown}}}\,.
\end{align*}

\noindent At the same time,  $\eurm{h}_{\eusm{E}} (z) = 1+N$, if $z\in \eusm{E}^{\hspace{0,05cm}n_{N}, ...,\, n_{1}}_{\!{\tn{\frown}}}$, $N\in \Bb{N}$, $ \eufm{n}=(n_{N}, ...,\, n_{1}) \!\in \!\Bb{Z}_{\neq 0}^{N} $.
But for such $z$, in view of \eqref{f22bsdenac},
\begin{align*}
    &  \hspace{-0,25cm}  \phi_{{\fo{\eufm{n}}}}\big(\Bb{D}_{\im > 0}\big)\setminus S^{\hspace{0,02cm}{\tn{||}}}_{\!{\tn{\frown}}}=\eusm{E}^{\hspace{0,05cm}n_{N}, ...,\, n_{1}}_{\!{\tn{\frown}}}\ \sqcup \!\!\!\!\!\!
 \bigsqcup\limits_{\
  {\fo{k_{K}, ...,\, k_{1} \in \Bb{Z}_{\neq 0}\,, \ K \in \Bb{N}  }} }
\!\!\!\!\!\! \eusm{E}^{\hspace{0,05cm}n_{N}, ...,\, n_{1}, k_{K}, ...,\, k_{1}}_{\!{\tn{\frown}}}
\ ,\tag{\ref{f22bsdenac}}
\end{align*}

\noindent we have $z\in \Bb{D}_{\im > 0}\setminus S^{\hspace{0,02cm}{\tn{||}}}_{\!{\tn{\frown}}} = \phi_{{\fo{\,0}}} \big(\Bb{D}_{\im > 0}\big)\setminus S^{\hspace{0,02cm}{\tn{||}}}_{\!{\tn{\frown}}}$, and it follows from
\begin{align}\label{f3case5}
    &   \chi_{{\fo{\phi_{{\fo{\eufm{m}}}}\left(\Bb{D}_{\im > 0}\right)}}}(z)\neq 0  \ , \quad  M\in \Bb{N}\,,  \ \  \eufm{m}=(m_{M}, ...,\, m_{1}) \!\in \!\Bb{Z}_{\neq 0}^{M}  \ \Leftrightarrow \\[0,4cm]  &
\eusm{E}^{\hspace{0,05cm}n_{N}, ...,\, n_{1}}_{\!{\tn{\frown}}}\!
\subset \!\phi_{{\fo{\eufm{m}}}}\big(\Bb{D}_{\im > 0}\big)\setminus S^{\hspace{0,02cm}{\tn{||}}}_{\!{\tn{\frown}}}   \!=\!\eusm{E}^{\hspace{0,05cm}m_{M}, ...,\, m_{1}}_{\!{\tn{\frown}}}\ \sqcup \!\!\!\!\!\!\!\!
 \bigsqcup\limits_{\
  {\fo{k_{K}, ...,\, k_{1} \in \Bb{Z}_{\neq 0}\,, \ K \in \Bb{N}  }} }
\!\!\!\!\!\! \eusm{E}^{\hspace{0,05cm}m_{M}, ...,\, m_{1}, \,k_{K}, ...,\, k_{1}}_{\!{\tn{\frown}}}
\nonumber \end{align}

\noindent that
\begin{multline*}
     z\in \phi_{{\fo{\,n_{N}, ..., n_{1}}}} \big(\Bb{D}_{\im > 0}\big) \, , \
z\in \phi_{{\fo{\,n_{N}, ..., n_{2}}}} \big(\Bb{D}_{\im > 0}\big) \, , \ \ldots   \, , \\[0,2cm]
z\in \phi_{{\fo{\,n_{N}}}} \big(\Bb{D}_{\im > 0}\big) \, , \
z\in \phi_{{\fo{\,0}}} \big(\Bb{D}_{\im > 0}\big) \, ,
\end{multline*}

\noindent while
\begin{align*}
    &  \chi_{{\fo{\phi_{{\fo{\eufm{n}}}}\left(\Bb{D}_{\im > 0}\right)}}}(z)=0 \ , \quad  \eufm{n} \!\in \!\Bb{Z}_{\neq 0}^{\hspace{0,02cm}\Bb{N}_{\hspace{-0,02cm}\eurm{f}}} \setminus  \{(n_{N}), (n_{N},n_{N-1}), \ldots, (n_{N}, ..., n_{1})\} \,,
\end{align*}

\noindent which yields the validity  of \eqref{f24bsdenac} for $z \in \eusm{E}^{\hspace{0,05cm}n_{N}, ...,\, n_{1}}_{\!{\tn{\frown}}}$,
\begin{align*}
    &  \sum\limits_{{\fo{\eufm{n} \!\in \!\Bb{Z}_{\neq 0}^{\hspace{0,02cm}\Bb{N}_{\hspace{-0,02cm}\eurm{f}}} \cup\{0\}}}}\chi_{{\fo{\phi_{{\fo{\eufm{n}}}}\left(\Bb{D}_{\im > 0}\right)}}}(z)= \eurm{h}_{\eusm{E}} (z) \ , \quad z\in  \eusm{E}^{\hspace{0,05cm}n_{N}, ...,\, n_{1}}_{\!{\tn{\frown}}}\,.
\end{align*}

\noindent  This completes the proof of \eqref{f24bsdenac}. Besides, we have proved that
\begin{align}\label{f1case5}
    & \sum\limits_{{\fo{\eufm{n} \!\in \!\Bb{Z}_{\neq 0}^{\hspace{0,02cm}\Bb{N}_{\hspace{-0,02cm}\eurm{f}}} \cup\{0\}}}}\chi_{{\fo{\phi_{{\fo{\eufm{n}}}}\left(\Bb{D}_{\im > 0}\right)}}}(z)= \eurm{h}_{\eusm{E}} (z) \ , \quad z\in \Bb{H}_{\re \leqslant 1 }\setminus S^{\hspace{0,02cm}{\tn{||}}}_{\!{\tn{\frown}}}  \, , \
\end{align}

\noindent and
\begin{align}\label{f2case5}
    & \sum\limits_{{\fo{\eufm{n} \!\in \!\Bb{Z}_{\neq 0}^{\hspace{0,02cm}\Bb{N}_{\hspace{-0,02cm}\eurm{f}}} \cup\{0\}}}}\chi_{{\fo{\phi_{{\fo{\eufm{n}}}}\left(\overline{\Bb{D}}_{\im > 0}\right)}}}(z)= \eurm{h}_{\eusm{E}} (z) \ , \quad z\in  \Bb{H}_{\re \leqslant 1 }  \, .
\end{align}

  }}\end{clash}
\end{subequations}


\subsection[\hspace{-0,25cm}. \hspace{0,05cm}Notes for Section~\ref{evagen}]{\hspace{-0,11cm}{\bf{.}} Notes for Section~\ref{evagen}}$\phantom{a}$


\vspace{0,2cm}
\begin{subequations}
\begin{clash}{\rm{\hypertarget{r28}{}\label{case28}\hspace{0,00cm}{\hyperlink{br28}{$\uparrow$}}
  \ We prove
the left-hand side inequality of \eqref{f1auxevagen}.

\vspace{0.15cm}
Since $x + i y  \! \in \! \eusm{E}^{\infty}_{\!{\tn{\frown}}}  \ \Rightarrow \ -x + i y  \! \in \! \eusm{E}^{\infty}_{\!{\tn{\frown}}} $,   and $\Theta_{2} (-x + i y) = \overline{\Theta_{2} (x + i y)}$, $x \in \Bb{R}$, $y> 0$, it is enough to
consider the case $\re\, z  \geqslant 0$ for which it is stated that
\begin{align}\label{f0case28}
    &  \left|\Theta_{2} (z)\right|^{4} \im^{2} z\! \leqslant\! \theta_{3}(e^{-\pi/2})^{4}
\! \leqslant\!  5  \ , \quad       z\! \in \! \eusm{E}^{\infty}_{\!{\tn{\frown}}} \cap \Bb{C}_{\re \geqslant 0}\,.
\end{align}

We first assume that $z\! \in \! \eusm{E}^{\infty}_{\!{\tn{\frown}}} \cap \Bb{C}_{\re \geqslant 0}$,
$\im\,z \leqslant 1$, and prove that
\begin{align}\label{f1case28}
    & \im (1- z)^{-1} \geqslant 1/2 \, , \quad  |z|> 1 \, , \quad   \re\, z \in [0,1] \, , \quad  \im \, z \leqslant 1 \ , \quad  z \in \Bb{H} \,.
\end{align}

\noindent For any such that there exists $a\in (0, \pi/2)$ such that
$z=\cos a +i y$,   $\sin a\leqslant y\leqslant 1$. Then it is necessary to obtain that
\begin{align*}
    &  \dfrac{1}{2} \leqslant \im \dfrac{1}{1-z} = \dfrac{\im\, z}{|1-z|^{2}} =
 \dfrac{y}{\left(1-\cos a\right)^{2} +y^{2}} \ , \ \   \sin a\leqslant y\leqslant 1  \, , \
\end{align*}

\noindent i.e.,
\begin{align*}
    &  \left(1-\cos a\right)^{2} +y^{2}\leqslant 2 y\, , \   \sin a\leqslant y\leqslant 1 \ \Leftrightarrow \
\left(1-\cos a\right)^{2}\leqslant y (2-y) \, ,  \   \sin a\leqslant y\leqslant 1 \,.
\end{align*}

\noindent But on the interval $[0,1]$ the function $y (2-y)$ increases from $0$ to $1$. Therefore the above inequality is equivalent to
\begin{align*}
    &  \left(1-\cos a\right)^{2}\leqslant  (2-\sin a)\sin a  \ \Leftrightarrow \
1+   \cos^{2} a-2 \cos a\leqslant 2\sin a-\sin^{2} a \\    &
1-\cos a \leqslant \sin a \ \Leftrightarrow \
\cos a + \sin a \geqslant 1 \ \Leftrightarrow \  \sin \left(a + \dfrac{\pi}{4}\right)\geqslant \sin \dfrac{\pi}{4} \, ,
\end{align*}

\noindent which is true because $a\in (0, \pi/2)$. Thus, \eqref{f1case28} is proved.

Then for such $z$ we deduce from the following consequence of \eqref{f3bint},
\begin{align*}
    &  \Theta_{2} (z)^{4}=(1- z)^{-2} \Theta_{4} \left({1}/{(1- z)}\right)^{4}, \ \ z \in \Bb{H}\,,
\end{align*}

\noindent and \eqref{f1case28} that
\begin{align*}
     \left|\Theta_{2} (z)\right|^{4} \im^{2} z & \leqslant  |1-z|^{2}  \left|\Theta_{2} (z)\right|^{4}=
  \left|\Theta_{4} \left(\dfrac{1}{1- z}\right)\right|^{4}
 \\    &    = \left|1  \!+ \! 2\sum\limits_{n\geqslant 1}
(-1)^{n}u^{n^2}\right|^{4}_{{\nor{u=
e^{{\fo{ \dfrac{i \pi}{1-z}}} }}}
}\leqslant \theta_{3}\left(e^{-\pi/2}\right)^{4} \ ,
\end{align*}

\noindent where by \cite[p.\! 325]{ber1},
\begin{align}\nonumber
    &  \theta_{3}\left(e^{-\pi/2}\right)^{4}= \dfrac{\pi}{\Gamma(3/4)^{4}}\dfrac{(1 + \sqrt{2})^{2}}{2}=
 \dfrac{\pi}{\Gamma(3/4)^{4}}\dfrac{3 + 2\sqrt{2}}{2} \, , \\  \nonumber   &2,914213562373095<
\dfrac{3 + 2\sqrt{2}}{2}< 2,9142135623730951  \, , \\  \nonumber   &1,393203929652002<
\dfrac{\pi}{\Gamma(3/4)^{4}} < 1,393203929652003 \ \Rightarrow \    \\    &4 < 4,06009<
\theta_{3} \left(e^{{\fo{-  \pi /2  }}}\right)^{4}<4,060094 < 5 \,.
\label{f2case28}\end{align}

\noindent It remains to prove \eqref{f0case28} for $z\! \in \! \eusm{E}^{\infty}_{\!{\tn{\frown}}} \cap \Bb{C}_{\re \geqslant 0}$,
$\im\,z > 1$, where due to $2$-periodicity of $\Theta_{2}$ it is possible to consider that $\re\, z \in [-1,1]$.
Then
\begin{align*}
    &  \left|\Theta_{2} (z)\right|^{4} \leqslant 16 e^{-\pi \im\, z}  \left|\theta_{2}\left({\rm{e}}^{\imag  \pi z}\right)\right|^{4} \leqslant 16 e^{-\pi \im\, z}  \theta_{2}\left( e^{-\pi }\right)^{4}
 \, ,
\end{align*}

\noindent and therefore
\begin{align*}
    &  \left|\Theta_{2} (z)\right|^{4} \im^{2} z \leqslant 16 \left(\im^{2} z\right) e^{-\pi \im\, z}  \theta_{2}\left( e^{-\pi }\right)^{4} \, , \
\end{align*}

\noindent where by \cite[p.\! 325]{ber1},
\begin{align*}
    &  \theta_{2}\left( e^{-\pi }\right)^{4}= \dfrac{e^{\pi}}{32} \dfrac{\pi}{\Gamma(3/4)^{4}} \, , \
\end{align*}

\noindent and the function $f (y) = y^{2}e^{-\pi y}$ has the derivative, satisfying
\begin{align*}
    & f^{\,\prime} (y) = y^{2}e^{-\pi y} = 2 y e^{-\pi y} - \pi y^{2}e^{-\pi y} =
y e^{-\pi y}\left(2 - \pi y\right) < 0  \, , \   y >2 / \pi  \, ,
\end{align*}

\noindent and therefore $f (y) \leqslant f (1) = e^{-\pi }$, $y \geqslant 1$. Thus,
\begin{align*}
    &  \left|\Theta_{2} (z)\right|^{4} \im^{2} z \leqslant 16 \left(\im^{2} z\right) e^{-\pi \im\, z}  \theta_{2}\left( e^{-\pi }\right)^{4} \leqslant 16 e^{-\pi } \dfrac{e^{\pi}}{32} \dfrac{\pi}{\Gamma(3/4)^{4}} =
\dfrac{\pi}{2 \Gamma(3/4)^{4}} \,,
\end{align*}

\noindent where
\begin{align*}
    &  \dfrac{\pi}{\Gamma(3/4)^{4}}< 1,3932039296520022532040318760206  \, , \
\end{align*}

\noindent and hence,
\begin{align*}
    &   \left|\Theta_{2} (z)\right|^{4} \im^{2} z <0,69660196482600112660201593801 < \theta_{3} \left(e^{{\fo{-  \pi /2  }}}\right)^{4} \, , \
\end{align*}

\noindent by \eqref{f2case28}. This completes the proof of \eqref{f0case28} and of the left-hand side inequality of \eqref{f1auxevagen} as well.

 }}\end{clash}
\end{subequations}

\vspace{0,2cm}
\begin{subequations}
\begin{clash}{\rm{\hypertarget{r26}{}\label{case26}\hspace{0,00cm}{\hyperlink{br26}{$\uparrow$}}
  \
\noindent In view of \eqref{f0apinttheor1}(a) and \eqref{f5case1}, $\lambda (i) =1/2$,
\begin{align*}
    &  \lambda (2i) = \left(\dfrac{1-\dfrac{1}{\sqrt{2}}}{1+\dfrac{1}{\sqrt{2}}}\right)^{2} =
\dfrac{3-2\sqrt{2}}{3+2\sqrt{2}} = 17 - 12\sqrt{2}=(0.029437251, 0.029437252) \, , \   \\    &
1-2 \lambda (2i) = 1- 2  \dfrac{3-2\sqrt{2}}{3+2\sqrt{2}} = \dfrac{3+2\sqrt{2}- 6 + 4 \sqrt{2}}{3+2\sqrt{2}}=
\dfrac{6\sqrt{2} - 3 }{3+2\sqrt{2}}= 3 \dfrac{2\sqrt{2} - 1 }{3+2\sqrt{2}}
\ ,  \\    &    \dfrac{1}{1-2 \lambda (2i)} = \dfrac{1}{3} \cdot\dfrac{3+2\sqrt{2}}{2\sqrt{2} - 1 }=
\dfrac{(2\sqrt{2} + 1)(3+2\sqrt{2})}{21}= \dfrac{6\sqrt{2} + 8 + 3 + 2\sqrt{2}}{21} \\    &   =
\dfrac{11 + 8\sqrt{2}}{21} \in \dfrac{(22.31370, 22.31371)}{21}
\in (1,062557547570, 1,062557547571) \, , \\    &
\dfrac{1}{5(1-2 \lambda (2i))}= \dfrac{11 + 8\sqrt{2}}{105} \in (0.21251150, 0.21251151)
\, .
\end{align*}

\noindent At the same time, by virtue of \eqref{f2aqinttheor1}  and \eqref{f0apinttheor1}(b),
\begin{align*}
    &   \dfrac{1}{\lambda (1 \!+\! i t)} +  \dfrac{1}{\lambda (i t)} =1 \ \Rightarrow \
\dfrac{1}{\lambda (1 \!+\! i t)}= \dfrac{\lambda (i t)-1}{\lambda (i t)} \ \Rightarrow \  \\    &
\left|\lambda (1 \!+\! i t)\right| = \dfrac{\lambda (i t)}{1- \lambda (i t)} =
\dfrac{\lambda (i t)}{ \lambda (i (1/t))} \geqslant \dfrac{\lambda (2i)}{16} e^{\pi /t} =
\dfrac{ 17 - 12\sqrt{2}}{16} e^{\pi /t} \ , \quad  t \in (0,2]\,.
\end{align*}

\noindent Here
\begin{align*}
    &  \dfrac{16}{ 17 - 12\sqrt{2}}= 16 \cdot \dfrac{3+2\sqrt{2}}{3-2\sqrt{2}} =
16 \cdot (3+2\sqrt{2})^{2} = 16 (17 + 12\sqrt{2}) \\    &    = 272 + 192 \sqrt{2}\in ( 543.52900, 543.52901)
\leqslant 555\,.
\end{align*}

\noindent Since $\lambda$ is $2$-periodic, we obtain
\begin{align}\label{f1case26}
    &  \left|\lambda (\pm 1 \!+\! i t)\right| \geqslant \dfrac{ 17 - 12\sqrt{2}}{16} e^{\pi /t}  \geqslant
\dfrac{e^{\pi /t}}{555}\ , \quad  t  \in (0,2]\,,  \\    &
\dfrac{1}{5(1-2 \lambda (2i))}= \dfrac{11 + 8\sqrt{2}}{105} \leqslant  0.21251151\,.
\label{f2case26}\end{align}

   }}\end{clash}
\end{subequations}

\vspace{0,2cm}
\begin{subequations}
\begin{clash}{\rm{\hypertarget{r27}{}\label{case27}\hspace{0,00cm}{\hyperlink{br27}{$\uparrow$}}
  \ The expression for the integral follows from the following transforms
  \begin{align*}
      \int\limits_{0}^{2}\! \dfrac{ d t}{t^{2} \left(a+ \dfrac{ 17 - 12\sqrt{2}}{16} e^{\pi /t}\right) }  &=
 \int\limits_{1/2}^{\infty} \dfrac{ d t}{a+ \dfrac{ 17 - 12\sqrt{2}}{16} e^{\pi t} } =
\left|t^{\,\prime} = e^{\pi t}\,, \ t = \dfrac{\log t^{\,\prime}}{\pi}\right| \\    &
= \dfrac{16}{\pi \left( 17 - 12\sqrt{2}\right)} \int\limits_{e^{\pi /2}}^{\infty} \dfrac{d t}{t (t+b)} =
\left|b := \dfrac{16 a}{17 - 12\sqrt{2}}\right|  \\    &   = \dfrac{16}{\pi  b\left( 17 - 12\sqrt{2}\right)}
\int\limits_{e^{\pi /2}}^{\infty} \left(\dfrac{1}{t}-\dfrac{1}{t+b}\right) d t  \\    &
= \dfrac{16}{\pi b \left( 17 - 12\sqrt{2}\right)} \left.\log \dfrac{t}{t+b}\right|_{t=e^{\pi /2}}^{t=\infty}
\\    & =  \dfrac{16}{\pi b \left( 17 - 12\sqrt{2}\right)} \log \left(1 + b e^{-\pi /2}\right) \\    &
= \dfrac{1}{\pi a}  \log  \left( 1 + \dfrac{16 e^{-\pi /2}}{17 - 12\sqrt{2}}a\right) \,.
\end{align*}

   }}\end{clash}
\end{subequations}

\vspace{0,2cm}
\begin{subequations}
\begin{clash}{\rm{\hypertarget{r29}{}\label{case29}\hspace{0,00cm}{\hyperlink{br29}{$\uparrow$}}
  \
We use the following results of the numerical calculations
\begin{align*}
    & \dfrac{55 + 40\sqrt{2}}{21} < 5,313  \, , \\    &
 1+ \dfrac{16 e^{-\pi /2}}{17 - 12\sqrt{2}}< 113,9884 \, , \\    &
\dfrac{5 \sqrt{2}}{\pi} <  2,251 < \dfrac{9}{4} \, , \\    &
\log \left(1 + \dfrac{16 e^{-\pi /2}}{17 - 12\sqrt{2}}\right)< 4,7361 \, , \\    &
 \dfrac{5 \sqrt{2}}{\pi} \log \left(1 +\dfrac{16 e^{-\pi /2}}{17 - 12\sqrt{2}}\right)<
10,6609536465591 \, ,  \\    &
\dfrac{55 + 40\sqrt{2}}{21}+ \dfrac{5 \sqrt{2}}{\pi} \log \left(1 +\dfrac{16 e^{-\pi /2}}{17 - 12\sqrt{2}}\right)< 15,974
\end{align*}

\noindent  which yield that
\begin{align*}
    &   I_{\delta} (z)\!\leqslant \! 16 + (9/4) \log \left( 1 + \left|\lambda (-1/z)\right|\right) \ , \quad
z \in \eusm{E}_{\!{\tn{\frown}}}^{\hspace{0,05cm}0}
 \,.
\end{align*}

\vspace{0.25cm} To prove \eqref{f12evagen}, observe that
\begin{align*}
   \log \left(1 + \left|\lambda(x+i y)\right|\right)  & \leqslant
     \log \left(1 + 16\, {\rm{e}}^{{\fo{-  \pi y + \pi / y }}}\right)  \\    &  \dfrac{\pi}{y} +  \log \left( {\rm{e}}^{{\fo{ - \pi / y }}}+ 16 {\rm{e}}^{{\fo{-  \pi y }}}\right)\leqslant   \dfrac{\pi}{y} + \log 17 \leqslant \pi \left(1
+  {1}/{y}\right)    \, , \
\end{align*}

\noindent and therefore it follows from \eqref{f6evagen}
\begin{align}
    &  \hspace{-0,3cm} I_{\delta} (z)\!\leqslant\! 16\! +\! (9/4) \log \Big( 1 \!+ \!\left|\lambda \big(\Bb{G}_{2} (z)\big)\right|\Big), \
\Bb{G}_{2} (z)\!\in\! \mathcal{F}^{\,{\tn{||}}}_{{\tn{\square}}} \setminus\overline{\Bb{D}}\!\subset \! \eusm{E}^{\infty}_{\!{\tn{\frown}}},
\    z\in\eusm{E}_{\!{\tn{\frown}}}^{\hspace{0,05cm}0}  \,,\hspace{-0,3cm}\tag{\ref{f6evagen}}
\end{align}

\noindent that
\begin{align*}
    &  I_{\delta} (z)\!\leqslant\! 16\! +\! \dfrac{9 \pi}{4} +   \dfrac{9 \pi}{4\im\,\Bb{G}_{2} (z) }\leqslant \dfrac{147 \pi}{20} \left(1 +  \dfrac{1}{\im\,\Bb{G}_{2} (z) }\right) \,,
\end{align*}

\noindent i.e., \eqref{f12evagen} holds, because
\begin{align*}
    &16 +   \dfrac{9 \pi}{4} <23,068583471 < 23,09 < \dfrac{147 \pi}{20} \ .
\end{align*}

 }}\end{clash}
\end{subequations}

\vspace{0,2cm}
\begin{subequations}
\begin{clash}{\rm{\hypertarget{r30}{}\label{case30}\hspace{0,00cm}{\hyperlink{br30}{$\uparrow$}}
  \  In view of \cite[p.\! 325]{ber1}, we have
\begin{align*}
    &  \dfrac{9 \pi}{4} < 7,068583471\, , \\    &
16 + \dfrac{9 \pi}{4} < 23,0685834705771\, , \\    &
\dfrac{16}{\pi} + \dfrac{9 }{4} < 7,342958178941 < \dfrac{147}{20}= 7,35  \ ,  \\    &
\theta_{3}(e^{-\pi/2})^{4} =   \dfrac{\pi}{\Gamma\left(\dfrac{3}{4}\right)^{4}} \dfrac{\left(1 + \sqrt{2}\right)^{2}}{2} <  4,0600937869433563 \, ,
 \\    &
\dfrac{147}{20} \cdot \theta_{3}(e^{-\pi/2})^{4} < \dfrac{147}{20} \cdot 4,0600937869433563 <
29,8416893341 < 30\,.
\end{align*}
 }}\end{clash}
\end{subequations}

\vspace{0,2cm}
\begin{subequations}
\begin{clash}{\rm{\hypertarget{r31}{}\label{case31}\hspace{0,00cm}{\hyperlink{br31}{$\uparrow$}}
  \
Here we use the inequality $1 + 60 \pi < 20 \pi^{2}$ which holds because
\begin{align*}
    & 60 \pi < 188,49555921538759430775860299677 \, , \\    &
1 + 60 \pi <189,49555921538759430775860299677 \, , \\    &  20 \pi^{2}=
197,39208802178717237668981999752... > 190 > 1 + 60 \pi \,.
\end{align*}

 }}\end{clash}
\end{subequations}

\vspace{0,2cm}
\begin{subequations}
\begin{clash}{\rm{\hypertarget{r32}{}\label{case32}\hspace{0,00cm}{\hyperlink{br32}{$\uparrow$}}
  \   Actually, we have
\begin{align*}
    & e^{\pi} <23,140692632779269005729086367949  \, , \\    &
40 e^{\pi} < 925,6277053111707602291634547178 \, , \\    &
 \pi^{6}= 961,38919357530443703021944365242...> 961 > 40 e^{\pi} \,.
\end{align*}

}}\end{clash}
\end{subequations}

\vspace{0,2cm}
\begin{subequations}
\begin{clash}{\rm{\hypertarget{r33}{}\label{case33}\hspace{0,00cm}{\hyperlink{br33}{$\uparrow$}}
  \ We prove \eqref{f2fevagen}. In view of \eqref{f2intth1},
\begin{align*}
     & \left| \eurm{H}_{n} (x)\right| \ , \
         \left| \eurm{M}_{n} (x)\right|\leqslant\dfrac{\pi^{6} n^{2}}{4}\, , \quad x\in \Bb{R}  \ , \ n \in \Bb{Z}_{\neq 0}\,.
\end{align*}

\noindent
 By using \eqref{f3intth1}, we get
\begin{align*}
    &  \left| \eurm{H}_{n} (x)\right| \ , \
         \left| \eurm{M}_{n} (x)\right|\leqslant
         \dfrac{\pi^{6}n^{2}}{4 x^{2}} \leqslant \dfrac{\pi^{6}n^{2}}{
         2(1+ x^{2})}\, , \quad |x| \geqslant 1 \,,
\end{align*}

\noindent while
\begin{align*}
     & \left| \eurm{H}_{n} (x)\right| \ , \
         \left| \eurm{M}_{n} (x)\right|\leqslant \dfrac{\pi^{6} n^{2}}{4}\leqslant \dfrac{\pi^{6} n^{2}}{2(1+ x^{2})}\, , \quad x\in [-1,1] \,.
\end{align*}

\noindent This proves \eqref{f2fevagen}.
}}\end{clash}
\end{subequations}

\vspace{0,2cm}
\begin{subequations}
\begin{clash}{\rm{\hypertarget{r34}{}\label{case34}\hspace{0,00cm}{\hyperlink{br34}{$\uparrow$}}
  \
We prove \eqref{f3fevagen}.  The explicit integral formula
for $\eurm{H}_{0}$ written after \eqref{f2intcor1},
\begin{align*}
      &   \eurm{H}_{0} (x) =\dfrac{i}{2 \pi^{2 }}\int\limits_{\Bb{R}}
\dfrac{\dfrac{1 \! - \! i \eurm{y}(t) }{ 1 \! + \! i \eurm{y}(t) }
}{ x^{2}- \left( \dfrac{1 \! - \! i \eurm{y}(t) }{ 1 \! + \! i \eurm{y}(t) }\right)^{2}} \, \dfrac{d t}{t^{2} + 1/4} \\  & =
\dfrac{i}{4 \pi^{2 }}\int\limits_{\Bb{R}}
\dfrac{1 }{ x- \dfrac{1 \! - \! i \eurm{y}(t) }{ 1 \! + \! i \eurm{y}(t) }} \, \dfrac{d t}{t^{2} + 1/4} -
\dfrac{i}{4 \pi^{2 }}\int\limits_{\Bb{R}}
\dfrac{1 }{ x+ \dfrac{1 \! - \! i \eurm{y}(t) }{ 1 \! + \! i \eurm{y}(t) }} \, \dfrac{d t}{t^{2} + 1/4}
\end{align*}

\noindent yields
\begin{align*}
     & \left| \eurm{H}_{0} (x) \right|\leqslant
     \dfrac{1}{2 \pi^{2 }}\int\limits_{\Bb{R}}
\dfrac{1 }{ \im \dfrac{i \eurm{y}(t) \! - \! 1 }{i \eurm{y}(t)  \! + \! 1 }} \, \dfrac{d t}{t^{2} + 1/4} =
 \dfrac{1}{4 \pi^{2 }}\int\limits_{\Bb{R}} \dfrac{1 \! + \!  \eurm{y}(t)^{2}}{\eurm{y}(t)}\dfrac{d t}{t^{2} + 1/4} \ , \
\end{align*}

\noindent where $\eurm{y}(-t) = 1/\eurm{y}(t)$ and therefore
\begin{align*}
     & \int\limits_{\Bb{R}} \dfrac{1 \! + \!  \eurm{y}(t)^{2}}{\eurm{y}(t)}\dfrac{d t}{t^{2} + 1/4} =
     \int\limits_{\Bb{R}} \left( \eurm{y}(t) + \eurm{y}(-t)\right)\dfrac{d t}{t^{2} + 1/4} = 2 \int\limits_{0}^{+\infty}\left( \eurm{y}(t) + \eurm{y}(-t)\right)\dfrac{d t}{t^{2} + 1/4}  \\  & =
     2 \int\limits_{0}^{+\infty}\left( \eurm{y}(t) + \dfrac{1}{\eurm{y}(t)}\right)\dfrac{d t}{t^{2} + 1/4}
     \ .
\end{align*}

\noindent Thus, by using $1 \leqslant \he (x)\leqslant 2$, $0 \leqslant x \leqslant 1/2$ (see \cite[p.\! 35, (A.9c)]{bh1}), which actually gives  $ \he (x)\geqslant 2$, $1/2 \leqslant x < 1$,  we get
\begin{align*}
     & \left| \eurm{H}_{0} (x) \right|\leqslant \dfrac{1}{2 \pi^{2 }}\int\limits_{0}^{+\infty}\left( \eurm{y}(t) + \dfrac{1}{\eurm{y}(t)}\right)\dfrac{d t}{t^{2} + 1/4} \\  & =
      \dfrac{1}{2 \pi^{2 }}\int\limits_{0}^{+\infty}
\left(
\dfrac{\he\left(\dfrac{1}{2}\! -\! \dfrac{ t}{\sqrt{4 t^{2}\! +\!1}}\right)}{\he \left(\dfrac{1}{2}\! + \!\dfrac{ t}{\sqrt{4 t^{2} \!+\!1}}\right)}
 +
\dfrac{\he\left(\dfrac{1}{2}\! +\! \dfrac{ t}{\sqrt{4 t^{2}\! +\!1}}\right)}{\he \left(\dfrac{1}{2}\! - \!\dfrac{ t}{\sqrt{4 t^{2} \!+\!1}}\right)}
\right) \dfrac{d t}{t^{2} + 1/4} \\  & \leqslant
      \dfrac{1}{2 \pi^{2 }}\int\limits_{0}^{+\infty}
  \left( 1+ \he\left(\dfrac{1}{2}\! +\! \dfrac{ t}{\sqrt{4 t^{2}\! +\!1}}\right)\right)
     \dfrac{d t}{t^{2} + 1/4}\\  & \leqslant
    \dfrac{1}{2 \pi^{2 }}\int\limits_{0}^{+\infty}    \dfrac{d t}{t^{2} + 1/4} +
   \dfrac{1}{2 \pi^{2 }}\int\limits_{0}^{+\infty}
     \he \left(\dfrac{1}{2}\! + \!\dfrac{ t}{\sqrt{4 t^{2} \!+\!1}}\right)
       \dfrac{d t}{t^{2} + 1/4}
     \ .
\end{align*}

\noindent Applying   $ \he (x)\leqslant 3 + \log (1-x)^{-1}$, $0 \leqslant x < 1$ (see \cite[p.\! 35, (A.9d)]{bh1}), we deduce that
\begin{align*}
    & \he \left(\dfrac{1}{2}\! + \!\dfrac{ t}{\sqrt{4 t^{2} \!+\!1}}\right) -3\leqslant  \log
\dfrac{1}{\dfrac{1}{2}\! - \!\dfrac{ t}{\sqrt{4 t^{2} \!+\!1}}} =
\log \dfrac{2\sqrt{4 t^{2} \!+\!1}}{\sqrt{4 t^{2} \!+\!1} - 2 t} \\    &   =
\log 2 + \log \sqrt{4 t^{2} \!+\!1} + \log \left(\sqrt{4 t^{2} \!+\!1} + 2 t\right) \leqslant
\log 2 + \log \sqrt{4 t^{2} \!+\!1}  \\    & + \log \left(\sqrt{4 t^{2} \!+\!1} + \sqrt{4 t^{2} \!+\!1}\right)=
2\log 2 + 2\log \sqrt{4 t^{2} \!+\!1} \, , \
\end{align*}

\noindent and therefore from the above inequality we obtain
\begin{align*}
    &  \left| \eurm{H}_{0} (x) \right|\leqslant  \dfrac{4 + 2\log 2}{ 2\pi^{2 }}\int\limits_{0}^{+\infty}    \dfrac{d t}{t^{2} + 1/4} +
\dfrac{1}{ \pi^{2 }}\int\limits_{0}^{+\infty} \dfrac{\log \sqrt{4 t^{2} \!+\!1}}{t^{2} + 1/4}
d t  \\ & =
  \dfrac{4 + 2\log 2}{ \pi^{2 }}\int\limits_{0}^{+\infty}    \dfrac{d t}{t^{2} + 1} +
\dfrac{2}{ \pi^{2 }}\int\limits_{0}^{+\infty}  \dfrac{\log \sqrt{ t^{2} \!+\!1}}{t^{2} + 1}
d t=\dfrac{2 + \log 2}{ \pi} \\    & +
\dfrac{1}{ \pi^{2 }}\int\limits_{0}^{+\infty}  \dfrac{\log ( t^{2} \!+\!1)}{t^{2} + 1}
d t    =
\dfrac{2 + \log 2}{ \pi} +\dfrac{1}{ 2\pi^{2 }}\int\limits_{0}^{+\infty}  \dfrac{\log ( 1 \!+\!t)}{
(t + 1)\sqrt{t}} d t \  \boxed{=}   \\    &
\int\limits_{0}^{+\infty} \dfrac{\log ( 1 \!+\!t)}{
(t + 1)\sqrt{t}} d t = \int\limits_{1}^{+\infty}  \dfrac{\log t}{
t\sqrt{(t - 1)}} d t =  \int\limits_{0}^{1} \dfrac{\log(1/ t)}{
\sqrt{\left(1 - t\right) t}} d t \\    & = \ \mbox{\cite[p.\! 490, 2.6.5.4]{pru}} \      =
- B (1/2, 1/2) \left[\psi (1/2) - \psi (1)\right] \\    &  = B (1/2, 1/2)\left(\psi (1) - \psi (1/2) \right)
  =
\pi \left((-\gamma) - (- \gamma - \log 4) \right) = \pi  \log 4 \, , \   \\    &
\boxed{=}   \ \dfrac{2 + \log 2}{ \pi} +\dfrac{\log 4}{ 2\pi} =\dfrac{2 + \log 2}{ \pi} +\dfrac{\log 2}{ \pi} =
2 \,\dfrac{1 + \log 2}{ \pi}< 1,077891\,.
\end{align*}

\noindent So that
\begin{align*}
    &  \left| \eurm{H}_{0} (x) \right|\leqslant  \dfrac{3}{2} \ , \quad x \in \Bb{R} \ .
\end{align*}

\noindent
 By using \eqref{f24int}, we get
\begin{align*}
    &  \left| \eurm{H}_{0} (x)\right| \leqslant
         \dfrac{3}{2 x^{2}} \leqslant \dfrac{3}{
         1+ x^{2}}\, , \quad |x| \geqslant 1 \,,
\end{align*}

\noindent while
\begin{align*}
     & \left| \eurm{H}_{0} (x)\right| \leqslant \dfrac{3}{2}\leqslant  \dfrac{3}{
         1+ x^{2}}\, , \quad x\in [-1,1] \,.
\end{align*}

\noindent This proves
\begin{align*}
    & \left| \eurm{H}_{0} (x) \right|\leqslant \dfrac{3}{1 + x^{2}} \ , \quad  x \in \Bb{R} \ ,
\end{align*}

\noindent and completes the proof of \eqref{f3fevagen}.

}}\end{clash}
\end{subequations}

\vspace{0,2cm}
\begin{subequations}
\begin{clash}{\rm{\hypertarget{r40}{}\label{case40}\hspace{0,00cm}{\hyperlink{br40}{$\uparrow$}}
  \ Actually, let $ x\!\in\! \Bb{R}$, $\delta \in \{0,1\}$, $\zeta\in \gamma (-1,1)$ and
$z\in \Bb{H}\setminus S_{\!{\tn{\frown}}}^{\infty}$. Since
\begin{align}\label{f1case40}
      &   \dfrac{d}{d \zeta} \ \ \dfrac{\zeta^{\delta}}{x^{\delta} \zeta-(-x)^{1-\delta}}=
      -\dfrac{1}{\big(x^{\delta} \zeta-(-x)^{1-\delta}\big)^{2}}  \,,
\end{align}

\noindent and
\begin{align*}
      &    \dfrac{d}{d z} \ \dfrac{\lambda (z)\lambda^{\,\prime} (\zeta)}{
      \lambda (\zeta)\left(\lambda (\zeta)-\lambda (z)\right)} =
      \dfrac{\lambda^{\,\prime} (\zeta)}{\lambda (\zeta)}\  \dfrac{d}{d z} \
      \left(-1 + \dfrac{\lambda (\zeta)}{\lambda (\zeta)-\lambda (z)}\right) \\    &   =
       \dfrac{\lambda^{\,\prime} (\zeta)}{\lambda (\zeta)}\ \dfrac{\lambda (\zeta)\lambda^{\,\prime} (z)}{\big(\lambda (\zeta)-\lambda (z)\big)^{2}}=
       \lambda^{\,\prime} (z)\dfrac{\lambda^{\,\prime} (\zeta)}{\big(
       \lambda (\zeta)-\lambda (z)\big)^{2}} \, , \
\end{align*}

\noindent we deduce from \eqref{f10fevagen} that
\begin{align*}
      \dfrac{d}{d z} \ \  \Psi^{\,\delta}_{\infty}(x;z) &  =  \dfrac{d}{d z} \ \
     \frac{1}{2\pi \imag }\int\limits_{{\fo{\gamma (-1,1)}}}
\dfrac{\lambda (z)\lambda^{\,\prime} (\zeta)}{\lambda (\zeta)\left(\lambda (\zeta)-\lambda (z)\right)}\dfrac{\zeta^{\delta}{\rm{d}}\zeta}{\big(x^{\delta}\zeta-(-x)^{1-\delta}\big)}  \\ & = \frac{1}{2\pi \imag }\int\limits_{{\fo{\gamma (-1,1)}}}
\dfrac{\lambda^{\,\prime} (z)\lambda^{\,\prime} (\zeta)}{\big(
       \lambda (\zeta)-\lambda (z)\big)^{2}}
\dfrac{\zeta^{\delta}{\rm{d}}\zeta}{\big(x^{\delta}\zeta-(-x)^{1-\delta}\big)}  \\ & =
\frac{1}{2\pi \imag }\int\limits_{{\fo{\gamma(-1,1)}}}
\dfrac{\zeta^{\delta}}{\big(x^{\delta}\zeta-(-x)^{1-\delta}\big)}
d \dfrac{\lambda^{\,\prime} (z)}{\lambda (z)-\lambda (\zeta)}= \left|\lambda (\pm 1) =\infty\right|  \\ & =
- \frac{1}{2\pi \imag }\int\limits_{{\fo{\gamma(-1,1)}}}
 \dfrac{\lambda^{\,\prime} (z)}{\lambda (z)-\lambda (\zeta)}
 d \dfrac{\zeta^{\delta}}{\big(x^{\delta}\zeta-(-x)^{1-\delta}\big)} \\ & \stackrel{\eqref{f1case40}}{=  }
 \frac{1}{2\pi \imag }\int\limits_{{\fo{\gamma(-1,1)}}}
 \dfrac{\lambda^{\,\prime} (z)}{\lambda (z)-\lambda (\zeta)} \dfrac{{\rm{d}} \zeta}{\big(x^{\delta}\zeta-(-x)^{1-\delta}\big)^{2}}\stackrel{\eqref{f1bsdenac}}{=  } \Phi^{\,\delta}_{\infty}(x;z)\,,
\end{align*}

\noindent i.e.,
\begin{align}\label{f2case40}
   \dfrac{d}{d z} \   \Psi^{\,\delta}_{\infty}(x;z) =  \Phi^{\,\delta}_{\infty}(x;z) \ , \quad z\in \Bb{H}\setminus S_{\!{\tn{\frown}}}^{\infty} \, , \  x\in \Bb{R}\setminus\{0\} \, , \  \delta \in \{0,1\} \,.
\end{align}

}}\end{clash}
\end{subequations}

\vspace{0,2cm}
\begin{subequations}
\begin{clash}{\rm{\hypertarget{r41}{}\label{case41}\hspace{0,00cm}{\hyperlink{br41}{$\uparrow$}}
  \ We prove \eqref{f12fevagen}. Let $ x\!\in\! \Bb{R}$,
$z\in \Bb{H}\setminus S_{\!{\tn{\frown}}}^{\infty}$  and
\begin{align*}\tag{\ref{f10fevagen}}
      & \Psi^{\,\delta}_{\infty}(x;z)= \frac{1}{2 }\!\!\!\!\!\int\limits_{{\fo{\gamma (-1,1)}}}\!\!\!
\dfrac{\lambda (z)\Theta_{4}(\zeta)^{4}\, \zeta^{\delta}{\rm{d}}\zeta}{
\big(\lambda(\zeta)-\lambda(z)\big)\big(x^{\delta}\zeta-(-x)^{1-\delta}\big)} \, , \   \  \delta\! \in\! \{0,1\}\,,
\end{align*}

\noindent or, by the identity  $\lambda^{\,\prime} (z)\!= \! \imag  \pi \lambda (z)
       \Theta_{4}(z)^{4}$, $z\in \Bb{H}$ (see \eqref{f19int}, \eqref{f2int}), it can be written as follows
\begin{align}\label{f1case41}
    &    \Psi^{\,\delta}_{\infty}(x;z)\!= \!\frac{1}{2\pi \imag }\!\!\!\!\!\int\limits_{{\fo{\gamma (-1,1)}}}\!\!\!\!\!
\dfrac{\lambda (z)\lambda^{\,\prime} (\zeta)}{\lambda (\zeta)\left(\lambda (\zeta)-\lambda (z)\right)}\dfrac{\zeta^{\delta}{\rm{d}}\zeta}{\big(x^{\delta}\zeta-(-x)^{1-\delta}\big)} \, , \ \delta \!\in \!\{0,1\} \,.
\end{align}

\noindent  For $z\in \fet \setminus [-1\!+\!2{\rm{i}}, 1+\!2{\rm{i}}] $ introduce
\begin{align}\label{f2case41}
    &  \Psi^{\hspace{0,02cm}\delta}_{{\tn{\sqcap}}}(x;z)\!= \!\frac{1}{2\pi \imag }\!\!\!\!\!\int\limits_{{\fo{\Pi (-1,1)}}}\!\!\!\!\!
\dfrac{\lambda (z)\lambda^{\,\prime} (\zeta)}{\lambda (\zeta)\left(\lambda (\zeta)-\lambda (z)\right)}\dfrac{\zeta^{\delta}{\rm{d}}\zeta}{\big(x^{\delta}\zeta-(-x)^{1-\delta}\big)} \, , \ \delta \!\in \!\{0,1\} \,,
\end{align}

\noindent where the contour $\Pi (-1,1)= (-1,-1+\!2{\rm{i}}]\cup [-1\!+\!2{\rm{i}}, 1+\!2{\rm{i}}] \cup [1+\!2{\rm{i}}, 1)$ passes from $-1$ to $1$.

By  transforming the contour $\gamma(-1,1)$ of integration in \eqref{f1case41} to $\Pi(-1,1)$ and using
 Lemma~\ref{lemoneto},  we obtain from the residue theorem \cite[p.\! 112]{con} that
\begin{align}\label{f3case41}
    &
\begin{array}{ll}
  \Psi^{\hspace{0,02cm}\delta}_{\infty}(x;z)\! =\!\Psi^{\hspace{0,02cm}\delta}_{{\tn{\sqcap}}}(x;z) \, ,   &   \quad
z\in \fetd  \sqcup  \Bb{H}_{\im >2} ,    \\[0,2cm]
   \Psi^{\hspace{0,02cm}\delta}_{\infty}(x;z)\! =\!\Psi^{\hspace{0,02cm}\delta}_{{\tn{\sqcap}}}(x;z) + \dfrac{z^{\delta}}{\big(x^{\delta} z-(-x)^{1-\delta}\big)}\, ,   &    \quad
z\in \fetu \cap \Bb{H}_{\im <2}\,  .
\end{array}
\end{align}

\noindent This means that the function $\Psi^{\hspace{0,02cm}\delta}_{{\tn{\sqcap}}}(x;z)  +  z^{\delta} (x^{\delta}z - (-x)^{1-\delta})^{-1}$, being holomorphic on
$\Bb{H}_{\im <2}\cap \fet $, coincides  on the set  $\Bb{H}_{\im <2} \cap \fetu  \subset  \Bb{H}_{\im <2}\cap \fet$  with the  function $\Psi^{\hspace{0,02cm}\delta}_{\infty}(x;z)$, which is  holomorphic on $\fetu  \supset  \Bb{H}_{\im <2} \cap \fetu$.

By the uniqueness theorem  for analytic functions (see \cite[p.\! 78]{con}),  we find that
 for $z \in  \fetd \subset \Bb{H}_{\im <2}\cap \fet$, the analytic
extension $\Psi^{\,\delta}_{\Bb{H}}(x;z)$ of the function $\Psi^{\hspace{0,02cm}\delta}_{\infty}(x;z)$ from
$\fetu $  to $\Bb{H}$ (see \eqref{f9fevagen})  equals the expression  $H(z):=\Psi^{\hspace{0,02cm}\delta}_{\infty}(x;z)  +  z^{\delta} (x^{\delta}z - (-x)^{1-\delta})^{-1}$ since
 $\Psi^{\hspace{0,02cm}\delta}_{{\tn{\sqcap}}}(x;z)=
\Psi^{\hspace{0,02cm}\delta}_{\infty}(x;z) $ holds for all  $z\in
\fetd $, in view of \eqref{f3case41}.

 But by  \eqref{f8bsdenac}(a) and \eqref{f6wdenac}, we see that $\fetd \!\subset \!\eusm{E}^{0}_{\!{\tn{\frown}}}\!\subset\!\Bb{H} \big\backslash  S_{\!{\tn{\frown}}}^{\infty}$, and
since $\eusm{E}^{0}_{\!{\tn{\frown}}}$ is simply connected it follows that
  the latter function $H(z)$  is actually holomorphic on $\eusm{E}^{0}_{\!{\tn{\frown}}}$.

  So that the equality
\begin{align*}\tag{\ref{f12fevagen}}
      &  \Psi^{\hspace{0,02cm}\delta}_{\Bb{H}}(x;z) = \Psi^{\,\delta}_{\infty}(x;z)
        + \dfrac{z^{\delta}}{x^{\delta} z-(-x)^{1-\delta}} \ , \quad
        z\in  \eusm{E}^{0}_{\!{\tn{\frown}}} \subset \Bb{H}\setminus S_{\!{\tn{\frown}}}^{\infty}\,,
\end{align*}

\noindent is proved  for arbitrary $\delta \in \{0,1\}$ and $x\in \Bb{R}$.

}}\end{clash}
\end{subequations}

\vspace{0,2cm}
\begin{subequations}
\begin{clash}{\rm{\hypertarget{r36}{}\label{case36}\hspace{0,00cm}{\hyperlink{br36}{$\uparrow$}}
  \ We prove \eqref{f1conhypfouser}. For $z= x + i y \in\Bb{H}$, $\overline{z}:=x + i y $ and
  $t \in \Bb{R}\setminus\{0\} $ the Poisson kernel
\begin{align*}
    &  2 \pi i P_{z} (t) :=  \dfrac{1}{t-z} - \dfrac{1}{t-\overline{z}} = \dfrac{1}{t - x - i y} - \dfrac{1}{t - x + i y} = \dfrac{2 i y}{(t - x)^{2} + y^{2}}  \, ,
\end{align*}

\noindent possesses the following property
\begin{align*}
     \dfrac{2 \pi i P_{z} (-1/t)}{t^{2}}&  = \dfrac{1}{t^{2}}\left(\dfrac{1}{(-1/t)-z} - \dfrac{1}{(-1/t)-\overline{z}}\right)= \dfrac{1}{t\left(-1-zt\right)} - \dfrac{1}{t\left(-1-t\overline{z}\right)} \\    &   =
\dfrac{t\left(-1-t\overline{z}\right) - t\left(-1-zt\right)}{t^{2}\left(-1-zt\right)\left(-1-t\overline{z}\right)}=
\dfrac{-t - t^{2}\overline{z} + t + t^{2} z}{t^{2}\left(-1-zt\right)\left(-1-t\overline{z}\right)} \\    &
=\dfrac{z-\overline{z}}{|z|^{2}\left(\dfrac{-1}{z}-t\right)\left(\dfrac{-1}{\overline{z}}-t\right) }   =
\dfrac{\dfrac{z-\overline{z}}{z \overline{z}}}{\left(t- \left(\dfrac{-1}{z}\right)\right)\left(t-\left(\dfrac{-1}{\overline{z}}\right)\right) } \\    &   =
\dfrac{\left(\dfrac{-1}{z}\right)-\left(\dfrac{-1}{\overline{z}}\right)}{
\left(t- \left(\dfrac{-1}{z}\right)\right)\left(t-\left(\dfrac{-1}{\overline{z}}\right)\right) } =
\dfrac{1}{t- \left(\dfrac{-1}{z}\right)}- \dfrac{1}{t-\left(\dfrac{-1}{\overline{z}}\right)} \\    &   =
2 \pi i P_{-1/z} (t) \, , \
\end{align*}

\noindent i.e.,
\begin{align}\label{f1case36}
    &  \dfrac{ P_{z} (-1/t)}{t^{2}} = P_{-1/z} (t ) \ , \quad  z \in\Bb{H}\, , \ \ t \in \Bb{R}\setminus\{0\}\,.
\end{align}

\noindent Since for every $z \in\Bb{H}$ the Poisson kernel $P_{z} (t)$ of the variable $t$ belongs to
$L^1 (\Bb{R})$, by \eqref{f05aint} and  Jordan's lemma \eqref{f4zmres}, for every $n \geqslant 1$ we have
\begin{align*}
    & 2 \pi i  \eurm{h}_{-n}^{\star}(P_{z})\!:=\!2 \pi i \int\limits_{\Bb{R}}\! P_{z}(t)\,
    e^{{\fo{{\rm{i}}\pi  n t }}} d t= \int\limits_{\Bb{R}}  e^{{\fo{{\rm{i}}\pi  n t }}} \left(\dfrac{1}{t-z} - \dfrac{1}{t-\overline{z}} \right)d t \\    &   =
    \lim\limits_{A\to+\infty} \int\limits_{-A}^{A} \dfrac{e^{{\fo{{\rm{i}}\pi  n t }}}d t}{t-z}-
       \lim\limits_{A\to+\infty}\int\limits_{-A}^{A}\dfrac{e^{{\fo{{\rm{i}}\pi  n t }}}d t}{t-\overline{z}} = 2 \pi i  e^{{\fo{{\rm{i}}\pi  n z }}} \,,
\end{align*}

\noindent from which
\begin{align}\label{f2case36}
    &  \eurm{h}_{-n}^{\star}(P_{z})=\int\limits_{\Bb{R}}\! P_{z}(t)\,
    e^{{\fo{{\rm{i}}\pi  n t }}} d t=e^{{\fo{{\rm{i}}\pi  n z }}} =
      e^{{\fo{  {\rm{i}}\pi  n (x+ i y) }}} \ , \quad  n \geqslant 1 \ , \quad  z \in\Bb{H}\,.
\end{align}

 \noindent Applying complex conjugation, taking into account that $P_{z}(t)\in \Bb{R}$ for arbitrary
 $z \in\Bb{H}$ and $t \in\Bb{R}$, we get
 \begin{align}\label{f3case36}
    & \hspace{-0,2cm} \eurm{h}_{n}^{\star}(P_{z})=\int\limits_{\Bb{R}}\! P_{z}(t)\,
    e^{{\fo{ - {\rm{i}}\pi  n t }}} d t=e^{{\fo{- {\rm{i}}\pi  n \overline{z} }}}=
      e^{{\fo{ - {\rm{i}}\pi  n (x-i y) }}}
      \ , \quad  n \geqslant 1 \ , \ \   z \in\Bb{H}\,,
\end{align}

\noindent while
\begin{align}\nonumber
   \eurm{h}_{0}^{\star}(P_{z}) &=  \int\limits_{\Bb{R}}\! P_{z}(t)\,
     d t =  \dfrac{1}{ \pi }\int\limits_{\Bb{R}} \dfrac{y d t}{(t- x)^{2} + y^{2}} = \dfrac{1}{ \pi }\left.\arctan \dfrac{t- x}{y}\right|_{t=-\infty}^{t = + \infty} \\    &   =1\ , \quad  z \in\Bb{H}\,.
\label{f4case36}\end{align}

\noindent Next,
\begin{align*}
     \eurm{m}_{-n}^{\star}(P_{z})  &:= \int\limits_{\Bb{R}}\! P_{z}(t) \,
    e^{{\fo{ - \dfrac{{\rm{i}} \pi n }{t} }}}  d t = \left|t^{\,\prime}= - 1/t \right| =
    \int\limits_{\Bb{R}}\! \dfrac{P_{z}(-1/t)}{t^{2}} \,
    e^{{\fo{  {\rm{i}}\pi  n t }}}  d t \\    &   =
     \int\limits_{\Bb{R}}\! P_{-1/z}(t)\,
    e^{{\fo{  {\rm{i}}\pi  n t }}}  d t = \eurm{h}_{-n}^{\star}(P_{-1/z})=e^{{\fo{ -\dfrac{{\rm{i}}\pi  n}{z}  }}} =e^{{\fo{ -\dfrac{{\rm{i}}\pi  n}{x+ i y}  }}}
       \ ,
\end{align*}

\noindent i.e.,
\begin{align}\label{f5case36}
    & \eurm{m}_{-n}^{\star}(P_{z})= \int\limits_{\Bb{R}}\! P_{z}(t) \,
    e^{{\fo{ - \dfrac{{\rm{i}} \pi n }{t} }}}  d t = e^{{\fo{ -\dfrac{{\rm{i}}\pi  n}{z}  }}} =e^{{\fo{ -\dfrac{{\rm{i}}\pi  n}{x+ i y}  }}}
       \ ,  \quad  n \geqslant 1 \ , \quad  z \in\Bb{H}\,.
\end{align}

\noindent Applying complex conjugation, as above ,we get
 \begin{align}\label{f6case36}
    &\hspace{-0,25cm} \eurm{m}_{n}^{\star}(P_{z})= \int\limits_{\Bb{R}}\! P_{z}(t) \,
    e^{{\fo{  \dfrac{{\rm{i}} \pi n }{t} }}}  d t = e^{{\fo{\dfrac{{\rm{i}}\pi  n}{\overline{z}}  }}} =e^{{\fo{ \dfrac{{\rm{i}}\pi  n}{x- i y}  }}}
       \ ,  \ \  n \geqslant 1 \ , \ \   z\!=x\!+\!iy \in\Bb{H}\,.
\end{align}

\noindent Finally,
\begin{align}\label{f7case36}
    & \hspace{-0,3cm} \begin{array}{ll}
        \eurm{h}_{n}^{\star}(P_{z})\!=\! e^{{\fo{ - {\rm{i}}\pi  n (x-i y)  }}}
       \! = \!e^{{\fo{ - \pi  n (y +i x)  }}}\,,   &  \ \
        \eurm{m}_{n}^{\star}(P_{z})\!= \!e^{{\fo{ \dfrac{{\rm{i}}\pi  n}{x- i y}  }}}\!=\! e^{{\fo{ -\dfrac{\pi  n}{y+ i x}  }}}\,,
                 \\
         \eurm{h}_{-n}^{\star}(P_{z})\!=\!e^{{\fo{  {\rm{i}}\pi  n (x+ i y) }}}\!= \!
          e^{{\fo{ - \pi  n (y- i x) }}} \,, &  \ \
          \eurm{m}_{-n}^{\star}(P_{z})\!= \!e^{{\fo{ -\dfrac{{\rm{i}}\pi  n}{x+ i y}  }}}\!= \!
          e^{{\fo{ -\dfrac{\pi  n}{y- i x}  }}}\,,
       \end{array}\hspace{-0,1cm}
\end{align}

\noindent for all $n\geqslant 1$, which together with \eqref{f3case36} and  Theorem~\ref{mresth1} give
\begin{align*}
      \dfrac{1}{ \pi } \dfrac{y}{(t- x)^{2} + y^{2}}& = P_{z} (t)\! =\! \eurm{h}_{0}^{\star}(P_{z})\eurm{H}_{0}(t)\! +\!
    \sum\nolimits_{n\in \Bb{Z}_{\neq 0} }\!
    \Big( \eurm{h}_{n}^{\star}(P_{z}) \eurm{H}_{n}(t)\! +\! \eurm{m}_{n}^{\star}(P_{z}) \eurm{M}_{n}(t)\Big) \\    &=\eurm{H}_{0}(t)\! +\!
     \sum\limits_{n\geqslant 1 }\!\Big( e^{{\fo{ - \pi  n (y +i x)  }}} \eurm{H}_{n}(t)\! +\! e^{{\fo{ -\dfrac{\pi  n}{y+ i x}  }}} \eurm{M}_{n}(t)\Big)   \\    &      +
      \sum\limits_{n\geqslant 1 }\!\Big(  e^{{\fo{ - \pi  n (y- i x) }}} \eurm{H}_{-n}(t)\! +\! e^{{\fo{ -\dfrac{\pi  n}{y- i x}  }}}\eurm{M}_{-n}(t)\Big) \ , \quad   \ t, x\in \Bb{R}, \ y > 0 \,,
\end{align*}

\noindent which completes the proof of \eqref{f1conhypfouser}.

}}\end{clash}
\end{subequations}


\subsection[\hspace{-0,25cm}. \hspace{0,05cm}Notes for Section~\ref{mres}]{\hspace{-0,11cm}{\bf{.}} Notes for Section~\ref{mres}}$\phantom{a}$


\vspace{0,2cm}
\begin{subequations}
\begin{clash}{\rm{\hypertarget{r37}{}\label{case37}\hspace{0,00cm}{\hyperlink{br37}{$\uparrow$}}
  \ For arbitrary $a > 0$ we prove \eqref{f4pintforth1}. Obviously, for any $b>0$ we have
  \begin{align*}
    &  \int\limits_{0}^{\infty}
e^{{\fo{-\dfrac{b}{2}\left(t + \dfrac{1}{t}\right) -\dfrac{ 2 a t  }{t^{2} +1} }}}\dfrac{d t}{t^{2} +1} =
\int\limits_{0}^{\infty}
e^{{\fo{-b \dfrac{t + \dfrac{1}{t}}{2} -\dfrac{  a }{\dfrac{t + \dfrac{1}{t}}{2}} }}}\dfrac{d t}{t^{2} +1}
 \\    &  =
 2\int\limits_{1}^{\infty}
e^{{\fo{-b \dfrac{t + \dfrac{1}{t}}{2} -\dfrac{  a }{\dfrac{t + \dfrac{1}{t}}{2}} }}}\dfrac{d t}{t^{2} +1}
  \boxed{=} \\    &
2 t^{\,\prime} = t + \dfrac{1}{t} \, , \  2 d t^{\,\prime} = \left(1-\dfrac{1}{t^{2}}\right) d t =
\left(t-\dfrac{1}{t}\right)\dfrac{d t}{t}
\, , \ \dfrac{d t}{t} = \dfrac{ 2 d t^{\,\prime} }{t-\dfrac{1}{t}} \,,
 \\    &
t^{2} - 2 t  t^{\,\prime} + 1 = 0 \, , \  t = t^{\,\prime} + \sqrt{\left(t^{\,\prime}\right)^{2} -1} \, , \  \dfrac{1}{t} = t^{\,\prime} - \sqrt{\left(t^{\,\prime}\right)^{2} -1} \\    &
t-\dfrac{1}{t}=2\sqrt{\left(t^{\,\prime}\right)^{2} -1} \, , \
\dfrac{d t}{t^{2} +1}=\dfrac{1}{ t + \dfrac{1}{t}} \dfrac{d t}{t}  = \dfrac{1}{2 t^{\,\prime}}\dfrac{ 2 d t^{\,\prime} }{2\sqrt{\left(t^{\,\prime}\right)^{2} -1}}
 \\    &
\boxed{=} \ \int\limits_{1}^{\infty}e^{{\fo{-b t - \dfrac{a}{t}}}}\dfrac{d t}{t \sqrt{t^{2}-1}}
 \, , \
  \end{align*}

\noindent and\vspace{-0,2cm}
\begin{align*}
    &  \int\limits_{0}^{\infty}
e^{{\fo{-\dfrac{b}{2}\left(t + \dfrac{1}{t}\right) -\dfrac{ 2 a t  }{t^{2} +1} }}}\dfrac{d t}{t} =
2 \int\limits_{1}^{\infty}
e^{{\fo{-b \dfrac{t + \dfrac{1}{t}}{2} -\dfrac{  a }{\dfrac{t + \dfrac{1}{t}}{2}} }}}\dfrac{d t}{t}\boxed{=} \\    &
2 t^{\,\prime} = t + \dfrac{1}{t} \, , \  2 d t^{\,\prime} = \left(1-\dfrac{1}{t^{2}}\right) d t =
\left(t-\dfrac{1}{t}\right)\dfrac{d t}{t}
\, , \ \dfrac{d t}{t} = \dfrac{ 2 d t^{\,\prime} }{t-\dfrac{1}{t}} \,, \\    &
t^{2} - 2 t  t^{\,\prime} + 1 = 0 \, , \  t = t^{\,\prime} + \sqrt{\left(t^{\,\prime}\right)^{2} -1} \, , \  \dfrac{1}{t} = t^{\,\prime} - \sqrt{\left(t^{\,\prime}\right)^{2} -1} \\    &
t-\dfrac{1}{t}=2\sqrt{\left(t^{\,\prime}\right)^{2} -1} \, , \
\dfrac{d t}{t}=\dfrac{ 2 d t^{\,\prime} }{2\sqrt{\left(t^{\,\prime}\right)^{2} -1}}
 \\    &
\boxed{=} \ 2
\int\limits_{1}^{\infty}e^{{\fo{-b t - \dfrac{a}{t}}}}\dfrac{d t}{ \sqrt{t^{2}-1}} \, .
\end{align*}

\noindent Therefore, it is necessary to estimate the integrals
\begin{align*}
    &  2^{1-\sigma} \int\limits_{1}^{\infty}
    \dfrac{ e^{{\fo{ -b t-\dfrac{a}{t}  }} } d t}{t^{\sigma} \sqrt{t^{2}-1} } \, , \   \quad
    b=\left\{\begin{array}{cc}
               \pi ,&  \ \mbox{if} \ \sigma=1 \ ;\\
               \pi/2 ,& \ \mbox{if} \ \sigma=0 \ ;
             \end{array}\right. \  \quad  \sigma\in\{0, 1\}\,, \quad a>0\ .
\end{align*}

\noindent By the identity (see \cite[p.\! 82, (23)]{erd})
\begin{align*}
    & \int\limits_{0}^{\infty} \dfrac{ e^{{\fo{ -b t - \dfrac{a+1}{t} }} } d t}{ t^{1+\sigma} }=
  \left(\dfrac{b}{a+1}\right)^{{\fo{\dfrac{\sigma}{2}}}}
  K_{\sigma} \left(2\sqrt{b (a+1)}\right)
  \,,
\end{align*}

\noindent
 we find that
\begin{align*}
    &  \int\limits_{1}^{\infty}
    \dfrac{ e^{{\fo{ -b t-\dfrac{a}{t}  }} } d t}{t^{\sigma} \sqrt{t^{2}-1} }
=
\int\limits_{1}^{\infty}
    \dfrac{ e^{{\fo{\dfrac{1}{t} -b t-\dfrac{a+1}{t}  }} } d t}{t^{\sigma} \sqrt{t^{2}-1} }
\leqslant e\cdot
\int\limits_{1}^{\infty}
    \dfrac{ e^{{\fo{ -b t-\dfrac{a+1}{t}  }} } d t}{t^{\sigma} \sqrt{t^{2}-1} } \, , \\    &
\int\limits_{1}^{2}
    \dfrac{ e^{{\fo{ -b t-\dfrac{a+1}{t}  }} } d t}{ t^{\sigma}\sqrt{t^{2}-1} }\leqslant
e^{{\fo{ -\dfrac{a+1}{2}  }} }\int\limits_{1}^{2}
    \dfrac{ e^{{\fo{ -b t  }} } d t}{t^{\sigma} \sqrt{t^{2}-1} }=
\int\limits_{1}^{2}
    \dfrac{ e^{{\fo{ -b t  }} } d t}{t^{\sigma}\sqrt{t^{2}-1} }
e^{{\fo{- \dfrac{a+1}{2}  }} }
\dfrac{\int\limits_{2}^{\infty}
    \dfrac{ e^{{\fo{ -b t  }} } d t}{ t^{1+\sigma}  }}{
\int\limits_{2}^{\infty}
    \dfrac{ e^{{\fo{ -b t  }} } d t}{ t^{1+\sigma}  }} \\ &\leqslant
\int\limits_{1}^{2}
    \dfrac{ e^{{\fo{ -b t  }} } d t}{t^{\sigma} \sqrt{t^{2}-1} }
\dfrac{\int\limits_{2}^{\infty}
    \dfrac{ e^{{\fo{ -b t - \dfrac{a+1}{t} }} } d t}{ t^{1+\sigma} }}{
\int\limits_{2}^{\infty}
    \dfrac{ e^{{\fo{ -b t  }} } d t}{ t^{1+\sigma} }}
\leqslant
2\dfrac{\int\limits_{1}^{2}
    \dfrac{ e^{{\fo{ -b t  }} } d t}{t^{\sigma} \sqrt{t^{2}-1} }}{
\int\limits_{2}^{\infty}
    \dfrac{ e^{{\fo{ -b t  }} } d t}{ t^{1+\sigma}}}
    \left(\dfrac{b}{a+1}\right)^{{\fo{\dfrac{\sigma}{2}}}}
  K_{\sigma}\! \left(2\sqrt{b (a+1)}\right) \, , \\    &
\int\limits_{2}^{\infty}
    \dfrac{ e^{{\fo{ -b t-\dfrac{a+1}{t}  }} } d t}{t^{\sigma} \sqrt{t^2 -1} }<
|t^2 -1>(3/4) t^2 \ \Leftrightarrow \ t^2> 4  |<\\ & <
\dfrac{2}{\sqrt{3}}\int\limits_{2}^{\infty}
    \dfrac{ e^{{\fo{ -b t-\dfrac{a+1}{t}  }} } d t}{ t^{1+\sigma} }\leqslant
\dfrac{4}{\sqrt{3}} \left(\dfrac{b}{a+1}\right)^{{\fo{\dfrac{\sigma}{2}}}}
  K_{\sigma}\! \left(2\sqrt{b (a+1)}\right)\,,
\end{align*}

\noindent i.e.,
\begin{align*}
    & \int\limits_{1}^{\infty}
    \dfrac{ e^{{\fo{ -b t-\dfrac{a}{t}  }} } d t}{t^{\sigma} \sqrt{t^{2}-1} }
    \leqslant 2 e\left(\dfrac{2}{\sqrt{3}} +\dfrac{\int\limits_{1}^{2}
    \dfrac{ e^{{\fo{ -b t  }} } d t}{t^{\sigma} \sqrt{t^{2}-1} }}{
\int\limits_{2}^{\infty}
    \dfrac{ e^{{\fo{ -b t  }} } d t}{ t^{1+\sigma}}}
    \right)
  \left(\dfrac{b}{a+1}\right)^{{\fo{\dfrac{\sigma}{2}}}}  K_{\sigma}\! \left(2\sqrt{b (a+1)}\right)\,,
\end{align*}

\noindent where, in view of \cite[5.1.1, 5.1.4, p.228]{abr},
\begin{align*}
    & \int\limits_{2}^{\infty}
    \dfrac{ e^{{\fo{ -b t  }} } d t}{ t }=
\int\limits_{2b}^{\infty}
    \dfrac{ e^{{\fo{ - t  }} } d t}{ t }= E_{1}(2b) \, , \ \
    \int\limits_{2}^{\infty}
    \dfrac{ e^{{\fo{ -b t  }} } d t}{ t^{2} }=\dfrac{1}{2}
\int\limits_{1}^{\infty}
    \dfrac{ e^{{\fo{ -2b t  }} } d t}{ t^{2} }=\dfrac{E_{2}(2b)}{2}\, .
\end{align*}

\noindent But the relations (see \cite[5.1.14, p.229]{abr}, \cite[6.8.2, p.152]{olv})
\begin{align*}
    &  E_{2}(2s)=e^{-2s}-2s E_{1}(2s) \ , \quad  \dfrac{ e^{{\fo{ -2s   }} }}{2s +1}\leqslant
    E_{1}(2s)\leqslant \dfrac{e^{-2s}}{2s}\dfrac{2s+1}{2s+2} \ , \quad  s> 0\,,
\end{align*}

\noindent yield that
\begin{align*}
    &  \int\limits_{2}^{\infty}
    \dfrac{ e^{{\fo{ -b t  }} } d t}{ t }=
 E_{1}(2b) \geqslant
\dfrac{ e^{{\fo{ -2b   }} }}{2b +1}
\,, \  \\    &
\int\limits_{2}^{\infty}
    \dfrac{ e^{{\fo{ -b t  }} } d t}{ t^{2} }=\dfrac{E_{2}(2b)}{2}=\dfrac{e^{-2b}-2b E_{1}(2b)}{2}\geqslant
\dfrac{e^{-2b}-2b \dfrac{e^{-2b}}{2b}\dfrac{2b+1}{2b+2}}{2}=\dfrac{e^{-2b}}{4b+4} \ .
\end{align*}

\noindent or, what is the same,
\begin{align*}
    & \int\limits_{2}^{\infty}
    \dfrac{ e^{{\fo{ -b t  }} } d t}{ t^{1+\sigma}}
    \geqslant \dfrac{ e^{ {\fo{ -2b   }} } }{ 2^{\sigma }(2b +1+\sigma) }\ ,
    \  \quad  \sigma\in\{0, 1\}\,.
\end{align*}

\noindent These inequalities together with (see  \cite[4.6.2, p.86]{abr},
\cite[4.6.21, p.87]{abr})
\begin{align*}
    &  \int\limits_{1}^{2}
    \dfrac{ e^{{\fo{ -b t  }} } d t}{t^{\sigma} \sqrt{t^{2}-1} }
    \leqslant \int\limits_{1}^{2}
    \dfrac{ e^{{\fo{ -b t  }} } d t}{ \sqrt{t^{2}-1} }
     \leqslant
e^{{\fo{ -b   }} }
\int\limits_{1}^{2}
    \dfrac{  d t}{ \sqrt{t^{2}-1} }=e^{{\fo{ -b   }} } {\rm{arccosh}}\, 2 =
    e^{{\fo{ -b   }} }\log\left(2 +\sqrt{3}\right) \, ,
\end{align*}

\noindent
imply
\begin{align*}
    & \dfrac{\int\limits_{1}^{2}
    \dfrac{ e^{{\fo{ -b t  }} } d t}{t^{\sigma} \sqrt{t^{2}-1} }}{
\int\limits_{2}^{\infty}
    \dfrac{ e^{{\fo{ -b t  }} } d t}{ t^{1+\sigma}}}
     \leqslant
     \dfrac{e^{{\fo{ -b   }} }\log\left(2 +\sqrt{3}\right)}{\dfrac{ e^{ {\fo{ -2b   }} } }{
     2^{\sigma }(2b +1+\sigma) }}
      \leqslant
     2^{\sigma } e^{{\fo{ b   }} }(2b +1+\sigma)\log\left(2 +\sqrt{3}\right)
\end{align*}

\noindent and therefore
\begin{align*}
    & 2^{1\!-\!\sigma}\!\!\! \int\limits_{1}^{\infty}
    \dfrac{ e^{{\fo{ -b t-\dfrac{a}{t}  }} } d t}{t^{\sigma} \sqrt{t^{2}-1} }
    \leqslant 4 e\left(\dfrac{2^{1-\sigma}}{\sqrt{3}} +
    e^{{\fo{ b   }} }(2b\! +\!1\!+\!\sigma)\log\left(2 +\sqrt{3}\right)
    \right)
  \dfrac{b^{{\fo{\dfrac{\sigma}{2}}}} K_{\sigma}\! \left(2\sqrt{b (a+1)}\right)}{(a+1)^{{\fo{\dfrac{\sigma}{2}}}}}
\end{align*}

\noindent Numerical calculations give
\begin{align*}
    & 2^{1\!-\!\sigma}\!\!\! \int\limits_{1}^{\infty}
    \dfrac{ e^{{\fo{ -b t-\dfrac{a}{t}  }} } d t}{t^{\sigma} \sqrt{t^{2}-1} }
    \leqslant
    \left\{
    \begin{array}{cc}
     \pi^{5} K_{0}\! \left(\sqrt{2\pi (a+1)}\right)\ , &  \ \mbox{if} \ \sigma =0\,;\\
   \pi^{8} \dfrac{K_{1}\! \left(2\sqrt{\pi (a+1)}\right)}{\sqrt{a+1}}\ , & \ \mbox{if} \ \sigma=1\,.
    \end{array}
    \right.
\end{align*}

\noindent Thus,
\begin{align*}
    &   \int\limits_{0}^{\infty}
e^{{\fo{-\dfrac{\pi}{2}\left(t + \dfrac{1}{t}\right) -\dfrac{ 2 a t  }{t^{2} +1} }}}\dfrac{d t}{t^{2} +1}
\leqslant
\pi^{8}\dfrac{K_{1} \left(2\sqrt{\pi (a+1)}\right)}{\sqrt{a+1}} \, , \\    &
\int\limits_{0}^{\infty}
e^{{\fo{-\dfrac{\pi}{4}\left(t + \dfrac{1}{t}\right) -\dfrac{ 2 a t  }{t^{2} +1} }}}\dfrac{d t}{t}
\leqslant  \pi^{5}
  K_{0} \left(\sqrt{2\pi (a+1)}\right)   \, , \  a > 0\,,
\end{align*}

\noindent which proves \eqref{f4pintforth1}.

}}\end{clash}
\end{subequations}



\vspace{0,1cm}

\addcontentsline{toc}{section}{A\hspace{0,03cm}.\hspace{0,07cm} References for Supplementary Notes}


\end{document}